\documentclass[a4paper, 11pt, twoside, final]{book}
\usepackage[T1]{fontenc}
\usepackage{amsmath}
\usepackage{amsthm}
\usepackage{amssymb}
\usepackage{amsfonts}
\usepackage{latexsym}
\usepackage{makeidx}
\usepackage{fancyhdr}
\usepackage{graphicx}

\usepackage{tabularx}

\usepackage{pstricks-add}

\usepackage{psfrag}
\usepackage{listings}
\usepackage{multicol}

\usepackage[notref,notcite]{showkeys}
\usepackage{mathrsfs}

\usepackage{epsfig}
\usepackage{mathrsfs}
\usepackage{stmaryrd}
\usepackage{textcomp}

\renewcommand{\baselinestretch}{1.3} \normalsize
\newcommand{\R}{\mathbb{R}}

\newcommand{\N}{\mathbb{N}}
\newcommand{\Z}{\mathbb{Z}}
\newcommand{\Sp}{\mathbb{S}}

\newcommand{\Ns}{\mathcal{N}}
\newcommand{\U}{\mathcal{U}}
\newcommand{\I}{\mathcal{I}}

\newcommand{\E}{\mathscr{E}}
\newcommand{\Le}{\mathscr{L}}
\newcommand{\Ha}{\mathscr{H}}
\newcommand{\M}{\mathscr{M}}

\newcommand{\eps}{\varepsilon}

\newcommand{\kla}[1]{\ensuremath{\left(#1\right)}}
\newcommand{\norm}[1]{\ensuremath{\left\lVert#1\right\rVert}}
\newcommand{\abs}[1]{\ensuremath{\lvert#1\rvert}}
\newcommand{\babs}[1]{\ensuremath{\Bigl\lvert#1\Bigr\rvert}}
\newcommand{\menge}[1]{\ensuremath{\left\{#1\right\}}}
\newcommand{\eins}[1]{\ensuremath{\Big\arrowvert_{#1}}}
\newcommand{\intl}{\int\limits}
\newcommand{\iiintl}{\iiint\limits}
\newcommand{\suml}{\sum\limits}
\newcommand{\name}[1]{\textsc{#1}}

\newcommand{\mppack}{MpPACK}
\newcommand{\lapack}{LAPACK}
\newcommand{\povray}{PovRay}
\newcommand{\gnuplot}{GnuPlot}
\newcommand{\python}{Python}
\newcommand{\maple}{Maple{\scriptsize{\texttrademark}}}
\newcommand{\cpp}{{C\footnotesize{++}}}
\newcommand{\knotplot}{KnotPlot}
\newcommand{\libbiarc}{libbiarc}

\newcounter{constc}

\newcounter{consttheta}

\newcounter{constTheta}

\def\mvint_#1{\mathchoice
          {\mathop{\vrule width 6pt height 3 pt depth -2.5pt
                  \kern -8pt \intop}\nolimits_{\kern -3pt #1}}%
          {\mathop{\vrule width 5pt height 3 pt depth -2.6pt
                  \kern -6pt \intop}\nolimits_{#1}}%
          {\mathop{\vrule width 5pt height 3 pt depth -2.6pt
                  \kern -6pt \intop}\nolimits_{#1}}%
          {\mathop{\vrule width 5pt height 3 pt depth -2.6pt
                  \kern -6pt \intop}\nolimits_{#1}}}

\usepackage{nextpage} 
\usepackage{enumerate}
\usepackage{float} 
\newenvironment{enumabc}{\begin{enumerate}[(a)]}{\end{enumerate}}

\newtheorem{thm}{Theorem}[chapter]
\newtheorem{lem}[thm]{Lemma}
\newtheorem{cor}[thm]{Corollary}
\newtheorem{prop}[thm]{Proposition}

\theoremstyle{definition}
\newtheorem{defn}[thm]{Definition}
\newtheorem{exmp}[thm]{Example}

\theoremstyle{remark}
\newtheorem{rem}[thm]{Remark}

\usepackage[a4paper, inner=3cm, top=3.5cm, bottom=3.5cm, textwidth=15cm]{geometry}

\usepackage{theoremref}
\newcommand{\mylabel}[1]{\thlabel{#1}\label{T#1}}

\begin{document}

\pagenumbering{Roman}
\pagestyle{empty} 

\newcommand{\globaltitle}{Analysis of the first variation and a numerical gradient flow for integral \name{Menger} curvature}
\title{\globaltitle}
\author{Tobias Hermes}
\date{\today}

\renewcommand{\baselinestretch}{1.3} \LARGE
\begin{titlepage}
\begin{center}
\vspace*{0.1cm}
\textbf{\LARGE \globaltitle}\\[2cm]
\Large Von der Fakult\"at f\"ur Mathematik, Informatik und Naturwissenschaften der RWTH Aachen University zur Erlangung des akademischen Grades eines Doktors der Naturwissenschaften genehmigte Dissertation\\[1.2cm]
\Large vorgelegt von\\[0.5cm]
\Large Diplom-Computermathematiker\\
\Large Tobias Hermes\\[0.1cm]
\Large aus Neuss\\[2.5cm]
\begin{tabbing}
\Large \hspace{1.5cm}Berichter: \=Univ.-Prof. Dr. Heiko von der Mosel\\
\Large \>AOR Priv.-Doz. Dr. Alfred Wagner\\[1cm]
\end{tabbing}
\Large Tag der m\"undlichen Pr\"ufung: 14.06.2012\\[2cm]
\end{center}
\textbf{\small Diese Dissertation ist auf den Internetseiten der Hochschulbibliothek online verf\"ugbar.}
\end{titlepage}

\renewcommand{\baselinestretch}{1.3} \normalsize

\pagestyle{plain}
\cleartooddpage[\thispagestyle{empty}]
\chapter*{Abstract}
In this thesis, we consider the knot energy ``integral \name{Menger} curvature'' which is the triple integral over the inverse of the classic circumradius of three distinct points on the given knot to the power $p\in [2,\infty)$. We prove the existence of the first variation for a subset of a certain fractional \name{Sobolev} space if $p>3$ and for a subset of a certain \name{H\"older} space otherwise. We also discuss how fractional \name{Sobolev} and \name{H\"older} spaces can be generalised for $2\pi$-periodic, closed curves. Since this energy is not invariant under scaling, we additionally consider a rescaled version of the energy, where we take the energy to the power one over $p$ and multiply by the length of the curve to a certain power. We prove that a circle is at least a stationary point of the rescaled energy. Furthermore, we show that in general a functional with a scaling behaviour like $E(r\gamma)=r^\alpha E(\gamma)$, $\alpha\in\R$ cannot have stationary points unless $\alpha=0$. Consequently ``integral \name{Menger} curvature'' for $p\neq 3$ can be used as a subsidiary condition for a \name{Lagrange}-multiplier rule.\\
We consider a numerical gradient flow for the rescaled energy. For this purpose we use trigonometric polynomials to approximate the knots and the trapezoidal rule for numerical integration, which is very efficient in this case. Moreover, we derive a suitable representation of the first variation. We present an algorithm for the adaptive choice of time step size and for the redistribution of the \name{Fourier} coefficients. After discussing the full discretization we present a wide collection of example flows.
\cleartooddpage[\thispagestyle{empty}]


\setlength{\parskip}{0.0ex} 

\tableofcontents
\cleartooddpage[\thispagestyle{empty}]

\listoffigures
\cleartooddpage[\thispagestyle{empty}]

\pagestyle{fancy}
\fancyhead[EL,OR]{\thepage}
\fancyhead[ER]{\nouppercase{\rightmark}}
\fancyhead[OL]{\nouppercase{\leftmark}}
\fancyfoot{}
\addtolength{\headheight}{12pt}

\setlength{\parskip}{2.0ex plus 1.0ex minus 0.5ex} 
\pagenumbering{arabic}


\chapter{Introduction}
In this thesis we are considering knots. These are embeddings of the unit circle (in other words the one dimensional unit sphere $\Sp^1$) into the \name{Euclidean} three-dimensional space, in particular a function $f: \Sp^1 \rightarrow \R^3$ such that $f: \Sp^1 \rightarrow f(\Sp^1)$ is a homeomorphism.

Knots are objects that appear in everyday life for example when tying your shoes. In this context one thinks about a knot as something like a piece of rope. The mathematical model for this describes the centreline of such a rope or in other words, a rope with infinitesimal diameter. In real life we normally are dealing with open knots. However, it is difficult to handle them mathematically, since they are topologically equivalent to a straight line, cf. \cite{pisa_rawdon}. Therefore, traditional mathematical knots are closed. Nevertheless, there are many applications for closed curves in physics and biology, for instance in the field of ring polymers \cite{millett} and supercoiled DNA sequences \cite{dewitt}.

We are interested in defining an equivalence class for knots, such that two knots are ``equal'' if they represent the same ``way of knotting'' in our opinion and can be transformed into each other. Using \emph{cut-and-seal} actions this can be achieved for every pair of knots. Proteins manipulate closed DNA molecules in such a way and thereby changing their knot type \cite{dewitt}. Hence, as mentioned before, this is a very important and interesting application. However, these actions are not permitted here. The exact mathematical definition can be found in Chapter~\ref{knottheory}. At least for closed polygons an equivalent definition exists, which goes back to \name{R. Reidemeister}, who introduced the so-called \emph{\name{Reidemeister} moves} in 1932. Two closed polygons that can be transformed into each other by a certain number of those moves are in the same knot class and vice versa. A knot, which is in the same knot class as the circle is called an \emph{unknot}. Even an unknot can assume very headstrong shapes, such that it is very difficult to identify its knot class at first sight. This clarifies the need for a method identifying the knot class. Another goal of research is to determine invariants for knots, but one that could distinguish all knot classes has not been found up to now. Moreover, it is interesting to consider knots mapping into an arbitrary manifold instead of mapping into $\R^3$, respectively knots moving on geometric objects, such as the surface of a torus. A detailed introduction can be found in \cite{BZ03} and \cite{CF77}.

Besides the classic knot theory the so-called \emph{geometric knot theory} has been developed, where among other topics \emph{knot energies} were considered. Such energies are often motivated by physical considerations and one hopes that physical phenomena could be modelled appropriately. In the field of fluid mechanics for instance, one considers knotted structures influenced by a continuous flow, that try to ``relax'', cf. \cite{Mof98}. Other examples are magnetohydrodynamics \cite{Mof85} and vortex filaments \cite{calini}. Another application is the search for the ``optimal'' shape of a knot by minimizing a corresponding energy within a given knot class. The problem to choose the ``optimal'' representative of each knot class in general is hard or even impossible to solve, hence we do not have an intuitive conception of such a configuration and the results may differ strongly. ``Optimality'' of a knot could mean that it has some pure geometric properties, that its parametrization is as simple as possible, or that there are some strong symmetries in its projection. Regardless of the kind of ``optimality'' a transformation into such an ``optimal'' configuration should have the advantage that some properties of the knot could be read off easily. The search for the ``optimal'' shape in the sense of minimizing a knot energy, can be done using a so-called \emph{gradient flow}, where the knot deforms in the direction of the negative gradient of the energy until a critical point is reached. This method helps determining whether two knots are in the same knot class, by applying this gradient flow to both knots and comparing their respective final configuration. However, to do so one has to assume that the knot class is not abandoned during the flow.
\name{J. O'Hara} gives an axiomatic definition of a knot energy functional in \cite{OH92}, for which it is essential that the value of the functional expands beyond all bounds if the knot degenerates to a continuous mapping with self-intersections and that knots with congruent images were assigned to the same value. A functional with the first property is called \emph{self-repulsive} \cite[Definition 1.1]{ohara} and is also known as a \emph{charge energy} \cite[Definition 2.5]{DEJvR}. The self-repulsiveness should prevent the knot from abandoning its knot class during a gradient flow, but there are repulsive functionals that do not punish the ``pulling-tight'' of small knots, cf. \cite[Theorem 3.1 (2)]{OH92}.

We start by imagining electrons being attached along a curve, which repel each other. A first approach like this
\begin{equation} \label{I001}
\intl_{\Sp^1} \intl_{\Sp^1} \frac{1}{\abs{\gamma(s)-\gamma(t)}^p} \, ds \, dt, \qquad p\geq 2,
\end{equation}
is infinite through nearest neighbour effects for every smooth and closed curve $\gamma$. Different examples for such repulsive energies were considered by \name{R. B. Kusner} and \name{J. M. Sullivan} \cite{KS98_Moeb} among others. However, first models of repulsive energies for smooth knots were defined by \name{O'Hara} in \cite{OH91} and \cite{OH92}. For $C^2$-embeddings he considered energies like this
\begin{equation} \label{defOHenergy}
E^\alpha(\gamma) := \intl_{\Sp^1} \intl_{\Sp^1} \kla{\frac{1}{\abs{\gamma(s)-\gamma(t)}^\alpha}-\frac{1}{\text{D}_\gamma(s,t)^\alpha}} \abs{\gamma'(s)} \abs{\gamma'(t)} \, ds \, dt, \quad \alpha \in [2,3),
\end{equation}
where $\text{D}_\gamma$ denotes the intrinsic distance function on the curve. The problem of \eqref{I001} is that all pairs of points that are close to each other are penalised. Here, on the contrary, only the pairs of points having small distance in space but being far away on a sufficiently smooth curve, are punished, because the \name{Euclidean}-distance and $\text{D}_\gamma$ are nearly the same for pairs of points close to each other. For $\alpha=2$ this energy is also called \emph{\name{M\"obius}~energy}, which is the first one for which satisfying existence and regularity results have been found. The name is due to the invariance under \name{M\"obius} transformations, which was found by \name{M. H. Freedman}, \name{Y.-X. He} and \name{Z. Wang}, who analysed it comprehensively in \cite{FHW94}. Using this invariance they prove that circles are the global minimizers and verify the existence and $C^{1,1}$-regularity of minimizers in prime knot classes. \emph{Prime knots} are knots that cannot be divided into simpler knots (see \cite{RS06} for details). Moreover, \name{Kusner} and \name{Sullivan} \cite{KS98_Moeb} conjecture that there are no minimizers among non-prime knots. Another important result of their work is the statement that given a fixed energy bound there are only finitely many knot classes having a representative, whose energy is bounded by this given value. The regularity of minimizers could be improved to $C^\infty$, see \name{He} \cite{He00} and the rigorous analysis of \name{P. Reiter} in \cite{reiter_reg} and his PhD thesis \cite{phd_reiter}, \cite{reiter_pdc}. Very recently \name{S. Blatt}, \name{Reiter} and \name{A. Schikorra} \cite{BRS12} proved that any stationary point of $E^2$ with finite \name{M\"obius}~energy is $C^\infty$, without using the \name{M\"obius} invariance. Also find a similar result for the family of energies $E^\alpha$, $\alpha\in(2,3)$ in \cite{blattreiter}. The uniqueness of minimizers for various repulsive knot energies including \name{M\"obius} and \name{O'Hara}'s energies is proven in \cite{circmin}. Using residue calculus \name{D. Kim} and \name{Kusner} in \cite{KK93} were able to explicitly compute the \name{M\"obius}~energy for $E^2$-critical torus knots. To the best of our knowledge \name{Blatt} is the only one so far, who proved long time existence for the gradient flow for \name{M\"obius}~energy \cite{blatt_flow}. A linear combination of an elastic energy and the \name{M\"obius}~energy is considered in \cite{LinSch10}.

Besides using repulsive energies in order to achieve self-avoidance, which could be categorised as a ``soft'' method, one can also consider the other extreme to assign a constant ``thickness'' to a curve $\gamma$ and using its inverse as a knot energy of $\gamma$. This models a ``hard'' or steric constraint and can be expressed by the \emph{global radius of curvature}, which was introduced by \name{O. Gonzales} and \name{J. H. Maddocks} in \cite{GM99} and extensively analysed together with \name{F. Schuricht} and \name{H. von der Mosel} in \cite{GMSvdM02} and \cite{SchvdM03_global}. Foundation of this approach is the classic \emph{circumradius} $R$, which is the radius of the unique circle passing through three distinct non-collinear points in $\R^3$. Three point are \emph{collinear} if they lie on one straight line. \name{Gonzales} and \name{Maddocks} defined the global radius of curvature for all $x\in \gamma(\Sp^1)$ as
\begin{equation} \label{defgrc}
\rho_G[\gamma](x) := \inf\limits_{\stackrel{y,z \in \gamma(\Sp^1)}{x\neq y\neq z\neq x}} R(x,y,z)
\end{equation}
and moreover,
\begin{equation} \label{thick}
\Delta[\gamma] := \inf\limits_{x\in \gamma(\Sp^1)} \rho_G[\gamma](x),
\end{equation}
as the \emph{thickness} of $\gamma$. Now there exists a ``tubular neighbourhood'' of $\gamma$ that has thickness $\eta:=\Delta[\gamma]$, since every open ball with radius $\eta$ that touches $\gamma$, is not intersected by the curve, see \cite[Lemma 3]{GMSvdM02}. For a $C^2$-curve $\gamma$ it is proven in \cite{GM99} that $\Delta[\gamma]$ is equal to its \emph{normal injectivity radius}, which is defined by considering normal circular discs with fixed radius around every point on the curve and increasing the radius until there are intersecting discs (see for instance \cite{Sim98}). The same is true for a curve $\gamma$ that is only continuous and rectifiable. The normal injectivity radius can than be defined via the arclength parametrization, which turns out to be of class $C^{1,1}$ if $\Delta[\gamma]>0$, see \cite[Lemmas 2 and 3]{GMSvdM02}. Definition \eqref{thick} is not just physically meaningful and mathematically precise but also analytically useful, in particular for the calculus of variations. More examples for defining ``thickness'' can be found in \cite{KS98_thick}. The idea of a completely tightened knot in a thick rope leads to the concept of \emph{ideal knots}. Those are the shortest representatives with a fixed thickness $\Delta[\cdot]>0$. For example one could image that a loose trefoil knot in a thick rope is pulled tighter and tighter by some process, finally becoming the ideal trefoil knot. Important in this context is the so-called \emph{ropelength}, which is the quotient of length over thickness. In \cite{GMSvdM02} and independently in \cite{CKS02} the existence of ideal knots in every knot class is established. In \cite{SchvdM04}, \cite{CKS02}, \cite{CFKSW06} and \cite{CFKS11} some characteristic properties of ideal knots are shown by means of techniques from non-smooth analysis.

As mentioned before, these methods are useful for analytical examinations but seem to be less suitable for numerical approaches. Therefore, we consider a somehow intermediate approach between ``soft'' and ``hard'' potentials. The idea is to replace the supremization of a variable, which  corresponds to the infimization of the inverse, by an integration over $\Sp^1$. The following energy, already proposed in \cite{GM99},
\begin{equation} \label{energyU}
\mathscr{U}_p(\gamma) := \kla{\;\intl_{\gamma(\Sp^1)} \frac{1}{\rho_G[\gamma](x)^p} \, d\Ha^1(x)}^\frac{1}{p}, \quad p\geq 1,
\end{equation}
which is actually the $L_p$-norm of the inverse of the global radius of curvature, was precisely analysed in \cite{StrvdM07}. \name{P. Strzelecki} and \name{von der Mosel} proved, that finite energy of a curve implies that it is embedded and that the representative parametrized by arclength is actually in the \name{Sobolev} space $W^{2,p}$. In turn, embedded curves in $W^{2,p}$ for $p>1$ can be characterised by finite $\mathscr{U}_p$-energy, see \cite[Theorem 1 and 2]{StrSzvdM09}. Moreover, it is shown, that $\mathscr{U}_p$-minimizers exist in each knot class, not only for prime knots, and that the circle is the unique global minimizer for closed curves with prescribed length. Replacing yet another suprimization by integration, we obtain
\begin{equation} \label{energyI}
\mathscr{I}_p(\gamma) := \intl_{\gamma(\Sp^1)} \intl_{\gamma(\Sp^1)} \frac{1}{\inf_{z\in\gamma(\Sp^1)} R(x,y,z)^p} \, d\Ha^1(x) \, d\Ha^1(y), \quad p\geq 2,
\end{equation}
which was considered in \cite{StrSzvdM09}. They proved, that finite energy implies that there are no self-intersections of the curve and that for $p>2$ the parametrization by arclength is actually in the \name{H\"older} space $C^{1,\alpha}$ with $\alpha=1-\frac{2}{p}$. In a last step we come to the so-called \emph{integral \name{Menger} curvature}, which is the main topic of this thesis
\begin{equation} \label{energyM}
\M_p(\gamma) := \intl_{\gamma(\Sp^1)} \intl_{\gamma(\Sp^1)} \intl_{\gamma(\Sp^1)} \frac{1}{R(x,y,z)^p} \, d\Ha^1(x) \, d\Ha^1(y) \, d\Ha^1(z), \quad p\geq 2.
\end{equation}
This energy was also already mentioned in \cite{GM99}. Observe that for $x,y$ and $z$ on the curve, with $y$ and $z$ moving to $x$, the inverse of $R(x,y,z)$ tends to the classic local curvature of $\gamma$ at $x$. Therefore, in contrast to our first attempt \eqref{I001}, it is not necessary to regularise the integrands in \eqref{energyU}, \eqref{energyI} and \eqref{energyM}. See also the advantages of integral \name{Menger} curvature over repulsive potentials from a physical point of view in \cite{banavar}. Measure theoretic results of \name{L\'eger} for sets $E\subset \R^n$ with finite one-dimensional \name{Hausdorff} measure can be found in \cite{Leg99}. For the generalised energy of those sets, he proved that $\M_2(E)<\infty$ implies that $E$ is \emph{rectifiable}, meaning that $E$ is contained in a countable family of \name{Lipschitz} images up to a null set with respect to \name{Hausdorff} measure.

In \cite[Theorem 1.2]{StrSzvdM10} \name{Strzelecki}, \name{M. Szuma\'nska} and \name{von der Mosel} proved for $p>3$ that a curve, which is a local homeomorphism, furthermore, parameterized by arclength and with finite integral \name{Menger} curvature, is of class $C^{1,1-\frac{3}{p}}(\Sp^1,\R^3)$. If there exists at least one simple point of $\gamma$, then $\gamma$ is injective. \name{Blatt} came up with the idea to consider the \name{Sobolev-Slobodeckij} function spaces for integral \name{Menger} curvature. In \cite[Theorem~1.1]{blatt_menger} he was able to prove that for $p>3$ an injective curve $\gamma\in C^1(\Sp^1,\R^3)$ parameterized by arclength has finite integral \name{Menger} curvature $\M_p(\gamma)<\infty$ if and only if $\gamma\in W^{2-\frac{2}{p},p}(\Sp^1,\R^3)$. The space $W^{2-\frac{2}{p},p}$ is a \name{Sobolev} space with a fractional order of differentiation. In particular, a $\gamma$ in this \name{Sobolev-Slobodeckij} space is of class $W^{1,p}(\Sp^1,\R^3)$ and in addition the following semi norm is finite
\begin{equation} \label{deffracSobolevS}
[\gamma']_{W^{s,p}(\Sp^1,\R^3)} := \kla{\,\intl_{\Sp^1} \intl_{\Sp^1} \;\frac{\abs{\gamma'(x)-\gamma'(y)}^p}{\abs{x-y}^{1+sp}}\,dx\,dy}^\frac{1}{p} \; < \; \infty \qquad \text{with } s=1-\frac{2}{p}.
\end{equation}
Together with the theorem of \name{Strzelecki}, \name{Szuma\'nska} and \name{von der Mosel}, this directly leads to the following statement (cf. \cite[Corollary~1.2]{blatt_menger}). For $p>3$ and a local homeomorphism $\gamma\in C^{0,1}(\Sp^1,\R^3)$ parametrized by arclength with at least one simple point, $\gamma$ has finite integral \name{Menger} curvature $\M_p(\gamma)<\infty$ if and only if $\gamma$ is embedded and belongs to $W^{2-\frac{2}{p},p}(\Sp^1,\R^3)$.

In this thesis we prove the existence of the first variation of $\M_p$, $p\in(3,\infty)$ and derive it for injective and regular functions in $W^{2-\frac{2}{p},p}(\Sp^1,\R^3)$ (see \thref{thmdiff}). Here \emph{regular} means that the norm of the tangent does not vanish on $\Sp^1$. The reason for the restriction to $p>3$ is our need for the embedding $W^{2-\frac{2}{p},p}(\Sp^1,\R^3)\subset C^{1,1-\frac{3}{p}}(\Sp^1,\R^3)$ \cite[Theorem 8.2]{hitchhiker}. After developing the proof of \thref{thmdiff} independently, these final result were reached by slightly improving \name{Blatt}'s methods in \cite[Theorem~1.1]{blatt_menger}. This proof can easily be extended to the case $p\in[2,\infty)$ for injective and regular functions in $C^{1,\alpha}(\Sp^1,\R^3)$ with $\alpha\in(1-\frac{2}{p},1]$ (see \thref{remsmallp}). Observe that $C^{1,\alpha}(\Sp^1,\R^3)\subset W^{2-\frac{2}{p},p}(\Sp^1,\R^3)$, for such $\alpha$, which we will prove later on. One part of the proof is to show that $\M_p$ is finite for knots in the respective function space. Another proof that $C^{1,\alpha}(\Sp^1,\R^3)$ regularity with $\alpha\in(1-\frac{2}{p},1]$ implies finite integral \name{Menger} curvature can also be found in \cite{KS11} for the case $p>2$. Moreover, \name{S. Kolasi\'nski} and \name{Szuma\'nska} showed in their paper that this \name{H\"older}-exponent is optimal by constructing a $C^{1,1-\frac{2}{p}}$-knot with infinite energy.


Integral \name{Menger} curvature is not invariant under scaling for $p\neq 3$, since we have
\begin{equation} \label{IscaleMp}
\M_p(r\gamma) = r^{3-p} \; \M_p(\gamma), \qquad r>0,
\end{equation}
which causes the effect that it is possible to reduce the energy by simply scaling the knot up (for $p>3$) respectively down (for $p<3$). This is something we are not interested in. Moreover, we will see later in this thesis, that there are no stationary points of $\M_p$, $p\neq 3$ at all. In particular, we prove that this is the case for every functional $E$ that scales like $E(r\gamma)=r^\alpha\,E(\gamma)$ for some $\alpha\neq 0$ and with $E(\gamma)\neq 0$ for every $\gamma$. However, this has the consequence that the $\M_p$-functional for $p\neq 3$ can always be used as a subsidiary condition for the calculus of variations in the sense of a \name{Lagrange}-multiplier rule (cf. \cite[2, Theorem 1]{GH96}. In order to find local minimizers, we introduce a modified integral \name{Menger} curvature energy
\begin{equation} \label{defmodmenger}
\E_p(\gamma) := \kla{\M_p(\gamma)}^\frac{1}{p}\,\Le(\gamma)^\frac{p-3}{p} \; \xrightarrow{p\rightarrow\infty} \; \frac{\Le(\gamma)}{\Delta[\gamma]},
\end{equation}
where this convergence establishes a connection between integral \name{Menger} curvature and ropelength. In addition we will have a short look on another variant $\E_p^\lambda := \kla{\M_p(\gamma)}^\frac{1}{p} + \lambda \Le(\gamma),\;\lambda>0$. We consider an example with a fixed $\lambda>0$. Furthermore, it would also be possible to choose $\lambda>0$ as a \name{Lagrange}-multiplier in order to keep the length of the knot fixed during the flow. 

In \cite{reiter_iso}, \name{Reiter} proved that all knots in a certain $C^1$-neighbourhood belong to the same knot class. Therefore, a (local) minimizer of some energy in a given knot class already is a local minimizer among \emph{all} knots. As we will see, this is an important step to prove that a local minimizer of $\M_p$ in a given knot class with prescribed length is indeed a stationary point of the modified integral \name{Menger} curvature $\E_p$. Up to now the very likely conjecture that also for integral \name{Menger} curvature the circle is the global minimizer (cf. \cite{circmin}) has \emph{not} been shown. Nevertheless, we present a proof that a circle is at least a stationary point of $\E_p$ and in particular of $\M_3$ (see \thref{thmcirclesp}).

We will now review what kind of numerical approaches have been considered for knot energies. In order to minimize the knot energy of a given knot, one can use a numerical gradient flow. To do so one has to approximate the knot in an appropriate finite dimensional subspace of functions. A common solution is to use polygons with sufficiently many points and to move these points into the desired direction. However, in the case of integral \name{Menger} curvature this will not work for $p\geq 3$, since this energy is infinite for piecewise linear functions (see \thref{rempwenergy}). We will quote a proof later on. For $p\in(0,3)$ in \cite{scholtes_fin} \name{S. Scholtes} proved that $\M_p$ is finite for piecewise linear functions. Since the first case is the one we are interested in, we need functions with higher regularity. We represent knots as trigonometric polynomials and will just move the corresponding \name{Fourier} coefficients in our flow.

As it is very difficult to calculate \name{M\"obius}~energy in most cases, \name{E. J. Rawdon} developed in his PhD thesis \cite{Raw98} a polygonal definition for the thickness of a knot. Impressive animations, visualising the evolution towards an ideal knot and drawing it with maximal thickness result from the work of \name{J. Cantarella}, \name{M. Piatek} and \name{Rawdon} in \cite{CPR05}. Very useful and comprehensive are the \texttt{knotplot} project by \name{R. Scharein} \cite{knotplot} and the \texttt{libbiarc} library by \name{M. Carlen} \cite{phd_carlen}. Remarkable in this context are also \cite{CKS02}, \cite{CFKSW06}, \cite{CFKS11} and \cite{CLR11}. Further examples of gradient flows for other energies may be found in \cite{CDIO04} and \cite{dipl}.

Choosing trigonometric polynomials for the representation of knots has the advantage that quite a few coefficients are sufficient to encode a large range of closed curves and we are able to choose an easy and at the same time very accurate numerical integration scheme for the triple integrals. Since this energy is highly non-linear and global, and since the support of the basis functions of our approximation space is not compact, we have to deal with long computation times. Nevertheless, it was possible to achieve significant speed-ups by performing some optimisations using trigonometric identities, summing up in a highly efficient way and using optimised data structures. The crucial issue of choosing an appropriate time step size is eased by implementing an adaptive choice of time step size. After all we get a very robust and very reliable algorithm.

This thesis is structured as described in the following.

\emph{Chapter~2} starts with a collection of meaningful facts about knot theory and calculus, in particular, we introduce some function spaces and discuss important details how to handle periodic functions in this context, especially for \name{H\"older} and \name{Sobolev-Slobodeckij} spaces. Then we define the circumradius as well as relative objects and give two examples why this function is discontinuous in $\R^3$. After that we define the thickness of a knot and present two examples of unknots with thickness equal to one. Next, we consider the integral \name{Menger} curvature and its properties and calculate the first variation of $\M_p$ and afterwards of $\E_p$. Moreover, we prove that for $C^2$-curves the integrand of the energy and its first variation are continuous. In the end we shortly consider another modification of the energy.

After that \emph{Chapter~3} starts with describing some details how the discretization is done. We discuss different alternatives for approximation spaces and numerical integration schemes, where possible improvements are pointed out. Then we introduce our approaches of an adaptive choice of time step size and of a redistribution algorithm for \name{Fourier} knots.

Finally, in \emph{Chapter~4} we consider a large collection of examples. Firstly we discuss how to get \name{Fourier} coefficients, needed as starting configurations for our flow. Afterwards we add some comments how the implementation has been done. Then we watch the evolution of some representatives of the first five prime knot classes with respect to the \name{Alexander-Briggs} notation. This notation categorises each knot class by its crossing number, i.e. the minimal number of crossings of the knot, when projected onto a plane. Additionally an index that simply counts the classes with the same crossing number in a specific tabular is added \cite{natclass}. The knots with crossing number equal to $0$, $3$ and $4$ only correspond to one knot class each, namely the \emph{unknot}, the \emph{trefoil} and the \emph{figure-eight knot}. There exist two types of knots with crossing number $5$, namely $5_1$ and $5_2$ in the \name{Alexander-Briggs} notation. An increasing crossing number means that the number of knot classes increases very fast. For instance there are $165$ knot classes with ten crossings but for $15$ crossings there are already $253293$ different classes \cite{knotbook}. The classes with up to five crossings are the five knot classes we are considering.

\newpage
\section*{Acknowledgements}
I would like to express my deepest gratitude to my advisor Prof. Dr. Heiko von der Mosel for giving me the opportunity to work on this highly interesting and exciting topic and for his preserving support. He was always able to encourage me by his enthusiastic way of proceeding. Due to his effort I was supported by the German Research Foundation (DFG) for most of the time. I want to express my gratitude to them as well. Moreover, I am very thankful to him for encouraging me to participate in the conference ``Knots and Links: From Form to Function'' at the Mathematical Research Center 'Ennio De Giorgi' in Pisa (Italy) July 2011 of the European Science Foundation (ESF), which was a sustainable inspiring experience. I would like to thank the organisers of the conference as well.

During a stay at the EPFL I very much enjoyed the hospitality of Prof. John H. Maddocks, Dr. Henryk Gerlach and Dr. Mathias Carlen, who in particular told me very much about their \texttt{libbiarc} project and how to visualise knots. Im am deeply indebted to Prof. Dr. Pawe\l{} Strzelecki for inviting me to Warsaw and also to him and his colleagues at the University of Warsaw, especially Prof. Maksymilian Dryja, Dr. Przemys\l{}aw Kiciak, S\l{}awomir Kolasi\'{n}ski and Prof. Henryk Wo\'{z}niakowski for their hospitality and for many fruitful discussions.

I want to thank Prof. Dr. Siegfried M\"uller and Dr. Karl-Heinz Brakhage for their readiness to help as well as Markus Bachmayr for supporting me in \lapack{}. I also want to thank all my colleagues at the Institut f\"ur Mathematik at the RWTH Aachen University for a very pleasant atmosphere.
Especially I want to thank for numerous fruitful discussions with Sebastian Scholtes, Martin Meurer and Patrick Overath. Sincere thanks to Heidi Bouj\'{e} and Doris Rihlmann, the secretaries of the institut, for solving all the bureaucratic problems. I am grateful for inspiring discussions with my former colleagues Dr. Philipp Reiter, Dr. Simon Blatt and Dr. Armin Schikorra not only during their time at the institut.

I am very thankful that Maria Nau and Patrick Overath proofread large parts of this thesis. I also want to deeply thank Maria, my brother and my parents for their moral support in times of need.

A special thank also goes to my former fellow students Christian Bagh, Thomas Bedbur, Michael Dahmen, Jan Jongen and Dr. Matthias Schlottbom for a great time as a student and for a joint project work that initially interested me in this topic.


\cleartooddpage[\thispagestyle{empty}]

\chapter{Integral \name{Menger} curvature for knots}
Before we define \emph{integral \name{Menger} curvature} and discuss related topics, we start with a collection of basics we need from calculus and knot theory.

\section{Essentials in knot theory and calculus}
\subsection*{Calculus} \label{calculus}
For a real number $a\in\R$ we denote its absolute value by $\abs{a}$. We also use this notation for vectors in $\R^3$, meaning the standard \name{Euclidean} norm, namely for $a\in\R^3$
\begin{equation*}
\abs{a} := \norm{a}_2 = \sqrt{a_1^2+a_2^2+a_3^2}.
\end{equation*}
The standard basis of $\R^3$ is denoted by $e_i$, $i=1,2,3$. We use the standard definition of the \emph{scalar vector product} in $\R^3$
\begin{equation} \label{defscalar}
\left( \begin{array}{c}
a_1\\
a_2\\
a_3
\end{array} \right) \cdot
\left( \begin{array}{c}
b_1\\
b_2\\
b_3
\end{array} \right) := a_1 b_1 + a_2 b_2 + a_3 b_3.
\end{equation}
Moreover, for $a,b\in\R^3$ we mention the essential connections to the \name{Euclidean} norm
\begin{equation} \label{scalarnorm}
\abs{a}^2 = a \cdot a
\end{equation}
and to the trigonometric functions
\begin{equation} \label{scalarcos}
a \cdot b = \abs{a} \abs{b} \cos( \sphericalangle(a,b) ),
\end{equation}
where $\sphericalangle(a,b)$ indicates the angle between the vectors $a$ and $b$. We define the standard \emph{wedge product} in three dimensional space, which is also known as \emph{cross product}, for $a, b \in \R^3$ by
\begin{equation} \label{defwedge}
\left( \begin{array}{c}
a_1\\
a_2\\
a_3
\end{array} \right) \wedge
\left( \begin{array}{c}
b_1\\
b_2\\
b_3
\end{array} \right) :=
\left( \begin{array}{c}
a_2 b_3 - a_3 b_2\\
a_3 b_1 - a_1 b_3\\
a_1 b_2 - a_2 b_1
\end{array} \right).
\end{equation}

Observe we have that $a\wedge a=0$. In addition this product is \emph{antisymmetric}, i.e. for all $a,b\in\R^3$
\begin{equation} \label{wedgeantisym}
a \wedge b \; + \; b \wedge a \; = \; 0.
\end{equation}
For $a, b, c, d \in \R^3$ we have (cf. \cite[1-4]{docarmo})
\begin{equation} \label{calScalarCross}
(a \wedge b) \cdot (c \wedge d) = (a \cdot c) (b \cdot d) - (a \cdot d) (b \cdot c).
\end{equation}
Consequently, we have for $a,b\in\R^3$
\begin{equation} \label{calabsWedge}
\abs{a \wedge b}^2 = \abs{a}^2\, \abs{b}^2 - (a \cdot b)^2,
\end{equation}
and the well known fact
\begin{equation} \label{wedgesin}
\abs{a \wedge b} = \abs{a}\abs{b}\, \abs{\sin( \sphericalangle(a,b) )}.
\end{equation}
Obviously $a\cdot b$ and $a\wedge b$ are linear in $a$ and $b$ respectively. Furthermore, for two differentiable functions $f:\R\rightarrow\R^3$ and $g:\R\rightarrow\R^3$ the product rule generalises for the scalar vector and the wedge product, more precisely
\begin{equation} \label{prodrules}
\begin{split}
\frac{d}{dx} \kla{f(x)\cdot g(x)} &= f'(x)\cdot g(x) + f(x)\cdot g'(x)\\
\frac{d}{dx} \kla{f(x)\wedge g(x)} &= f'(x)\wedge g(x) + f(x)\wedge g'(x).
\end{split}
\end{equation}
\begin{rem} \mylabel{threeintegrals}
Let $a>b$. For a \name{Lebesgue} integrable function $f: [a,b] \rightarrow \R^3$ we use the standard integral notation for integrating each scalar component of $f$:
\begin{equation*}
\int_a^b f(x) \, dx \quad := \quad \left(
\begin{array}{c}
\int_a^b f_1(x) \, dx\\
\int_a^b f_2(x) \, dx\\
\int_a^b f_3(x) \, dx
\end{array} \right).
\end{equation*}
\end{rem}
We use the elementary fact that for $p\geq 1$ the function $\R\rightarrow\R, x\mapsto\abs{x}^p$ is convex. From $\abs{\frac{1}{2}a+\frac{1}{2}b}^p\leq\frac{1}{2}\abs{a}^p+\frac{1}{2}\abs{b}^p$ we gain the useful estimate
\begin{equation} \label{estabssum}
\abs{a+b}^p\leq 2^{p-1} \kla{\abs{a}^p+\abs{b}^p}.
\end{equation}
We define the characteristic function for sets in the ordinary way
\begin{defn} \mylabel{defcharfct}
Let $A\subset\R$ be a set. Then we define the \emph{characteristic function of $A$} $\chi_A:\R\rightarrow\R$ as follows
\begin{equation*}
\chi_A(x) :=
\begin{cases}
1, &x\in A\\
0, &x\notin A.
\end{cases}
\end{equation*}
\end{defn}
\begin{lem} \mylabel{lemintch}
Let $a,b,c\in\R$. For all $v,w\in\R$ we have
\begin{align*}
\chi_{[a+b,c]}(v)\,\chi_{[a,v-b]}(w)\quad &= \quad\chi_{[w+b,c]}(v)\,\chi_{[a,c-b]}(w),\\
\chi_{[b-c,a]}(v)\,\chi_{[b-v,c]}(w)\quad &= \quad\chi_{[b-w,a]}(v)\,\chi_{[b-a,c]}(w).
\end{align*}
\end{lem}
\begin{proof}
We prove that the left-hand side is equal to $1$ if and only if the right-hand side is equal to $1$. For the first equation we assume $v\in[a+b,c]$ and $w\in[a,v-b]$, then we have
\begin{equation*}
a\leq w\leq v-b\leq c-b \qquad \text{and} \qquad a+b\leq w+b\leq v\leq c,
\end{equation*}
and hence $v\in[w+b,c]$ and $w\in[a,c-b]$. Now assume that $w\in[a,c-b]$ and $v\in[w+b,c]$, and we get
\begin{equation*}
a+b\leq w+b\leq v\leq c \qquad \text{and} \qquad a\leq w\leq v-b\leq c-b,
\end{equation*}
and hence $v\in[a+b,c]$ and $w\in[a,v-b]$. For the second equation assume that $v\in[b-c,a]$ and $w\in[b-v,c]$, and we get
\begin{equation*}
b-a\leq b-v\leq w\leq c \qquad \text{and} \qquad b-c\leq b-w \leq v\leq a,
\end{equation*}
and hence $v\in[b-w,a]$ and $w\in[b-a,c]$. Now we assume $w\in[b-a,c]$ and $v\in[b-w,a]$, then we have
\begin{equation*}
b-c\leq b-w\leq v\leq a \qquad \text{and} \qquad b-a\leq b-v\leq w\leq c,
\end{equation*}
and hence $v\in[b-c,a]$ and $w\in[b-v,c]$.
\end{proof}
\begin{rem} \mylabel{remproofbw}
Let $f$ be a function, which is \name{Lebesgue} measurable on $\R^n$. To prove \name{Lebesgue} integrability of $f$ we can use \name{Fubini}'s theorem \cite[8.8 Theorem and the notes afterwards]{rudin} by showing that one so-called ``iterated integral'' of the absolute value of $f$ is finite. ``Iterated integral'' means that we compute $n$ one-dimensional \name{Lebesgue} integrals step-by-step. \name{Fubini}'s theorem \cite{rudin} also tells us that we are allowed to interchange the order of integration if $f$ is \emph{non-negative} or if $f$ is \name{Lebesgue} integrable on $\R^n$ or if an ``iterated integral'' of the absolute value of $f$ is finite.\\
Assume $f$ is non-negative and that we want to estimate the \name{Lebesgue} integral of $f$ from above. If we need the integrability of $f$ for one step in between and we end up with an expression that is finite, then we are allowed to write it down this way, since we could apply the whole chain of estimates in reversed order.
\end{rem}

In the following part we define some function spaces that we need later on. We denote the classic spaces of continuously differentiable functions by $C^k$ for $k\in\N_0$ and the space of the functions whose $p$-th power is \name{Lebesgue} integrable by $L^p$, where $1<p<\infty$. Recall that the space $L^2$ is a \name{H\"older} space with the following scalar product for $f,g\in L^2$
\begin{equation*}
\left\langle f,g \right\rangle_{L^2} := \int f(x)\cdot g(x) \, dx. 
\end{equation*}
\begin{rem} \mylabel{remperiodic}
Let $f:[0,2\pi) \rightarrow \R^3$ be a function. Then $f$ can be expanded to a $2\pi$-periodic function $\tilde{f}:\R \rightarrow \R^3$, by defining $\tilde{f}(x+2\pi k) := f(x)$ for all $x \in [0,2\pi)$ and all $k \in \Z$. We express this by writing $f: \R/2\pi\Z \rightarrow \R^3$, because $\R/2\pi\Z \cong \Sp^1 \cong [0,2\pi)$. If there is no cause for confusion we do not distinguish between $f$ and $\tilde{f}$.

\begin{tabularx}{\textwidth}{llX}
$f \in L^p(\R/2\pi\Z,\R^3)$ & :$\Leftrightarrow \tilde{f} \in L_\text{loc}^p(\R,\R^3)$ & The function $\abs{\tilde{f}}^p$ is \name{Lebesgue} integrable over every compact set, in particular over every finite interval in $\R$.\\
$f \in C^k(\R/2\pi\Z,\R^3)$ & :$\Leftrightarrow \tilde{f} \in C^k(\R,\R^3)$ & The function $\tilde{f}$ is $k$-times continuously differentiable (with $k \in \N_0 \cup \menge{\infty}$) and therefore closed, i.e. $f(x)\rightarrow f(0)$ if $x\rightarrow 2\pi$.
\end{tabularx}

Obviously this can be transferred to a function $f: \R/\ell\Z \rightarrow \R^3$, with $\ell>0$.
\end{rem}
\begin{prop} \mylabel{propshift}
Let $f: \R/2\pi\Z \rightarrow \R^3$ be a \name{Lebesgue} integrable function. For $a \in \R$ we have
\begin{equation*}
\int_0^{2\pi} f(s) \, ds = \int_{a-\pi}^{a+\pi} f(s) \, ds.
\end{equation*}
\end{prop}
\begin{proof}
As we mentioned in \thref{remperiodic} the function $f$ is \name{Lebesgue} integrable over every finite interval in $\R$. We use the substitution $\tilde{s}:=s-2\pi$ and get
\begin{align*}
\int_0^{2\pi} f(s) \, ds &= \int_0^{a+\pi} f(s) \, ds + \int_{a+\pi}^{2\pi} f(s) \, ds\\
&= \int_0^{a+\pi} f(s) \, ds + \int_{a-\pi}^0 f(\tilde{s}+2\pi) \, d\tilde{s}\\
&= \int_{a-\pi}^{a+\pi} f(s) \, ds,
\end{align*}
because f is $2\pi$-periodic.
\end{proof}
We denote the class of continuous functions by $C^0(\R/2\pi\Z,\R^3)$ with the norm
\begin{equation*}
\norm{\gamma}_{C^0(\R/2\pi\Z,\R^3)} \quad := \quad \sup_{x\in[0,2\pi]} \abs{\gamma(x)} \; = \; \sup_{x\in\R}\;\abs{\gamma(x)},
\end{equation*}
which is finite since $\tilde{\gamma}$ is in particular continuous on the compact interval $[0,2\pi]$. The class $C^\infty(\R/2\pi\Z,\R^3)$ stands for those functions, which are continuously differentiable to any order. Now we come to an extension of these classic spaces, the so-called \name{H\"older} spaces, cf. \cite[5.1. \name{H\"older} spaces]{evans}.
\begin{defn} \mylabel{defhoelder}
Let $k\in\N$ and $\alpha\in (0,1]$. A function $\gamma: \R/2\pi\Z \rightarrow \R^3$ belongs to the \name{H\"older} class $C^{k,\alpha}(\R/2\pi\Z,\R^3)$, if $\gamma \in C^k(\R/2\pi\Z,\R^3)$ and if the following semi norm is finite
\begin{equation*}
[\gamma^{(k)}]_{C^{0,\alpha}(\R/2\pi\Z,\R^3)} \quad := \quad \sup_{\stackrel{s,t\in[0,2\pi]}{s\neq t}} \frac{\abs{\gamma^{(k)}(s)-\gamma^{(k)}(t)}}{\abs{s-t}^\alpha} \quad < \quad \infty.
\end{equation*}
These functions form a \name{Banach} space (\cite[5.1, Theorem 1]{evans}) with the norm
\begin{equation*}
\norm{\gamma}_{C^{k,\alpha}(\R/2\pi\Z,\R^3)} \quad := \quad \sum_{i=0}^k \;\lVert\gamma^{(i)}\rVert_{C^0(\R/2\pi\Z,\R^3)} + [\gamma^{(k)}]_{C^{0,\alpha}(\R/2\pi\Z,\R^3)}.
\end{equation*}
A function in $C^{0,1}$ is also called \emph{\name{Lipschitz} continuous}.
\end{defn}
Observe that it is sufficient to define this semi norm with respect to an arbitrary interval of length $2\pi$. Let $a\in\R$, we consider the interval $[a,a+2\pi]$. Then we find an upper bound for the supremum with respect to the interval $[a,a+4\pi]$, which is useful since the semi norm on $[a,a+4\pi]$ is again an upper bound for the corresponding semi norm on every interval of length $2\pi$, which we will see afterwards. We start to show the first estimate
\begin{align*}
\sup_{\stackrel{s,t\in[a,a+4\pi]}{s\neq t}} \frac{\abs{\gamma(s)-\gamma(t)}}{\abs{s-t}^\alpha} &\leq \sup_{\stackrel{s,t\in[a,a+2\pi]}{s\neq t}} \frac{\abs{\gamma(s)-\gamma(t)}}{\abs{s-t}^\alpha} + \sup_{\stackrel{s,t\in[a+2\pi,a+4\pi]}{s\neq t}} \frac{\abs{\gamma(s)-\gamma(t)}}{\abs{s-t}^\alpha}\\
& \qquad + 2 \sup_{\stackrel{s\in[a+2\pi,a+4\pi]}{\stackrel{t\in[a,a+2\pi]}{s\neq t}}} \frac{\abs{\gamma(s)-\gamma(t)}}{\abs{s-t}^\alpha}.
\end{align*}
We have
\begin{align*}
\frac{\abs{\gamma(s)-\gamma(t)}}{\abs{s-t}^\alpha} &\quad=& &\frac{\abs{\gamma(s-2\pi)-\gamma(t-2\pi)}}{\abs{(s-2\pi)-(t-2\pi)}^\alpha}\\ \Rightarrow \qquad \sup_{\stackrel{s,t\in[a+2\pi,a+4\pi]}{s\neq t}} \frac{\abs{\gamma(s)-\gamma(t)}}{\abs{s-t}^\alpha} &\quad=& \sup_{\stackrel{s,t\in[a,a+2\pi]}{s\neq t}} &\frac{\abs{\gamma(s)-\gamma(t)}}{\abs{s-t}^\alpha}.
\end{align*}
Let $s\in[a+2\pi,a+4\pi]$ and $t\in[a,a+2\pi]$ with $s\neq t$. Define $\tilde{s}=s-2\pi\in [a,a+2\pi]$ and we have $\max\menge{\abs{\tilde{s}-a},\abs{(a+2\pi)-t}}\leq \abs{\tilde{s}-a}+\abs{(a+2\pi)-t} = (\tilde{s}-a)+(a+2\pi-t) = \abs{\tilde{s}+2\pi-t}$. Now we get
\begin{align*}
\frac{{\abs{\gamma^{(k)}(s)-\gamma^{(k)}(t)}}}{\abs{s-t}^\alpha} &= \frac{{\abs{\gamma^{(k)}(\tilde{s})-\gamma^{(k)}(t)}}}{\abs{\tilde{s}+2\pi-t}^\alpha} \leq \frac{{\abs{\gamma^{(k)}(\tilde{s})-\gamma^{(k)}(a)}}}{\abs{\tilde{s}-a}^\alpha} + \frac{{\abs{\gamma^{(k)}(a+2\pi)-\gamma^{(k)}(t)}}}{\abs{(a+2\pi)-t}^\alpha}\\
&\leq 2\;\sup_{\stackrel{s,t\in[a,a+2\pi]}{s\neq t}} \frac{\abs{\gamma^{(k)}(s)-\gamma^{(k)}(t)}}{\abs{s-t}^\alpha},
\end{align*}
and finally,
\begin{equation*}
\sup_{\stackrel{s,t\in[a,a+4\pi]}{s\neq t}} \frac{\abs{\gamma(s)-\gamma(t)}}{\abs{s-t}^\alpha} \leq 6\;\sup_{\stackrel{s,t\in[a,a+2\pi]}{s\neq t}} \frac{\abs{\gamma(s)-\gamma(t)}}{\abs{s-t}^\alpha}.
\end{equation*}
Now we want to show the second estimate mentioned above, more precisely for $b\in\R$
\begin{equation*}
\sup_{\stackrel{s,t\in[b,b+2\pi]}{s\neq t}} \frac{\abs{\gamma^{(k)}(s)-\gamma^{(k)}(t)}}{\abs{s-t}^\alpha} \quad \leq \quad  [\gamma^{(k)}]_{C^{0,\alpha}(\R/2\pi\Z,\R^3)} \quad < \quad \infty.
\end{equation*}
Since there exists a $l\in\Z$ such that $b+2\pi l\in [a,a+2\pi) \subset [a,a+4\pi]$ we conclude that $(b+2\pi)+2\pi l\in [a,a+4\pi]$ as well and
\begin{align*}
\frac{\abs{\gamma^{(k)}(s)-\gamma^{(k)}(t)}}{\abs{s-t}^\alpha} &\quad=& &\frac{\abs{\gamma^{(k)}(s+2\pi l)-\gamma^{(k)}(t+2\pi l)}}{\abs{(s+2\pi l)-(t+2\pi l)}^\alpha}\\
\Rightarrow \qquad \sup_{\stackrel{s,t\in[b,b+2\pi]}{s\neq t}} \frac{\abs{\gamma^{(k)}(s)-\gamma^{(k)}(t)}}{\abs{s-t}^\alpha} &\quad\leq& \sup_{\stackrel{\tilde{s},\tilde{t}\in[a,a+4\pi]}{\tilde{s}\neq \tilde{t}}} &\frac{\abs{\gamma^{(k)}(\tilde{s})-\gamma^{(k)}(\tilde{t})}}{\abs{\tilde{s}-\tilde{t}}^\alpha}.
\end{align*}

In the following we introduce the very useful \name{Sobolev} function spaces, which are quite common for instance in the calculus of variations, cf. \cite[5.2.2. Definition of \name{Sobolev} spaces]{evans}.
\begin{defn} \mylabel{defsobolev}
Let $k\in\N$ and $p\in(1,\infty)$. A function $\gamma: \R/2\pi\Z \rightarrow \R^3$ belongs to the \name{Sobolev} class $W^{k,p}(\R/2\pi\Z,\R^3)$, if $\gamma \in L^p(\R/2\pi\Z,\R^3)$ and if for all $j=1,\dots,k$ there exists a function $\gamma_j \in L^p(\R/2\pi\Z,\R^3)$ such that
\begin{equation*}
\int_0^{2\pi} \gamma(x) \cdot \frac{d^j}{dx^j}\phi(x) \, dx = (-1)^j \int_0^{2\pi} \gamma_j(x) \cdot \phi(x) \, dx \qquad \text{for all } \phi \in C^\infty(\R/2\pi\Z,\R^3).
\end{equation*}
The functions $\gamma^{(j)} := \gamma_j$ for $j=1,\dots,k$ are called \emph{weak derivatives} of $\gamma$.
\end{defn}
The \name{Sobolev} space is a \name{Banach} space \cite[5.2.3 Theorem 2]{evans} with the norm
\begin{equation*}
\norm{\gamma}_{W^{k,p}(\R/2\pi\Z,\R^3)} \quad := \quad \sum_{i=0}^k \kla{\int_0^{2\pi} \babs{\gamma^{(i)}(x)}^p \, dx}^{\frac{1}{p}}.
\end{equation*}
Observe that due to \thref{propshift} we could equivalently integrate over an arbitrary interval of length $2\pi$.

In addition, we need the so-called \emph{fractional \name{Sobolev} spaces} or \name{Sobolev-Slobodeckij} spaces, where we have a fractional order of differentiation. However, we will not define them directly for $2\pi$-periodic functions.
\begin{defn} \mylabel{deffracsobolev}
Let $a<b$, $p\in(1,\infty)$, $k\in\N$ and $s\in(0,1)$. A function $\gamma:[a,b] \rightarrow \R^3$ belongs to the \name{Sobolev-Slobodeckij} class $W^{k+s,p}([a,b],\R^3)$, if $\gamma\in W^{k,p}([a,b],\R^3)$ and if the following semi norm is finite
\begin{equation*}
[\gamma^{(k)}]_{W^{s,p}([a,b],\R^3)} \quad := \quad  \kla{\int_a^b \int_a^b \;\frac{\abs{\gamma^{(k)}(x)-\gamma^{(k)}(y)}^p}{\abs{x-y}^{1+sp}}\,dx\,dy}^\frac{1}{p} \quad < \quad \infty.
\end{equation*}
\end{defn}
Due to \cite[(2.9)]{hitchhiker} these classes are as well \name{Banach} spaces with the norm
\begin{equation*}
\norm{\gamma}_{W^{k+s,p}([a,b],\R^3)} \quad := \quad \norm{\gamma}_{W^{k,p}([a,b],\R^3)} + [\gamma^{(k)}]_{W^{s,p}([a,b],\R^3)}.
\end{equation*}
For $k\in\Z$ we have
\begin{align*}
\int_{a-\pi}^{a+\pi} \int_{a-\pi}^{a+\pi} &\frac{\abs{\gamma^{(k)}(x)-\gamma^{(k)}(y)}^p}{\abs{x-y}^{1+sp}}\,dx\,dy\\
&= \int_{a+(2k-1)\pi}^{a+(2k+1)\pi} \int_{a-\pi}^{a+\pi} \;\frac{\abs{\gamma^{(k)}(x)-\gamma^{(k)}(\tilde{y}-2\pi k)}^p}{\abs{x-\tilde{y}+2\pi k}^{1+sp}}\,dx\,d\tilde{y}\\
&= \int_{a+(2k-1)\pi}^{a+(2k+1)\pi} \int_{a+(2k-1)\pi}^{a+(2k+1)\pi} \;\frac{\abs{\gamma^{(k)}(\tilde{x}-2\pi k)-\gamma^{(k)}(\tilde{y}-2\pi k)}^p}{\abs{\tilde{x}-\tilde{y}}^{1+sp}}\,d\tilde{x}\,d\tilde{y}\\
&= \int_{a+(2k-1)\pi}^{a+(2k+1)\pi} \int_{a+(2k-1)\pi}^{a+(2k+1)\pi} \;\frac{\abs{\gamma^{(k)}(x)-\gamma^{(k)}(y)}^p}{\abs{x-y}^{1+sp}}\,dx\,dy,
\end{align*}
where we substituted $\tilde{y}:=y+2\pi k$ as well as $\tilde{x}:=x+2\pi k$ and used the periodicity of $\gamma$. Observe that these computations are only possible since we translate the domain by a multiple of $2\pi$.
\begin{figure}[H]
\begin{center}
\includegraphics{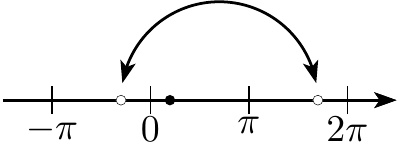}
\end{center}
\caption{A consequence of the $2\pi$-periodic setting}
\label{figperset}
\end{figure}
In Figure~\ref{figperset} we see an example of the problem. The two white points represent the same point on the curve, since their distance is $2\pi$. If we integrate over the interval $[0,2\pi]$ the distance of the white and the black point is quite large, however, if we integrate over the interval $[-\pi,\pi]$ it is very short. In particular we could not exclude the possibility that there exists a function $\gamma$ such that the integral is finite for $[a-\pi,a+\pi]$ but for $[a,a+2\pi]$ it is not. Therefore, in contrary to the case of the \name{Sobolev} norm, it seems to be insufficient to define the semi norm for one fixed interval $[a-\pi,a+\pi]$ in order to guarantee the finiteness of the semi norm for every interval. At first we consider the following definition for periodic functions $\gamma:\R/2\pi\Z \rightarrow \R^3$, where we use another semi norm, cf. for instance \cite{blatt_menger}.
\begin{defn} \mylabel{defperfracsobolev}
Let $p\in(1,\infty)$, $k\in\N$ and $s\in(0,1)$. A periodic function $\gamma:\R/2\pi\Z \rightarrow \R^3$ is in the \name{Sobolev-Slobodeckij} class $W^{k+s,p}(\R/2\pi\Z,\R^3)$, if $\gamma\in W^{k,p}(\R/2\pi\Z,\R^3)$ and the following semi norm is finite
\begin{equation*}
[\gamma^{(k)}]_{W^{s,p}(\R/2\pi\Z,\R^3)} \quad := \quad  \kla{\int_0^{2\pi} \int_{-\pi}^{\pi} \;\frac{\abs{\gamma^{(k)}(x+w)-\gamma^{(k)}(x)}^p}{\abs{w}^{1+sp}}\,dw\,dx}^\frac{1}{p} \quad < \quad \infty.
\end{equation*}
\end{defn}
Then the norm is
\begin{equation*}
\norm{\gamma}_{W^{k+s,p}(\R/2\pi\Z,\R^3)} \quad := \quad \norm{\gamma}_{W^{k,p}(\R/2\pi\Z,\R^3)} + [\gamma^{(k)}]_{W^{s,p}(\R/2\pi\Z,\R^3)}.
\end{equation*}
Observe, that this norm is the more natural one for closed curves, since the integrand is $2\pi$-periodic with respect to $x$. Moreover, $w$ measures the distance of the points $x$ and $x+w$ on the unit circle, which is also a natural domain to define a curve. We have seen that we need to consider more than one interval of length $2\pi$. However, we only have to use one additional interval, namely $[a,a+2\pi]$. Hence, we could use for instance $[-\pi,\pi]$ and $[0,2\pi]$. We will show this in two steps. At first we estimate the semi norm for periodic functions by a sum of semi norms on the intervals $[a-\pi,a+\pi]$ and $[a,a+2\pi]$.
\begin{align*}
& \int_0^{2\pi} \int_{-\pi}^{\pi} \frac{\abs{\gamma^{(k)}(x+w)-\gamma^{(k)}(x)}^p}{\abs{w}^{1+sp}}\,dw\,dx\\
&= \int_a^{a+2\pi} \int_{-\pi}^{0} \;\frac{\abs{\gamma^{(k)}(x+w)-\gamma^{(k)}(x)}^p}{\abs{w}^{1+sp}}\,dw\,dx + \int_{a-\pi}^{a+\pi} \int_{0}^{\pi} (\dots) \,dw\,dx\\
&= \int_{a+\pi}^{a+2\pi} \int_{-\pi}^{0} \;\frac{\abs{\gamma^{(k)}(x+w)-\gamma^{(k)}(x)}^p}{\abs{w}^{1+sp}}\,dw\,dx + \int_{a}^{a+\pi} \int_{-\pi}^{0} (\dots) \,dw\,dx\\
& \qquad \qquad + \int_{a-\pi}^{a} \int_{0}^{\pi} (\dots) \,dw\,dx + \int_{a}^{a+\pi} \int_{0}^{\pi} (\dots) \,dw\,dx\\
&= \int_{a+\pi}^{a+2\pi} \int_{x-\pi}^{x} \;\frac{\abs{\gamma^{(k)}(y)-\gamma^{(k)}(x)}^p}{\abs{y-x}^{1+sp}}\,dy\,dx + \int_{a}^{a+\pi} \int_{x-\pi}^{x} (\dots) \,dy\,dx\\
& \qquad \qquad + \int_{a-\pi}^{a} \int_{x}^{x+\pi} (\dots) \,dy\,dx + \int_{a}^{a+\pi} \int_{x}^{x+\pi} (\dots) \,dy\,dx\\
&\leq \int_{a+\pi}^{a+2\pi} \int_{a}^{a+2\pi} \;\frac{\abs{\gamma^{(k)}(y)-\gamma^{(k)}(x)}^p}{\abs{y-x}^{1+sp}}\,dy\,dx + \int_{a}^{a+\pi} \int_{a-\pi}^{a+\pi} (\dots) \,dy\,dx\\
&\qquad \qquad + \int_{a-\pi}^{a} \int_{a-\pi}^{a+\pi} (\dots) \,dy\,dx + \int_{a}^{a+\pi} \int_{a}^{a+2\pi} (\dots) \,dy\,dx\\
&\leq 2\,\int_{a}^{a+2\pi} \int_{a}^{a+2\pi} \;\frac{\abs{\gamma^{(k)}(y)-\gamma^{(k)}(x)}^p}{\abs{y-x}^{1+sp}}\,dy\,dx + 2\,\int_{a-\pi}^{a+\pi} \int_{a-\pi}^{a+\pi} (\dots) \,dy\,dx,
\end{align*}
where we used \thref{propshift} and substituted $w$ by $y:=w+x$. On the other hand we have for every $b\in\R$
\begin{align*}
\int_{b-\pi}^{b+\pi} \int_{b-\pi}^{b+\pi} \frac{\abs{\gamma^{(k)}(x)-\gamma^{(k)}(y)}^p}{\abs{x-y}^{1+sp}}\,dy\,dx &= \int_{b-\pi}^{b+\pi} \int_{b-\pi-x}^{b+\pi-x} \frac{\abs{\gamma^{(k)}(x+w)-\gamma^{(k)}(x)}^p}{\abs{w}^{1+sp}}\,dw\,dx\\
&\leq \int_{b-\pi}^{b+\pi} \int_{-2\pi}^{2\pi} \frac{\abs{\gamma^{(k)}(x+w)-\gamma^{(k)}(x)}^p}{\abs{w}^{1+sp}}\,dw\,dx\\
&= \int_{0}^{2\pi} \int_{-2\pi}^{2\pi} \frac{\abs{\gamma^{(k)}(x+w)-\gamma^{(k)}(x)}^p}{\abs{w}^{1+sp}}\,dw\,dx,
\end{align*}
where we substituted $y$ by $w:=y-x$. We continue with
\begin{align*}
&= \int_{0}^{2\pi} \int_{-\pi}^{\pi} \frac{\abs{\gamma^{(k)}(x+w)-\gamma^{(k)}(x)}^p}{\abs{w}^{1+sp}}\,dw\,dx + \int_{0}^{2\pi} \int_{\pi}^{2\pi} (\dots)\,dw\,dx + \int_{0}^{2\pi} \int_{-2\pi}^{-\pi} (\dots)\,dw\,dx\\
&\leq \int_{0}^{2\pi} \int_{-\pi}^{\pi} \frac{\abs{\gamma^{(k)}(x+w)-\gamma^{(k)}(x)}^p}{\abs{w}^{1+sp}}\,dw\,dx + 2\,C\,\norm{\gamma^{(k)}}_{L^p(\R/2\pi\Z,\R^3)}^p,
\end{align*}
since
\begin{align*}
\int_{0}^{2\pi} \int_{-2\pi}^{-\pi} \frac{\abs{\gamma^{(k)}(x+w)-\gamma^{(k)}(x)}^p}{\abs{w}^{1+sp}}\,dw\,dx &= \int_{0}^{2\pi} \int_{\pi}^{2\pi} \frac{\abs{\gamma^{(k)}(x-\tilde{w})-\gamma^{(k)}(x)}^p}{\abs{\tilde{w}}^{1+sp}}\,d\tilde{w}\,dx\\
&= \int_{\pi}^{2\pi} \int_{-w}^{2\pi-w} \frac{\abs{\gamma^{(k)}(\tilde{x})-\gamma^{(k)}(\tilde{x}+\tilde{w})}^p}{\abs{\tilde{w}}^{1+sp}}\,d\tilde{x}\,d\tilde{w}\\
&= \int_{0}^{2\pi} \int_{\pi}^{2\pi} \frac{\abs{\gamma^{(k)}(x+w)-\gamma^{(k)}(x)}^p}{\abs{w}^{1+sp}}\,dw\,dx\\
&\leq 3\pi \frac{2^{p-1}}{\pi^{1+sp}} \norm{\gamma^{(k)}}_{L^p(\R/2\pi\Z,\R^3)}^p := C\,\norm{\gamma^{(k)}}_{L^p(\R/2\pi\Z,\R^3)}^p,
\end{align*}
where we substituted $\tilde{w}:=-w$ as well as $\tilde{x}:=x-\tilde{w}$. Consequently, a periodic function $\gamma\in W^{k+s,p}(\R/2\pi\Z,\R^3)$ is always also a function in $W^{k+s,p}([a-\pi,a+\pi],\R^3)$ for all $a\in\R$ and we can apply the standard theorems for fractional \name{Sobolev} spaces.

A detailed introduction to these functions spaces can be found in \cite{evans} and \cite{adams}, respectively in \cite{RS96} and \cite{triebel}.
\begin{rem} \mylabel{remembedd}
Later on we need the following (continuous) embeddings
\begin{enumabc}
\item For $p\in(3,\infty)$ we have
\begin{equation*}
W^{2-\frac{2}{p},p}(\R/2\pi\Z,\R^3) \quad\subset\quad C^{1,1-\frac{3}{p}}(\R/2\pi\Z,\R^3).
\end{equation*}
Observe that for the standard \name{Sobolev} spaces we ``only'' have these embeddings (see \cite[8.13 \textlangle 1\textrangle]{alt}) for $p\in(1,\infty)$
\begin{equation*}
W^{1,p}(\R/2\pi\Z,\R^3) \subset C^{0,1-\frac{1}{p}}(\R/2\pi\Z,\R^3) \quad \text{and} \quad W^{2,p}(\R/2\pi\Z,\R^3) \subset C^{1,1-\frac{1}{p}}(\R/2\pi\Z,\R^3).
\end{equation*}
\item For $p\in[2,\infty)$ and $1-\frac{2}{p}<\alpha\leq 1$ we have
\begin{equation*}
C^{1,\alpha}(\R/2\pi\Z,\R^3) \quad\subset\quad W^{2-\frac{2}{p},p}(\R/2\pi\Z,\R^3).
\end{equation*}
\end{enumabc}
\end{rem}
\begin{proof}
\begin{enumabc}
\item This embedding is provided by \cite[Theorem 8.2]{hitchhiker}.
\item Let $p\in[2,\infty)$, $\alpha\in (1-\frac{2}{p},1]$ and $\gamma\in C^{1,\alpha}(\R/2\pi\Z,\R^3)$. At first we note that $C^1(\R/2\pi\Z,\R^3)\subset W^{1,p}(\R/2\pi\Z,\R^3)$, see \cite[1.25]{alt}. In the case $p=2$ we are done, because $2-\frac{2}{p}=1$ is an integer. Moreover, due to $1+(1-\frac{2}{p})\,p=p-1$, we have that the semi norm $[\gamma']^p_{W^{1-\frac{2}{p},p}(\R/2\pi\Z,\R^3)}$ is equal to
\begin{align*}
\int_0^{2\pi} \int_{-\pi}^\pi \frac{\babs{\gamma'(y+w) - \gamma'(y)}^p}{\abs{w}^{p-1}} \, dw \, dy &\leq \int_0^{2\pi} \int_{-\pi}^\pi [\gamma']^p_{C^{0,\alpha}(\R/2\pi\Z,\R^3)} \frac{\abs{w}^{\alpha p}}{\abs{w}^{p-1}} \, dw \, dy\\
&\leq 2\pi\,\norm{\gamma}^p_{C^{1,\alpha}(\R/2\pi\Z,\R^3)}\; \int_{-\pi}^\pi \abs{w}^{-(p-1-\alpha p)} \, dw,
\end{align*}
where we substituted $x$ by $w:=x-y$. The integral on the right-hand side is finite if $(p-1-\alpha p)<1\;\Leftrightarrow\;1-\frac{2}{p}<\alpha$. Observe that $\int_{-\pi}^\pi \abs{w}^{-1} \, dw$ is infinite. \qedhere
\end{enumabc}
\end{proof}

Now we introduce the basic concept of the \emph{calculus of variations}, cf. \cite[p. 9]{GH96}.
\begin{defn} \mylabel{defFVar}
Let $\Omega\subset\R$ and $E$ be a functional defined on a subset $V$ of a function space $X \subset C^0(\Omega,\R^3)$. Let further be $\gamma_0$ some point in $V$ and $\Phi$ be a vector in $X$. Assume that $\menge{\gamma \,\arrowvert\, \gamma=\gamma_0+\tau\Phi, \tau\in[-\tau_0,\tau_0]}$ is contained in $V$ for some $\tau_0>0$.  We want to investigate the behaviour of $E$ while perturbing $\gamma_0$ slightly in direction $\Phi$. Therefore, we define the \emph{first Variation of $E$ at $\gamma_0$ in direction $\Phi$} as
\begin{equation*}
\delta E(\gamma_0,\Phi) := \frac{d}{d\tau} \eins{\tau=0} E(\gamma_0+\tau\Phi) = \lim_{\tau\rightarrow 0} \frac{E(\gamma_0+\tau\Phi)-E(\gamma_0)}{\tau},
\end{equation*}
if the limit exists. Observe that whenever $\gamma_\text{sp}$ is a stationary point (e.g. a local minimum) of $E$ in $V$ we have that $\delta E(\gamma_\text{sp},\Phi)=0$ for every admissible $\Phi\in X$.
\end{defn}
A quite simple but typical situation would be an interval $\Omega := [a,b]$ for $a<b$ and the function space $X:=C^2([a,b])$. One could be interested in minimizing a functional $E$ in this space $X$ with fixed boundary conditions, which could be reflected in the definition of $V$. In this case we would not be allowed to perturb the boundary and therefore, the choice of $\Phi$ would be restricted to functions of $C_0^2([a,b])$. In our situation we do not need boundary conditions, hence we handle with periodic functions. Taking for example $X=W^{2-\frac{2}{p}}(\R/2\pi\Z,\R^3)$
we are able to choose $\Phi$ to be any function in $X$. This will be convenient for the discretization, since we can use the \name{Ritz-Galerkin} method, see Chapter~\ref{discspace}.

\begin{prop} \mylabel{propmeas}
Let $n \in \N$ and $f, g: \R^n \rightarrow \R$ be \name{Lebesgue} measurable functions on $\R^n$ with $g\neq 0$ almost everywhere, i.e. $\Ns:= \menge{x \in \R^n \arrowvert g(x)=0}$ is a \name{Lebesgue} null set in $\R^n$. Then
\begin{equation*}
h: \R^n \setminus \Ns \rightarrow \R,\; x \mapsto \dfrac{f(x)}{g(x)}
\end{equation*}
is an \emph{almost everywhere defined \name{Lebesgue} measurable} function (see \cite[consequence of 13.6 Theorem]{bauer}) and so can be extended to
\begin{equation*}
\tilde{h}: \R^n \rightarrow \R, \; x \mapsto
\begin{cases}
h(x), & x\in \R^n \setminus \Ns\\
0, & x\in \Ns
\end{cases},
\end{equation*}
which is a \name{Lebesgue} measurable function. If $\tilde{h}$ is \name{Lebesgue} integrable on $\R^n$ then every \name{Lebesgue} measurable extension of $h$ is \name{Lebesgue} integrable on $\R^n$ as well and therefore, we say that $h$ itself is \name{Lebesgue} integrable on $\R^n$, see \cite[13.7 Definition]{bauer}.
\end{prop}
\begin{exmp}
Observe that in general $h$ is even not defined in $\overline{\R} = \R \cup \menge{-\infty,\infty}$, for instance if we choose for $n:=1$
\begin{equation*}
f(x):=1, \quad g(x):=x \qquad \text{for all $x \in \R$}
\end{equation*}
then
\begin{equation*}
\lim\limits_{x\searrow 0} h(x)=\infty \qquad\lim_{x\nearrow 0} h(x)=-\infty
\end{equation*}
and we are not able to define a suitable value in $\overline{\R}$ for $h$ at $x=0$.
\end{exmp}
\begin{proof}[Proof of \thref{propmeas}]
Since $\Ns$ is a \name{Lebesgue} null set on $\R^n$, $\Ns$ and $\R^n\setminus \Ns$ are \name{Lebesgue} measurable. The function $h$ is $\mathcal{B}^n\setminus \Ns$-measurable, due to \cite[1.1.2 Theorem 6(i)]{evans_gariepy}, since $g(x) \neq 0$ for all $x \in \R^n\setminus \Ns$, where $\mathcal{B}^n$ is the $n$-dimensional \name{Borel}-algebra. As mentioned in the first part of the proof of \cite[1.1.2 Theorem 6(i)]{evans_gariepy} it is sufficient to show that $\tilde{h}^{-1}( (-\infty,a) )$ is \name{Lebesgue} measurable for each $a\in\R$, which is the case as
\begin{equation*}
\tilde{h}^{-1}( (-\infty,a) ) =
\begin{cases}
h^{-1}( (-\infty,a) ), & a \leq 0\\
h^{-1}( (-\infty,a) ) \cup \Ns, & a > 0.
\end{cases}
\end{equation*}
Now we follow the arguments in the part about consequences of \cite[13.6 Theorem]{bauer}. Another \name{Lebesgue} measurable extension of $h$ differs from $\tilde{h}$ only on a null set. If $\tilde{h}$ is \name{Lebesgue} integrable on $\R^n$, the same is true for other extensions and furthermore, their integrals are the same, see \cite[13.4 Theorem]{bauer}.
\end{proof}
%

\subsection*{Knot theory} \label{knottheory}
\begin{defn} \mylabel{defsimple}
A closed curve $\gamma \in C^0(\R/2\pi\Z,\R^3)$ is called \emph{simple} if
\begin{equation*}
\gamma(s) \neq \gamma(t) \qquad \text{for all } s,t \in \R \text{ with } \abs{s-t}\notin 2\pi\Z.
\end{equation*}
In particular this means that $\gamma$ is injective on $[0,2\pi)$. From another point of view one can say that the image of the curve $\gamma$ is embedded into $\R^3$.
\end{defn}

We follow the definition of \cite[48. Arc Length]{lass}.
\begin{defn} \mylabel{defrectifiable}
Let $a<b$. A curve $\gamma:[a,b] \rightarrow \R^3$, not necessary closed, is called \emph{rectifiable} if the length of all inscribed polygons of the curve is bounded. More precisely, if there exists a fixed constant $C>0$ such that for all $n\in\N$
\begin{equation*}
L_n := \sup\menge{\sum_{i=1}^n \abs{\gamma(t_i)-\gamma(t_{i-1})} \,\Big\arrowvert\, a = t_0 < t_1 < \dots < t_n=b} < C.
\end{equation*}
For such a curve we define its \emph{length} as
\begin{equation*}
\Le(\gamma) = \sup_{n\in\N} L_n.
\end{equation*}
Moreover, if $\gamma$ is $C^1$ then due to \cite[Proposition 2.1.18 and Exercise 2.3]{baer} we have the following well-known equation
\begin{equation*}
\Le(\gamma) = \int_a^b \abs{\gamma'(x)} \, dx.
\end{equation*}
\end{defn}
Later on we consider functions that are continuous and rectifiable. The following example shows that continuity is not enough.
\begin{exmp}
We consider $f:[0,1] \rightarrow \R$
\begin{equation*}
x \mapsto
\begin{cases}
x \sin\kla{\frac{\pi}{2x}} & x\neq 0\\
0 & x=0.
\end{cases}
\end{equation*}
This function is continuous in $[0,1]$ but in \cite[47. Functions of Bounded Variation]{lass} it is shown that $f$ is not rectifiable.
\end{exmp}
\begin{prop} \mylabel{propabscont}
Additionally, we introduce the space $\text{AC}(\R/2\pi\Z,\R^3)$ of \emph{absolutely continuous} functions (see \cite[Definition 3.31 and the lines below]{AFP}). It coincides with the \name{Sobolev} space $W^{1,1}(\R/2\pi\Z,\R^3)$ (see \thref{defsobolev}) and a function $f\in \text{AC}(\R/2\pi\Z,\R^3)$ is continuous and satisfies the fundamental theorem of calculus
\begin{equation*}
f(s)-f(t) = \int_t^s f'(x) \, dx \qquad \text{for all $s,t\in\R$}.
\end{equation*}
Moreover, it is well known that $f$ is rectifiable, as for $0=t_0<t_1<\dots<t_n=2\pi$ we have
\begin{equation*}
\sum_{i=1}^n\,\abs{f(t_i)-f(t_{i-1})} = \sum_{i=1}^n\,\babs{\int_{t_{i-1}}^{t_i} f'(x) \, dx} \leq \sum_{i=1}^n \int_{t_{i-1}}^{t_i} \abs{f'(x)} \, dx = \int_a^b \abs{f'(x)} \, dx < \infty,
\end{equation*}
because $f' \in L^1([0,2\pi],\R^3)$.
\end{prop}
\begin{rem}
Observe that the fact that a function $f\in W^{1,1}(\R/2\pi\Z,\R^3)$ is continuous in general is only valid for a one dimensional domain.
\end{rem}

\begin{defn} \mylabel{defregular}
An absolutely continuous curve $\gamma:\R/2\pi\Z\rightarrow\R^3$ is called \emph{regular} if there exists a constant $C>0$ such that
\begin{equation*}
\abs{\gamma'(x)} \geq C > 0 \qquad \text{for almost every $x\in\R$}.
\end{equation*}
Observe that for a continuously differentiable curve $\gamma$ it is sufficient that $\abs{\gamma'(x)}\neq 0$ for all $x\in\R$ in order to be regular. For every rectifiable curve $\gamma$ there exists a reparametrization $\Gamma$, which is parametrized by \emph{arclength} \cite[2.5.16]{federer}, i.e. $\abs{\Gamma'(x)}=1$ for almost every $x\in\R$ and hence $\Gamma$ is obviously regular.
\end{defn}

A curve $\gamma \in AC(\R/2\pi\Z,\R^3)$ that is closed, simple and regular will be called \emph{a knot}.

\begin{defn} \mylabel{defloccur}
For a curve $\gamma\in C^2(\R/2\pi\Z,\R^3)$ parametrized by arclength, we define the \emph{classic local curvature} of $\gamma$ at $t\in\R$ by
\begin{equation*}
\kappa_\gamma(t) := \abs{\gamma''(t)}.
\end{equation*}
Moreover, for a regular curve, not necessarily parametrized by arclength, we have
\begin{equation*}
\kappa_\gamma(t) = \frac{\abs{\gamma'(t)\wedge \gamma''(t)}}{\abs{\gamma'(t)}^3}.
\end{equation*}
\end{defn}

For an exact mathematical description of the so-called \emph{knot classes} we follow the presentation in \cite[Definition 1.1 and 1.2]{BZ03} and let $X$ and $Y$ be \name{Hausdorff} spaces. A function $f: X \rightarrow Y$ is called \emph{embedding}, if $f: X \rightarrow f(X)$ is a homeomorphism.
\begin{defn}[Isotopy]
Two embeddings $f_0, f_1: X \rightarrow Y$ are \emph{isotopic}, if there exists an embedding
\begin{equation*}
F: X \times [0,1] \rightarrow Y \times [0,1],
\end{equation*}
such that $F(x,t) = (f(x,t),t)$ for all $x\in X$ and $t \in [0,1]$, with $f(\cdot,0)=f_0$ and $f(\cdot,1)=f_1$. F is called a \emph{level-preserving isotopy} connecting $f_0$ and $f_1$.
\end{defn}
\begin{defn}[Ambient isotopy]
Two embeddings $f_0, f_1:X\rightarrow Y$ are \emph{ambient isotopic} if there is a level preserving isotopy
\begin{equation*}
H:Y\times[0,1] \rightarrow Y\times[0,1], \qquad H(y,t) = (h_t(y),t).
\end{equation*}
with $f_1=h_1 f_0$ and $h_0=id_Y$. The mapping $H$ is called an \emph{ambient isotopy}.
\end{defn}
For our purpose we use $X=\Sp^1$ and $Y=\R^3$. Two knots belong to the same \emph{knot class}, if they are ambient isotopic. An Isotopy describes the transformation of one knot to another. However, the concept of isotopy is not sufficient here, hence small (sub-)knots could tighten more and more and even vanish in the end. However, this is not part of our concept of equivalency. The advantage of ambient isotopy is in particular that the complete surrounding space of the knot is transformed into the end configuration in a continuous way. Closed polygons that can be transformed into each other by \name{Reidemeister} moves are ambient isotopic and vice versa \cite[Proposition 1.14]{BZ03}.
\begin{prop} \mylabel{propc1neigh}
Let $\gamma \in C^1(\R/2\pi\Z,\R^3)$ be simple and regular and $\Phi \in C^1(\R/2\pi\Z,\R^3)$. Then there exists a $\tau_0>0$ such that all $\gamma+\tau\Phi$ are simple, regular and ambient isotopic to $\gamma$ for $\tau\in[-\tau_0,\tau_0]$.\\
Since $\gamma$ is regular there exists a constant $C>0$ such that $\abs{\gamma'(x)} \geq C$ for all $x\in\R/2\pi\Z$. Moreover, we have
\begin{equation*}
\abs{\gamma'(x)+\tau\Phi'(x)} \geq \frac{1}{2}\,C >0 \qquad \text{for all $x\in\R/2\pi\Z$ and all $\tau\in[-\tau_0,\tau_0]$}.
\end{equation*}
One could say that $\gamma+\tau\Phi$, $\tau\in [-\tau_0,\tau_0]$ is even uniformly regular, in the sense that the constant does not depend on $\tau$.
\end{prop}
\begin{proof}
Due to \cite[Lemma]{reiter_iso} there exists an $\eps^\ast>0$ such that for $\eps_0:=\frac{\eps^\ast}{\norm{\Phi'}_{C^0}\,+1}$ we have that $\gamma$ and $\gamma+\tau\Phi$ are ambient isotopic for all $\tau\in[-\eps_0,\eps_0]$ since
\begin{equation*}
\norm{\gamma'-(\gamma'+\tau\Phi')}_{C^0} = \abs{\tau} \norm{\Phi'}_{C^0} \leq \eps^\ast.
\end{equation*}
Let $\eps_1 := \frac{C}{2\norm{\Phi'}_{C^0}\,+1}$. Then we have for $\tau\in[-\eps_1,\eps_1]$
\begin{equation*}
\abs{\gamma'(x)+\tau\Phi'(x)} \geq \abs{\gamma'(x)} - \abs{\tau} \abs{\Phi'(x)} \geq \frac{1}{2}\,C > 0.
\end{equation*}
Therefore, $\gamma+\tau\Phi$ is regular for $\tau\in[-\eps_1,\eps_1]$.\\
We prove the simplicity by contradiction. We assume that for all $\eps_2>0$ there exist $x\in[0,2\pi)$, $y\in[x-\pi,x+\pi)\setminus\menge{x}$ and $\tau\in [-\eps_2,\eps_2]$ such that
\begin{equation*}
(\gamma+\tau\Phi)(x) = (\gamma+\tau\Phi)(y) \qquad \Leftrightarrow \qquad \abs{\gamma(x)-\gamma(y)+\tau(\Phi(x)-\Phi(y))}=0.
\end{equation*}
Let $\sigma_k>0$ for $k\in\N$ be a null sequence and let $x_k\in[0,2\pi)$, $y_k\in[x_k-\pi,x_k+\pi)\setminus\menge{x_k}$ as well as $\tau_k\in[-\sigma_k,\sigma_k]$ for $k\in\N$ be the corresponding sequences from the assumption. At first we observe that $\tau_k$ is a null sequence. Due to the \name{Bolzano-Weierstrass} theorem, as $y_k\in[-\pi,3\pi]$ for all $k\in\N$, there exists a convergent subsequence, which will be denoted by the index $k$ again, with
\begin{equation*}
x_k \rightarrow \bar{x}, \quad y_k \rightarrow \bar{y}, \qquad \text{where $\bar{x}\in [0,2\pi]$ and $\bar{y}\in [\bar{x}-\pi,\bar{x}+\pi]$}.
\end{equation*}
Due to
\begin{equation*}
0 = \abs{\gamma(x_k)-\gamma(y_k)+\tau_k(\Phi(x_k)-\Phi(y_k))} \geq \babs{\abs{\gamma(x_k)-\gamma(y_k)}-\abs{\tau_k}\,\abs{\Phi(x_k)-\Phi(y_k)}} \geq 0,
\end{equation*}
we have $\abs{\gamma(x_k)-\gamma(y_k)}=\abs{\tau_k}\,\abs{\Phi(x_k)-\Phi(y_k)}$. Since $\abs{\tau_k}$ is a null sequence and
\begin{equation*}
\abs{\Phi(x_k)-\Phi(y_k)} \leq \abs{\Phi(x_k)}+\abs{\Phi(y_k)} \leq 2 \norm{\Phi}_{C^0} < \infty,
\end{equation*}
we have $\abs{\gamma(x_k)-\gamma(y_k)} \xrightarrow{k\rightarrow\infty} 0$ and consequently
\begin{equation*}
\abs{\gamma(\bar{x})-\gamma(\bar{y})}=0.
\end{equation*}
As $\gamma$ is simple we have $z:=\bar{x}=\bar{y} \in [0,2\pi]$. Therefore,
\begin{equation*}
\frac{\Phi(x_k)-\Phi(y_k)}{x_k-y_k} \xrightarrow{k\rightarrow\infty} \Phi'(z), \qquad \abs{\Phi'(z)} \leq \norm{\Phi'}_{C^0} < \infty.
\end{equation*}
We see this if we apply the mean value theorem with respect to the points $x_k, y_k$ to each scalar component $\Phi_i\in C^1(\R/2\pi\Z,\R)$ getting a $\zeta_k$ between $x_k$ and $y_k$ each time with
\begin{equation*}
\frac{\Phi_i(x_k)-\Phi_i(y_k)}{x_k-y_k} = \Phi_i'(\zeta_k) \xrightarrow{k\rightarrow\infty} \Phi_i'(z),
\end{equation*}
since $\zeta_k \xrightarrow{k\rightarrow\infty}z$. The same is valid for $\gamma$ instead of $\Phi$. Hence,
\begin{equation*}
0 \xleftarrow{k\rightarrow\infty} \abs{\tau_k} \babs{\frac{\Phi(x_k)-\Phi(y_k)}{x_k-y_k}} = \babs{\frac{\gamma(x_k)-\gamma(y_k)}{x_k-y_k}} \xrightarrow{k\rightarrow\infty} \abs{\gamma'(z)},
\end{equation*}
in contradiction to the fact that $\gamma$ is regular. Consequently, there exists a constant $\eps_2$ such that $\gamma+\tau\Phi$ is simple for all $\tau\in[-\eps_2,\eps_2]$.\\
Finally, we complete the proof by defining $\tau_0 := \min\menge{\eps_0,\eps_1,\eps_2} > 0$.
\end{proof}

\subsection*{Trigonometric identities} \label{trigid}
At first we mention a well known property of the sine function
\begin{equation} \label{sinelb}
\abs{\sin(x)} \;\geq\; \frac{2}{\pi}\,\abs{x} \quad \text{for all } x\in[-\frac{\pi}{2},\frac{\pi}{2}],
\end{equation}
which can be explained by the convexity of the sine function in this area. Recall that sine is \emph{odd}, meaning $\sin(-x)=-\sin(x)$ for all $x\in\R$ and cosine is \emph{even}, meaning $\cos(-x)=\cos(x)$ for all $x\in\R$. Remember that for all $x\in\R$ we have $\abs{\sin(x)}\leq 1$ and $\abs{\cos(x)}\leq 1$.

Now we collect some useful trigonometric identities. We start with the \emph{addition theorems} (see for example \cite[VIII. Additionstheoreme]{trig_ebene})
\begin{align}
\sin(\alpha+\beta) &= \sin(\alpha) \cos(\beta) + \cos(\alpha) \sin(\beta) \label{trigadds}\\
\cos(\alpha+\beta) &= \cos(\alpha) \cos(\beta) - \sin(\alpha) \sin(\beta) \label{trigaddc}.
\end{align}
Next we mention these direct conclusions
\begin{align}
\sin(2 \alpha) &= 2 \sin(\alpha) \cos(\alpha) \label{trigtwos}\\
\cos(2 \alpha) &= 2 \cos(\alpha)^2 -1 \;=\; 1 -2 \sin(\alpha)^2 \label{trigtwoc}
\end{align}
and
\begin{align}
\sin(3 \alpha) &= \sin(2\alpha)\cos(\alpha)+\cos(2\alpha)\sin(\alpha) \notag\\
&= \sin(\alpha) \kla{4 \cos(\alpha)^2 -1} \label{trigthrees}\\
\cos(3 \alpha) &= \cos(2\alpha)\cos(\alpha)-\sin(2\alpha)\sin(\alpha) \notag\\
&= \cos(\alpha) \kla{1 -4 \sin(\alpha)^2}. \label{trigthreec}
\end{align}
Now we gain formulas for the sum of two $\sin$ respectively $\cos$ terms
\begin{alignat}{2}
\sin(\alpha) + \sin(\beta) &=\;&& \sin(\frac{\alpha+\beta}{2}+\frac{\alpha-\beta}{2}) + \sin(\frac{\alpha+\beta}{2}-\frac{\alpha-\beta}{2}) \notag\\
&=\;&& \sin(\frac{\alpha+\beta}{2}) \cos(\frac{\alpha-\beta}{2}) + \cos(\frac{\alpha+\beta}{2}) \sin(\frac{\alpha-\beta}{2}) + \notag\\
&&& \sin(\frac{\alpha+\beta}{2}) \cos(\frac{\alpha-\beta}{2}) - \cos(\frac{\alpha+\beta}{2}) \sin(\frac{\alpha-\beta}{2}) \notag\\
&=\;&& 2 \sin(\frac{\alpha+\beta}{2}) \cos(\frac{\alpha-\beta}{2}), \label{trigsums}
\end{alignat}
analog
\begin{equation} \label{trigdiffs}
\sin(\alpha) - \sin(\beta) = 2 \cos(\frac{\alpha+\beta}{2}) \sin(\frac{\alpha-\beta}{2}),
\end{equation}
and
\begin{alignat}{2}
\cos(\alpha) + \cos(\beta) &=\;&& \cos(\frac{\alpha+\beta}{2}+\frac{\alpha-\beta}{2}) + \cos(\frac{\alpha+\beta}{2}-\frac{\alpha-\beta}{2}) \notag\\
&=\;&& \cos(\frac{\alpha+\beta}{2}) \cos(\frac{\alpha-\beta}{2}) - \sin(\frac{\alpha+\beta}{2}) \sin(\frac{\alpha-\beta}{2}) + \notag\\
&&& \cos(\frac{\alpha+\beta}{2}) \cos(\frac{\alpha-\beta}{2}) + \sin(\frac{\alpha+\beta}{2}) \sin(\frac{\alpha-\beta}{2}) \notag\\
&=\;&& 2 \cos(\frac{\alpha+\beta}{2}) \cos(\frac{\alpha-\beta}{2}), \label{trigsumc}
\end{alignat}
analog
\begin{equation} \label{trigdiffc}
\cos(\alpha) - \cos(\beta) = 2 \sin(\frac{\alpha+\beta}{2}) \sin(\frac{\alpha-\beta}{2}),
\end{equation}
using that $\cos$ is even and $\sin$ is odd. The next identity is also standard
\begin{equation} \label{trigssum}
\sin(\alpha)^2 + \cos(\alpha)^2 = 1.
\end{equation}

\section{The circumcircle}
For three distinct and non-collinear points in three dimensional space $X,Y,Z \in \R^3$ we denote the radius of their circumcircle as $R(X,Y,Z)$, which is also called \emph{circumradius}.
\begin{figure}[H]
\begin{center}
\includegraphics{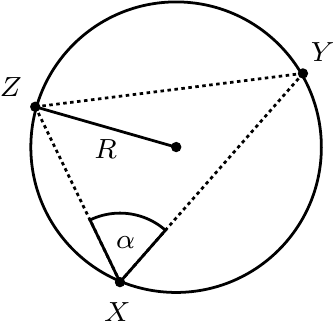}
\end{center}
\caption{The circumradius $R(X,Y,Z)$ of three points $X, Y, Z \in \R^3$}
\label{figcircumcircle}
\end{figure}
The points $X,Y,Z \in \R^3$ always form a triangle which lies in one particular plane. Moreover, these points also lie on a unique circle called the \emph{circumcircle}. Due to the symmetry of a circle the perpendicular bisectors of the three sides of the triangle all intersect in one point $O$, which is called the \emph{circumcentre}, see \cite[1.5, p. 12f]{coxeter}. We represent the well known proof of \name{Euclid}'s theorem ``In a circle the angle at the centre is double the angle at the circumference, when the rays forming the angles meet the circumference in the same two points.'', cf. \cite[1.3 Pons Asinorum]{coxeter}. By $\,\overline{\!XY}$ we denote the line segment between $X$ and $Y$. Moreover, we denote by $\sphericalangle(XY\!Z)$ the smaller angle enclosed by $\,\overline{\!XY}$ and $\overline{Y\!Z}$ as well as by $\triangle(XY\!Z)$ the triangle defined by the sides $\,\overline{\!XY}$, $\overline{Y\!Z}$ and $\,\overline{\!X\!Z}$. We consider Figure~\ref{figcompcircumcircle}, where we have the circumcircle of the triangle $\triangle(XY\!Z)$ around $O$ and we split the angle $\alpha = \sphericalangle(Z\!XY)$ into the angles $\alpha_1$ and $\alpha_2$ by the line segment $\,\overline{\!O\!X}$.
\begin{figure}[H]
\begin{center}
\includegraphics{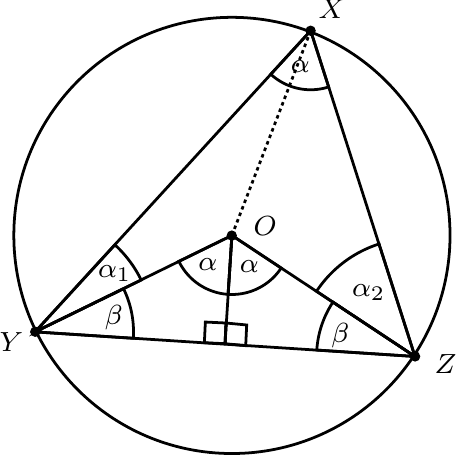}
\end{center}
\caption{Computing the circumradius $R(X,Y,Z)$ of three points $X,Y,Z \in \R^3$}
\label{figcompcircumcircle}
\end{figure}
Since the triangles $\triangle(Y\!O\!X)$, $\triangle(Z\!X\!O)$ and $\triangle(ZOY)$ are isosceles, the angles at each base are equal. The perpendicular of the side $\,\overline{Y\!Z}$ through the centre $O$ splits the triangle $\triangle(ZOY)$ into two congruent triangles and therefore, the angle at the centre is halved. Now we can compute, using the fact that the sum of the angles in a triangle is equal to $\pi$
\begin{alignat*}{2}
&&\alpha + (\alpha_1+\beta) + (\alpha_2+\beta) &= \pi\\
\Rightarrow\qquad &&\alpha &= \frac{\pi}{2} - \beta
\end{alignat*}
and therefore, the halved angle at the centre is equal to $\alpha$. Again as in \cite[1.5, p. 12f]{coxeter} we can now compute
\begin{equation*}
R \sin(\alpha) = \frac{1}{2}\, \text{d}_{Y,Z},
\end{equation*}
where $\text{d}_{X,Y} = \abs{Y-X}$ measures the distance between the points $X$ and $Y$. For the inverse of the circumradius we now have
\begin{alignat}{2}
\frac{1}{R(X,Y,Z)} &= \frac{2\,\sin\kla{\sphericalangle(\mathit{ZXY})}}{\text{d}_{Y,Z}} &\quad=\quad& \frac{2\, \abs{\frac{Y-X}{\abs{Y-X}} \wedge \frac{Z-X}{\abs{Z-X}}}}{\abs{Z-Y}} \notag\\
&= \frac{2\,\abs{(Y-X) \wedge (Z-X)}}{\abs{Y-X}\abs{X-Z}\abs{Z-Y}} \label{R001}\\
&= \frac{4\,\text{A}\kla{\triangle(X,Y,Z)}}{\text{d}_{X,Y}\,\text{d}_{X,Z}\,\text{d}_{Y,Z}} &\quad=\quad& \frac{4\kla{\frac{1}{2}\,\text{d}_{Y,Z}\; \text{d}(X,L_{YZ})}}{\text{d}_{X,Y}\,\text{d}_{X,Z}\,\text{d}_{Y,Z}} \notag\\
&= \frac{2\,\text{d}(L_{YZ},X)}{\text{d}_{X,Y}\,\text{d}_{X,Z}}, \notag
\end{alignat}
where $\text{A}(\triangle)$ is the area of a given triangle, $\text{d}(X,L_{YZ})$ gives the distance between the point $X$ to the line through the points $Y$ and $Z$. In Figure~\ref{figcircumcircle} we see the points $X,Y,Z\in\R^3$, the radius $R=R(X,Y,Z)$, the angle $\alpha = \sphericalangle(\mathit{ZXY})$ and the triangle $\triangle(X,Y,Z)$.

Another interesting alternative is the following in which we express the whole formula only by distance functions. We start with \eqref{R001} and compute for $X,Y,Z \in \R^3$
\begin{align}
& 4\,\abs{(Y-X)\wedge (Z-X)}^2 \;=\; 4\,\kla{\abs{Y-X}^2\,\abs{Z-X}^2-\kla{(Y-X)\cdot (Z-X)}^2} \notag\\
&= 4\,\text{d}_{X,Y}^2\,\text{d}_{X,Z}^2 - \kla{\text{d}_{X,Y}^2+\text{d}_{X,Z}^2-\text{d}_{Y,Z}^2}^2 \label{dist001}\\
&= \kla{2\,\text{d}_{X,Y}\,\text{d}_{X,Z}+\text{d}_{X,Y}^2+\text{d}_{X,Z}^2-\text{d}_{Y,Z}^2}\kla{2\,\text{d}_{X,Y}\,\text{d}_{X,Z}-\text{d}_{X,Y}^2-\text{d}_{X,Z}^2+\text{d}_{Y,Z}^2} \notag\\
&= \kla{\kla{\text{d}_{X,Y}+\text{d}_{X,Z}}^2-\text{d}_{Y,Z}^2}\kla{\text{d}_{Y,Z}^2-\kla{\text{d}_{X,Y}-\text{d}_{X,Z}}^2} \notag\\
&= \kla{\text{d}_{X,Y}+\text{d}_{X,Z}+\text{d}_{Y,Z}}\kla{\text{d}_{X,Y}+\text{d}_{X,Z}-\text{d}_{Y,Z}}\kla{\text{d}_{X,Y}-\text{d}_{X,Z}+\text{d}_{Y,Z}}\kla{-\text{d}_{X,Y}+\text{d}_{X,Z}+\text{d}_{Y,Z}}, \notag
\end{align}
where we used the binomial theorems and in \eqref{dist001} 
\begin{align*}
&\abs{X-Y}^2+\abs{X-Z}^2-\abs{Y-Z}^2\\
&= \abs{X}^2 -2\,X\cdot Y+\abs{Y}^2+\abs{X}^2-2\,X\cdot Z+\abs{Z}^2-\abs{Y}^2+2\,Y\cdot Z-\abs{Z}^2\\
&= 2\,\abs{X}^2 -2\,X\cdot Z-2\,Y\cdot (Z-X) = 2\,(Y-X)\cdot (Z-X).
\end{align*}
Finally, we obtain
\begin{align}
&\frac{1}{R(X,Y,Z)}= \label{R002}\\
&\frac{\sqrt{\kla{\text{d}_{X,Y}+\text{d}_{X,Z}+\text{d}_{Y,Z}}\kla{\text{d}_{X,Y}+\text{d}_{X,Z}-\text{d}_{Y,Z}}\kla{\text{d}_{X,Y}-\text{d}_{X,Z}+\text{d}_{Y,Z}}\kla{-\text{d}_{X,Y}+\text{d}_{X,Z}+\text{d}_{Y,Z}}}}{\text{d}_{X,Y}\,\text{d}_{X,Z}\,\text{d}_{Y,Z}}. \notag
\end{align}
To numerically compute the inverse of the circumradius the formulas \eqref{R001} and \eqref{R002} are reasonable candidates at first sight, but some experiments indicated that the latter is worse due to loss of significance effects. Therefore, we will use this representation, which is symmetric in its three arguments,
\begin{equation} \label{oneoverr}
\frac{1}{R(X,Y,Z)} = \frac{2\,\abs{(Y-X) \wedge (Z-X)}}{\abs{Y-X}\abs{X-Z}\abs{Z-Y}} = \frac{2\,\abs{Y\wedge Z + X\wedge Y +Z\wedge X}}{\abs{Y-X}\abs{X-Z}\abs{Z-Y}},
\end{equation}
see \eqref{defwedge}. Observe that all representations are also well-defined for collinear triples such that $\frac{1}{R}$ is ``equal'' to $0$ in this case. This leads to an infinite radius for such a triple, which corresponds to the view of a straight line as a circle with infinite radius. Since the wedge product is bi-linear we immediately obtain for $r>0$
\begin{equation} \label{scaleR}
\frac{1}{R(rX,rY,rZ)} = \frac{1}{r\,R(X,Y,Z)}.
\end{equation}

The function $\frac{1}{R}$ is continuous at triples of distinct points. Nevertheless, if two or more points coincide, the corresponding circumcircle is not unique any more leading to the fact that the function is not continuous at those triples.
\begin{exmp}
We consider $\frac{1}{R}$ in the $x,y$-plane and show that it is not continuous at $(-e_2,e_2,e_2)$ respectively $(e_2,e_2,e_2)$. To do so we take the shifted unit circle given by
\begin{equation*}
\gamma(s):=\left(
\begin{array}{c}
\cos(\frac{\pi}{2}-s)\\
\sin(\frac{\pi}{2}-s)
\end{array}
\right), \qquad s\in [0,2\pi].
\end{equation*}
This way, $[0,\pi]\ni s\mapsto \gamma(s)$ parameterizes the right semicircle, in particular $\gamma(0)=e_2$ and $\gamma(\pi)=-e_2$.
\begin{enumabc}
\item Now we consider a second circle. Choose $r>1$ and define $a:=\sqrt{r^2-1}$ as well as $\alpha:=\arcsin(\frac{1}{r})\in (0,\frac{\pi}{2})$, where $\arcsin$ is the inverse function of the sine function.
\begin{figure}[H]
\begin{center}
\includegraphics{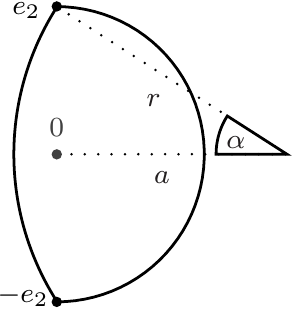}
\end{center}
\caption{A planar closed curve where $R$ is discontinuous}
\label{figRdiscont}
\end{figure}
We define
\begin{equation}
\tilde{\gamma}(s):=\left(
\begin{array}{c}
r \cos(s+\pi-\alpha) + a\\
r \sin(s+\pi-\alpha)
\end{array}
\right), \qquad s\in [0,2\pi],
\end{equation}
which is the circle with radius $r$ around the point $a e_1$. Observe, that $\tilde{\gamma}(\alpha)=(a-r)e_1$, i.e. $\tilde{\gamma}(\alpha)$ lies on the x-axis. Moreover, $\tilde{\gamma}$ is constructed such that $\tilde{\gamma}(0)=e_2$ and $\tilde{\gamma}(2\alpha)=-e_2$. Consequently we have that the image $\gamma([0,\pi])\cup\tilde{\gamma}([0,2\alpha])$ is exactly the curve of Figure~\ref{figRdiscont}. Let $(s_n)_{n\in\N}\subset (0,\pi)$ and $(c_n)_{n\in\N}\subset (0,2\alpha)$ be two null series. Therefore,
\begin{equation*}
\frac{1}{R(-e_2,e_2,\gamma(s_n))} = 1 \qquad \text{and} \qquad \frac{1}{R(-e_2,e_2,\tilde{\gamma}(c_n))} = r,
\end{equation*}
however, $\gamma(s_n)$ and $\tilde{\gamma}(c_n)$ both converge to $e_2$ for $n \rightarrow \infty$. As we found two possible values for $(-e_2,e_2,e_2)$, $\frac{1}{R}$ cannot be continuous there.
\item Even if we consider a smoother function like this \emph{stadium-curve}, which is $C^{1,1}$, we can get problems if all three points coincide.
\begin{figure}[H]
\begin{center}
\includegraphics{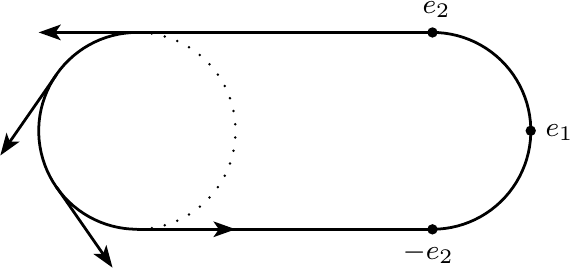}
\end{center}
\caption{A stadium-curve}
\label{figStadionCurve}
\end{figure}
Let $(s^1_n)_{n\in\N}\subset (0,\pi)$ and $(c^1_n)_{n\in\N}\subset (0,1]$ be two null series and furthermore, for $\alpha_s,\alpha_c\in (0,1)$ we define $s^2_n:=\alpha_s s^1_n$ and $c^2_n:=\alpha_c s^1_n$ for all $n\in\N$, which are also null series. Since the points $e_2$, $e_2-c^1_n e_1$ and $e_2-c^2_n e_1$ are collinear, we obtain
\begin{equation*}
\frac{1}{R(e_2,e_2-c^1_n e_1,e_2-c^2_n e_1)} = 0 \qquad \text{and} \qquad \frac{1}{R(e_2,\gamma(s^1_n),\gamma(s^2_n))} = 1,
\end{equation*}
for all $n\in\N$. We get a contradiction by the same argument as before, as both triples converges to $(e_2,e_2,e_2)$ for $n \rightarrow \infty$.
\end{enumabc}
\end{exmp}

\section{Thickness of a knot}
We use the following definition that goes back to \cite{GM99}.
\begin{defn} \mylabel{defthickness}
Let $\gamma: \R/2\pi\Z \rightarrow \R^3$ be a closed, simple and rectifiable curve. We define the \emph{thickness of $\gamma$} as
\begin{equation*}
\Delta[\gamma] := \inf\limits_{\stackrel{x,y,z \in \gamma(\I)}{x\neq y\neq z\neq x}} R(x,y,z).
\end{equation*}
\end{defn}
To get used to that quantity we have a look at the following
\begin{exmp} \mylabel{exmpthickone}
We consider two examples of knots with thickness $\Delta[\gamma] = 1$.
\begin{enumerate}
\item A trivial case is the \emph{unit circle}, because for every triple of points on the curve the circumcircle is again the unit circle. Therefore, the restriction of $R$ to triples of points on the unit circle is constantly equal to $1$ and we immediately get that the thickness is $1$ as well.
\item Here comes another example:
\begin{figure}[H]
\begin{center}
\includegraphics{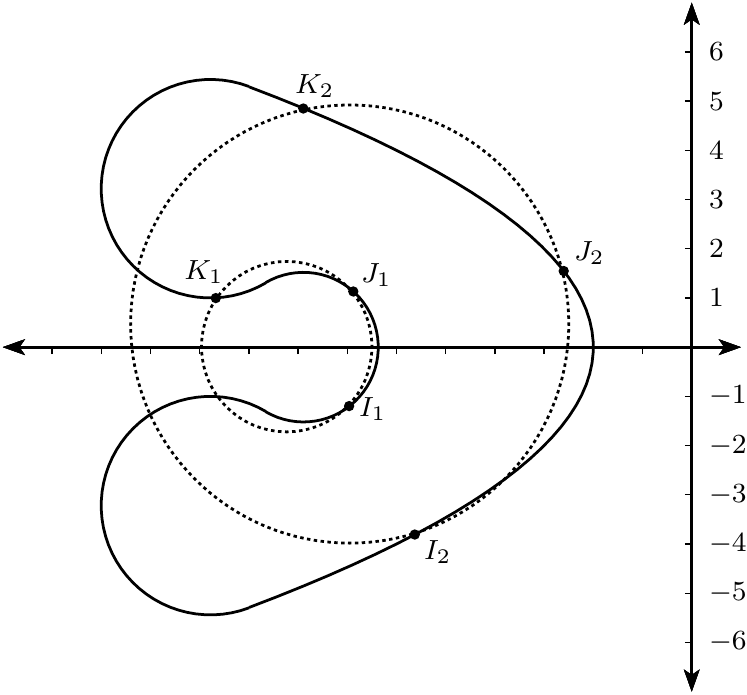}
\end{center}
\caption{A non-trivial example of an unknot with $\Delta[\gamma] = 1$ --- two arbitrary circumcircle}
\label{exmpcircumcircle}
\end{figure}
where we see the considered knot as a solid line and the two circumcircle of the triples $I_1,J_1,K_1$ and $I_2,J_2,K_2$ as pointed lines. To understand why this knot has thickness $1$ we look at the following picture:
\begin{figure}[H]
\begin{center}
\includegraphics{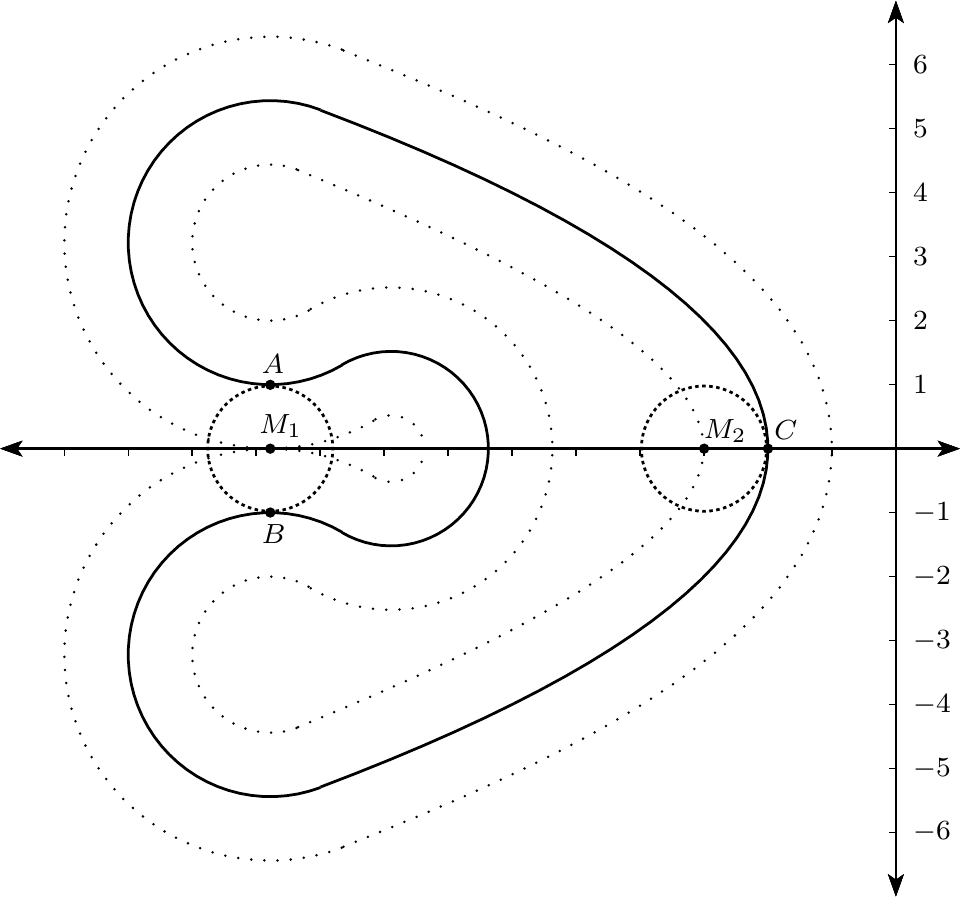}
\end{center}
\caption{A non-trivial example of an unknot with $\Delta[\gamma] = 1$ --- the minimal limit circles}
\label{exmpthickness}
\end{figure}
Again the knot is a solid line surrounded by two pointed lines in such a way that the distance to the solid line is exactly one everywhere. A circumcircle corresponding to three points on the curve as in Figure~\ref{exmpcircumcircle} converges, if the limit exists, to a circle that only has one or two points in common with the curve \cite{phd_smutny}, see also \cite{GMS02}. The circles drawn in dashed lines here are two examples of those. The one on the left-hand side has two distinct points $A,B$ in common with the knot and touches it of order one at the point where two points collapse. The one on the right-hand side touches the knot only at point $C$ of order two. Such a circle is also called an \emph{osculating circle}. As proven by \cite[Lemma 2 and Proposition 1]{SchvdM03_global} 
the thickness of a $C^\infty$-curve is achieved by the radius of such ``limit circles'' and the two circles mentioned before are responsible for the fact that $\Delta(\gamma)=1$. The one on the left is minimal because the distance between $A$ and $B$ is the minimal self-distance of the knot. The radius of an osculating circle is the inverse of the classic local curvature $\kappa_\gamma$ (cf. \thref{defloccur}). Based on this the circle on the right is minimal, because the classic local curvature is maximal at point $C$.
Finally, hence the two circles in dashed lines (Figure~\ref{exmpthickness}) have radius $1$ we have shown that the thickness of $\gamma$ is one.
\end{enumerate}
\end{exmp}

\section{Integral \name{Menger} curvature}
Let $\gamma \in AC(\R/2\pi\Z,\R^3)$ be a closed, simple and regular curve. For a knot $\gamma$ we want to consider the integral of the inverse of $R$ to the fixed power $p>0$ over all triples of points on the curve, more specific
\begin{equation} \label{MenergySet}
\int_\gamma \int_\gamma \int_\gamma \frac{1}{R(x,y,z)^p} \, d\Ha^1(x) \, d\Ha^1(y) \, d\Ha^1(z),
\end{equation}
where we denote the image of $\gamma(\R/2\pi\Z)$ also by $\gamma$ and $d\Ha^1$ is the one dimensional \name{Hausdorff} measure (see for instance \cite[Chapter 2]{evans_gariepy}). As it is a common way of notation we introduce $c_\gamma:\R^3\setminus \Ns \rightarrow \R$ as
\begin{equation} \label{defcg}
c_\gamma(s,t,\sigma) := \frac{1}{R(\gamma(s),\gamma(t),\gamma(\sigma))} = \frac{2 \abs{(\gamma(t)-\gamma(s)) \wedge (\gamma(\sigma)-\gamma(s))}}{\abs{\gamma(t)-\gamma(s)}\abs{\gamma(s)-\gamma(\sigma)}\abs{\gamma(\sigma)-\gamma(t)}},
\end{equation}
which is defined for $(s,t,\sigma) \in \R^3$ almost everywhere except for
\begin{equation} \label{nullset1}
\Ns := \menge{ (s,t,\sigma) \in \R^3 \,\Big\arrowvert\, \abs{s-t}\in 2\pi\Z \text{ or } \abs{s-\sigma}\in 2\pi\Z \text{ or } \abs{t-\sigma}\in 2\pi\Z},
\end{equation}
and if we restrict our self on the cube $[0,2\pi]^3$, we get
\begin{multline} \label{nullset1r}
\Ns\cap[0,2\pi]^3 = \menge{(s,t,\sigma) \in [0,2\pi]^3 \,\Big\arrowvert\, s=t \text{ or } s=\sigma \text{ or } t=\sigma } \cup \\
\menge{ (s,t,\sigma) \in [0,2\pi]^3 \,\Big\arrowvert\, \{s,t\}=\{0,2\pi\} \text{ or } \{s,\sigma\}=\{0,2\pi\} \text{ or } \{t,\sigma\}=\{0,2\pi\}},
\end{multline}
since $\gamma$ is simple but closed. $\Ns\cap[0,2\pi]^3$ is a null-set as a finite union of hypersurfaces and lines. Let additionally $\gamma$ be parameterized by arclength. For those curves we define a first version of \emph{integral \name{Menger} curvature} like this
\begin{equation} \label{MenergyAL}
\M_p(\gamma) := \int_0^L \int_0^L \int_0^L c_\gamma(s,t,\sigma)^p \; d\sigma\,dt\,ds,
\end{equation}
where $L>0$ is the length of $\gamma$ and remark that this expression equals to \eqref{MenergySet}.
However, we want to express the energy without arclength parametrization, because this cannot be guaranteed in our task. Therefore, we use the following definition
\begin{defn} \mylabel{defMenergy}
Let $p>0$ and $\gamma \in AC(\R/2\pi\Z,\R^3)$ be a closed, simple and regular curve. Then we define \emph{integral \name{Menger} curvature} as
\begin{equation*}
\M_p (\gamma) := \int_0^{2\pi} \int_0^{2\pi} \int_0^{2\pi} c(s,t,\sigma)^p \; \abs{\gamma'(s)}\abs{\gamma'(t)}\abs{\gamma'(\sigma)} \, d\sigma \, dt \, ds.
\end{equation*}
\end{defn}
The advantage of this expression is, that it is invariant under reparametrization, see \cite[7.26 Theorem]{rudin} and the special case afterwards. From \eqref{scaleR} we gain that $\M_p$ is not invariant under scaling for $p\neq 3$. In particular for $r>0$ we have
\begin{equation} \label{scaleMp}
\M_p(r\gamma) = r^{3-p} \M_p(\gamma),
\end{equation}
where $r\gamma$ means that each scalar component of $\gamma$ is multiplied by $r$.

For the analytic part of this thesis it is much more convenient to represent integral \name{Menger} curvature in a slightly other way. Due to the fact that the integrand is symmetric in $s,t,\sigma$ we will use \thref{propshift} from time to time. Assuming $\M_p(\gamma) < \infty$, we get
\begin{align}
\M_p (\gamma) & = \int_0^{2\pi} \int_0^{2\pi} \int_0^{2\pi} \frac{\abs{\gamma'(s)}\abs{\gamma'(t)}\abs{\gamma'(\sigma)}}{R^p(\gamma(s), \gamma(t), \gamma(\sigma))} \, d\sigma \, dt \, ds \notag\\
& = \int_0^{2\pi} \int_0^{2\pi} \int_{-s}^{2\pi-s} \frac{\abs{\gamma'(s)}\abs{\gamma'(t)}\abs{\gamma'(s+w)}}{R^p(\gamma(s), \gamma(t), \gamma(s+w))} \, dw \, dt \, ds \label{Z001}\\
& = \int_0^{2\pi} \int_0^{2\pi} \int_{-\pi}^{\pi} \frac{\abs{\gamma'(s)}\abs{\gamma'(t)}\abs{\gamma'(s+w)}}{R^p(\gamma(s), \gamma(t), \gamma(s+w))} \, dw \, dt \, ds \label{Z002}\\
& = \int_0^{2\pi} \int_{-s}^{2\pi-s} \int_{-\pi}^{\pi} \frac{\abs{\gamma'(s)}\abs{\gamma'(s+v)}\abs{\gamma'(s+w)}}{R^p(\gamma(s), \gamma(s+v), \gamma(s+w))} \, dw \, dv \, ds \label{Z003}\\
& = \int_0^{2\pi} \int_{-\pi}^\pi \int_{-\pi}^\pi \frac{\abs{\gamma'(s)}\abs{\gamma'(s+v)}\abs{\gamma'(s+w)}}{R^p(\gamma(s), \gamma(s+v), \gamma(s+w))} \, dw \, dv \, ds, \label{altMenergyOne}
\intertext{where we used the substitution $w:=\sigma-s$ in \eqref{Z001} respectively $v:=t-s$ in \eqref{Z003} and \thref{propshift} in \eqref{Z002} and \eqref{altMenergy}. Using the latter one once more we get}
& = \int_0^{2\pi} \int_{-\frac{\pi}{2}}^{\frac{3\pi}{2}} \int_{0}^{\pi} \frac{\abs{\gamma'(s)}\abs{\gamma'(s+v)}\abs{\gamma'(s+w)}}{R^p(\gamma(s), \gamma(s+v), \gamma(s+w))} \, dw \, dv \, ds \notag\\
& \qquad + \; \int_0^{2\pi} \int_{-\frac{3\pi}{2}}^{\frac{\pi}{2}} \int_{-\pi}^{0} \frac{\abs{\gamma'(s)}\abs{\gamma'(s+v)}\abs{\gamma'(s+w)}}{R^p(\gamma(s), \gamma(s+v), \gamma(s+w))} \, dw \, dv \, ds. \label{altMenergy}
\end{align}
We define the abbreviation $\I := \kla{[0,2\pi]\times[-\frac{\pi}{2},\frac{3\pi}{2}]\times[0,\pi]}\cup\kla{[0,2\pi]\times[-\frac{3\pi}{2},\frac{\pi}{2}]\times[-\pi,0]}$ and get
\begin{align}
\M_p(\gamma) &= \iiintl_\I & c_\gamma(s,s+v,s+w)^p \qquad \abs{\gamma'(s)} \abs{\gamma'(s+v)} \abs{\gamma'(s+w)} \,dw\,dv\,ds \notag\\
&= \iiintl_\I & \frac{2^p \abs{(\gamma(s+v)-\gamma(s)) \wedge (\gamma(s+w)-\gamma(s))}^p}{\abs{\gamma(s+v)-\gamma(s)}^p \abs{\gamma(s)-\gamma(s+w)}^p \abs{\gamma(s+w)-\gamma(s+v)}^p} \quad \notag\\
&& \abs{\gamma'(s)} \abs{\gamma'(s+v)} \abs{\gamma'(s+w)} \,dw\,dv\,ds. \label{altMenergy2}
\end{align}
Observe that $\I\setminus\widetilde{\Ns}\ni(s,v,w)\mapsto c_\gamma(s,s+v,s+w)$ is consequently (see \eqref{nullset1}) defined almost everywhere on $\I$ except for
\begin{align*}
&\widetilde{\Ns} := \menge{(s,v,w)\in\R^3 \,\Big\arrowvert\, v\in2\pi\Z \text{ or } w\in2\pi\Z \text{ or } \abs{w-v}\in2\pi\Z}, \\
\text{where}\quad &\widetilde{\Ns}\cap\I = \menge{(s,v,w)\in\I \,\Big\arrowvert\, v=0 \text{ or } w=0 \text{ or } v=w}.
\end{align*}
\begin{rem} \mylabel{rempropI}
The advantage of the domain $\I$ is that it is compact and moreover, that the distance between $s,s+v,s+w$ is at most $\frac{3\pi}{2}$ which is strictly smaller than $2\pi$, hence $\abs{(s+v)-s}=\abs{v}\leq\frac{3\pi}{2}$, $\abs{(s+w)-s}=\abs{w}\leq\pi$ and $-\frac{3\pi}{2}\leq(s+v)-(s+w)=v-w\leq\frac{3\pi}{2}$.
\end{rem}
In the introduction we already mentioned regularity results by \name{P. Strzelecki}, \name{M. Szuma\'nska} and \name{H. von der Mosel} \cite[Theorem 1.2]{StrSzvdM10} as well as by \name{S. Blatt} \cite[Theorem~1.1]{blatt_menger}. In the next section there will be another proof for the integrability of $c_\gamma$ for knots in the fractional \name{Sobolev} space $W^{2-\frac{2}{p},p}(\R/2\pi\Z,\R^3)$ with $p\in(3,\infty)$, see \thref{deffracsobolev} and at least in a \name{H\"older} space $C^{1,\alpha}(\R/2\pi\Z,\R^3)$ with $\alpha>1-\frac{2}{p}$ for $p\in[2,3]$, see \thref{defhoelder}.
\begin{rem} \mylabel{rempwenergy}
Additionally, it can be shown in a very elegant way that $\M_p$ energy is infinite for piecewise linear knots, which are $C^{0,1}(\R/2\pi\Z,\R^3)$, if $p \geq 3$.
\end{rem}
The following proof is due to \name{P. Strzelecki} (personal communication, June 2009).
\begin{proof}
Without loss of generality we can assume, that the origin is one vertex of the polygon $\gamma$ and that we have two more vertices $a \neq b \in \R^3$ such that $a, b$ and $0$ are not collinear. That is the case, since the energy is invariant under translation and the polygon is closed and embedded. Now we consider the two corresponding edges
\begin{equation*}
S := [0,a] \cup [0,b] \subset \R^3
\end{equation*}
and we define a covering of $S$ for $k \in \N_0$ like that
\begin{equation*}
S_k := \Bigl[ \kla{\frac{1}{2}}^{k+1} a, \kla{\frac{1}{2}}^k a \Bigr] \cup \Bigl[ \kla{\frac{1}{2}}^{k+1} b, \kla{\frac{1}{2}}^k b \Bigr] \qquad \longrightarrow \qquad S = \bigcup_{j=0}^\infty S_j.
\end{equation*}
Now we define
\begin{align*}
M_k &:= \intl_{S_k} \intl_{S_k} \intl_{S_k} \frac{1}{R^p(x,y,z)} \, d\Ha^1(x) \, d\Ha^1(y) \, d\Ha^1(z)\\
&= \intl_{S_{k+1}} \intl_{S_{k+1}} \intl_{S_{k+1}} 2^3 \frac{1}{R^p(2\tilde{x},2\tilde{y},2\tilde{z})} \, d\Ha^1(\tilde{x}) \, d\Ha^1(\tilde{y}) \, d\Ha^1(\tilde{z})\\
&= \intl_{S_{k+1}} \intl_{S_{k+1}} \intl_{S_{k+1}} 2^{3-p} \frac{1}{R^p(\tilde{x},\tilde{y},\tilde{z})} \, d\Ha^1(\tilde{x}) \, d\Ha^1(\tilde{y}) \, d\Ha^1(\tilde{z}),
\end{align*}
where we used $\eqref{scaleR}$ and substituted $\tilde{x}:=\frac{1}{2}x$, $\tilde{y}:=\frac{1}{2}y$ and $\tilde{z}:=\frac{1}{2}z$. Observe that $M_k>0$ for all $k\in\N_0$. Since $\frac{1}{R}$ is a nonnegative function $M_k=0$ would imply that $\frac{1}{R}$ is equal to zero on $S_k$, which is not the case due to the non-collinearity of $a, b$ and $0$. We set $d := 2^{p-3} \geq 1$ as $p \geq 3$ and gain $M_{k+1} = d \, M_k$. Therefore,
\begin{equation*}
M_k = d^k \, M_0 \geq M_0 > 0.
\end{equation*}
After all we get
\begin{equation*}
\M_p(\gamma) \geq \intl_S \intl_S \intl_S \frac{1}{R^p(x,y,z)} \, d\Ha^1(x) \, d\Ha^1(y) \, d\Ha^1(z) = \suml_{k=0}^\infty M_k \geq \suml_{k=0}^\infty M_0 = \infty,
\end{equation*}
which completes the proof.
\end{proof}

In the last part of this section we assume that $\gamma \in C^2(\R/2\pi\Z,\R^3)$. This way we are able to show that $\I\setminus\widetilde{\Ns} \ni (s,v,w) \mapsto c_\gamma(s,s+v,s+w)$ can be extended continuously on $\I$. From \thref{propabscont} we get for $s,t\in \R$
\begin{equation} \label{gdiff_st}
\gamma(t)-\gamma(s) = \int_s^t \gamma'(\tilde{y}) \, d\tilde{y} = (t-s) \; \int_0^1 \gamma'(yt+(1-y)s) \, dy,
\end{equation}
where we substituted $y:=\frac{\tilde{y}-s}{t-s}$. Due to the simplicity and closeness of $\gamma$ the expression $\abs{\gamma(s)-\gamma(t)}$ is zero if and only if $(s-t) \in 2\pi\Z$. For $s=t$ we have $\abs{\int_0^1 \gamma'(s) \, dy} = \abs{\gamma'(s)} \neq 0$ and therefore,
\begin{equation} \label{intzero_iff}
\babs{\int_0^1 \gamma'(ys+(1-y)t) \, dy} = 0 \qquad \Leftrightarrow \qquad \kla{s-t} \in 2\pi\Z\setminus\menge{0}.
\end{equation}
Consequently, we get for $s,v,w\in\R$
\begin{align}
\gamma(s+v)-\gamma(s) &= v \int_0^1 \gamma'(s+yv) \, dy \label{gdiff_sv}\\
\gamma(s+w)-\gamma(s) &= w \int_0^1 \gamma'(s+yw) \, dy \label{gdiff_sw}\\
\gamma(s+w)-\gamma(s+v) &= (w-v) \int_0^1 \gamma'(s+yw+(1-y)v) \, dy \label{gdiff_vw}
\end{align}
and there exist positive constants $C_1,C_2,C_3 > 0$ independent of $(s,v,w)\in\I$ with
\begin{equation} \label{intconsts}
C_1 \leq \babs{\int_0^1 \gamma'(s+yv) \, dy},\, C_2 \leq \babs{\int_0^1 \gamma'(s+yw) \, dy},\, C_3 \leq \babs{\int_0^1 \gamma'(s+yw+(1-y)v) \, dy}.
\end{equation}
This is true, because we can regard the integrand of each right-hand side as a continuous function $\I\times[0,1]\rightarrow\R^3$ and therefore, the corresponding integrals are continuous as well due to \cite[16.1 Continuity lemma]{bauer}. Moreover, $\I$ is compact and these integrals are not zero on $\I$, because the right-hand side of \eqref{intzero_iff} is never true as it is mentioned in \thref{rempropI}. We go ahead to
\begin{equation} \label{wedgeinta}
\abs{(\gamma(s+v)-\gamma(s)) \wedge (\gamma(s+w)-\gamma(s))} = \abs{v}\abs{w}\,\babs{\int_0^1 \gamma'(s+yv) \, dy \wedge \int_0^1 \gamma'(s+yw) \, dy}.
\end{equation}
Observe that up to now we only used the fact that $\gamma$ is $C^1(\R/2\pi\Z,\R^3)$. The integrand of the second factor of the cross product can be regarded as a continuously differentiable function $f$ of $w$, hence $\gamma \in C^2(\R/2\pi\Z,\R^3)$. Therefore, we can use \eqref{gdiff_st} to get $f(w) = f(v) + (w-v) \int_0^1 f'(\tilde{y}w+(1-\tilde{y})v) \, d\tilde{y}$, which was an easy consequence of the fundamental theorem and we get for every $(s,v,w) \in \I$
\begin{align}
= & \abs{v}\abs{w}\,\Bigl\lvert \int_0^1 \gamma'(s+yv) \, dy \,\wedge \notag\\
& \qquad \qquad \Bigl( \int_0^1 \gamma'(s+yv) \, dy + \int_0^1 (w-v) \int_0^1 y\,\gamma''(s+y(\tilde{y}w+(1-\tilde{y})v)) \,d\tilde{y}\,dy \Bigr) \Bigr\rvert \notag\\
= &\abs{v}\abs{w}\abs{w-v}\,\babs{\int_0^1 \gamma'(s+yv) \, dy \wedge \int_0^1 \int_0^1 y\,\gamma''(s+y(\tilde{y}w+(1-\tilde{y})v)) \, d\tilde{y} \, dy}. \label{wedgeintb}
\end{align}
Together with \eqref{gdiff_sv}, \eqref{gdiff_sw} and \eqref{gdiff_vw} we see that the first factors cancel out in $c_\gamma$ and we obtain
\begin{equation} \label{cgammaint}
c_\gamma(s,s+v,s+w) = \frac{2\, \babs{\int_0^1 \gamma'(s+yv) \, dy \wedge \int_0^1 \int_0^1 y\, \gamma''(s+y(\tilde{y}w+(1-\tilde{y})v)) \, d\tilde{y} \, dy}}{\babs{\int_0^1 \gamma'(s+yv) \, dy } \babs{\int_0^1 \gamma'(s+yw) \, dy} \babs{\int_0^1 \gamma'(s+yw+(1-y)v)) \, dy}}.
\end{equation}
Let $s_i \xrightarrow{i \rightarrow \infty} s$, $t_i \xrightarrow{i \rightarrow \infty} t$ and $\sigma_i \xrightarrow{i \rightarrow \infty} \sigma$ where $(s,v,w) \in \I$ and all $(s_i,v_i,w_i) \in \I\setminus\widetilde{\Ns}$. We can choose $y\mapsto\norm{\gamma'}_{C^0(\R/2\pi\Z,\R^3)}$, which is \name{Lebesgue} integrable on $\I$, as a dominating function of $\gamma'(s_i+yw_i+(1-y)v_i))$, which converges pointwise to $\gamma'(s+yw+(1-y)v))$ for every $y \in [0,1]$ if $i \rightarrow \infty$, since $\gamma' \in C^0(\R/2\pi\Z,\R^3)$. Hence we gain
\begin{equation*}
\int_0^1 \gamma'(s_i+yw_i+(1-y)v_i)) \, dy \quad \xrightarrow{i \rightarrow \infty} \quad \int_0^1 \gamma'(s+yw+(1-y)v)) \, dy,
\end{equation*}
applying \name{Lebesgue}'s dominated convergence theorem. Since $\gamma \in C^2(\R/2\pi\Z,\R^3)$, we obtain the estimate $c_\gamma(s,s+v,s+w)\leq \frac{2}{C_1 C_2 C_3}\norm{\gamma}^2_{C^2(\R/2\pi\Z,\R^3)}$ for all $(s,v,w)\in\I$. After handling the other expressions analogously we get
\begin{multline*}
\frac{2\, \babs{\int_0^1 \gamma'(s_i+yv_i) \, dy \wedge \int_0^1 \int_0^1 y\, \gamma''(s_i+y(\tilde{y}w_i+(1-\tilde{y})v_i)) \, d\tilde{y} \, dy}}{\babs{\int_0^1 \gamma'(s_i+yv_i) \, dy } \babs{\int_0^1 \gamma'(s_i+yw_i) \, dy} \babs{\int_0^1 \gamma'(s_i+yw_i+(1-y)v_i)) \, dy}}\\
\xrightarrow{i \longrightarrow \infty} \frac{2\, \babs{\int_0^1 \gamma'(s+yv) \, dy \wedge \int_0^1 \int_0^1 y\, \gamma''(s+y(\tilde{y}w+(1-\tilde{y})v)) \, d\tilde{y} \, dy}}{\babs{\int_0^1 \gamma'(s+yv) \, dy } \babs{\int_0^1 \gamma'(s+yw) \, dy} \babs{\int_0^1 \gamma'(s+yw+(1-y)v)) \, dy}},
\end{multline*}
which means
\begin{equation*}
c_\gamma(s_i,s_i+v_i,s_i+w_i) \xrightarrow{i \longrightarrow \infty} c_\gamma(s,s+v,s+w)
\end{equation*}
and this completes the proof. Moreover, we are now able to easily compute the value of $c_\gamma$ when at least two points coincide. As an example we now compute $c_\gamma(s,s+v,s+v)$ for $(s,v,v)\in\menge{(s,v,w)\in\I\,\arrowvert\,w=v\neq 0}\subset\widetilde{\Ns}$
\begin{equation*}
c_\gamma(s,s+v,s+v) = \frac{2\, \babs{\int_0^1 \gamma'(s+yv) \, dy \wedge \int_0^1 y\,\gamma''(s+yv) \,dy}}{\babs{\int_0^1 \gamma'(s+yv) \, dy}^2 \babs{\gamma'(s+v)}}.
\end{equation*}
First we use partial integration to get
\begin{equation*}
v\,\int_0^1 y\,\gamma''(s+yv) \, dy = \Bigl[y\,\gamma'(s+yv)\Bigr]_{y=0}^{y=1} -\int_0^1 \gamma'(s+yv)\, dy = \gamma'(s+v) -\int_0^1 \gamma'(s+yv) \,dy.
\end{equation*}
Now we use the fundamental theorem of calculus again and do the same computations that lead to \eqref{gdiff_sv} but in opposite direction.
After all we gain
\begin{equation} \label{csvv}
c_\gamma(s,s+v,s+v) = \frac{2\, \abs{(\gamma(s+v) - \gamma(s)) \wedge \gamma'(s+v)}}{\abs{\gamma(s+v)-\gamma(s)}^2\, \abs{\gamma'(s+v)}}.
\end{equation}
The interesting case that all three points coincide will be considered now. Let $s\in[0,2\pi]$
\begin{equation} \label{csss}
c_\gamma(s,s,s) = \frac{2\, \babs{\gamma'(s) \wedge \int_0^1 y\, \gamma''(s) \, dy}}{\abs{\gamma'(s)}^3} = \frac{\abs{\gamma''(s) \wedge \gamma'(s)}}{\abs{\gamma'(s)}^3} = \kappa_\gamma(s),
\end{equation}
which reflects the fact that $\kla{R(\gamma(s),\gamma(t),\gamma(\sigma))}^{-1}$ converges to the classic local curvature (see \thref{defloccur}) when all three points tend to one.

We have proven that $(s,v,w) \mapsto c_\gamma(s,s+v,s+w)$ can be extended continuously on $\I$. Hence we are now able to come back to \thref{defMenergy} and to define a extension of $c_\gamma$ on $[0,2\pi)$, which will be called $c_\gamma$ again. Observe that we have to choose the half-opened interval here, which is not compact, in order to avoid the singularities due to the $2\pi$-periodicity of $\gamma$. Let $(s,t,\sigma)\in[0,2\pi)$. Then there exist $k_1,k_2\in\Z$ such that for $\bar{v}:=(t-s)+2\pi k_1$ and $\bar{w}:=(t-s)+2\pi k_2$ we get $(s,\bar{v},\bar{w})\in\I$. Therefore, we have
\begin{equation*}
c_\gamma(s,t,\sigma) \;=\; c_\gamma(s,s+(t-s),s+(\sigma-s)) \;=\; c_\gamma(s,s+\bar{v},s+\bar{w}),
\end{equation*}
and we gain the desired extension. Of course this extension respects the definition of $c_\gamma$ on $[0,2\pi)^3 \setminus\Ns$ from \eqref{defcg}. For instance, we can carry over the example above and define for $s\neq t$
\begin{equation} \label{cstt}
c_\gamma(s,t,t) := \frac{2\, \abs{(\gamma(t) - \gamma(s)) \wedge \gamma'(t)}}{\abs{\gamma(t)-\gamma(s)}^2\, \abs{\gamma'(t)}}.
\end{equation}
Since the integrand is symmetric in its arguments it is sufficient to calculate this expression for the case that two points are equal and the third one is distinct, because
\begin{alignat}{3}
c_\gamma(s,t,\sigma) &=& \; c_\gamma(t,s,\sigma) &=& \; c_\gamma(t,\sigma,s) \notag\\
\downarrow_{\sigma \rightarrow t} && \downarrow_{\sigma \rightarrow t} && \downarrow_{\sigma \rightarrow t} \label{symmconv}\\
c_\gamma(s,t,t) &=& c_\gamma(t,s,t) &=& c_\gamma(t,t,s). \notag
\end{alignat}
The definition of $c_\gamma(s,s,s)$, which is the extension to the case that all three points coincide, stays the same as in \eqref{csss}.

\section{The first variation}
In order to get compacter notations we use the following abbreviation. For a function $f:\R/2\pi\Z \rightarrow \R^3$ in a function space $X$ we write $\norm{f}_X$ instead of $\norm{f}_{X(\R/2\pi\Z,\R^3)}$.\\

Instead of switching from the domain $\U := [0,2\pi]\times[-\pi,\pi]\times[-\pi,\pi]$ to $\I$ we need to consider the set $G$ this time, which consists of the triples in $\U$, where $w$ and $v$ have different signs and the absolute value of at least one of $w$ and $v$ is less of equal to $\frac{3\pi}{4}$, more precisely
\begin{multline} \label{defG}
G := \kla{\menge{(s,v,w)\in\U\;\arrowvert\;v\leq 0, w\geq 0} \cup \menge{(s,v,w)\in\U\;\arrowvert\;v\geq 0, w\leq 0}} \cap\\
\kla{\menge{(s,v,w)\in\U\;\arrowvert\; \abs{v}\leq\tfrac{3\pi}{4}} \cup \menge{(s,v,w)\in\U\;\arrowvert\; \abs{w}\leq\tfrac{3\pi}{4}}}
\end{multline}
\begin{rem} \mylabel{rempropG}
Like domain $\I$ the set $G$ is compact and the distance between $s,s+v,s+w$ is at most $\frac{7\pi}{4}$ which is again strictly smaller than $2\pi$, hence $\abs{v}\leq\pi$, $\abs{w}\leq\pi$ and $\abs{w-v}\leq \abs{w}+\abs{v}\leq \pi+\frac{3\pi}{4}=\frac{7\pi}{4}$.
\end{rem}
\begin{lem} \mylabel{lemestG}
Let $g:(\R/2\pi\Z)^3 \rightarrow \R$ be a locally \name{Lebesgue} integrable, non-negative function on $\R^3$, which is symmetric in its arguments. The triple integral of $g$ can be transformed as in \eqref{altMenergyOne}. Moreover, there exist sets $G_1,\dots, G_4\subseteq G$ independent of $g$, such that
\begin{equation} \label{lemestG1}
\int_0^{2\pi} \int_{-\pi}^\pi \int_{-\pi}^\pi g(s,s+v,s+w) \, dw \, dv \, ds = 2\,\sum_{k=1}^4\; \iiintl_{G_k} g(s,s+v,s+w) \, dw \, dv \, ds.
\end{equation}
In particular, we have
\begin{equation} \label{lemestG2}
\int_0^{2\pi} \int_{-\pi}^\pi \int_{-\pi}^\pi g(s,s+v,s+w) \, dw \, dv \, ds \leq 8 \iiintl_G g(s,s+v,s+w) \, dw \, dv \, ds.
\end{equation}
\end{lem}
\begin{proof}
We start using \name{Fubini}'s theorem (consider \thref{remproofbw})
\begin{align*}
\int_0^{2\pi} \int_0^\pi \int_v^\pi g(s,s+v,s+w) \, dw \, dv \, ds &= \int_0^\pi \int_v^\pi \int_v^{2\pi+v} g(\tilde{s}-v,\tilde{s},\tilde{s}+(w-v))  \, d\tilde{s} \, dw \, dv\\
&= \int_0^\pi \int_0^{\pi-v} \int_v^{2\pi+v} g(\tilde{s}-v,\tilde{s},\tilde{s}+\tilde{w})  \, d\tilde{s} \, d\tilde{w} \, dv\\
&= \int_{-\pi}^0 \int_0^{\pi+\tilde{v}} \int_{-\tilde{v}}^{2\pi-\tilde{v}} g(\tilde{s}+\tilde{v},\tilde{s},\tilde{s}+\tilde{w})  \, d\tilde{s} \, d\tilde{w} \, d\tilde{v}\\
&= \int_0^{2\pi} \int_{-\pi}^0 \int_0^{\pi+v} g(s,s+v,s+w) \, dw \, dv \, ds,
\end{align*}
where we substituted $\tilde{s}:=s+v$, $\tilde{w}:=w-v$ and $\tilde{v}:=-v$, used \thref{propshift} as well as the symmetry of the arguments and \name{Fubini}'s theorem again to get the last line. From \thref{lemintch} we get for the characteristic function $\chi$ the fact that $\chi_{[0,\pi]}(v)\chi_{[0,v]}(w)=\chi_{[w,\pi]}(v)\chi_{[0,\pi]}(w)$ for all $v,w\in\R$. Hence,
\begin{align*}
\int_0^{2\pi} \int_0^\pi \int_0^v g(s,s+v,s+w) \, dw \, dv \, ds &= \int_0^{2\pi} \int_0^\pi \int_w^\pi  g(s,s+w,s+v)  \, dv \, dw \, ds\\
&= \int_0^{2\pi} \int_0^\pi \int_v^\pi g(s,s+v,s+w) \, dw \, dv \, ds,
\end{align*}
where we interchanged the variables $v$ and $w$ in the last step. Using the same reasoning as before, we obtain
\begin{align*}
\int_0^{2\pi} \int_{-\pi}^0 \int_v^0 g(s,s+v,s+w) \, dw \, dv \, ds &= \int_{-\pi}^0 \int_v^0 \int_v^{2\pi+v} g(\tilde{s}-v,\tilde{s},\tilde{s}+(w-v))  \, d\tilde{s} \, dw \, dv\\
&= \int_0^\pi \int_{-\tilde{v}}^0 \int_v^{2\pi+v} g(\tilde{s}+\tilde{v},\tilde{s},\tilde{s}+(w+\tilde{v}))  \, d\tilde{s} \, dw \, d\tilde{v}\\
&= \int_0^\pi \int_0^{\tilde{v}} \int_{-\tilde{v}}^{2\pi-\tilde{v}} g(\tilde{s}+\tilde{v},\tilde{s},\tilde{s}+\tilde{w})  \, d\tilde{s} \, d\tilde{w} \, d\tilde{v}\\
&= \int_0^{2\pi} \int_0^\pi \int_0^v g(s,s+v,s+w) \, dw \, dv \, ds,
\end{align*}
where we substituted $\tilde{s}:=s+v$, $\tilde{v}:=-v$ and $\tilde{w}:=w+\tilde{v}$. Furthermore, \thref{lemintch} leads us to $\chi_{[-\pi,0]}(v)\chi_{[-\pi,v]}(w)=\chi_{[w,0]}(v)\chi_{[-\pi,0]}(w)$ for all $v,w\in\R$ and we gain
\begin{align*}
\int_0^{2\pi} \int_{-\pi}^0 \int_{-\pi}^v g(s,s+v,s+w) \, dw \, dv \, ds &= \int_0^{2\pi} \int_{-\pi}^0 \int_w^0  g(s,s+w,s+v)  \, dv \, dw \, ds\\
&= \int_0^{2\pi} \int_{-\pi}^0 \int_v^0 g(s,s+v,s+w) \, dw \, dv \, ds.
\end{align*}
We define $G_1:=G_2:=[0,2\pi]\times [-\pi,0]\times [0,\pi+v]$, which are subsets of $G$, since for $v\in [-\pi,-\frac{3\pi}{4}]$ we have $0\leq \pi+v\leq \frac{\pi}{4}$. Therefore, we have
\begin{align*}
\int_0^{2\pi} \int_0^\pi \int_0^\pi g(s,s+v,s+w) \, dw \, dv \, ds &= 2\,\iiintl_{G_1} g(s,s+v,s+w) \, dw \, dv \, ds\\
=\,\int_0^{2\pi} \int_{-\pi}^0 \int_{-\pi}^0 g(s,s+v,s+w) \, dw \, dv \, ds &= 2\,\iiintl_{G_2} g(s,s+v,s+w) \, dw \, dv \, ds.
\end{align*}
Next we decompose $[0,2\pi]\times[0,\pi]\times[-\pi,0]$ into the subset where $\abs{v}, \abs{w}\geq \frac{3\pi}{4}$ and the set $G_4:=[0,2\pi]\times[0,\pi]\times[-\pi,0] \cap G$. Now we move on to
\begin{align*}
& \int_0^{2\pi} \int_0^\pi \int_{-\pi}^0 g(s,s+v,s+w) \, dw \, dv \, ds\\
&\qquad= \int_0^{2\pi} \int_\frac{3\pi}{4}^\pi \int_{-\pi}^{-\frac{3\pi}{4}} g(s,s+v,s+w) \, dw \, dv \, ds + \iiintl_{G_4} g(s,s+v,s+w) \, dw \, dv \, ds
\end{align*}
and
\begin{align}
&\int_0^{2\pi} \int_\frac{3\pi}{4}^\pi \int_{-\pi}^{-\frac{3\pi}{4}} g(s,s+v,s+w) \, dw \, dv \, ds \notag\\
&\qquad= \int_\frac{3\pi}{4}^\pi \int_{-\pi}^{-\frac{3\pi}{4}} \int_v^{2\pi+v} g(\tilde{s}-v,\tilde{s},\tilde{s}+(w-v)) \, d\tilde{s} \, dw \, dv \notag\\
&\qquad= \int_\frac{3\pi}{4}^\pi \int_{\pi-v}^{\frac{5\pi}{4}-v} \int_v^{2\pi+v} g(\tilde{s}-v,\tilde{s},\tilde{s}+\tilde{w}-2\pi) \, d\tilde{s} \, d\tilde{w} \, dv \notag\\
&\qquad= \int_{-\pi}^{-\frac{3\pi}{4}} \int_{\pi+\tilde{v}}^{\frac{5\pi}{4}+\tilde{v}} \int_{-\tilde{v}}^{2\pi-\tilde{v}} g(\tilde{s}+\tilde{v},\tilde{s},\tilde{s}+\tilde{w}) \, d\tilde{s} \, d\tilde{w} \, d\tilde{v} \notag\\
&\qquad= \int_0^{2\pi} \int_{-\pi}^{-\frac{3\pi}{4}} \int_{\pi+v}^{\frac{5\pi}{4}+v} g(s,s+v,s+w) \, dw \, dv \, ds, \label{estG001}
\end{align}
where we substituted $\tilde{s}:=s+v$, $\tilde{w}:=2\pi+w-v$ and $\tilde{v}:=-v$. Consequently, we define $G_3:=[0,2\pi]\times [-\pi,-\frac{3\pi}{4}]\times [\pi+v,\frac{5\pi}{4}+v]$, which is a subset of $G$ since $0\leq\pi+v\leq w\leq\frac{5\pi}{4}+v\leq\frac{\pi}{2}$. Again using \name{Fubini}'s theorem and the symmetry of the arguments, we obtain
\begin{align*}
\int_0^{2\pi} \int_{-\pi}^0 \int_0^\pi g(s,s+v,s+w) \, dw \, dv \, ds &= \int_0^{2\pi} \int_0^\pi \int_{-\pi}^0  g(s,s+w,s+v)  \, dv \, dw \, ds\\
&= \int_0^{2\pi} \int_0^\pi \int_{-\pi}^0 g(s,s+v,s+w) \, dw \, dv \, ds.
\end{align*}
Finally, we have
\begin{align*}
&\int_0^{2\pi} \int_{-\pi}^\pi \int_{-\pi}^\pi g(s,s+v,s+w) \, dw \, dv \, ds\\
&\qquad= \int_0^{2\pi} \int_{-\pi}^0 \int_0^\pi g(s,s+v,s+w) \, dw \, dv \, ds + \int_0^{2\pi} \int_0^\pi \int_{-\pi}^0 g(s,s+v,s+w) \, dw \, dv \, ds\\
&\qquad\;+ \int_0^{2\pi} \int_0^\pi \int_0^\pi g(s,s+v,s+w) \, dw \, dv \, ds + \int_0^{2\pi} \int_{-\pi}^0 \int_{-\pi}^0 g(s,s+v,s+w) \, dw \, dv \, ds\\
&\qquad= 2\,\sum_{k=1}^4\; \iiintl_{G_k} g(s,s+v,s+w) \, dw \, dv \, ds. \qedhere
\end{align*}
\end{proof}
We are interested in knots that are simple, regular and of sufficient regularity. Therefore, we define the class for $p>3$
\begin{equation} \label{defvarclass}
W_{\text{s,r}}^{2-\frac{2}{p},p}(\R/2\pi\Z,\R^3) \; := \; \menge{\gamma \in W^{2-\frac{2}{p},p}(\R/2\pi\Z,\R^3) \,\arrowvert\, \text{$\gamma$ simple and regular}},
\end{equation}
which is well-defined since $W^{2-\frac{2}{p},p}(\R/2\pi\Z,\R^3)\subset C^1(\R/2\pi\Z,\R^3)$, due to \thref{remembedd}. For $\gamma\in W_{\text{s,r}}^{2-\frac{2}{p},p}(\R/2\pi\Z,\R^3)$, $\Phi\in W^{2-\frac{2}{p},p}(\R/2\pi\Z,\R^3)$ and $\tau\in\R$ we use from now on the abbreviation
\begin{equation} \label{gtabb}
\gamma_\tau := \gamma + \tau \Phi.
\end{equation} 
In some applications we want to find minimizers in a given knot class $k$, with or without prescribed length, thus for $\ell,L>0$ we define
\begin{equation} \label{knotclasslen}
\begin{split}
C_{L,k}(\R/\ell\Z,\R^3) &:= \menge{\gamma \in C^0(\R/\ell\Z,\R^3) \,\arrowvert\, \Le(\gamma)=L,\; \gamma \text{ ambient isotopic to } k},\\
C_k(\R/\ell\Z,\R^3) &:= \menge{\gamma \in C^0(\R/\ell\Z,\R^3) \,\arrowvert\, \Le(\gamma) < \infty,\; \gamma \text{ ambient isotopic to } k}.
\end{split}
\end{equation}
Now we come to our main result.
\begin{thm} \mylabel{thmdiff}
Let $\gamma$ be a function in $W_{\text{s,r}}^{2-\frac{2}{p},p}(\R/2\pi\Z,\R^3)$, $\Phi \in W^{2-\frac{2}{p},p}(\R/2\pi\Z,\R^3)$ and $p\in(3,\infty)$ then
\begin{equation*}
\int_0^{2\pi} \int_{-\pi}^\pi \int_{-\pi}^\pi c_{\gamma_\tau}(s,s+v,s+w)^p \,\abs{\gamma_\tau'(s)}\abs{\gamma_\tau'(s+v)}\abs{\gamma_\tau'(s+w)} \, dw \, dv \, ds
\end{equation*}
is continuously differentiable with respect to $\tau$ in a small neighbourhood of $0$. Consequently, the first variation of $\M_p(\gamma)$ exists.
\end{thm}
\begin{rem} \mylabel{remdiff}
Assume we have an energy $E$ like $\M_p$, where the first variation exists for the given classes of functions. Let $k$ be a knot class. For $\gamma \in W_{\text{s,r}}^{2-\frac{2}{p},p}(\R/2\pi\Z,\R^3)\cap C_k(\R/2\pi\Z,\R^3)$ and $\Phi \in W^{2-\frac{2}{p},p}(\R/2\pi\Z,\R^3)$ \thref{propc1neigh} provides a $0<\tau_0\ll1$ such that $\gamma + \tau \Phi$ is in $W_{\text{s,r}}^{2-\frac{2}{p},p}(\R/2\pi\Z,\R^3)\cap C_k(\R/2\pi\Z,\R^3)$ as well for all $\tau \in [-\tau_0,\tau_0]$. Therefore, if $\gamma$ is for instance a local minimizer of $E$ among the given knot class $k$, then $\gamma$ is a stationary point, i.e.
\begin{equation*}
\delta E(\gamma,\Phi) =0 \qquad \text{for all $\Phi\in W^{2-\frac{2}{p},p}(\R/2\pi\Z,\R^3)$}.
\end{equation*}
Later in \thref{remnostpts} we will see that there are no minimizers for $\M_p$ due to its scaling behaviour. There are minimizers if we restrict ourself to $C_{L,k}(\R/2\pi\Z,\R^3)$, i.e. if the length of the considered curves is fixed to $L>0$. However, we use a variant of $\M_p$ that is scale invariant and which has (local) minimizers, which are consequently stationary points.
\end{rem}
For the case $p\in[2,3]$, please consider \thref{remsmallp}.
\begin{proof}[Proof of \thref{thmdiff}]
In the last section we have seen that we are abel to extend $c_\gamma$ on $\I$ for $\gamma\in C^2(\R/2\pi\Z,\R^3)$. Here we do not have that much regularity and therefore, we have to extend the function by zero and consider $\tilde{c}_\gamma: (\R/2\pi\Z)^3 \rightarrow\R$
\begin{equation} \label{defctilde}
\tilde{c}_\gamma(s,t,\sigma) :=
\begin{cases}
c_\gamma(s,t,\sigma)^p \,\abs{\gamma'(s)}\abs{\gamma'(t)}\abs{\gamma'(\sigma)}, & (s,t,\sigma) \in \R^3\setminus\Ns\\
0, & (s,t,\sigma) \in \Ns.
\end{cases}
\end{equation}
Keeping \thref{remproofbw} in mind, we apply \thref{lemestG} to $\tilde{c}_{\gamma_\tau}$, which is symmetric in its arguments due to \eqref{oneoverr}. Since \eqref{altMenergyOne} and equation \eqref{lemestG1} of \thref{lemestG} we aspire to compute
\begin{align*} 
\delta \M_p(\gamma,\Phi) &= \frac{d}{d\tau}\eins{\tau=0} \int_0^{2\pi} \int_{-\pi}^\pi \int_{-\pi}^\pi \tilde{c}_{\gamma_\tau}(s,s+v,s+w) \, dw \, dv \, ds,\\
&= \frac{d}{d\tau}\eins{\tau=0} \;2\,\sum_{k=1}^4 \iiintl_{G_k} \tilde{c}_{\gamma_\tau}(s,s+v,s+w) \, dw \, dv \, ds
\end{align*}
where $G_1,\dots,G_4\subset G$ are the sets from \thref{lemestG}, which are independent of the integrand. Our goal is to prove that this triple integral is differentiable and we will use \cite[16.2 Differentiation lemma]{bauer} on every $G_k, k=1,\dots,4$ for this. Afterwards we use \thref{lemestG} again to get
\begin{align*}
\delta \M_p(\gamma,\Phi) &= 2\,\sum_{k=1}^4 \iiintl_{G_k} \frac{d}{d\tau}\eins{\tau=0}\; \tilde{c}_{\gamma_\tau}(s,s+v,s+w) \, dw \, dv \, ds\\
&= \int_0^{2\pi} \int_{-\pi}^\pi \int_{-\pi}^\pi \frac{d}{d\tau}\eins{\tau=0}\; \tilde{c}_{\gamma_\tau}(s,s+v,s+w) \, dw \, dv \, ds,
\end{align*}
which can be applied since $\frac{d}{d\tau} \tilde{c}_{\gamma\tau}(s,s+v,s+w)$ is also symmetric in its arguments. To do so we will show the following statements, which prove a bit more than necessary:
\begin{enumerate}
\item $(s,v,w) \mapsto \tilde{c}_{\gamma_\tau}(s,s+v,s+w)$ is \name{Lebesgue} integrable on $\U$ for all $\tau \in (-\tau_0,\tau_0)$.
\item $\tau \mapsto \tilde{c}_{\gamma_\tau}(s,s+v,s+w)$ is differentiable on $(-\tau_0,\tau_0)$ for all $(s,v,w) \in \U$.
\item There exists a \name{Lebesgue} integrable function $h\geq0$ on $G$ such that\\
$\abs{\frac{d}{d\tau}\,\tilde{c}_{\gamma_\tau}(s,s+v,s+w)}\leq h(s,v,w)$ for all $(s,v,w)\in G$ and all $\tau\in[-\tau_0,\tau_0]$.
\end{enumerate}
As mentioned in \thref{remdiff} we have a $\tau_0>0$ such that $\gamma_\tau$ is simple and regular for all $\tau\in[-\tau_0,\tau_0]$. However, \thref{propc1neigh} also provides a constant $C_0>0$ such that $\abs{\gamma'_\tau(x)}\geq\frac{1}{2}\,C_0$ for all $x\in\R$ and all $\tau \in [-\tau_0,\tau_0]$.

\emph{Ad statement 1}. Observe that we want to prove the integrability for the whole domain $\U$ here. However, we can restrict our considerations to the set $G$ here as well, because we apply estimate \eqref{lemestG1} of \thref{lemestG} to $\tilde{c}_{\gamma_\tau}$. Obviously the numerator and denominator of $c_{\gamma_\tau}$ are measurable as continuous functions and \thref{propmeas} provides that $\tilde{c}_{\gamma_\tau}$ is measurable. Since $\gamma_\tau$ is in $C^1(\R/2\pi\Z,\R^3)$, due to \thref{remembedd}, we can argue in exactly the same way as in the last section between \eqref{gdiff_st} and \eqref{wedgeinta}. Among these, in analogy to \eqref{intconsts}, we get positive constants $\widetilde{C}_1,\widetilde{C}_2,\widetilde{C}_3>0$ independent of $(s,v,w)\in G$ and $\tau\in[-\tau_0,\tau_0]$
\begin{equation} \label{intconsts2}
\widetilde{C}_1 \leq \babs{\int_0^1 \gamma_\tau'(s+yv) dy}, \widetilde{C}_2 \leq \babs{\int_0^1 \gamma_\tau'(s+yw) dy}, \widetilde{C}_3 \leq \babs{\int_0^1 \gamma_\tau'(s+yw+(1-y)v) dy},
\end{equation}
since we can regard the corresponding integrands as continuous functions on the compact set $[-\tau_0,\tau_0]\times G\times[0,1]$ and the integrals are not zero on $[-\tau_0,\tau_0]\times G$ since $\gamma_\tau$ is simple for every $\tau\in[-\tau_0,\tau_0]$. Observe that for this statement it would be sufficient to have constants that do depend on $\tau$, but hence we need these independent ones later on we directly introduce them here. Now we proceed with \eqref{wedgeinta} to get
\begin{equation} \label{wedgeintdiff}
\begin{split}
&\abs{(\gamma_\tau(s+v)-\gamma_\tau(s)) \wedge (\gamma_\tau(s+w)-\gamma_\tau(s))}\\
=\;& \abs{v} \abs{w}\; \babs{\int_0^1 \gamma_\tau'(s+yv) \, dy \,\wedge \Bigl( \int_0^1 \gamma_\tau'(s+yw) \,dy - \int_0^1 \gamma_\tau'(s+yv) \, dy \Bigr)}\\
\leq\;& \abs{v} \abs{w}\; \kla{\int_0^1 \abs{\gamma_\tau'(s+yv)} \, dy} \kla{\int_0^1 \babs{\gamma_\tau'(s+yw) - \gamma_\tau'(s+yv)} \, dy}.
\end{split}
\end{equation}
Of course we also have $\int_0^1 \abs{\gamma_\tau'(s+yv)} \, dy \leq \norm{\gamma_\tau'}_{C^0} \leq \norm{\gamma_\tau}_{W^{2-\frac{2}{p},p}}$, due to the embedding mentioned in \thref{remembedd}. Therefore, together with the analogies of \eqref{gdiff_sv}, \eqref{gdiff_sw} and \eqref{gdiff_vw} as well as the estimates in \eqref{intconsts2} we finally have
\begin{equation*}
\tilde{c}_{\gamma_\tau}(s,v,w) \leq 2^p\, \widetilde{C}_1^{-p} \widetilde{C}_2^{-p} \widetilde{C}_3^{-p}\, \norm{\gamma_\tau}_{W^{2-\frac{2}{p},p}}^{p+3}\, \frac{\kla{\int_0^1 \babs{\gamma_\tau'(s+yw) - \gamma_\tau'(s+yv)} \, dy}^p}{\abs{w-v}^p},
\end{equation*}
which is a \name{Lebesgue} integrable majorant on $G$, as we will see in a moment. We get that
\begin{align}
& \iiintl_G \frac{\kla{\int_0^1 \babs{\gamma_\tau'(s+yw) - \gamma_\tau'(s+yv)} \, dy}^p}{\abs{w-v}^p} \, dw \, dv \, ds \label{SE011}\\
&\leq 2\,\int_0^{2\pi} \int_{-\pi}^0 \int_0^\pi \int_0^1 \frac{\babs{\gamma_\tau'(s+yw) - \gamma_\tau'(s+yv)}^p}{\abs{w-v}^p} \, dy \, dw \, dv \, ds, \label{SE012}
\end{align}
by applying \name{Jensen}'s inequality and \name{Fubini}'s theorem, using the fact that the integrand is symmetric in $w$ and $v$. Observe that this integrand is not $2\pi$-periodic any longer in $w$ and $v$ but still in $s$. We go ahead to
\begin{align}
& \int_0^{2\pi} \int_{-\pi}^0 \int_0^\pi \int_0^1 \frac{\babs{\gamma_\tau'(s+yw) - \gamma_\tau'(s+yv)}^p}{\abs{w-v}^p} \, dy \, dw \, dv \, ds \notag\\
= \; & \int_0^1 \int_{-\pi}^0 \int_0^\pi \int_0^{2\pi} \frac{\babs{\gamma_\tau'(s+yw) - \gamma_\tau'(s+yv)}^p}{\abs{w-v}^p} \, ds \, dw \, dv \, dy \label{SE001}\\
= \; & \int_0^1 \int_{-\pi}^0 \int_0^\pi \int_{yv}^{yv+2\pi} \frac{\babs{\gamma_\tau'(\tilde{s}+y(w-v)) - \gamma_\tau'(\tilde{s})}^p}{\abs{w-v}^p} \, d\tilde{s} \, dw \, dv \, dy\label{SE002}\\
= \; & \int_0^1 \int_{-\pi}^0 \int_{-v}^{\pi-v} \int_0^{2\pi} \frac{\babs{\gamma_\tau'(s+y\tilde{w}) - \gamma_\tau'(s)}^p}{\abs{\tilde{w}}^p} \, ds \, d\tilde{w} \, dv \, dy \label{SE003}\\
= \; & \int_0^1 \int_0^\pi \int_{-w}^0 \int_0^{2\pi} \frac{\babs{\gamma_\tau'(s+yw) - \gamma_\tau'(s)}^p}{\abs{w}^p} \, ds \, dv \, dw \, dy \notag\\
& \quad + \int_0^1 \int_{-\pi}^0 \int_\pi^{\pi-v} \int_0^{2\pi} \frac{\babs{\gamma_\tau'(s+yw) - \gamma_\tau'(s)}^p}{\abs{w}^p} \, ds \, dw \, dv \, dy, \label{SE004}
\end{align}
where we used \name{Fubini} in \eqref{SE001} and substituted $\tilde{s}:=s+yv$ in \eqref{SE002} and $\tilde{w}:=w-v$ in \eqref{SE003} where we used \thref{propshift} additionally. In \eqref{SE004} we split up the interval $[-v,\pi-v]$ at $\pi$ and apply \thref{lemintch} to the first integral. The second one is bounded from above by
\begin{multline*}
\int_0^1 \int_{-\pi}^0 \int_\pi^{\pi-v} \int_0^{2\pi} \frac{\babs{\gamma_\tau'(s+yw) - \gamma_\tau'(s)}^p}{\abs{w}^p} \, ds \, dw \, dv \, dy\\
\leq \int_0^1 \int_{-\pi}^0 \int_\pi^{2\pi} \int_0^{2\pi} \pi^{-p}\, 2^p\, \norm{\gamma'_\tau}^p_{C^0} \, ds \, dw \, dv \, dy := C,
\end{multline*}
using the embedding in \thref{remembedd} again. We proceed \eqref{SE004} with
\begin{align}
\leq \; & \int_0^1 \int_0^\pi \int_0^{2\pi} \frac{\babs{\gamma_\tau'(s+yw) - \gamma_\tau'(s)}^p}{\abs{w}^{p-1}} \, ds \, dw \, dy + C \label{SE005}\\
\leq \; & \int_0^1 \int_0^{2\pi} \int_{-\pi}^\pi \frac{\babs{\gamma_\tau'(s+yw) - \gamma_\tau'(s)}^p}{\abs{yw}^{p-1}} y^{p-1}\, dw \, ds \, dy + C \label{SE006}\\
= \; & \int_0^1 \int_0^{2\pi} \int_{-y\pi}^{y\pi} \frac{\babs{\gamma_\tau'(s+\tilde{w}) - \gamma_\tau'(s)}^p}{\abs{\tilde{w}}^{p-1}} y^{p-2}\, d\tilde{w} \, ds \, dy + C \label{SE007}\\
\leq \; & \int_0^{2\pi} \int_{-\pi}^\pi \frac{\babs{\gamma_\tau'(s+w) - \gamma_\tau'(s)}^p}{\abs{w}^{1+\sigma p}} \, dw \, ds + C \;=\; [\gamma'_{\tau}]_{W^{\sigma,p}}+C, \qquad \sigma=1-\frac{2}{p}, \label{SE008}
\end{align}
where we took advantage of the fact that the integrand is independent of $v$ in \eqref{SE005}, used \name{Fubini}'s theorem in \eqref{SE006} and substituted $\tilde{w}:=yw$ in \eqref{SE007}. Observe that for $\tau=0$ we have proven in particular that $\tilde{c}_\gamma(s,s+v,s+w)$ is \name{Lebesgue} integrable on $\U$ and therefore, $\M_p(\gamma)$ is finite for $\gamma\in W_{\text{s,r}}^{2-\frac{2}{p},p}(\R/2\pi\Z,\R^3)$.

\emph{Ad statement 2}. In order to get nicer representations of the following formulas we define for $(s,v,w)\in\U$ the variables $\tilde{v}:=s+v$ and $\tilde{w}:=s+w$. We only have to consider $(s,v,w)\in\U\setminus\widetilde{\Ns}$ since $\tilde{c}_{\gamma_\tau}$ is constant on $\widetilde{\Ns}$. The function $(-\tau_0,\tau_0) \rightarrow\R^3, \tau\mapsto\gamma_\tau(s)$, for instance, is linear for fixed $s$. Observe that the function $(-1,1)\rightarrow\R, x\mapsto x^\rho$ is continuous for every $\rho \geq 0$ and it is also well known that $(-1,1)\rightarrow\R, x\mapsto x^\rho$ is differentiable if $\rho \geq 1$. Therefore, $(-\tau_0,\tau_0) \rightarrow\R$
\begin{equation*}
\tau \mapsto \abs{\gamma_\tau(s)-\gamma_\tau(\tilde{v})}^p = \kla{(\gamma^1_\tau(s)-\gamma^1_\tau(\tilde{v}))^2+(\gamma^2_\tau(s)-\gamma^2_\tau(\tilde{v}))^2+(\gamma^3_\tau(s)-\gamma^3_\tau(\tilde{v}))^2}^\frac{p}{2}
\end{equation*}
for instance, is differentiable as a composition of differentiable functions, because $p\geq2$. In addition $\abs{\gamma_\tau(s)-\gamma_\tau(\tilde{v})}^p \neq 0$, since $(s,v,w)\notin \widetilde{\Ns}$ and $\gamma_\tau$ is simple for every $\tau\in[-\tau_0,\tau_0]$. Furthermore, $(0,1)\rightarrow\R, x\mapsto x^\rho$ is differentiable for all $\rho\in\R$. Consequently, the function $(-\tau_0,\tau_0) \rightarrow\R$
\begin{equation*}
\tau\rightarrow\abs{\gamma'_\tau(s)} = \kla{\gamma'^1_\tau(s)^2+\gamma'^2_\tau(s)^2+\gamma'^3_\tau(s)^2}^\frac{1}{2}
\end{equation*}
is differentiable, since $\gamma_\tau$ is regular for all $\tau\in[-\tau_0,\tau_0]$. After all $\tau \mapsto \tilde{c}_{\gamma_\tau}(s,v,w)$ is differentiable on $(-\tau_0,\tau_0)$ for all $(s,v,w)\in\U$. Now we can compute
\begin{align} 
\frac{d}{d\tau} & \kla{\frac{2^p\,\abs{\kla{\gamma_\tau(\tilde{v})-\gamma_\tau(s)} \wedge \kla{\gamma_\tau(\tilde{w})-\gamma_\tau(s)}}^p \; \abs{\gamma_\tau'(s)} \abs{\gamma_\tau'(\tilde{v})} \abs{\gamma_\tau'(\tilde{w})}}{\abs{\gamma_\tau(\tilde{v})-\gamma_\tau(s)}^p\,\abs{\gamma_\tau(s)-\gamma_\tau(\tilde{w})}^p\,\abs{\gamma_\tau(\tilde{w})-\gamma_\tau(\tilde{v})}^p}} \notag\\
=\; & \frac{2^p\,\abs{\gamma_\tau'(s)} \abs{\gamma_\tau'(\tilde{v})} \abs{\gamma_\tau'(\tilde{w})}}{\abs{\gamma_\tau(\tilde{v})-\gamma_\tau(s)}^p\,\abs{\gamma_\tau(s)-\gamma_\tau(\tilde{w})}^p\,\abs{\gamma_\tau(\tilde{w})-\gamma_\tau(\tilde{v})}^p} \Biggl[ \notag\\
& \;\babs{(\gamma_\tau(\tilde{v})-\gamma_\tau(s))\wedge (\gamma_\tau(\tilde{w})-\gamma_\tau(s))}^p \Biggl( \Bigl(\frac{\gamma_\tau'(s) \cdot \Phi'(s)}{\abs{\gamma_\tau'(s)}^2} + \frac{\gamma_\tau'(\tilde{v}) \cdot \Phi'(\tilde{v})}{\abs{\gamma_\tau'(\tilde{v})}^2} \notag\\
& \qquad + \frac{\gamma_\tau'(\tilde{w}) \cdot \Phi'(\tilde{w})}{\abs{\gamma_\tau'(\tilde{w})}^2} \Bigr) -p \Bigl( \frac{(\gamma_\tau(\tilde{v})-\gamma_\tau(s)) \cdot (\Phi(\tilde{v})-\Phi(s))}{\abs{\gamma_\tau(\tilde{v})-\gamma_\tau(s)}^2} \label{diffintegrand}\\
& \qquad + \frac{(\gamma_\tau(s)-\gamma_\tau(\tilde{w})) \cdot (\Phi(s)-\Phi(\tilde{w}))}{\abs{\gamma_\tau(s)-\gamma_\tau(\tilde{w})}^2} + \frac{(\gamma_\tau(\tilde{w})-\gamma_\tau(\tilde{v})) \cdot (\Phi(\tilde{w})-\Phi(\tilde{v}))}{\abs{\gamma_\tau(\tilde{w})-\gamma_\tau(\tilde{v})}^2} \Bigr) \Biggr) \notag\\
& \,+ p\,\babs{(\gamma_\tau(\tilde{v})-\gamma_\tau(s))\wedge (\gamma_\tau(\tilde{w})-\gamma_\tau(s))}^{p-2} \Biggl( \Bigl( (\gamma_\tau(\tilde{v})-\gamma_\tau(s)) \wedge (\gamma_\tau(\tilde{w})-\gamma_\tau(s)) \Bigr) \cdot \notag\\
& \qquad \Bigl( \Bigl( (\Phi(\tilde{v})-\Phi(s)) \wedge (\gamma_\tau(\tilde{w})-\gamma_\tau(s)) \Bigr) + \Bigl( (\gamma_\tau(\tilde{v})-\gamma_\tau(s)) \wedge (\Phi(\tilde{w})-\Phi(s)) \Bigr) \Bigr) \Biggr) \Biggr], \notag
\end{align}
using the product rules \eqref{prodrules}. With the same arguments as above we can conclude that the derivative of the integrand $\tau \mapsto \frac{d}{d\tau} \tilde{c}_{\gamma_\tau}(s,\tilde{v},\tilde{w})$ is continuous on $(-\tau_0,\tau_0)$ as well.

\emph{Ad statement 3}. We take the absolute value of \eqref{diffintegrand} and get the following estimate
\begin{align}
\babs{\frac{d}{d\tau} & \kla{\frac{2^p\,\abs{\kla{\gamma_\tau(\tilde{v})-\gamma_\tau(s)} \wedge \kla{\gamma_\tau(\tilde{w})-\gamma_\tau(s)}}^p \; \abs{\gamma_\tau'(s)} \abs{\gamma_\tau'(\tilde{v})} \abs{\gamma_\tau'(\tilde{w})}}{\abs{\gamma_\tau(\tilde{v})-\gamma_\tau(s)}^p\,\abs{\gamma_\tau(s)-\gamma_\tau(\tilde{w})}^p\,\abs{\gamma_\tau(\tilde{w})-\gamma_\tau(\tilde{v})}^p}}} \label{absdiffintegrand}\\
\leq\; & \frac{2^p\,\abs{\gamma_\tau'(\tilde{v})} \abs{\gamma_\tau'(s)} \abs{\gamma_\tau'(\tilde{w})}}{\abs{\gamma_\tau(s)-\gamma_\tau(s)}^p\,\abs{\gamma_\tau(s)-\gamma_\tau(\tilde{w})}^p\,\abs{\gamma_\tau(\tilde{w})-\gamma_\tau(\tilde{v})}^p} \Biggl[ \notag\\
& \;\babs{(\gamma_\tau(\tilde{v})-\gamma_\tau(s))\wedge (\gamma_\tau(\tilde{w})-\gamma_\tau(s))}^p \Biggl( \Bigl(\frac{\abs{\Phi'(s)}}{\abs{\gamma_\tau'(s)}} + \frac{\abs{\Phi'(\tilde{v})}}{\abs{\gamma_\tau'(\tilde{v})}} + \frac{\abs{\Phi'(\tilde{w})}}{\abs{\gamma_\tau'(\tilde{w})}} \Bigr) \notag\\
& +p \Bigl( \frac{\abs{\Phi(\tilde{v})-\Phi(s)}}{\abs{\gamma_\tau(\tilde{v})-\gamma_\tau(s)}} + \frac{\abs{\Phi(s)-\Phi(\tilde{w})}}{\abs{\gamma_\tau(s)-\gamma_\tau(\tilde{w})}} + \frac{\abs{\Phi(\tilde{w})-\Phi(\tilde{v})}}{\abs{\gamma_\tau(\tilde{w})-\gamma_\tau(\tilde{v})}} \Bigr) \Biggr) \notag\\
& \,+ p\,\babs{(\gamma_\tau(\tilde{v})-\gamma_\tau(s))\wedge (\gamma_\tau(\tilde{w})-\gamma_\tau(s))}^{p-1} \notag\\
& \qquad \babs{\Bigl( (\Phi(\tilde{v})-\Phi(s)) \wedge (\gamma_\tau(\tilde{w})-\gamma_\tau(s)) \Bigr) + \Bigl( (\gamma_\tau(\tilde{v})-\gamma_\tau(s)) \wedge (\Phi(\tilde{w})-\Phi(s)) \Bigr)} \Biggr] \notag
\end{align}
In the equations \eqref{gdiff_sv}, \eqref{gdiff_sw} and \eqref{gdiff_vw} one can replace $\gamma$ by $\Phi$, since $\Phi$ is $C^1(\R/2\pi\Z,\R^3)$ as well. Again we have $\int_0^1 \abs{\gamma_\tau'(s+yv)} \, dy \leq \norm{\gamma_\tau'}_{C^0} \leq \norm{\gamma}_{W^{2-\frac{2}{p},p}}+\norm{\Phi}_{W^{2-\frac{2}{p},p}}$ and $\int_0^1 \abs{\Phi'(s+yv)} \, dy \leq \norm{\Phi'}_{C^0} \leq \norm{\gamma}_{W^{2-\frac{2}{p},p}}+\norm{\Phi}_{W^{2-\frac{2}{p},p}}$, due to the embedding mentioned in \thref{remembedd} and since $\abs{\tau}<1$. Therefore, we can estimate exemplarily in the following way
\begin{equation} \label{SE100}
\frac{\abs{\Phi(\tilde{v})-\Phi(s)}}{\abs{\gamma_\tau(\tilde{v})-\gamma_\tau(s)}} \leq \frac{\abs{v}\norm{\Phi'}_{C^0}}{\abs{v}\babs{\int_0^1 \gamma_\tau'(s+yv) \, dy}} \leq \frac{\norm{\gamma}_{W^{2-\frac{2}{p},p}}+\norm{\Phi}_{W^{2-\frac{2}{p},p}}}{\widetilde{C}_1}.
\end{equation}
As in \eqref{wedgeintdiff} we get
\begin{equation} \label{wedgeint2}
\begin{split}
& \babs{\Bigl((\Phi(\tilde{v})-\Phi(s)) \wedge (\gamma_\tau(\tilde{w})-\gamma_\tau(s))\Bigr)+\Bigl((\gamma_\tau(\tilde{v})-\gamma_\tau(s)) \wedge (\Phi(\tilde{w})-\Phi(s))\Bigr)}\\
& \;= \abs{v} \abs{w}\; \Bigl\lvert \Bigl(\int_0^1 \Phi'(s+yv) \, dy \wedge \int_0^1 \gamma_\tau'(s+yw) \, dy\Bigr)+\\
& \; \qquad \qquad \; \Bigl(\int_0^1 \gamma_\tau'(s+yv) \, dy \wedge \int_0^1 \Phi'(s+yw) \, dy \Bigr) \Bigr\rvert \\
& \;\leq \abs{v} \abs{w}\; \Bigl(\Bigl(\int_0^1 \abs{\Phi'(s+yv)} \, dy\Bigr) \Bigl(\int_0^1 \abs{\gamma_\tau'(s+yw)-\gamma_\tau'(s+yv)} \, dy\Bigr)+\\
& \; \qquad \qquad \;\; \Bigl(\int_0^1 \abs{\gamma_\tau'(s+yv)} \, dy\Bigr) \Bigl(\int_0^1 \abs{\Phi'(s+yw)-\Phi'(s+yv)} \, dy\Bigr)\Bigr)\\
& \;\leq 2\; \abs{v} \abs{w}\; (\norm{\gamma}_{W^{2-\frac{2}{p},p}}+\norm{\Phi}_{W^{2-\frac{2}{p},p}})\,\Bigl(\int_0^1 \abs{\gamma'(s+yw)-\gamma'(s+yv)} \, dy\,+\\
& \; \qquad \qquad \qquad \qquad \qquad \qquad \qquad \qquad \int_0^1 \abs{\Phi'(s+yw)-\Phi'(s+yv)} \, dy\Bigr),
\end{split}
\end{equation}
where we used the anti-symmetry of the wedge product. Moreover, the last line of \eqref{wedgeint2} divided by $2$ is a upper bound for \eqref{wedgeintdiff}, which we get analogously. For the sake of clarity we define $\widehat{C} := \norm{\gamma}_{C^{1,\alpha}}+\norm{\Phi}_{C^{1,\alpha}}$. After all, using the basic estimate \eqref{estabssum} we get for \eqref{absdiffintegrand}
\begin{multline*}
\leq 2^p\,\tilde{C}_1^{-p}\tilde{C}_2^{-p}\tilde{C}_3^{-p}\;\widehat{C}^3\; \Biggl( 2^{p-1}\widehat{C}^{p} \Bigl( 3 \cdot 2\,C_0^{-2}\;\widehat{C} + p\;\widehat{C} \kla{\widetilde{C}_1^{-1}+\widetilde{C}_2^{-1}+\widetilde{C}_3^{-1}} \Bigr) +p\,2^p\;\widehat{C}^p \Biggr)\\
\Biggl( \frac{\kla{\int_0^1 \abs{\gamma'(s+yw)-\gamma'(s+yv)} \, dy}^p}{\abs{w-v}^p}\,+\frac{\kla{\int_0^1 \abs{\Phi'(s+yw)-\Phi'(s+yv)} \, dy}^p}{\abs{w-v}^p} \Biggr) =: h(s,v,w).
\end{multline*}
This is the required \name{Lebesgue} integrable majorant on $G$, which can be seen by just repeating the steps from \eqref{SE011} up to \eqref{SE008} again for $\gamma$ and $\Phi$ instead of $\gamma_\tau$. This completes the proof.
\end{proof}
\begin{rem} \mylabel{remsmallp}
In the case $p\in[2,\infty)$ (in particular the case $p=3$) we can prove the existence of the first variation at least for functions in $C^{1,\alpha}(\R/2\pi\Z,\R^3)$ with $\alpha\in (1-\frac{2}{p},1]$ in an analogous way. For $p<2$ this method fails since differentiating the integrand in \eqref{diffintegrand} we get a factor
\begin{equation*}
\babs{(\gamma_\tau(\tilde{v})-\gamma_\tau(s))\wedge (\gamma_\tau(\tilde{w})-\gamma_\tau(s))}^{p-2},
\end{equation*}
which cannot be estimated appropriately if the exponent is negative.
\end{rem}
\begin{proof}
Let $p\in[2,\infty)$ and $1-\frac{2}{p}<\alpha\leq 1$. We use the very same argumentation as in the proof before, except for the norm used there. Here we estimate $\norm{\gamma'_\tau}_{C^0} \leq \norm{\gamma_\tau}_{C^{1,\alpha}}$ and $\norm{\gamma'_\tau}_{C^0},\norm{\Phi'}_{C^0} \leq \norm{\gamma}_{C^{1,\alpha}}+\norm{\Phi}_{C^{1,\alpha}}$. Moreover, we exchange the estimate from \eqref{SE008} by
\begin{align*}
\int_0^{2\pi} \int_{-\pi}^\pi \frac{\babs{\gamma_\tau'(s+w) - \gamma_\tau'(s)}^p}{\abs{w}^{p-1}} \, dw \, ds &\leq \int_0^{2\pi} \int_{-\pi}^\pi \frac{[\gamma'_\tau]^p_{C^{0,\alpha}}\;\abs{w}^{\alpha p}}{\abs{w}^{p-1}} \, dw \, ds\\
&\leq 2\pi\,\norm{\gamma_\tau}^p_{C^{1,\alpha}}\; \int_{-\pi}^\pi \abs{w}^{-(p-1-\alpha p)} \, dw,
\end{align*}
which is finite if $(p-1-\alpha p)<1\;\Leftrightarrow\;1-\frac{2}{p}<\alpha$ in analogy to \thref{remembedd}.
\end{proof}
After all we get for the first Variation of $\M_p$, $p\in[2,\infty)$
\begin{multline} \label{long1var}
\frac{d}{d\tau} \M_p(\gamma_\tau) \eins{\tau=0} =\\
\int_0^{2\pi} \int_0^{2\pi} \int_0^{2\pi} \frac{2^p \abs{\gamma'(s)} \abs{\gamma'(t)} \abs{\gamma'(\sigma)}}{\abs{\gamma(s)-\gamma(t)}^p \abs{\gamma(t)-\gamma(\sigma)}^p \abs{\gamma(s)-\gamma(\sigma)}^p} \Biggl[ \babs{(\gamma(t)-\gamma(s))\wedge (\gamma(\sigma)-\gamma(s))}^p\\
\Biggl( \Bigl( \frac{\gamma'(s) \cdot \Phi'(s)}{\abs{\gamma'(s)}^2} + \frac{\gamma'(t) \cdot \Phi'(t)}{\abs{\gamma'(t)}^2} + \frac{\gamma'(\sigma) \cdot \Phi'(\sigma)}{\abs{\gamma'(\sigma)}^2} \Bigr) -p \Bigl( \frac{(\gamma(s)-\gamma(t)) \cdot (\Phi(s)-\Phi(t))}{\abs{\gamma(s)-\gamma(t)}^2} + \\
\frac{(\gamma(t)-\gamma(\sigma)) \cdot (\Phi(t)-\Phi(\sigma))}{\abs{\gamma(t)-\gamma(\sigma)}^2} + \frac{(\gamma(s)-\gamma(\sigma)) \cdot (\Phi(s)-\Phi(\sigma))}{\abs{\gamma(s)-\gamma(\sigma)}^2} \Bigr) \Biggr)\\
+p \, \babs{(\gamma(t)-\gamma(s))\wedge (\gamma(\sigma)-\gamma(s))}^{p-2} \Bigl( (\gamma(t)-\gamma(s)) \wedge (\gamma(\sigma)-\gamma(s)) \Bigr) \cdot\\
\Bigl( (\Phi(t)-\Phi(s)) \wedge (\gamma(\sigma)-\gamma(s)) + (\gamma(t)-\gamma(s)) \wedge (\Phi(\sigma)-\Phi(s)) \Bigr) \Biggr] \, ds \, dt \, d\sigma.
\end{multline}

If we assume again, that $\gamma$ and $\Phi$ are actually $C^2$, we obtain as in \eqref{wedgeintb}
\begin{align*}
& \Bigl((\Phi(\tilde{v})-\Phi(s)) \wedge (\gamma_\tau(\tilde{w})-\gamma_\tau(s))\Bigr)+\Bigl((\gamma_\tau(\tilde{v})-\gamma_\tau(s)) \wedge (\Phi(\tilde{w})-\Phi(s))\Bigr)\\
=&\; v\, w\; \Biggl( \Bigl(\int_0^1 \Phi'(s+yv) \, dy \,\wedge\\
& \Bigl( \int_0^1 \gamma_\tau'(s+yv) \, dy + \int_0^1 \kla{w-v} \int_0^1 y\,\gamma_\tau''(s+y(\tilde{y}w+(1-\tilde{y})v)) \,d\tilde{y}\,dy \Bigr) \Bigr)\\
&\;\quad + \Bigl( \int_0^1 \gamma_\tau'(s+yv) \, dy \,\wedge\\
& \Bigl( \int_0^1 \Phi'(s+yv) \, dy + \int_0^1 \kla{w-v} \int_0^1 y\,\Phi''(s+y(\tilde{y}w+(1-\tilde{y})v)) \,d\tilde{y}\,dy \Bigr) \Bigr) \Biggr)\\
=&\; v\,w \kla{w-v}\; \Biggl( \Bigl( \int_0^1 \Phi'(s+yv) \, dy \wedge \int_0^1 \int_0^1 y\, \gamma_\tau''(s+y(\tilde{y}w+(1-\tilde{y})v)) \, d\tilde{y} \, dy \Bigr)\\
& \qquad \qquad \qquad +\Bigl( \int_0^1 \gamma_\tau'(s+yv) \, dy \wedge \int_0^1 \int_0^1 y\, \Phi''(s+y(\tilde{y}w+(1-\tilde{y})v)) \, d\tilde{y} \, dy \Bigr) \Biggr),
\end{align*}
using the anti-symmetric property of the wedge product. Together with analogous operations on $\kla{(\gamma(\tilde{v})-\gamma(s))\wedge(\gamma(\tilde{w})-\gamma(s))}$ (cf. \eqref{wedgeintb}) we get the needed factors $\abs{v}^2$, $\abs{w}^2$ and $\abs{w-v}^2$. Finally, we gain for the full derivative of the integrand
\begin{align}
& \frac{d}{d\tau}\eins{\tau=0} \tilde{c}_{\gamma_\tau} (s,\tilde{v},\tilde{w}) = \notag\\
& \frac{2^p \, \abs{\gamma'(s)} \abs{\gamma'(\tilde{v})} \abs{\gamma'(\tilde{w})}}{\babs{\int_0^1 \gamma'(s+yv) \, dy}^p \babs{\int_0^1 \gamma'(s+yw) \, dy}^p \babs{\int_0^1 \gamma'(s+yw+(1-y)v)) \, dy}^p} \Biggl[ \notag\\
& \qquad \babs{\int_0^1 \gamma'(s+yv) \, dy \wedge \int_0^1 \int_0^1 y\,\gamma''(s+y(\tilde{y}w+(1-\tilde{y})v)) \,d\tilde{y} \,dy}^p \quad \Biggl( \notag\\
& \qquad \qquad \Bigl( \frac{\gamma'(s) \cdot \Phi'(s)}{\abs{\gamma'(s)}^2} + \frac{\gamma'(\tilde{v}) \cdot \Phi'(\tilde{v})}{\abs{\gamma'(\tilde{v})}^2} + \frac{\gamma'(\tilde{w}) \cdot \Phi'(\tilde{w})}{\abs{\gamma'(\tilde{w})}^2} \Bigr) \notag\\
& -p\, \Bigl( \frac{\Bigl( \int_0^1 \gamma'(s+yv) \, dy \Bigr) \cdot \Bigl( \int_0^1 \Phi'(s+yv) \, dy \Bigr)}{\babs{\int_0^1 \gamma'(s+yv) \, dy}^2} + \frac{\Bigl( \int_0^1 \gamma'(s+yw) \, dy \Bigr) \cdot \Bigl( \int_0^1 \Phi'(s+yw) \, dy \Bigr)}{\babs{\int_0^1 \gamma'(s+yw) \, dy}^2} \notag\\
& \qquad \qquad + \frac{\Bigl( \int_0^1 \gamma'(s+yw+(1-y)v)) \, dy \Bigr) \cdot \Bigl( \int_0^1 \Phi'(s+yw+(1-y)v)) \, dy \Bigr)}{\babs{\int_0^1 \gamma'(s+yw+(1-y)v)) \, dy}^2} \Bigr) \Biggr) \notag\displaybreak[0]\\
& +p \,\babs{\int_0^1 \gamma'(s+yv) \, dy \wedge \int_0^1 \int_0^1 y\,\gamma''(s+y(\tilde{y}w+(1-\tilde{y})v)) \,d\tilde{y}\,dy}^{p-2} \notag\\
& \qquad \qquad \Bigl( \int_0^1 \gamma'(s+yv) \, dy \wedge \int_0^1 \int_0^1 y\,\gamma''(s+y(\tilde{y}w+(1-\tilde{y})v)) \,d\tilde{y}\,dy \Bigr) \, \cdot \notag\\
& \qquad \Biggl( \int_0^1 \Phi'(s+yv) \, dy \wedge \int_0^1 \int_0^1 y\,\gamma''(s+y(\tilde{y}w+(1-\tilde{y})v)) \,d\tilde{y}\,dy \,+ \notag\\
& \qquad \quad \int_0^1 \gamma'(s+yv) \, dy \wedge \int_0^1 \int_0^1 y\,\Phi''(s+y(\tilde{y}w+(1-\tilde{y})v)) \,d\tilde{y}\,dy \Biggr) \Biggr]. \label{M2.2}
\end{align}

Again applying \name{Lebesgue}'s dominated convergence theorem, just like for $c_\gamma$ in \eqref{cgammaint}, we gain that there exists a continuous extension of the integrand on $(s,s+v,s+w)\in G$ and consequently on $(s,t,\sigma)\in [0,2\pi)^3$.

In order to derive a simplified version of \eqref{long1var}, we use \eqref{calScalarCross} to obtain
\begin{multline} \label{scalarcrossshort}
\Bigl( (\gamma(t)-\gamma(s)) \wedge (\gamma(\sigma)-\gamma(s)) \Bigr) \cdot \Bigl( (\Phi(t)-\Phi(s)) \wedge (\gamma(\sigma)-\gamma(s)) + (\gamma(t)-\gamma(s)) \wedge (\Phi(\sigma)-\Phi(s)) \Bigr)\\
= \abs{\gamma(\sigma)-\gamma(s)}^2 \, (\gamma(t)-\gamma(s)) \cdot (\Phi(t)-\Phi(s)) + \abs{\gamma(t)-\gamma(s)}^2 \, (\gamma(\sigma)-\gamma(s)) \cdot (\Phi(\sigma)-\Phi(s))\\
-(\gamma(t)-\gamma(s)) \cdot (\gamma(\sigma)-\gamma(s)) \Bigl( (\gamma(\sigma)-\gamma(s)) \cdot (\Phi(t)-\Phi(s)) + (\gamma(t)-\gamma(s)) \cdot (\Phi(\sigma)-\Phi(s)) \Bigr).
\end{multline}
Moreover, we want to apply the following simplification method. Let $F: [0,2\pi]^3\rightarrow\R^3$ be a function, which is \name{Lebesgue} integrable on $[0,2\pi]^3$ and we consider exemplarily the bounded function $x\mapsto\frac{\gamma'(x)\cdot\Phi'(x)}{\abs{\gamma'(x)}^2}$. Therefore,
\begin{align} 
& \int_0^{2\pi} \int_0^{2\pi} \int_0^{2\pi} F(s,t,\sigma) \Bigl( \frac{\gamma'(s) \cdot \Phi'(s)}{\abs{\gamma'(s)}^2} + \frac{\gamma'(t) \cdot \Phi'(t)}{\abs{\gamma'(t)}^2} + \frac{\gamma'(\sigma) \cdot \Phi'(\sigma)}{\abs{\gamma'(\sigma)}^2} \Bigr) \, ds \, dt \, d\sigma \notag\\
= & \int_0^{2\pi} \int_0^{2\pi} \int_0^{2\pi} F(s,t,\sigma) \frac{\gamma'(s) \cdot \Phi'(s)}{\abs{\gamma'(s)}^2} \, ds \, dt \, d\sigma + \int_0^{2\pi} \int_0^{2\pi} \int_0^{2\pi} F(t,s,\sigma) \frac{\gamma'(s) \cdot \Phi'(s)}{\abs{\gamma'(s)}^2} \, dt \, ds \, d\sigma \notag\\
& \qquad \qquad \qquad \qquad \qquad \qquad + \int_0^{2\pi} \int_0^{2\pi} \int_0^{2\pi} F(\sigma,t,s)\frac{\gamma'(s) \cdot \Phi'(s)}{\abs{\gamma'(s)}^2} \, d\sigma \, dt \, ds \notag\\
= & \int_0^{2\pi} \int_0^{2\pi} \int_0^{2\pi} \kla{F(s,t,\sigma)+F(t,s,\sigma)+F(\sigma,t,s)} \frac{\gamma'(s) \cdot \Phi'(s)}{\abs{\gamma'(s)}^2} \, ds \, dt \, d\sigma, \label{fubinitrick}\\
\intertext{where we interchange the variable names in the first step and apply \name{Fubini}'s theorem (see \thref{remproofbw}) in the second. If $F$ is additionally symmetric in its arguments we have}
= & \int_0^{2\pi} \int_0^{2\pi} \int_0^{2\pi} 3\, F(s,t,\sigma) \frac{\gamma'(s) \cdot \Phi'(s)}{\abs{\gamma'(s)}^2} \, ds \, dt \, d\sigma. \label{fubinishort}
\end{align}
Since this is the case for $R(X,Y,Z)^{-1}$ (see \eqref{oneoverr}), we gain
\begin{multline} \label{simp1var}
\int_0^{2\pi} \int_0^{2\pi} \int_0^{2\pi} \frac{2^p \abs{\gamma'(s)} \abs{\gamma'(t)} \abs{\gamma'(\sigma)}}{\abs{\gamma(s)-\gamma(t)}^p \abs{\gamma(t)-\gamma(\sigma)}^p \abs{\gamma(s)-\gamma(\sigma)}^p} \Biggl[ \babs{(\gamma(t)-\gamma(s))\wedge (\gamma(\sigma)-\gamma(s))}^p\\
\Biggl( 3 \frac{\gamma'(s) \cdot \Phi'(s)}{\abs{\gamma'(s)}^2} -3p \frac{(\gamma(s)-\gamma(t)) \cdot (\Phi(s)-\Phi(t))}{\abs{\gamma(s)-\gamma(t)}^2} \Biggr)\\
+p \babs{(\gamma(t)-\gamma(s))\wedge (\gamma(\sigma)-\gamma(s))}^{p-2} \Biggl( \abs{\gamma(\sigma)-\gamma(s)}^2 \, (\gamma(t)-\gamma(s)) \cdot (\Phi(t)-\Phi(s))\\
+ \abs{\gamma(t)-\gamma(s)}^2 \, (\gamma(\sigma)-\gamma(s)) \cdot (\Phi(\sigma)-\Phi(s)) - (\gamma(t)-\gamma(s)) \cdot (\gamma(\sigma)-\gamma(s))\\
\Bigl( (\gamma(\sigma)-\gamma(s)) \cdot (\Phi(t)-\Phi(s)) + (\gamma(t)-\gamma(s)) \cdot (\Phi(\sigma)-\Phi(s)) \Bigr) \Biggr) \Biggr] \, ds \, dt \, d\sigma.
\end{multline}
Observe that \name{Fubini}'s theorem must be used very carefully here. It is, for instance, not allowed to ``simplify'' the expression $\frac{(\gamma(s)-\gamma(t)) \cdot (\Phi(s)-\Phi(t))}{\abs{\gamma(s)-\gamma(t)}^2}$ to $2 \frac{(\gamma(s)-\gamma(t)) \cdot \Phi(s)}{\abs{\gamma(s)-\gamma(t)}^2}$ because the second one is singular!

Later on for the discretization in Chapter~\ref{discspace} we have to compute the corresponding values of the integrand when two or more points coincide. As already mentioned in the proof, with $c_\gamma$ also the integrand of the first variation \eqref{long1var} is symmetric in its arguments. Therefore, we are again able to apply the method used in \eqref{symmconv} here. Observe that the symmetry is broken in \eqref{simp1var}. Due to this method, for the case that two points coincide and the third one is distinct, it is sufficient to determine the integrand for $(s,t,t)$ and multiply the result by $3$. Therefore, we take the integrand of \eqref{M2.2} and compute the case $w=v$, in analogy to \eqref{csvv}. After applying \eqref{calScalarCross} again we obtain
\begin{multline} \label{stt1var}
\frac{2^p \abs{\gamma'(s)}}{\abs{\gamma(s)-\gamma(t)}^{2p} \abs{\gamma'(t)}^{p-2}} \Biggl[\\
\babs{(\gamma(t)-\gamma(s))\wedge \gamma'(t)}^p \Biggl( \frac{\gamma'(s) \cdot \Phi'(s)}{\abs{\gamma'(s)}^2} +(2-p) \frac{\gamma'(t) \cdot \Phi'(t)}{\abs{\gamma'(t)}^2} -2p \frac{(\gamma(s)-\gamma(t)) \cdot (\Phi(s)-\Phi(t))}{\abs{\gamma(s)-\gamma(t)}^2} \Biggr)\\
+p \babs{(\gamma(t)-\gamma(s))\wedge \gamma'(t)}^{p-2} \Biggl( \abs{\gamma'(t)}^2 \, (\gamma(t)-\gamma(s)) \cdot (\Phi(t)-\Phi(s)) + \abs{\gamma(t)-\gamma(s)}^2 \, \gamma'(t) \cdot \Phi'(t)\\
-(\gamma(t)-\gamma(s)) \cdot \gamma'(t) \Bigl( \gamma'(t) \cdot (\Phi(t)-\Phi(s)) + (\gamma(t)-\gamma(s)) \cdot \Phi'(t) \Bigr) \Biggr) \Biggr].
\end{multline}

For the very last case that all three points coincide in the point $s$ we obtain analogously from \eqref{M2.2}, by using \eqref{calScalarCross} once more
\begin{multline} \label{sss1var}
\abs{\gamma'(s)}^{3-3p} \Biggl[ (3-3p) \abs{\gamma'(s) \wedge \gamma''(s)}^p \frac{\gamma'(s) \cdot \Phi'(s)}{\abs{\gamma'(s)}^2} + p \abs{\gamma'(s) \wedge \gamma''(s)}^{p-2} \Bigl( \abs{\gamma'(s)}^2 \gamma''(s) \cdot \Phi''(s)\\
+ \abs{\gamma''(s)}^2 \gamma'(s) \cdot \Phi'(s) - (\gamma'(s) \cdot \gamma''(s)) (\gamma'(s) \cdot \Phi''(s) + \gamma''(s) \cdot \Phi'(s)) \Bigr) \Biggr].
\end{multline}

Due to the scaling behaviour of $\M_p$ \eqref{scaleMp} it would be possible to decrease the energy by scaling up the knot. To avoid this we multiply the energy by a specific power of its length. By taking the whole energy to the power $\frac{1}{p}$ we are able to compare our results with ideal knots:
\begin{equation*}
\E_p (\gamma) := \sqrt[p]{\M_p(\gamma)} \diagup \Le(\gamma)^{\frac{3-p}{p}} = \sqrt[p]{\M_p(\gamma)} \, \Le(\gamma)^{\frac{p-3}{p}} \xrightarrow{p \rightarrow \infty} \frac{\Le(\gamma)}{\Delta[\gamma]}.
\end{equation*}
The last expression converges, because due to \cite[U1.4]{alt}
\begin{equation*}
\kla{\;\intl_{[0,2\pi]^3} c_\gamma(z_1,z_2,z_3) \, d\lambda^3(z)}^{\frac{1}{p}} = \norm{c_\gamma}_{L^p([0,2\pi]^3)} \xrightarrow{p \rightarrow \infty} \norm{c_\gamma}_{L^\infty([0,2\pi]^3)} = \frac{1}{\Delta[\gamma]},
\end{equation*}
where $\lambda^3$ is the \name{Lebesgue} measure on $\R^3$. Moreover, the functional $\M_p$ $\Gamma$-converges to the functional $\frac{1}{\Delta}$, what one will find in \name{S. Scholtes}' forthcoming PhD-thesis. Furthermore, we have
\begin{equation} \label{estm_mprl}
\E_p(\gamma) \leq \kla{\int_0^{2\pi} \int_0^{2\pi} \int_0^{2\pi} \frac{\abs{\gamma'(s)}\abs{\gamma'(t)}\abs{\gamma'(\sigma)}}{\Delta[\gamma]^p} \, ds \, dt \, d\sigma}^{\frac{1}{p}} \, \Le(\gamma)^{\frac{p-3}{p}} = \frac{\Le(\gamma)}{\Delta[\gamma]}.
\end{equation}
We can easily compute the first variation of $\E_p$
\begin{align}
\delta\E_p(\gamma,\Phi) & \; = \; \frac{\Le(\gamma)^{\frac{p-3}{p}} \M_p(\gamma)^{\frac{1-p}{p}}}{p} \, \delta\M_p(\gamma,\Phi) + \kla{\frac{p-3}{p} \M_p(\gamma)^{\frac{1}{p}} \Le(\gamma)^{-\frac{3}{p}}} \, \delta \Le(\gamma,\Phi) \notag\\
& \; = \; \sqrt[p]{\frac{\M_p(\gamma)}{\Le(\gamma)^3}} \Bigl( \frac{\Le(\gamma)}{p \M_p(\gamma)} \, \delta\M_p(\gamma,\Phi) + \frac{p-3}{p} \, \delta \Le(\gamma,\Phi) \Bigr), \label{generalFVar}
\end{align}
where one can directly verify that $\delta \Le(\gamma,\Phi) = \int_0^{2\pi} \frac{\gamma'(s) \cdot \Phi'(s)}{\abs{\gamma'(s)}} \, ds$.

\begin{rem} \mylabel{remtestcinfty}
By a standard convolution argument (cf. \cite[Chapter~7.2]{GT} and \cite[Theorem~2.4]{hitchhiker}) we get that it is sufficient to show that $\delta \M_p(\gamma,\Phi)=0$ for all $\Phi\in C^\infty(\R/2\pi\Z,\R^3)$ in order to get that $\gamma$ is a stationary point of $\M_p$.
\end{rem}

\begin{lem} \mylabel{regfrac}
Let $k\in\N$ and $\gamma,\Phi \in C^k(\R/2\pi\Z,\R^3)$. Moreover, $\gamma$ be simple and regular, then
\begin{equation*}
(s,t) \mapsto\quad \frac{(\gamma(s)-\gamma(t))\cdot(\Phi(s)-\Phi(t))}{\abs{\gamma(s)-\gamma(t)}^2}
\end{equation*}
is in $C^{k-1}([0,2\pi]\times(0,2\pi))$. In particular for $\gamma,\Phi\in C^\infty(\R/2\pi\Z,\R^3)$ this quotient is in $C^\infty([0,2\pi]\times(0,2\pi))$ as well.
\end{lem}
\begin{proof}
We have seen before that
\begin{equation}
\frac{(\gamma(s)-\gamma(t))\cdot(\Phi(s)-\Phi(t))}{\abs{\gamma(s)-\gamma(t)}^2} = \frac{\int_0^1 \gamma'(ys+(1-y)t) \, dy \cdot \int_0^1 \Phi'(ys+(1-y)t) \, dy}{\abs{\int_0^1 \gamma'(ys+(1-y)t) \, dy}^2}.
\end{equation}
Let $a<b$. For a continuously differentiable function $f:[a,b] \rightarrow \R^3$, with $\abs{f(x)}>0$ for all $x\in[a,b]$ we have for $k\in\N$
\begin{equation*}
\frac{d}{dx}\,\abs{f(x)}^{-k} = \kla{-k} \abs{f(x)}^{-(k+2)} \kla{f(x)\cdot \frac{d}{dx}\,f(x)}.
\end{equation*}
For $(s,t)\in [0,2\pi]\times(0,2\pi)$ the integral $\int_0^1 \gamma'(ys+(1-y)t) \, dy$ could be such a function $f$ with respect to $s$ or $t$.
Obviously we have $\frac{d}{ds} f(s)=\int_0^1 y\,\gamma''(ys+(1-y)t) \, dy$ and analogous expressions for $t$ and $\Phi$. Now the product rule implies the claim.
\end{proof}

It is very likely that a circle is the global minimizer of $\M_p$ among all closed and simple curves with prescribed length or equivalently (cf. \thref{remnostpts}(b)) of $\E_p$ among all closed and simple curves, but it has not been proven yet. However, we present a proof that a circle is at least a stationary point of this energy.
\begin{thm} \mylabel{thmcirclesp}
For $p\geq 2$ a circle is a stationary point of $\E_p$. In particular it is a stationary point of $\M_3$.
\end{thm}
\begin{proof}
Without loss of generality we consider the unit circle $\gamma(x) = \left(\begin{smallmatrix}
\cos(x)\\
\sin(x)\\
0
\end{smallmatrix} \right)$. Consequently, we have $\gamma'(x) = \left(\begin{smallmatrix}
-\sin(x)\\
\cos(x)\\
0
\end{smallmatrix} \right)$,  $\gamma'' \equiv -\gamma$,  $\norm{\gamma} \equiv 1$, $\norm{\gamma'} \equiv 1$ and $\gamma\cdot \gamma' \equiv 0$ as well as $c_\gamma \equiv 1$ (cf. \thref{exmpthickone}). Let $s\in [0,2\pi]$. Then for $x$ with $(x-s)\notin 2\pi\Z$ the vectors $\menge{\gamma(x)-\gamma(s), \gamma'(x)-\gamma'(s), e_3}=:B$ form an orthogonal basis of $\R^3$, because $(\gamma(x)-\gamma(s)) \cdot (\gamma'(x)-\gamma'(s))=-\gamma(x)\cdot\gamma'(s)-\gamma(s)\cdot\gamma'(x)=\cos(x)\sin(s)-\sin(x)\cos(s)+\cos(s)\sin(x)-\sin(s)\cos(x)=0$ and $\abs{\gamma(x)-\gamma(s)} = \abs{\gamma'(x)-\gamma'(s)} \neq 0$. The latter is true since
\begin{equation} \label{pcmquad}
\begin{split}
\abs{\gamma(x)-\gamma(s)}^2 &= 4 \sin(\frac{x+s}{2})^2 \sin(\frac{x-s}{2})^2 + 4 \cos(\frac{x+s}{2})^2 \sin(\frac{x-s}{2})^2\\
&= 4 \sin(\frac{x-s}{2})^2,
\end{split}
\end{equation}
where we used the trigonometric identities \eqref{trigdiffs} and \eqref{trigdiffc}. This will be done implicitly several times during the proof. Additionally, we needed \eqref{trigssum} here. Furthermore, observe that $\Ns\cap[0,2\pi]^3$ is a null-set (see \eqref{nullset1r}).\\
Let $\Phi \in C^{\infty}(\R/2\pi\Z, \R^3)$. We want to compute the first variation of $\M_p$ at $\gamma$ in direction $\Phi$. At first we express $\Phi(x)-\Phi(s)$ with respect to the basis $B$
\begin{equation} \label{pcmdiff}
\Phi(x) -\Phi(s) =: \alpha_1(x,s) (\gamma(x)-\gamma(s)) + \alpha_2(x,s) (\gamma'(x)-\gamma'(s)) + \alpha_3(x,s)\,e_3.
\end{equation}
Equivalently we can say
\begin{equation} \label{pcmalpha}
\begin{split}
\alpha_1(x,s) &= \frac{(\Phi(x)-\Phi(s)) \cdot (\gamma(x)-\gamma(s))}{\abs{\gamma(x)-\gamma(s)}^2},\\
\alpha_2(x,s) &= \frac{(\Phi(x)-\Phi(s)) \cdot (\gamma'(x)-\gamma'(s))}{\abs{\gamma'(x)-\gamma'(s)}^2},\\
\alpha_3(x,s) &= \Phi_3(x)-\Phi_3(s).
\end{split}
\end{equation}
We know from \thref{regfrac} that these three functions are at least $C^1$. Since $\gamma\in C^\infty$ the first variation exists due to \thref{remsmallp} also for $p\in [2,\infty)$ and in analogy to \eqref{simp1var} we obtain
\begin{multline*}
\delta\M_p(\gamma,\Phi) = \int_0^{2\pi} \int_0^{2\pi} \int_0^{2\pi} c_\gamma^p(s,t,\sigma) \abs{\gamma'(s)} \abs{\gamma'(t)} \abs{\gamma'(\sigma)} \Biggl( 3 \frac{\gamma'(t) \cdot \Phi'(t)}{\abs{\gamma'(t)}^2}\\
-2p \frac{(\gamma(s)-\gamma(t)) \cdot (\Phi(s)-\Phi(t))}{\abs{\gamma(s)-\gamma(t)}^2} -p \, \frac{(\gamma(s)-\gamma(\sigma)) \cdot (\Phi(s)-\Phi(\sigma))}{\abs{\gamma(s)-\gamma(\sigma)}^2} \Biggr) +\\
+4 p \frac{c_\gamma^{p-2}(s,t,\sigma)\; \abs{\gamma'(s)} \abs{\gamma'(t)} \abs{\gamma'(\sigma)}}{\abs{\gamma(s)-\gamma(t)}^2 \abs{\gamma(t)-\gamma(\sigma)}^2 \abs{\gamma(s)-\gamma(\sigma)}^2} \Biggl( \abs{\gamma(\sigma)-\gamma(s)}^2 \, (\gamma(t)-\gamma(s)) \cdot (\Phi(t)-\Phi(s))\\
+ \abs{\gamma(t)-\gamma(s)}^2 \, (\gamma(\sigma)-\gamma(s)) \cdot (\Phi(\sigma)-\Phi(s)) - (\gamma(t)-\gamma(s)) \cdot (\gamma(\sigma)-\gamma(s))\\
\Bigl( (\gamma(\sigma)-\gamma(s)) \cdot (\Phi(t)-\Phi(s)) + (\gamma(t)-\gamma(s)) \cdot (\Phi(\sigma)-\Phi(s)) \Bigr) \Biggr) \, ds \, dt \, d\sigma.
\end{multline*}
From the representation \eqref{pcmdiff} we conclude
\begin{equation} \label{pcmmultp}
\begin{split}
(\gamma(\sigma)-\gamma(s))\cdot (\Phi(t)-\Phi(s)) &= \alpha_1(t,s) (\gamma(\sigma)-\gamma(s))\cdot (\gamma(t)-\gamma(s))\\
&+\alpha_2(t,s) (\gamma(\sigma)-\gamma(s))\cdot (\gamma'(t)-\gamma'(s)).
\end{split}
\end{equation}
By using \eqref{pcmalpha}, \eqref{pcmmultp} and the properties of the unit circle mentioned above we obtain
\begin{multline*}
= \int_0^{2\pi} \int_0^{2\pi} \int_0^{2\pi} 3 \gamma'(t) \cdot \Phi'(t) -2p \alpha_1(t,s) -p \alpha_1(\sigma,s) +\\
+4 p \frac{(\alpha_1(t,s)+\alpha_1(\sigma,s)) \Bigl( (\abs{\gamma(t)-\gamma(s)}^2 \abs{\gamma(\sigma)-\gamma(s)}^2 - ((\gamma(t)-\gamma(s)) \cdot (\gamma(\sigma)-\gamma(s)))^2 \Bigr)}{\abs{\gamma(s)-\gamma(t)}^2 \abs{\gamma(t)-\gamma(\sigma)}^2 \abs{\gamma(s)-\gamma(\sigma)}^2}\\
-4 p \frac{\alpha_2(t,s) \, ((\gamma(t)-\gamma(s)) \cdot (\gamma(\sigma)-\gamma(s)))\,((\gamma(\sigma)-\gamma(s)) \cdot (\gamma'(t)-\gamma'(s)))}{\abs{\gamma(s)-\gamma(t)}^2 \abs{\gamma(t)-\gamma(\sigma)}^2 \abs{\gamma(s)-\gamma(\sigma)}^2}\\
-4 p \frac{\alpha_2(\sigma,s)\,((\gamma(t)-\gamma(s)) \cdot (\gamma(\sigma)-\gamma(s)))\,((\gamma(t)-\gamma(s)) \cdot (\gamma'(\sigma)-\gamma'(s)))}{\abs{\gamma(s)-\gamma(t)}^2 \abs{\gamma(t)-\gamma(\sigma)}^2 \abs{\gamma(s)-\gamma(\sigma)}^2} \, ds \, dt \, d\sigma.
\end{multline*}
Furthermore, we can compute using \eqref{trigaddc} and \eqref{trigadds} additionally
\begin{align}
&(\gamma(t)-\gamma(s)) \cdot (\gamma(\sigma)-\gamma(s)) \notag\\
=\;& (\cos(t)-\cos(s))(\cos(\sigma)-\cos(s)) + (\sin(t)-\sin(s))(\sin(\sigma)-\sin(s)) \notag\\
=\;&\quad\; 4 \sin(\frac{t+s}{2})\sin(\frac{t-s}{2})\sin(\frac{\sigma+s}{2})\sin(\frac{\sigma-s}{2}) \notag\\
& +4 \cos(\frac{t+s}{2})\sin(\frac{t-s}{2})\cos(\frac{\sigma+s}{2})\sin(\frac{\sigma-s}{2}) \notag\\
=\;& 4 \sin(\frac{t-s}{2}) \sin(\frac{\sigma-s}{2}) \cos(\frac{t-\sigma}{2}) \label{pcmmult20}
\intertext{and}
&(\gamma(t)-\gamma(s)) \cdot (\gamma'(\sigma)-\gamma'(s)) \notag\\
=\;& (\cos(t)-\cos(s))(-\sin(\sigma)+\sin(s)) + (\sin(t)-\sin(s))(\cos(\sigma)-\cos(s)) \notag\\
=\;& -4 \sin(\frac{t+s}{2})\sin(\frac{t-s}{2})\cos(\frac{\sigma+s}{2})\sin(\frac{\sigma-s}{2}) \notag\\
& +4 \cos(\frac{t+s}{2})\sin(\frac{t-s}{2})\sin(\frac{\sigma+s}{2})\sin(\frac{\sigma-s}{2}) \notag\\
=\;& -4 \sin(\frac{t-s}{2}) \sin(\frac{\sigma-s}{2}) \sin(\frac{t-\sigma}{2}). \label{pcmmult21}
\end{align}
Therefore, with \eqref{pcmmult20}, \eqref{pcmmult21} and \eqref{calabsWedge} the integrand simplifies to
\begin{multline*}
3 \gamma'(t)\cdot\Phi'(t) -2p \alpha_1(t,s) -p \alpha_1(\sigma,s) + p(\alpha_1(t,s)+\alpha_1(\sigma,s)) \,c_\gamma^2(s,t,\sigma)\\
-4 p \frac{-4^2 \alpha_2(t,s) \sin(\frac{t-s}{2})\sin(\frac{\sigma-s}{2})\cos(\frac{t-\sigma}{2})\sin(\frac{\sigma-s}{2})\sin(\frac{t-s}{2})\sin(\frac{\sigma-t}{2})}{4^3 \sin(\frac{s-t}{2})^2 \sin(\frac{t-\sigma}{2})^2 \sin(\frac{s-\sigma}{2})^2}\\
-4 p \frac{-4^2 \alpha_2(\sigma,s) \sin(\frac{t-s}{2}) \sin(\frac{\sigma-s}{2}) \cos(\frac{t-\sigma}{2}) \sin(\frac{t-s}{2}) \sin(\frac{\sigma-s}{2}) \sin(\frac{t-\sigma}{2})}{4^3 \sin(\frac{s-t}{2})^2 \sin(\frac{t-\sigma}{2})^2 \sin(\frac{s-\sigma}{2})^2}
\end{multline*}
\begin{equation*}
= 3 \gamma'(t)\cdot\Phi'(t) -p\alpha_1(t,s) -p \frac{(\alpha_2(t,s)-\alpha_2(\sigma,s))\cos(\frac{t-\sigma}{2})}{\sin(\frac{t-\sigma}{2})}.
\end{equation*}
By using \thref{lemcircleint1}, which will be our next issue, we see that the integral over the fraction vanishes. For the first expression we differentiate \eqref{pcmdiff}
\begin{multline}
\Phi'(x) = \alpha_1(x,s) \gamma'(x) -\alpha_2(x,s)\gamma(x)\\
+\alpha_1'(x,s) (\gamma(x)-\gamma(s)) +\alpha_2'(x,s) (\gamma'(x)-\gamma'(s)) +\alpha_3'(x,s)\,e_3,
\end{multline}
where $\alpha_i'$ denotes the derivative of $\alpha_i$ with respect to $x$ and we get
\begin{equation} \label{pcmfirst}
\gamma'(t)\cdot\Phi'(t) = \alpha_1(t,s)-\alpha_1'(t,s) \gamma'(t)\cdot\gamma(s)+\alpha_2'(t,s)-\alpha_2'(t,s) \gamma'(t)\cdot\gamma'(s).
\end{equation}
Observe, that $\int_0^{2\pi} \alpha_2'(t,s)\,dt=0$, since with $\gamma$ and $\gamma'$ also $\alpha_1, \alpha_2$ and $\alpha_3$ are $2\pi$-periodic. We consider the following expression
\begin{equation} \label{pcmrest}
\begin{split}
& \int_0^{2\pi} \int_0^{2\pi} \alpha_1'(t,s)\,\gamma'(t)\cdot\gamma(s)+\alpha_2'(t,s)\,\gamma'(t)\cdot\gamma'(s) \, ds \, dt\\
= & \int_0^{2\pi} \int_0^{2\pi} \alpha_1(t,s)\,(-\gamma''(t))\cdot\gamma(s)+\alpha_2(t,s)\,(-\gamma''(t))\cdot\gamma'(s) \, ds \, dt\\
= & \int_0^{2\pi} \int_0^{2\pi} \alpha_1(t,s)\,\gamma(t)\cdot\gamma(s)+\alpha_2(t,s)\,\gamma(t)\cdot\gamma'(s) \, ds \, dt,
\end{split}
\end{equation}
where we use integration by parts in addition. We continue with
\begin{align*}
\gamma(s)\cdot\gamma(t) &=& \cos(s)\cos(t)&+\sin(s)\sin(t) &=& \cos(s-t) \\
\gamma(s)\cdot\gamma'(t) &=& -\cos(s)\sin(t)&+\sin(s)\cos(t) &=& \sin(s-t)
\end{align*}
and therefore, together with \eqref{pcmalpha} the integrand of \eqref{pcmrest} is equal to
\begin{align*}
& \frac{(\Phi^1(s)-\Phi^1(t)) \Bigl[(\cos(s)-\cos(t))\cos(s-t)+(-\sin(s)+\sin(t))\sin(s-t)\Bigr]}{4 \sin(\frac{s-t}{2})^2} + \\
& \frac{(\Phi^2(s)-\Phi^2(t)) \Bigl[(\sin(s)-\sin(t))\cos(s-t)+(\cos(s)-\cos(t))\sin(s-t)\Bigr]}{4 \sin(\frac{s-t}{2})^2}.
\end{align*}
For the first line we get
\begin{align}
= & \frac{(\Phi^1(s)-\Phi^1(t)) \Bigl[\sin(\frac{s+t}{2})\sin(\frac{s-t}{2})\cos(s-t)-\cos(\frac{s+t}{2})\sin(\frac{s-t}{2})\sin(s-t)\Bigr]}{2 \sin(\frac{s-t}{2})^2} \notag\\
= & \frac{(\Phi^1(s)-\Phi^1(t)) \sin(t-\frac{s-t}{2})}{2 \sin(\frac{s-t}{2})} \label{pcialphasum1}
\end{align}
and for the second line
\begin{align}
= & \frac{(\Phi^2(s)-\Phi^2(t)) \Bigl[\cos(\frac{s+t}{2})\sin(\frac{s-t}{2})\cos(s-t)+\sin(\frac{s+t}{2})\sin(\frac{s-t}{2})\sin(s-t)\Bigr]}{2 \sin(\frac{s-t}{2})^2} \notag\\
= & \frac{(\Phi^2(s)-\Phi^2(t)) \cos(t-\frac{s-t}{2})}{2 \sin(\frac{s-t}{2})}. \label{pcialphasum2}
\end{align}
The integrals over \eqref{pcialphasum1} and \eqref{pcialphasum1} vanish, which we will prove in \thref{lemcircleint2,lemcircleint3}. Therefore, the only part of the expression \eqref{pcmfirst} that remains if we integrate is $\alpha_1(t,s)$ and we finally get
\begin{equation*}
\delta \M_p(\gamma,\Phi) = 2\pi \int_0^{2\pi} \int_0^{2\pi} (3-p)\,\alpha_1(t,s) \, ds \, dt,
\end{equation*}
which is zero for $p=3$. In addition we have $\Le(\gamma) = 2\pi$, $\M_p(\gamma) = (2\pi)^3$ and $\delta \Le(\gamma,\Phi) = \int_0^{2\pi} \gamma'(s)\cdot\Phi'(s) \, ds = \frac{1}{2\pi} \int_0^{2\pi} \int_0^{2\pi} \gamma'(s)\cdot\Phi'(s) \, ds \, dt = \frac{1}{2\pi} \int_0^{2\pi} \int_0^{2\pi} \alpha_1(t,s) \, ds \, dt$, as we have seen above. For the scale invariant energy this leads to
\begin{equation*}
\delta \E_p(\gamma,\Phi) = \frac{1}{2\pi p} \int_0^{2\pi} \int_0^{2\pi} (3-p)\,\alpha_1(t,s) \, ds \, dt + \frac{p-3}{2\pi p} \int_0^{2\pi} \int_0^{2\pi} \alpha_1(t,s) \, ds \, dt = 0. \qedhere
\end{equation*}
\end{proof}
\begin{lem} \mylabel{lemcircleint1}
Let $f$ be a $2\pi$-periodic function that is \name{H\"{o}lder}-continuous to some exponent $0 < \alpha \leq 1$. Then
\begin{equation*}
\int_0^{2\pi} \int_0^{2\pi} \frac{(f(s)-f(\sigma)) \cos(\frac{s-\sigma}{2})}{\sin(\frac{s-\sigma}{2})} \, ds \, d\sigma = 0.
\end{equation*}
\end{lem}
\begin{proof}
At first we consider
\begin{equation} \label{pci001}
\int_0^{2\pi} \int_{-\pi}^\pi \frac{(f(s+w)-f(s)) \cos(\frac{w}{2})}{\sin(\frac{w}{2})} \, dw \, ds.
\end{equation}
Observe that $\dfrac{\cos(\frac{x}{2})}{\sin(\frac{x}{2})}$ is $2\pi$-periodic as well, since $\dfrac{\cos(\frac{x}{2}+\pi)}{\sin(\frac{x}{2}+\pi)}=\dfrac{-\cos(\frac{x}{2})}{-\sin(\frac{x}{2})}$. Therefore, the whole integrand  is $2\pi$-periodic and we can use \thref{propshift} all the time. Due to \eqref{sinelb}, we have
\begin{equation*}
\abs{\sin(\frac{w}{2})} \geq \frac{1}{\pi}\,\abs{w} \qquad \text{for all $w \in [-\pi,\pi]$}.
\end{equation*}
Hence, we can estimate the integrand for all $s \in [0,2\pi]$ and $w \in [-\pi,\pi]$ by
\begin{equation} \label{pciconst001}
\babs{\frac{(f(s+w)-f(s)) \cos(\frac{w}{2})}{\sin(\frac{w}{2})}} \leq \pi [f]_{C^{0,\alpha}}\,\frac{\abs{w}^\alpha}{\abs{w}} =: C\,\abs{w}^{-(1-\alpha)}.
\end{equation}
Hence for $0 \leq 1-\alpha < 1$ this function is \name{Lebesgue} integrable around $0$ and due to \thref{remproofbw} we get that the integral \eqref{pci001} exists. Next we consider for $\eps >0$
\begin{equation*}
\int_0^{2\pi} \int_{\eps}^\pi \frac{f(s+w) \cos(\frac{w}{2})}{\sin(\frac{w}{2})} \, dw \, ds \qquad \text{and} \qquad \int_0^{2\pi} \int_{\eps}^\pi \frac{f(s) \cos(\frac{w}{2})}{\sin(\frac{w}{2})} \, dw \, ds
\end{equation*}
In both cases the integrand is continuous on $[0,2\pi]\times[\eps,\pi]$ and its absolute value can be estimated by the constant $\norm{f}_{C^0} \kla{\sin(\frac{\eps}{2})}^{-1}$, since sine is strictly increasing on $[0,\frac{\pi}{2}]$.
\begin{alignat*}{2}
& \int_0^{2\pi} \int_\eps^\pi \frac{f(s+w) \cos(\frac{w}{2})}{\sin(\frac{w}{2})} \, dw \, ds &\;=& \int_\eps^\pi \int_0^{2\pi} \frac{f(s+w) \cos(\frac{w}{2})}{\sin(\frac{w}{2})} \, ds \, dw\\
= & \int_\eps^\pi \int_w^{2\pi-w} \frac{f(\tilde{s}) \cos(\frac{w}{2})}{\sin(\frac{w}{2})} \, d\tilde{s} \, dw &\;=& \int_\eps^\pi \int_0^{2\pi} \frac{f(s) \cos(\frac{w}{2})}{\sin(\frac{w}{2})} \, ds \, dw\\
= & \int_0^{2\pi} \int_\eps^\pi \frac{f(s) \cos(\frac{w}{2})}{\sin(\frac{w}{2})} \, dw \, ds,
\end{alignat*}
where we substituted $\tilde{s}:=s+w$ and used \name{Fubini}'s theorem twice. Therefore,
\begin{equation*}
\int_0^{2\pi} \int_\eps^\pi \frac{(f(s+w)-f(s)) \cos(\frac{w}{2})}{\sin(\frac{w}{2})} \, dw \, ds = 0 \qquad \text{for all $\eps>0$}.
\end{equation*}
Moreover, applying \name{Lebesgue}'s dominated convergence theorem we conclude that
\begin{multline} \label{pci003}
\int_0^{2\pi} \int_\eps^\pi \frac{(f(s+w)-f(s)) \cos(\frac{w}{2})}{\sin(\frac{w}{2})} \, dw \, ds \\
\xrightarrow{\eps\searrow 0} \int_0^{2\pi} \int_0^\pi \frac{(f(s+w)-f(s)) \cos(\frac{w}{2})}{\sin(\frac{w}{2})} \, dw \, ds = 0 
\end{multline}
since the function $\frac{(f(s+w)-f(s)) \cos(\frac{w}{2})}{\sin(\frac{w}{2})} \chi_{[\eps,\pi]}$ converges pointwise to $\frac{(f(s+w)-f(s)) \cos(\frac{w}{2})}{\sin(\frac{w}{2})} \chi_{[0,\pi]}$ and can be estimated by the majorant \eqref{pciconst001}. In addition we get
\begin{equation} \label{pci002}
\begin{split}
\int_0^{2\pi} \int_{-\pi}^0 &\frac{(f(s+w)-f(s)) \cos(\frac{w}{2})}{\sin(\frac{w}{2})} \, dw \, ds\\
&= \int_{-\pi}^0 \int_0^{2\pi} \frac{(f(s+w)-f(s)) \cos(\frac{w}{2})}{\sin(\frac{w}{2})} \, ds \, dw\\
&= -\int_\pi^0 \int_0^{2\pi} \frac{(f(s-\tilde{w})-f(s)) \cos(\frac{\tilde{w}}{2})}{-\sin(\frac{\tilde{w}}{2})} \, d\tilde{w} \, ds\\
&= -\int_0^\pi \int_0^{2\pi} \frac{(f(s-w)-f(s)) \cos(\frac{w}{2})}{\sin(\frac{w}{2})} \, ds \, dw\\
&= -\int_0^\pi \int_{-w}^{2\pi-w} \frac{(f(\tilde{s})-f(\tilde{s}+w)) \cos(\frac{w}{2})}{\sin(\frac{w}{2})} \, d\tilde{s} \, dw\\
&= -\int_0^\pi \int_0^{2\pi} \frac{(f(s)-f(s+w)) \cos(\frac{w}{2})}{\sin(\frac{w}{2})} \, ds \, dw\\
&= \int_0^{2\pi} \int_0^\pi \frac{(f(s+w)-f(s)) \cos(\frac{w}{2})}{\sin(\frac{w}{2})} \, dw \, ds,
\end{split}
\end{equation}
where we substituted $\tilde{w}:=-w$ respectively $\tilde{s}:=s-w$ and used \name{Fubini}'s theorem twice. Finally \eqref{pci003} and \eqref{pci002} yields
\begin{align*}
0 &= \int_0^{2\pi} \int_{-\pi}^\pi \frac{(f(s+w)-f(s)) \cos(\frac{w}{2})}{\sin(\frac{w}{2})} \, dw \, ds\\
&= \int_0^{2\pi} \int_{s-\pi}^{s+\pi} \frac{(f(\sigma)-f(s)) \cos(\frac{s-\sigma}{2})}{-\sin(\frac{s-\sigma}{2})} \, d\sigma \, ds\\
&= \int_0^{2\pi} \int_{0}^{2\pi} \frac{(f(s)-f(\sigma)) \cos(\frac{s-\sigma}{2})}{\sin(\frac{s-\sigma}{2})} \, d\sigma \, ds,
\end{align*}
where we used the substitution $\sigma:=s+w$.
\end{proof}
\begin{lem} \mylabel{lemcircleint2}
Let $f$ be a $2\pi$-periodic function that is \name{H\"{o}lder}-continuous to some exponent $0 < \alpha \leq 1$. Then
\begin{equation}
\int_0^{2\pi} \int_0^{2\pi} \frac{(f(s)-f(t)) \sin(t-\frac{s-t}{2})}{2\sin(\frac{s-t}{2})} \, ds \, dt = 0.
\end{equation}
\end{lem}
\begin{proof}
The proof is quite analog to the one of \thref{lemcircleint1}. Again we first consider
\begin{equation} \label{pci004}
\int_0^{2\pi} \int_{-\pi}^\pi \frac{(f(t+w)-f(t)) \sin(t-\frac{w}{2})}{2 \sin(\frac{w}{2})} \, dw \, dt.
\end{equation}
Observe that $\dfrac{\sin(t-\frac{x}{2})}{\sin(\frac{x}{2})}$ is $2\pi$-periodic since$\dfrac{\sin(t-\frac{x}{2}+\pi)}{\sin(\frac{x}{2}+\pi)}=\dfrac{-\sin(\pi-(\frac{x}{2}-t)}{-\sin(\frac{x}{2}+\pi)}$. The periodicity with respect to $t$ is obvious. With the same reasoning as in \thref{lemcircleint1} we get that integral \eqref{pci004} exists. Next we consider for $\eps >0$
\begin{equation} \label{pci005}
\int_0^{2\pi} \int_\eps^\pi \frac{(f(t+w)-f(t)) \sin(t-\frac{w}{2})}{2\sin(\frac{w}{2})} \, dw \, dt.
\end{equation}
The integrand is continuous on $[0,2\pi]\times[\eps,\pi]$ and its absolute value can be estimated by the constant $\norm{f}_{C^0} \kla{2\sin(\frac{\eps}{2})}^{-1}$. Furthermore,
\begin{align*}
& \int_0^{2\pi} \int_\eps^\pi \frac{(f(t+w)-f(t)) \sin(t-\frac{w}{2})}{2\sin(\frac{w}{2})} \, dw \, dt \\
= & \int_0^{2\pi} \int_\eps^\pi \kla{\frac{f(t+w) \sin((t+w)-\frac{3w}{2})}{2\sin(\frac{w}{2})} -\frac{f(t) \sin(t-\frac{w}{2})}{2\sin(\frac{w}{2})}} \, dw \, dt \\
= & \int_0^{2\pi} \int_\eps^\pi f(t) \kla{\frac{\sin(t-\frac{3w}{2})}{2\sin(\frac{w}{2})} -\frac{\sin(t-\frac{w}{2})}{2\sin(\frac{w}{2})}} \, dw \,dt,
\end{align*}
where the last step is due to
\begin{align*}
& \int_0^{2\pi} \int_\eps^\pi \frac{f(t+w) \sin((t+w)-\frac{3w}{2})}{2\sin(\frac{w}{2})} \, dw \, dt & = & \int_\eps^\pi \int_0^{2\pi} \frac{f(t+w) \sin((t+w)-\frac{3w}{2})}{2\sin(\frac{w}{2})} \, dt \, dw \\
= & \int_\eps^\pi \int_w^{2\pi+w} \frac{f(\tilde{t}) \sin(\tilde{t}-\frac{3w}{2})}{2\sin(\frac{w}{2})} \, d\tilde{t} \, dw & = & \int_\eps^\pi \int_0^{2\pi} \frac{f(t) \sin(t-\frac{3w}{2})}{2\sin(\frac{w}{2})} \, dt \, dw \\
= & \int_0^{2\pi} \int_\eps^\pi \frac{f(t) \sin(t-\frac{3w}{2})}{2\sin(\frac{w}{2})} \, dw \,dt,
\end{align*}
using the substitution $\tilde{t}:=t+w$. Next we apply \eqref{trigthrees} and \eqref{trigthreec} to observe that
\begin{align*}
\frac{\sin(t-\frac{w}{2})}{2 \sin(\frac{w}{2})} &= \frac{\sin(t) \cos(\frac{w}{2})}{2 \sin(\frac{w}{2})} -\frac{1}{2} \cos(t),\\
\frac{\sin(t-3\frac{w}{2})}{2 \sin(\frac{w}{2})} &= \frac{\sin(t) \cos(3\frac{w}{2})}{2 \sin(\frac{w}{2})} - \frac{\cos(t) \sin(3\frac{w}{2})}{2 \sin(\frac{w}{2})} \\
&= \frac{\sin(t) \cos(\frac{w}{2}) \kla{1-4\sin(\frac{w}{2})^2}}{2 \sin(\frac{w}{2})} - \frac{\cos(t) \sin(\frac{w}{2}) \kla{4\cos(\frac{w}{2})^2-1}}{2 \sin(\frac{w}{2})} \\
&= \frac{\sin(t) \cos(\frac{w}{2})}{2 \sin(\frac{w}{2})} -2 \sin(t) \cos(\frac{w}{2}) \sin(\frac{w}{2}) -2 \cos(t) \cos(\frac{w}{2})^2+\frac{1}{2} \cos(t).
\end{align*}
Therefore, \eqref{pci005} simplifies to
\begin{align}
& \int_0^{2\pi} \int_\eps^\pi f(t) \kla{-2 \sin(t) \cos(\frac{w}{2}) \sin(\frac{w}{2}) -\cos(t) \kla{2\cos(\frac{w}{2})^2-1}} \, dw \,dt \notag\\
= & \int_0^{2\pi} \int_\eps^\pi -f(t) \kla{2 \sin(t) \cos(\frac{w}{2}) \sin(\frac{w}{2}) +\cos(t) \cos(w)} \, dw \,dt \notag\\
\xrightarrow{\eps \rightarrow 0} & \int_0^{2\pi} \int_0^\pi -f(t) \kla{2 \sin(t) \cos(\frac{w}{2}) \sin(\frac{w}{2}) +\cos(t) \cos(w)} \, dw \,dt, \label{pci006}
\end{align}
applying \name{Lebesgue}'s dominated convergence theorem, since the integrand can be estimated by $3 \norm{f}_{C_0}$. Moreover, we conclude that
\begin{multline} \label{pci007}
\int_0^{2\pi} \int_\eps^\pi \frac{(f(s+w)-f(s)) \sin(t-\frac{w}{2})}{\sin(\frac{w}{2})} \, dw \, ds \\
\xrightarrow{\eps\searrow 0} \int_0^{2\pi} \int_0^\pi \frac{(f(s+w)-f(s)) \sin(t-\frac{w}{2})}{\sin(\frac{w}{2})} \, dw \, ds,
\end{multline}
using the same arguments than as \eqref{pci003}. Combining \eqref{pci006} and \eqref{pci007} we get
\begin{multline} \label{pci008}
\int_0^{2\pi} \int_0^\pi \frac{(f(s+w)-f(s)) \sin(t-\frac{w}{2})}{\sin(\frac{w}{2})} \, dw \, ds = \\
\int_0^{2\pi} \int_0^\pi -f(t) \kla{2 \sin(t) \cos(\frac{w}{2}) \sin(\frac{w}{2}) +\cos(t) \cos(w)} \, dw \,dt.
\end{multline}
Analogously we get
\begin{align}
& \int_0^{2\pi} \int_{-\pi}^0 \frac{(f(t+w)-f(t)) \sin(t-\frac{w}{2})}{2\sin(\frac{w}{2})} \, dw \, dt \notag\\
= & \int_0^{2\pi} \int_{-\pi}^0 -f(t) \kla{2 \sin(t) \cos(\frac{w}{2}) \sin(\frac{w}{2}) +\cos(t) \cos(w)} \, dw \,dt \notag\\
= & \int_0^{2\pi} \int_0^\pi -f(t) \kla{2 \sin(t) \cos(\frac{\tilde{w}}{2}) (-\sin(\frac{\tilde{w}}{2})) +\cos(t) \cos(\tilde{w})} \, d\tilde{w} \,dt \notag\\
= & \int_0^{2\pi} \int_0^\pi -f(t) \kla{-2 \sin(t) \cos(\frac{w}{2}) \sin(\frac{w}{2}) +\cos(t) \cos(w)} \, dw \,dt, \label{pci009}
\end{align}
where we substituted $\tilde{w}:=-w$. Furthermore, using \eqref{pci008} and \eqref{pci009} we see that
\begin{align*}
& \int_0^{2\pi} \int_{-\pi}^\pi \frac{(f(t+w)-f(t)) \sin(t-\frac{w}{2})}{2\sin(\frac{w}{t})} \, dw \, dt \\
= & \int_0^{2\pi} \int_0^\pi -2\,f(t) \cos(t) \cos(w) \, dw \,dt \\
= & \int_0^{2\pi} -2\,f(t) \cos(t) \Bigl[ \sin(w) \Bigr]_0^\pi \, dt = 0.
\end{align*}
And finally,
\begin{align*}
0 &= \int_0^{2\pi} \int_{-\pi}^\pi \frac{(f(t+w)-f(t)) \sin(t-\frac{w}{2})}{2\sin(\frac{w}{t})} \, dw \, dt\\
&= \int_0^{2\pi} \int_{t-\pi}^{t+\pi} \frac{(f(s)-f(t)) \sin(t-\frac{s-t}{2})}{2\sin(\frac{s-t}{t})} \, ds \, dt\\
&= \int_0^{2\pi} \int_{0}^{2\pi} \frac{(f(s)-f(t)) \sin(t-\frac{s-t}{2})}{2\sin(\frac{s-t}{t})} \, ds \, dt,
\end{align*}
where we used the substitution $s:=t+w$.
\end{proof}
\begin{lem} \mylabel{lemcircleint3}
Let $f$ be a $2\pi$-periodic function that is \name{H\"{o}lder}-continuous to some exponent $0 < \alpha \leq 1$. Then
\begin{equation}
\int_0^{2\pi} \int_0^{2\pi} \frac{(f(s)-f(t)) \cos(t-\frac{s-t}{2})}{2\sin(\frac{s-t}{2})} \, ds \, dt = 0.
\end{equation}
\end{lem}
\begin{proof}
This proof is completely analogous to the proof of \thref{lemcircleint2}. The first differences appear in the following equations
\begin{align*}
\frac{\cos(t-\frac{w}{2})}{2 \sin(\frac{w}{2})} &= \frac{\cos(t) \cos(\frac{w}{2})}{2 \sin(\frac{w}{2})} +\frac{1}{2} \sin(t),\\
\frac{\cos(t-3\frac{w}{2})}{2 \sin(\frac{w}{2})} &= \frac{\cos(t) \cos(3\frac{w}{2})}{2 \sin(\frac{w}{2})} + \frac{\sin(t) \sin(3\frac{w}{2})}{2 \sin(\frac{w}{2})} \\
&= \frac{\cos(t) \cos(\frac{w}{2}) \kla{1-4\sin(\frac{w}{2})^2}}{2 \sin(\frac{w}{2})} + \frac{\sin(t) \sin(\frac{w}{2}) \kla{4\cos(\frac{w}{2})^2-1}}{2 \sin(\frac{w}{2})} \\
&= \frac{\cos(t) \cos(\frac{w}{2})}{2 \sin(\frac{w}{2})} -2 \cos(t) \cos(\frac{w}{2}) \sin(\frac{w}{2}) +2 \sin(t) \cos(\frac{w}{2})^2-\frac{1}{2} \sin(t),
\end{align*}
using \eqref{trigthrees} and \eqref{trigthreec}. Therefore, the analogue of \eqref{pci005} simplifies to
\begin{align*}
& \int_0^{2\pi} \int_\eps^\pi f(t) \kla{-2 \cos(t) \cos(\frac{w}{2}) \sin(\frac{w}{2}) +\sin(t) \kla{2\cos(\frac{w}{2})^2-1}} \, dw \,dt\\
= & \int_0^{2\pi} \int_\eps^\pi -f(t) \kla{2 \cos(t) \cos(\frac{w}{2}) \sin(\frac{w}{2}) -\sin(t) \cos(w)} \, dw \,dt\\
\xrightarrow{\eps \rightarrow 0} & \int_0^{2\pi} \int_0^\pi -f(t) \kla{2 \cos(t) \cos(\frac{w}{2}) \sin(\frac{w}{2}) -\sin(t) \cos(w)} \, dw \,dt.
\end{align*}
And consequently we obtain
\begin{multline*}
\int_0^{2\pi} \int_0^\pi \frac{(f(s+w)-f(s)) \sin(t-\frac{w}{2})}{\cos(\frac{w}{2})} \, dw \, ds=\\
\int_0^{2\pi} \int_0^\pi -f(t) \kla{2 \cos(t) \cos(\frac{w}{2}) \sin(\frac{w}{2}) -\sin(t) \cos(w)} \, dw \,dt.
\end{multline*}
Again analogously we get
\begin{align*}
& \int_0^{2\pi} \int_{-\pi}^0 \frac{(f(t+w)-f(t)) \cos(t-\frac{w}{2})}{2\sin(\frac{w}{2})} \, dw \, dt\\
= & \int_0^{2\pi} \int_{-\pi}^0 -f(t) \kla{2 \cos(t) \cos(\frac{w}{2}) \sin(\frac{w}{2}) -\sin(t) \cos(w)} \, dw \,dt\\
= & \int_0^{2\pi} \int_0^\pi -f(t) \kla{2 \cos(t) \cos(\frac{\tilde{w}}{2}) (-\sin(\frac{\tilde{w}}{2})) -\sin(t) \cos(\tilde{w})} \, d\tilde{w} \,dt\\
= & \int_0^{2\pi} \int_0^\pi -f(t) \kla{-2 \cos(t) \cos(\frac{w}{2}) \sin(\frac{w}{2}) -\sin(t) \cos(w)} \, dw \,dt,
\end{align*}
where we substituted $\tilde{w}:=-w$. Finally, we gain
\begin{equation*}
\int_0^{2\pi} \int_{-\pi}^\pi \frac{(f(t+w)-f(t)) \cos(t-\frac{w}{2})}{2\sin(\frac{w}{2})} \, dw \, dt = \int_0^{2\pi} \int_0^\pi 2\,f(t) \sin(t) \cos(w)\, dw \,dt = 0.
\end{equation*}
This completes the proof.
\end{proof}

\begin{lem} \mylabel{lemFVar_scalesec}
Let $E$ be a functional for a function space $X$. Let further $\gamma$ and $\Phi$ be functions in $X$ such that the first variation of $E$ at $\gamma$ in direction $\Phi$ exists. Then we have for all $r>0$
\begin{equation*}
\delta E(\gamma, r\Phi) = r \, \delta E(\gamma, \Phi )
\end{equation*}
\end{lem}
\begin{proof}
From the \thref{defFVar} of the first variation we get
\begin{alignat*}{2}
\delta E(\gamma, r\Phi) &= \lim_{\tau \rightarrow 0} \frac{E(\gamma+\tau(r \Phi)) - E(\gamma)}{\tau} &&= r \, \lim_{\tau \rightarrow 0} \frac{E(\gamma+(r\tau)\Phi) - E(\gamma)}{r\tau} \\
&= r \, \lim_{\rho \rightarrow 0} \frac{E(\gamma+\rho \Phi) - E(\gamma)}{\rho} &&= r \, \delta E(\gamma, \Phi ),
\end{alignat*}
where we replace $\tau$ by $\rho := r \tau$ in the last line.
\end{proof}
\begin{rem} \mylabel{remFVar_linearsec}
Observe, that it is easy to see that for the energies we are considering here we have that the first variation is even linear in the second argument.
\end{rem}
Using the \thref{lemFVar_scalesec} we get the following statement
\begin{lem} \mylabel{lemFVar_scalefir}
Let $E$ be a functional for a function space $X$. Let further $\gamma$ and $\Phi$ be functions in $X$ such that the first variation of $E$ at $\gamma$ in direction $\Phi$ exists. Moreover, let $E$ scales like $E(r \gamma) = r^\alpha E(\gamma)$ for all $r>0$ and all admissible $\gamma\in X$. Then we have for all $r>0$
\begin{equation}
\delta E(r\gamma,\Phi) = r^{\alpha-1} \, \delta E(\gamma, \Phi).
\end{equation}
\end{lem}
\begin{proof}
From the \thref{defFVar} of the first variation we get
\begin{alignat*}{2}
\delta E(r\gamma,\Phi) &= \frac{d}{d\tau} \eins{\tau =0} E(r\gamma+\tau\Phi) &&= r^\alpha \, \frac{d}{d\tau} \eins{\tau =0} E\kla{\gamma +\tau \frac{\Phi}{r}} \\
&= r^\alpha \, \delta E\kla{\gamma, \frac{\Phi}{r}} &&= r^{\alpha-1} \, \delta E(\gamma, \Phi),
\end{alignat*}
where we used \thref{lemFVar_scalesec} in the last equation.
\end{proof}
\begin{cor} \mylabel{corFVar_stpts}
Let $X$ be a function space, $\alpha\in\R$ and $E$ be a functional that scales like $E(r \gamma) = r^\alpha E(\gamma)$ for all $r>0$ and all admissible $\gamma\in X$. Assume there exists a stationary point $\gamma \in X$ of $E$, i.e. $\delta E(\gamma,\Phi)=0$ for all functions $\Phi\in X$, then we have
\begin{equation*}
\alpha=0 \qquad \textbf{or} \qquad E(\gamma)=0.
\end{equation*}
In this case we have that $\delta E(r\gamma, \Phi)=0$ for all $\Phi\in X$ and all $r>0$.
\end{cor}
\begin{rem} \mylabel{remnostpts}
\begin{enumabc}
\item Observe that if $\gamma$ is a (local) minimizer or maximizer of $E$ then this statement is trivial, because if $E(\gamma)\neq 0$, the energy can be increased or decreased by scaling up or down the knot.
\item In our situation it is the case that
\begin{equation*}
\M_p(\gamma) > 0,
\end{equation*}
because if it were zero then the integrand would need to be zero as it is a nonnegative function. This would imply that for each triple $s,t,\sigma \in \R^3\setminus \Ns$
\begin{equation*}
\abs{\gamma'(s)} \abs{\gamma'(t)} \abs{\gamma'(\sigma)}=0 \quad \text{or} \quad \abs{(\gamma(t)-\gamma(s)) \wedge (\gamma(\sigma)-\gamma(s))}=0,
\end{equation*}
where the first one is impossible, because $\gamma$ is regular. This would guarantee that for each triple $s,t,\sigma \in \R^3\setminus \Ns$ there exists a constant $c=c(s,t,\sigma)\in\R$ such that
\begin{alignat*}{2}
&& (\gamma(t)-\gamma(s)) & = c\,(\gamma(\sigma)-\gamma(s))\\
\Leftrightarrow\qquad && \gamma(t) & = c\,\gamma(\sigma) + (1-c)\,\gamma(s).
\end{alignat*}
Therefore, by fixing $s,\sigma\in\R^2$, with $(s-t)\notin 2\pi\Z$ we would gain that $\gamma(t)$ lies on the straight line through $\gamma(s)$ and $\gamma(\sigma)$ for all $t\in\R$. This is not possible since $\gamma$ is closed and simple. Consequently, the \emph{only} variant of our energy that could have stationary points is the one that is scale invariant! In this case with $\gamma$ all scales of $\gamma$ are stationary points. Furthermore, \name{Strzelecki}, \name{Szuma\'nska} and \name{von der Mosel} mentioned in \cite[Remark 4.6.]{StrSzvdM10} that minima of the $\M_p$-energy are achieved in prescribed knot classes. More precisely, they show the existence of a knot $\Gamma$ in $C_{L,k}:=C_{L,k}(\R/L\Z,\R^3)$, which was defined in \eqref{knotclasslen}, parametrized by arclength such that $\M_p(\Gamma) = \inf_{C_{L,k}} \M_p(\cdot)$. Applying \cite[Theorem 1.1]{blatt_menger} we obtain that the energy of $\Gamma$ is finite, since each knot class contains a smooth representative, and moreover, that $\Gamma$ actually is in $W^{2-\frac{2}{p},p}(\R/L\Z) \cap C_{L,k}$. Next we will prove that the same $\Gamma$ is also a minimizer of the $\E_p$-energy in $C_{k}(\R/L\Z)$, which was defined in \eqref{knotclasslen} as well. Assume there exists another knot $\widetilde{\Gamma}$ with a lower $\E_p$-energy value. Then $\widehat{\Gamma} := \frac{L}{\Le(\widetilde{\Gamma})} \widetilde{\Gamma}$ has the same energy level since $\E_p$ is invariant under scaling. Moreover, in contradiction to the minimality of $\Gamma$, we have $\M_p(\widehat{\Gamma}) < \M_p(\Gamma)$, hence $\widehat{\Gamma} \in C_{L,k}$ and consequently $\Le(\widehat{\Gamma})=\Le(\Gamma)=L$. In \thref{remdiff} after \thref{thmdiff} we proved that $\Gamma$ is a stationary point of $\E_p$.
\item Observe that in general an energy $E$ does \emph{not} have to be invariant under scaling in order to have a stationary point with a non-zero energy value. Necessary for the existence of such a point is that there is no $\alpha \in \R\setminus\menge{0}$ such that $E(r\gamma)=r^\alpha E(\gamma)$ for all $r>0$. For instance consider the energy \eqref{energy_pMp+L} and the associated \thref{mpcirclelambda}.
\end{enumabc}
\end{rem}
\begin{proof}[Proof of \thref{corFVar_stpts}]
Let $r>0$. As a consequence of the assumptions we get from \thref{lemFVar_scalefir} $\delta E(r\gamma,\Phi)=0$ for all functions $\Phi\in X$
\begin{equation*}
\delta E(r\gamma,\gamma) = \lim_{\tau\rightarrow 0} \frac{E(r\gamma+\tau\gamma)-E(r\gamma)}{\tau} = \lim_{\tau\rightarrow 0} \frac{(r+\tau)^\alpha-r^\alpha}{\tau} E(\gamma) = 0.
\end{equation*}
We consider the \name{Taylor}-series of $\tau \mapsto (r+\tau)^\alpha$ around $0$ that reads $r^\alpha + \alpha (r+\xi(0,\tau))^{\alpha-1}\,\tau$, where $(0,\tau) \ni \xi(0,\tau) \xrightarrow{\tau\rightarrow 0} 0$. Therefore,
\begin{equation*}
\lim_{\tau\rightarrow 0} \frac{(r+\tau)^\alpha-r^\alpha}{\tau}\, E(\gamma) = \lim_{\tau\rightarrow 0} \alpha\,(r+\xi(0,\tau))^{\alpha-1}\,E(\gamma) = \alpha\, r^\alpha E(\gamma) = 0.
\end{equation*}
This is exactly the case if $E(\gamma)=0$ or $\alpha=0$.
\end{proof}

In our situation we have to solve an equation like this
\begin{equation}
\int_0^{2\pi} \partial_t \gamma(x,t) \abs{\gamma'(x,t)} \, \Phi(x) \, dx + \delta E(\gamma, \Phi ) = 0.
\end{equation}
Now we compute
\begin{align}
\int_0^{2\pi} \partial_t (r\gamma)(x,t) \abs{(r \gamma')(x,t)} \, \Phi(x) \, dx &= r^2 \int_0^{2\pi} \partial_t \gamma(x,t) \abs{\gamma'(x,t)} \, \Phi(x) \, dx \notag\\
&= r^2 \kla{- \delta E(\gamma,\Phi)} \notag\\
&= r^{3-\alpha} \kla{- \delta E(r\gamma,\Phi)},
\end{align}
by applying \thref{lemFVar_scalefir} in the last step. Furthermore, we can conclude
\begin{equation} \label{firstscale}
\int_0^{2\pi} \frac{\partial_t(r\gamma)(x,t)}{r^3} \abs{(r\gamma)'(x,t)} \, \Phi(x) \, dx + \delta \E_p(r\gamma, \Phi) = 0,
\end{equation}
which implies that we may multiply the time step size by a factor of $r^3$ when scaling the knot by a factor of $r$.

\section{Modified integral \name{Menger} curvature energies}
We consider this modification of the energy
\begin{equation} \label{energy_pMp+L}
\E_p^\lambda (\gamma) := \kla{\M_p(\gamma)}^{\frac{1}{p}} + \lambda \Le(\gamma), \qquad \lambda > 0.
\end{equation}
A result of the calculus of variations is that if $\gamma$ is a minimizer of $(\M_p(\gamma))^{\frac{1}{p}}$ under the restriction that $\Le(\gamma)=c$ for a fixed given $c>0$ then there exists a $\lambda >0$ such that $\gamma$ is a minimizer of \eqref{energy_pMp+L}.\\
In a gradient flow it would be necessary to evaluate that $\lambda$ in every time step. In order to achieve constant length during the flow one could alternatively rescale the knot after every step, since the energy stays the same. Of course the resulting flow would not be exactly the same as using a \name{Lagrange} multiplier, but it would be faster. In addition it is interesting to consider the energy $\E_p^\lambda$ with a fixed $\lambda>0$. Observe that in this case we do not have any scaling behaviour as mentioned above, because $\E_p^\lambda(r\gamma) = r^{\frac{3}{p}-1} (\M_p(\gamma))^\frac{1}{p} + r \lambda \Le(\gamma)$. This is equal to $r \E_p^\lambda(\gamma)$ if and only if $p=\frac{3}{2}$, but we are only interested in $p>3$. Therefore, it is reasonable to search for roots of the first variation of $\E_p^\lambda$ among circles with different radii. For the first variation we get
\begin{align*}
\delta \E_p^\lambda(\gamma,\Phi) &= \frac{1}{p} \kla{\M_p(\gamma)}^{\frac{1}{p}-1} \delta\M_p(\gamma,\Phi) + \lambda\, \delta \Le(\gamma,\Phi).
\end{align*}
Let $\Theta$ be the unit circle and $r>0$, using \thref{lemFVar_scalefir} we get
\begin{align*}
\delta \E_p^\lambda(r\Theta,\Phi) &= \frac{1}{p} \kla{r^{3-p} \M_p(\Theta)}^{\frac{1}{p}-1} r^{2-p}\, \delta\M_p(\Theta,\Phi) + \lambda\, \delta \Le(\Theta,\Phi) \\
&= \int_0^{2\pi} \kla{-\frac{p-3}{p}\, (2\pi)^\frac{3-p}{p}\, r^\frac{3-2p}{p} + \lambda} \alpha_1(s) \, ds.
\end{align*}
This expression is zero if
\begin{align*}
\lambda &= (2\pi)^\frac{3-p}{p}\, \frac{p-3}{p}\, r^\frac{3-2p}{p} \\
\Longleftrightarrow \qquad r &= \kla{(2\pi)^\frac{p-3}{p}\, \frac{p}{p-3}\, \lambda}^\frac{p}{3-2p}.
\end{align*}
\begin{exmp} \mylabel{mpcirclelambda}
In this experiment we start with the unit circle $\Theta$, which has the length $\Le(\Theta) = 2\pi \approx 6.28319$, and want it to grow until it has a length of $\Le(r\Theta)=7.0$. We choose $p=4$ and set $r=\frac{7}{2\pi}$ accordingly to our issue and therefore, we gain
\begin{equation*}
\lambda = \frac{\pi}{2\cdot 7^\frac{5}{4}} \approx 0.13796.
\end{equation*}
The flow of this example is presented in the example part of this thesis (see Figure~\ref{exmpcirclelambda}).
\end{exmp}

\cleartooddpage[\thispagestyle{empty}]

\chapter{Numerical gradient flow}
In this chapter we want to discretize the first variation that we calculated and discussed in the last chapter in order to get a numerical gradient flow for integral \name{Menger} curvature.

\section{The \name{Ritz-Galerkin} Method} \label{discspace}
A \emph{gradient flow} moves a curve $\gamma$ along the steepest descent concerning a given energy $E$. Therefore, we need to expand the function space by a time component. For $T>0$ we regard a given knot as $\gamma: [0,T] \rightarrow W_{\text{s,r}}^{2-\frac{2}{p},p}(\R/2\pi\Z,\R^3)$, using notation \eqref{defvarclass}. By $\gamma'$ we denote the function $t\mapsto \frac{d}{dx} \gamma(t)$. We consider the so-called \emph{weak formulation} of our problem
\begin{equation*}
\left\langle \frac{d}{dt} \gamma(t)\;\abs{\gamma'(t)}, \Phi \right\rangle_{L^2} + \frac{d}{d\tau} E (\gamma(t) + \tau \Phi) \eins{\tau=0} = 0 \quad \text{for all } \Phi \in W^{2-\frac{2}{p},p}(\R/2\pi\Z,\R^3).
\end{equation*}

The main idea of the \name{Ritz-Galerkin} (cf. \cite[8.1]{hackbusch}) method is to choose an appropriate approximation space $V$ of finite dimension and to replace the full function space by $V$. We will turn to the approximation space in the next section. Further, we need that the first variation of $E$ is linear in its second argument, see \thref{remFVar_linearsec}. If $\varphi_i$ for $i=1,\dots,M$ and $M\in\N$ are the basis functions, we represent our knot by
\begin{equation*}
\gamma(t) := \sum_{i=1}^{M} c_i(t)\;\varphi_i, \qquad f_i(t) \in \R^3.
\end{equation*}
This way the weak formulation is equivalent to (cf. \cite[Lemma 8.1.2]{hackbusch})
\begin{equation*}
\left\langle \frac{d}{dt} \gamma(t)\;\abs{\gamma'(t)}, \varphi_i e_\ell \right\rangle_{L^2} + \frac{d}{d\tau} E (\gamma(t) + \tau \varphi_i e_\ell) \eins{\tau=0} = 0, 
\end{equation*}
for all $i=1, \dots, M$ and $\ell=1,2,3$.

\section{Approximation spaces}
We use (real) trigonometric polynomials, i.e. partial sums of \name{Fourier} series, cf. \cite[1.2]{butzer}, to approximate the knots in order to store and handle them for the discretization. Let $N \in \N$ and
\begin{equation*}
\gamma_N (x) = \frac{a_0}{2} + \sum_{k=1}^{N} \kla{a_k \cos(k x) + b_k \sin(k x)} \in C^\infty(\R/2\pi\Z,\R^3).
\end{equation*}
For $\gamma$ in $L_1(\R/2\pi\Z,\R^3)$ we can compute the following coefficients \cite[Definition 1.2.1]{butzer}
\begin{equation*}
a_k = \frac{1}{\pi} \int_0^{2\pi} \gamma(t) \cos(k t) \, dt, \quad b_k = \frac{1}{\pi} \int_0^{2\pi} \gamma(t) \sin(k t) \, dt.
\end{equation*}
If we have $\gamma \in L_2(\R/2\pi\Z,\R^3)$, $\gamma_N$ converges pointwise almost everywhere which is due to \name{Carleson}. $a_k$ and $b_k$ are called \name{Fourier} coefficients. Observe that the $2\pi$-periodicity is not a restriction at all. If $T>0$ and if $\gamma$ is $T$-periodic, then $\tilde{\gamma}(x):=\gamma(\frac{T}{2\pi} x)$ is $2\pi$-periodic. In our scenario we are not interested in translations of the knots. Therefore, we always set $a_0:=0$. This way our approximation space is $2N$ dimensional, with respect to $\R^3$.

To get a simpler notation we introduce the following abbreviation for the basis functions
\begin{equation}
\varphi_l(x) :=
\begin{cases}
\cos(kx), & l=2k-2\\
\sin(kx), & l=2k-1
\end{cases}
\quad \text{for $k=1,\dots,N$}.
\end{equation}
This way we get $\varphi_l$ for $l=0,\dots,2N-1$ and we can rewrite the \name{Fourier} approximation
\begin{equation*} \
\gamma_N (x) = \sum_{l=0}^{2N-1} c_l\, \varphi_l(x),
\end{equation*}
where we have for the coefficients
\begin{equation*}
c_l = \frac{1}{\pi} \int_0^{2\pi} \gamma(t)\;\varphi_l\!\kla{t} \, dt.
\end{equation*}
Moreover, we have (cf. \cite[7.9 Beispiel]{alt}) for $a,b=0,\dots,2N-1$
\begin{equation} \label{l2ortho}
\left\langle \varphi_a, \varphi_b \right\rangle_{L^2} =
\begin{cases}
\pi, &a=b\\
0, &a\neq b.
\end{cases}
\end{equation}

Later on we want to use an equidistant distribution on the interval $[0,2\pi]$ into $M\in\N$ segments for the discretization. If we define the distance between two points $h:=\frac{2\pi}{M}$, the points are $x_i=ih$ for $i=0,\dots,M$. Furthermore, we define
\begin{equation} \label{defdiscbasisq}
q_i^l :=
\begin{cases}
\cos(kih), & l=2k-2,\\
\sin(kih), & l=2k-1,
\end{cases}
\quad \text{for $l=0,\dots,2N-1$ and $i=0,\dots,M-1$},
\end{equation}
respectively
\begin{equation*}
{q'}_i^l := \varphi'_l(x_i) \quad\text{and}\quad {q''}_i^l := \varphi''_l(x_i).
\end{equation*}

A very common approximation space is the space of piecewise-linear functions. This is not an option here, since those functions have infinite energy. One could also try to use splines instead of \name{Fourier} curves. This way one would be able to create knots, which are in $C^{1,1}$ and not in $C^\infty$. However, it is well known that \name{Fourier} knots can be very close to those configurations and it is often reasonable to work with smooth functions. Another advantage of spline curves would be the fact that their basis functions have compact support. However, \name{Fourier} knots need very few coefficients in order to represent a quite large range of knots. For instance a circle which is supposed to be a global minimizer only needs one pair of \name{Fourier} coefficients.

\section{Numerical integration}
The simplest way to numerically integrate a function $f:[a,b] \rightarrow \R$ over $[a,b]$ is using the \emph{trapezoidal rule}. Let $M \in \N$ and $x_i \in [a,b]$ for $i=0,\dots, M$, where $x_0=a$, $a_M=b$ and $x_{j-1}-x_j =: h > 0$ for $j=1,\dots, M$. Then we approximate the integral by
\begin{equation*}
T_M(f) := h \kla{\frac{1}{2} f(a) + \sum_{k=1}^{M-1} f(x_k) + \frac{1}{2} f(b)}.
\end{equation*}
In our case we can simplify this a bit since we are considering closed curves defined on $[0,2\pi]$. Let $h := \frac{2\pi}{M}$ and 
\begin{equation} \label{trap}
T_M(\gamma) :=  h \sum_{k=0}^{M-1} \gamma(kh)
\end{equation}
which can be interpreted as the average or the \emph{centre of mass} of the integration points. These points are the images of equidistantly distributed points on the interval. Although it is a simple technique it can be a very efficient under certain conditions. We take the following theorem from \cite[Theorem 1]{weideman}
\begin{thm} \mylabel{trapErrEst}
Let $m \geq 0$, $M \geq 1$, and define $h:=2\pi/M$, $x_j = jh$ for $j = 0,1,\dots,M$. Further assume that $f(x)$ is $2m+2$ times continuously differentiable on $\I$ for some $m \geq 0$. Then, for the error in the trapezoidal rule \eqref{trap},
\begin{align*}
I(f) - T_M(f) = & - \suml_{k=1}^m \frac{B_{2k}}{(2k)!} \, h^{2k} \left[f^{(2k-1)} (2\pi) - f^{(2k-1)} (0)\right]\\
& - 2\pi \, h^{2m+2} \frac{B_{2m+2}}{(2m+2)!} f^{(2m+2)}(\xi),
\end{align*}
for some $\xi \in \I$. The $B_k$ are the Bernoulli numbers (see for instance \cite{number}).
\end{thm}
Since we are dealing with periodic functions the sum on the right-hand size vanishes and it only remains a high order term. To the best of our knowledge, the trapezoidal rule is the optimal choice for one-dimensional integrals in our situation. Moreover, the usability of \thref{trapErrEst} can be seen by the following examples
\begin{exmp} \mylabel{exmptrapExact}
\begin{enumabc}
\item If we integrate a trigonometric polynomial
\begin{equation*}
\int_0^{2\pi} \sum_{k=1}^N \kla{a_k \cos(kx) + b_k \sin(kx)} \, dx = 0
\end{equation*}
the trapezoidal rule is exact. For any $n\in\N$ the function $\gamma_N$ is $2n$ times continuously differentiable. We have
\begin{align*}
\gamma_N'(x) &= \sum_{k=1}^N \kla{-k a_k \sin(kx) + k b_k \cos(kx)}\\
\gamma_N''(x) &= \sum_{k=1}^N (-k^2) \kla{a_k \cos(kx) + b_k \sin(kx)}
\end{align*}
and therefore,
\begin{equation*}
\babs{\gamma^{(2n)}(x)} \leq \sum_{k=1}^N k^{2n} \kla{\abs{a_k}+\abs{b_k}} \leq N^{2n} \sum_{k=1}^N \kla{\abs{a_k}+\abs{b_k}}.
\end{equation*}
Now we use \cite[(2.2)]{number}
\begin{equation*}
(-1)^{n+1} \frac{B_{2n} (2\pi)^{2n}}{(2n)!} = 2\, \zeta(2n),
\end{equation*}
where $\zeta$ is the \emph{\name{Riemann} zeta function}. Moreover, we need the well known fact that $\zeta(\sigma)\rightarrow 1$ if $\sigma>1$ tends to $\infty$. We shortly recall the proof, cf. for instance \cite{fwerner}. Let $N\in\N$ and $\sigma>1$
\begin{alignat*}{2}
1 \leq \sum_{n=1}^N \frac{1}{n^\sigma} &= 1 + \sum_{n=2}^N \frac{1}{n^\sigma} &\;=\;& \sum_{n=1}^N \int_{n-1}^n \frac{1}{n^\sigma} \, dx\\
&\leq 1 + \sum_{n=1}^N \int_{n-1}^n \frac{1}{x^\sigma} \, dx &\;=\;& 1 + \int_1^N \frac{1}{n^\sigma} \, dx\\
&= 1 + \left[\frac{1}{1-\sigma}\,x^{1-\sigma} \right]_1^N &\;=\;& 1 + \frac{1}{1-\sigma} N^{1-\sigma} + \frac{1}{\sigma-1}.
\end{alignat*}
Consequently,
\begin{equation*}
1 \leq \sum_{n=1}^\infty \frac{1}{n^\sigma} \leq 1 + \frac{1}{\sigma-1} \qquad \text{and} \qquad \lim_{\sigma\rightarrow\infty} \sum_{n=1}^\infty \frac{1}{n^\sigma} = 1.
\end{equation*}
Now we conclude for $M>N$
\begin{equation*}
\abs{I(\gamma_N)-T_M(\gamma)} = 4 \pi\, \zeta(2n) \frac{\abs{\gamma^{(2n)}(\zeta)}}{M^{2n}} \leq C \kla{\frac{N}{M}}^{2n} \zeta(2n) \xrightarrow{n \rightarrow \infty} 0.
\end{equation*}
\item The same is true for functions like ($j,k\in\menge{1,\dots,N}$)
\begin{equation*}
f(x)=\cos(jx)\cos(kx),\quad f(x)=\sin(jx)\cos(kx)\quad\text{or}\quad f(x)=\sin(jx)\sin(kx).
\end{equation*}
We have
\begin{align*}
\frac{d}{dx} \sin(jx)\cos(kx) &= j\,\cos(jx)\cos(kx)-k\,\sin(jx)\sin(kx)\\
\frac{d}{dx} \sin(jx)\sin(kx) &= j\,\cos(jx)\sin(kx)+k\,\sin(jx)\cos(kx)\\
\frac{d}{dx} \cos(jx)\cos(kx) &= -j\,\sin(jx)\cos(kx)-k\,\cos(jx)\sin(kx)
\end{align*}
and therefore,
\begin{equation*}
\babs{f^{(2n)}(x)} \leq \kla{2N}^{2n}.
\end{equation*}
Analogous to the example above we obtain that for $M>2N$ the trapezoidal rule is exact.
\end{enumabc}
\end{exmp}
\begin{exmp} \mylabel{exmpCinfty}
However, even if the integrand is a $2\pi$-periodic function of class $C^\infty$, it is possible that one needs numerous integration points in order to get a good approximation of the integral. Consider
\begin{equation}
f(x) := \exp\kla{\sin(20x)+\cos(20x)+\sin(15x)+\cos(15x)}.
\end{equation}
Then we have for the trapezoidal rule

\begin{tabular}{lcccc}
number of integration points: & $50$ & $100$ & $200$ & $500$ \\
approximate error: & $2.37422$ & $0.02226$ & $0.21448\,10^{-5}$ & $0.25435\,10^{-24}$.
\end{tabular}
\end{exmp}

A suitable option to compute the triple integrals we are considering here, is to apply the trapezoidal rule to each integral. We can do this, since we are able to compute the value of the integrand for points where two or all three points coincide. Otherwise, techniques to avoid these triples as for instance the choice of different integration schemes for each integral would have been necessary. However, there are other methods which could be faster and could lead to a better convergence behaviour. These ideas are due to \name{Wo\'{z}niakowski} (personal communication, October 2011). Since the trapezoidal rule works fine for one-dimensional integrals, it could be used with the method of \emph{Sparse Grids}, where only a certain subset of the ``full grid'' of tuples of integration points is used. Other alternatives could be the \emph{(quasi) Monte Carlo} method.
%
\section{Discretization in time}
In order to discretize our problem in time we introduce discrete time steps, by adding a new index for each time step $m$
\begin{equation} \label{timebasis}
\gamma^m = \sum_{i=0}^{2N-1} c_i^m\, \varphi_i \; \in \R^3.
\end{equation}
Now we have to decide for each $\gamma$ in the discretization, which time step should be chosen. Our goal is to solve a linear system of equations for each time step $m$ and each scalar component $\ell=1,2,3$:
\begin{equation*}
S^m\,\overrightarrow{c}^{m+1}_\ell = \overrightarrow{r}^m_\ell, \qquad \overrightarrow{c}^{m+1}_\ell = \kla{\kla{c^{m+1}_i}_\ell}_{i=0,\cdots,2N-1},
\end{equation*}
where $S$ denotes the $2N\times 2N$-matrix corresponding to the linear system of equations. To do so we use a so-called \emph{implicit single step} scheme. We substitute $\dfrac{d}{dt} \gamma$ by $\dfrac{\gamma^{m+1}-\gamma^m}{\tau}$ and for the first variation we choose the next time step, i.e. $\gamma^{m+1}$ where $\gamma$ appears linearly and the current time step, i.e. $\gamma^m$ otherwise. This is a \emph{single step} scheme, since we only use data from the current time step while considering the next step and \emph{implicit}, because the successive time step also appears outside the difference quotient. Now we consider a small example that clarifies the procedure. For $\ell=1,2,3$ we test with $\varphi_a e_\ell$
\begin{multline*}
\sum_{i} h\,\abs{\gamma'(x_i)} \partial_t \gamma_\ell(x_i) \varphi_a(x_i)\\
+\sum_{i,j,k} -6 h^3 p\,\frac{\abs{\gamma'(x_i)} \abs{\gamma'(x_j)} \abs{\gamma'(x_k)}}{R^p(\gamma(x_i),\gamma(x_j),\gamma(x_k))} \frac{(\gamma_\ell(x_i)-\gamma_\ell(x_j))\,\varphi_a(x_i)}{\abs{\gamma(x_i)-\gamma(x_j)}^2}+\,(\dots) = 0
\end{multline*}
Applying the implicit single step scheme we conclude
\begin{multline*}
\sum_{i} h\,\abs{\gamma'^{m}(x_i)} \gamma_\ell^{m+1}(x_i) \varphi_a(x_i)\\
+\tau \Bigl( \sum_{i,j,k} -6 h^3 p\,\frac{\abs{\gamma'^{m}(x_i)} \abs{\gamma'^{m}(x_j)} \abs{\gamma'^{m}(x_k)}}{R^p(\gamma^{m}(x_i),\gamma^{m}(x_j),\gamma^{m}(x_k))} \frac{(\gamma_\ell^{m+1}(x_i)-\gamma_\ell^{m+1}(x_j))\,\varphi_a(x_i)}{\abs{\gamma^{m}(x_i)-\gamma^{m}(x_j)}^2}\\
+\,(\dots) \Bigr) = \sum_{i} h\,\abs{\gamma'^{m}(x_i)} \gamma_\ell^{m}(x_i) \varphi_a(x_i).
\end{multline*}
As we can write the next time step of $\gamma$ like this
\begin{equation*}
\gamma_\ell^{m+1} = \suml_{b=0}^{2N-1} \kla{c_b^{m+1}}_\ell \; \varphi_b,
\end{equation*}
we obtain that the left-hand size is actually a matrix vector multiplication and the entries of the matrix $S$ are
\begin{multline} \label{matCompAB}
\sum_{i} h\,\abs{\gamma'(x_i)} \varphi_b(x_i) \varphi_a(x_i)\,+\\
\tau \Bigl( \sum_{i,j,k} -6 h^3 p\,\frac{\abs{\gamma'(x_i)} \abs{\gamma'(x_j)} \abs{\gamma'(x_k)}}{R^p(\gamma(x_i),\gamma(x_j),\gamma(x_k))} \frac{(\varphi_b(x_i)-\varphi_b(x_j))\,\varphi_a(x_i)}{\abs{\gamma(x_i)-\gamma(x_j)}^2} +\,(\dots) \Bigr).
\end{multline}
Therefore, this matrix can be composed as $A+\tau B$, where $A$ and $B$ are symmetric matrices. The right-hand side stays the same and hence we get 3 symmetric $2N \times 2N$ linear systems of equations. As $N$ is quite small and the matrix $S$ is dense we use \lapack{} to solve the linear system of equations in every time step. $S$ is dense, since our basis functions do \emph{not} have compact support. Moreover, the energy effects the curve globally. Consider for instance the following expression
\begin{equation*}
\kla{\varphi_b(x_i) \varphi_a(x_i) - \varphi_b(x_j) \varphi_a(x_i)}.
\end{equation*}
If the basis functions had compact support, then the sum over $i$ of the first  summand would vanish if the distance of $a$ and $b$ is sufficiently large. However, this is not true for the sum over $i$ and $j$ of the second summand, as for every $a$ and $b$ there are $i$ and $j$ such that $x_i$ is in the support of $\varphi_a$ and $x_j$ is in the support of $\varphi_b$. Observe that the number of needed summations would be smaller. Nevertheless, later on we present some optimisations that also decrease the number of operations. 

Another effect of using basis functions with compact support, e.g. splines, is that we would need much more coefficients in order to represent our knots. Hence, it could be the case that we need to solve large linear systems of equations and therefore, it would be reasonable to use techniques like hierarchical matrices \cite{hMat}.

\section{The full discretization}
Assume that we have to compute an $M$-dimensional array $a$ or an $M\times M$ triangular matrix $b$ with
\begin{align*}
a_i &= \suml^{M-1}_{\stackrel{k=0}{k\neq i}} \suml^{k-1}_{\stackrel{j=0}{j\neq i}} F(i,j,k) &&\;\text{for $i = 0, \cdots, M-1$},\\
b_{jk} &= \suml^{M-1}_{\stackrel{i=0}{i\notin \menge{j,k}}} F(i,j,k) &&\begin{array}{l}
\text{for $k = 0, \cdots, M-1$},\\ 
\text{for $j = 0, \cdots, k-1$},
\end{array}
\end{align*}
where $F$ is a function $\menge{0,\cdots, M-1}^3 \rightarrow \R$. At first we consider a short example of two pseudo-code fragments
\begin{center}
\begin{tabular}{c|c}
\begin{lstlisting}
for i = 0 to M-1
 for j = i+1 to M-1
  (...)
\end{lstlisting}&
\begin{lstlisting}
for j = 0 to M-1
 for i = 0 to j-1
  (...)
\end{lstlisting}
\end{tabular}
\end{center}
which obviously both do the same since we have $0\leq i<j\leq M-1$ inside this nested loop in each case. One could say that we have \emph{flipped} the summation order. Observe that on the left-hand side for $i=M-1$ the second loop is skipped since $M>M-1$ and therefore nothing happens. The same is true on the right-hand side for $j=0$, because $0>-1$. Now we come back to $a_i$ and $b_{jk}$ and obtain

\begin{tabular}{c|c|c}
\begin{lstlisting}
for k = 0 to M-1
 for j = 0 to k-1
  for i = 0 to j-1
   a[i] += F(i,j,k)
   b[j,k] += F(i,j,k)
\end{lstlisting}&
\begin{lstlisting}
for k = 0 to M-1
 for j = 0 to k-1
  for i = j+1 to k-1
   a[i] += F(i,j,k)
   b[j,k] += F(i,j,k)
\end{lstlisting}&
\begin{lstlisting}
for k = 0 to M-1
 for j = 0 to k-1
  for i = k+1 to M-1
   a[i] += F(i,j,k)
   b[j,k] += F(i,j,k)
\end{lstlisting}\\
\hline
\begin{lstlisting}
for k = 0 to M-1
 for j = 0 to k-1
  for i = 0 to j-1
   a[i] += F(i,j,k)
   b[j,k] += F(i,j,k)
\end{lstlisting}&
\begin{lstlisting}
for k = 0 to M-1
 for i = 0 to k-1
  for j = 0 to i-1
   a[i] += F(i,j,k)
   b[j,k] += F(i,j,k)
\end{lstlisting}&
\begin{lstlisting}
for i = 0 to M-1
 for k = 0 to i-1
  for j = 0 to k-1
   a[i] += F(i,j,k)
   b[j,k] += F(i,j,k)
\end{lstlisting}\\
\hline
\begin{lstlisting}
for k = 0 to M-1
 for j = 0 to k-1
  for i = 0 to j-1
   a[i] += F(i,j,k)
   b[j,k] += F(i,j,k)
\end{lstlisting}&
\begin{lstlisting}
for k = 0 to M-1
 for j = 0 to k-1
  for i = 0 to j-1
   a[j] += F(j,i,k)
   b[i,k] += F(j,i,k)
\end{lstlisting}&
\begin{lstlisting}
for k = 0 to M-1
 for j = 0 to k-1
  for i = 0 to j-1
   a[k] += F(k,i,j)
   b[i,j] += F(k,i,j)
\end{lstlisting}
\end{tabular}

where we flip the summation order $0,1$ and $2$ times from the first to the second row, like in the example above and perform a variable interchange $0,1$ and $2$ times from the second to the third row. Hence, there remains only one nested loop to compute all $a_i$ and all $b_{jk}$ which is faster than the direct solution.

In order to compute the full matrix $S$ we have 5-times nested loops, because we have to approximate the triple integral for each entry of the matrix. Since we want to use the trapezoidal rule to approximate each integral, we get triple sums for each entry. Therefore, it is crucial to do this very efficiently in order to get acceptable calculation times. On the one hand we have to exploit the symmetry of the matrix and of the integrand as much as possible. On the other hand it is import to reduce the level of nestedness for some parts of the calculation. This can be done by reducing the number of variables being applied to $\Phi$, which could break symmetries of the equation. Now an approach is presented where we combine both ideas in a reasonable way.

In the last chapter we computed a simplified version \eqref{simp1var} of the first variation of $\M_p$. However, we can even simplify more. For instance we continue with \eqref{scalarcrossshort} and by a small calculation we obtain
\begin{multline*}
\Bigl( (\gamma(t)-\gamma(s)) \wedge (\gamma(\sigma)-\gamma(s)) \Bigr) \cdot \Bigl( (\Phi(t)-\Phi(s)) \wedge (\gamma(\sigma)-\gamma(s)) + (\gamma(t)-\gamma(s)) \wedge (\Phi(\sigma)-\Phi(s)) \Bigr)\\
= \abs{\gamma(t)-\gamma(\sigma)}^2\,\gamma(s)\cdot\Phi(s) + \abs{\gamma(s)-\gamma(t)}^2\,\gamma(\sigma)\cdot\Phi(\sigma) + \abs{\gamma(s)-\gamma(\sigma)}^2\,\gamma(t)\cdot\Phi(t)\\
-(\gamma(s)-\gamma(\sigma))\cdot(\gamma(t)-\gamma(\sigma)) \kla{\gamma(s)\cdot\Phi(t)+\gamma(t)\cdot\Phi(s)}\\
-(\gamma(s)-\gamma(t))\cdot(\gamma(s)-\gamma(\sigma)) \kla{\gamma(t)\cdot\Phi(\sigma)+\gamma(\sigma)\cdot\Phi(t)}\\
-(\gamma(t)-\gamma(s))\cdot(\gamma(t)-\gamma(\sigma)) \kla{\gamma(\sigma)\cdot\Phi(s)+\gamma(s)\cdot\Phi(\sigma)}.\;\;\;\quad\qquad\qquad
\end{multline*}
Now we replace in \eqref{simp1var} the variables $s,t,\sigma$ by the discrete samples $x_i,x_j,x_k$ for $i,j,k=0,\dots,M-1$ and the triple integral by a triple sum, which is already the trapezoidal rule up to a constant factor. In analogy to using \name{Fubini}'s theorem in \eqref{fubinishort} we now change the order of the three-times nested summation, since we are dealing with finite sums here. Using the analogue of \eqref{fubinishort}, we obtain that the sum over $i\neq j\neq k\neq i$ of
\begin{multline*}
\frac{2^p \abs{\gamma'(x_i)} \abs{\gamma'(x_j)} \abs{\gamma'(x_k)}}{\abs{\gamma(x_i)-\gamma(x_j)}^p \abs{\gamma(x_j)-\gamma(x_k)}^p \abs{\gamma(x_i)-\gamma(x_k)}^p} \Biggl[ \babs{(\gamma(x_j)-\gamma(x_i))\wedge (\gamma(x_k)-\gamma(x_i))}^p\\
\Biggl( 3 \frac{\gamma'(x_i) \cdot \Phi'(x_i)}{\abs{\gamma'(x_i)}^2} -3p \frac{(\gamma(x_i)-\gamma(x_j)) \cdot (\Phi(x_i)-\Phi(x_j))}{\abs{\gamma(x_i)-\gamma(x_j)}^2} \Biggr)\\
+3p \babs{(\gamma(x_j)-\gamma(x_i))\wedge (\gamma(x_k)-\gamma(x_i))}^{p-2} \Biggl(  \abs{\gamma(x_j)-\gamma(x_k)}^2 \gamma(x_i)\cdot\Phi(x_i)\\
 - (\gamma(x_i)-\gamma(x_k))\cdot(\gamma(x_j)-\gamma(x_k)) \kla{\gamma(x_i)\cdot\Phi(x_j)+\gamma(x_j)\cdot\Phi(x_i)} \Biggr) \Biggr]
\end{multline*}
is the same as the triple sum from above with respect to distinct points. However, observe that the triple integral of the function over $(s,t,\sigma)$ would be infinite and in particular not the same as the triple integral \eqref{simp1var}, as \name{Fubini}'s theorem cannot be applied here. For the case that two points coincide and the third point is distinct, as already mentioned earlier, we only have to consider the integrand for $(s,t,t)$, since we can use the method described in \eqref{symmconv}. We can take \eqref{stt1var} and apply an analogue of \eqref{fubinitrick} for triple sums to obtain that for $i\neq j$ we have to sum up
\begin{multline*}
\frac{2^p}{\abs{\gamma(x_i)-\gamma(x_j)}^{2p}} \Biggl[ \Biggl( \frac{\abs{(\gamma(x_j)-\gamma(x_i))\wedge\gamma'(x_j)}^p}{\abs{\gamma'(x_i)} \abs{\gamma'(x_j)}^{p-2}} - (p-2) \frac{\abs{(\gamma(x_i)-\gamma(x_j))\wedge\gamma'(x_i)}^p \abs{\gamma'(x_j)}}{\abs{\gamma'(x_i)}^p}\\
+p \frac{\abs{(\gamma(x_i)-\gamma(x_j))\wedge\gamma'(x_i)}^{p-2} \abs{\gamma'(x_j)}}{\abs{\gamma'(x_i)}^{p-2}} \abs{\gamma(x_i)-\gamma(x_j)}^2 \Biggr) \gamma'(x_i) \cdot \Phi'(x_i) +\\
\Biggl( -2p \frac{\abs{(\gamma(x_j)-\gamma(x_i))\wedge\gamma'(x_j)}^p \abs{\gamma'(x_i)}}{\abs{\gamma'(x_j)}^{p-2} \abs{\gamma(x_i)-\gamma(x_j)}^2} +p \frac{\abs{(\gamma(x_j)-\gamma(x_i))\wedge\gamma'(x_j)}^{p-2} \abs{\gamma'(x_i)}}{\abs{\gamma'(x_j)}^{p-4}} \Biggl)\\
\Bigl( (\gamma(x_i)-\gamma(x_j)) \cdot (\Phi(x_i)-\Phi(x_j)) \Bigr) +\\
\Biggl( -p \frac{\abs{(\gamma(x_i)-\gamma(x_j))\wedge\gamma'(x_i)}^{p-2} \abs{\gamma'(x_j)}}{\abs{\gamma'(x_i)}^{p-2}} \Bigl( (\gamma(x_i)-\gamma(x_j)) \cdot \gamma'(x_i)\Bigr) \Biggr)\\
\Bigl( \gamma'(x_i) \cdot (\Phi(x_i)-\Phi(x_j)) + (\gamma(x_i)-\gamma(x_j)) \cdot \Phi'(x_i) \Bigr) \Biggr].
\end{multline*}
For the last step that all three points coincide in the point $s$ we can directly take the sum \eqref{sss1var} after substituting $s$ by $x_i$.

For simplification we introduce the following abbreviations
\begin{align*}
p'_i &\;:=\; \gamma'(x_i), \qquad p''_i \;:=\; \gamma''(x_i),\\
dp_{ij} &\;:=\; \gamma(x_i)-\gamma(x_j),\\
\mathit{X0}_{ijk} &\;:=\; (\gamma(x_j)-\gamma(x_i))\wedge (\gamma(x_k)-\gamma(x_i)),\\
\mathit{X1}_{ij} &\;:=\; (\gamma(x_i)-\gamma(x_j))\wedge\gamma'(x_i),\\
\mathit{X2}_{i} &\;:=\; \gamma'(x_i)\wedge\gamma''(x_i).
\end{align*}
Since $\abs{dp_{ij}}$ and $\abs{\mathit{X0}_{ijk}}$ are symmetric they are stored in triangular structures. After all we gain the following entries of the matrix $S$ for all rows $a$ and all columns $b$. Observe that $a$ is the index of the testing basis function $\varphi_a$ and that $b$ corresponds to the basis function $\varphi_b$ representing the current curve $\gamma^m$.
\begin{multline*}
\suml_{i=0}^{M-1} \Biggl( \sigma_i^1 \, q_i^a q_i^b + \sigma_i^2 \, {q'}_i^a {q'}_i^b + \sigma_i^3 \, {q''}_i^a {q''}_i^b + \sigma_i^4 \, ({q'}_i^a {q''}_i^b+{q''}_i^a {q'}_i^b) +\\
\suml_{j=i+1}^{M-1} \Bigl( \sigma_{ij}^5 \, (q_i^a q_j^b + q_j^a q_i^b) + \sigma_{ij}^6 \, ({q'}_i^a d_{ij}^b + {q'}_i^b d_{ij}^a) + \sigma_{ij}^7 \, (d_{ij}^a d_{ij}^b) \Bigr) \Biggr),
\end{multline*}
where $d_{ij}^c := q_i^c-q_j^c$. We collect the terms that belong to the different $\sigma^k$, $k=1,\dots,7$ expressions which are part of the discrete first variation. In addition, powers with exponent close $p$ are collected, resulting in a speed-up for large $p$.\\
The vector $\sigma^1$:
\begin{align}
\sigma_i^1 &= h^3 \tau \Bigg\{ \suml_{\stackrel{k=0}{k\neq i}}^{M-1} \suml_{\stackrel{j=0}{j\neq i}}^{k-1} \mathbf{2} \kla{3p \frac{2^p \abs{p'_i} \abs{p'_j} \abs{p'_k}}{\abs{dp_{ij}}^p \abs{dp_{jk}}^p \abs{dp_{ik}}^p} \abs{\mathit{X0}_{ijk}}^{p-2} \abs{dp_{jk}}^2} \Biggr\}\\
&= h^3 \tau \Biggl\{\suml_{\stackrel{k=0}{k\neq i}}^{M-1} \suml_{\stackrel{j=0}{j\neq i}}^{k-1} 24 p \kla{\frac{2 \abs{\mathit{X0}_{ijk}}}{\abs{dp_{ij}}\abs{dp_{jk}}\abs{dp_{ik}}}}^{p-2} \frac{\abs{p'_i}\abs{p'_j}\abs{p'_k}}{\abs{dp_{ij}}^2 \abs{dp_{ik}}^2} \Biggr\}.
\end{align}
The vector $\sigma^2$:
\begin{align}
\sigma_i^2 &=&& h^3 \tau \Biggl\{ \suml_{\stackrel{k=0}{k\neq i}}^{M-1} \suml_{\stackrel{j=0}{j\neq i}}^{k-1} \mathbf{2} \kla{\frac{2^p \abs{p'_i} \abs{p'_j} \abs{p'_k}}{\abs{dp_{ij}}^p\abs{dp_{jk}}^p\abs{dp_{ik}}^p} \abs{\mathit{X0}_{ijk}}^p \frac{3}{\abs{p'_i}^2}} + \notag\\
&&&\abs{p'_i}^{-(3p-3)} \kla{\frac{3-3p}{\abs{p'_i}^2} \abs{\mathit{X2}_i}^p + p \abs{\mathit{X2}_i}^{p-2} \abs{p''_i}^2} + \notag\\
&&&\sum_{\stackrel{j=0}{j \neq i}}^{M-1} \mathbf{3} \Biggl[ \frac{2^p}{\abs{dp_{ij}}^{2p}} \Bigl( \frac{\abs{\mathit{X1}_{ji}}^p}{\abs{p'_i} \abs{p'_j}^{p-2}} -(p-2) \frac{\abs{\mathit{X1}_{ij}}^p \abs{p'_j}}{\abs{p'_i}^p} +p \frac{\abs{\mathit{X1}_{ij}}^{p-2} \abs{p'_j}}{\abs{p'_i}^{p-2}} \abs{dp_{ij}}^2 \Bigr) \Biggr] \Biggr\} \notag\\
&=&& h^3 \tau \Biggl\{ (3-3p) \kla{\frac{\abs{\mathit{X2}_i}}{\abs{p'_i}^3}}^p \abs{p'_i} + p \kla{\frac{\abs{\mathit{X2}_i}}{\abs{p'_i}^3}}^{p-2} \frac{\abs{p''_i}^2}{\abs{p'_i}^3} + \notag\\
&&&\sum_{\stackrel{j=0}{j \neq i}}^{M-1} \Biggl( 3 \kla{\frac{2 \abs{\mathit{X1}_{ji}}}{\abs{dp_{ij}}^2 \abs{p'_j}}}^p \frac{\abs{p'_j}^2}{\abs{p'_i}} - 3(p-2) \kla{\frac{2\abs{\mathit{X1}_{ij}}}{\abs{dp_{ij}}^2 \abs{p'_i}}}^p \abs{p'_j} + 12 p \kla{\frac{2 \abs{\mathit{X1}_{ij}}}{\abs{dp_{ij}}^2 \abs{p'_i}}}^{p-2} \frac{\abs{p'_j}}{\abs{dp_{ij}}^2} \Biggr) + \notag\\
&&&\suml_{\stackrel{k=0}{k\neq i}}^{M-1} \suml_{\stackrel{j=0}{j\neq i}}^{k-1} \kla{6 \kla{\frac{2 \abs{\mathit{X0}_{ijk}}}{\abs{dp_{ij}}\abs{dp_{jk}}\abs{dp_{ik}}}}^p \frac{\abs{p'_j}\abs{p'_k}}{\abs{p'_i}}} \Biggr\}.
\end{align}
The vector $\sigma^3$:
\begin{align}
\sigma_i^3 &= h^3 \tau \Biggl\{ \abs{p'_i}^{-(3p-3)} \kla{p\,\abs{\mathit{X2}_i}^{p-2} \abs{p'_i}^2} \Biggr\} &=& h^3 \tau \Biggl\{ p \kla{\frac{\abs{\mathit{X2}_i}}{\abs{p'_i}^3}}^{p-2} \frac{1}{\abs{p'_i}} \Biggr\}.
\end{align}
The vector $\sigma^4$:
\begin{align}
\sigma_i^4 &= h^3 \tau \Biggl\{ \abs{p'_i}^{-(3p-3)} \kla{ -p\, \abs{\mathit{X2}_i}^{p-2} (p'_i \cdot p''_i)} \Biggr\} &=& h^3 \tau \Biggl\{ (-p) \kla{\frac{\abs{\mathit{X2}_i}}{\abs{p'_i}^3}}^{p-2} \frac{p'_i \cdot p''_i}{\abs{p'_i}^3} \Biggr\}.
\end{align}
The upper triangular matrix $\sigma^5$ stored without the diagonal entries:
\begin{align}
\sigma_{ij}^5 &= h^3 \tau \Biggl\{ \sum_{\stackrel{k=0}{k \notin \menge{i,j}}}^{M-1} \mathbf{2} \kla{ \frac{2^p \abs{p'_i}\abs{p'_j}\abs{p'_k}}{\abs{dp_{ij}}^p\abs{dp_{jk}}^p\abs{dp_{ik}}^p} (-3p) \abs{\mathit{X0}_{ijk}}^{p-2} (dp_{jk} \cdot dp_{ik})} \Biggr\} \notag\\
&= h^3 \tau \Biggl\{ \sum_{\stackrel{k=0}{k \notin \menge{i,j}}}^{M-1} (-24 p) \kla{\frac{2 \abs{\mathit{X0}_{ijk}}}{\abs{dp_{ij}}\abs{dp_{jk}}\abs{dp_{ik}}}}^{p-2} \frac{\abs{p'_i}\abs{p'_j}\abs{p'_k}}{\abs{dp_{ij}}^2 \abs{dp_{jk}}^2 \abs{dp_{ik}}^2} (dp_{jk} \cdot dp_{ik}) \Biggr\}.
\end{align}
The full matrix $\sigma^6$:
\begin{align}
\sigma_{ij}^6 &= h^3 \tau \Biggl\{ \mathbf{3} \Biggl[ \frac{2^p}{\abs{dp_{ij}}^{2p}} \kla{(-p) \frac{\abs{\mathit{X1}_{ij}}^{p-2} \abs{p'_j}}{\abs{p'_i}^{p-2}} (dp_{ij} \cdot p'_i)} \Biggr] \Biggr\} \notag\\
&= h^3 \tau \Biggl\{ (-12 p) \kla{\frac{2 \abs{\mathit{X1}_{ij}}}{\abs{dp_{ij}}^2 \abs{p'_i}}}^{p-2} \frac{\abs{p'_j}}{\abs{dp_{ij}}^4} (dp_{ij} \cdot p'_i) \Biggr\}.
\end{align}
The upper triangular matrix $\sigma^7$ stored without the diagonal entries:
\begin{align}
\sigma_{ij}^7 &=&& h^3 \tau \Biggl\{ \mathbf{3} \Biggl[ \frac{2^p}{\abs{dp_{ij}}^{2p}} \kla{-2p \frac{\abs{p'_i} \abs{\mathit{X1}_{ji}}^p}{\abs{p'_j}^{p-2} \abs{dp_{ij}}^2} +p \frac{\abs{p'_i} \abs{\mathit{X1}_{ji}}^{p-2}}{\abs{p'_j}^{p-4}}}+ \notag\\
&&& \frac{2^p}{\abs{dp_{ij}}^{2p}} \kla{-2p \frac{\abs{p'_j} \abs{\mathit{X1}_{ij}}^p}{\abs{p'_i}^{p-2} \abs{dp_{ij}}^2} +p \frac{\abs{p'_j} \abs{\mathit{X1}_{ij}}^{p-2}}{\abs{p'_i}^{p-4}}} \Biggr] + \notag\\
&&& \sum_{\stackrel{k=0}{k \notin \menge{i,j}}}^{M-1} \mathbf{2} \kla{ \frac{2^p \abs{p'_i}\abs{p'_j}\abs{p'_k}}{\abs{dp_{ij}}^p \abs{dp_{jk}}^p \abs{dp_{ik}}^p} \abs{\mathit{X0}_{ijk}}^p \frac{-3p}{\abs{dp_{ij}}^{p-4}}} \Biggr\} \notag\\
&=&& h^3 \tau \Biggl\{  (-6p) \Biggl[ \kla{\frac{2 \abs{\mathit{X1}_{ji}}}{\abs{dp_{ij}}^2 \abs{p'_j}}}^p \frac{\abs{p'_j}^2 \abs{p'_i}}{\abs{dp_{ij}}^2} + \kla{\frac{2 \abs{\mathit{X1}_{ij}}}{\abs{dp_{ij}}^2 \abs{p'_i}}}^p \frac{\abs{p'_i}^2 \abs{p'_j}}{\abs{dp_{ij}}^2} \Biggr] + \notag\\
&&& 12p \Biggl[ \kla{\frac{2 \abs{\mathit{X1}_{ji}}}{\abs{dp_{ij}}^2 \abs{p'_j}}}^{p-2} \frac{\abs{p'_j}^2 \abs{p'_i}}{\abs{dp_{ij}}^4} + \kla{\frac{2 \abs{\mathit{X1}_{ij}}}{\abs{dp_{ij}}^2 \abs{p'_i}}}^{p-2} \frac{\abs{p'_i}^2 \abs{p'_j}}{\abs{dp_{ij}}^4} \Biggr] + \notag\\
&&& \sum_{\stackrel{k=0}{k \notin \menge{i,j}}}^{M-1} (-6p) \kla{\frac{2 \abs{\mathit{X0}_{ijk}}}{\abs{dp_{ij}}\abs{dp_{jk}}\abs{dp_{ik}}}}^p \frac{\abs{p'_i}\abs{p'_j}\abs{p'_k}}{\abs{dp_{ij}}^2} \Biggr\}.
\end{align}
The factor $\mathbf{2}$ comes from the symmetry in $i$ and $j$ in each case and the factor $\mathbf{3}$ from the symmetry of the integrand, see \eqref{symmconv}.

For the calculation of $\M_p$ in \mppack{} we use the following trick (\name{H. Wo\'{z}niakowski}, personal communication, October 2011). Let $y_i\in\R$, $i=1,\dots,N$, then we have for $y^\ast:=\max_{i=1,\dots,N}\abs{y_i}$
\begin{equation*}
y^\ast \sqrt{\sum_{k=1}^N \kla{\frac{y_k}{y^\ast}}^2} = \sqrt{\sum_{k=1}^N y_k^2}
\end{equation*}
which avoids buffer overflows for very large $p$. The disadvantage is that we have to find this maximum at first. However, one should easily be possible to adapt this for computing the first variation.

\section{Enhancing calculation speed}

During the implementation many different methods to increase performance has been applied. We now have a closer look on purely mathematical improvements, optimisations with respect to the used data structures and enhancements by C++ -- programming issues.

\paragraph{Mathematical improvements}
Let $i\in\menge{0,\dots,M-1}$ and $k,l \in \N$ with $1 \leq k \leq l \leq N$
\begin{alignat*}{3}
& \cos(kx_i) \cos(lx_i) & \; = \frac{1}{2} \Bigl( & &  \cos((l-k)x_i) &+ \cos((l+k)x_i) \Bigr)\\
& \sin(kx_i) \sin(lx_i) & \; = \frac{1}{2} \Bigl( & &  \cos((l-k)x_i) &- \cos((l+k)x_i) \Bigr)\\
& \sin(kx_i) \cos(lx_i) & \; = \frac{1}{2} \Bigl( & & -\sin((l-k)x_i) &+ \sin((l+k)x_i) \Bigr)\\
& \cos(kx_i) \sin(lx_i) & \; = \frac{1}{2} \Bigl( & &  \sin((l-k)x_i) &+ \sin((l+k)x_i) \Bigr),
\end{alignat*}
due to \eqref{trigsums}, \eqref{trigdiffs}, \eqref{trigsumc} and \eqref{trigdiffc}. The only factors that appear on the right-hand-side are $(l-k)$ and $(l+k)$ and due to $1 \leq k \leq l \leq N$ we have in the first case $0 \leq l-k \leq N-1$ and in the second $2 \leq l+k \leq 2N$. We consider those terms in the first variation where two basis functions appear with the same integration variable $x_i$, namely $q_i^a q_i^b$, ${q'}_i^a {q'}_i^b$, ${q''}_i^a {q''}_i^b$ and ${q'}_i^a {q''}_i^b+{q''}_i^a {q'}_i^b$. We can use these formulas in order to shorten the calculation time and additionally receive a higher accuracy. Instead of calculating $\sum_{i=0}^{M-1} \sigma_i^1\, q_i^a q_i^b$ for all $a, b=1,\dots, 2N$, that is $4 N^2$ times, for example, we can now compute $\sum_{i=0}^{M-1} \sigma_i^1\, \cos(kih)$ and $\sum_{i=0}^M \sigma_i^1\, \sin(kih)$ for $k=0,\dots,2N$. Since $\sin(0)=0$ we can skip one and eventually have to do $4N+1$ calculations. We save those values in arrays named $\vartheta_\iota$, for some index $\iota$. Afterwards we only have to use the formulas above to get the values for the products of basis functions like $q_i^a q_i^b$.

\paragraph{Better data structures}
We want to store $\abs{\mathit{X0}_{ijk}}$ for $i,j,k=0,\dots,M-1$ efficiently. It is symmetric in $i,j,k$ and zero if two indices are equal. Therefore, we only have to save values for $i<j<k$. Next we compute how many values we have to store
\begin{align*}
\sum_{k=2}^{M-1} \sum_{j=1}^{k-1} j &= \frac{1}{2} \kla{\sum_{k=2}^{M-1} k^2 -\sum_{k=2}^{M-1} k} \;=\; \frac{1}{2} \kla{\sum_{k=1}^{M-1} k^2 -\sum_{k=1}^{M-1} k}\\
&= \frac{1}{2} \kla{\frac{(M-1)M(2M-1)}{6}-\frac{(M-1)M}{2}} \;=\; \frac{(M-2)(M-1)M}{6},
\end{align*}
which is less than $\frac{1}{6}$ of the amount for the full structure $M^3$. Instead of using a nested array structure it is more efficient to save these values in a one-dimensional array. Consequently, the triple index $(i,j,k)$ matches the following index of this array
\begin{equation*}
i+\sum_{\iota=1}^{j-1} \iota+\sum_{\kappa=2}^{k-1}\sum_{\iota=1}^{\kappa-1} \iota \;=\; i+\frac{(j-1)j}{2}+\frac{(k-2)(k-1)k}{6}.
\end{equation*}
Analogously a symmetric matrix with a non-zero diagonal can be stored in an array with $\frac{M(M+1)}{2}$ elements and the double index $(i,j)$ with $i\leq j$ correspond to the index
\begin{equation*}
i + \frac{j(j+1)}{2}.
\end{equation*}

\paragraph{\cpp{} programming issues}
An important issue in \cpp{} programming is to avoid overhead by objective orientated programming. However, there is more that we taken into account.
\begin{exmp} \mylabel{altpow}
If we want to calculate powers of numbers close to $1$ to a fractional exponent, for instance $\frac{7}{2}$, using the standard \cpp{} \texttt{pow} function might be extremely slow. Applying the flow on the unit circle with $p=3.5$ we get
a calculation time of more than $1$h and $24$min. But by using a self-written variant, that always uses the fact that for $a>0$ and $b\in\R$
\begin{equation*}
a^b \quad = \quad \exp( b\; \ln(a) ),
\end{equation*}
we get a very significant speed-up
as this only takes about $10$s. However, for $p=3$ it is better to choose the standard method because the self-written solution takes 
about $11$s, but if we switch back to the standard method it only takes
$1.5$s. The situation for $p=50$ is $12.97$s with the self-written and $2.23$s for the standard \texttt{pow} function.
\end{exmp}
Since the application only needs a minimal amount of memory we can run multiple copies in parallel. However, it could additionally be beneficial to use real parallelisation. For instance \name{M. Dryja} (personal communication, October 2011) suggested to use \emph{PETSc}.

\section{An adaptive choice of time step size}
The stability of a solver for a linear system of equations strongly depends on the condition number of a matrix $A$. If we consider the spectral norm of $A$ the \emph{condition number} is defined like this, see \cite[(3.4)]{quarteroni}
\begin{equation}
K_2(A) := \norm{A}_2\,\norm{A^{-1}}_2.
\end{equation}
Furthermore, if $A$ is symmetric and positive definite, i.e. all eigenvalues are positive, we have the following useful representation, see \cite[(3.5)]{quarteroni}
\begin{equation}
K_2(A) = \frac{\lambda_A^{\max}}{\lambda_A^{\min}},
\end{equation}
where we call $\lambda_A^{\min}$ the minimal and $\lambda_A^{\max}$ the maximal eigenvalue of $A$. Hence, we need to estimate the eigenvalues of a matrix in order to control its condition number. A symmetric matrix $A$ has exactly $n$ real eigenvalues (up to multiplicity) to be denoted by $(\lambda_A^i)_{i=1,\dots,n}$. We have already seen that the matrix \eqref{matCompAB} can be decomposed into $A+\tau B$, where $A$ and $B$ are symmetric matrices and $\tau>0$. Moreover, $A$ is positive definite, which we will see next. Recall that $A=\sum\limits_i h\,\abs{\gamma'(x_i)}\,\varphi_b(x_i)\varphi_a(x_i)$ and for $0\neq v\in \R^{2N}$ we have
\begin{align*}
v^\text{T} A v &= \sum_{a=0}^{2N-1} \sum_{b=0}^{2N-1} \sum_i h\,v_a v_b\,\abs{\gamma'(x_i)}\,\varphi_b(x_i)\varphi_a(x_i)\\
&\geq C\,\sum_{a=0}^{2N-1} \sum_{b=0}^{2N-1} \int_0^{2\pi} v_a v_b\,\varphi_b(x)\varphi_a(x) \, dx \stackrel{\eqref{l2ortho}}{=} C\,\pi\sum_{a=0}^{2N-1} v_a^2 >0,
\end{align*}
which is correct because $\gamma$ is regular and the trapezoidal rule is exact for these integrands, see \thref{exmptrapExact}(b).\\
Consequently, for the eigenvalues of the matrices $A$ and $B$ we have $0 < \lambda_A^{\min} \leq \lambda_A^{\max}$ and $\lambda_B^{\min} \leq \lambda_B^{\max}$. Furthermore, for $n=2N$ there exists an orthonormal basis of the eigenspaces of $A$ and $B$ which we call $(x_i)_{i=1,\dots,n}$ and $(\tilde{x_i})_{i=1,\dots,n}$ respectively. Let $\hat{x}$ be an eigenvector of the matrix $A+\tau B$ with respect to the eigenvalue $\lambda$. Therefore, $x := \frac{\hat{x}}{\norm{\hat{x}}}$ is also a eigenvalue. Let $(\alpha_i)_{i=1,\dots,n}$ and $(\tilde{\alpha_i})_{i=1,\dots,n}$ be the coefficients of $x$ concerning the basis $(x_i)$ and $(\tilde{x_i})$. Due to the orthonormality we gain
\begin{equation*}
1 = \norm{x}^2 = \kla{\sum_{i=1}^n \alpha_i x_i}^T \kla{\sum_{i=1}^n \alpha_i x_i} = \sum_{i=1}^n \alpha_i^2 \qquad \text{and} \qquad \sum_{i=1}^n \tilde{\alpha_i}^2 = 1.
\end{equation*}
Now we conclude
\begin{eqnarray*}
\lambda & = & \lambda \norm{x}^2 = x^T \lambda x = x^T \kla{A+\tau B} x = x^T A x + \tau \, x^T B x\\
& = & \kla{\sum_{i=1}^n \alpha_i x_i}^T A \kla{\sum_{i=1}^n \alpha_i x_i} + \tau \kla{\sum_{i=1}^n \tilde{\alpha_i} \tilde{x_i}}^T B \kla{\sum_{i=1}^n \tilde{\alpha_i} \tilde{x_i}}\\
& = & \kla{\sum_{i=1}^n \alpha_i x_i}^T \kla{\sum_{i=1}^n \alpha_i \lambda_A^i x_i} + \tau \kla{\sum_{i=1}^n \tilde{\alpha_i} \tilde{x_i}}^T \kla{\sum_{i=1}^n \tilde{\alpha_i} \lambda_B^i \tilde{x_i}}\\
& = & \sum_{i=1}^n \alpha_i^2 \lambda_A^i + \tau \sum_{i=1}^n \tilde{\alpha_i}^2 \lambda_B^i.
\end{eqnarray*}
This is equivalent to
\begin{equation*}
\lambda \in \left[ \lambda_A^{\min} + \tau \lambda_B^{\min}, \lambda_A^{\max} + \tau \lambda_B^{\max} \right].
\end{equation*}
The aim of the algorithm we want to design for the adaptive choice of time step size is the following. For $\eps>0$ and we choose $\tau\geq 0$ such that
\begin{enumerate}
\item $A+\tau B$ is positive definite
\item $K_2(A+\tau B) \leq (1+\eps) K_2(A)$
\end{enumerate}
The first claim is fulfilled if the minimal eigenvalue of $A+\tau B$ is positive in particular if $0<\lambda_A^{\min}+\tau\lambda_B^{\min}$. Consequently, we require
\begin{equation}
\begin{cases} \label{HMpd}
\tau \geq 0, & \lambda_B^{\min} \geq 0\\
0 \leq \tau < \dfrac{\lambda_A^{\min}}{-\lambda_B^{\min}}, & \lambda_B^{\min} < 0.
\end{cases}
\end{equation}
For the second claim we consider
\begin{equation*}
K_2(A+\tau B) = \frac{\lambda_{A+\tau B}^{\max}}{\lambda_{A+\tau B}^{\min}}\leq \frac{\lambda_{A}^{\max}+\tau\lambda_{B}^{\max}}{\lambda_{A}^{\min}+\tau\lambda_{B}^{\min}}\leq (1+\eps)\frac{\lambda_{A}^{\max}}{\lambda_{A}^{\min}}=(1+\eps)K_2(A)
\end{equation*}
as $A+\tau B$ is positive definite. Therefore, we demand
\begin{equation}
\begin{cases} \label{HKs}
0 \leq \tau \leq \dfrac{\eps \lambda_A^{\min} \lambda_A^{\max}}{\lambda_B^{\max}\lambda_A^{\min}-(1+\eps) \lambda_A^{\max} \lambda_B^{\min}}, & \lambda_B^{\max}\lambda_A^{\min}-(1+\eps) \lambda_A^{\max} \lambda_B^{\min} > 0\\
\tau \geq 0, & \lambda_B^{\max}\lambda_A^{\min}-(1+\eps) \lambda_A^{\max} \lambda_B^{\min} \leq 0.
\end{cases}
\end{equation}
Now let $T:=\lambda_B^{\max}\lambda_A^{\min}-(1+\eps) \lambda_A^{\max} \lambda_B^{\min}$. At first we consider the case $T>0$. If $\lambda_B^{\max}\geq 0$ in addition, we may set
\begin{equation*}
0 \leq \tau:= \dfrac{\eps \lambda_A^{\min} \lambda_A^{\max}}{T},
\end{equation*}
as \eqref{HKs} and \eqref{HMpd} are satisfied. The latter is true, since assuming $\lambda_B^{\min} < 0$ we have
\begin{equation*}
0 \leq \tau = \dfrac{\eps \lambda_A^{\min} \lambda_A^{\max}}{T} \leq \dfrac{\eps\lambda_A^{\min} \lambda_A^{\max}}{-\lambda_A^{\max} \lambda_B^{\min}-\eps\lambda_A^{\max} \lambda_B^{\min}} <
\dfrac{\eps \lambda_A^{\min} \lambda_A^{\max}}{-\eps\lambda_A^{\max} \lambda_B^{\min}} =
\dfrac{\lambda_A^{\min}}{-\lambda_B^{\min}}.
\end{equation*}
However, if $\lambda_B^{\min} \geq 0$ then \eqref{HMpd} is trivial. Now we come to the case $\lambda_B^{\max} < 0$. This implies $\lambda_B^{\min} < 0$ and we define
\begin{equation*}
0 \leq \tau := \min \left\{ \dfrac{\eps \lambda_A^{\min} \lambda_A^{\max}}{T}, \dfrac{\lambda_A^{\min}}{-\lambda_B^{\min}} \right\}.
\end{equation*}
\begin{exmp}
Assume $\lambda_B^{\min}=\lambda_B^{\max}$ and $\lambda_A^{\max}=2$. If we consider
\begin{equation*}
\dfrac{\eps \lambda_A^{\min} \lambda_A^{\max}}{T} = \dfrac{2 \eps\,\lambda_A^{\min}}{-\lambda_B^{\min} (2+2\eps-\lambda_A^{\min})} =: c\,\frac{\lambda_A^{\min}}{-\lambda_B^{\min}},
\end{equation*}
we have $c=\frac{2\eps}{1+2\eps} < 1$ for $\lambda_A^{\min}=1$ and $c=2 > 1$ for $\lambda_A^{\min}=2$. Consequently, in general both options for $\tau$ are possible.
\end{exmp}
Secondly, we assume that $T\leq 0$. This case is trivial, since
\begin{equation*}
\lambda_B^{\min} \kla{(\lambda_A^{\min} - \lambda_A^{\max}) -\eps\lambda_A^{\max}} = \lambda_B^{\min} (\lambda_A^{\min} -(1+\eps) \lambda_A^{\max}) \leq T \leq 0
\end{equation*}
and we gain $\lambda_B^{\min} \geq 0$. Therefore, we may choose any $\tau \geq 0$ in this case.

In the last part of this section we want to examine how the matrices $A$ and $B$ and therefore, their eigenvalues scale if we apply $r\gamma$ for $r>0$ instead of $\gamma$. We denote the matrix $A$ with respect to the knot $\gamma$ by $A(\gamma)$ and for $B$ respectively. In \eqref{matCompAB} one can see that $A(r\gamma) = r\,A(\gamma)$. Recall that in the first variation a $\gamma$ which appears linearly in the equation is substituted by the next time step (cf. \eqref{matCompAB}) and does not appear in the matrix $B$. Moreover, $\delta\E_p(r\gamma,\Phi)=r^{-1}\delta\E_p(\gamma,\Phi)$, because $\E_p(\gamma)$ is scale invariant. Consequently, we have that $B(r\gamma)=r^{-2}\,B(\gamma)$. Finally, we obtain that by using $r\gamma$ instead of $\gamma$ we get a time step size multiplied by $r^3$ which is the same result as for the continuous first variation \eqref{firstscale}. Observe that if $T\leq 0$ this is also true for $r\gamma$ instead of $\gamma$ and in this case we also may choose $r^3\,\tau$ instead of $\tau$.

\section{A redistribution algorithm}
In numerical flows it is often helpful to redistribute the control points of the approximation in order to get a better convergence behaviour (see for instance \cite{DKS02}). Therefore, we transfer this idea to \name{Fourier} knots.
\begin{figure}[H]
\begin{tabular}{ccc}
\includegraphics{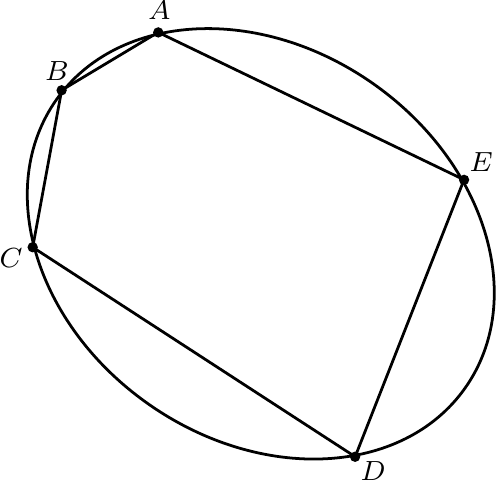}&
\hspace{2cm}
&
\includegraphics{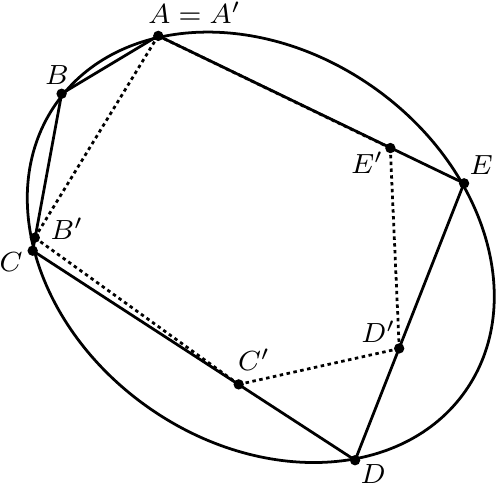}
 \\
step~1 & &step~2
\end{tabular}
\caption{A redistribution algorithm}
\label{figredista}
\end{figure}
\begin{figure}[H]
\begin{tabular}{ccc}
\includegraphics{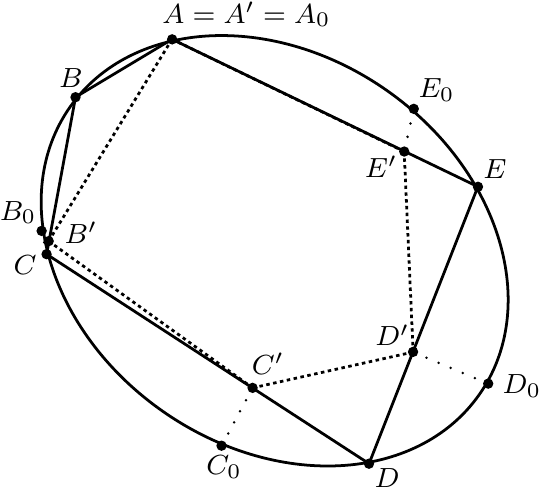}&
\hspace{2cm}
 &
\includegraphics{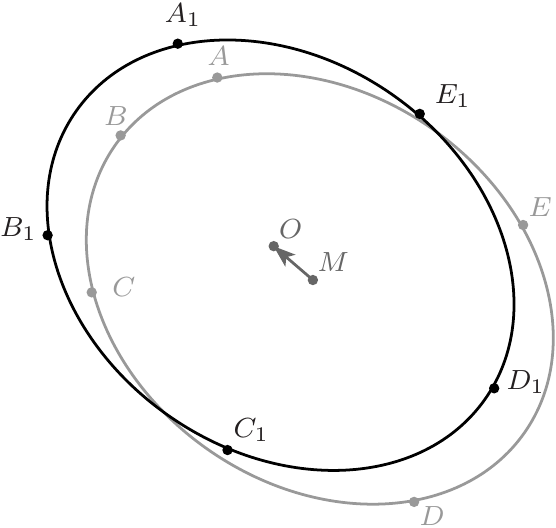}
 \\
step~3 & &step~4
\end{tabular}
\caption{A redistribution algorithm (continued)}
\label{figredistb}
\end{figure}
Firstly, we create an inscribed polygon $\mathcal{P}$ by choosing the vertices corresponding to equidistantly distributed points on $[0,2\pi]$ and compute its total length $L$, see (step~1). Then we apply the ``standard" redistribution algorithm to $P$. More precisely, if $n\in\N$ is the number of vertices of $P$, we create $n$ points on $\mathcal{P}$ such that the distance on the polygon (i.e. \emph{not} the \name{Euclidean} distance) between two points which are adjacent on $P$ is exactly $\frac{L}{n}$. Then these points define the vertices of a new polygon $\widetilde{\mathcal{P}}$ which has altered shape compared to $P$, see (step~2). Next we compute $n$ new points on the curve according to the vertices of $\widetilde{\mathcal{P}}$ which can be done since they also belong to the polygon $P$, see (step~3). Finally, we compute new \name{Fourier} coefficients with respect to a new parametrization of the curve such that these new points on the curve are the points corresponding to equidistantly distributed points on $[0,2\pi]$. Since we leave the first coefficient out the curve is slightly translated in such a way that the centre of mass of the new points is in the origin, see (step~4). Moreover, there will also be a slightly change in the shape of the curve as we only use $20$ coefficients which is not visible in Figure~\ref{figredistb}.

\section{Alternative approaches}
In principal one could also apply other optimisation algorithms instead of a gradient flow. An approach similar to a \name{Newton} method would lead to faster convergence. However, it would be necessary to compute the Hessian respectively the second variation of $\M_p$. Some hints and ideas in this direction are due to \name{P. Kiciak} (personal communication, October 2011). Another idea is to use the so-called \emph{simulated annealing} (see \cite{simanneal} and the PhD theses of \name{M. Carlen} \cite{phd_carlen} and \name{H. Gerlach} \cite{phd_gerlach}) where to put it simple the coefficients of the approximation are moved randomly such that the energy decreases. It would be interesting to compare the results with those in this thesis.

\cleartooddpage[\thispagestyle{empty}]

\chapter{Examples}
Before we come to our collection of examples we firstly discuss how to get \name{Fourier} parametrizations. Moreover, we shortly mention the tool box \mppack{}.

\section{Creating parametrizations of \name{Fourier} knots}
A collection of numerous examples of \name{Fourier} knots can be found in \cite{3d-meier}. Among others the prime knot classes up to $7_7$ are present. Moreover, on page \texttt{Seite27.html} you find the formula for \emph{torus knots}, which are curves defined on a torus. For instance, the representative of a $5_1$ knot is such a knot.

More parametrizations of \name{Fourier} knots can be found in \cite{harmknots}, where the so-called \emph{harmonic knots} are introduced and where we take a $4_1$ knot from and int \cite{kauffman}.

The discretization of a stadium-curve (see Figure~\ref{figStadionCurve}) can be found in \cite[3.4]{dipl}. The ``saddle'' knot is taken from the project work \cite{projekt}. The data for the deformed trefoil knot comes from \name{H. Gerlach}.

For instance one could generate a polygonal knot using \knotplot{} \cite{knotplot} and save it as an list of coordinates in $\R^3$. Then we can use \libbiarc{} \cite{phd_carlen}. One can convert the file into the pkf-format using \texttt{\libbiarc{}/tools/xyz2pkf} and then \texttt{\libbiarc{}/fourierknots/guess\_coeffs} in order to get the \name{Fourier} coefficients. For instance the configuration used for the $5_2$ knot (Section~\ref{knot52}) was created like this.

\section{Implementation}
The tool box \mppack{} contains a family of programs to handle \name{Fourier} knots and to compute the gradient flows considered in this thesis. For this purpose the programming languages \cpp{} and \python{} has been used. The computer algebra system \maple{} is also used to create previews of the gradient flows. The final visualisations have been done with \povray. \gnuplot{} can be used to create graphs related to the flows.

All knots we are considering in the following sections are \name{Fourier} knots with $20$ pairs of coefficients. A corresponding flow is visualised by a series of pictures of significant time steps. The caption of such a visualisation contains the information what $p$ has been chosen. For $\tau_{\max}$ and $\eps$, which are important for the adaptive choice of time step size, we use $\tau_{\max}=0.01$ and $\eps=0.05$ in most examples. If this is not the case there will be a hint in the caption. Unless otherwise noted the scale-invariant flow is used. Remember that for $p=3$ we have $\E_3\equiv\M_3$. Each picture is labelled with the index of the time step, the approximate length and $\E_p$-energy as well as the total time that has elapsed so far. The knots are rendered with a constant thickness during the flow, which does not depend on the actual thickness in the sense of \thref{defthickness} of the knot.

\newpage
\section{Unknots ($0_1$)} \label{knot0}
We start with a planar stadium-curve. Observe, that the curve used here is a \name{Fourier} approximation of Figure~\ref{figStadionCurve}, which is of class $C^\infty$, and therefore, some segments of this curve only seems to be straight. As was to be expected it converges to a circle. The energy value of a circle for $\eps=0.05$ stays constant at \texttt{6.28318530718} for all digits. Observe, that the redistribution applied to a circle does \emph{not} change its coefficients. Even for the unit circle it is crucial to choose the time step size carefully, for example with a time step size of $0.1$ the circle deforms strongly after only a few steps.
\begin{figure}[H]
\begin{center}
\begin{scriptsize}
\begin{tabular}{cc}
0/25000 & 1200/25000 \\
$\Le(\gamma)\approx 5.14154$ & $\Le(\gamma)\approx 5.03248$ \\
$\E_p(\gamma)\approx 7.18548$ & $\E_p(\gamma)\approx 6.83197$ \\
$\tau=0.0$ & $\tau=0.02728$ \\
\includegraphics[width=0.4\textwidth,keepaspectratio]{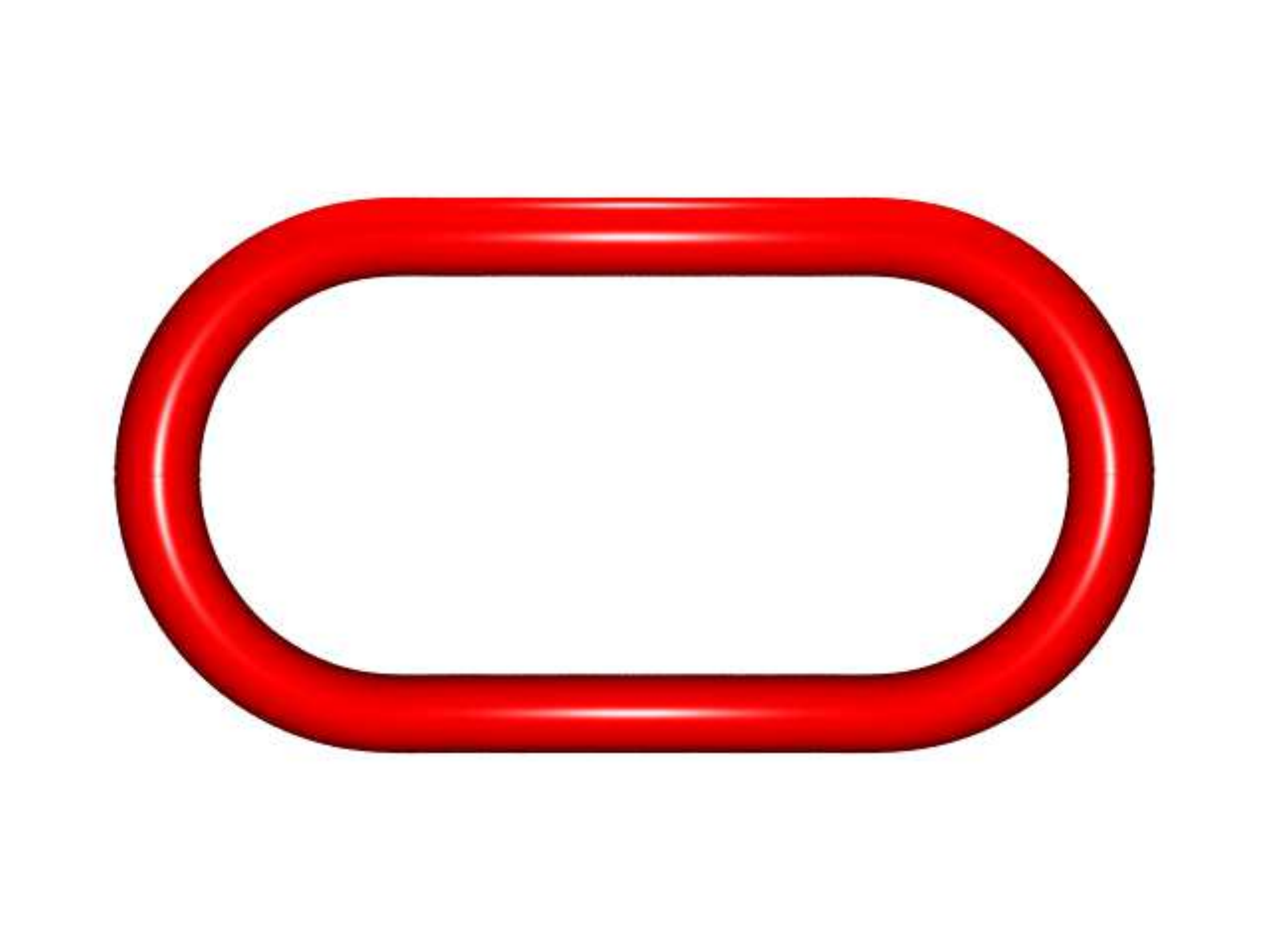} & \includegraphics[width=0.4\textwidth,keepaspectratio]{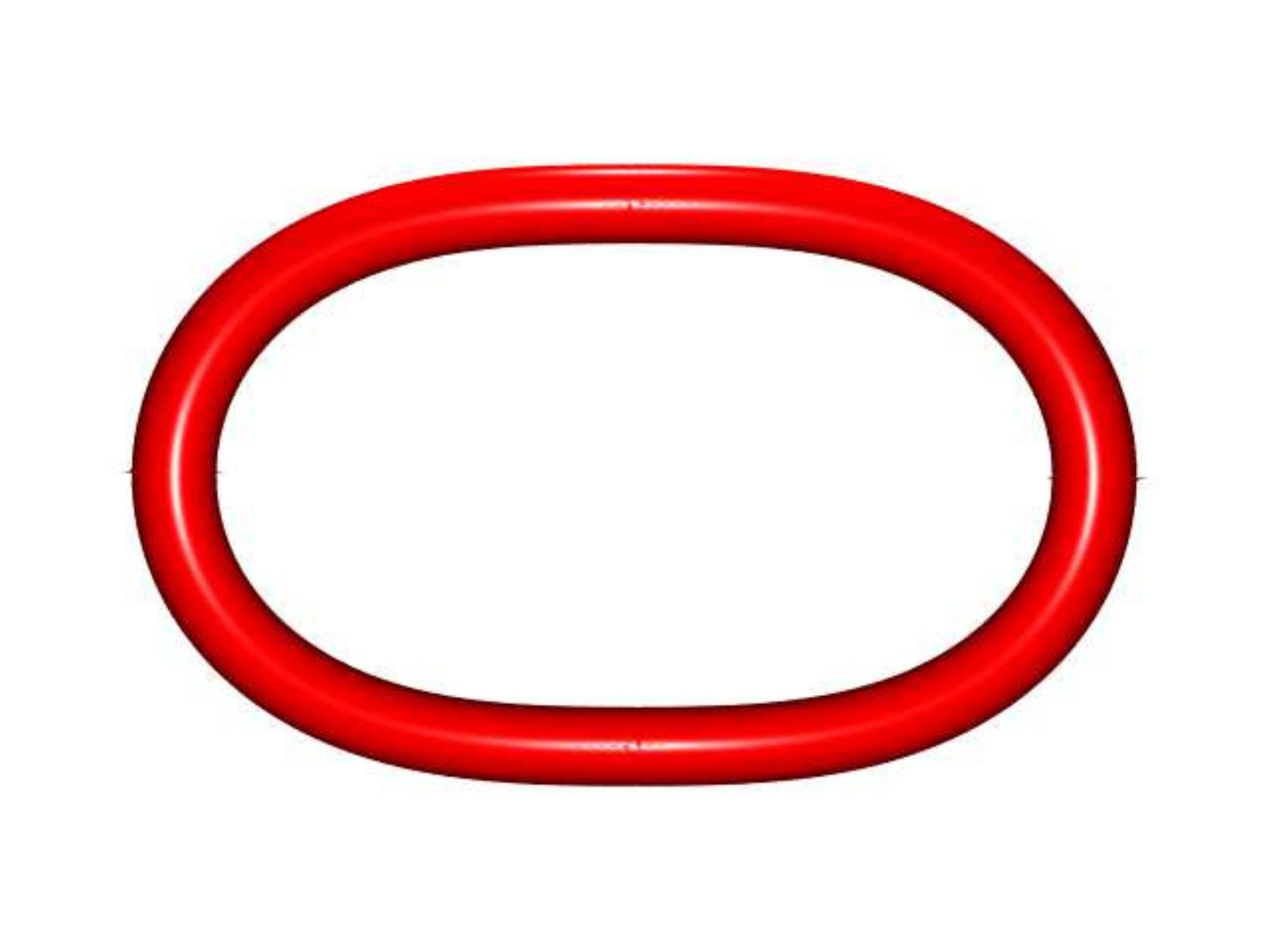}\\
7000/25000 & 25000/25000 \\
$\Le(\gamma)\approx 4.88948$ & $\Le(\gamma)\approx 4.88032$ \\
$\E_p(\gamma)\approx 6.31843$ & $\E_p(\gamma)\approx 6.28319$ \\
$\tau=0.14839$ & $\tau=0.51834$ \\
\includegraphics[width=0.4\textwidth,keepaspectratio]{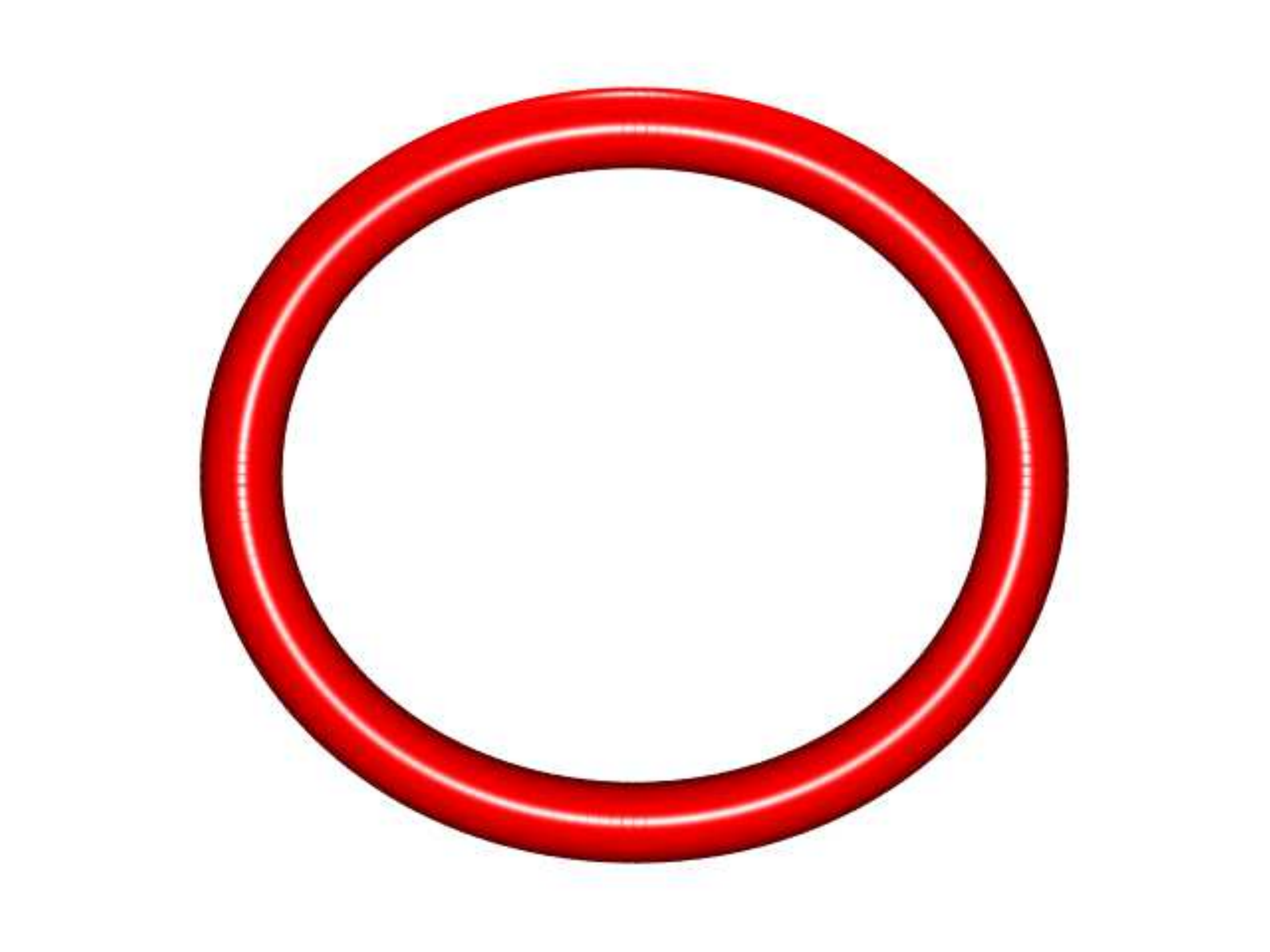} & \includegraphics[width=0.4\textwidth,keepaspectratio]{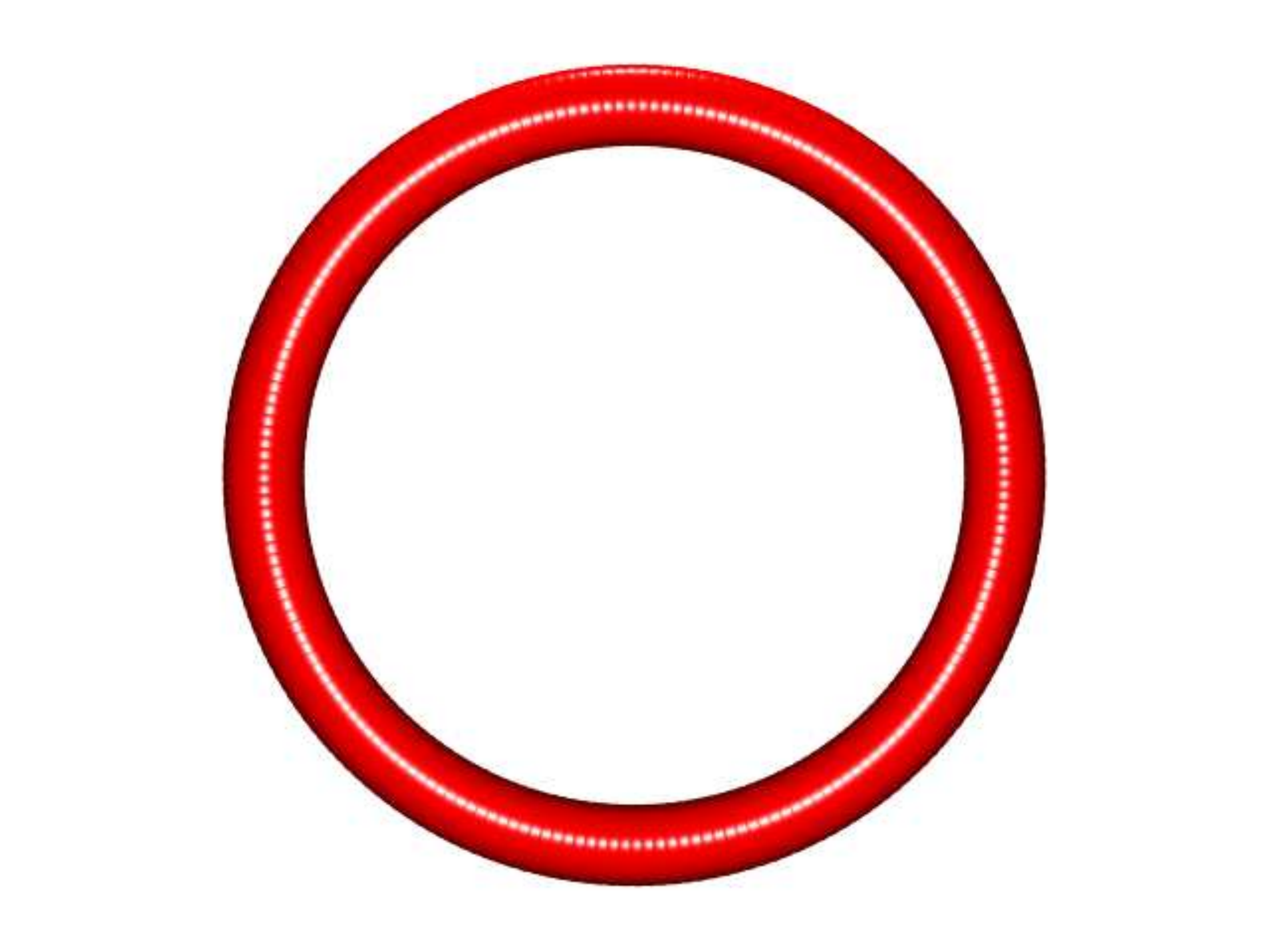}
\end{tabular}
\end{scriptsize}
\end{center}
\caption{A stadium-curve -- $p=3.0$}
\end{figure}

\newpage
This was the flow for $p=3$. If we choose $p=50$ instead, the stadium-curve again converges to a circle. However, there are some differences
\begin{figure}[H]
\begin{center}
\begin{scriptsize}
\begin{tabular}{cc}
0/10000 & 1200/10000 \\
$\Le(\gamma)\approx 5.14154$ & $\Le(\gamma)\approx 5.08261$ \\
$\E_p(\gamma)\approx 9.7989$ & $\E_p(\gamma)\approx 8.59715$ \\
$\tau=0.0$ & $\tau=0.01058$ \\
\includegraphics[width=0.4\textwidth,keepaspectratio]{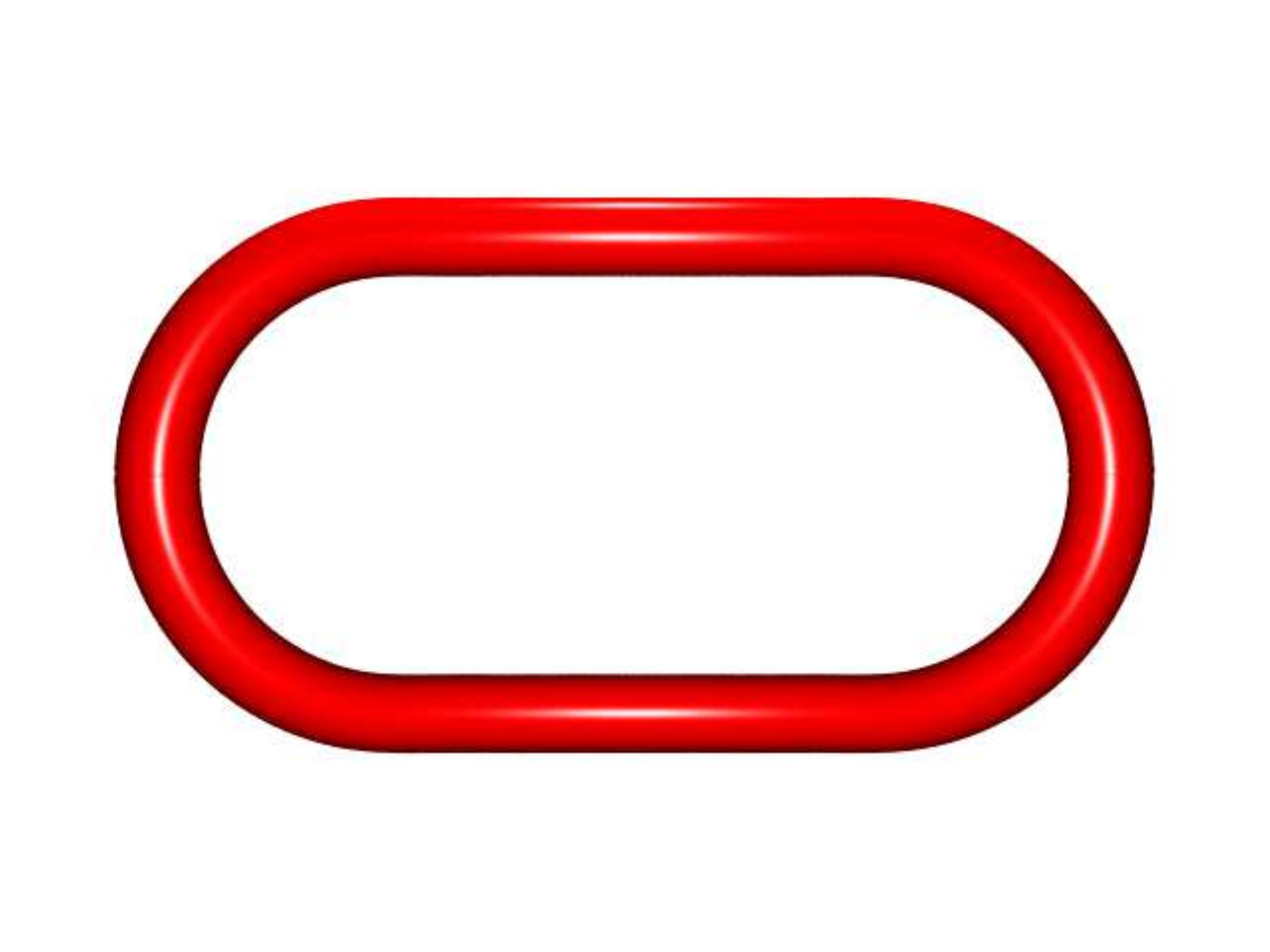} & \includegraphics[width=0.4\textwidth,keepaspectratio]{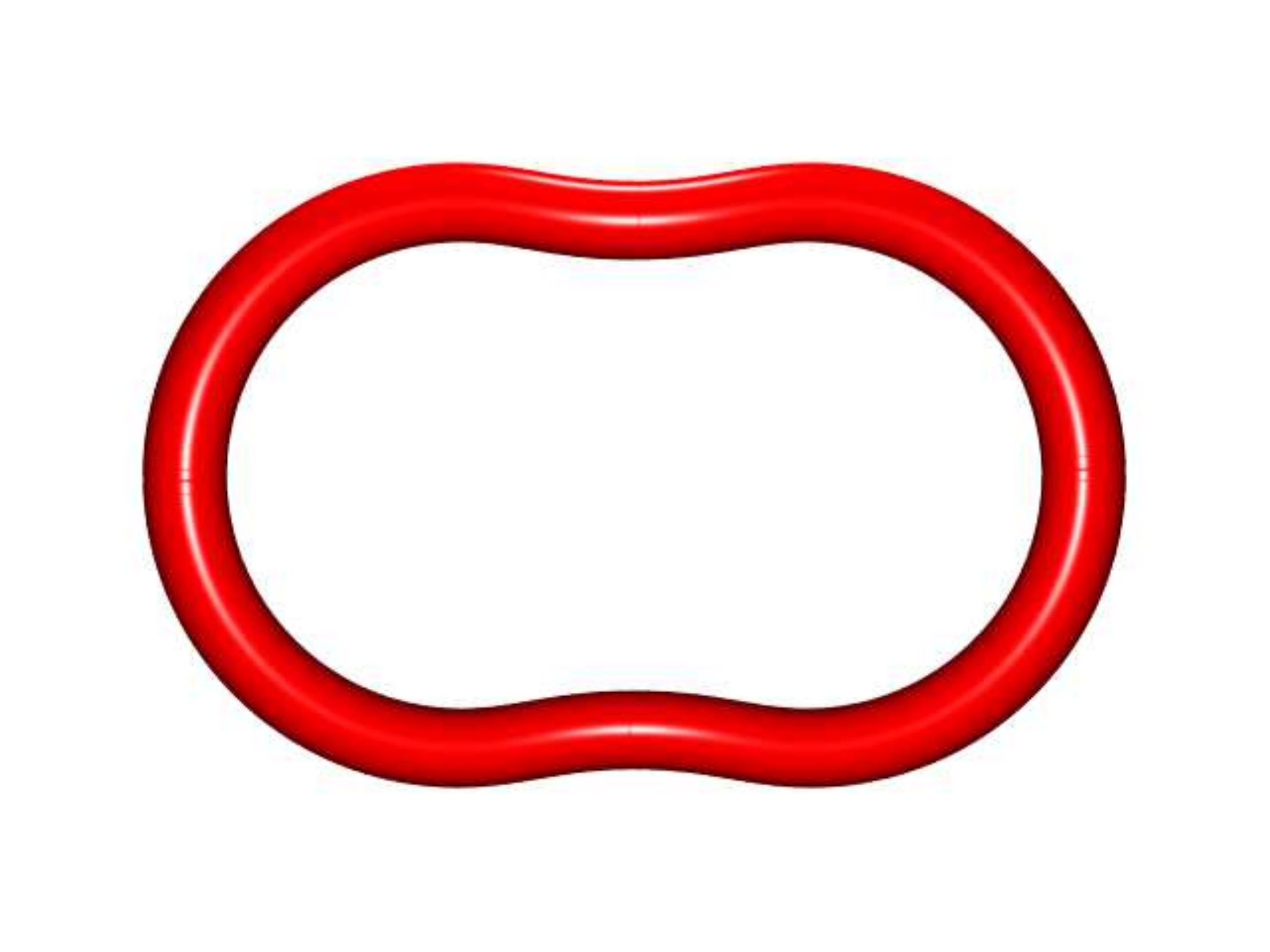} \\
4000/10000 & 10000/10000 \\
$\Le(\gamma)\approx 4.93668$ & $\Le(\gamma)\approx 4.90877$ \\
$\E_p(\gamma)\approx 6.97758$ & $\E_p(\gamma)\approx 6.28319$ \\
$\tau=0.03576$ & $\tau=0.1423$ \\
\includegraphics[width=0.4\textwidth,keepaspectratio]{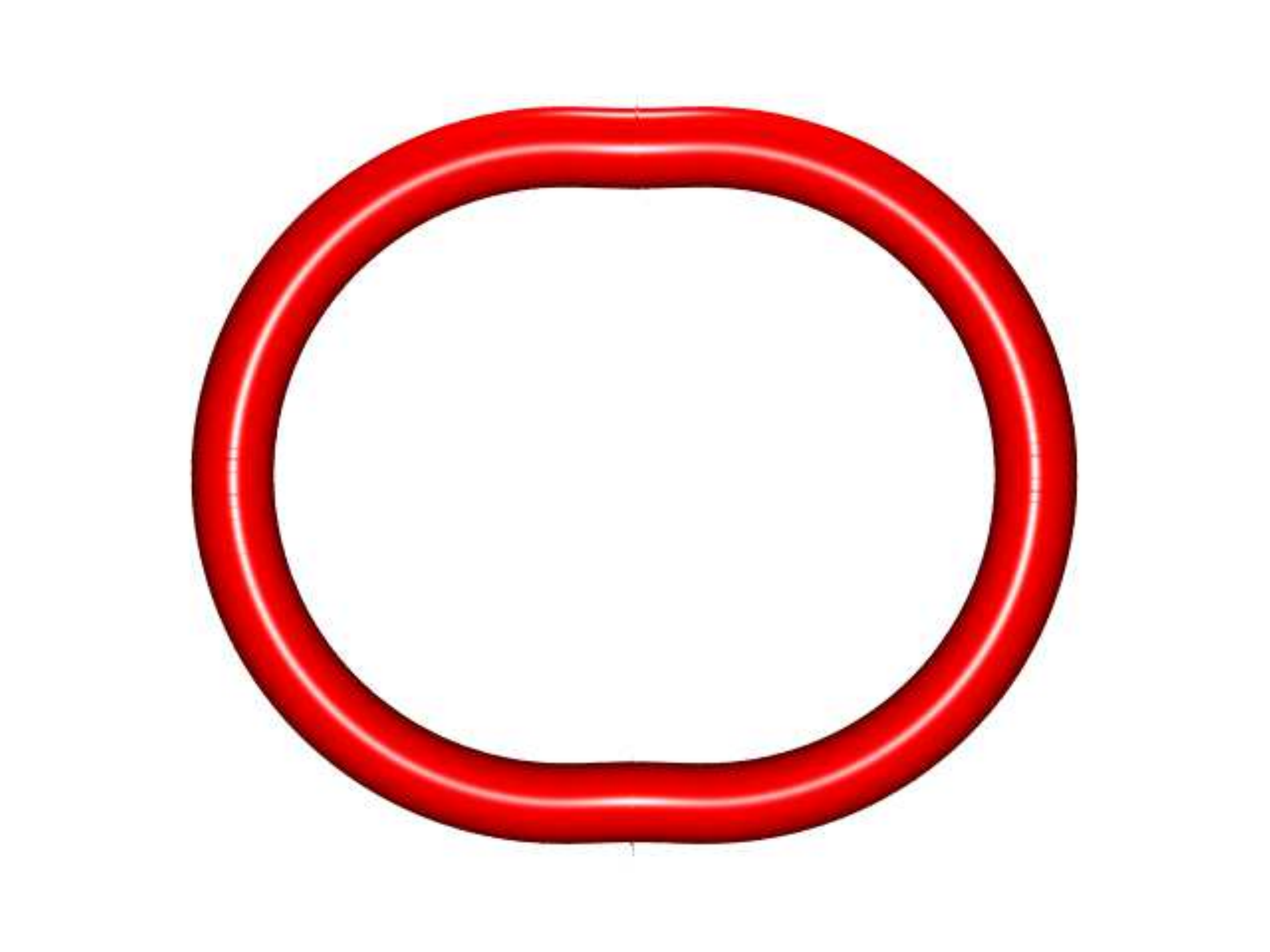} & \includegraphics[width=0.4\textwidth,keepaspectratio]{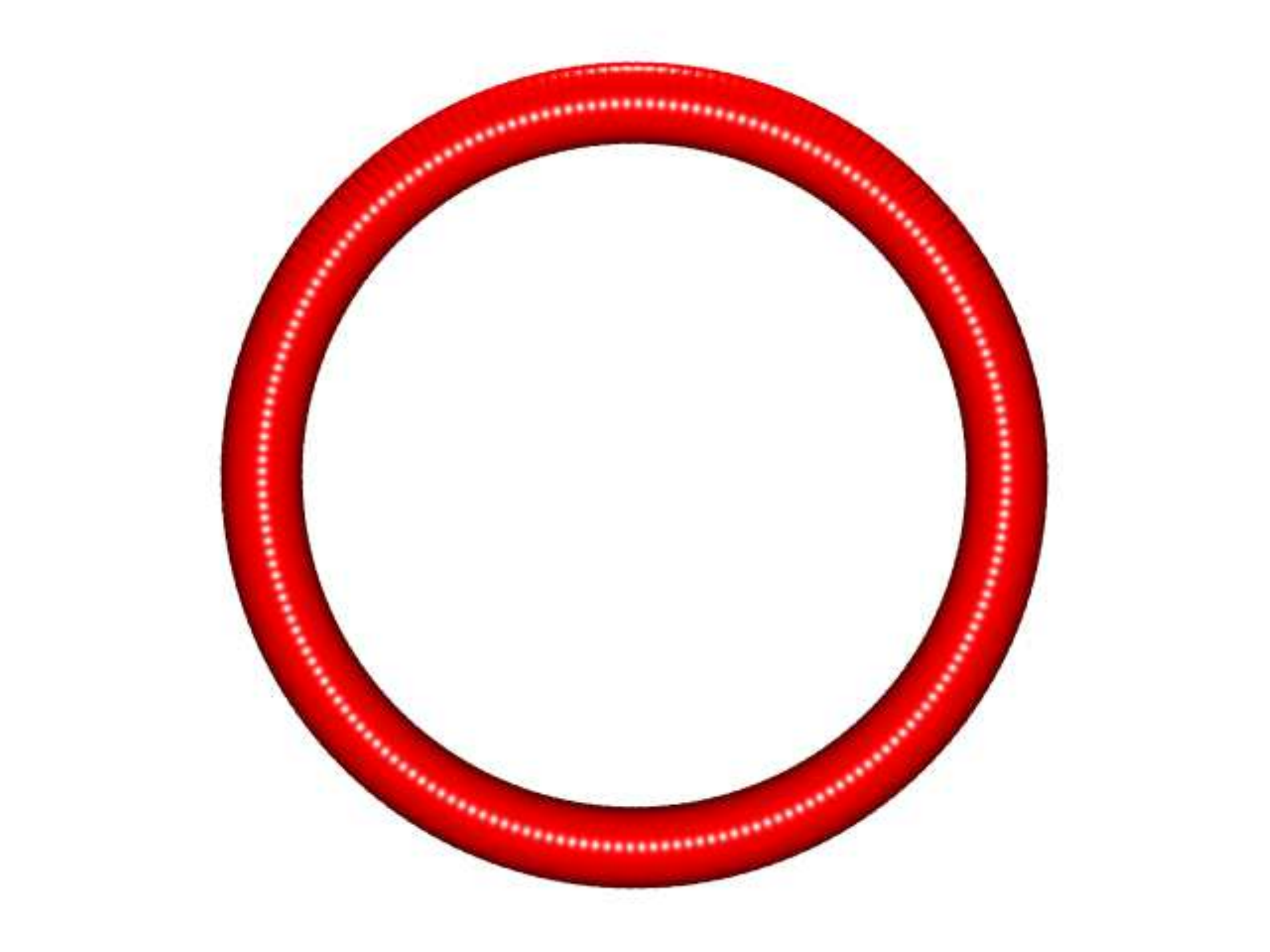}
\end{tabular}
\end{scriptsize}
\end{center}
\caption{A stadium-curve -- $p=50.0$}
\end{figure}

\newpage
Something similar happens to this ``saddle''-curve. Firstly for $p=3$
\begin{figure}[H]
\begin{center}
\begin{scriptsize}
\begin{tabular}{cc}
0/50000 & 300/50000 \\
$\Le(\gamma)\approx 7.64039$ & $\Le(\gamma)\approx 7.47673$ \\
$\E_p(\gamma)\approx 8.07063$ & $\E_p(\gamma)\approx 7.7699$ \\
$\tau=0.0$ & $\tau=0.02285$ \\
\includegraphics[width=0.4\textwidth,keepaspectratio]{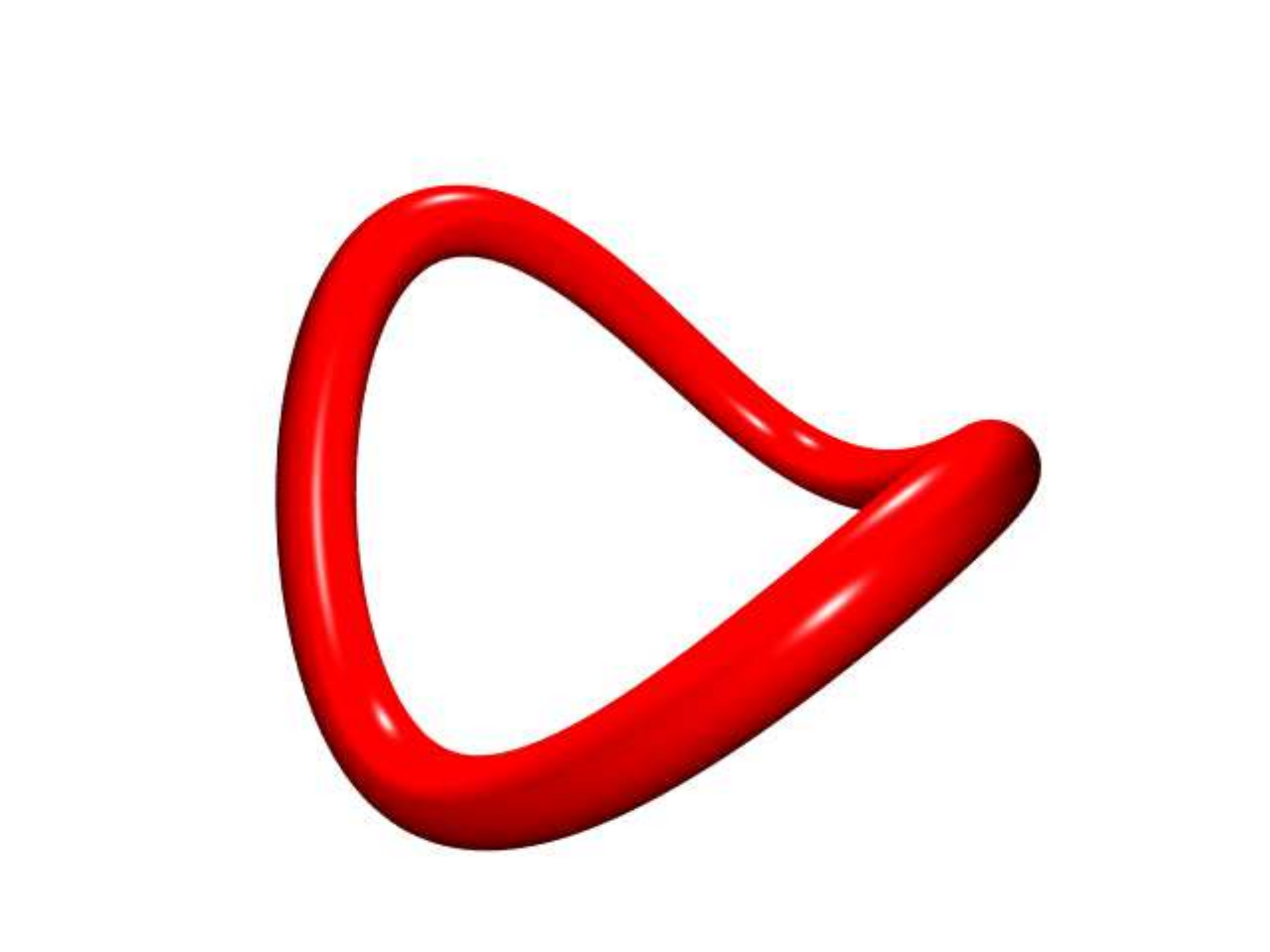} & \includegraphics[width=0.4\textwidth,keepaspectratio]{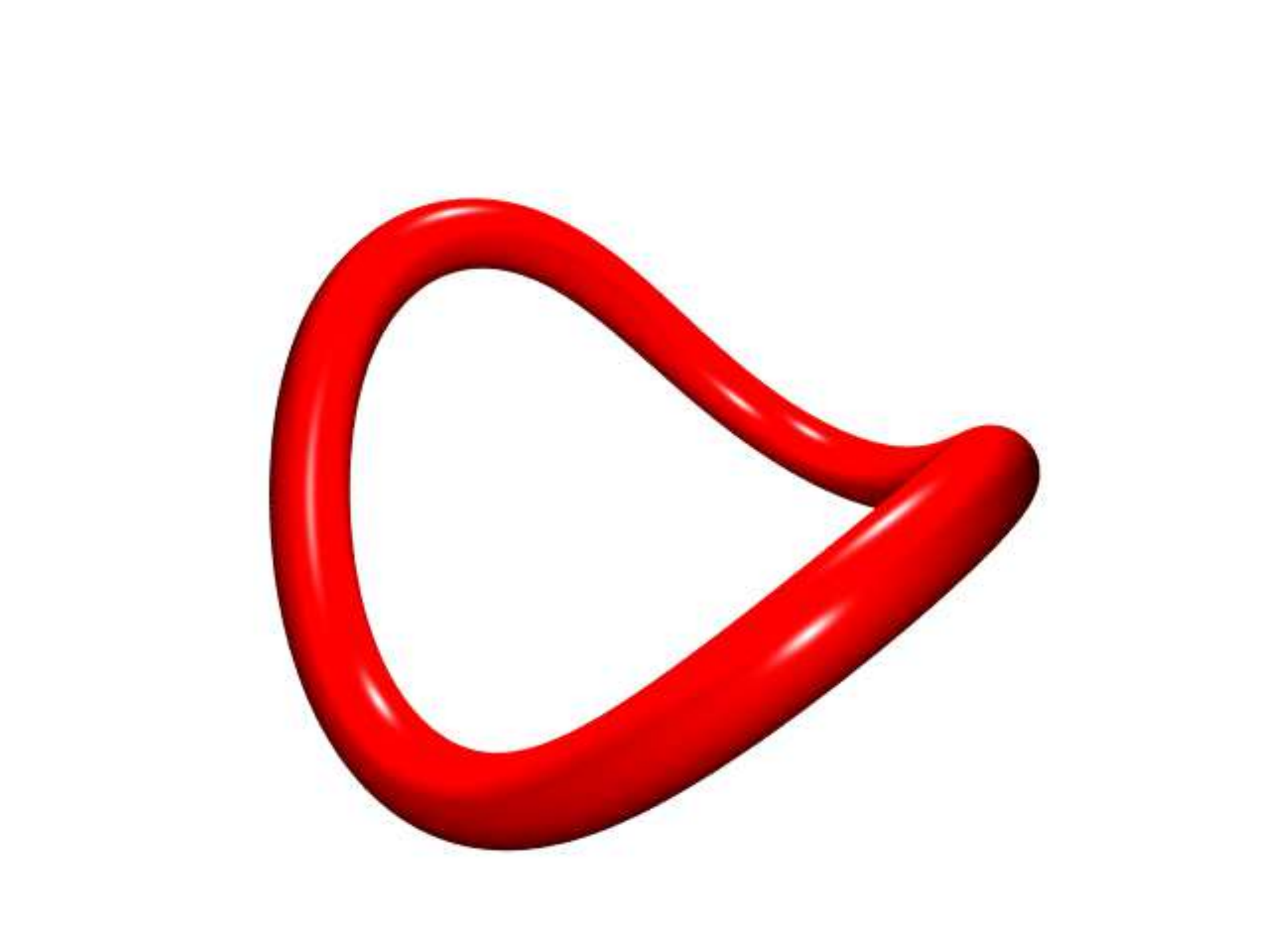} \\
5000/50000 & 10000/50000 \\
$\Le(\gamma)\approx 6.70235$ & $\Le(\gamma)\approx 6.6282$ \\
$\E_p(\gamma)\approx 6.43338$ & $\E_p(\gamma)\approx 6.29425$ \\
$\tau=0.30021$ & $\tau=0.55978$ \\
\includegraphics[width=0.4\textwidth,keepaspectratio]{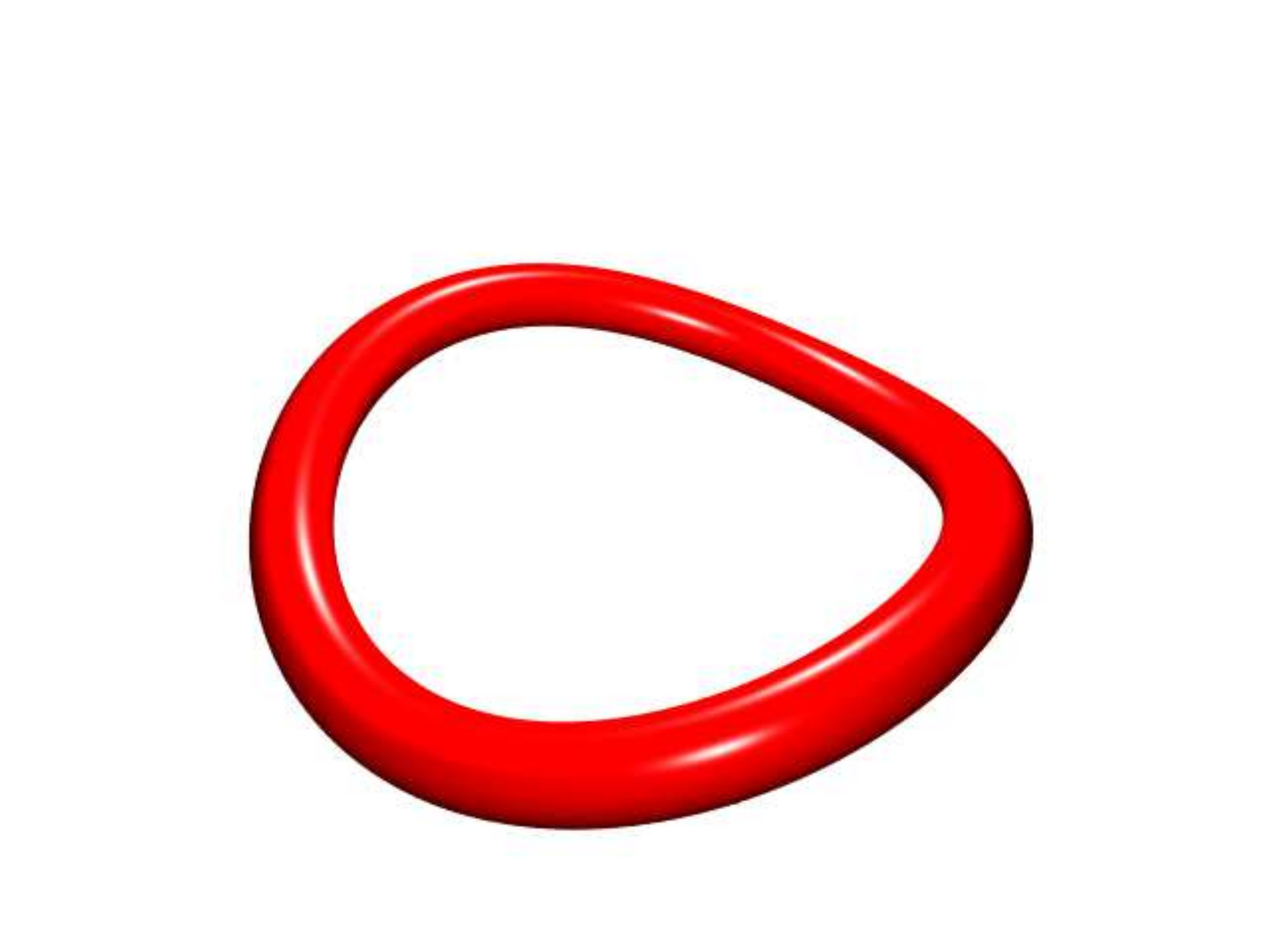} & \includegraphics[width=0.4\textwidth,keepaspectratio]{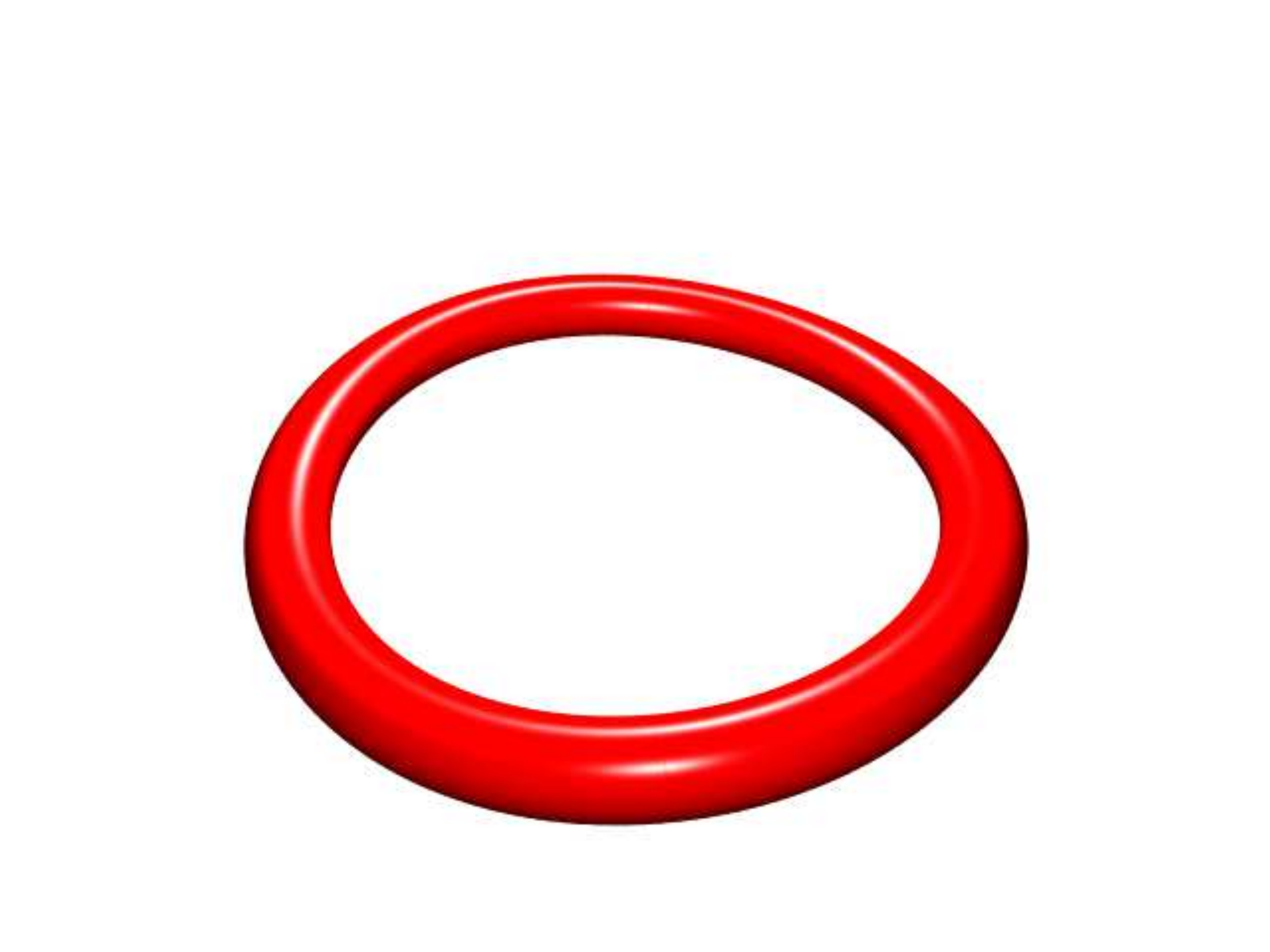} \\
20000/50000 & 50000/50000 \\
$\Le(\gamma)\approx 6.6224$ & $\Le(\gamma)\approx 6.62237$ \\
$\E_p(\gamma)\approx 6.28324$ & $\E_p(\gamma)\approx 6.28319$ \\
$\tau=1.07353$ & $\tau=2.61411$ \\
\includegraphics[width=0.4\textwidth,keepaspectratio]{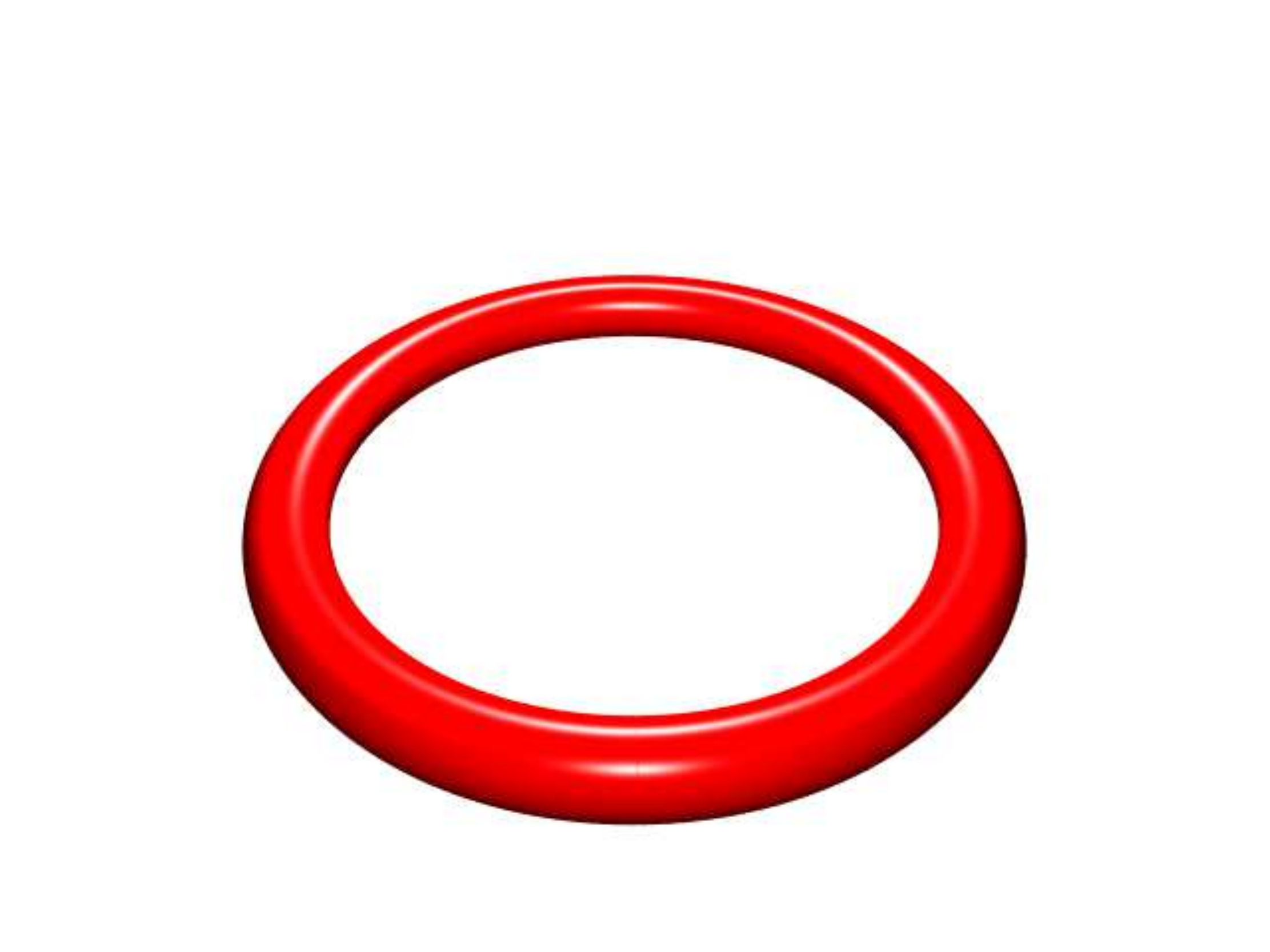} & \includegraphics[width=0.4\textwidth,keepaspectratio]{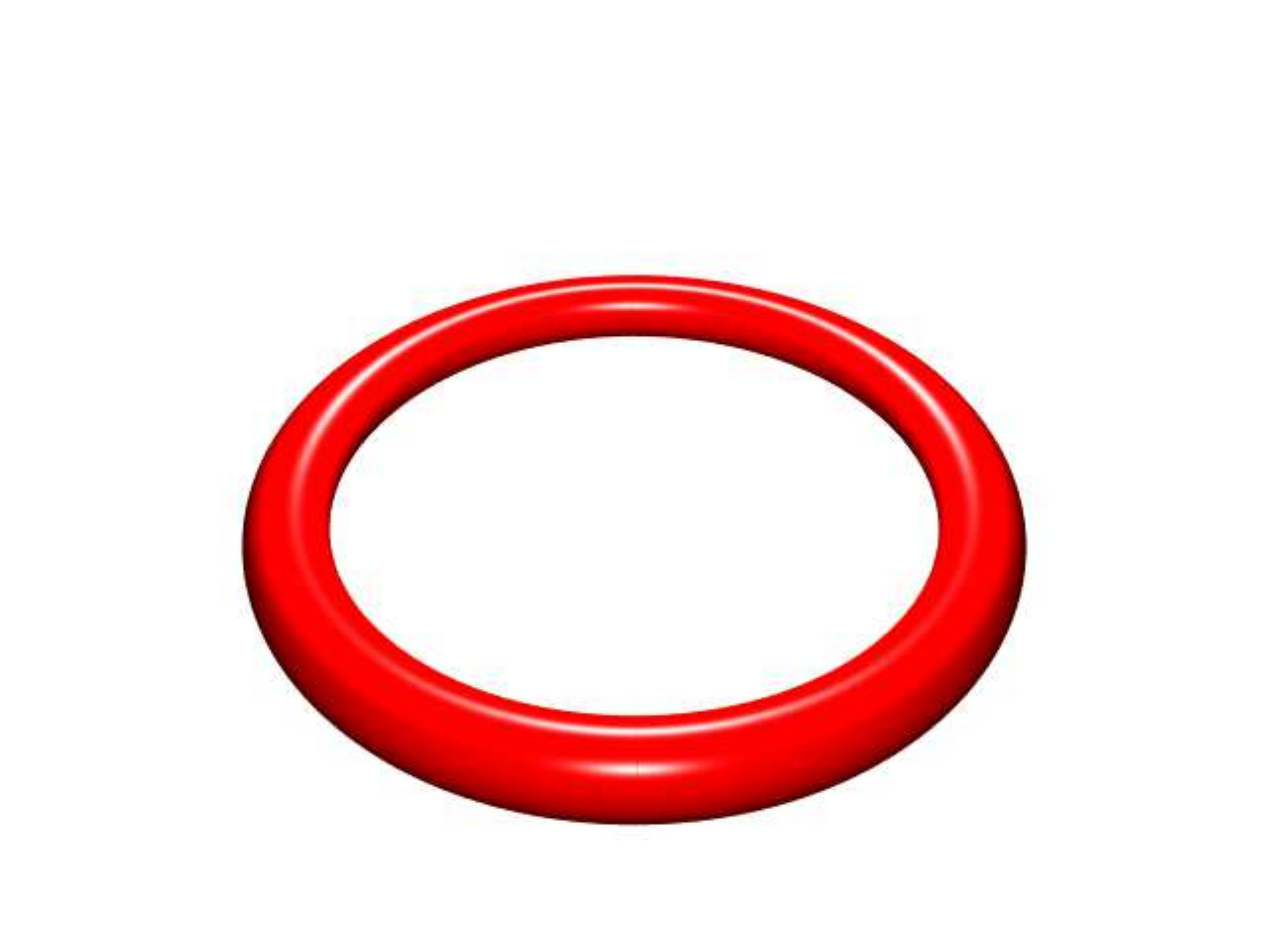}
\end{tabular}
\end{scriptsize}
\end{center}
\caption{A saddle-curve -- $p=3.0$}
\end{figure}

\newpage
and now for $p=50$
\begin{figure}[H]
\begin{center}
\begin{scriptsize}
\begin{tabular}{cc}
0/50000 & 300/50000 \\
$\Le(\gamma)\approx 7.64039$ & $\Le(\gamma)\approx 7.59426$ \\
$\E_p(\gamma)\approx 14.34766$ & $\E_p(\gamma)\approx 10.01585$ \\
$\tau=0.0$ & $\tau=0.00257$ \\
\includegraphics[width=0.4\textwidth,keepaspectratio]{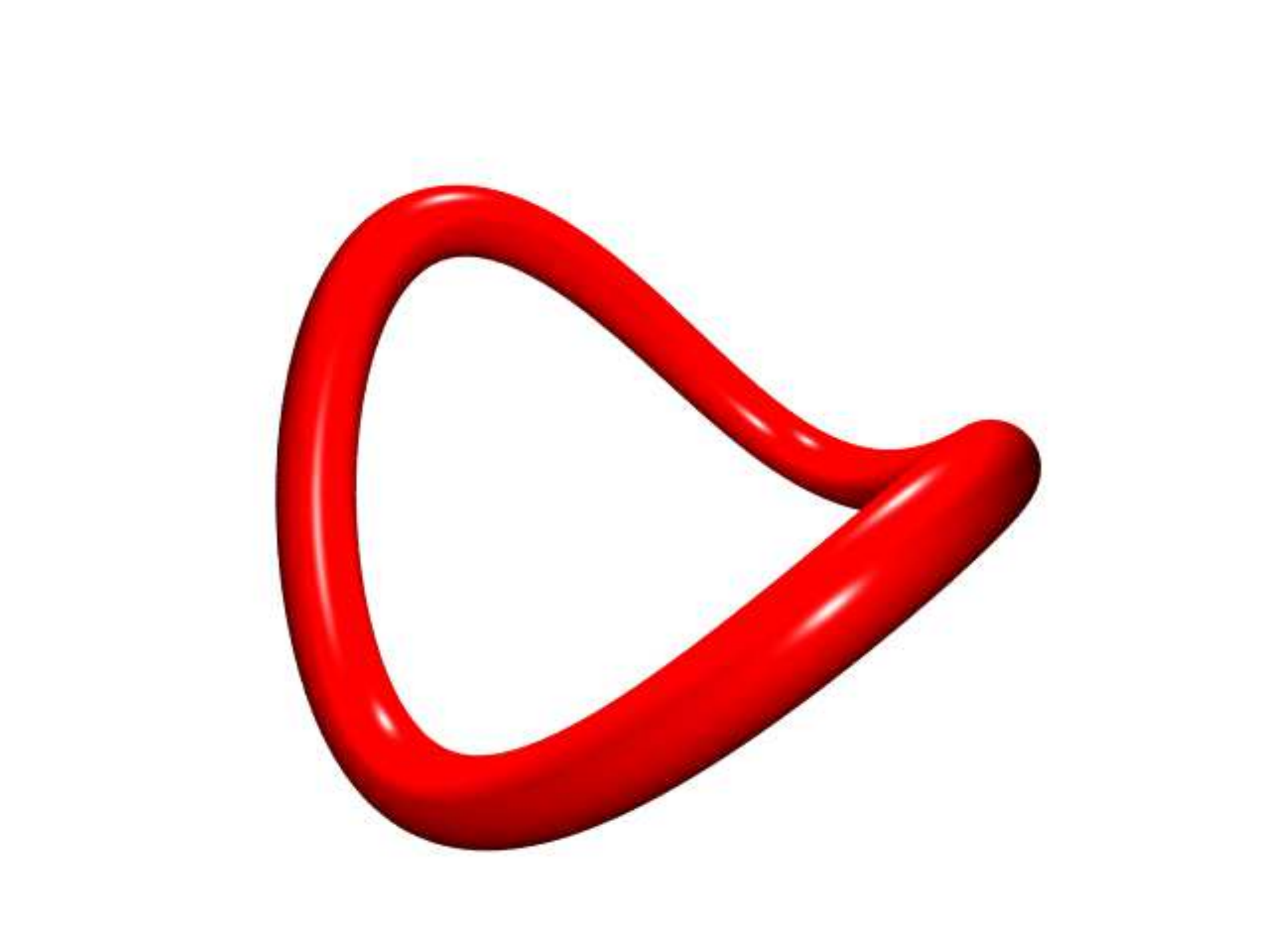} & \includegraphics[width=0.4\textwidth,keepaspectratio]{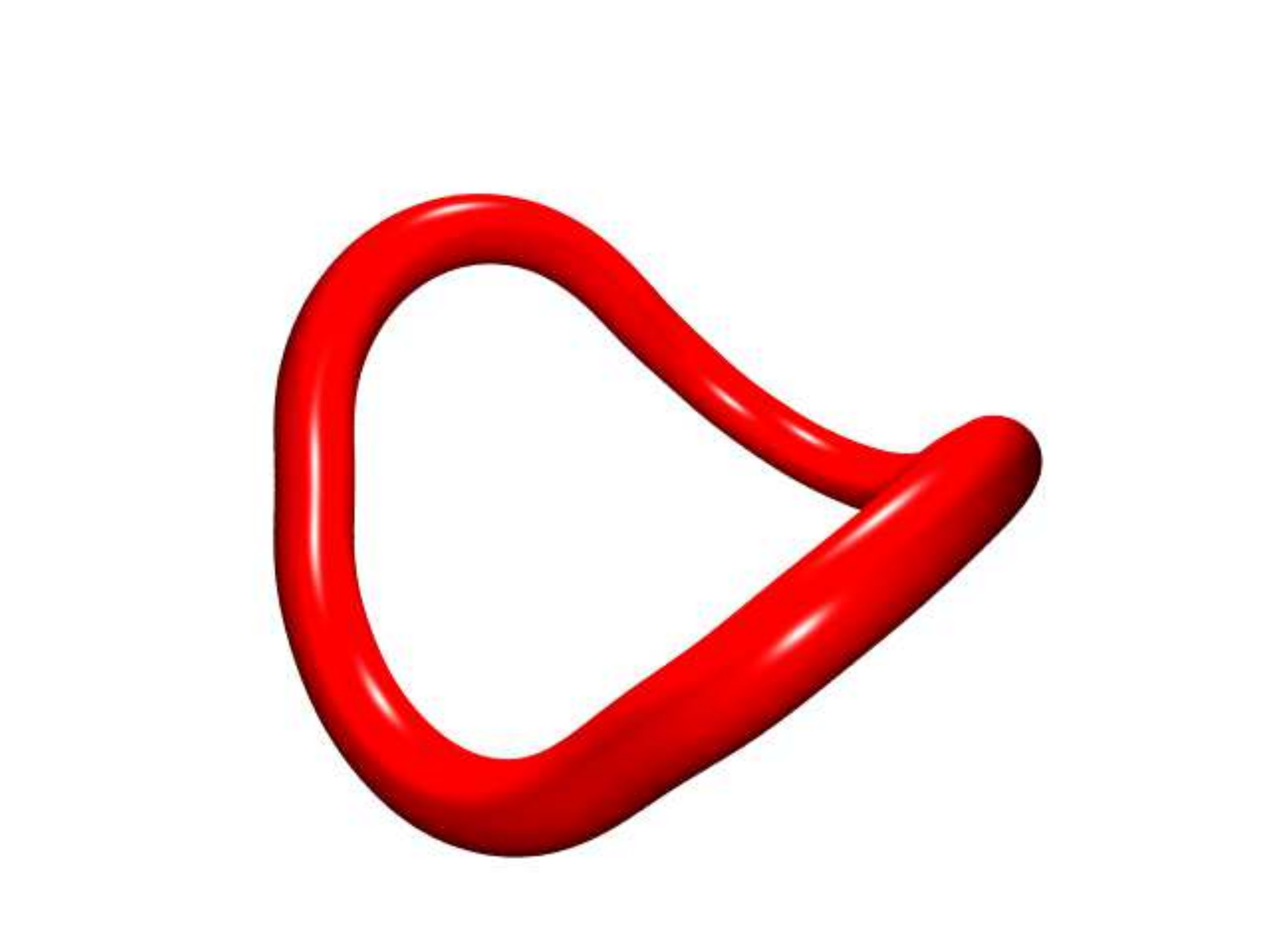} \\
5000/50000 & 10000/50000 \\
$\Le(\gamma)\approx 7.19844$ & $\Le(\gamma)\approx 6.92135$ \\
$\E_p(\gamma)\approx 7.7091$ & $\E_p(\gamma)\approx 6.99055$ \\
$\tau=0.05116$ & $\tau=0.10411$ \\
\includegraphics[width=0.4\textwidth,keepaspectratio]{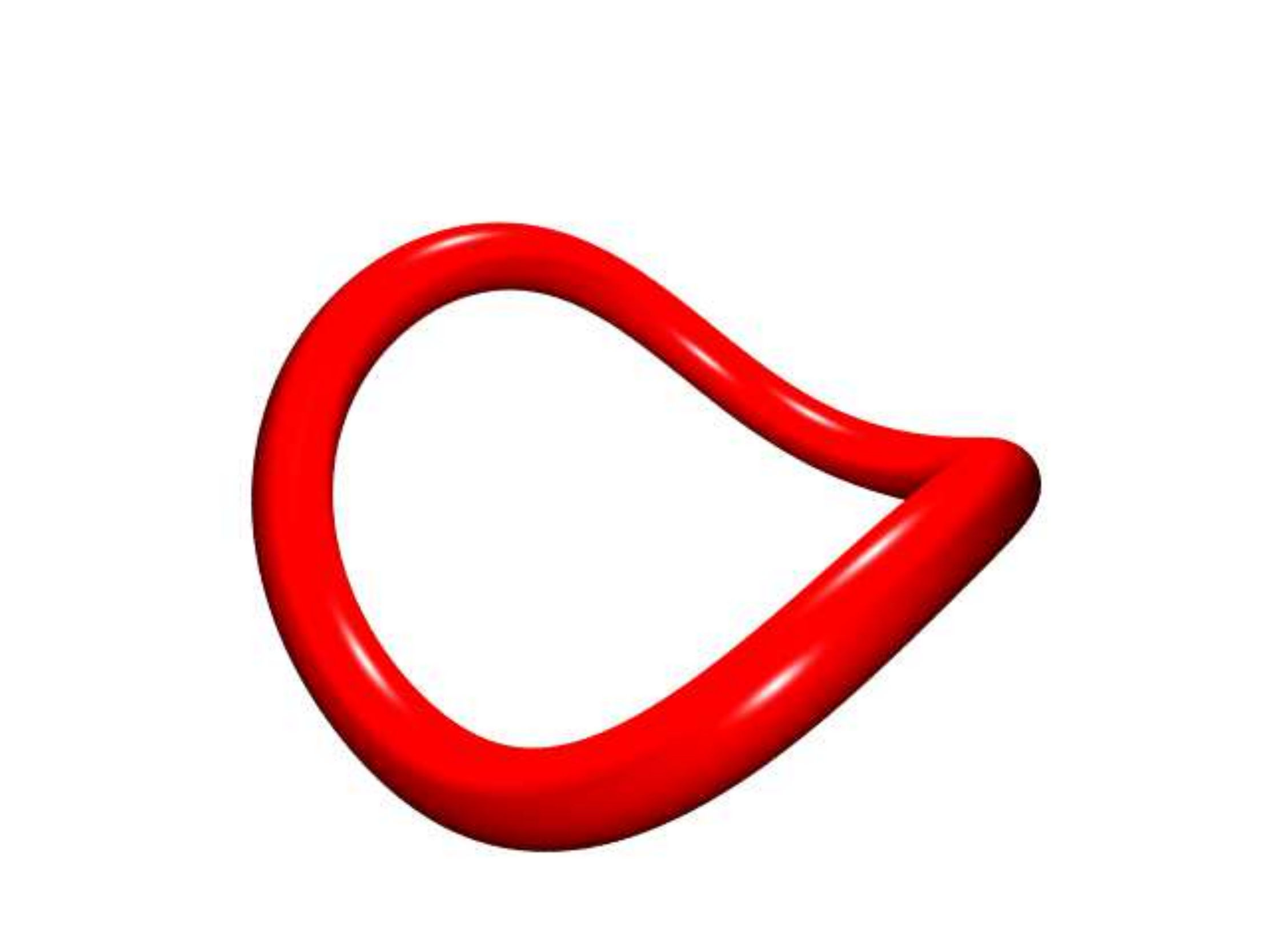} & \includegraphics[width=0.4\textwidth,keepaspectratio]{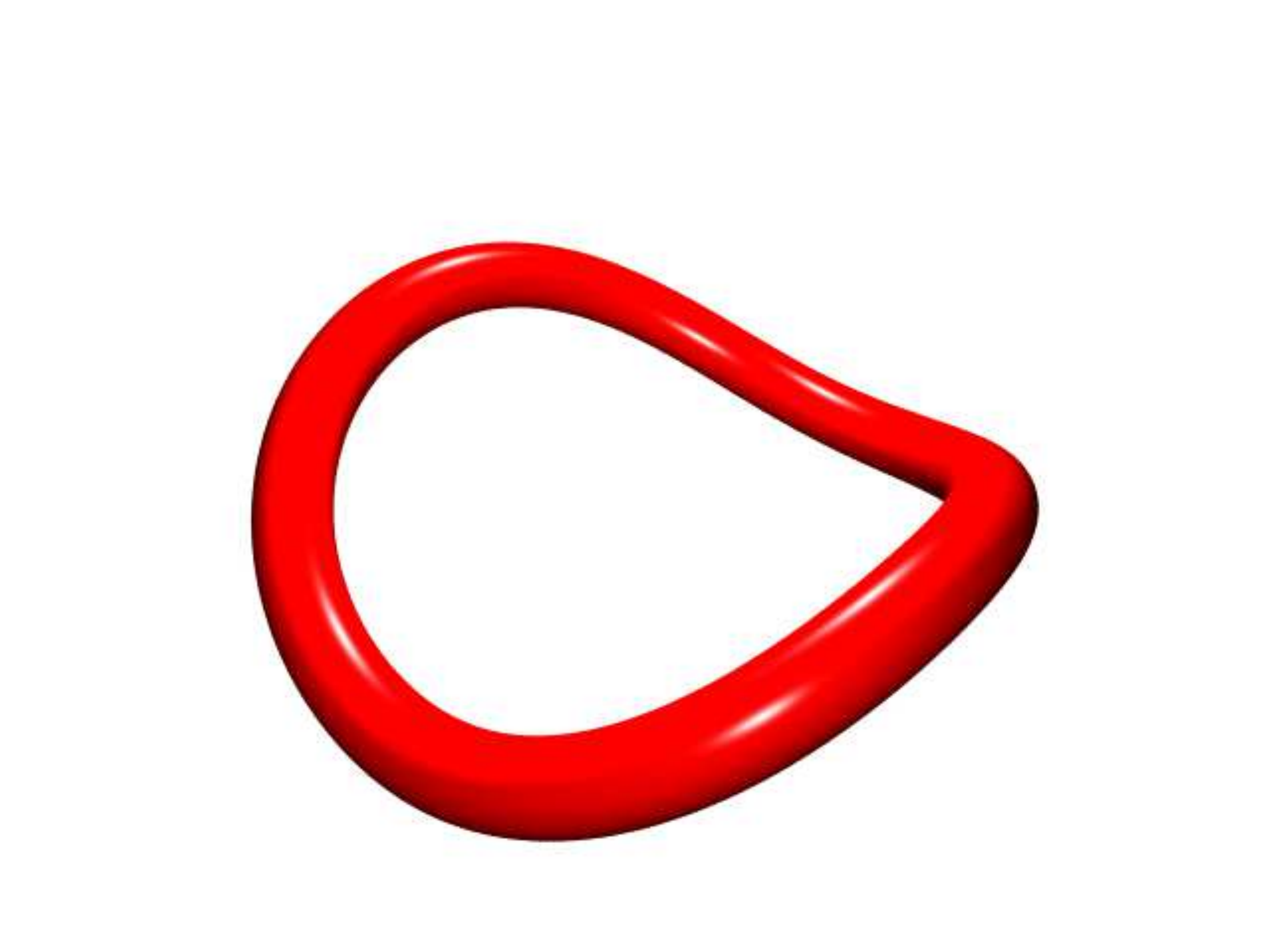} \\
20000/50000 & 50000/50000 \\
$\Le(\gamma)\approx 6.64259$ & $\Le(\gamma)\approx 6.63093$ \\
$\E_p(\gamma)\approx 6.30566$ & $\E_p(\gamma)\approx 6.28319$ \\
$\tau=0.37589$ & $\tau=1.91282$ \\
\includegraphics[width=0.4\textwidth,keepaspectratio]{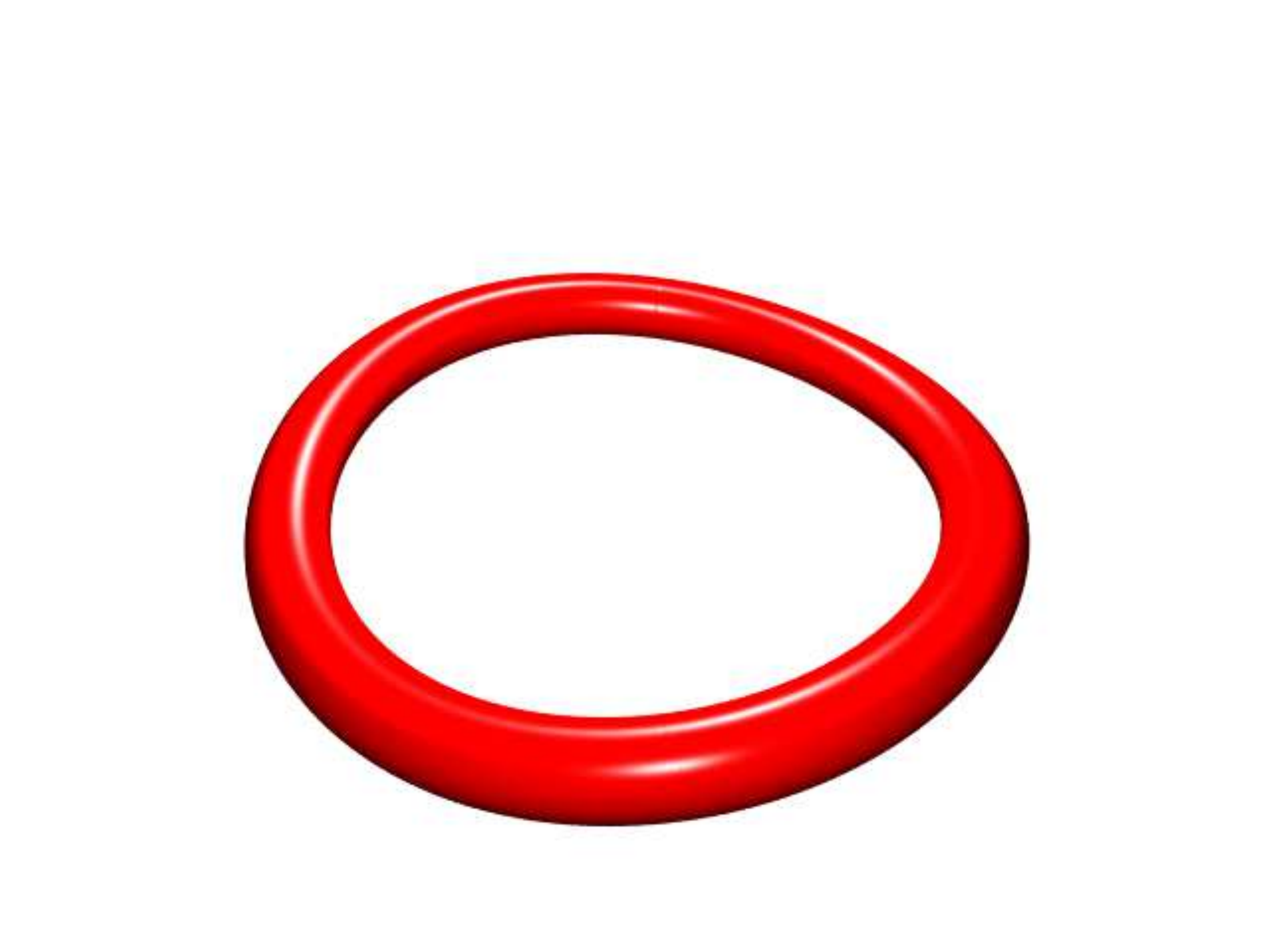} & \includegraphics[width=0.4\textwidth,keepaspectratio]{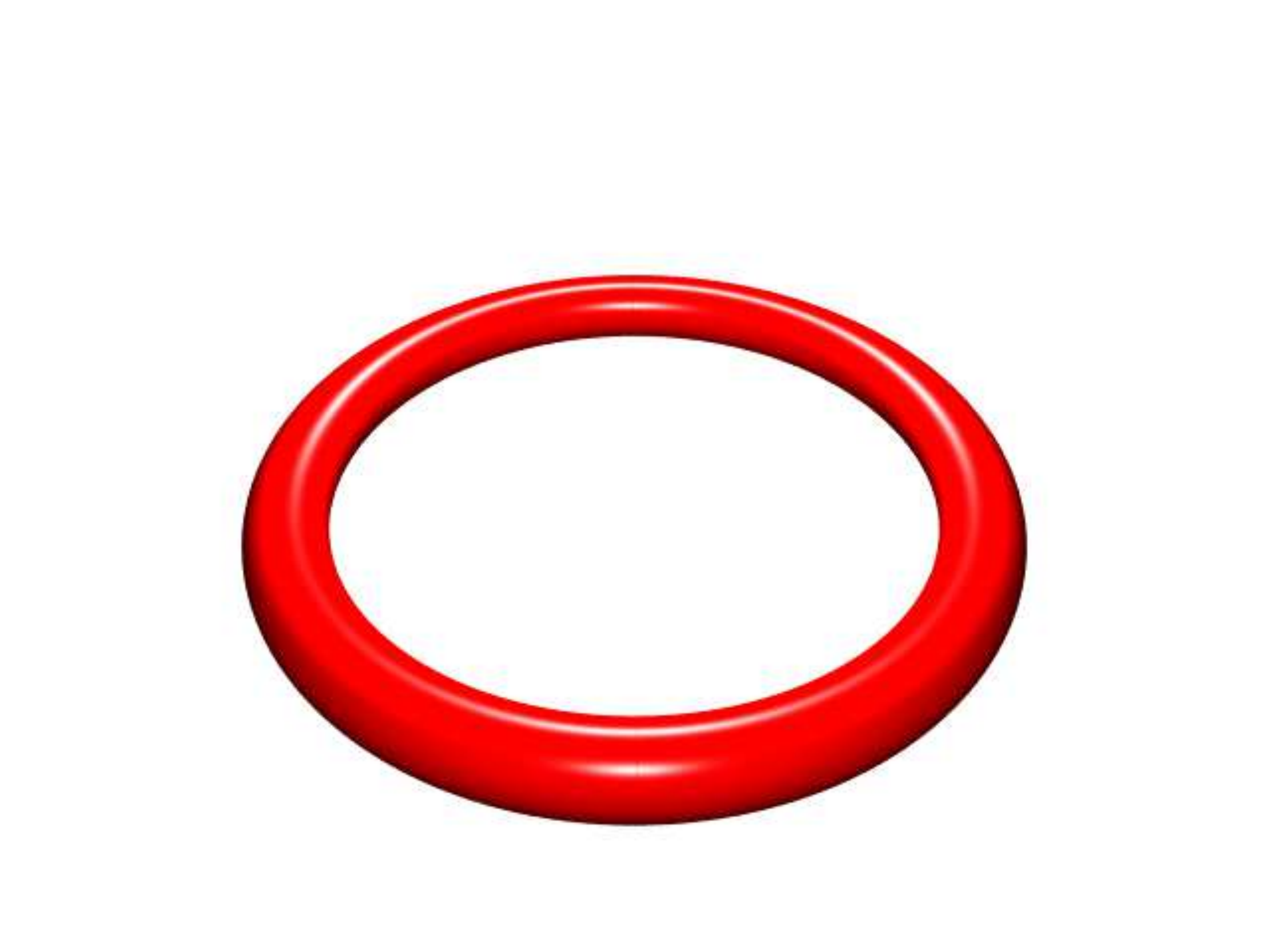}
\end{tabular}
\end{scriptsize}
\end{center}
\caption{A saddle-curve -- $p=50.0$}
\end{figure}

\newpage
We can take this example to further demonstrate the difficulties in choosing time step size. $\eps=0.05$ seems to be a suitable value for the flow with redistribution. However, if we run the flow without redistribution $\eps=0.05$ could be too large under certain circumstances. If we run this flow without redistribution for $\eps=0.05$ the minimal energy of \texttt{6.28319} is not reached and the energy even starts to slightly increase in the end. Nevertheless, the shape of the knot is very close to a circle and if we apply the redistribution algorithm to such a configuration and restart the algorithm without redistribution the knot immediately flows to a perfect circle.

Now we shortly switch to the flow for the modified integral \name{Menger} curvature, defined in \eqref{energy_pMp+L}. The following flow is discussed in \thref{mpcirclelambda}
\begin{figure}[H]
\begin{center}
\begin{scriptsize}
\begin{tabular}{ccc}
0/450000 & 90000/450000 & 180000/450000 \\
$\Le(\gamma)\approx 6.28319$ & $\Le(\gamma)\approx 6.93279$ & $\Le(\gamma)\approx 6.99425$ \\
$\E_p(\gamma)\approx 6.28319$ & $\E_p(\gamma)\approx 6.28319$ & $\E_p(\gamma)\approx 6.28319$ \\
$\tau=0.0$ & $\tau=15.66016$ & $\tau=33.2186$ \\
\includegraphics[width=0.33\textwidth,keepaspectratio]{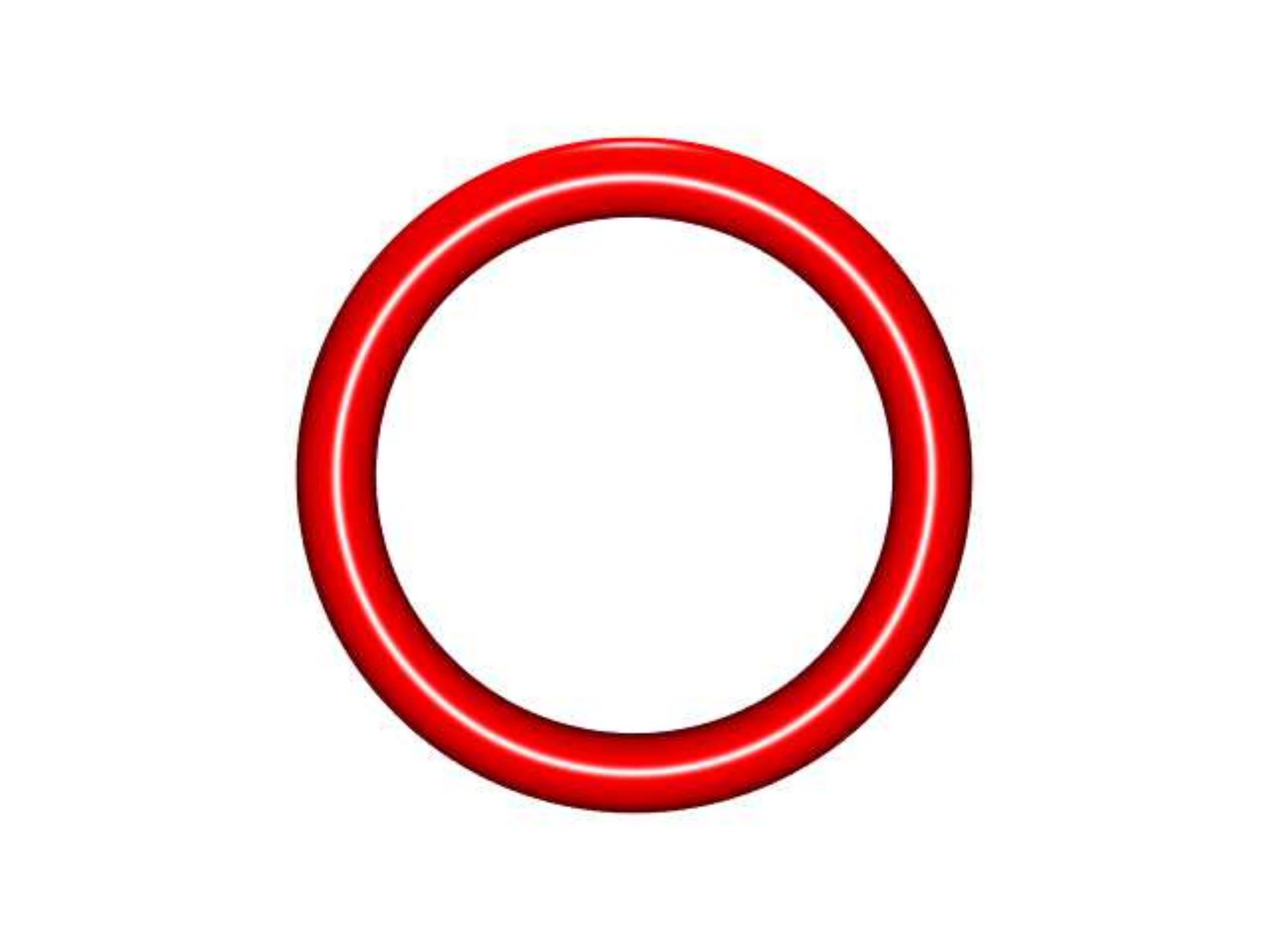} & \includegraphics[width=0.33\textwidth,keepaspectratio]{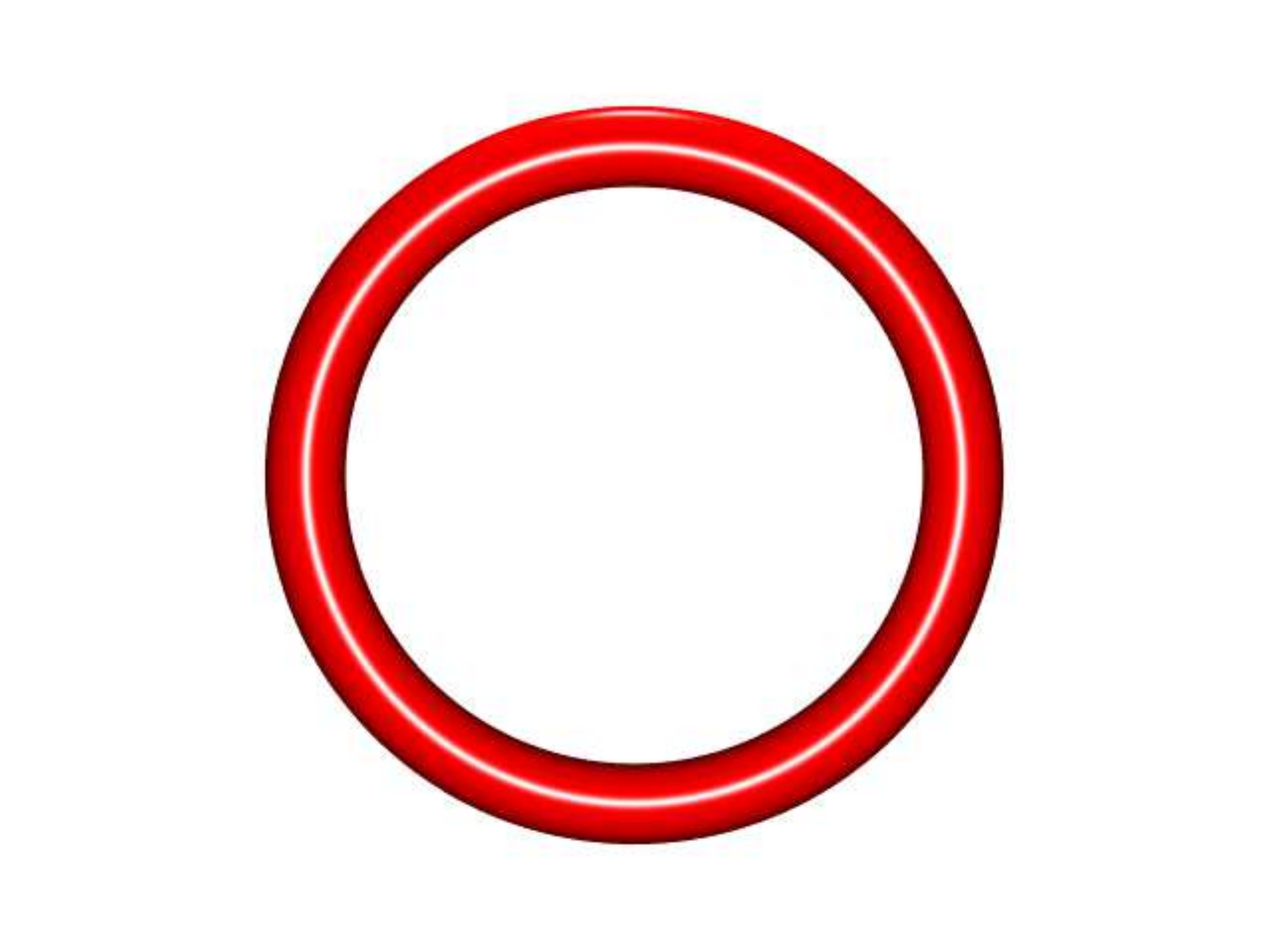} & \includegraphics[width=0.33\textwidth,keepaspectratio]{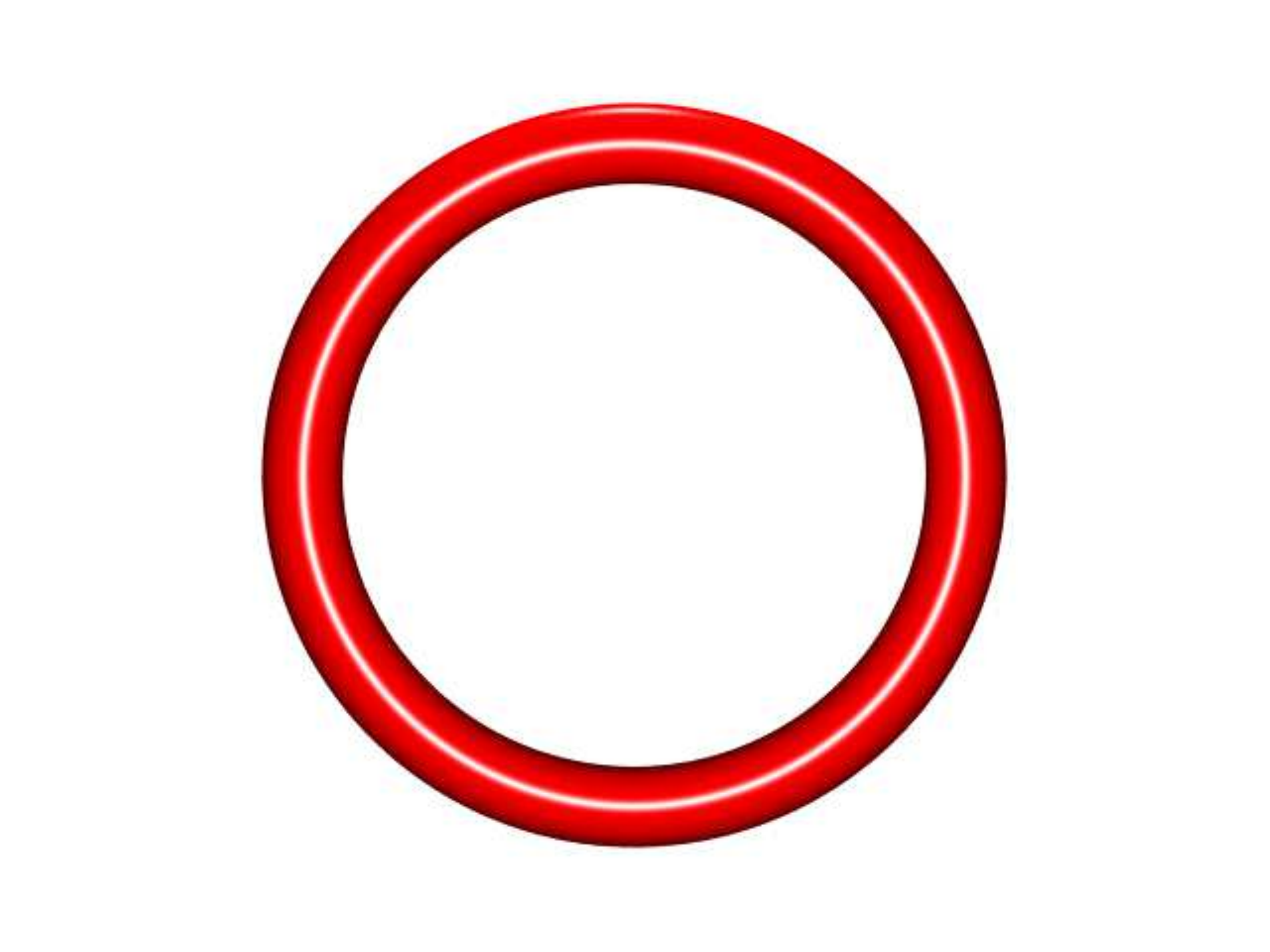} \\
270000/450000 & 360000/450000 & 450000/450000 \\
$\Le(\gamma)\approx 6.99951$ & $\Le(\gamma)\approx 6.99996$ & $\Le(\gamma)\approx 7.0$ \\
$\E_p(\gamma)\approx 6.28319$ & $\E_p(\gamma)\approx 6.28319$ & $\E_p(\gamma)\approx 6.28319$ \\
$\tau=50.95848$ & $\tau=68.71392$ & $\tau=86.47068$ \\
\includegraphics[width=0.33\textwidth,keepaspectratio]{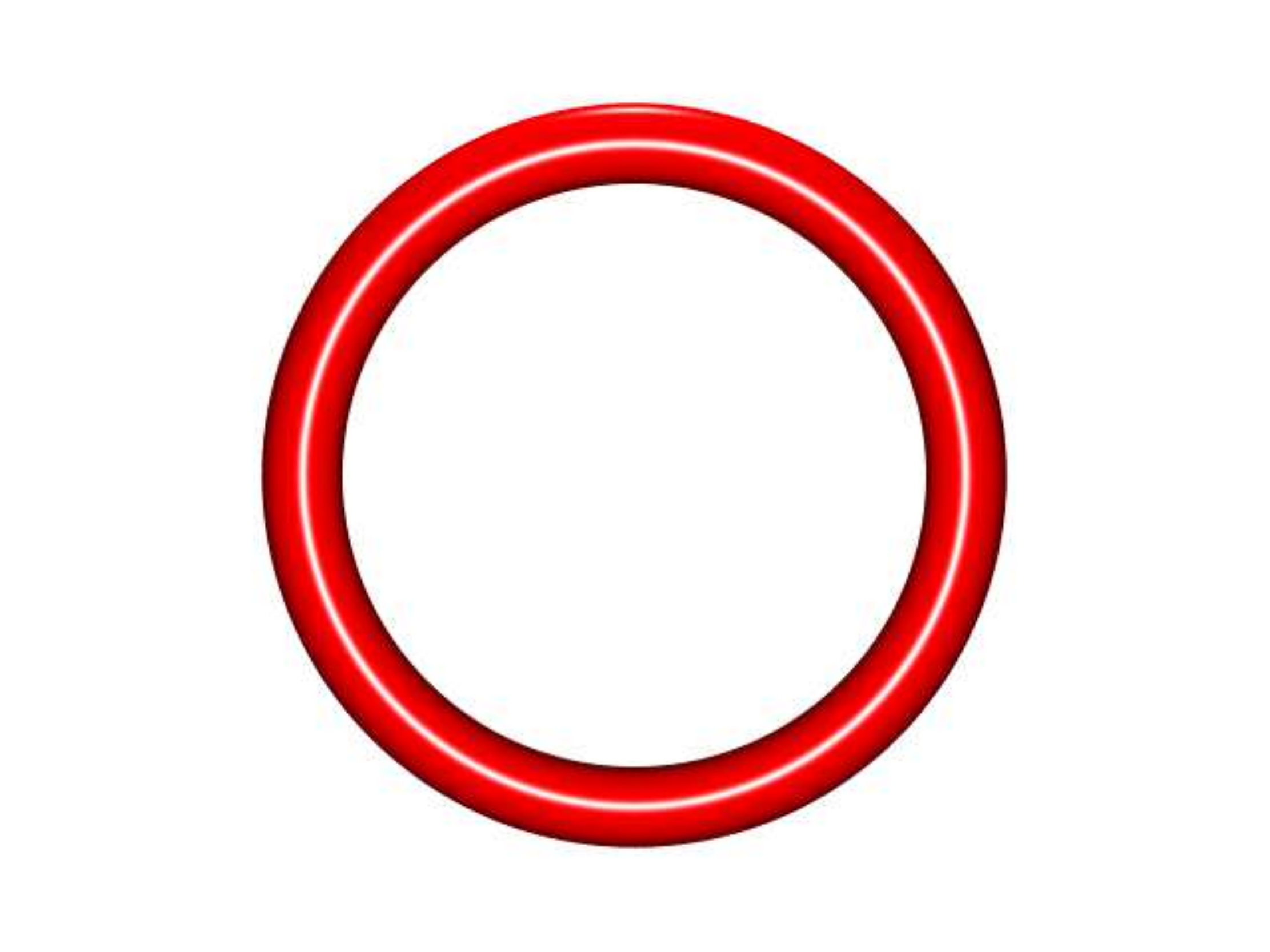} & \includegraphics[width=0.33\textwidth,keepaspectratio]{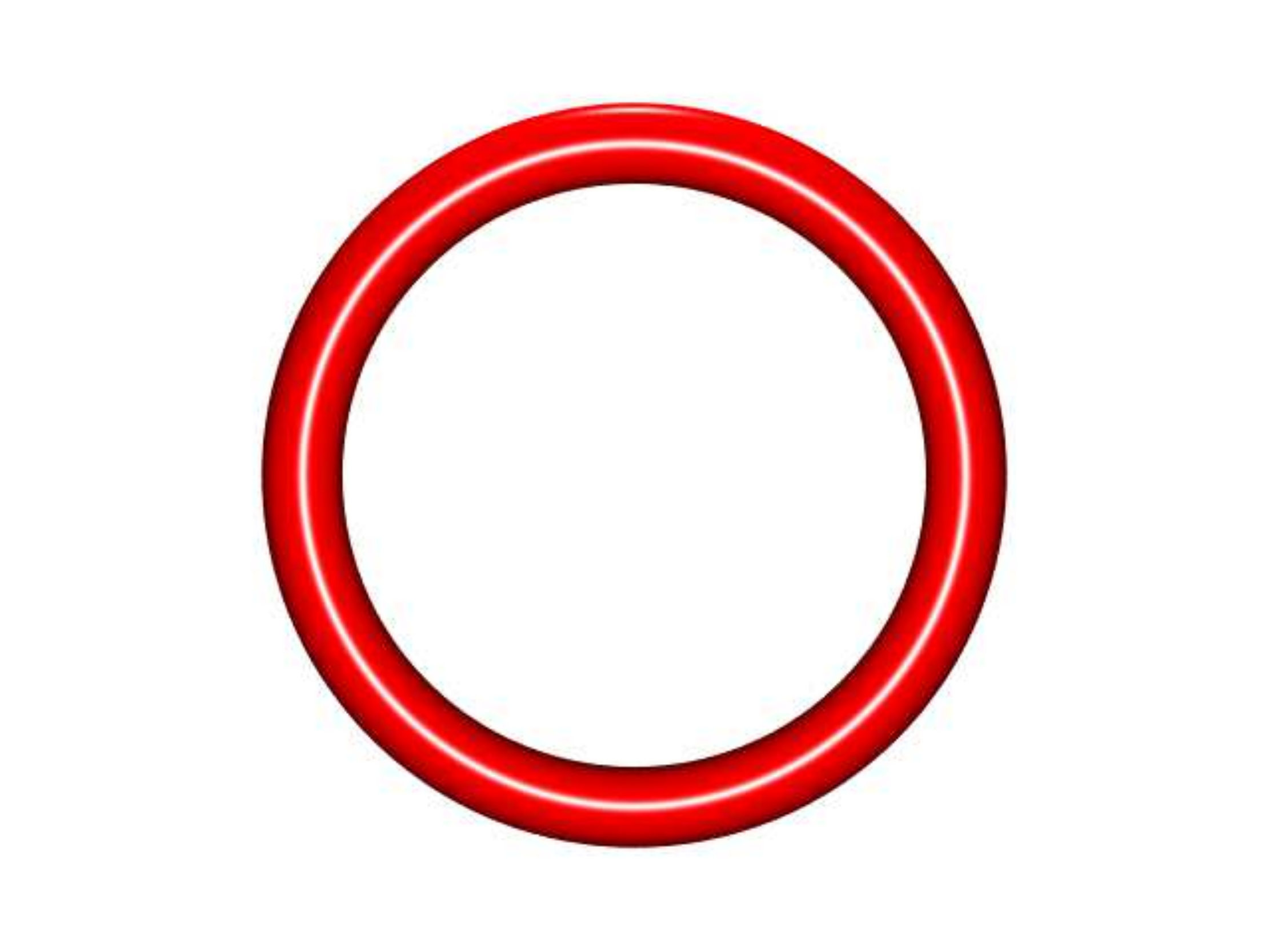} & \includegraphics[width=0.33\textwidth,keepaspectratio]{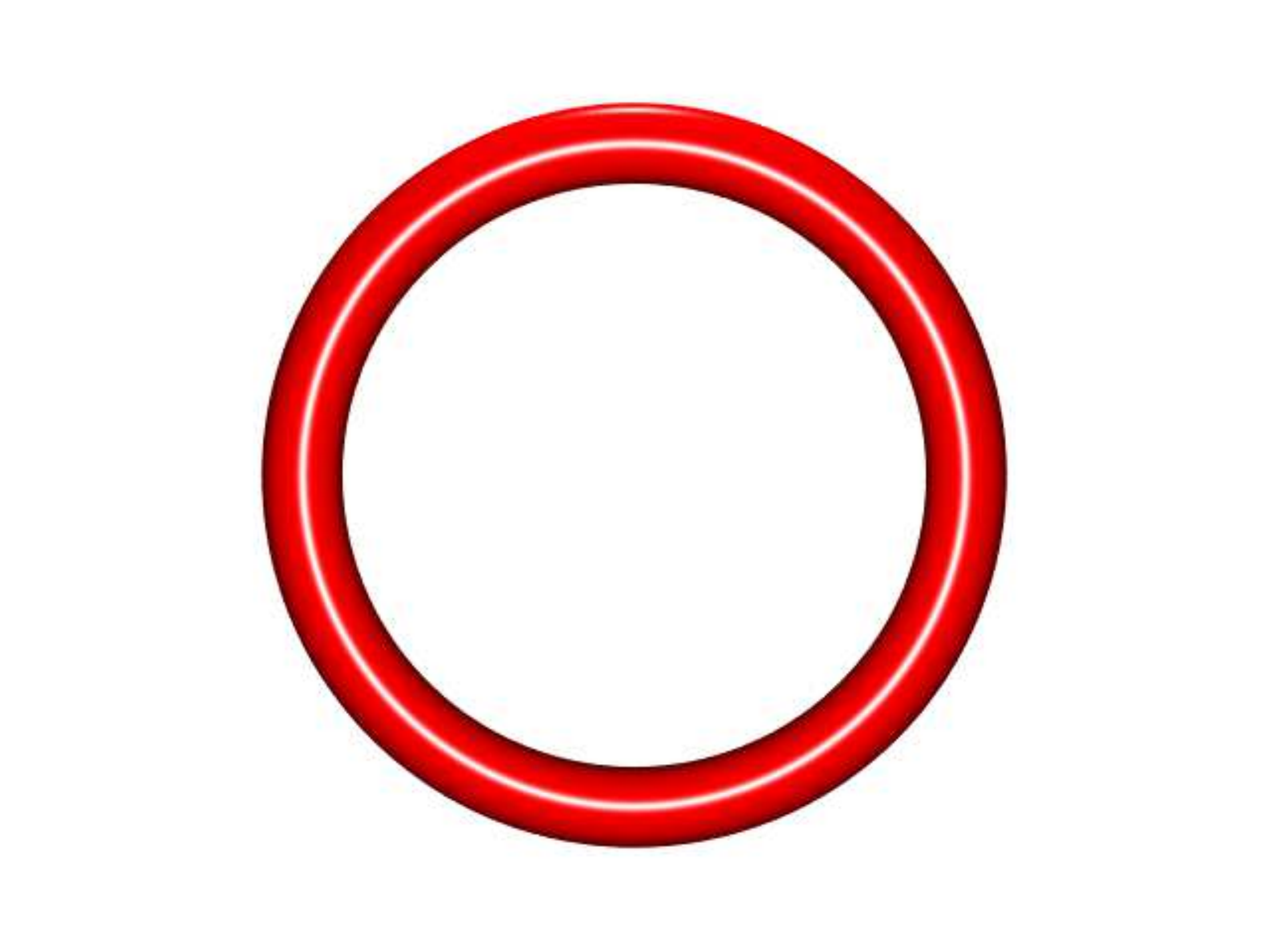}
\end{tabular}
\end{scriptsize}
\end{center}
\caption{A circle -- $p=4.0$ and fixed $\lambda=0.13796$ ($\tau_{\text{max}}=0.1$, $\eps=0.1$)}
\label{exmpcirclelambda}
\end{figure}

Next we come back to our standard flow for integral \name{Menger} curvature with redistribution. We consider an example of a non-trivial unknot and we see that the algorithm is able to untangle it. The idea for this configuration is due to \name{R. Kusner} (personal communication, July 2011).

\newpage
At first we look at this unknot, which is rendered with a smaller thickness here, in order to get a better idea of its configuration.
\begin{figure}[H]
\begin{center}
\includegraphics[width=0.4\textwidth,keepaspectratio]{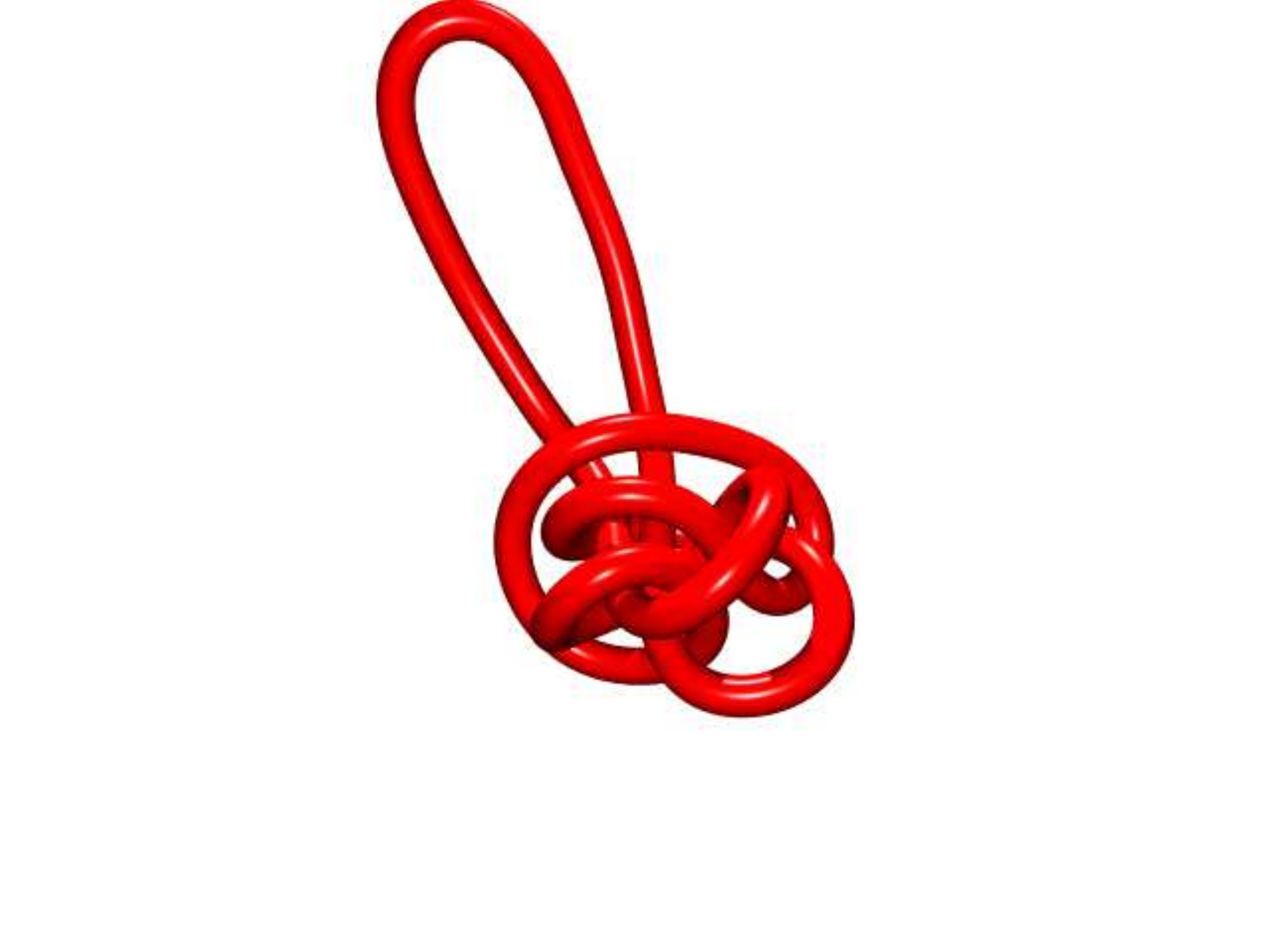} 
\end{center}
\caption{An entangled unknot}
\end{figure}
Now we consider the corresponding flow, where we watch the same unknot from another perspective.
\begin{figure}[H]
\begin{center}
\begin{scriptsize}
\begin{tabular}{ccc}
0/5000 & 960/5000 & 1060/5000 \\
$\Le(\gamma)\approx 97.85994$ & $\Le(\gamma)\approx 97.59822$ & $\Le(\gamma)\approx 87.09058$ \\
$\E_p(\gamma)\approx 147.84663$ & $\E_p(\gamma)\approx 93.1437$ & $\E_p(\gamma)\approx 81.7187$ \\
$\tau=0.0$ & $\tau=7.79625$ & $\tau=8.79271$ \\
\includegraphics[width=0.33\textwidth,keepaspectratio]{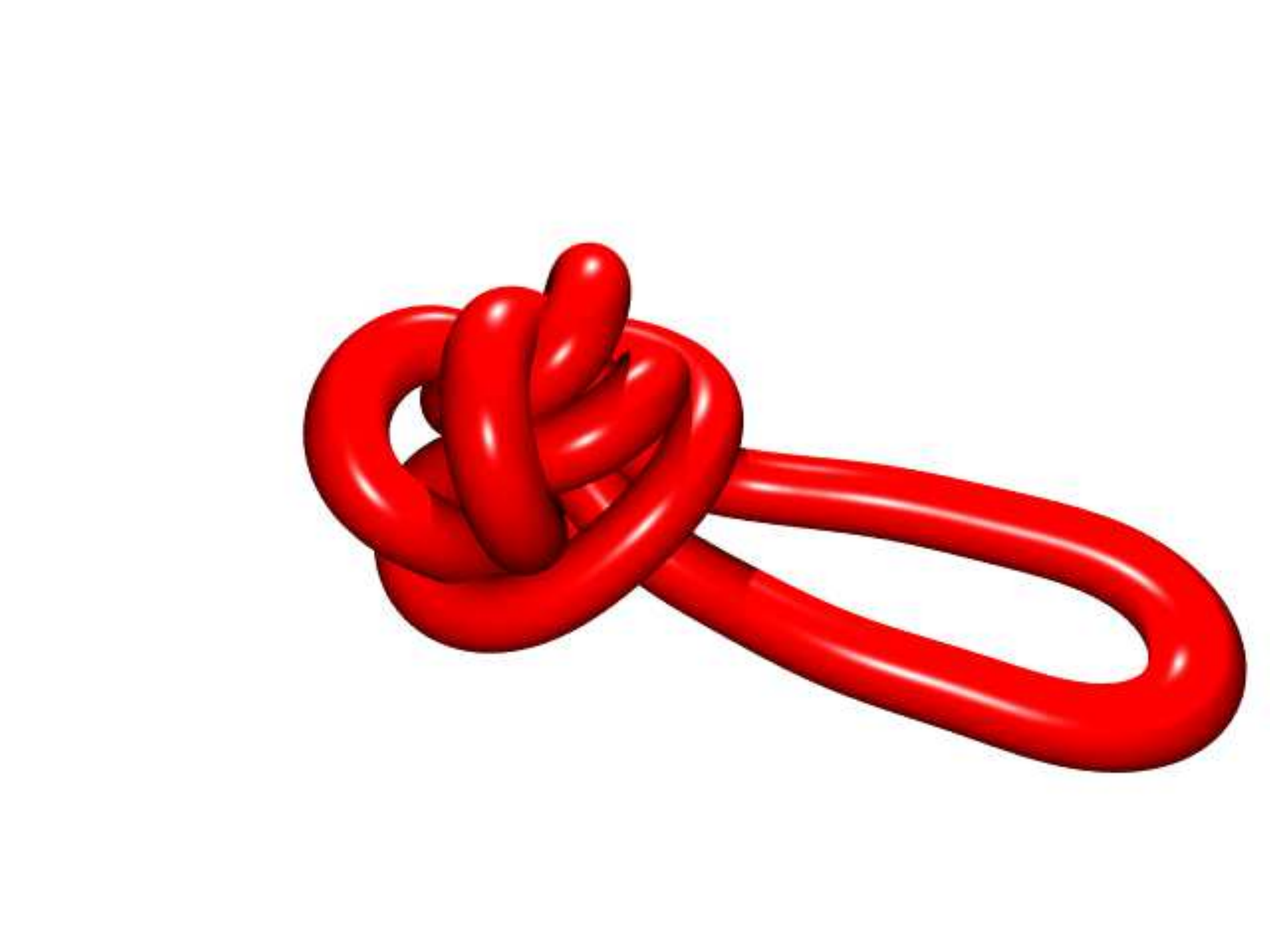} & \includegraphics[width=0.33\textwidth,keepaspectratio]{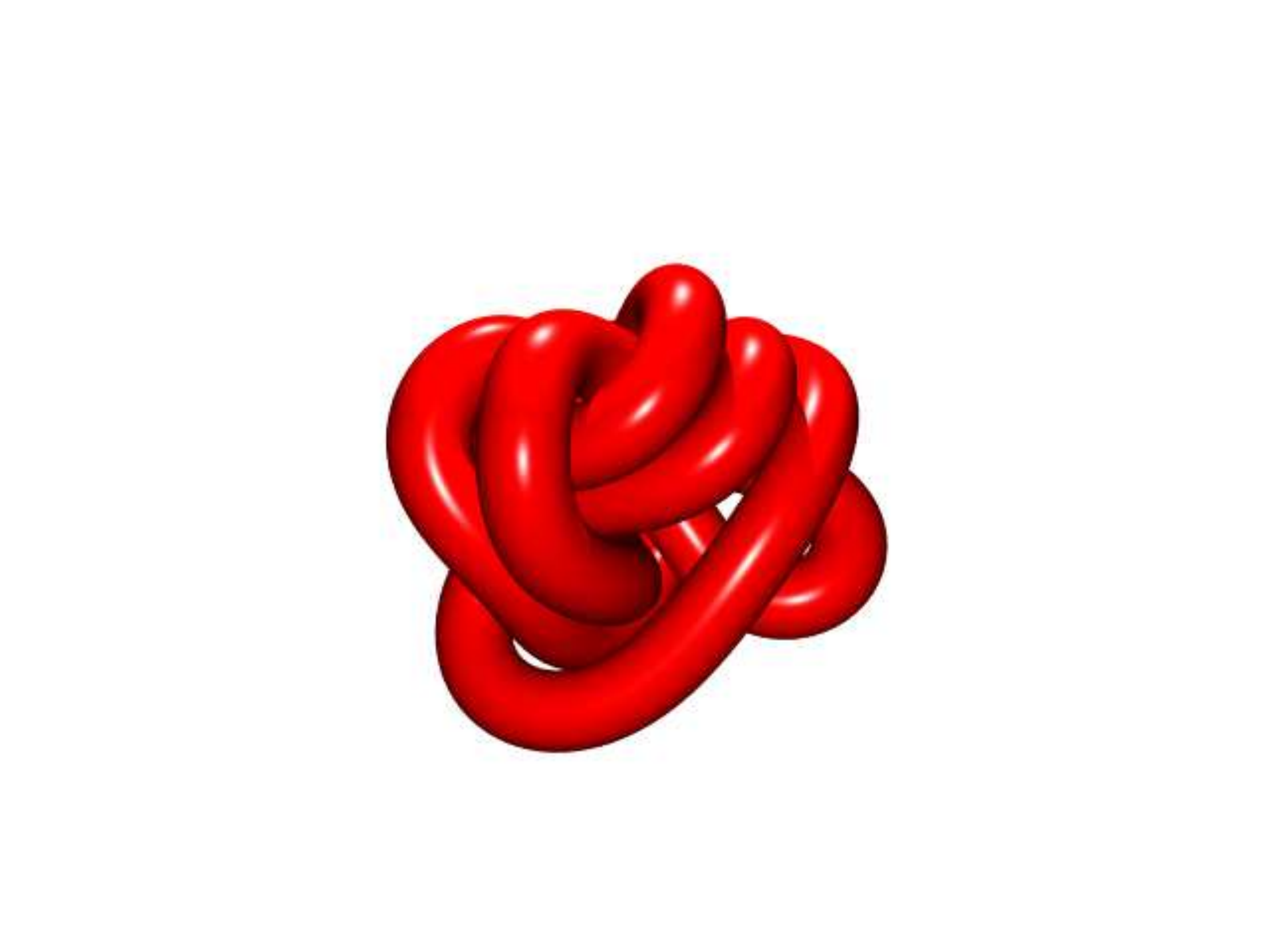} & \includegraphics[width=0.33\textwidth,keepaspectratio]{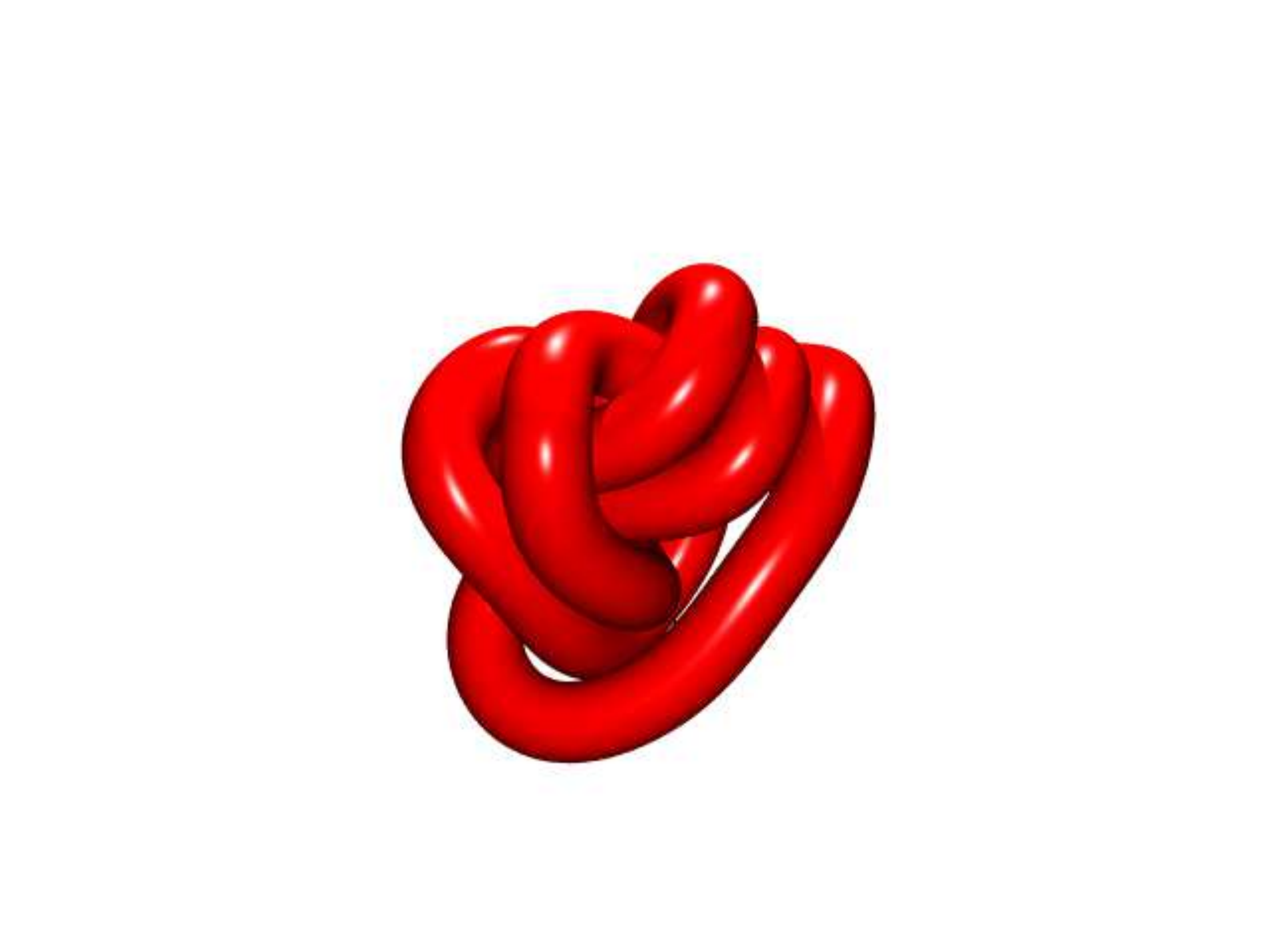} \\
1560/5000 & 2600/5000 & 5000/5000 \\
$\Le(\gamma)\approx 64.73693$ & $\Le(\gamma)\approx 49.18245$ & $\Le(\gamma)\approx 41.7716$ \\
$\E_p(\gamma)\approx 47.72924$ & $\E_p(\gamma)\approx 32.91997$ & $\E_p(\gamma)\approx 22.92765$ \\
$\tau=13.79271$ & $\tau=20.6022$ & $\tau=26.5256$ \\
\includegraphics[width=0.33\textwidth,keepaspectratio]{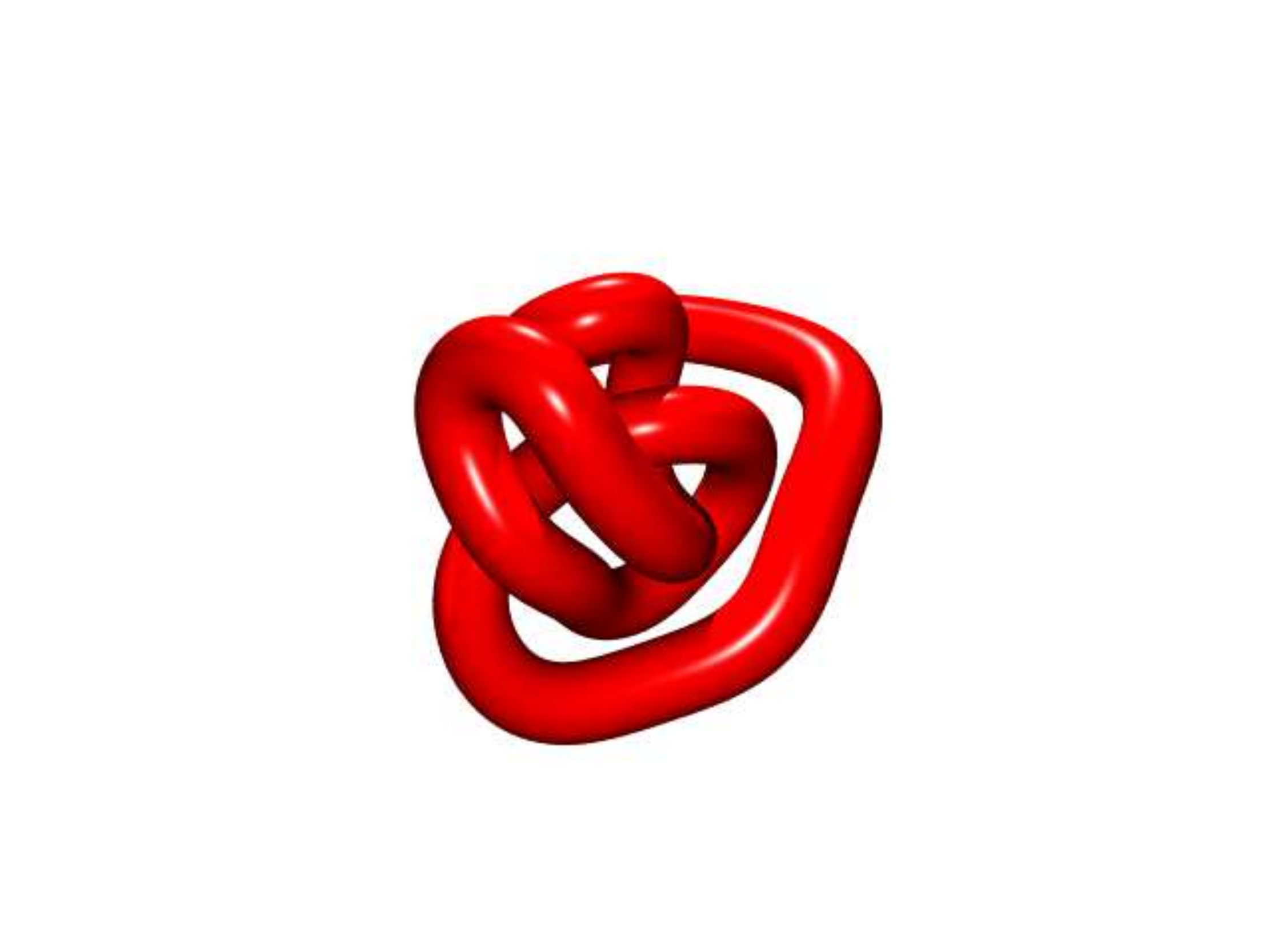} & \includegraphics[width=0.33\textwidth,keepaspectratio]{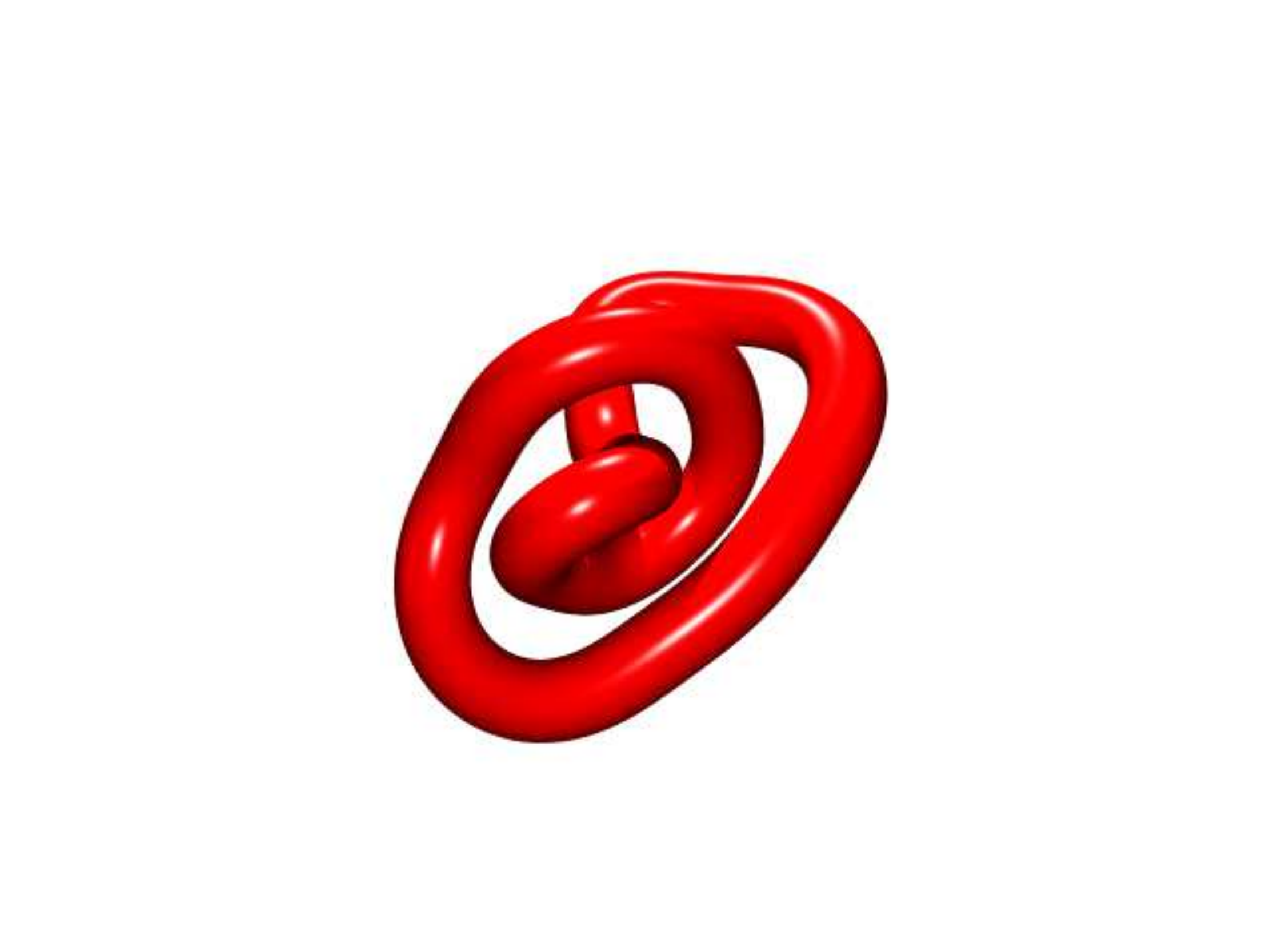} & \includegraphics[width=0.33\textwidth,keepaspectratio]{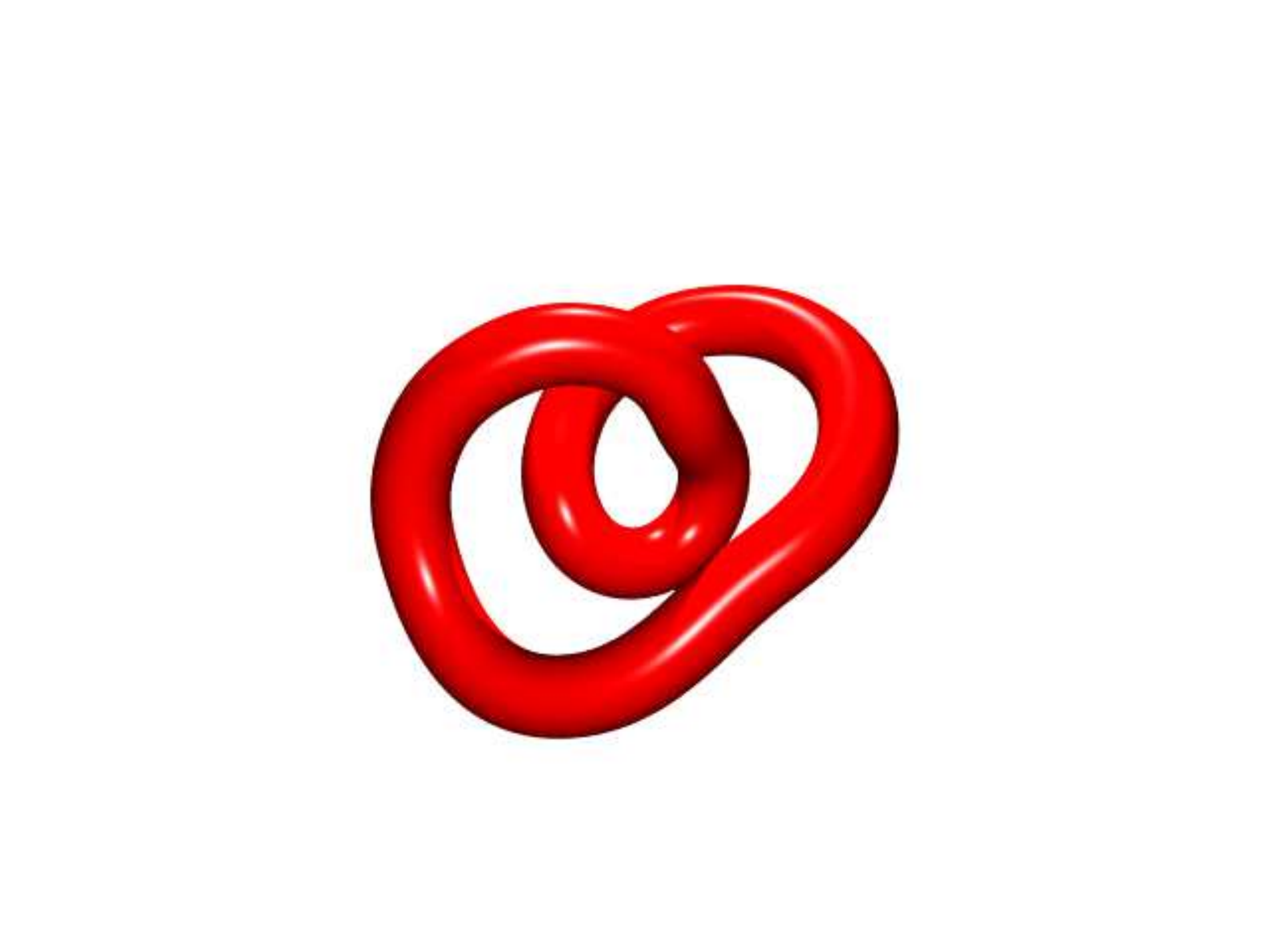}
\end{tabular}
\end{scriptsize}
\end{center}
\caption{An entangled unknot -- $p=50.0$}
\end{figure}

We restart the flow with the last configuration. Observe, that the redistribution algorithm is also applied to the initial configuration. Hence there are differences in the length and energy values between the last configuration of the prior flow and the starting configuration of the following flow. To avoid these, it is also possible to suppress the first redistribution in the algorithm.
\begin{figure}[H]
\begin{center}
\begin{scriptsize}
\begin{tabular}{ccc}
0/195000 & 32500/195000 & 65000/195000 \\
$\Le(\gamma)\approx 41.77149$ & $\Le(\gamma)\approx 27.13459$ & $\Le(\gamma)\approx 25.83265$ \\
$\E_p(\gamma)\approx 22.92858$ & $\E_p(\gamma)\approx 8.0456$ & $\E_p(\gamma)\approx 6.28319$ \\
$\tau=0.0$ & $\tau=23.052$ & $\tau=106.26745$ \\
\includegraphics[width=0.33\textwidth,keepaspectratio]{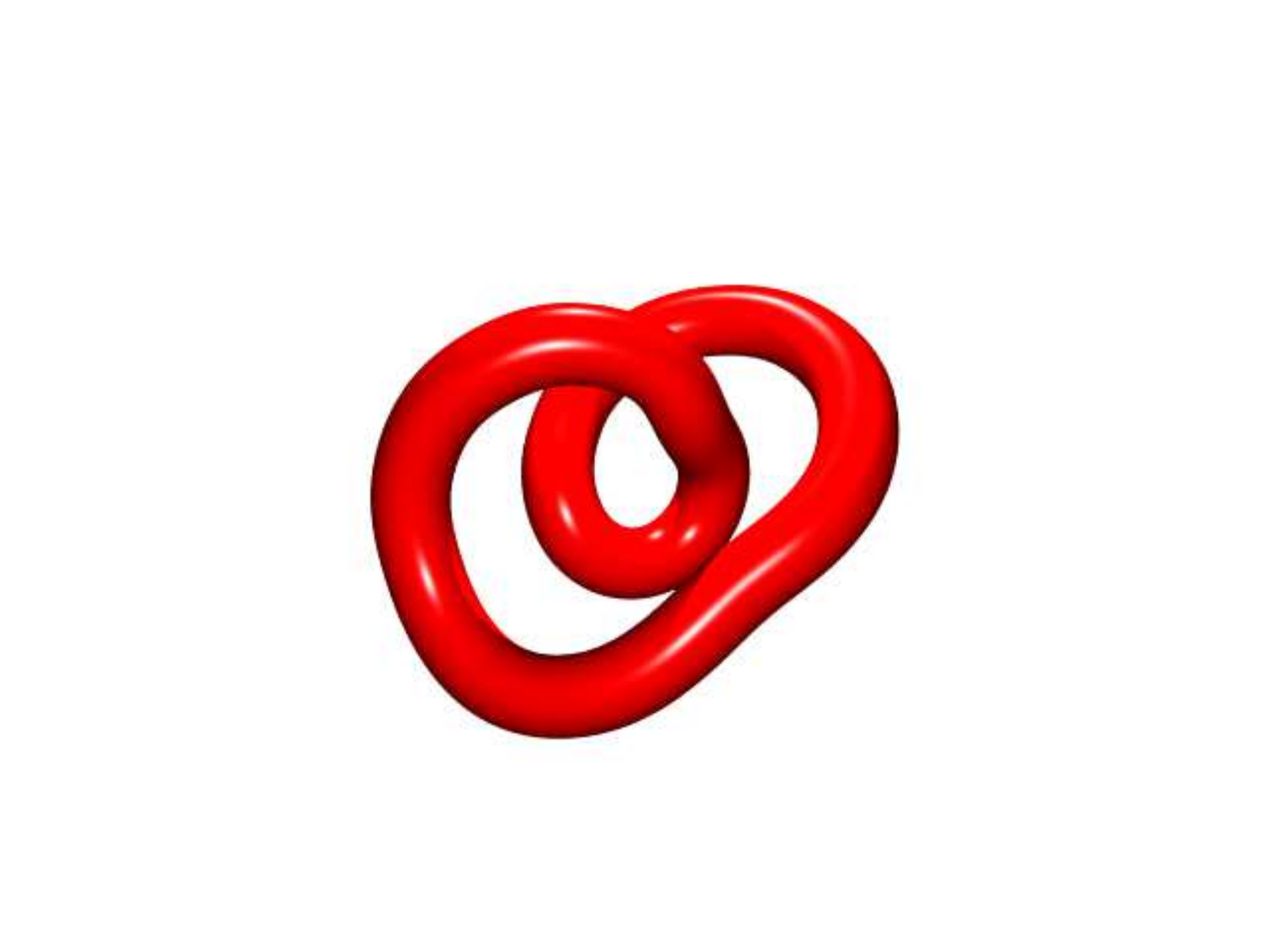} & \includegraphics[width=0.33\textwidth,keepaspectratio]{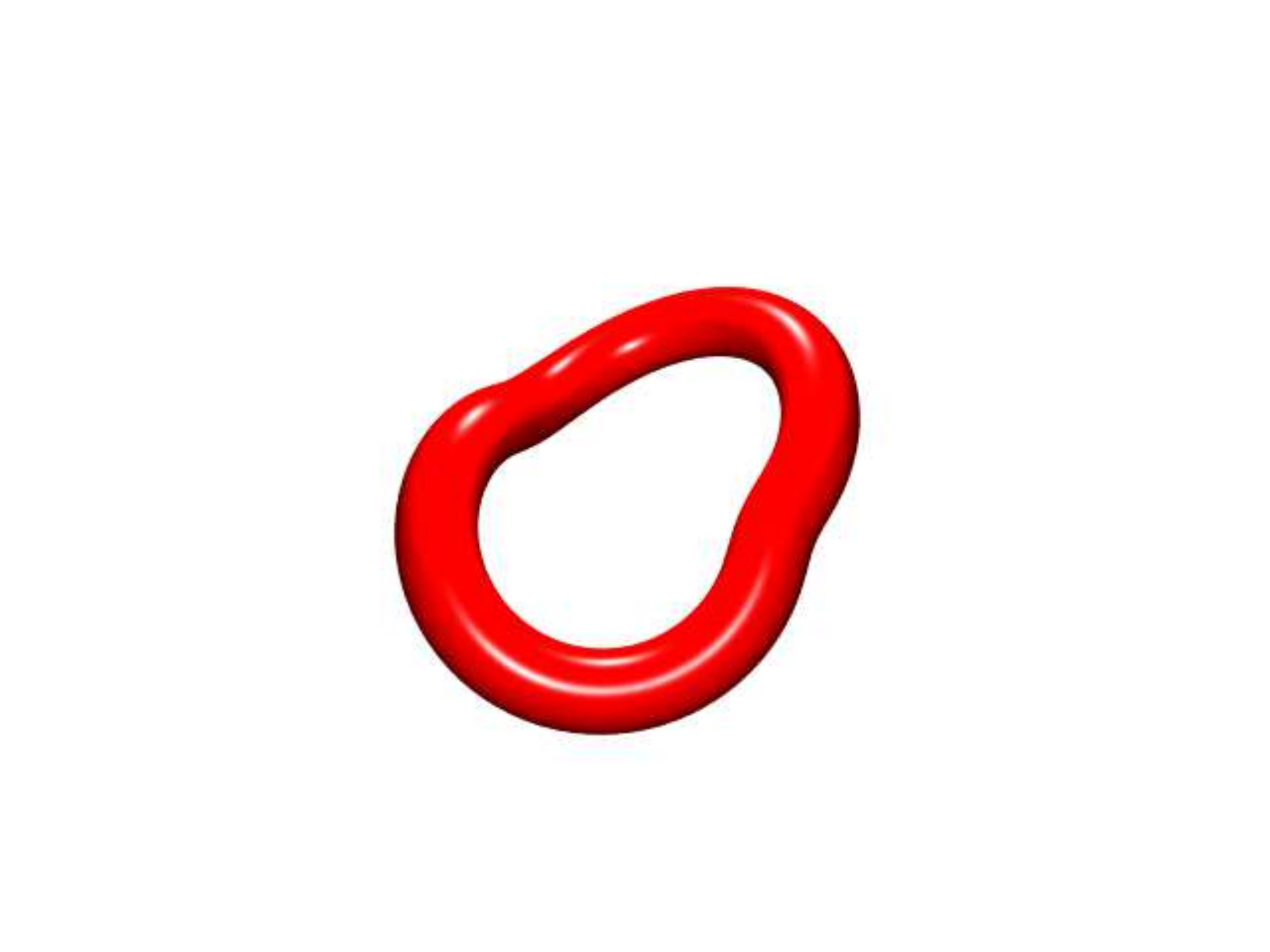} & \includegraphics[width=0.33\textwidth,keepaspectratio]{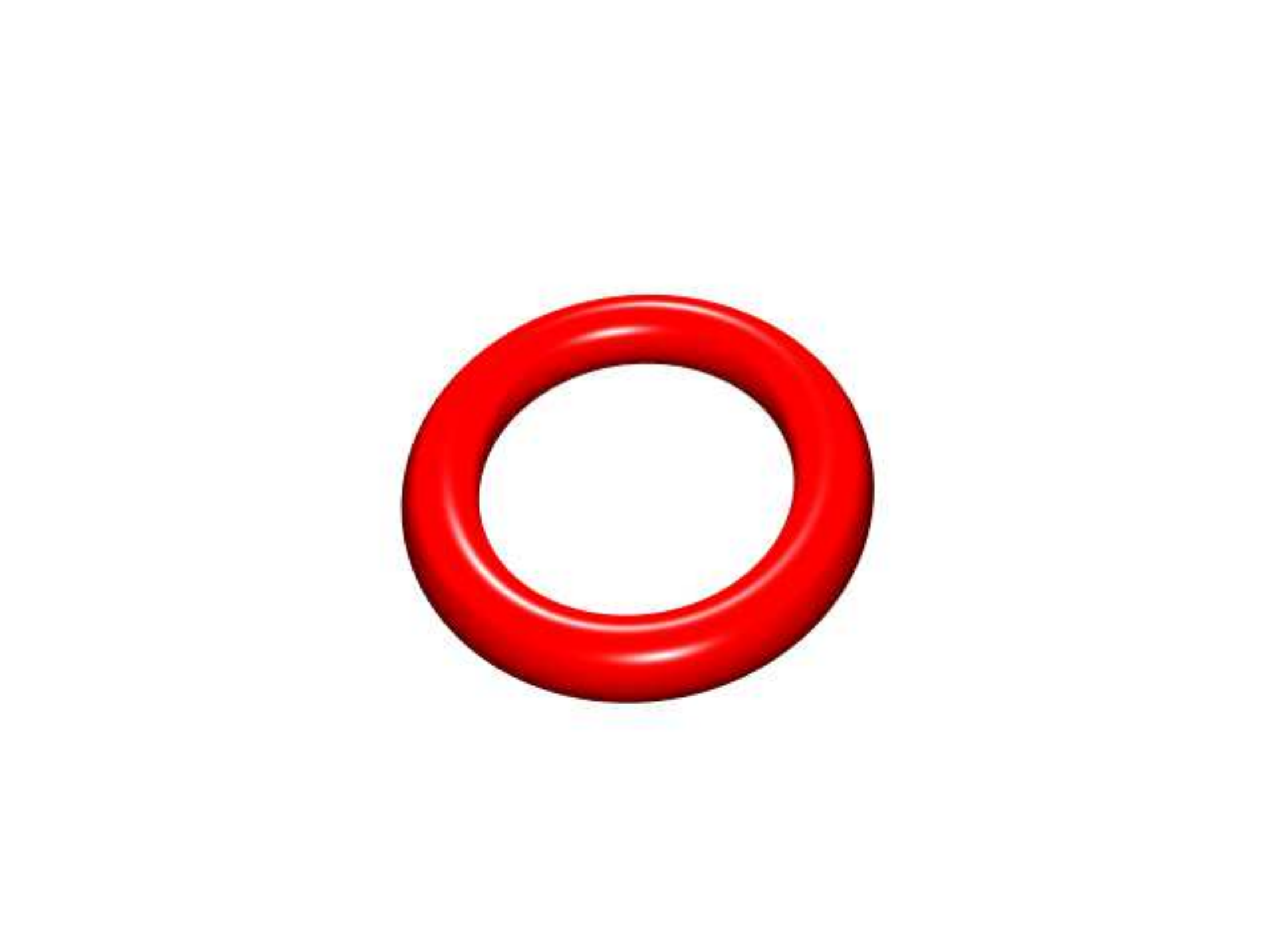}
\end{tabular}
\end{scriptsize}
\end{center}
\caption{An entangled unknot -- $p=50.0$ (continued)}
\end{figure}

\newpage
\section{Trefoils ($3_1$)} \label{knot31}
In this section we are dealing with the easiest knot one can imagine. The trefoil knot has a crossing number of $3$. We start with applying a trefoil-knot to the flow without redistribution
\begin{figure}[H]
\begin{center}
\begin{scriptsize}
\begin{tabular}{ccc}
0/400000 & 1000/400000 & 50000/400000 \\
$\Le(\gamma)\approx 35.43384$ & $\Le(\gamma)\approx 33.24154$ & $\Le(\gamma)\approx 32.79838$ \\
$\E_p(\gamma)\approx 17.20705$ & $\E_p(\gamma)\approx 16.46783$ & $\E_p(\gamma)\approx 16.40439$ \\
$\tau=0.0$ & $\tau=9.22458$ & $\tau=429.17874$ \\
\includegraphics[width=0.33\textwidth,keepaspectratio]{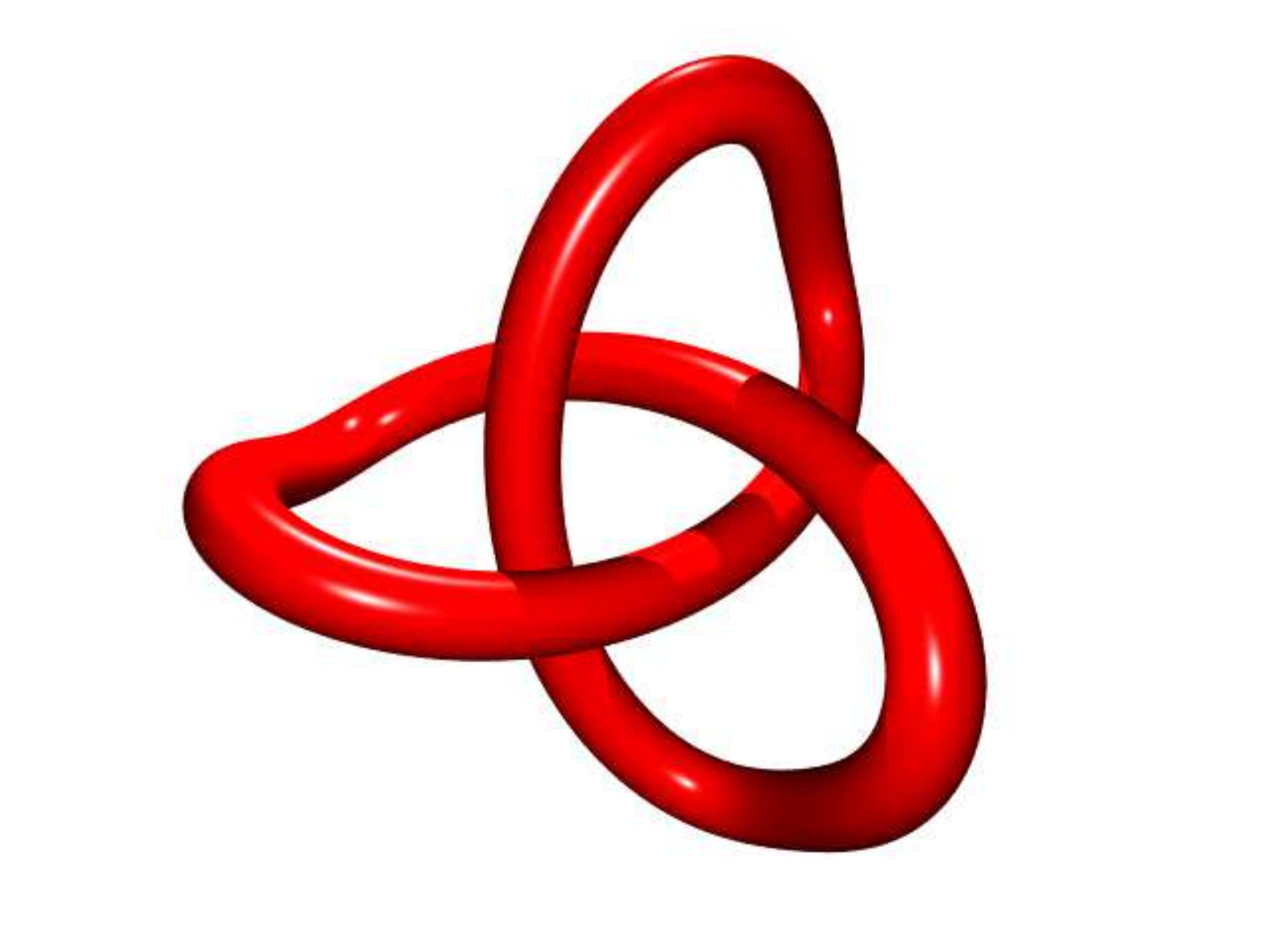} & \includegraphics[width=0.33\textwidth,keepaspectratio]{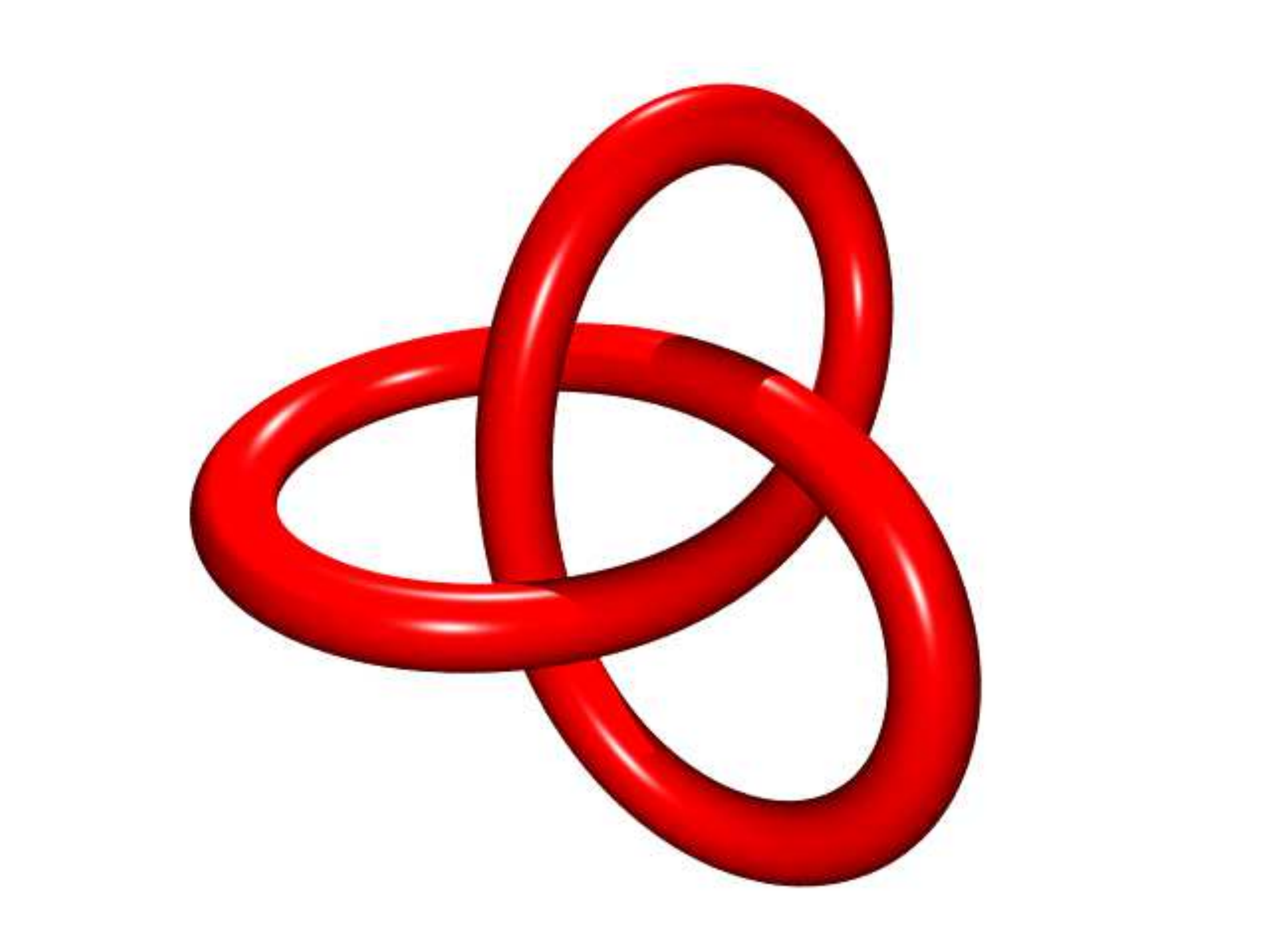} & \includegraphics[width=0.33\textwidth,keepaspectratio]{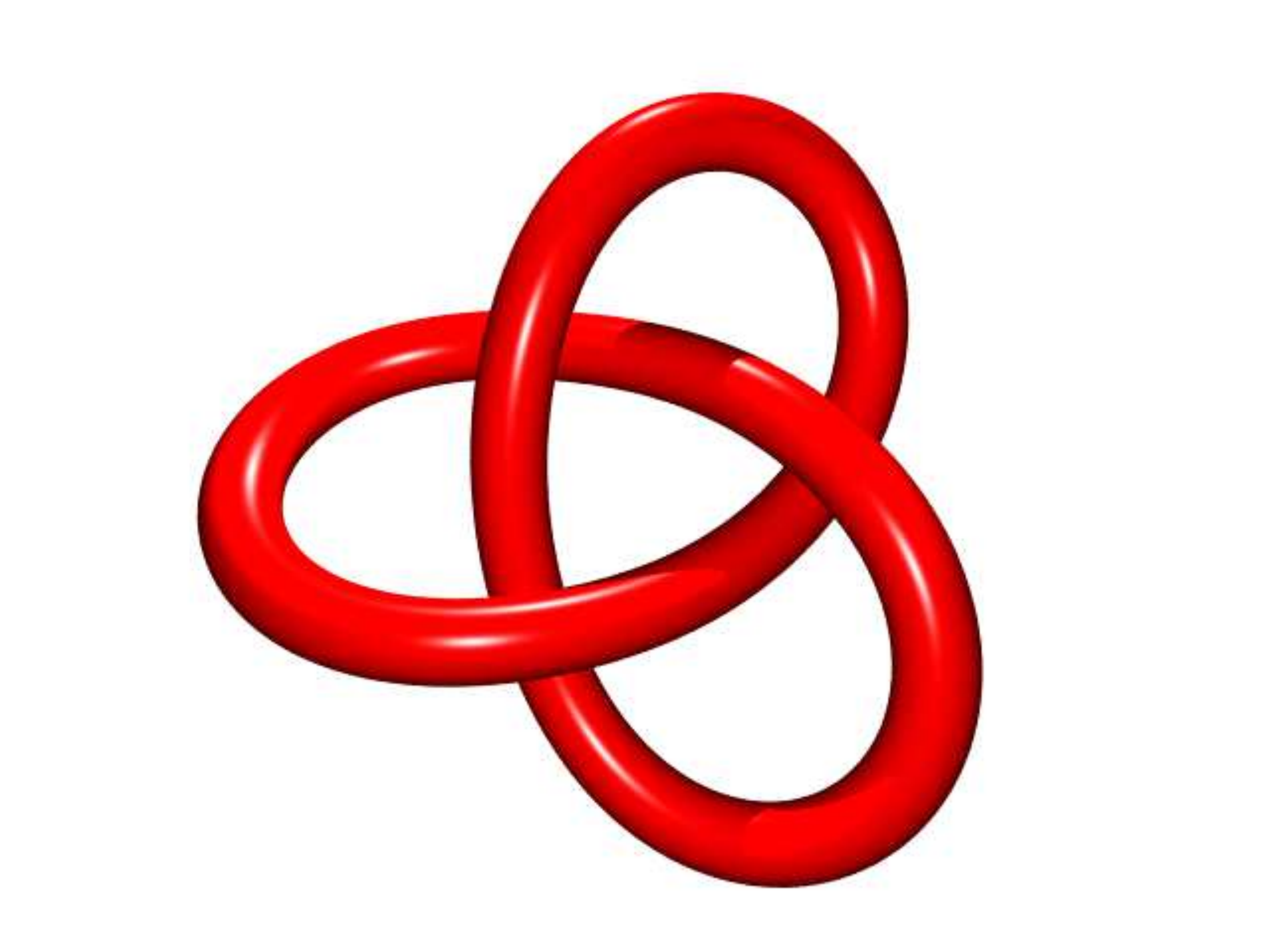} \\
85000/400000 & 100000/400000 & 400000/400000 \\
$\Le(\gamma)\approx 31.56549$ & $\Le(\gamma)\approx 30.25214$ & $\Le(\gamma)\approx 30.04025$ \\
$\E_p(\gamma)\approx 16.28163$ & $\E_p(\gamma)\approx 16.05689$ & $\E_p(\gamma)\approx 16.01051$ \\
$\tau=681.12347$ & $\tau=692.24391$ & $\tau=693.48698$ \\
\includegraphics[width=0.33\textwidth,keepaspectratio]{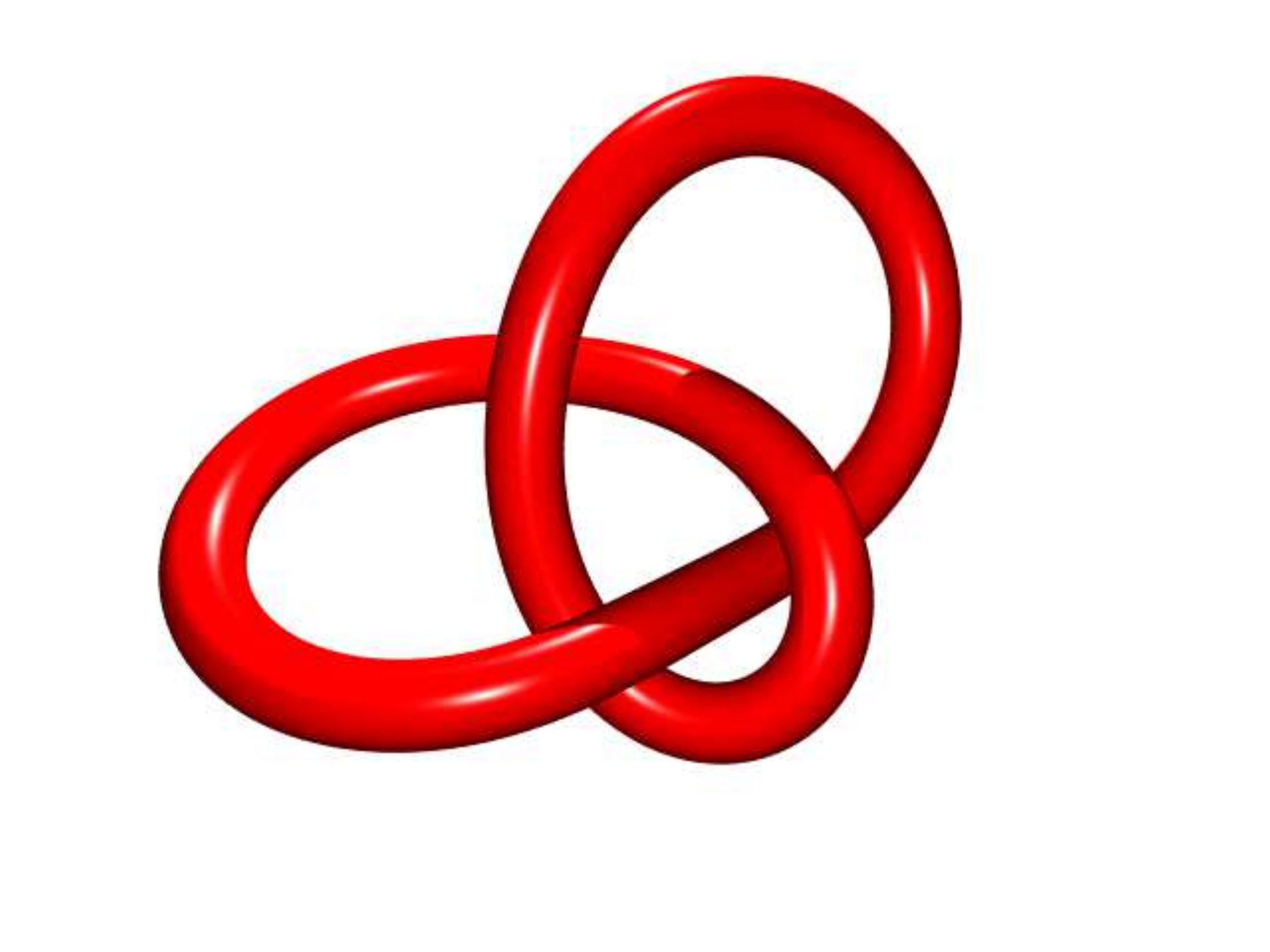} & \includegraphics[width=0.33\textwidth,keepaspectratio]{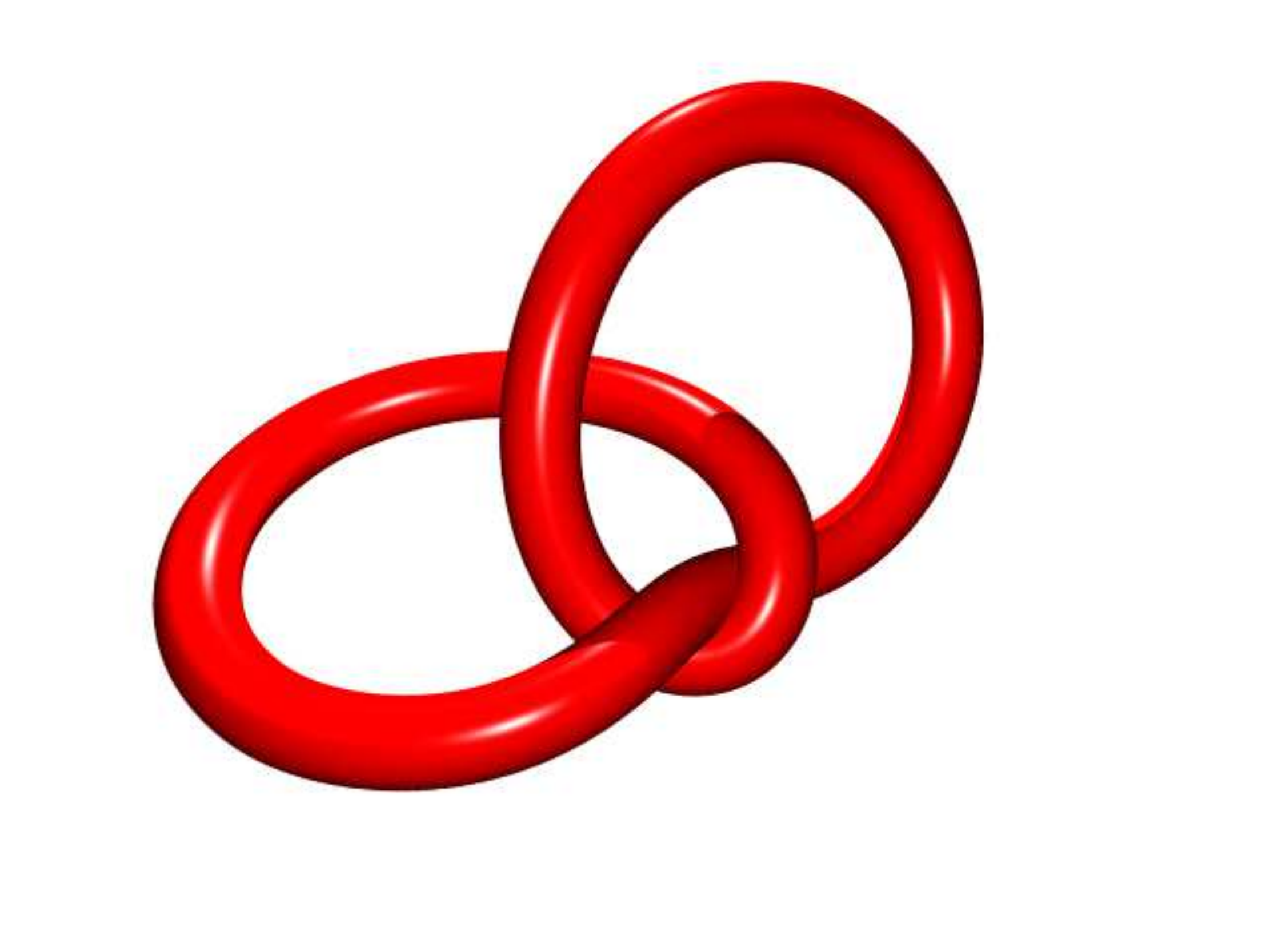} & \includegraphics[width=0.33\textwidth,keepaspectratio]{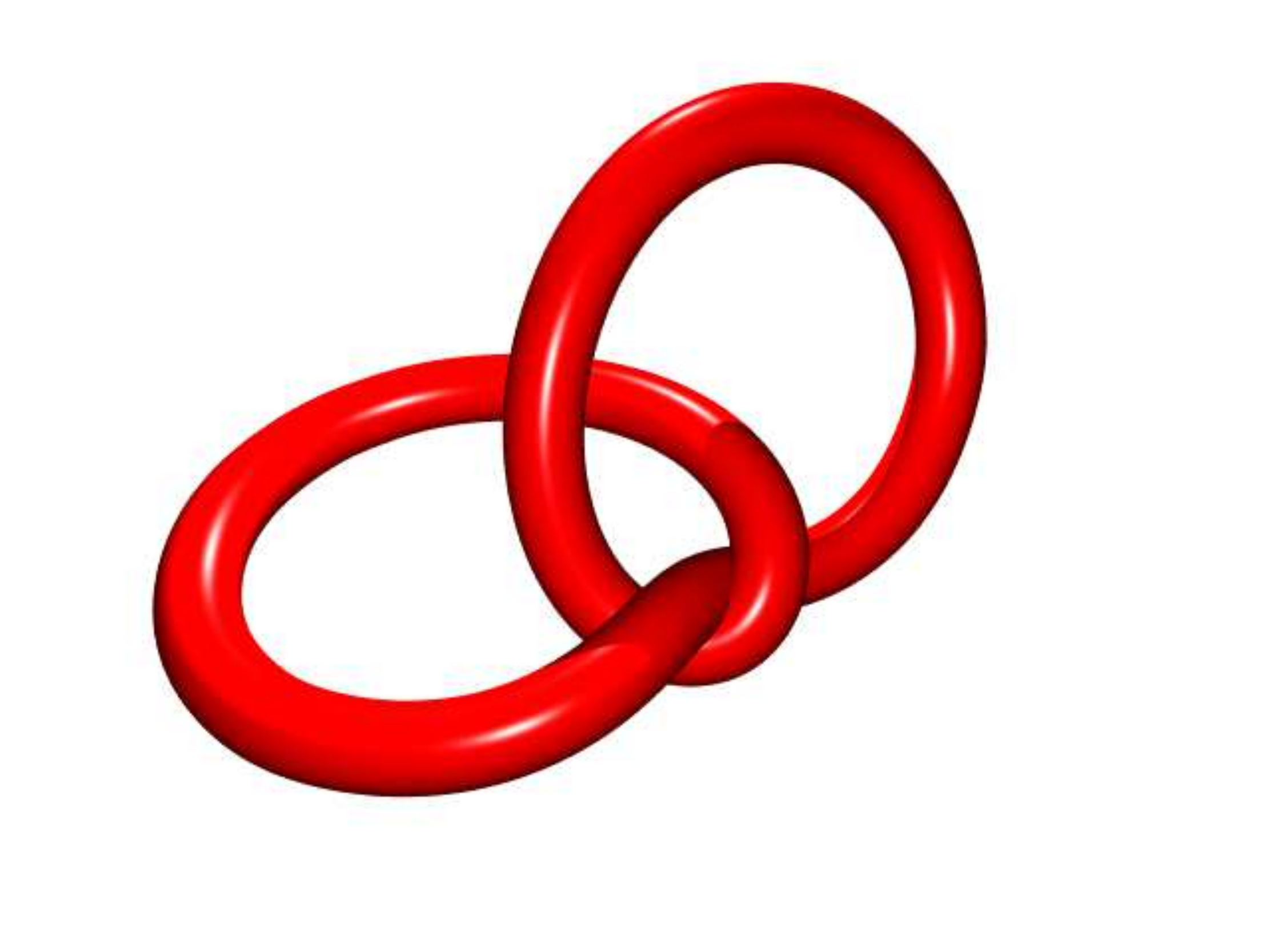}
\end{tabular}
\end{scriptsize}
\end{center}
\caption{A trefoil -- $p=3.0$ without redistribution ($\tau_{\text{max}}=0.1$, $\eps=0.05$)}
\end{figure}

However, at this position the flow seems to get stuck, which is indicated by a strong decrease of the calculated time step size. Now we add redistribution and use the same starting configuration as before, but viewed from another angle. This way we get the following
\begin{figure}[H]
\begin{center}
\begin{scriptsize}
\begin{tabular}{ccc}
0/400000 & 1000/400000 & 10000/400000 \\
$\Le(\gamma)\approx 35.42829$ & $\Le(\gamma)\approx 33.56804$ & $\Le(\gamma)\approx 32.78996$ \\
$\E_p(\gamma)\approx 17.20426$ & $\E_p(\gamma)\approx 16.54996$ & $\E_p(\gamma)\approx 16.40439$ \\
$\tau=0.0$ & $\tau=5.62177$ & $\tau=51.64631$ \\
\includegraphics[width=0.33\textwidth,keepaspectratio]{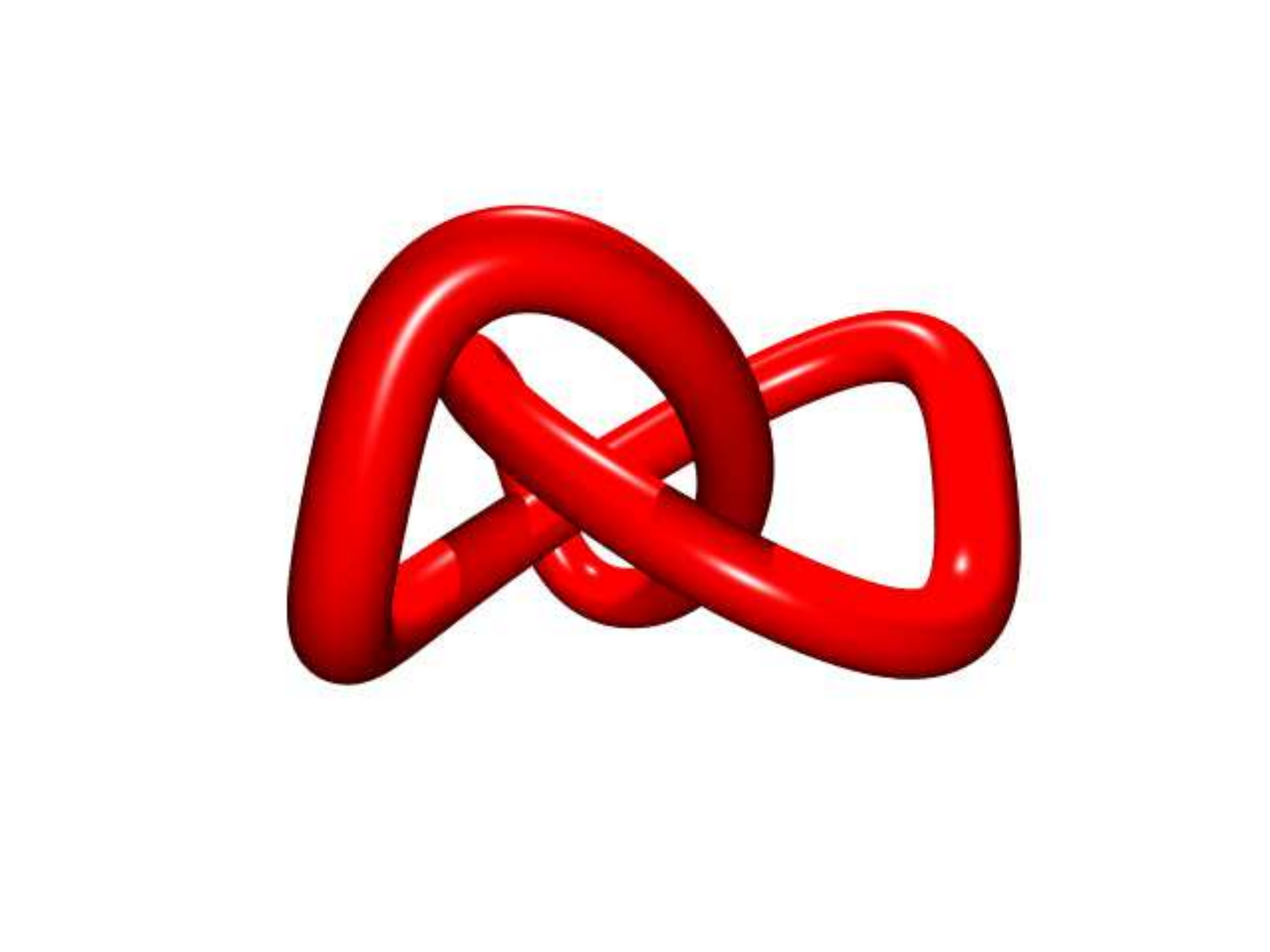} & \includegraphics[width=0.33\textwidth,keepaspectratio]{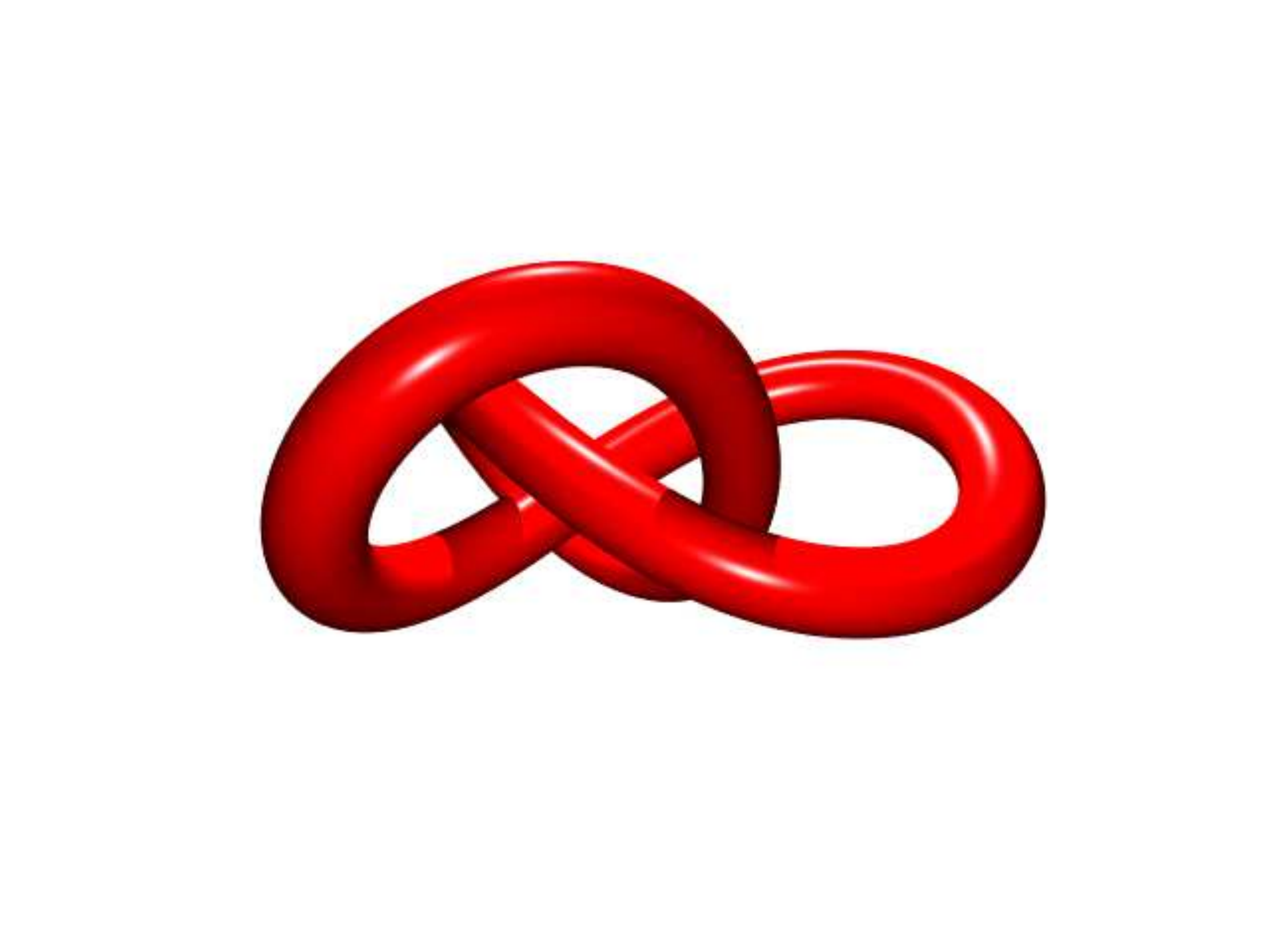} & \includegraphics[width=0.33\textwidth,keepaspectratio]{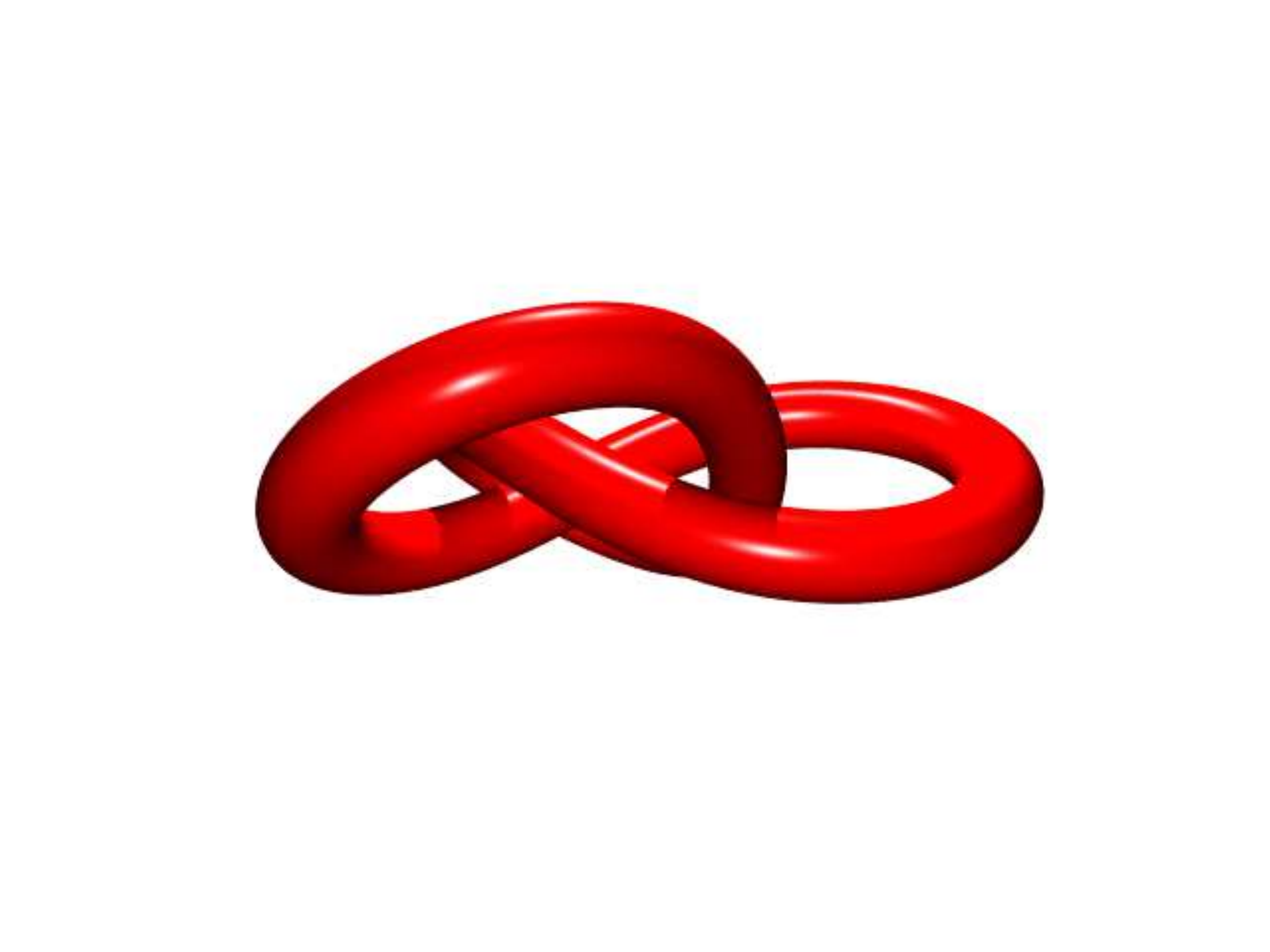} \\
170000/400000 & 200000/400000 & 400000/400000 \\
$\Le(\gamma)\approx 31.85679$ & $\Le(\gamma)\approx 27.7953$ & $\Le(\gamma)\approx 31.04354$ \\
$\E_p(\gamma)\approx 16.32394$ & $\E_p(\gamma)\approx 15.48498$ & $\E_p(\gamma)\approx 15.30027$ \\
$\tau=859.45056$ & $\tau=889.64543$ & $\tau=1006.80042$ \\
\includegraphics[width=0.33\textwidth,keepaspectratio]{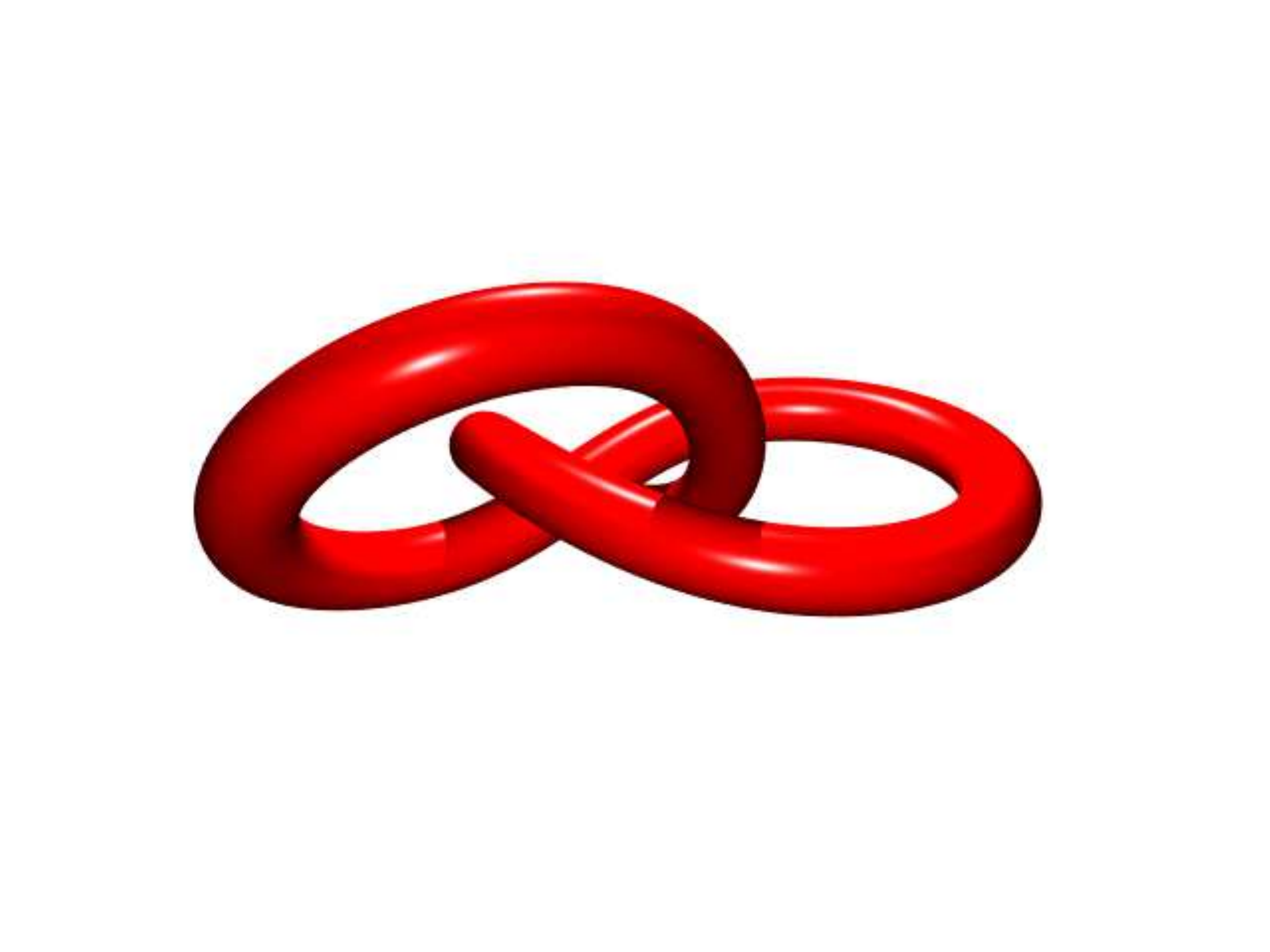} & \includegraphics[width=0.33\textwidth,keepaspectratio]{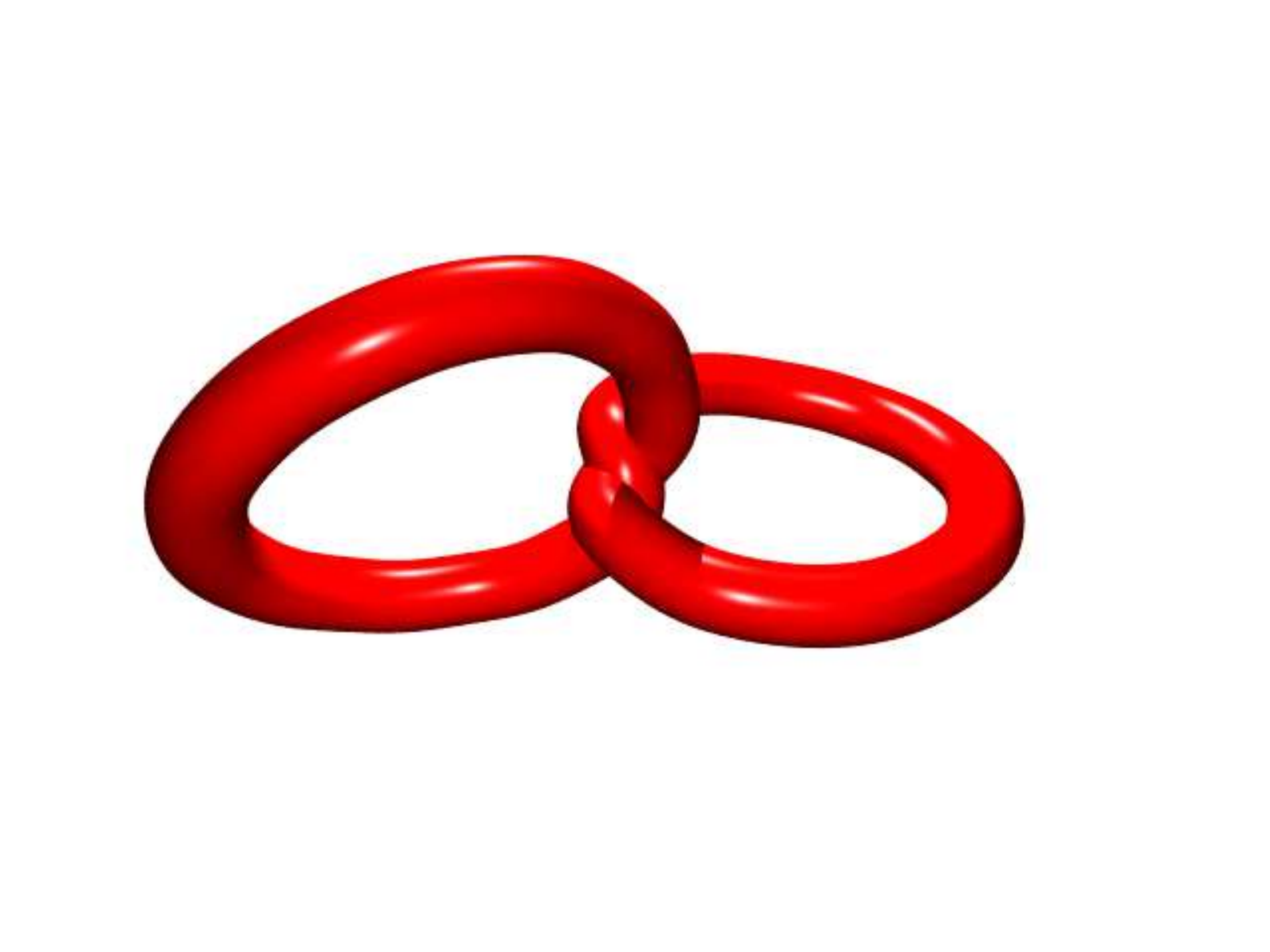} & \includegraphics[width=0.33\textwidth,keepaspectratio]{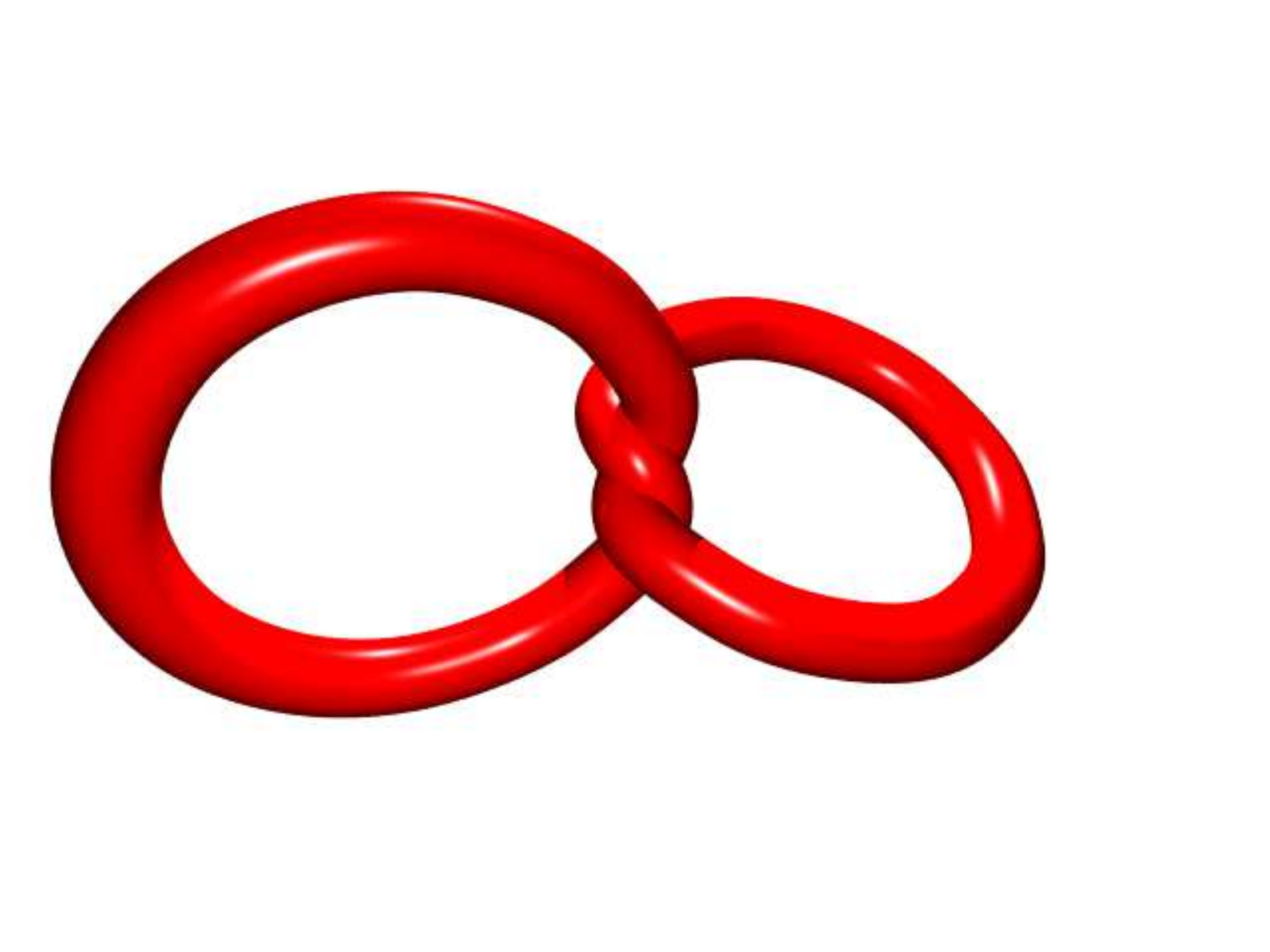}
\end{tabular}
\end{scriptsize}
\end{center}
\caption{A trefoil -- $p=3.0$ ($\tau_{\text{max}}=0.1$, $\eps=0.05$)}
\end{figure}

We continue without redistribution and we obtain that the knot class is abandoned here. Finally, we end up with a circle, which is also indicated by the $\E_p$ value of the last configuration.
\begin{figure}[H]
\begin{center}
\begin{scriptsize}
\begin{tabular}{ccc}
0/200000 & 55000/200000 & 200000/200000 \\
$\Le(\gamma)\approx 31.04354$ & $\Le(\gamma)\approx 27.51339$ & $\Le(\gamma)\approx 21.06661$ \\
$\E_p(\gamma)\approx 15.30027$ & $\E_p(\gamma)\approx 9.90118$ & $\E_p(\gamma)\approx 6.28319$ \\
$\tau=0.0$ & $\tau=24.55333$ & $\tau=317.84845$ \\
\includegraphics[width=0.33\textwidth,keepaspectratio]{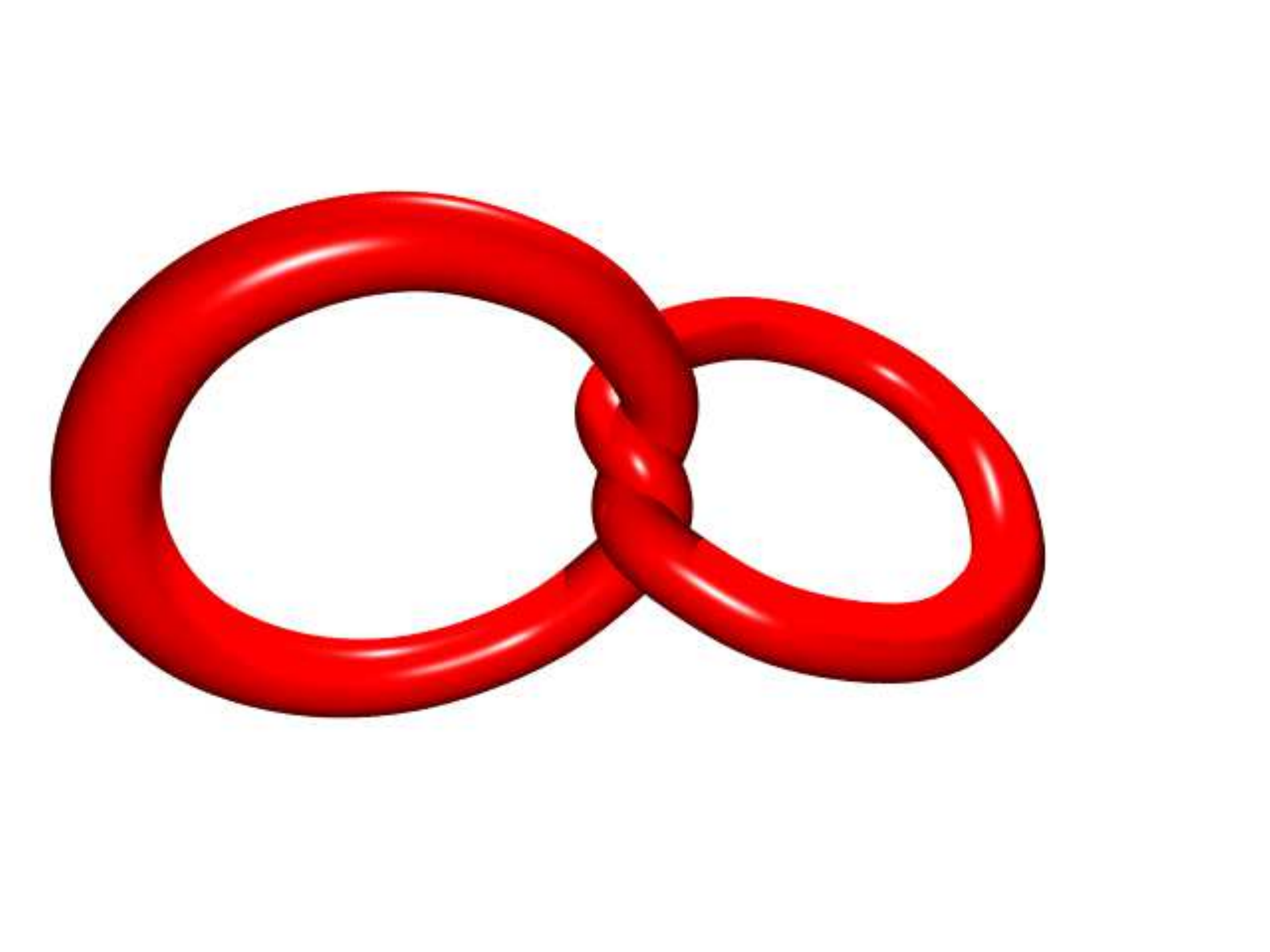} & \includegraphics[width=0.33\textwidth,keepaspectratio]{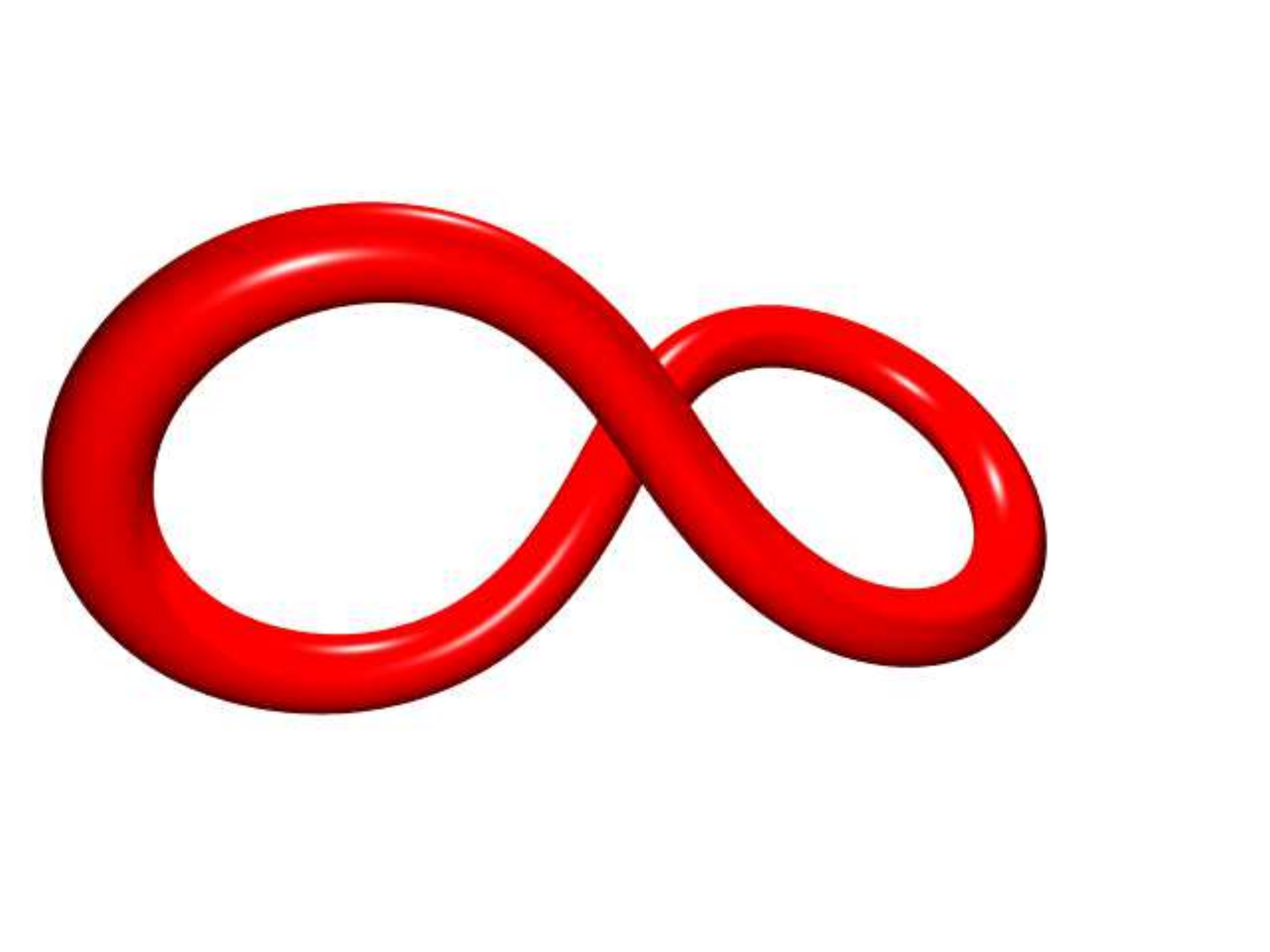} & \includegraphics[width=0.33\textwidth,keepaspectratio]{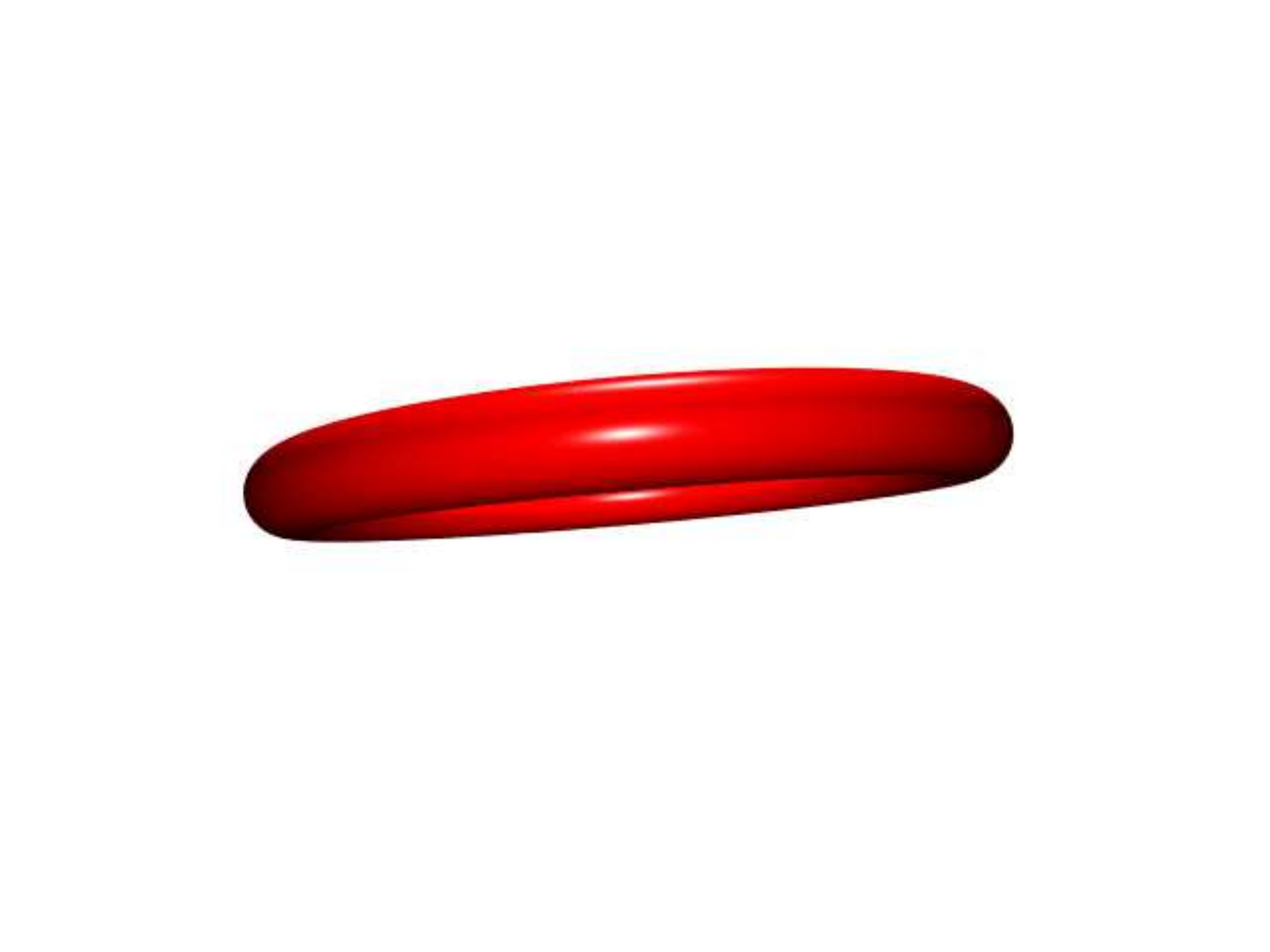}
\end{tabular}
\end{scriptsize}
\end{center}
\caption{A trefoil -- $p=3.0$ (continued) without redistribution ($\tau_{\text{max}}=0.1$, $\eps=0.05$)}
\end{figure}

Now we start with the very same configuration as in the flow before but with $p=3.5$ and again without redistribution
\begin{figure}[H]
\begin{center}
\begin{scriptsize}
\begin{tabular}{ccc}
0/100000 & 5000/100000 & 100000/100000 \\
$\Le(\gamma)\approx 31.04354$ & $\Le(\gamma)\approx 33.81871$ & $\Le(\gamma)\approx 34.37755$ \\
$\E_p(\gamma)\approx 18.44063$ & $\E_p(\gamma)\approx 17.28253$ & $\E_p(\gamma)\approx 17.25392$ \\
$\tau=0.0$ & $\tau=12.4572$ & $\tau=333.04893$ \\
\includegraphics[width=0.33\textwidth,keepaspectratio]{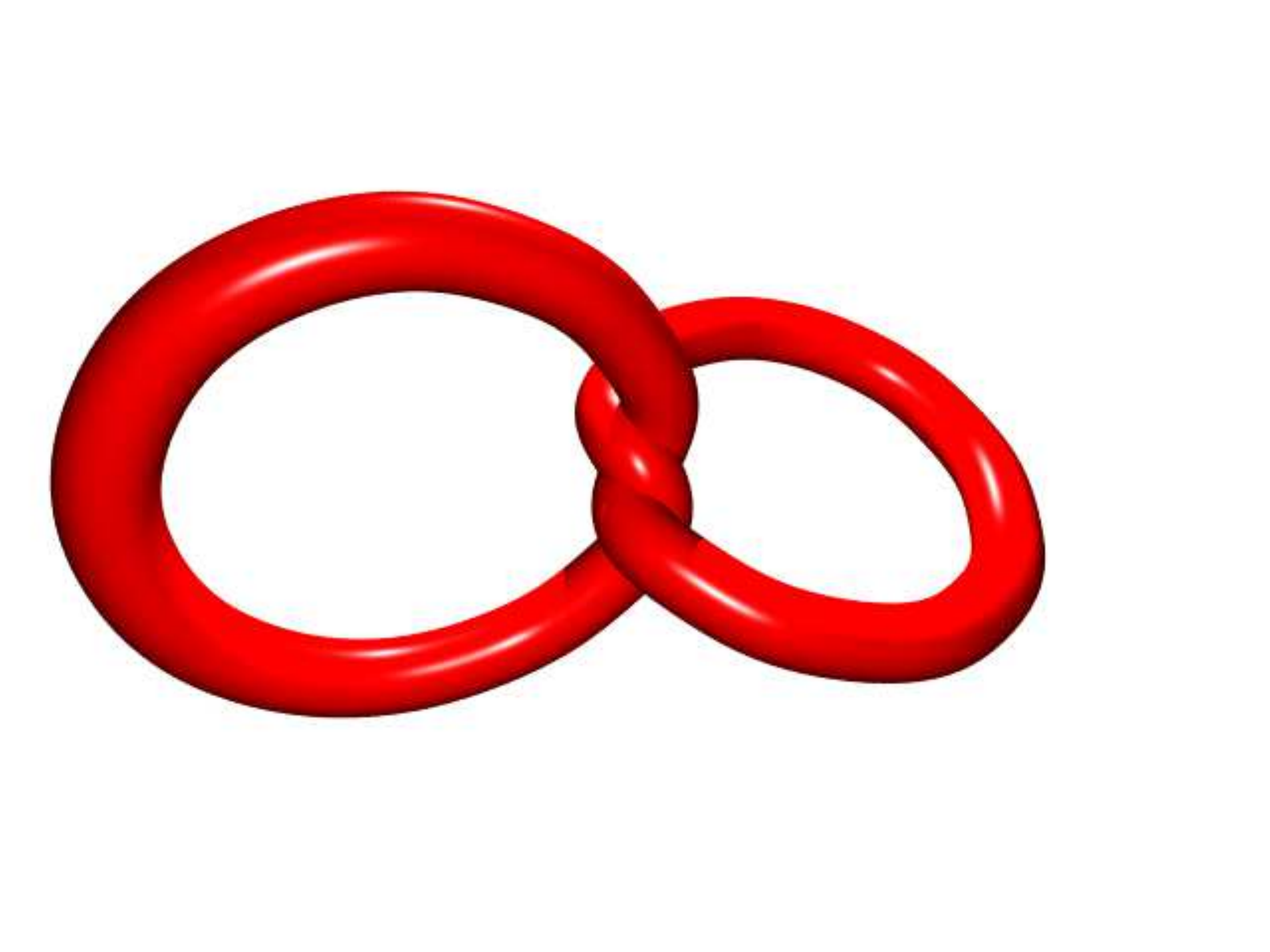} & \includegraphics[width=0.33\textwidth,keepaspectratio]{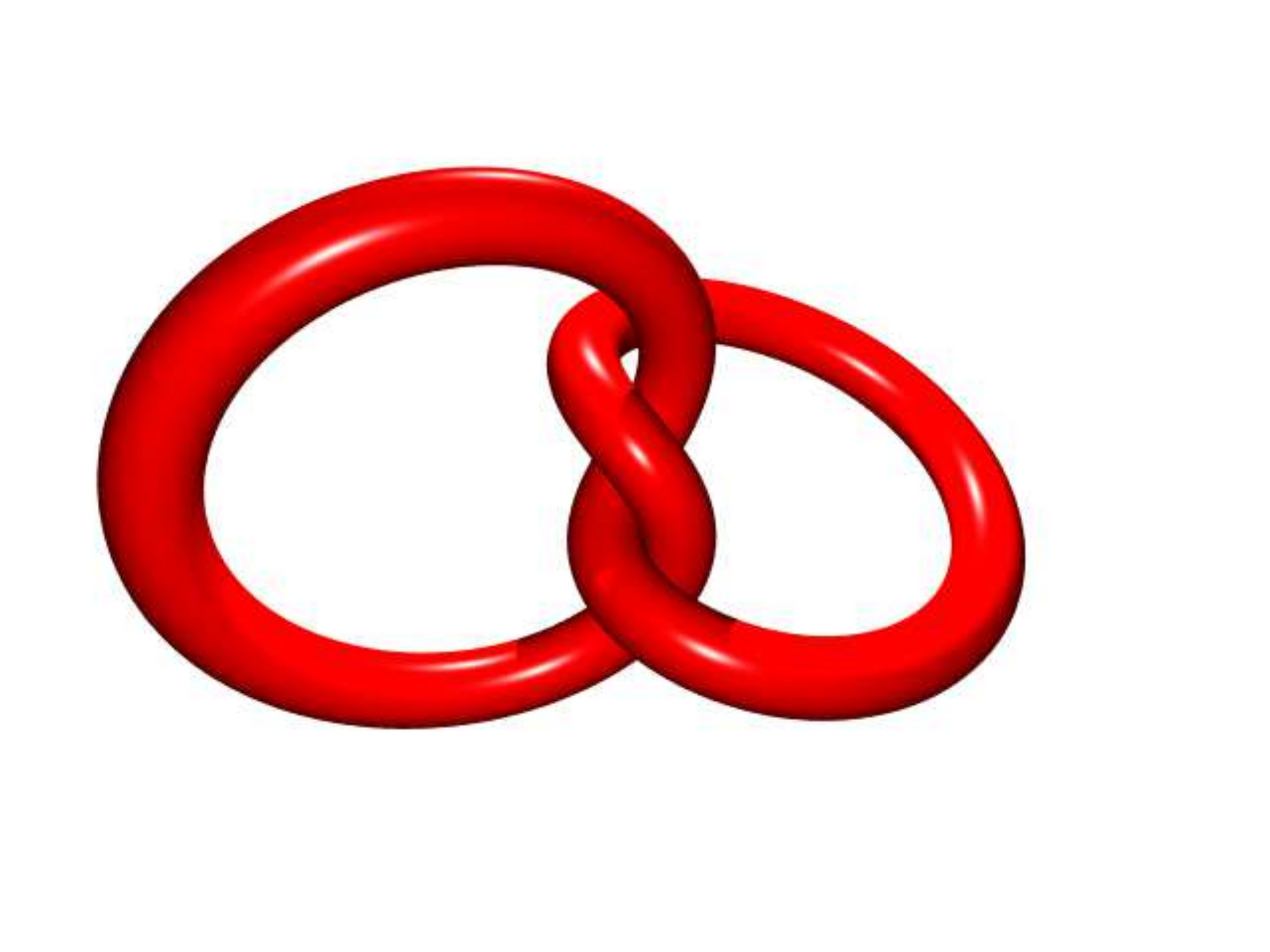} & \includegraphics[width=0.33\textwidth,keepaspectratio]{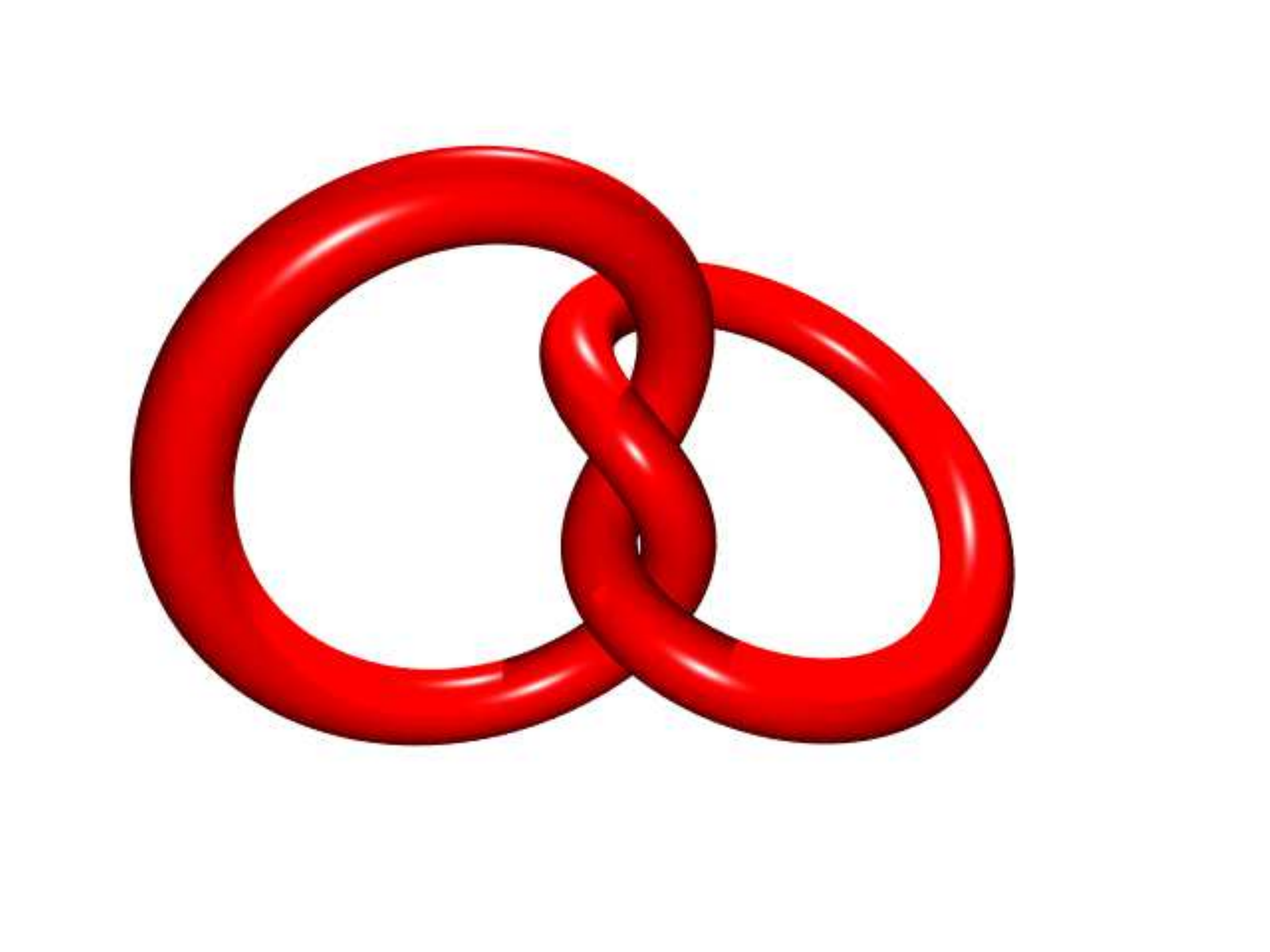}
\end{tabular}
\end{scriptsize}
\end{center}
\caption{A trefoil -- $p=3.5$ (continued) without redistribution ($\tau_{\text{max}}=0.1$, $\eps=0.05$)}
\end{figure}

The knot class is not abandoned this time. However, starting with the trefoil from the beginning of this section and using redistribution we end up in another configuration, which is symmetric and has a lower energy value for this $p$.
\begin{figure}[H]
\begin{center}
\begin{scriptsize}
\begin{tabular}{ccc}
0/500000 & 10000/500000 & 20000/500000 \\
$\Le(\gamma)\approx 35.42829$ & $\Le(\gamma)\approx 33.27166$ & $\Le(\gamma)\approx 33.27144$ \\
$\E_p(\gamma)\approx 17.84245$ & $\E_p(\gamma)\approx 17.0783$ & $\E_p(\gamma)\approx 17.0783$ \\
$\tau=0.0$ & $\tau=53.84697$ & $\tau=107.13935$ \\
\includegraphics[width=0.33\textwidth,keepaspectratio]{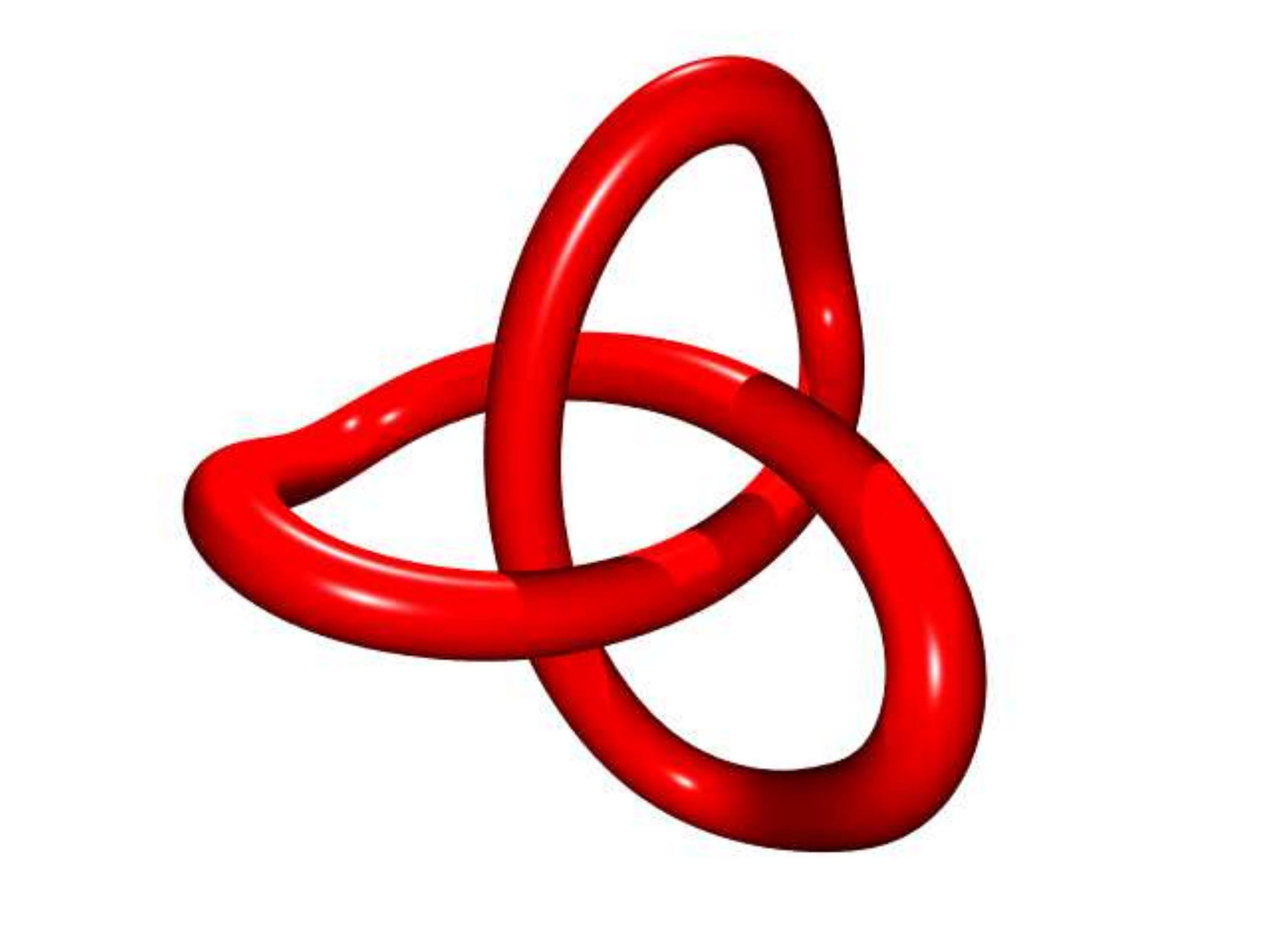} & \includegraphics[width=0.33\textwidth,keepaspectratio]{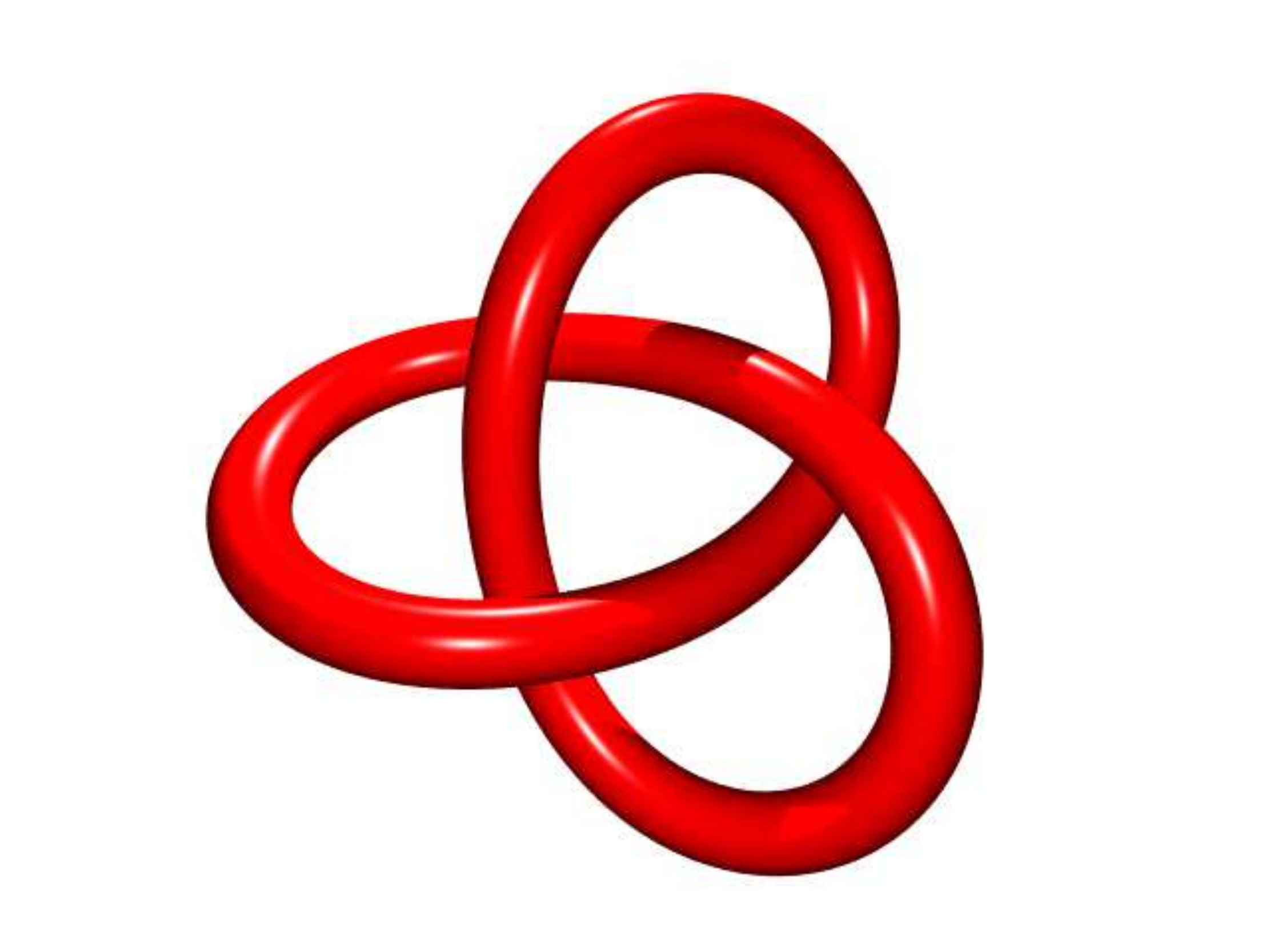} & \includegraphics[width=0.33\textwidth,keepaspectratio]{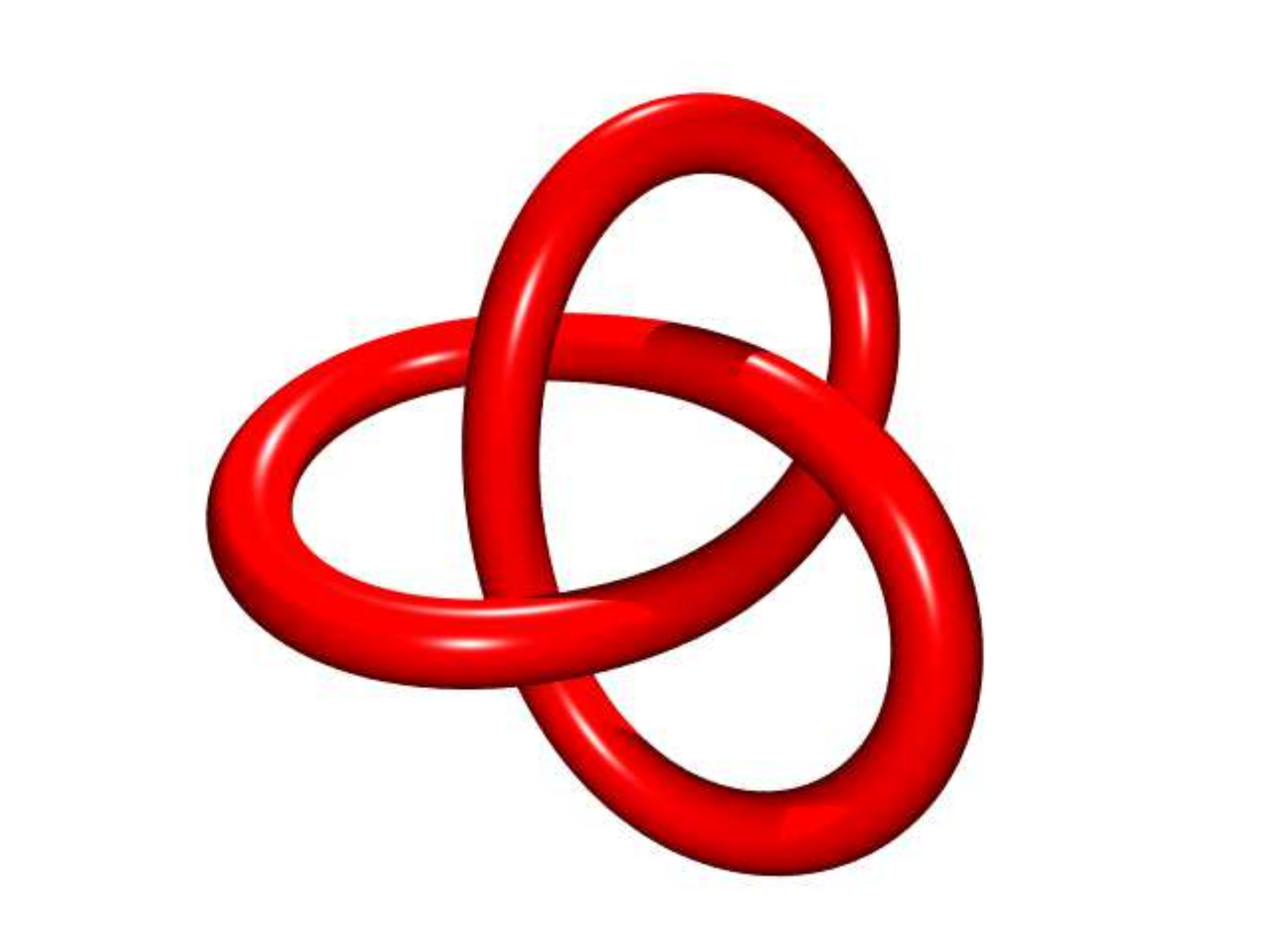}
\end{tabular}
\end{scriptsize}
\end{center}
\caption{A trefoil -- $p=3.5$ ($\tau_{\text{max}}=0.1$, $\eps=0.05$)}
\end{figure}

and the energy stays constant for the rest of the flow, i.e. $480000$ steps.

The examples show, that the flow for integral \name{Menger} curvature converges to configurations with small self-distance for smaller $p$ and large self-distance for larger $p$. In theory the energy value of a curve with self-intersection is infinite for each $p>3$. Here, and in most of the following examples $p=3.5$ seems to be large enough to pretend self-penetration.

The following example is a flow with $p=50$ and with redistribution
\begin{figure}[H]
\begin{center}
\begin{scriptsize}
\begin{tabular}{ccc}
0/2000 & 10/2000 & 60/2000 \\
$\Le(\gamma)\approx 35.42829$ & $\Le(\gamma)\approx 35.53827$ & $\Le(\gamma)\approx 35.36989$ \\
$\E_p(\gamma)\approx 45.11733$ & $\E_p(\gamma)\approx 37.77995$ & $\E_p(\gamma)\approx 31.68628$ \\
$\tau=0.0$ & $\tau=0.02376$ & $\tau=0.20622$ \\
\includegraphics[width=0.33\textwidth,keepaspectratio]{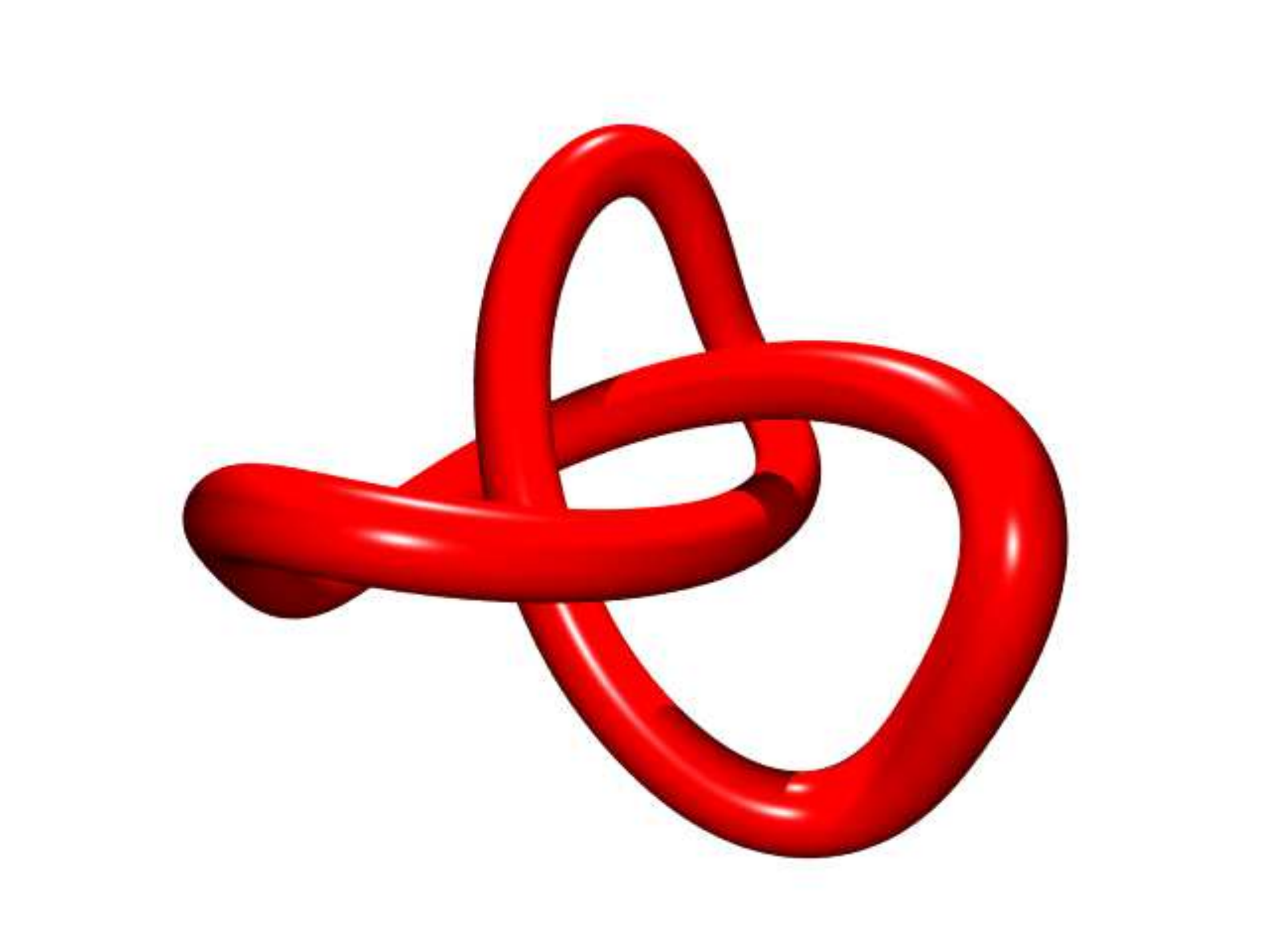} & \includegraphics[width=0.33\textwidth,keepaspectratio]{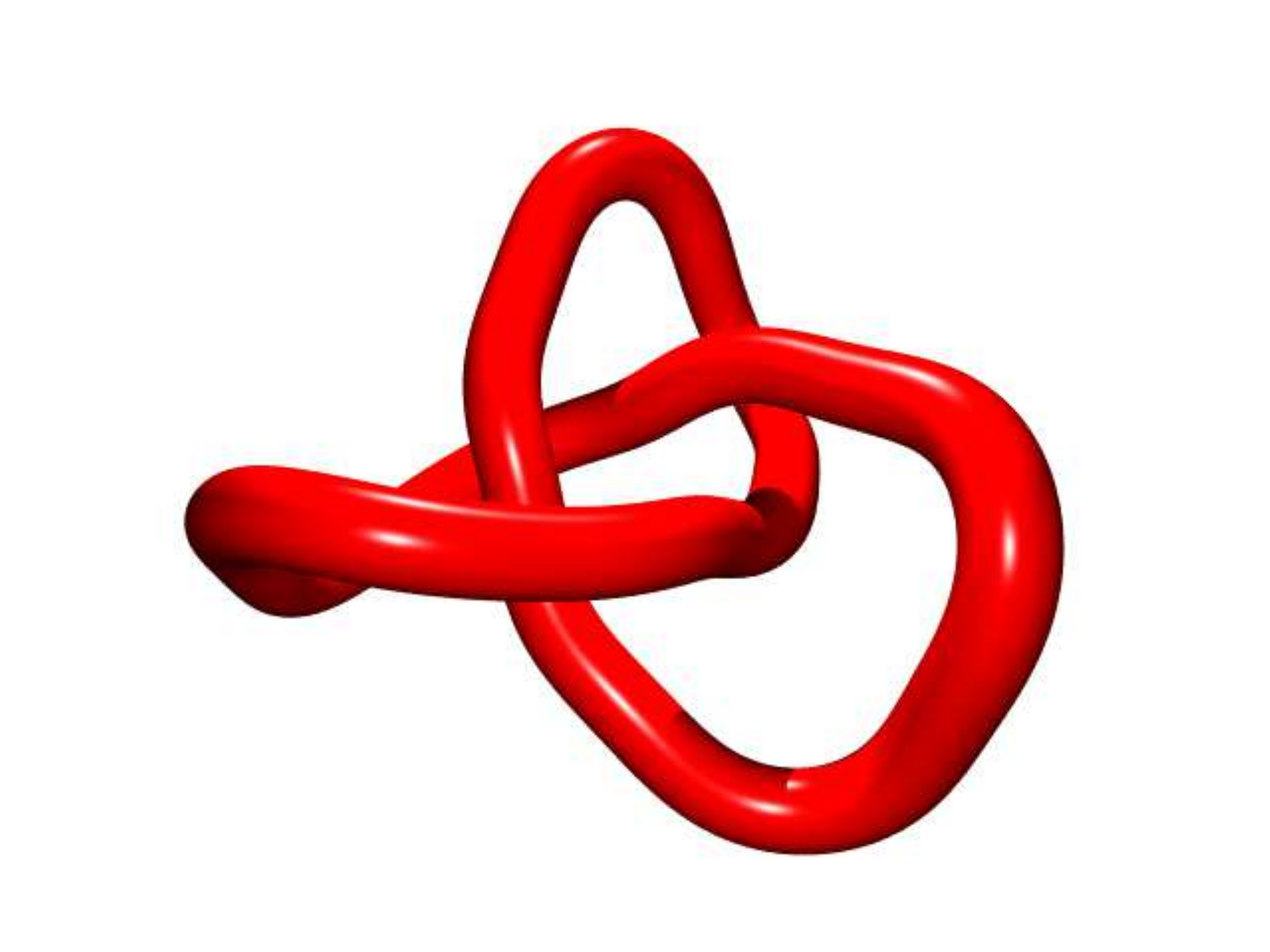} & \includegraphics[width=0.33\textwidth,keepaspectratio]{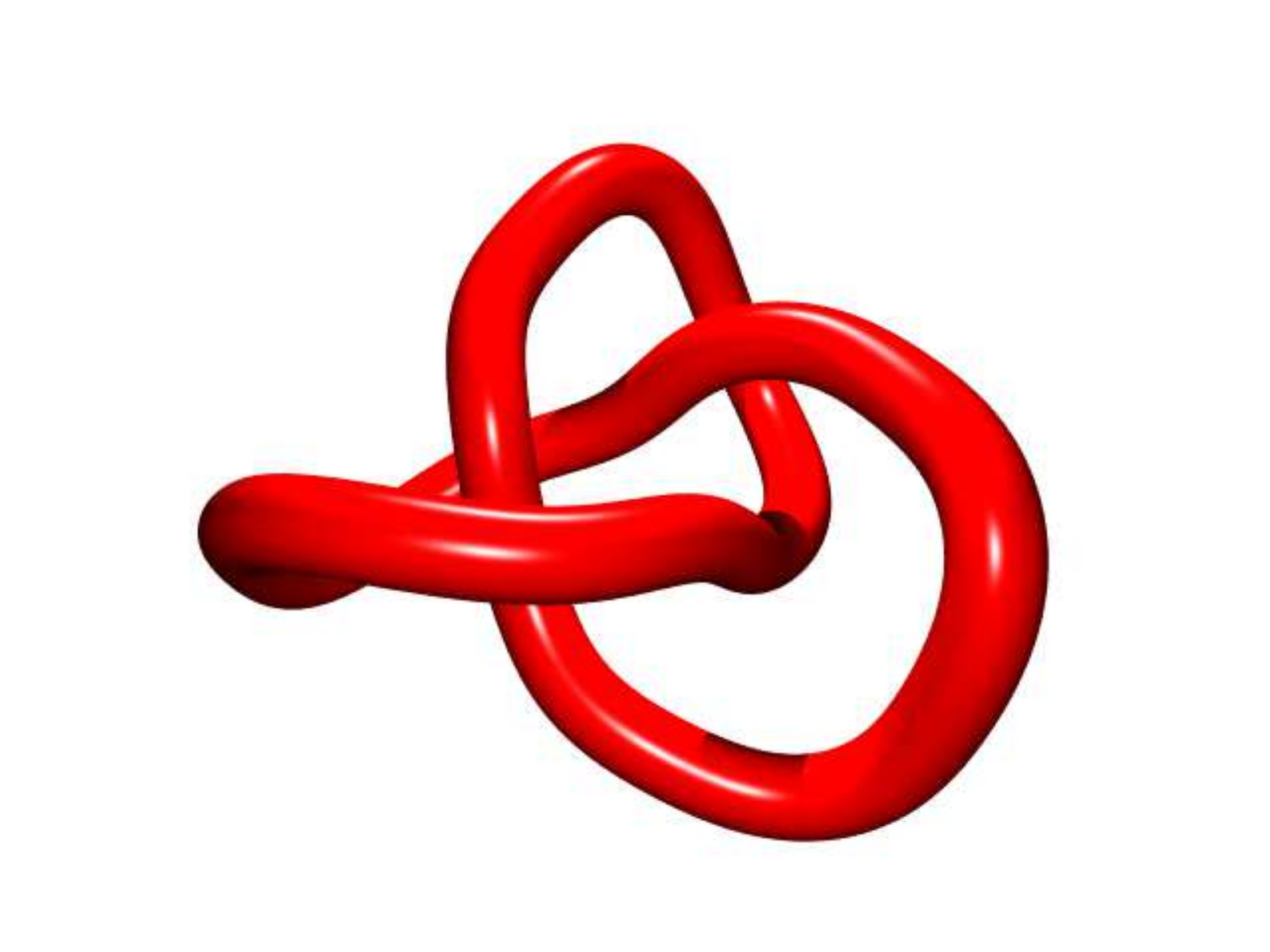} \\
200/2000 & 1200/2000 & 2000/2000 \\
$\Le(\gamma)\approx 35.03015$ & $\Le(\gamma)\approx 34.57029$ & $\Le(\gamma)\approx 34.55507$ \\
$\E_p(\gamma)\approx 30.07876$ & $\E_p(\gamma)\approx 29.5923$ & $\E_p(\gamma)\approx 29.57521$ \\
$\tau=0.78832$ & $\tau=4.91227$ & $\tau=8.19921$ \\
\includegraphics[width=0.33\textwidth,keepaspectratio]{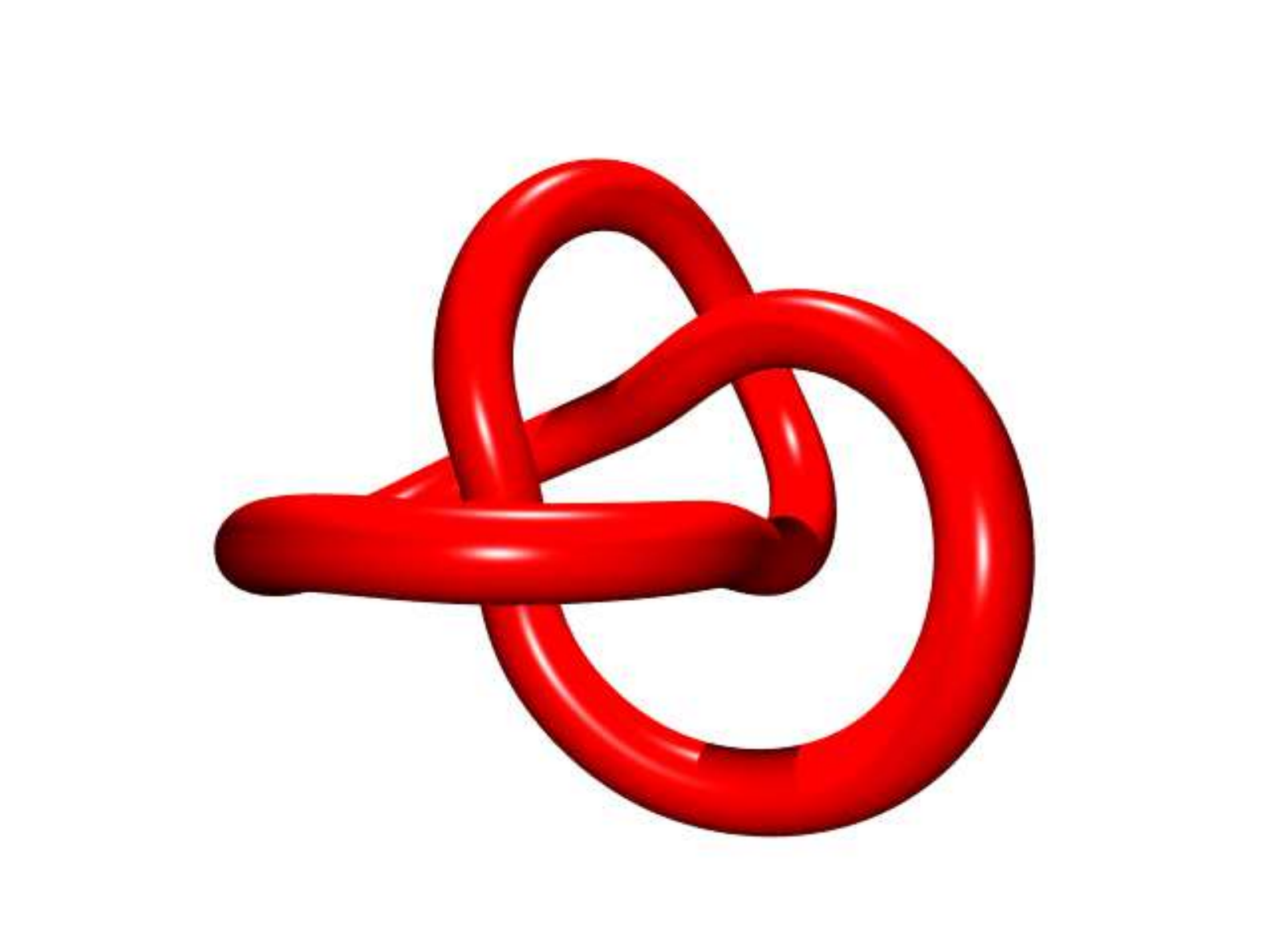} & \includegraphics[width=0.33\textwidth,keepaspectratio]{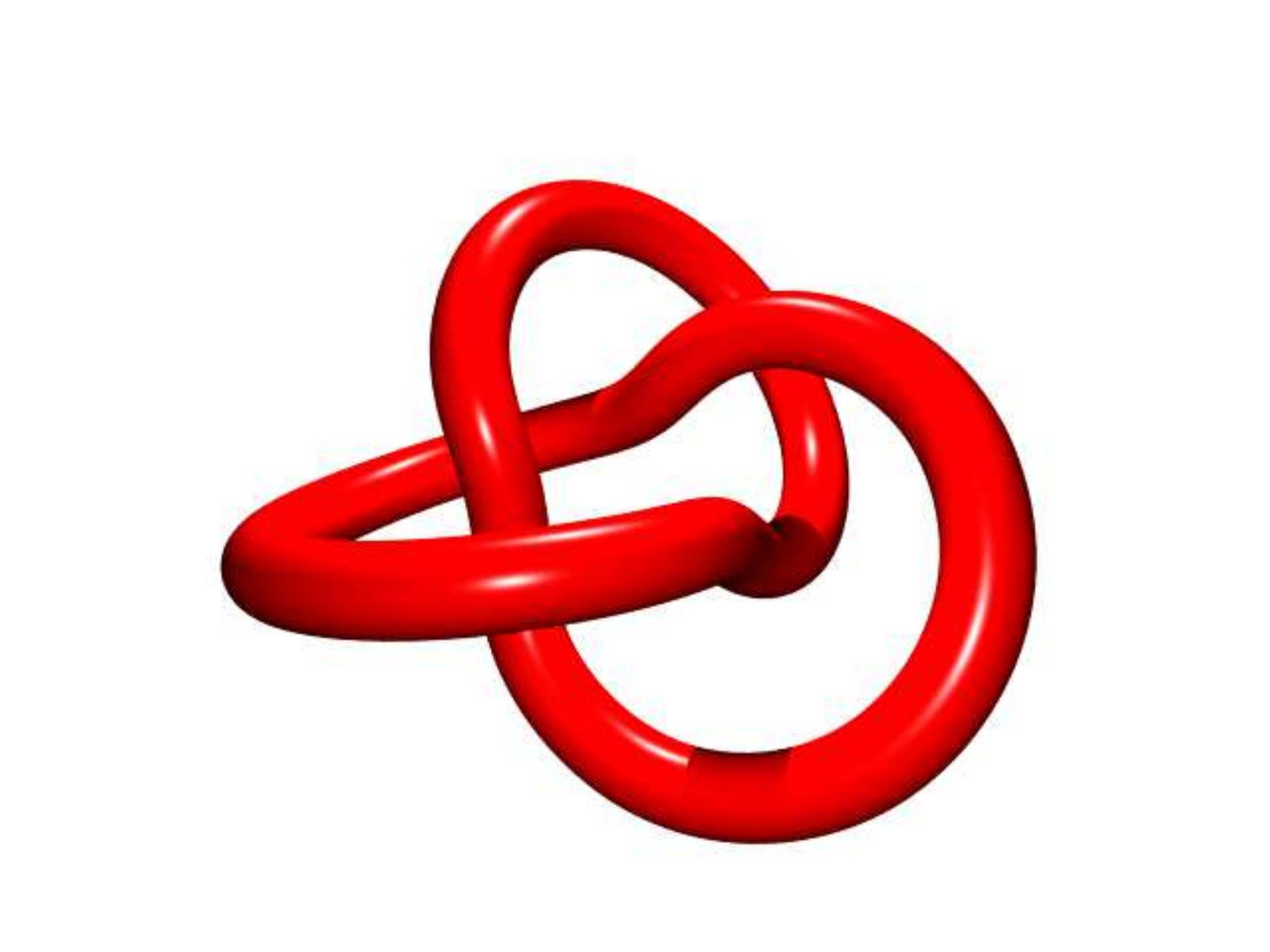} & \includegraphics[width=0.33\textwidth,keepaspectratio]{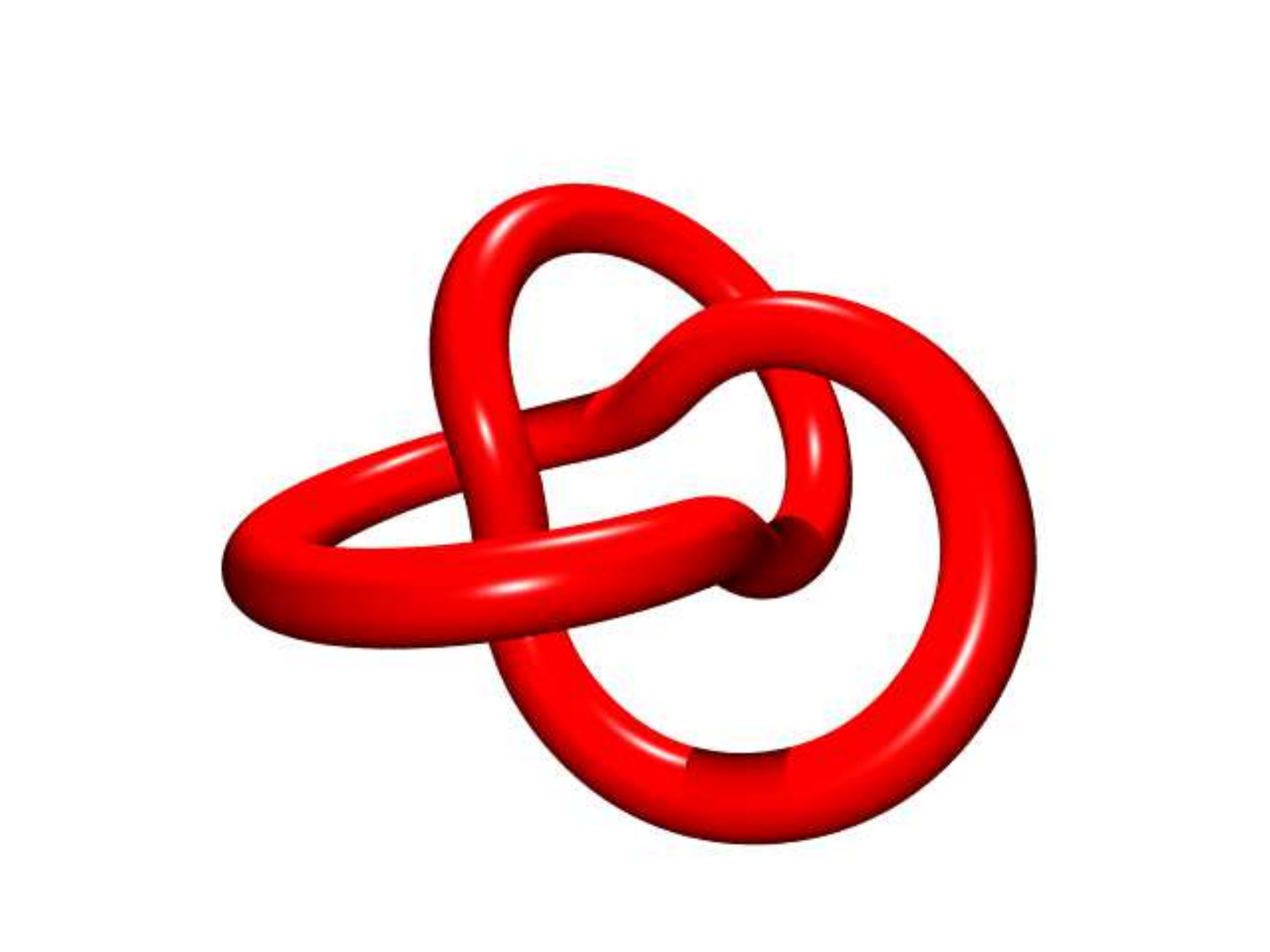}
\end{tabular}
\end{scriptsize}
\end{center}
\caption{A trefoil -- $p=50.0$}
\label{exmptrefoilp50}
\end{figure}

Afterwards there are only minor changes in shape, however, the $\E_p$ energy decreases a little bit until it reaches its minimum.
\begin{figure}[H]
\begin{center}
\begin{scriptsize}
\begin{tabular}{ccc}
0/8000 & 4000/8000 & 8000/8000 \\
$\Le(\gamma)\approx 34.55502$ & $\Le(\gamma)\approx 34.57705$ & $\Le(\gamma)\approx 34.59993$ \\
$\E_p(\gamma)\approx 29.57525$ & $\E_p(\gamma)\approx 29.57435$ & $\E_p(\gamma)\approx 29.57435$ \\
$\tau=0.0$ & $\tau=16.46821$ & $\tau=32.97004$ \\
\includegraphics[width=0.33\textwidth,keepaspectratio]{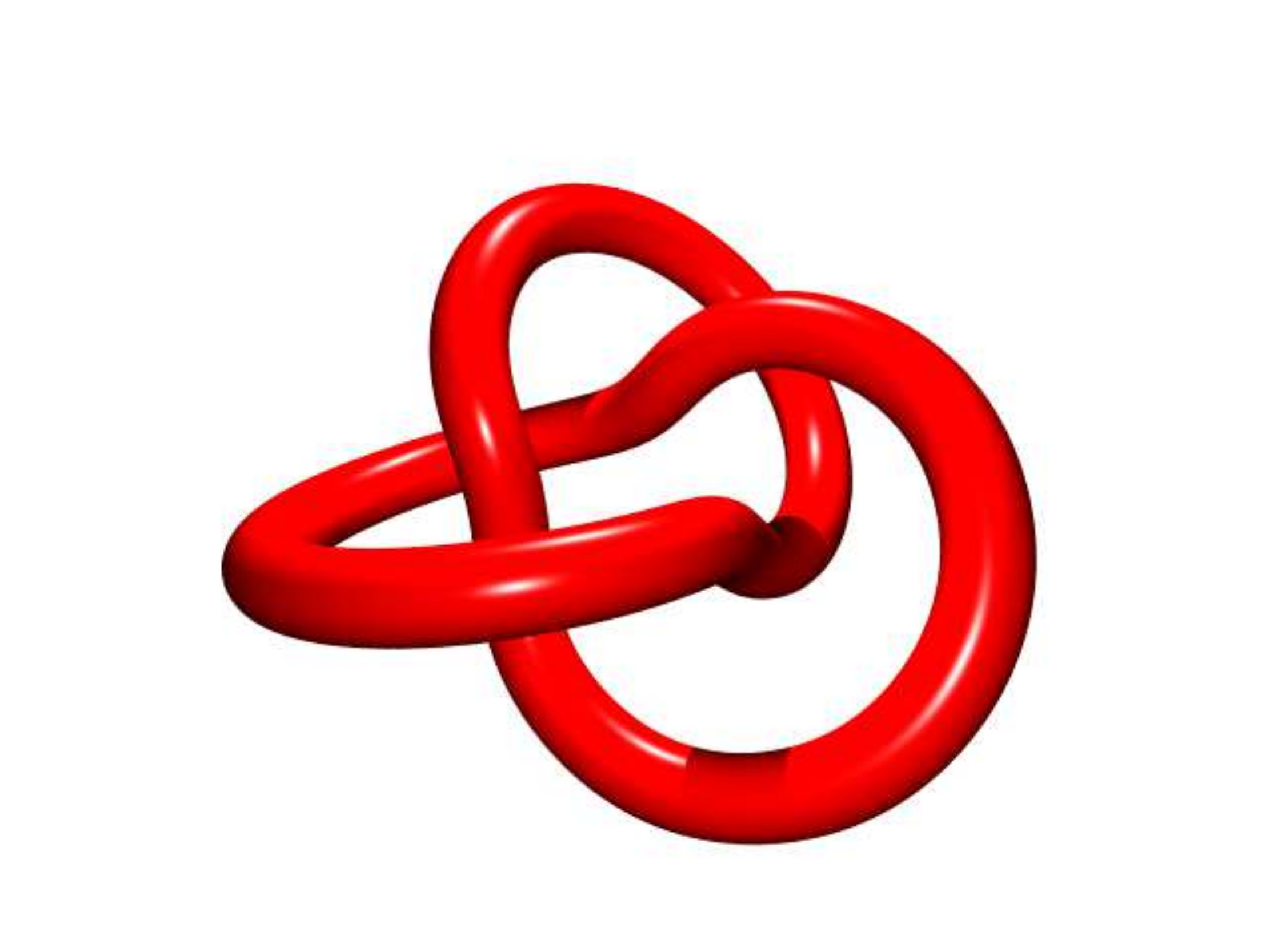} & \includegraphics[width=0.33\textwidth,keepaspectratio]{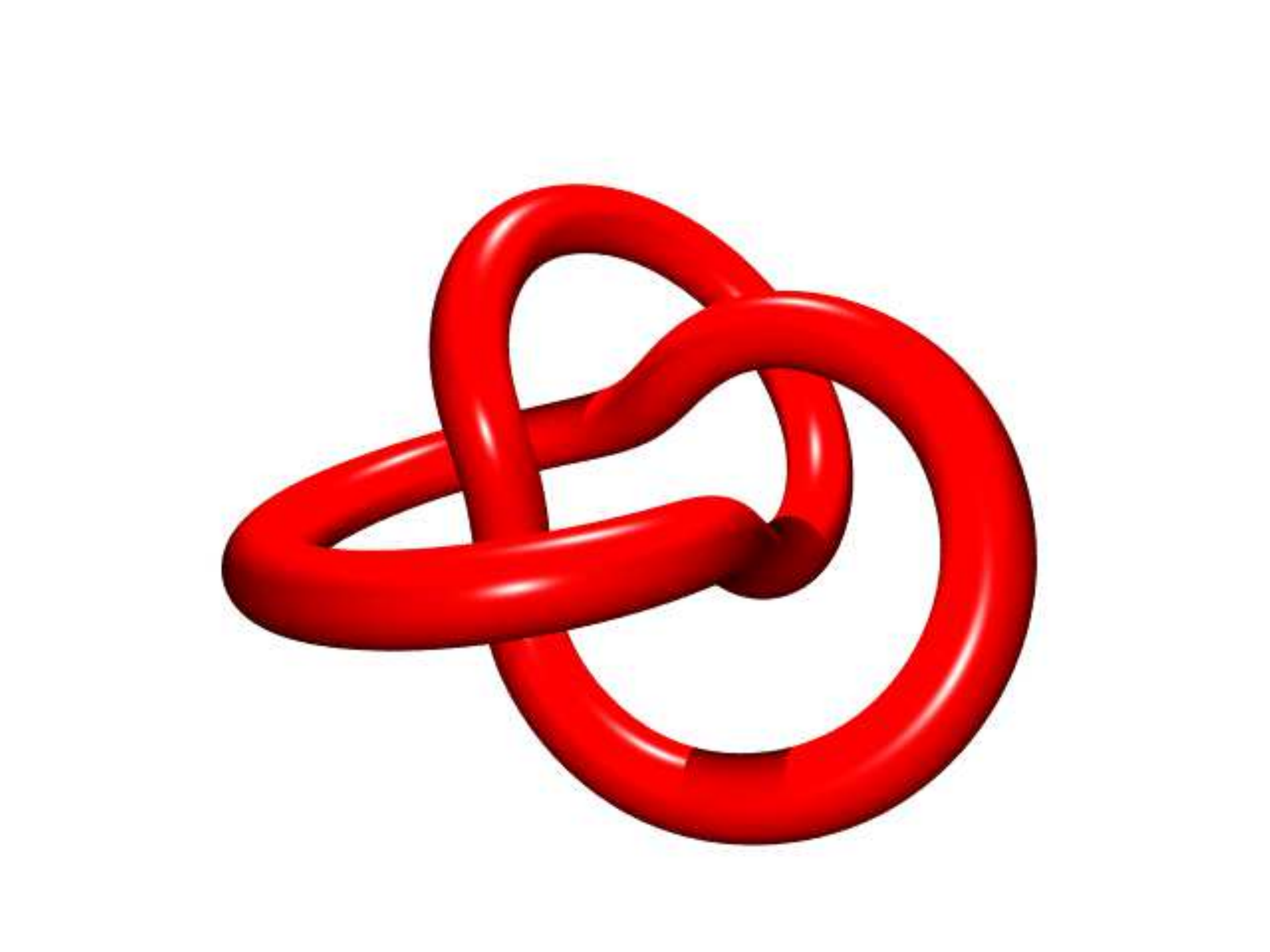} & \includegraphics[width=0.33\textwidth,keepaspectratio]{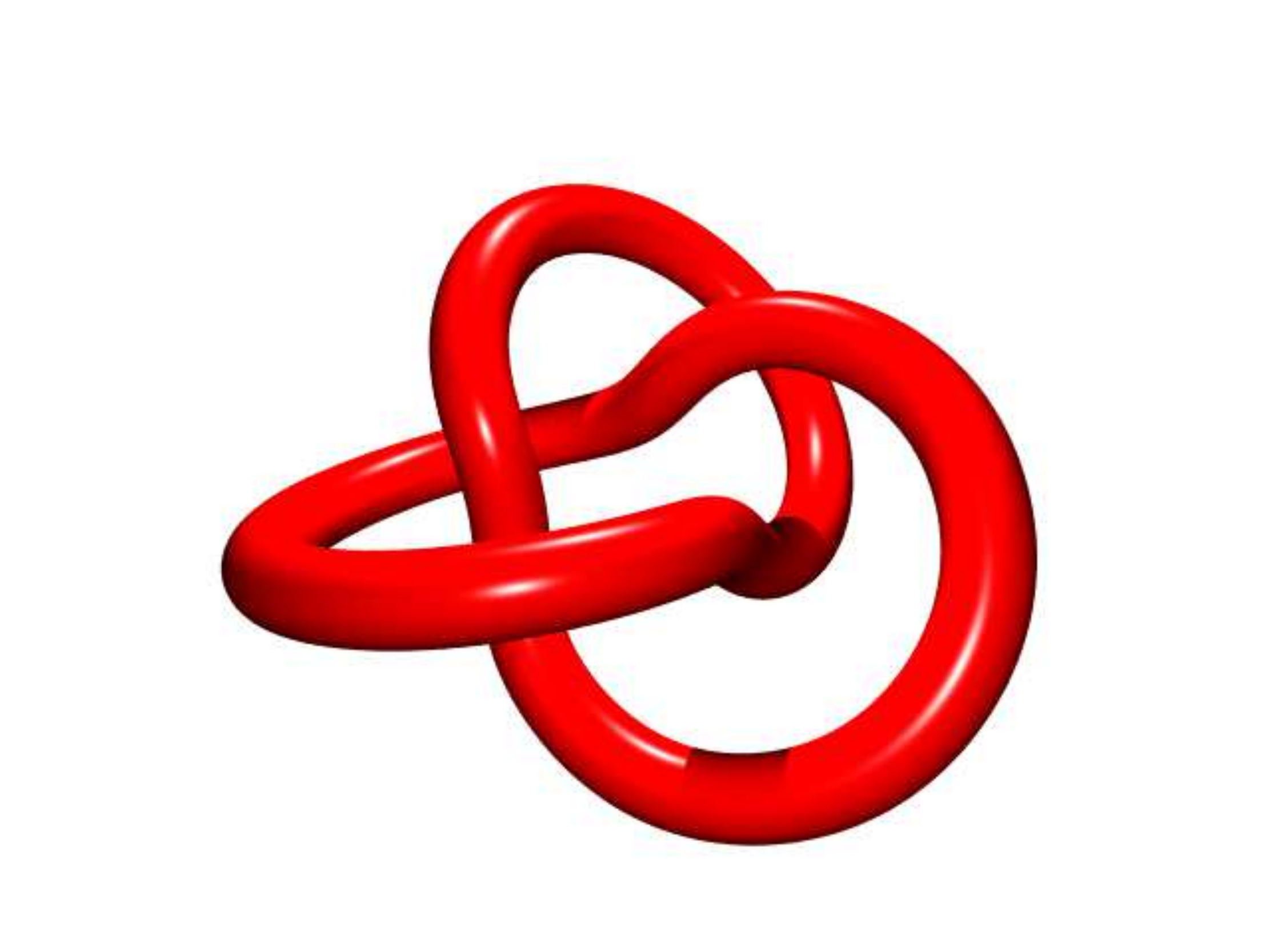}
\end{tabular}
\end{scriptsize}
\end{center}
\caption{A trefoil -- $p=50.0$ (continued)}
\end{figure}

Even running up to $1.000.000$ steps the energy level stays the same. Probably due to the redistribution there is a small increase of the length.

Observe that this is only the minimum when using redistribution. Otherwise, after $200.000$ steps the energy goes down to a value of $29.57199$. Even if we continue and allow larger time steps, by setting $\eps=0.1$, we reach a level of $29.57191$ after $1.000.000$ steps and after $800.000$ steps more we reach $29.57190$. However, the energy are not completely stable at this point. Nevertheless the changes in shape are far beyond visibility. It is interesting, that if we start with this minimal configuration using the flow with redistribution, except for the initial configuration, the energy increases until the value $29.57435$ is reached again.

\newpage
Here comes an example of a deformed trefoil knot. For $p=3$ using the flow with redistribution the knot class is abandoned and the final configuration is the circle once more.
\begin{figure}[H]
\begin{center}
\begin{scriptsize}
\begin{tabular}{cc}
0/30000 & 300/30000 \\
$\Le(\gamma)\approx 2.52383$ & $\Le(\gamma)\approx 1.88497$ \\
$\E_p(\gamma)\approx 28.04599$ & $\E_p(\gamma)\approx 20.96723$ \\
$\tau=0.0$ & $\tau=0.00028$ \\
\includegraphics[width=0.4\textwidth,keepaspectratio]{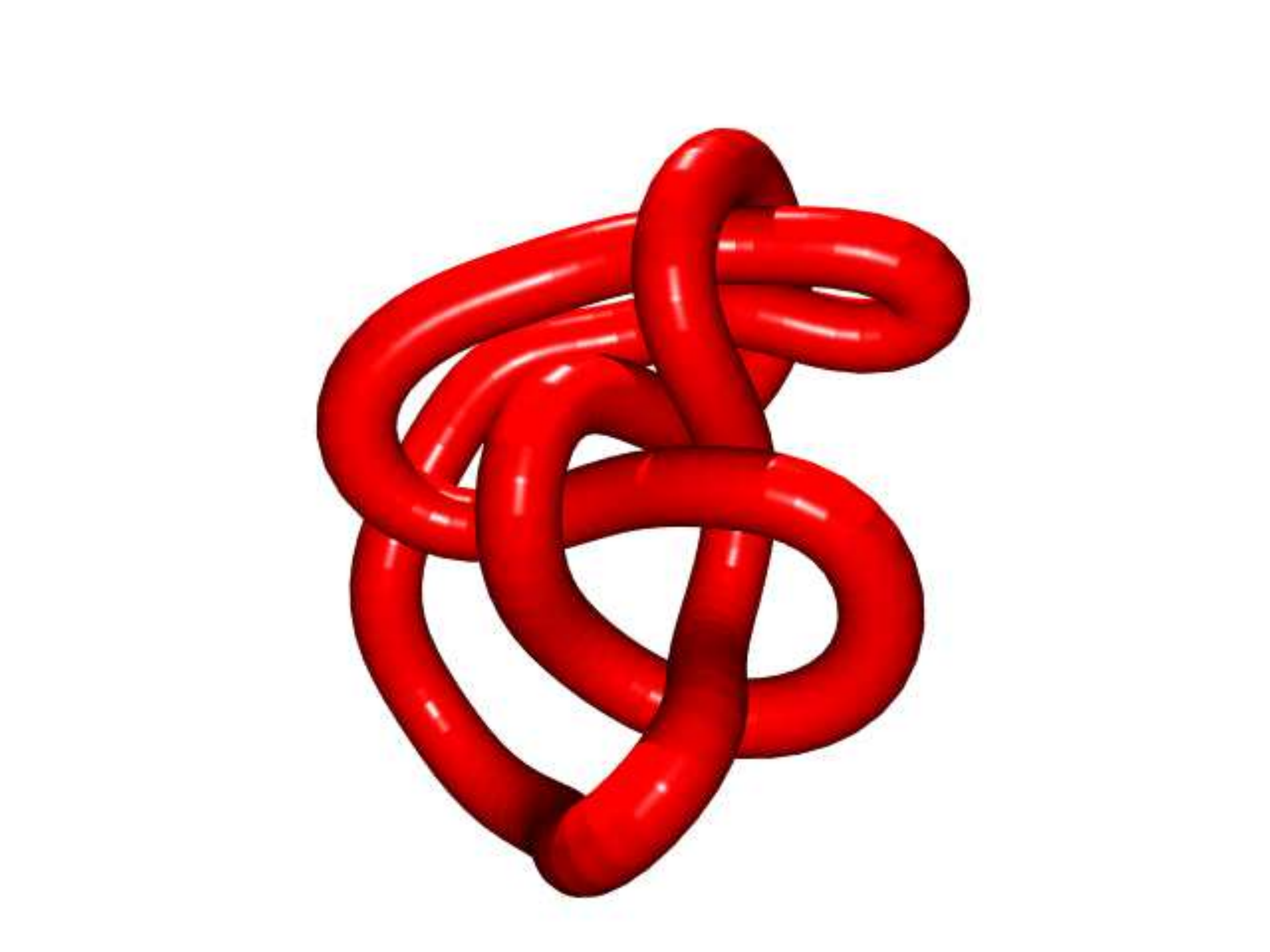} & \includegraphics[width=0.4\textwidth,keepaspectratio]{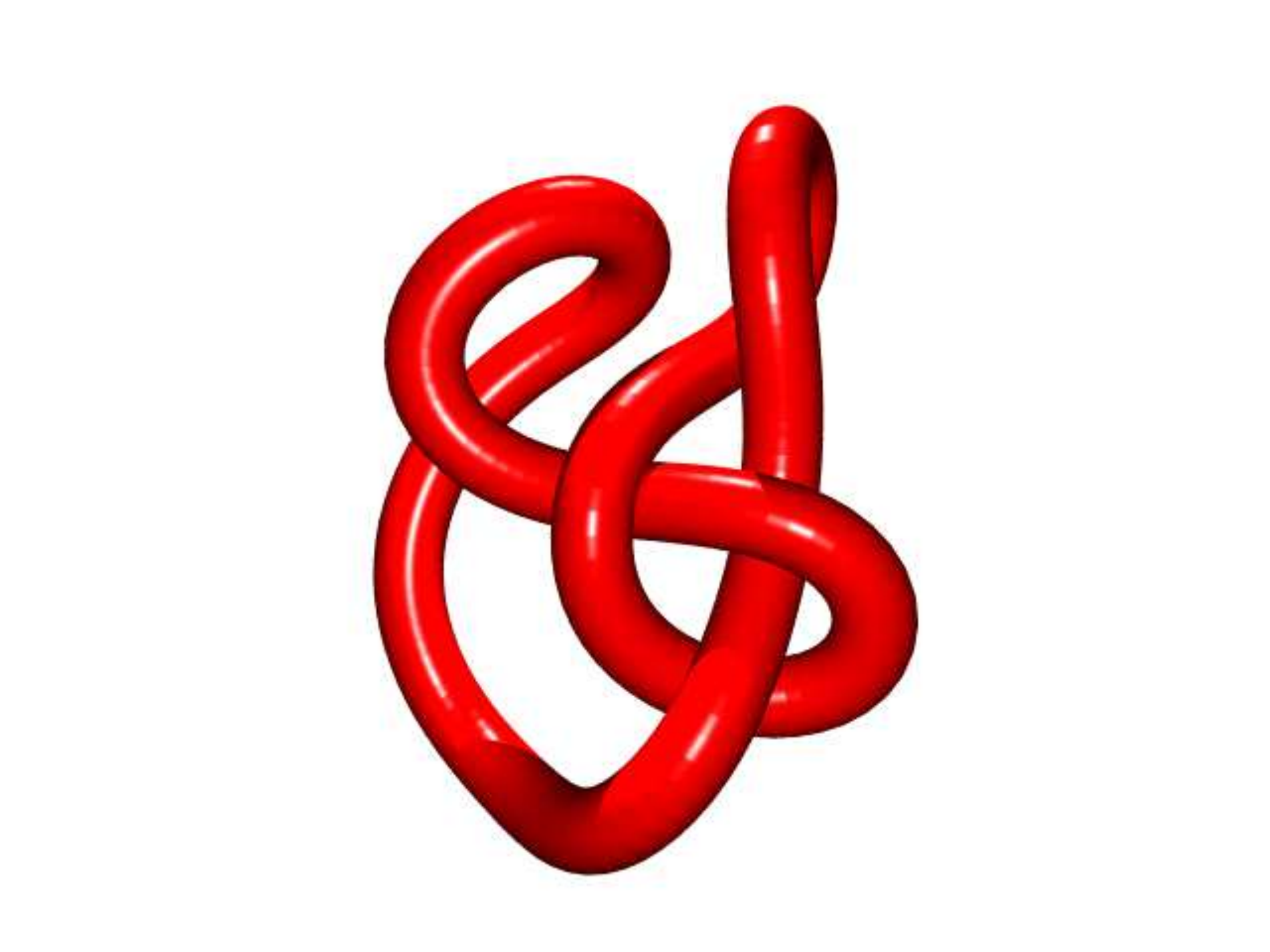} \\
1500/30000 & 4000/30000 \\
$\Le(\gamma)\approx 1.43415$ & $\Le(\gamma)\approx 1.28511$ \\
$\E_p(\gamma)\approx 17.73554$ & $\E_p(\gamma)\approx 16.61257$ \\
$\tau=0.00075$ & $\tau=0.001$ \\
\includegraphics[width=0.4\textwidth,keepaspectratio]{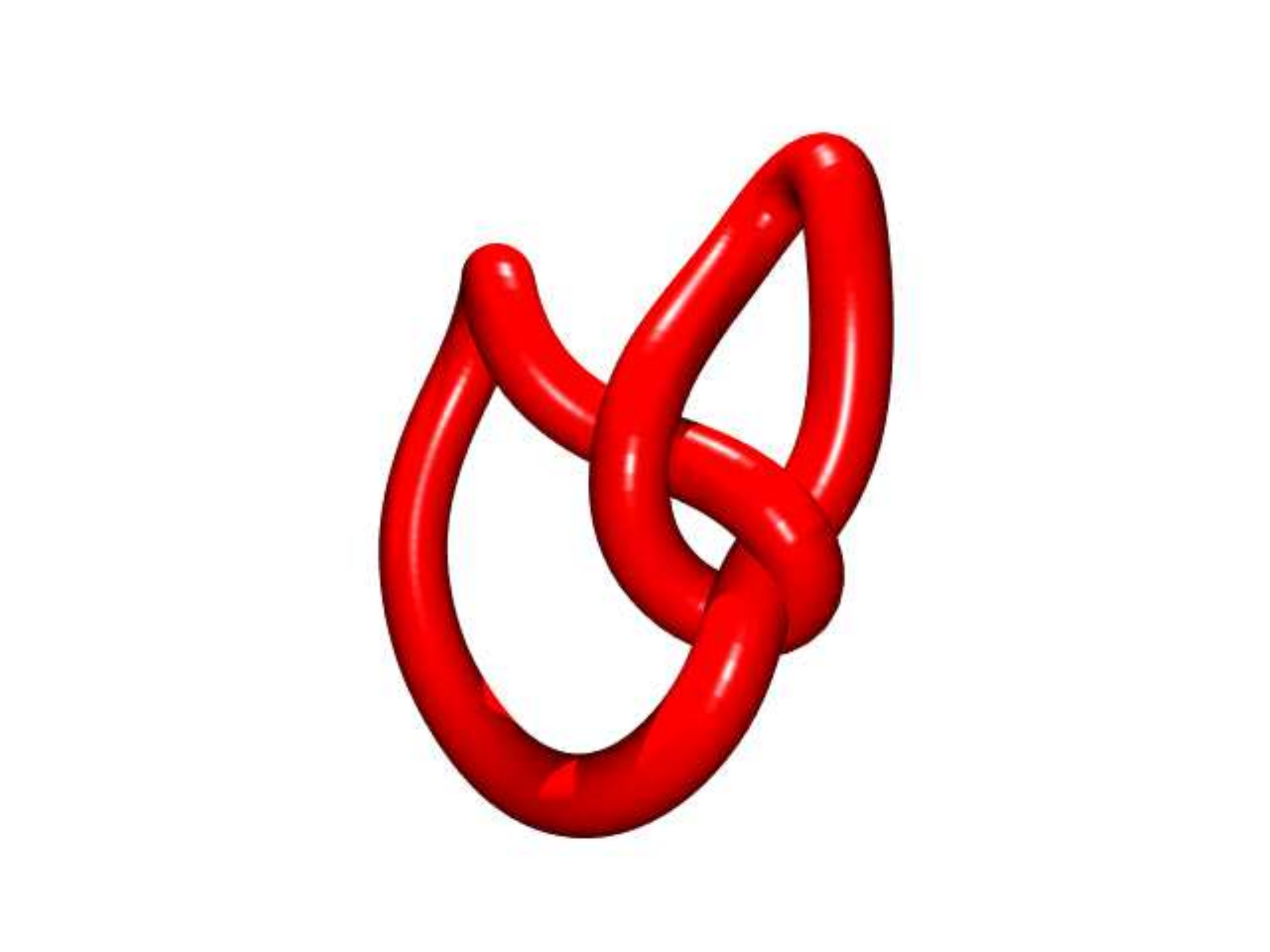} & \includegraphics[width=0.4\textwidth,keepaspectratio]{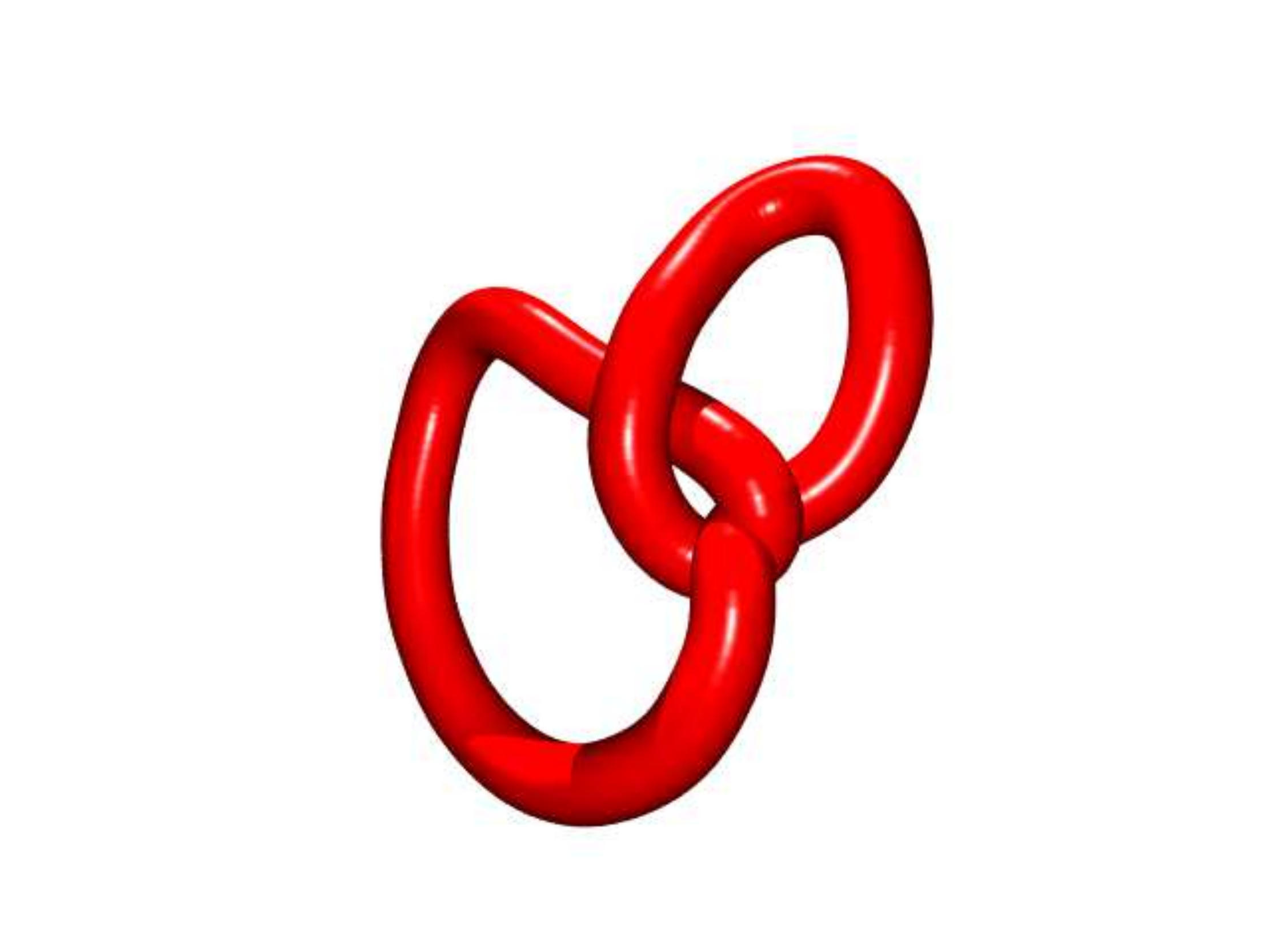} \\
5000/30000 & 30000/30000 \\
$\Le(\gamma)\approx 1.15102$ & $\Le(\gamma)\approx 0.85353$ \\
$\E_p(\gamma)\approx 12.34798$ & $\E_p(\gamma)\approx 6.28319$ \\
$\tau=0.00106$ & $\tau=0.00384$ \\
\includegraphics[width=0.4\textwidth,keepaspectratio]{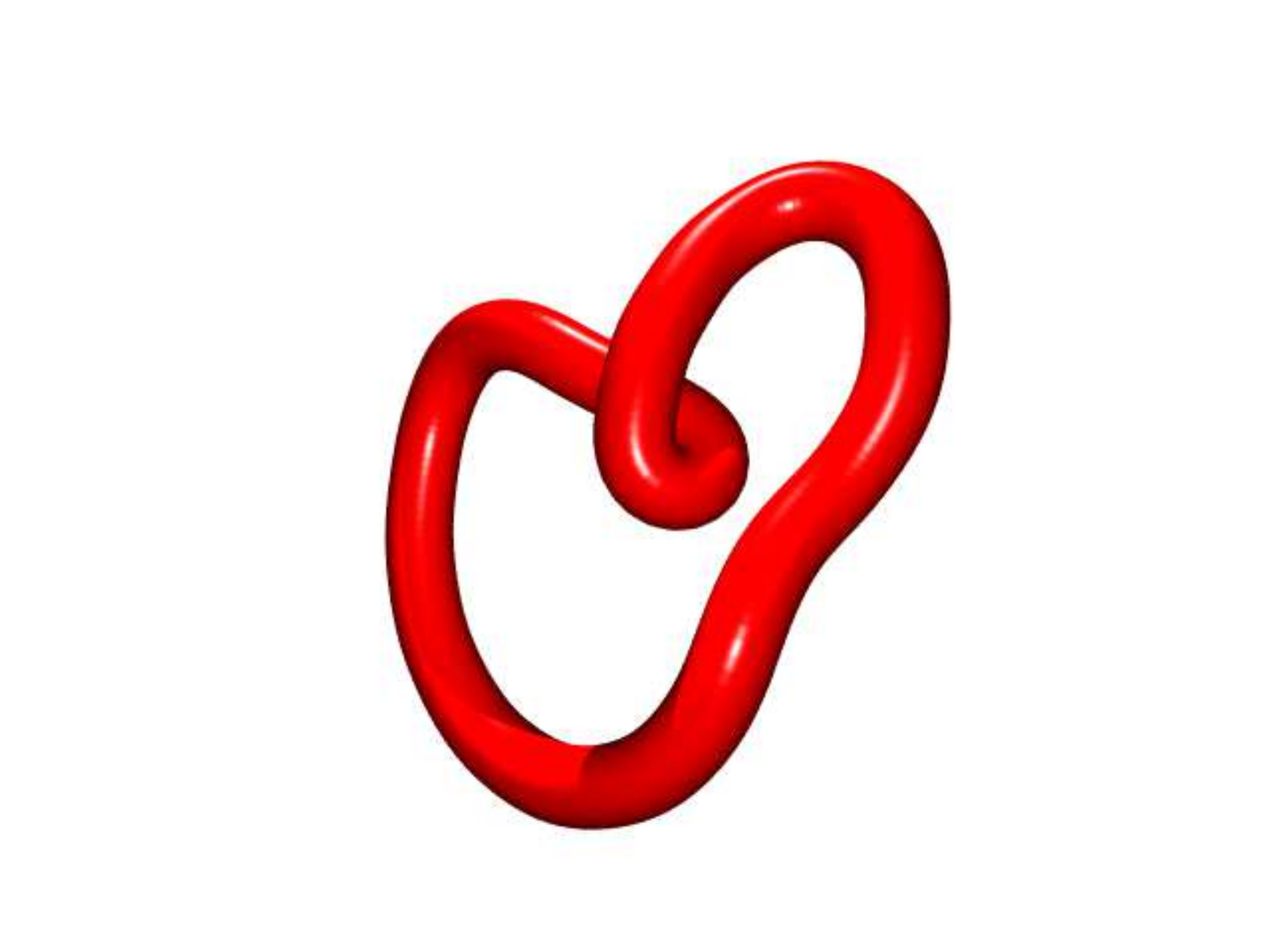} & \includegraphics[width=0.4\textwidth,keepaspectratio]{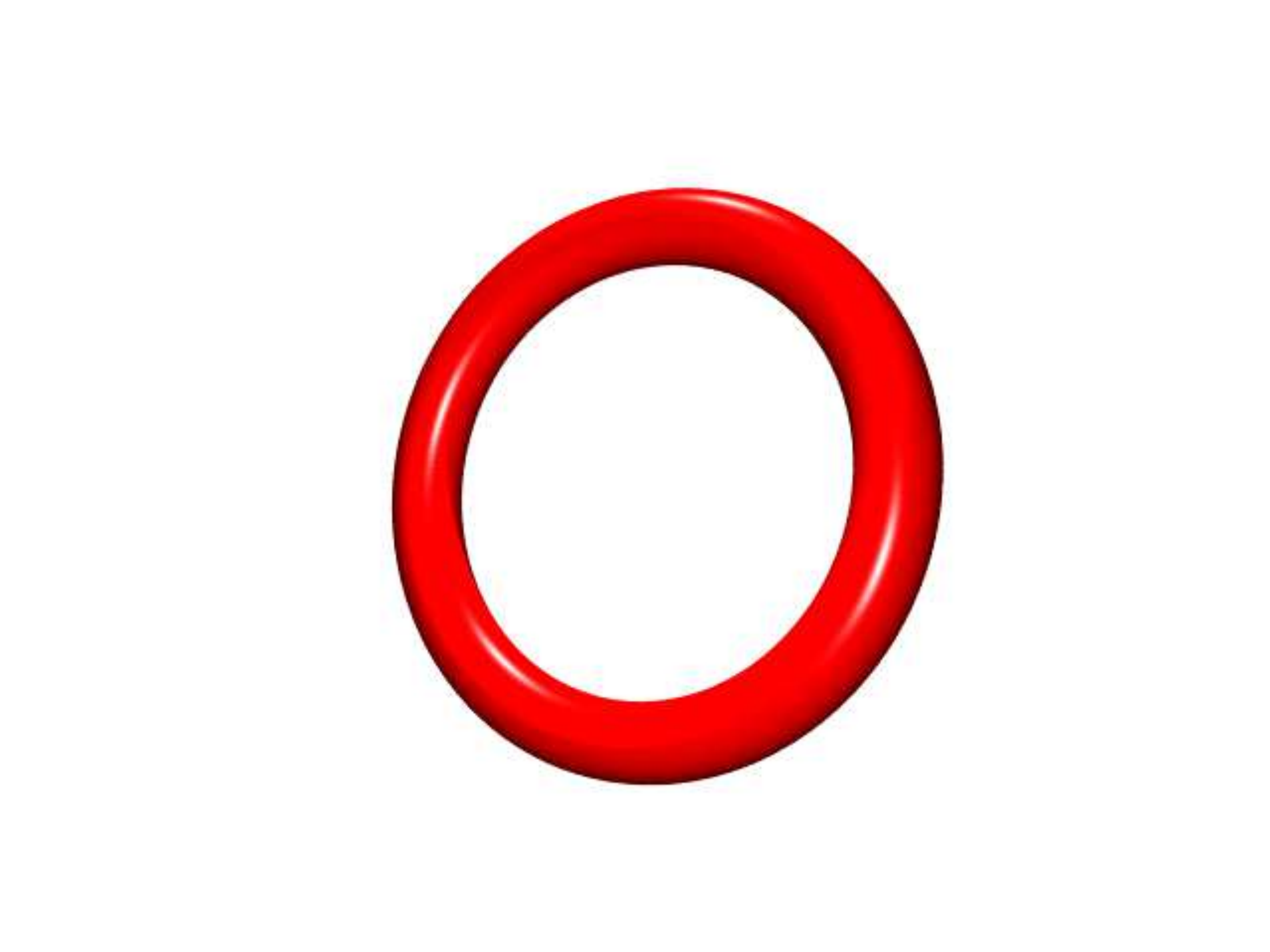}
\end{tabular}
\end{scriptsize}
\end{center}
\caption{A deformed trefoil -- $p=3.0$}
\end{figure}

\newpage
Again for $p=3.5$ the self-penetration is prevented and the knot untangles to the trefoil knot.
\begin{figure}[H]
\begin{center}
\begin{scriptsize}
\begin{tabular}{cc}
0/100000 & 100/100000 \\
$\Le(\gamma)\approx 2.52383$ & $\Le(\gamma)\approx 2.17273$ \\
$\E_p(\gamma)\approx 29.63985$ & $\E_p(\gamma)\approx 25.35132$ \\
$\tau=0.0$ & $\tau=0.00011$ \\
\includegraphics[width=0.4\textwidth,keepaspectratio]{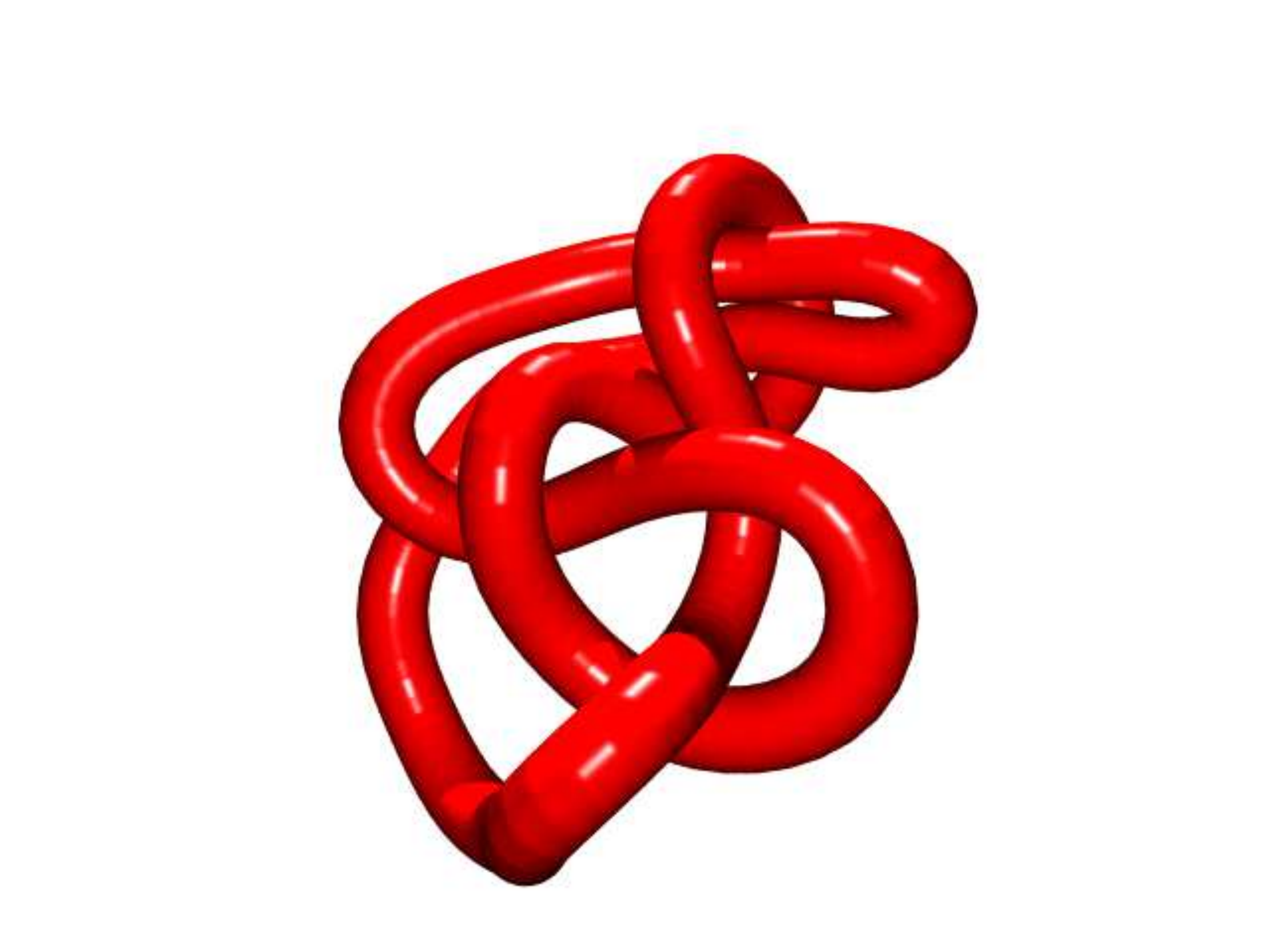} & \includegraphics[width=0.4\textwidth,keepaspectratio]{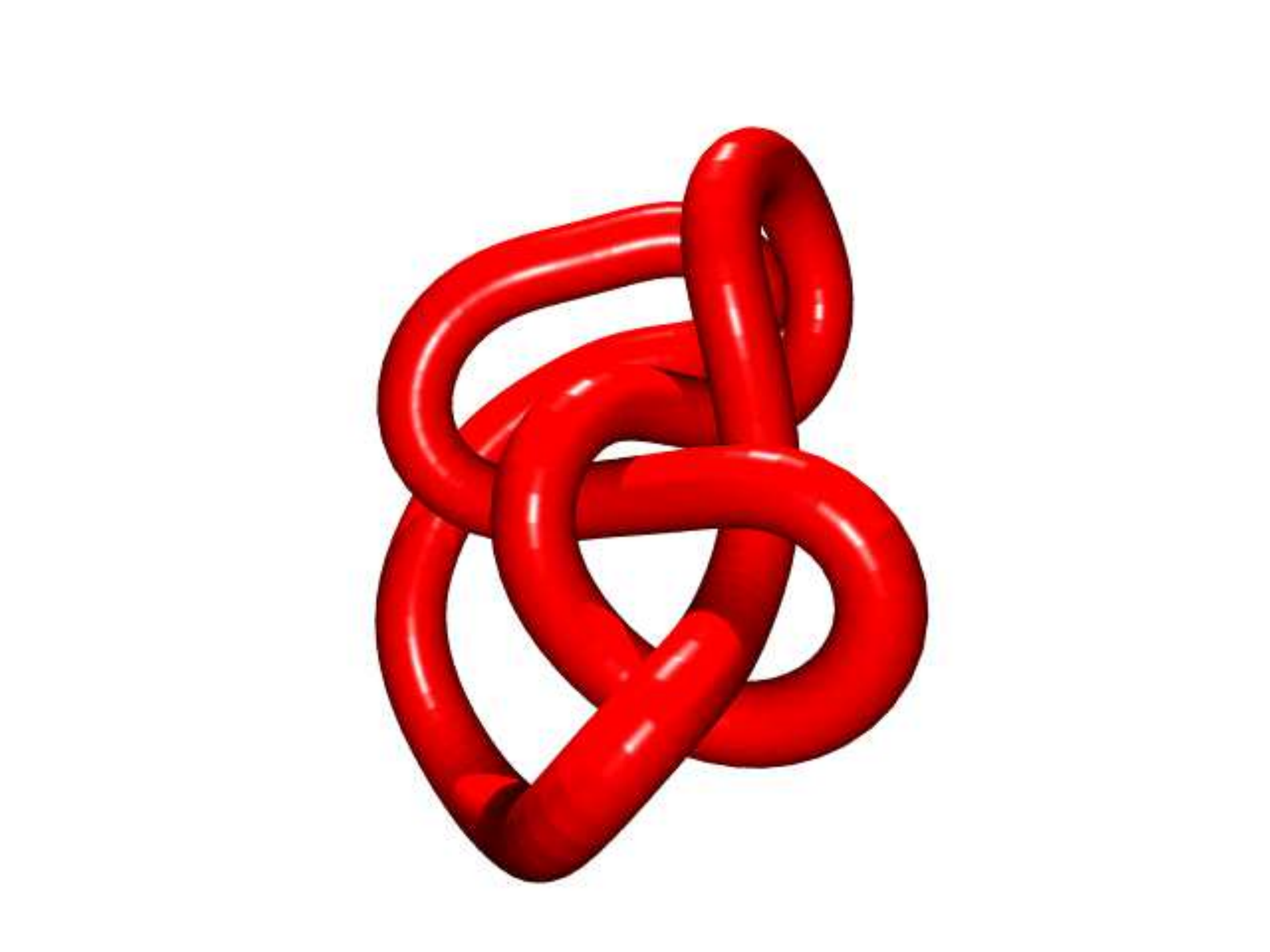} \\
500/100000 & 1000/100000 \\
$\Le(\gamma)\approx 1.80156$ & $\Le(\gamma)\approx 1.60106$ \\
$\E_p(\gamma)\approx 21.37764$ & $\E_p(\gamma)\approx 19.76579$ \\
$\tau=0.00039$ & $\tau=0.00063$ \\
\includegraphics[width=0.4\textwidth,keepaspectratio]{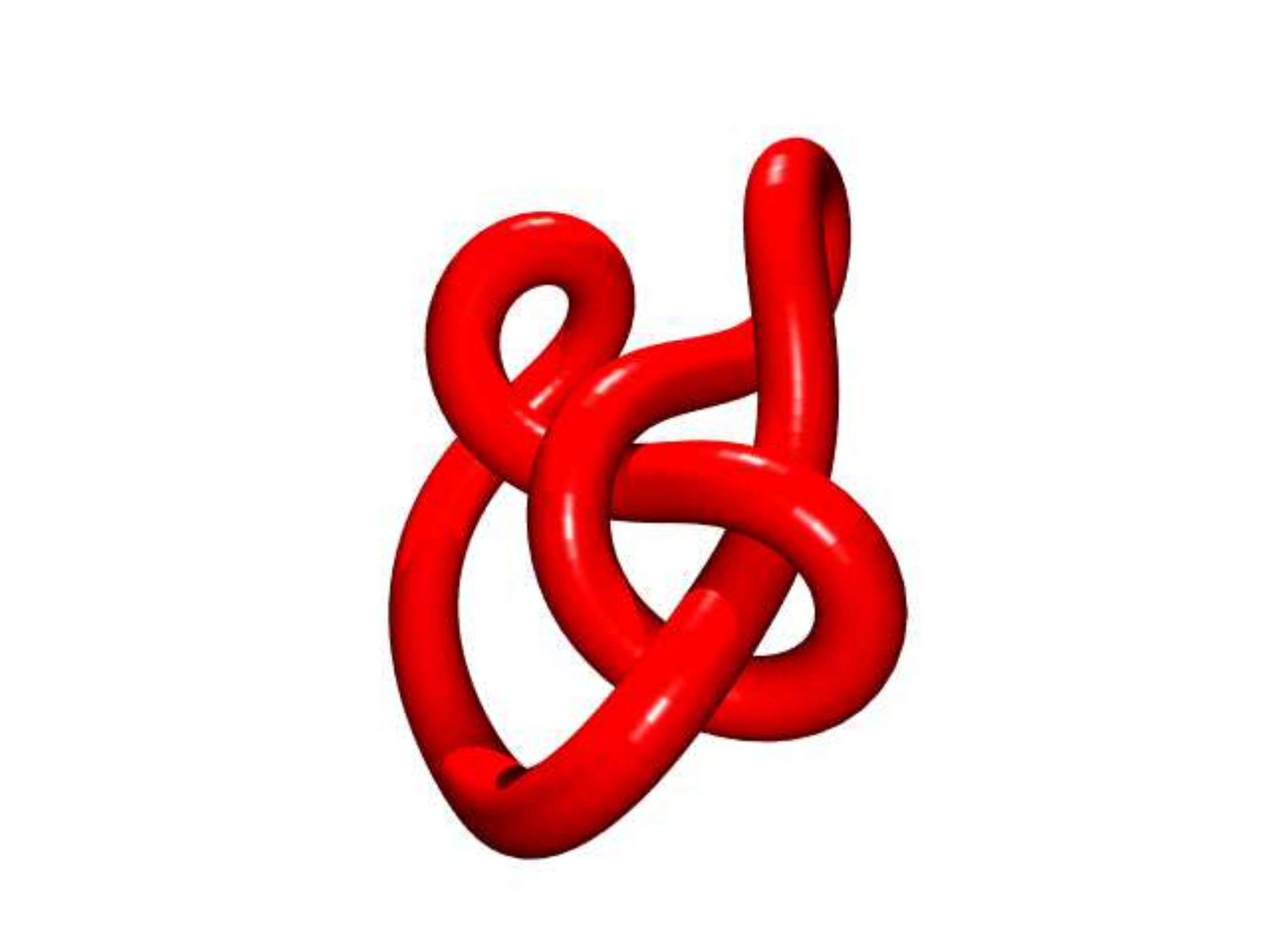} & \includegraphics[width=0.4\textwidth,keepaspectratio]{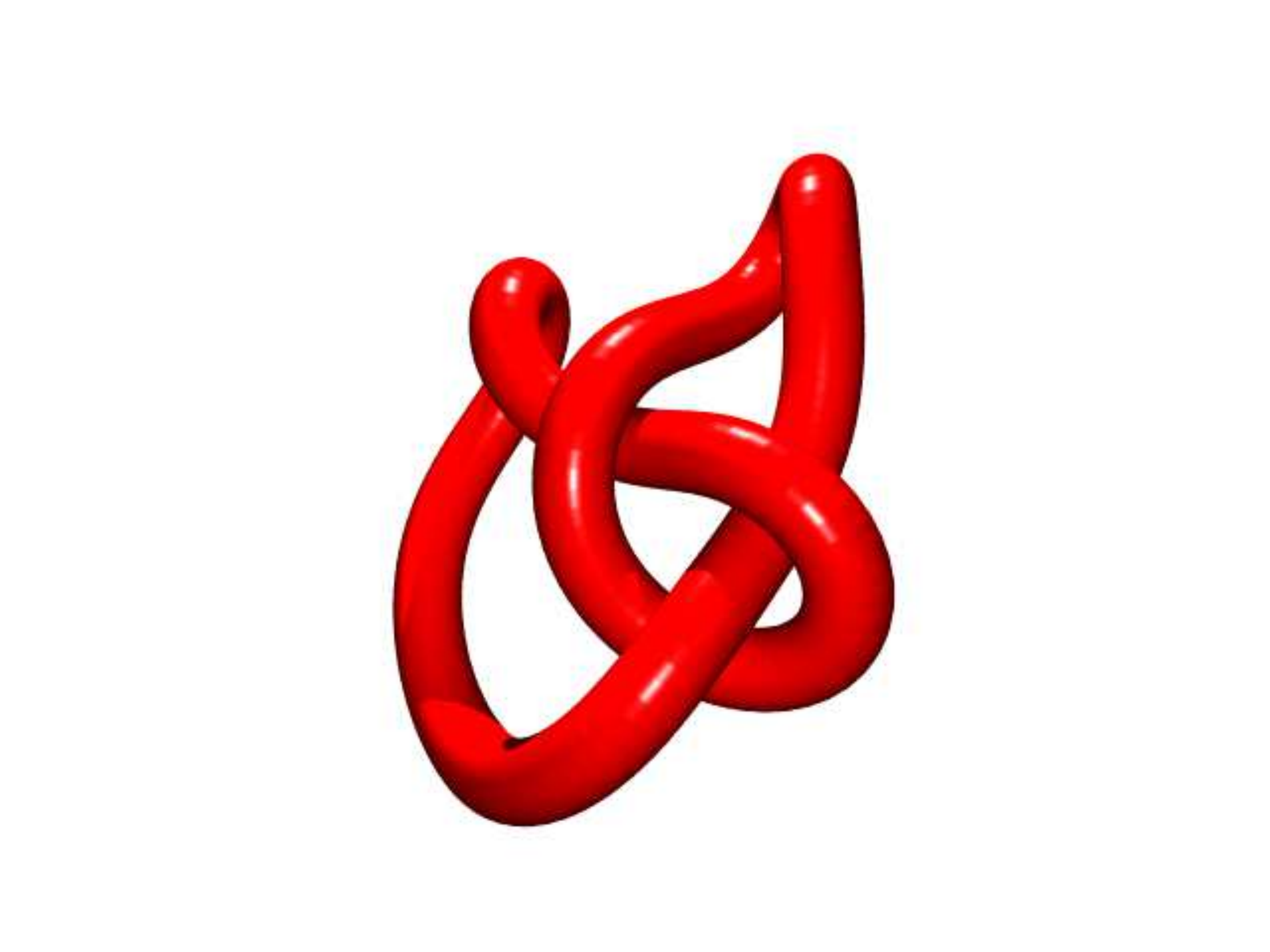} \\
5500/100000 & 100000/100000 \\
$\Le(\gamma)\approx 1.43246$ & $\Le(\gamma)\approx 1.46025$ \\
$\E_p(\gamma)\approx 17.13617$ & $\E_p(\gamma)\approx 17.0783$ \\
$\tau=0.00233$ & $\tau=0.0446$ \\
\includegraphics[width=0.4\textwidth,keepaspectratio]{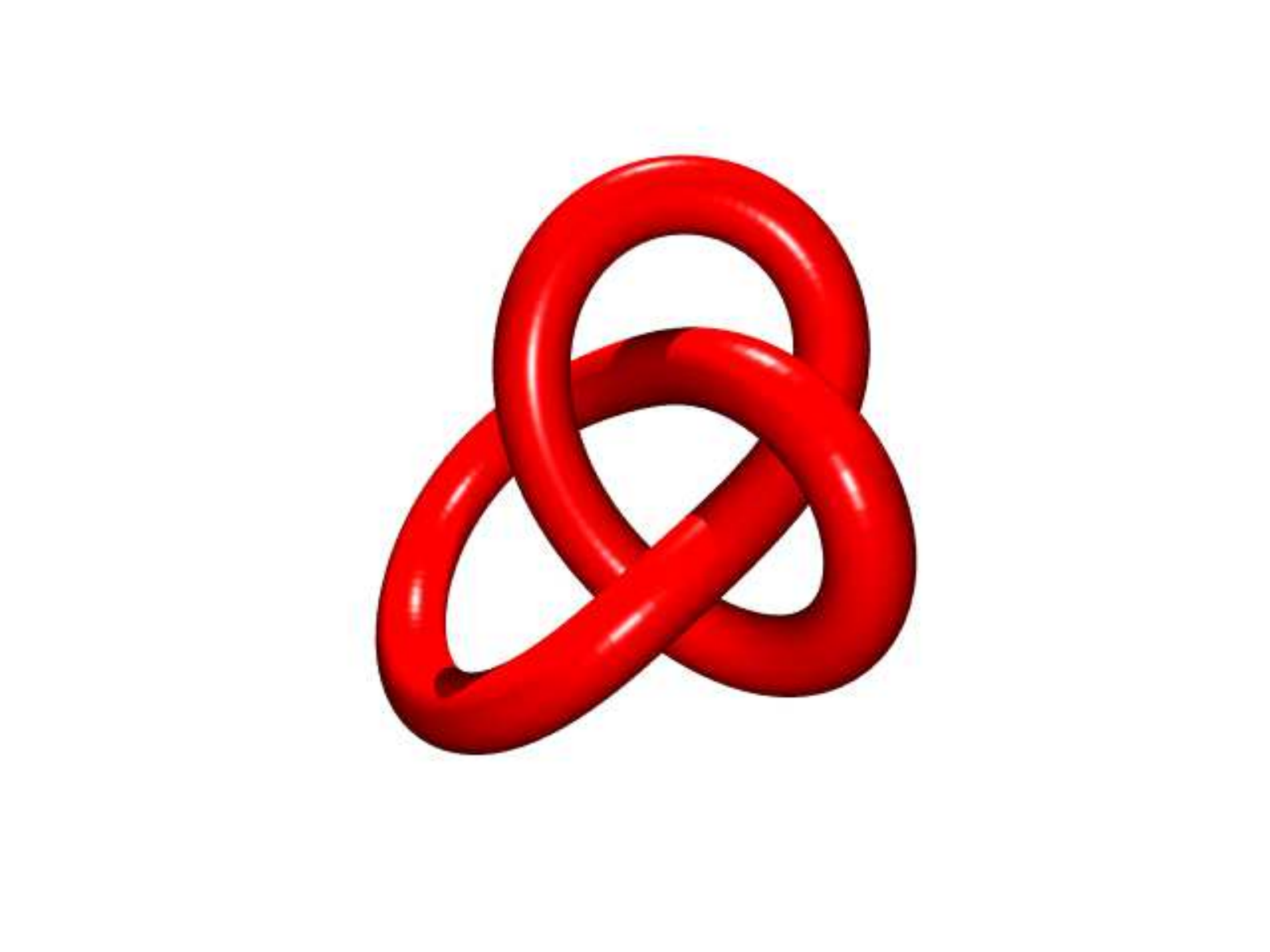} & \includegraphics[width=0.4\textwidth,keepaspectratio]{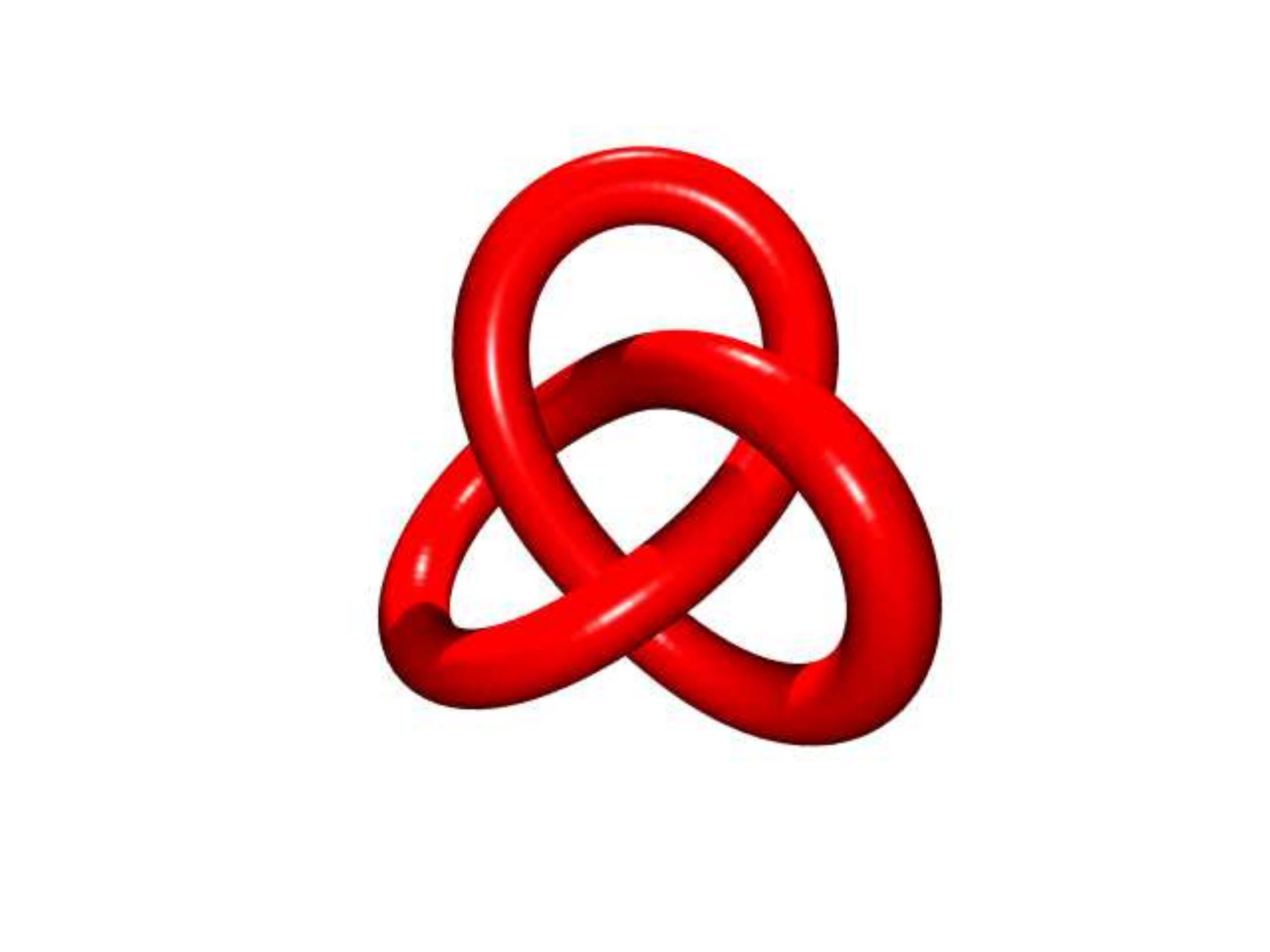}
\end{tabular}
\end{scriptsize}
\end{center}
\caption{A deformed trefoil -- $p=3.5$}
\end{figure}

The same is true for $p=50$
\begin{figure}[H]
\begin{center}
\begin{scriptsize}
\begin{tabular}{cc}
0/50000 & 100/50000 \\
$\Le(\gamma)\approx 2.52383$ & $\Le(\gamma)\approx 2.26165$ \\
$\E_p(\gamma)\approx 126.45736$ & $\E_p(\gamma)\approx 57.34568$ \\
$\tau=0.0$ & $\tau=2e-05$ \\
\includegraphics[width=0.4\textwidth,keepaspectratio]{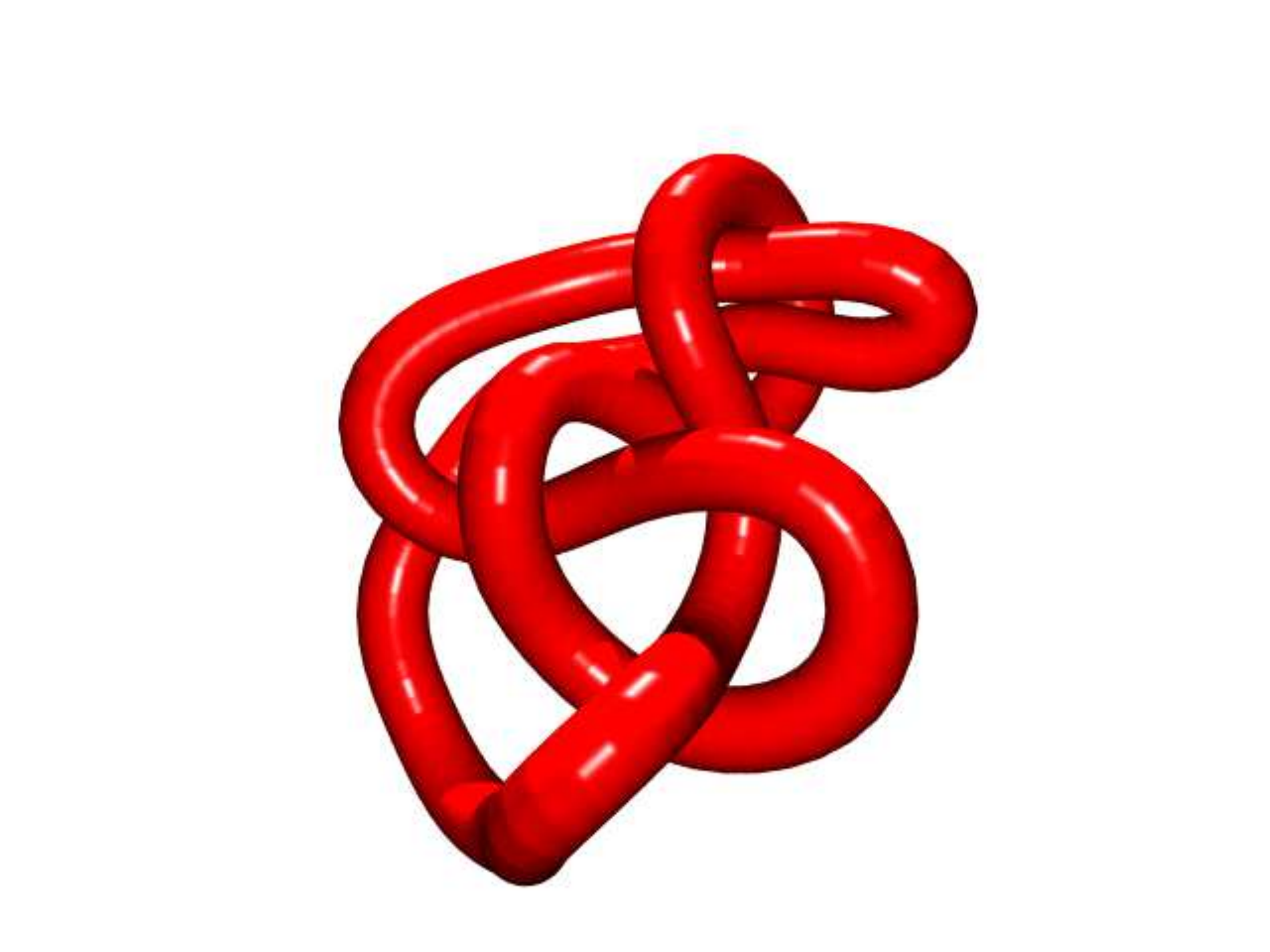} & \includegraphics[width=0.4\textwidth,keepaspectratio]{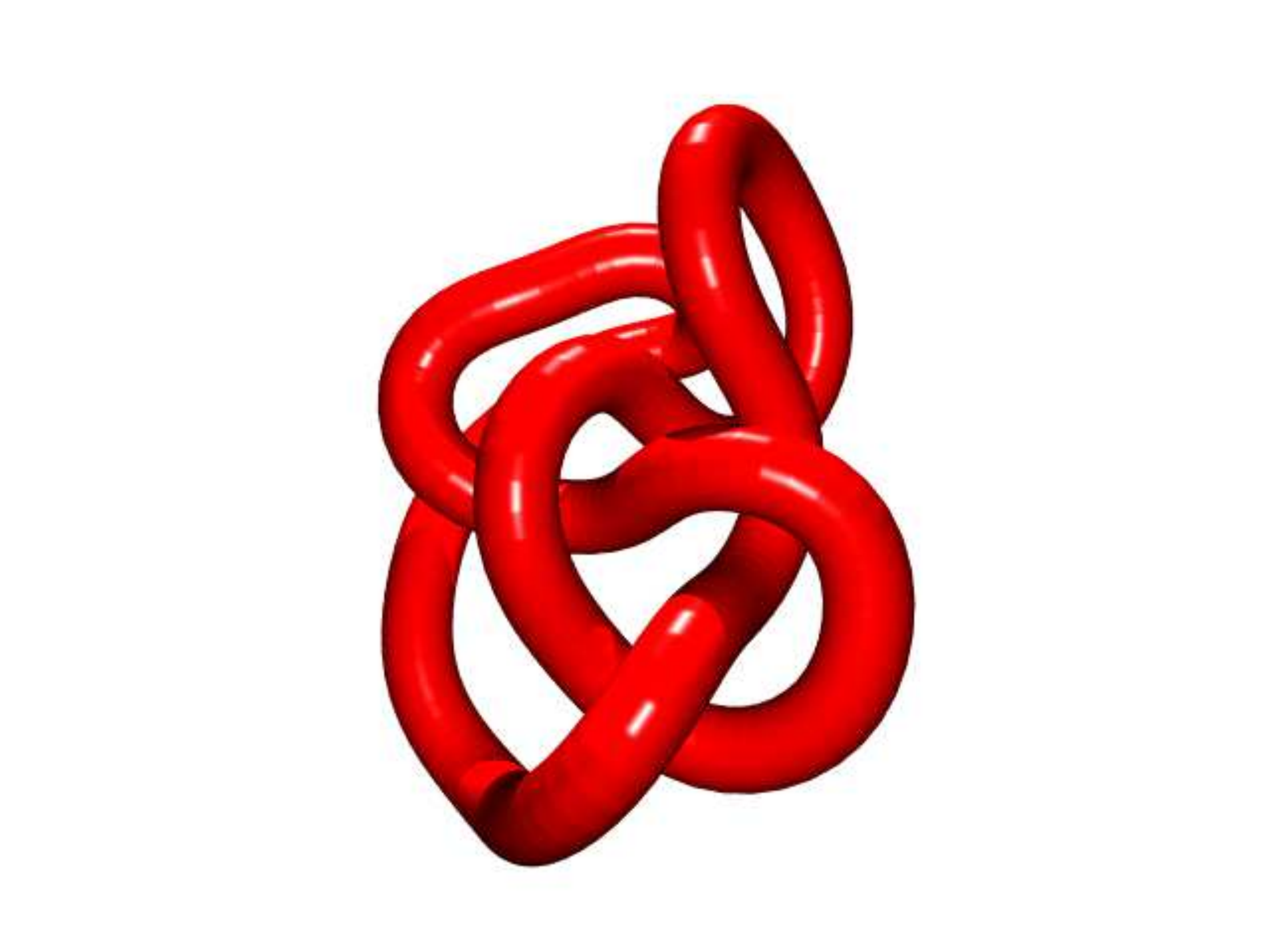} \\
300/50000 & 600/50000 \\
$\Le(\gamma)\approx 2.03835$ & $\Le(\gamma)\approx 1.82524$ \\
$\E_p(\gamma)\approx 45.96166$ & $\E_p(\gamma)\approx 38.70333$ \\
$\tau=9e-05$ & $\tau=0.00018$ \\
\includegraphics[width=0.4\textwidth,keepaspectratio]{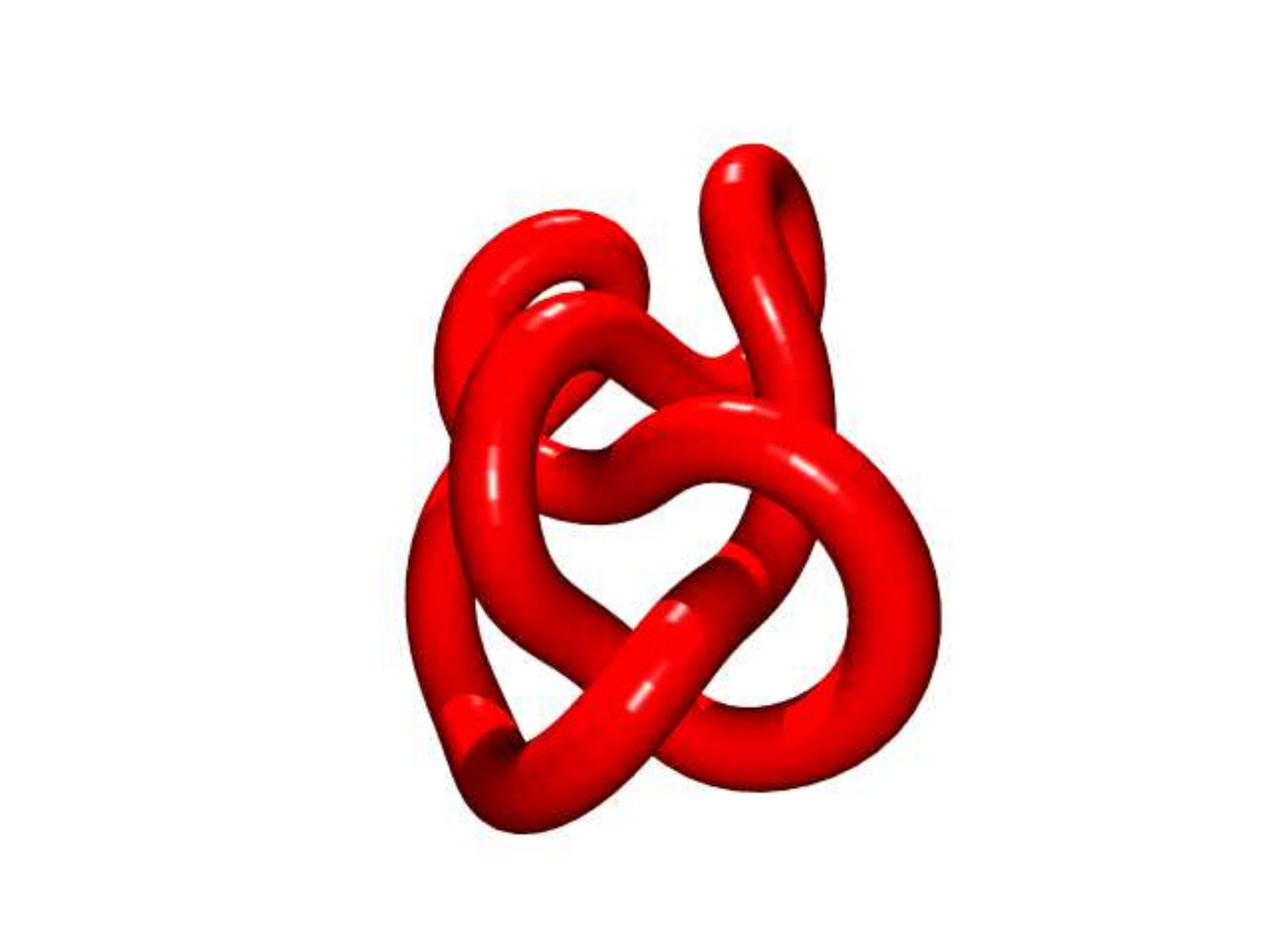} & \includegraphics[width=0.4\textwidth,keepaspectratio]{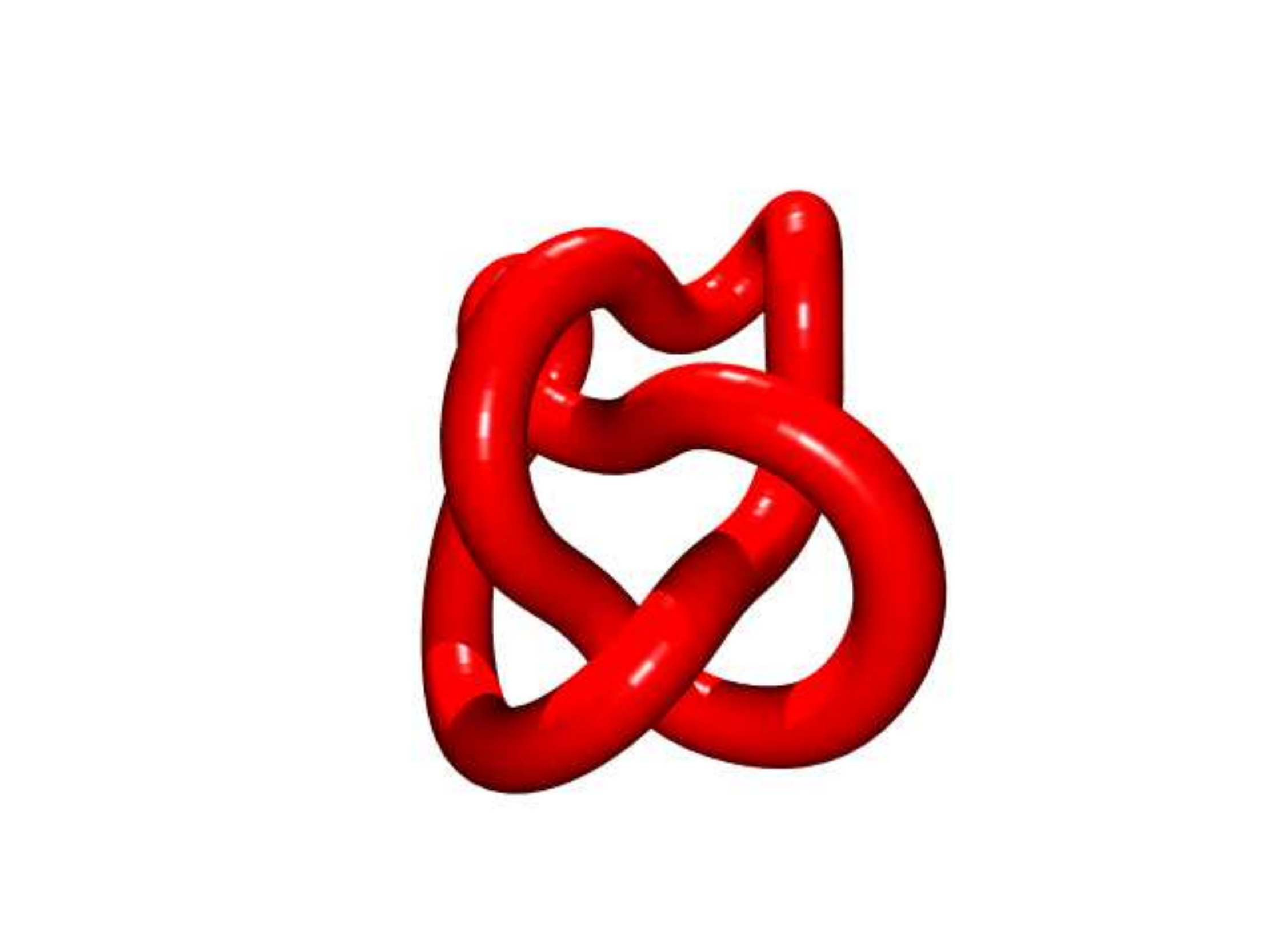} \\
1200/50000 & 10000/50000 \\
$\Le(\gamma)\approx 1.57742$ & $\Le(\gamma)\approx 1.55291$ \\
$\E_p(\gamma)\approx 30.4483$ & $\E_p(\gamma)\approx 29.57435$ \\
$\tau=0.00033$ & $\tau=0.00361$ \\
\includegraphics[width=0.4\textwidth,keepaspectratio]{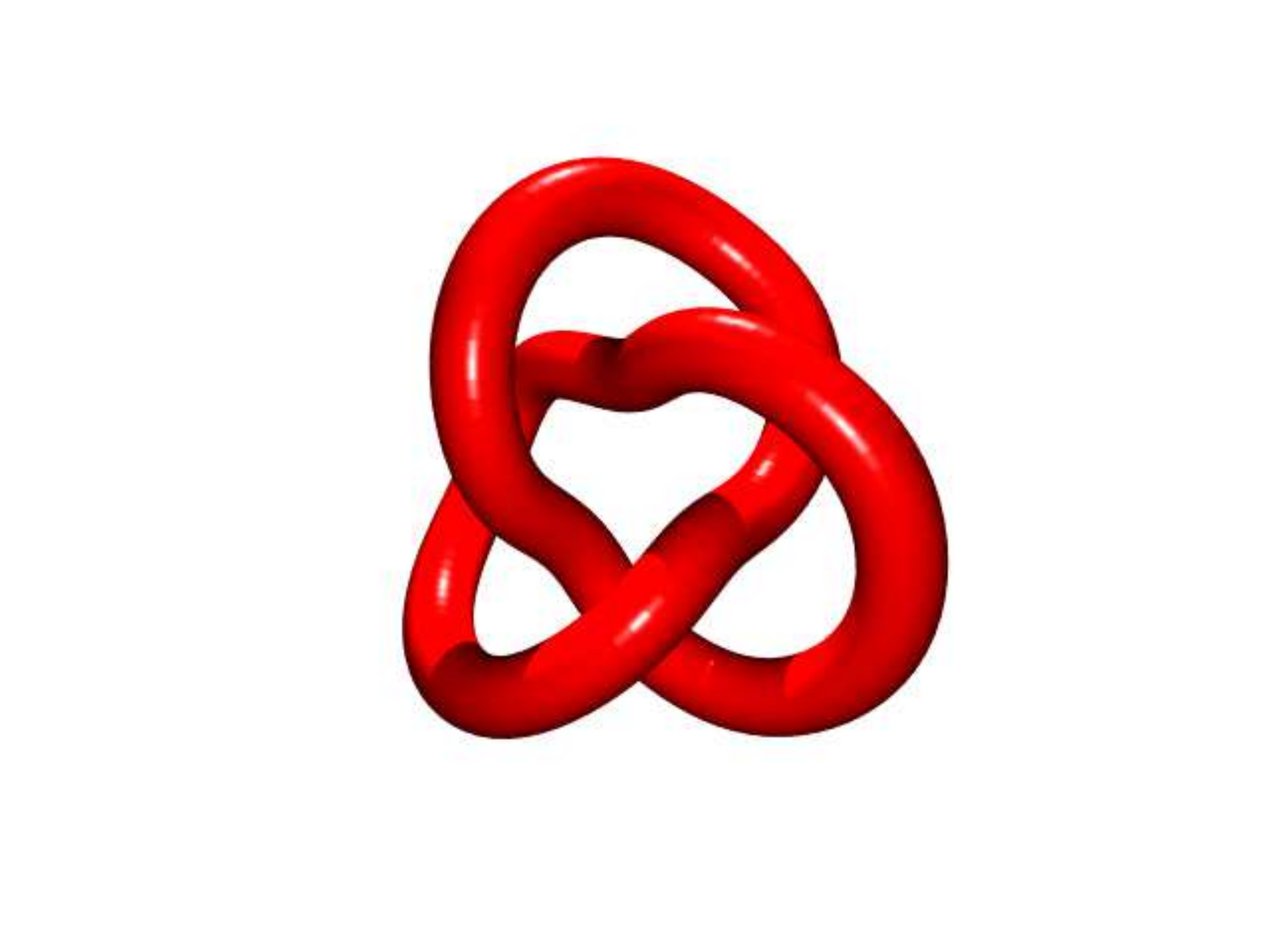} & \includegraphics[width=0.4\textwidth,keepaspectratio]{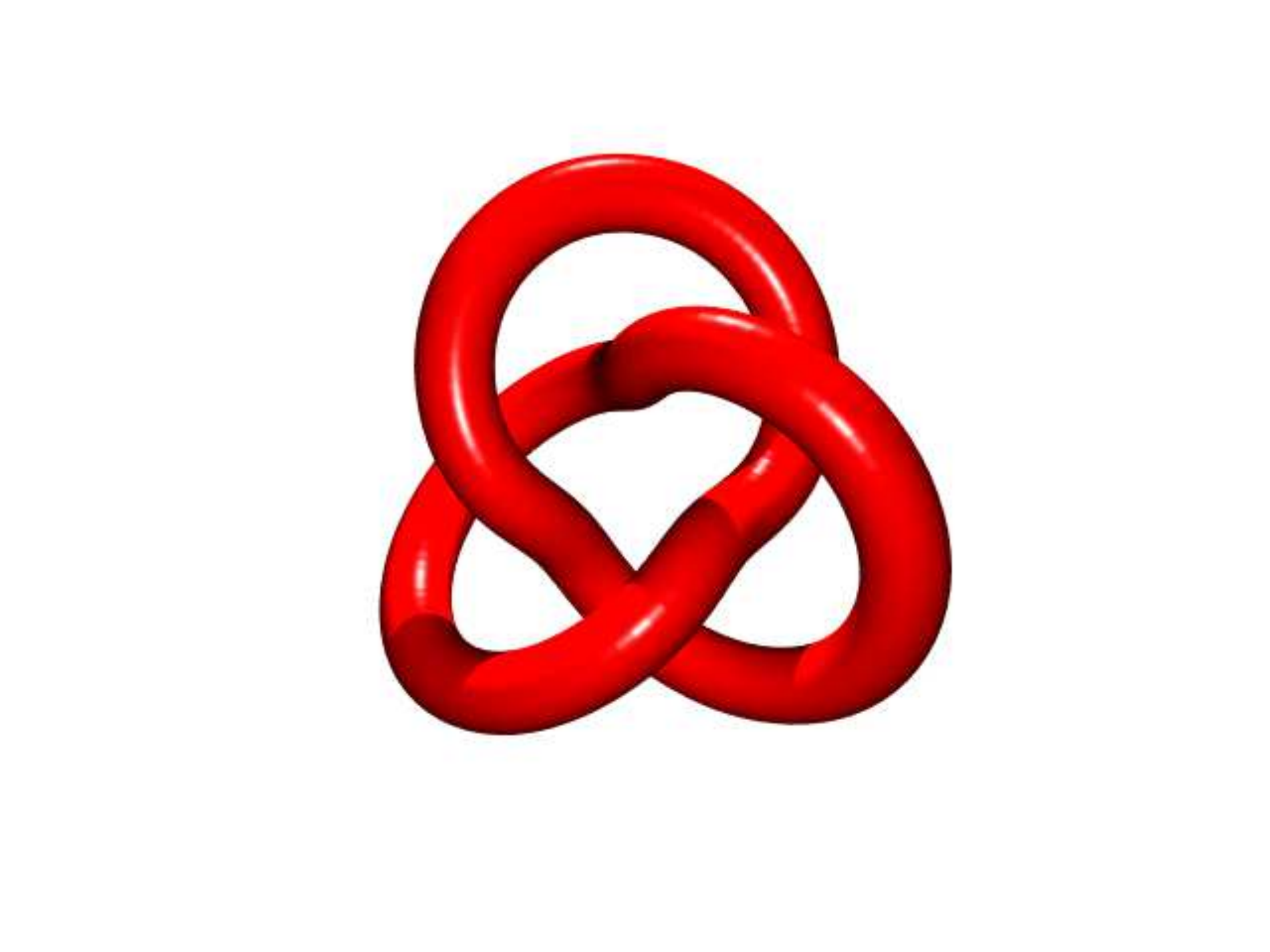}
\end{tabular}
\end{scriptsize}
\end{center}
\caption{A deformed trefoil -- $p=50.0$}
\end{figure}

The shapes of the last configurations of these flows and those for the trefoil we used first, are the same for $p=3.5$ and respectively for $p=50$. However, the shapes of the last configurations for $p=3.5$ and $p=50$ differ strongly.

Observe, that the flow for $p=3$ and for $p=3.5$ is not able to untangle the non-trivial unknot from above without self-penetration. However, if we choose $p=4$ instead the untanglement succeeds as well.

\section{Figure-eight knots ($4_1$)} \label{knot41}
We proceed with the knot, which has crossing number $4$. For $p=3$ the knot class is abandoned and we end up with a circle
\begin{figure}[H]
\begin{center}
\begin{scriptsize}
\begin{tabular}{ccc}
0/300000 & 5000/300000 & 25000/300000 \\
$\Le(\gamma)\approx 38.03944$ & $\Le(\gamma)\approx 34.49324$ & $\Le(\gamma)\approx 33.72952$ \\
$\E_p(\gamma)\approx 19.31638$ & $\E_p(\gamma)\approx 18.40864$ & $\E_p(\gamma)\approx 17.58617$ \\
$\tau=0.0$ & $\tau=18.84871$ & $\tau=45.60539$ \\
\includegraphics[width=0.33\textwidth,keepaspectratio]{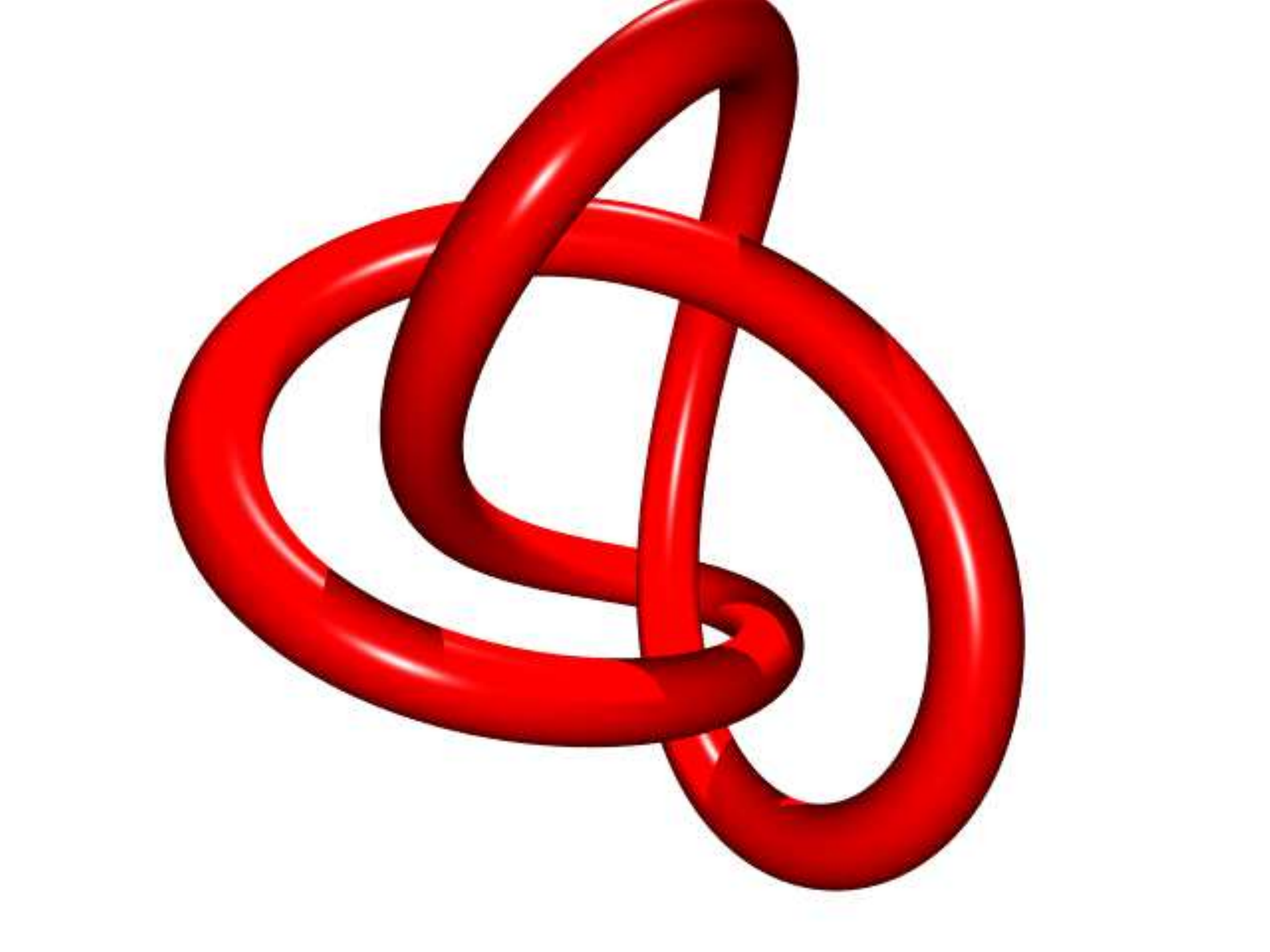} & \includegraphics[width=0.33\textwidth,keepaspectratio]{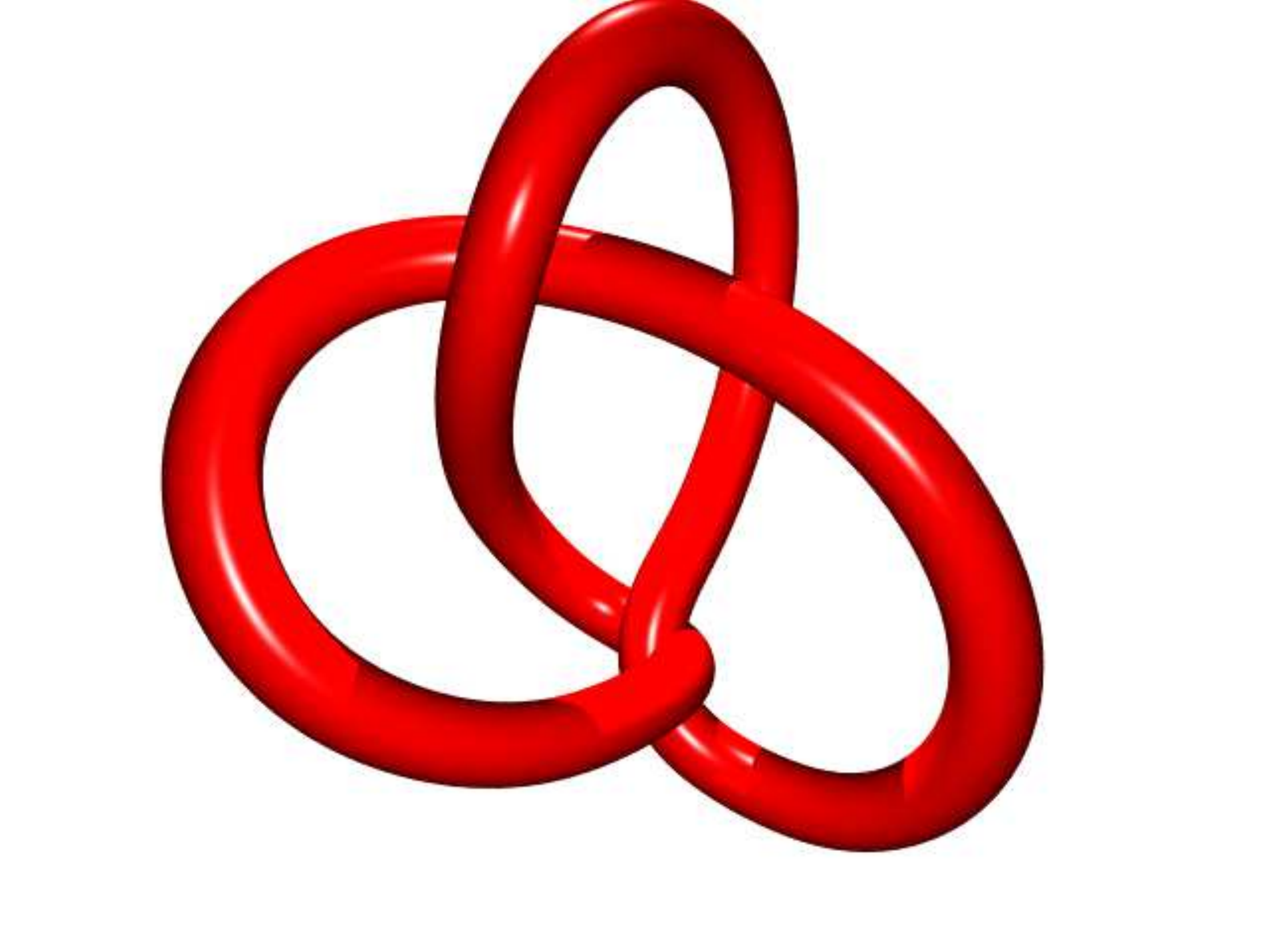} & \includegraphics[width=0.33\textwidth,keepaspectratio]{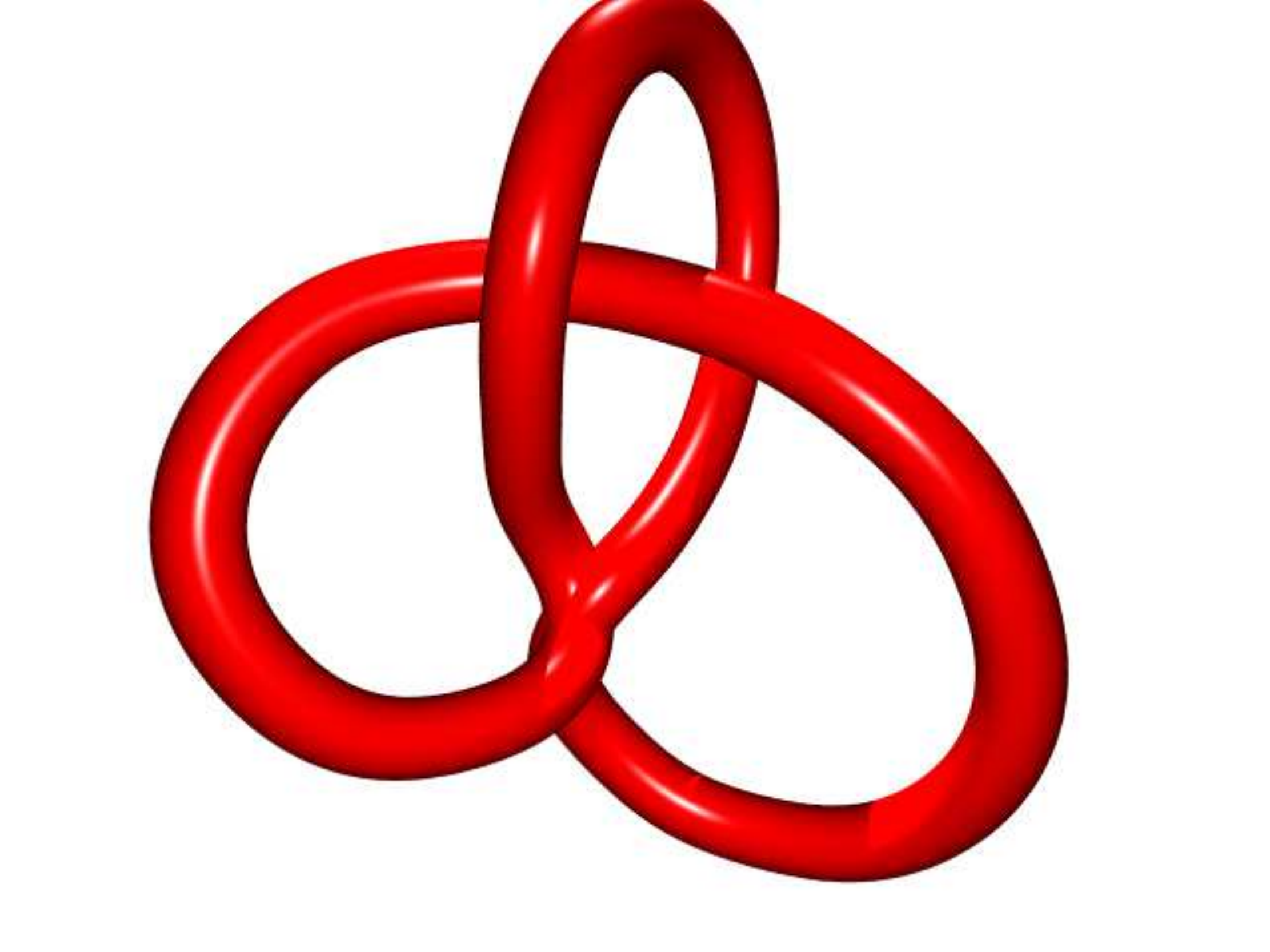} \\
100000/300000 & 126000/300000 & 300000/300000 \\
$\Le(\gamma)\approx 32.62229$ & $\Le(\gamma)\approx 27.78634$ & $\Le(\gamma)\approx 17.11188$ \\
$\E_p(\gamma)\approx 17.01334$ & $\E_p(\gamma)\approx 13.21349$ & $\E_p(\gamma)\approx 6.28319$ \\
$\tau=51.98823$ & $\tau=57.08792$ & $\tau=214.89181$ \\
\includegraphics[width=0.33\textwidth,keepaspectratio]{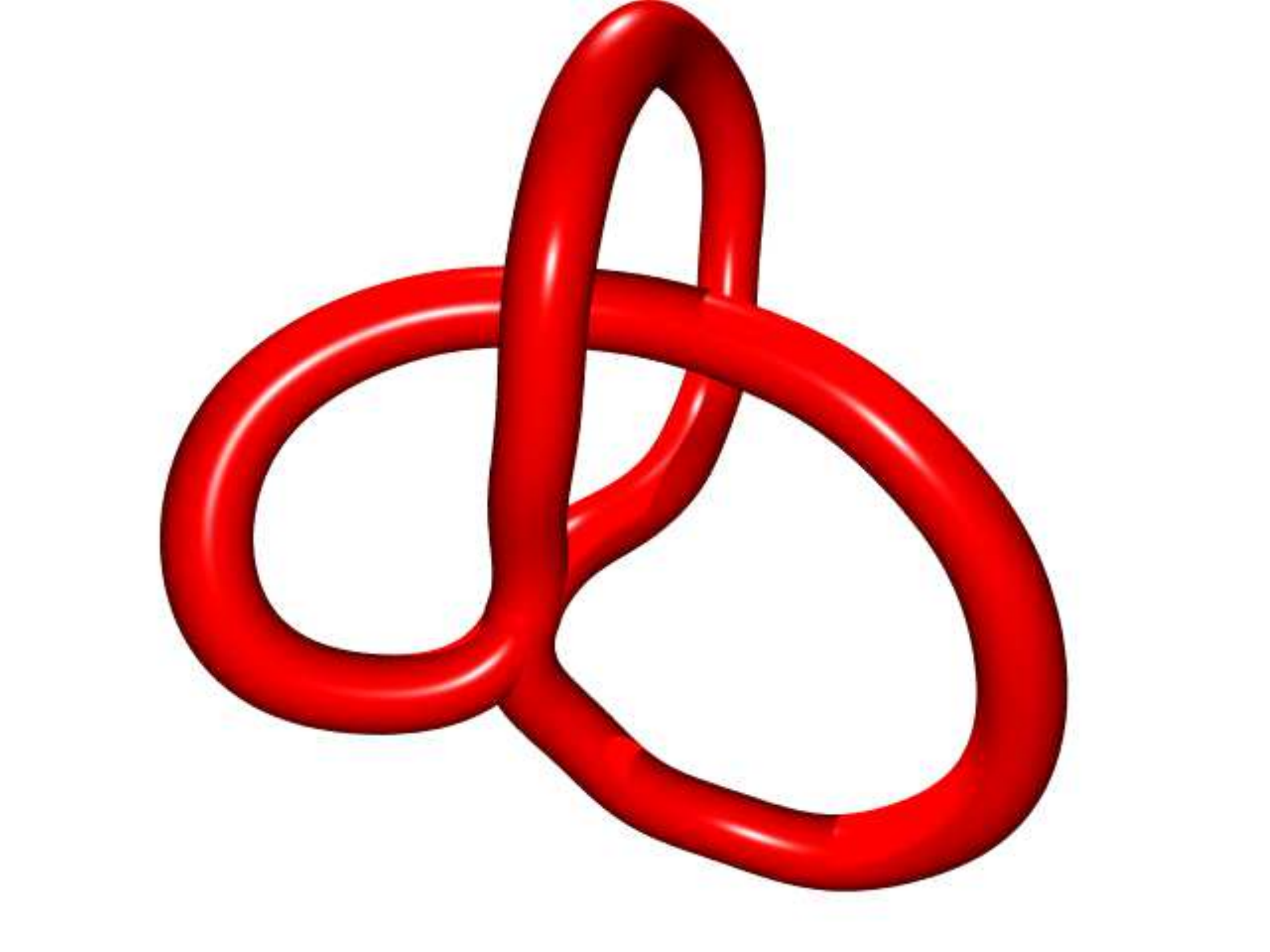} & \includegraphics[width=0.33\textwidth,keepaspectratio]{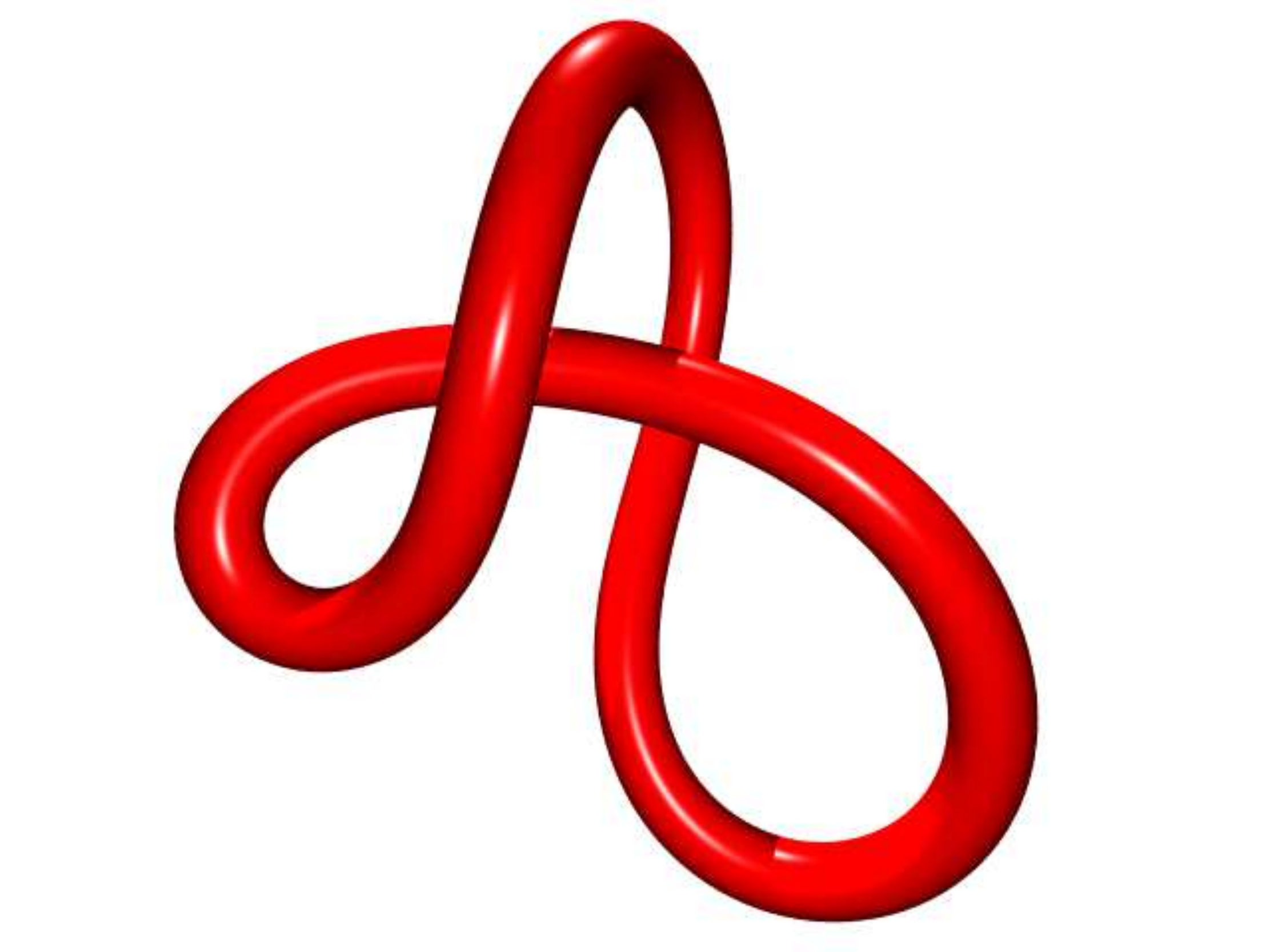} & \includegraphics[width=0.33\textwidth,keepaspectratio]{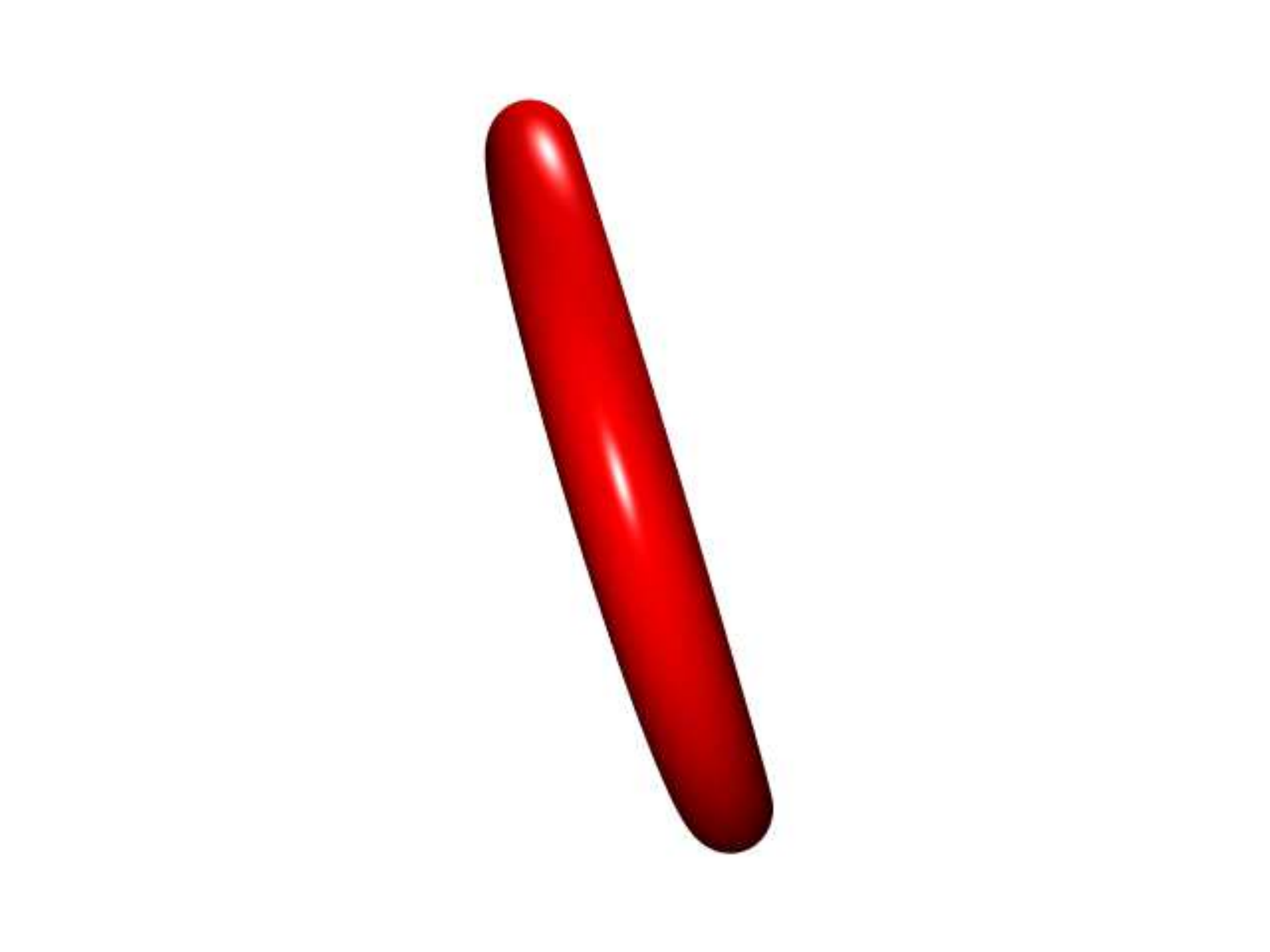}
\end{tabular}
\end{scriptsize}
\end{center}
\caption{A harmonic figure-eight -- $p=3.0$}
\end{figure}

Therefore, we consider the case $p=3.5$ next
\begin{figure}[H]
\begin{center}
\begin{scriptsize}
\begin{tabular}{cc}
0/300000 & 46000/300000 \\
$\Le(\gamma)\approx 38.03944$ & $\Le(\gamma)\approx 36.805$ \\
$\E_p(\gamma)\approx 20.1996$ & $\E_p(\gamma)\approx 19.72164$ \\
$\tau=0.0$ & $\tau=272.76018$ \\
\includegraphics[width=0.33\textwidth,keepaspectratio]{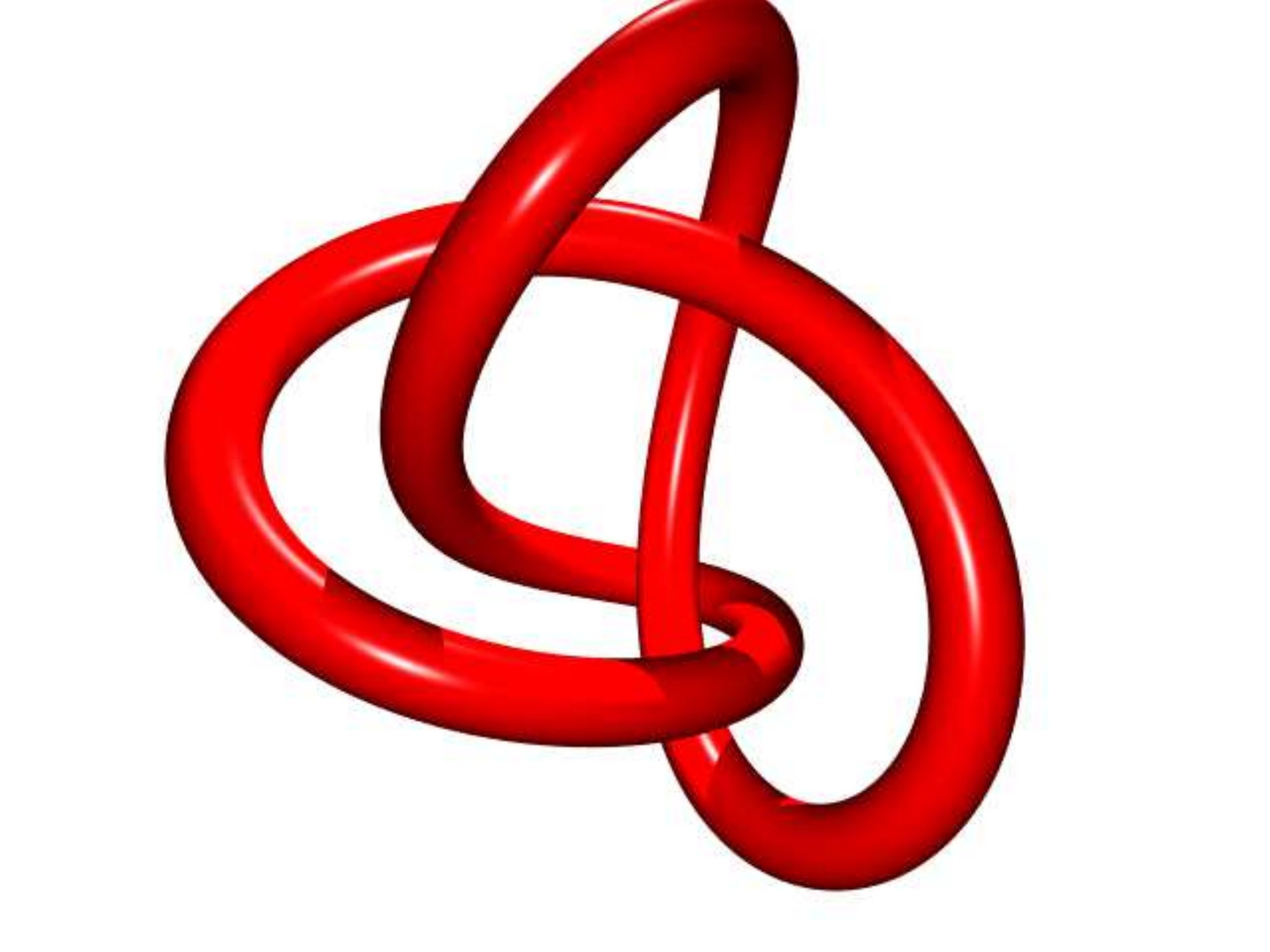} & \includegraphics[width=0.33\textwidth,keepaspectratio]{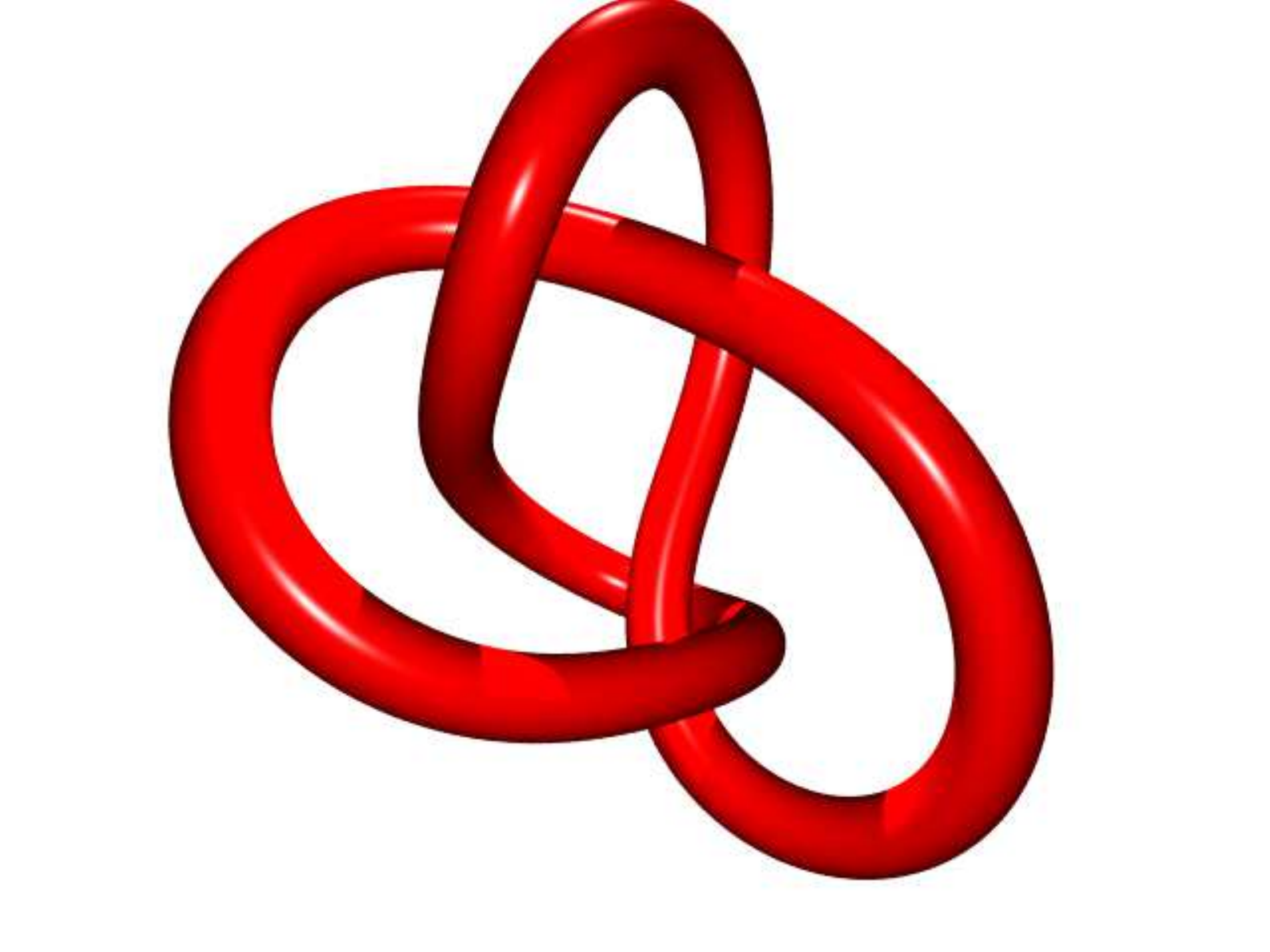} \\
170000/300000 & 300000/300000 \\
$\Le(\gamma)\approx 35.76346$ & $\Le(\gamma)\approx 35.14579$ \\
$\E_p(\gamma)\approx 19.68515$ & $\E_p(\gamma)\approx 19.6115$ \\
$\tau=972.00521$ & $\tau=1477.14949$ \\
\includegraphics[width=0.33\textwidth,keepaspectratio]{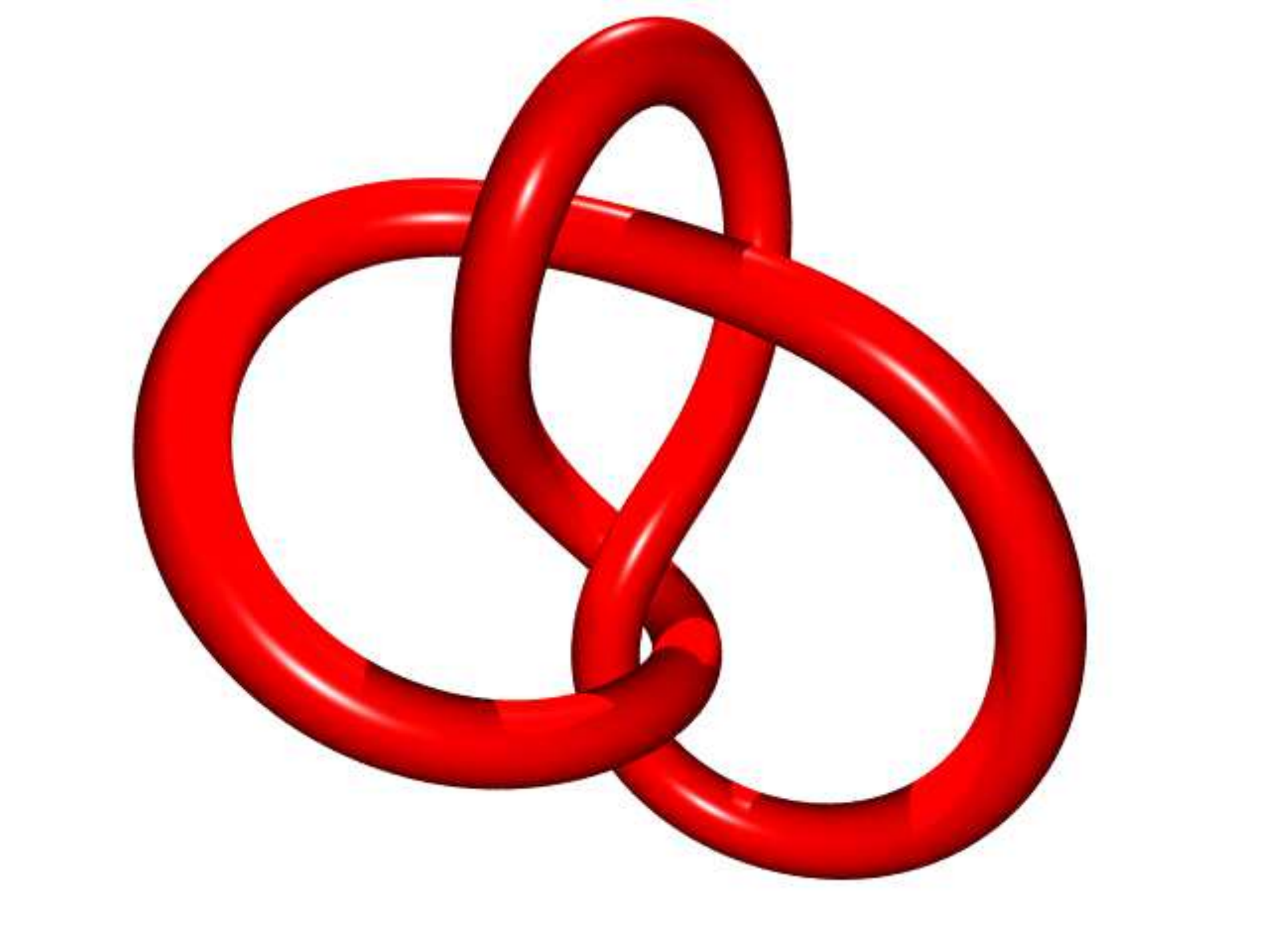} & \includegraphics[width=0.33\textwidth,keepaspectratio]{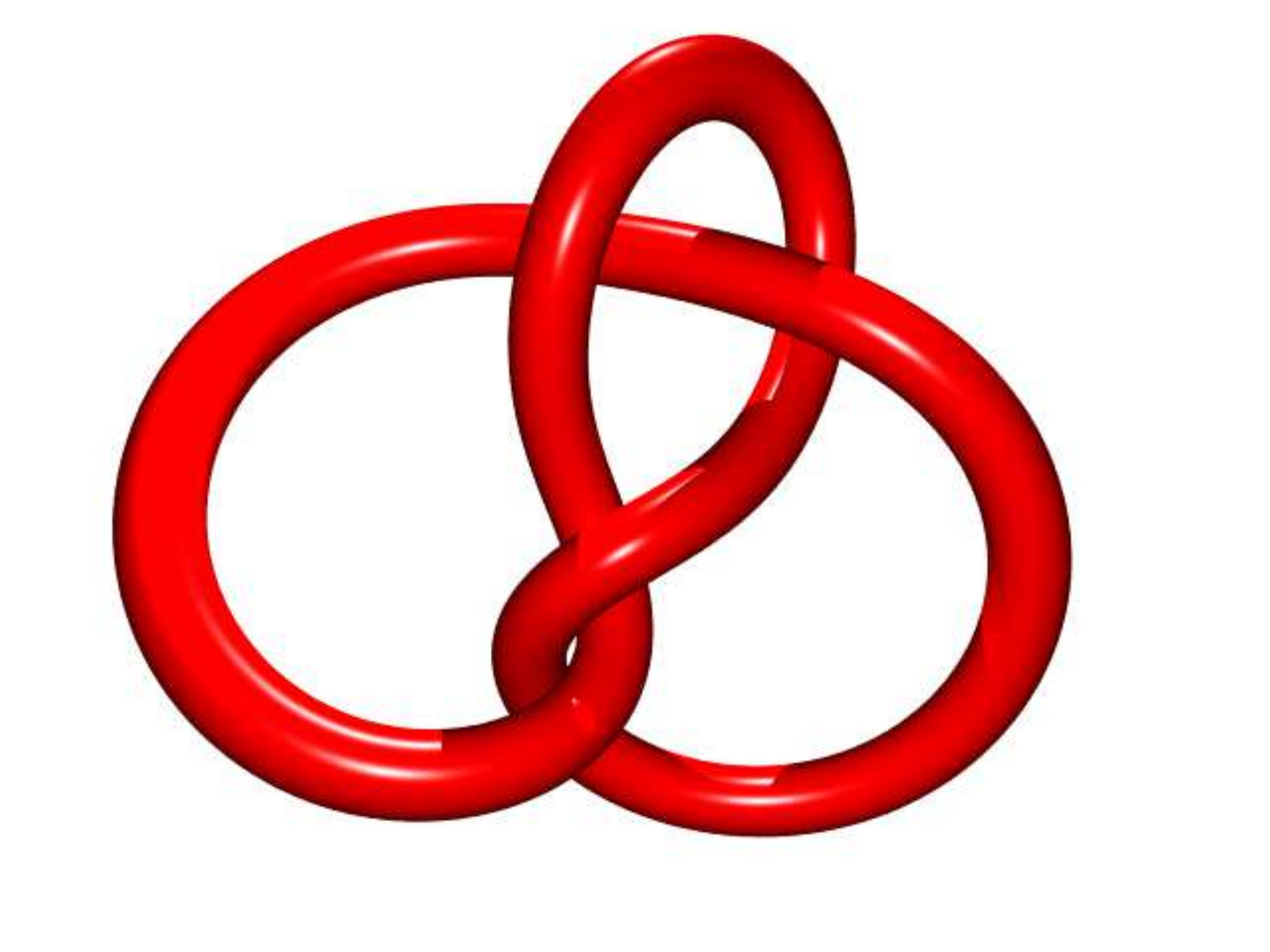}
\end{tabular}
\end{scriptsize}
\end{center}
\caption{A harmonic figure-eight -- $p=3.5$}
\end{figure}

By switching to the flow without redistribution after $500.000$ steps we reach an energy value of $19.61146$. Now we proceed with $p=50$
\begin{figure}[H]
\begin{center}
\begin{scriptsize}
\begin{tabular}{ccc}
0/30000 & 1000/30000 & 30000/30000 \\
$\Le(\gamma)\approx 38.03944$ & $\Le(\gamma)\approx 40.86402$ & $\Le(\gamma)\approx 41.89326$ \\
$\E_p(\gamma)\approx 68.70305$ & $\E_p(\gamma)\approx 37.51595$ & $\E_p(\gamma)\approx 37.49672$ \\
$\tau=0.0$ & $\tau=4.61167$ & $\tau=156.45102$ \\
\includegraphics[width=0.33\textwidth,keepaspectratio]{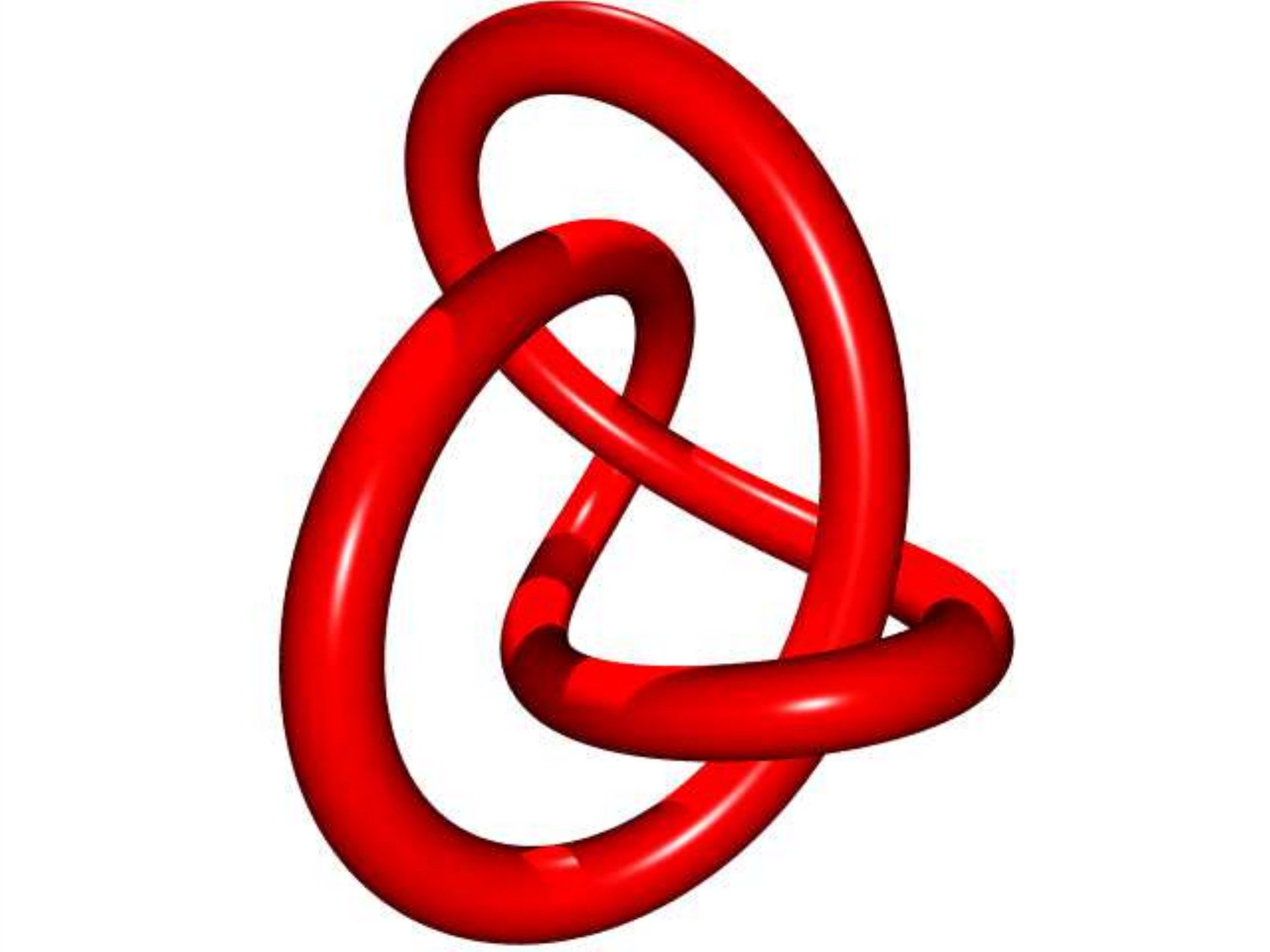} & \includegraphics[width=0.33\textwidth,keepaspectratio]{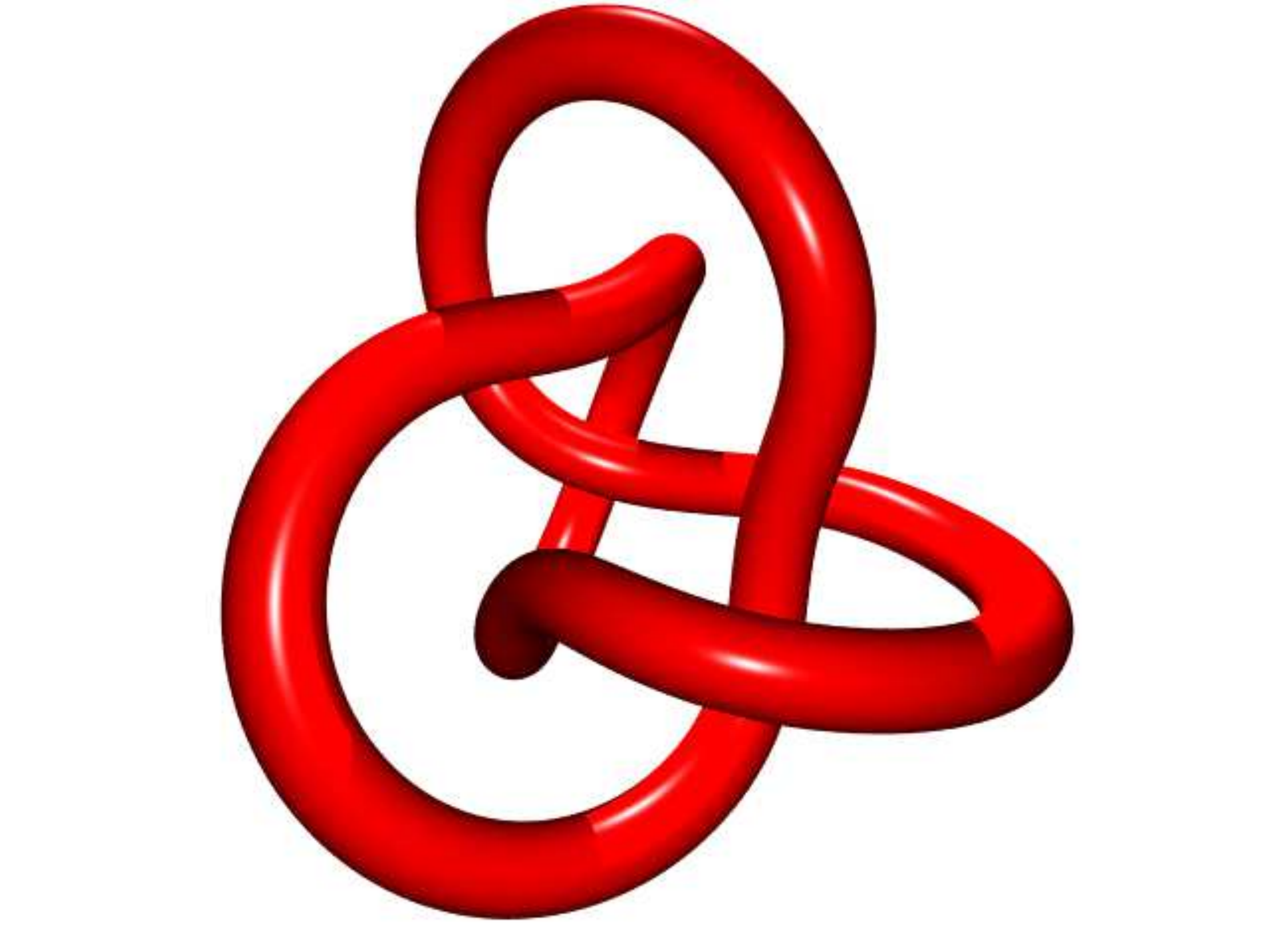} & \includegraphics[width=0.33\textwidth,keepaspectratio]{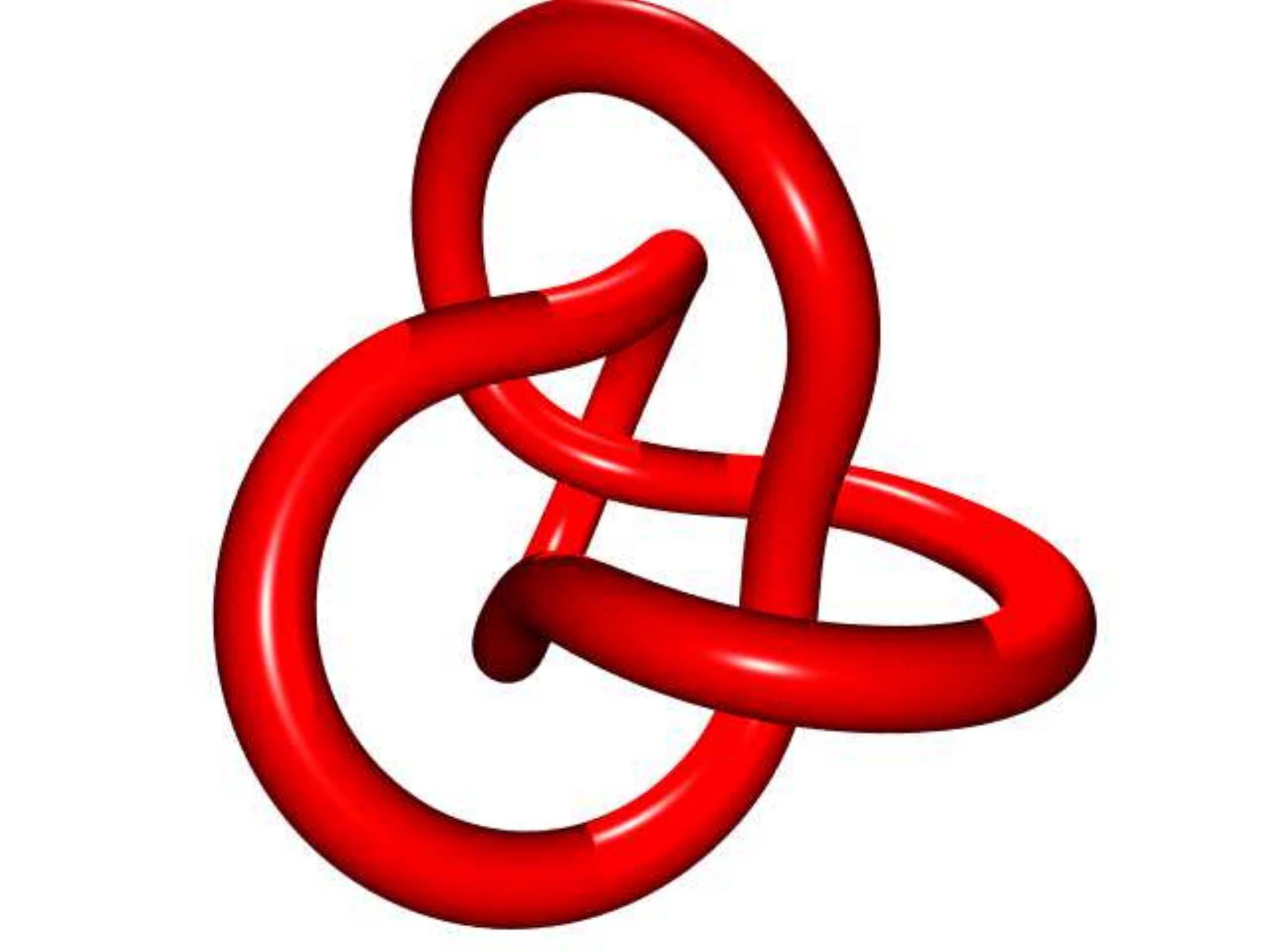}
\end{tabular}
\end{scriptsize}
\end{center}
\caption{A harmonic figure-eight -- $p=50.0$}
\end{figure}

After some time the energy starts to increase again, probably due to the redistribution. However, by switching to the flow without redistribution the energy decreases again and after $500.000$ steps the knot has the energy $37.47163$. Next we consider another representative of this knot class. Firstly, we use the flow with redistribution for $p=3.5$.
\begin{figure}[H]
\begin{center}
\begin{scriptsize}
\begin{tabular}{cc}
0/300000 & 10000/300000 \\
$\Le(\gamma)\approx 31.68443$ & $\Le(\gamma)\approx 19.72678$ \\
$\E_p(\gamma)\approx 30.34843$ & $\E_p(\gamma)\approx 20.90206$ \\
$\tau=0.0$ & $\tau=2.64537$ \\
\includegraphics[width=0.4\textwidth,keepaspectratio]{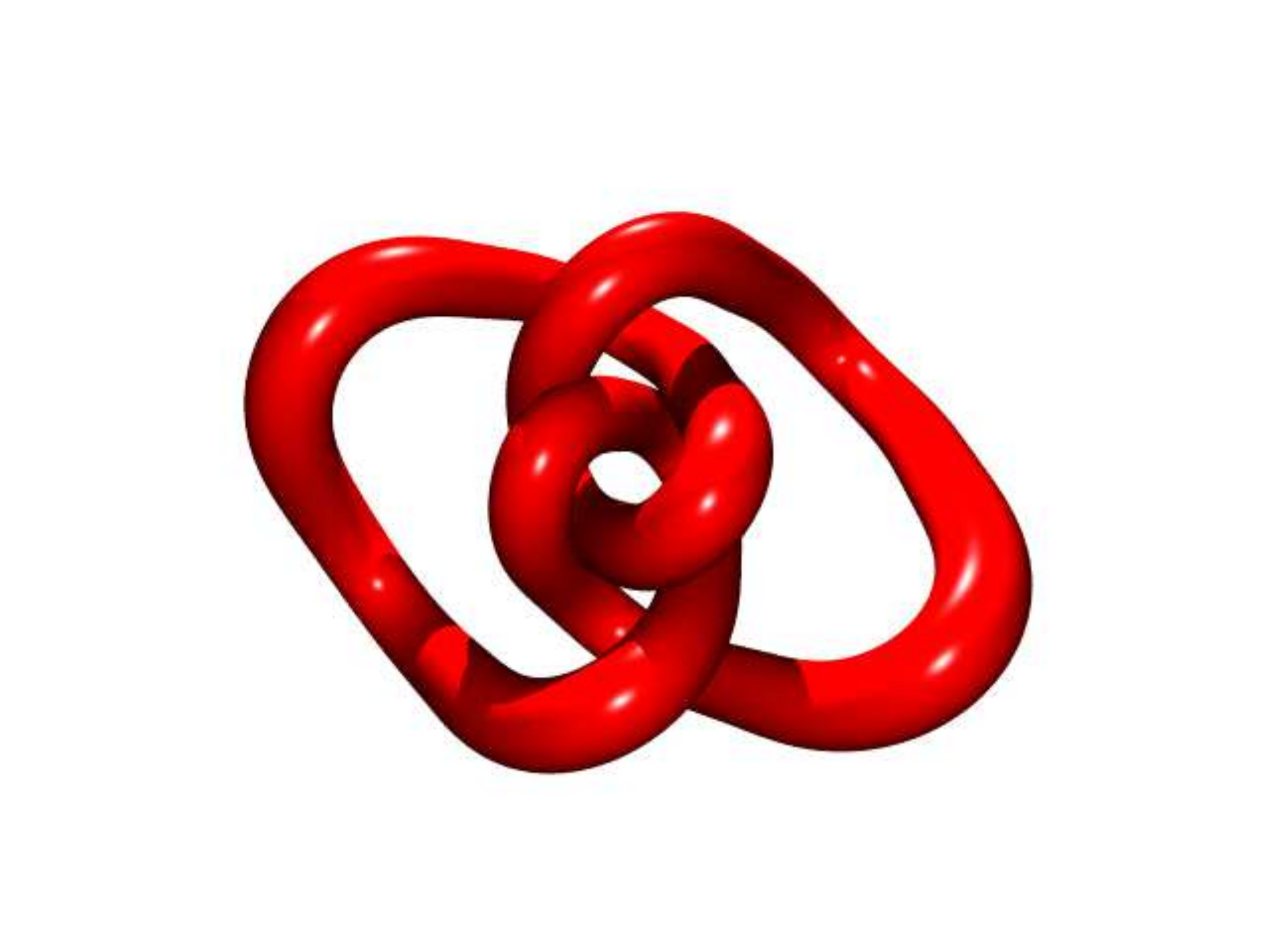} & \includegraphics[width=0.4\textwidth,keepaspectratio]{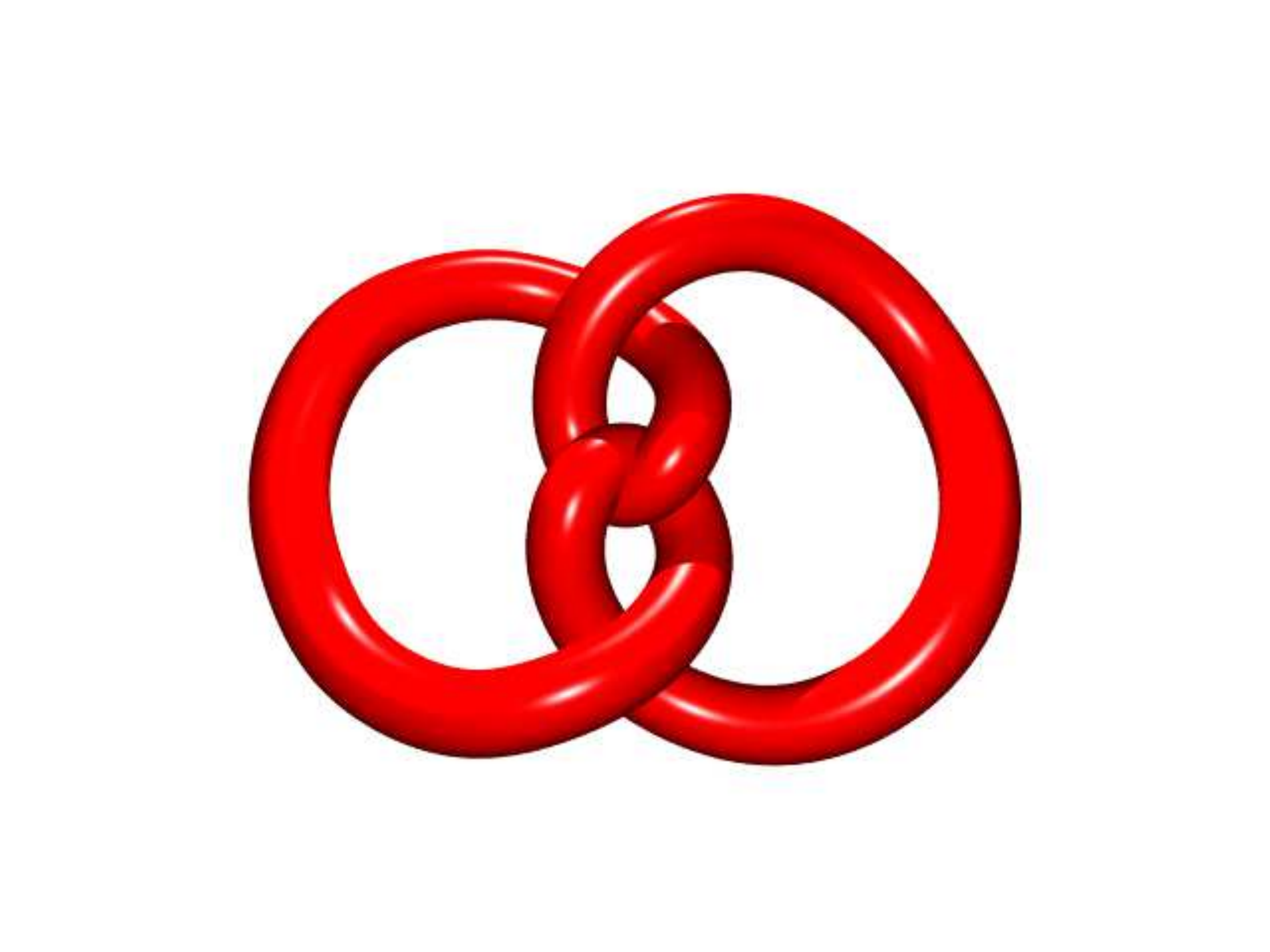} \\
100000/300000 & 174000/300000 \\
$\Le(\gamma)\approx 21.06465$ & $\Le(\gamma)\approx 21.72393$ \\
$\E_p(\gamma)\approx 20.64248$ & $\E_p(\gamma)\approx 20.30596$ \\
$\tau=24.86407$ & $\tau=47.02145$ \\
\includegraphics[width=0.4\textwidth,keepaspectratio]{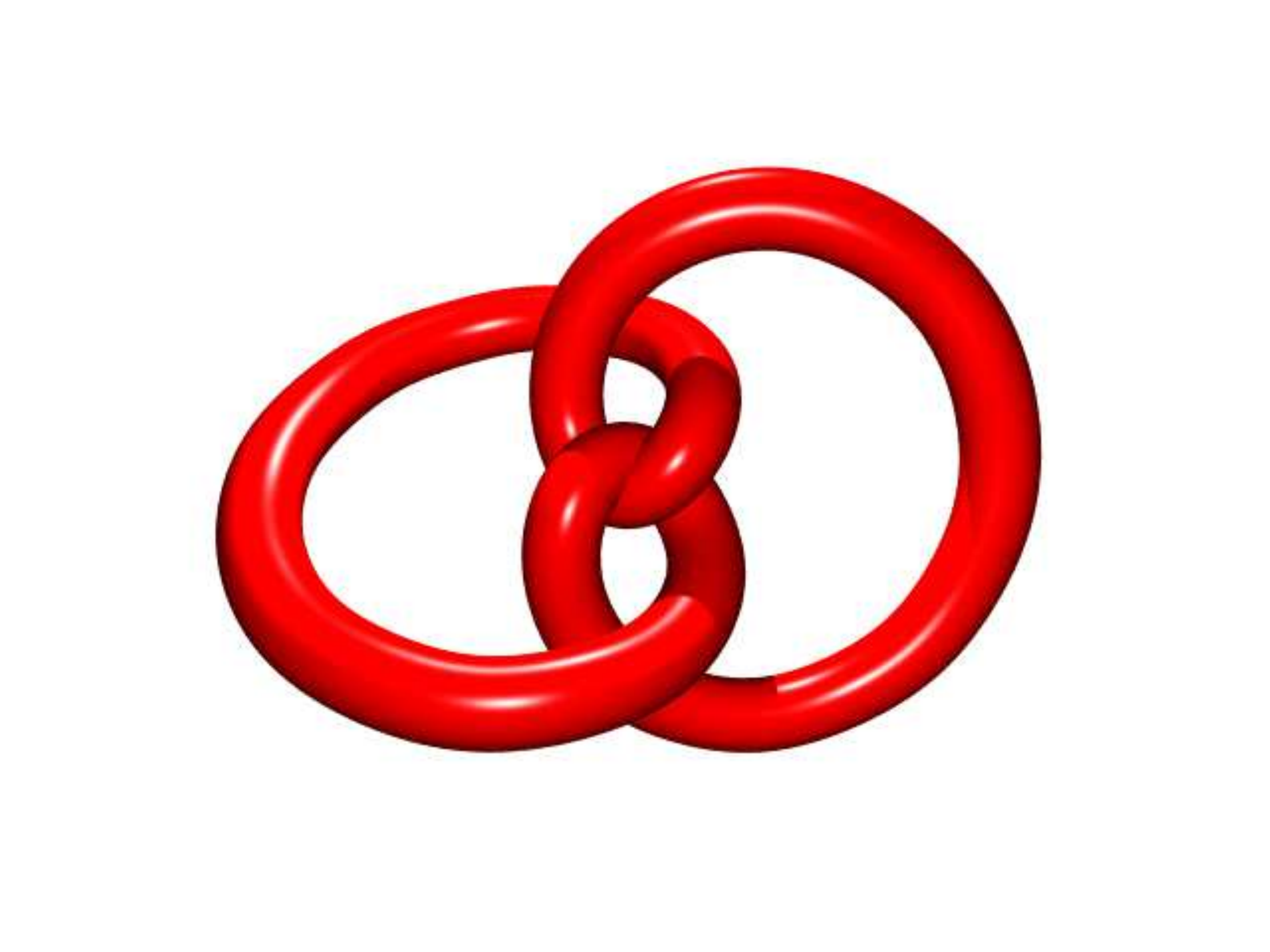} & \includegraphics[width=0.4\textwidth,keepaspectratio]{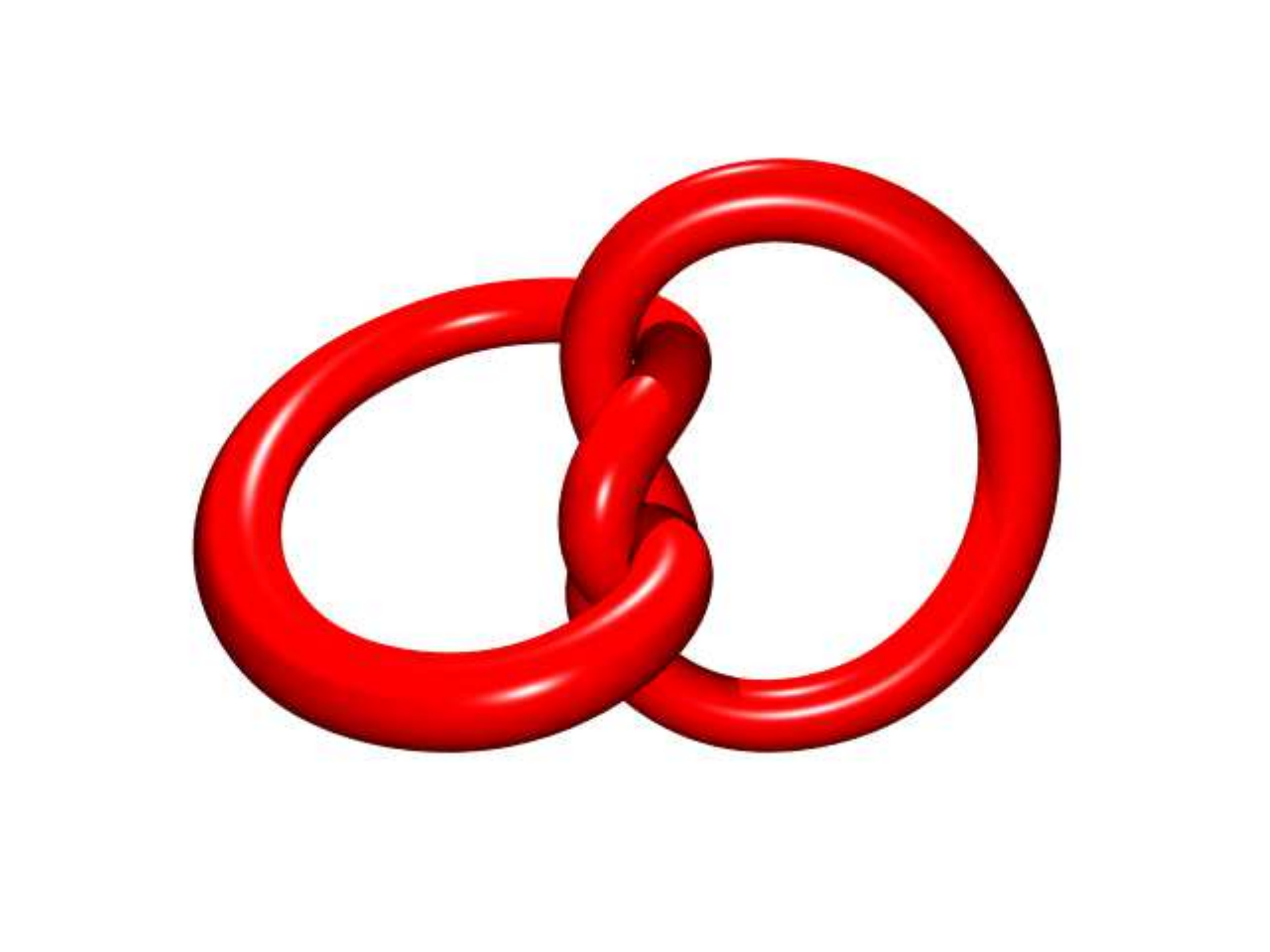} \\
178000/300000 & 300000/300000 \\
$\Le(\gamma)\approx 21.50917$ & $\Le(\gamma)\approx 22.37257$ \\
$\E_p(\gamma)\approx 19.88296$ & $\E_p(\gamma)\approx 19.6115$ \\
$\tau=49.36716$ & $\tau=168.33763$ \\
\includegraphics[width=0.4\textwidth,keepaspectratio]{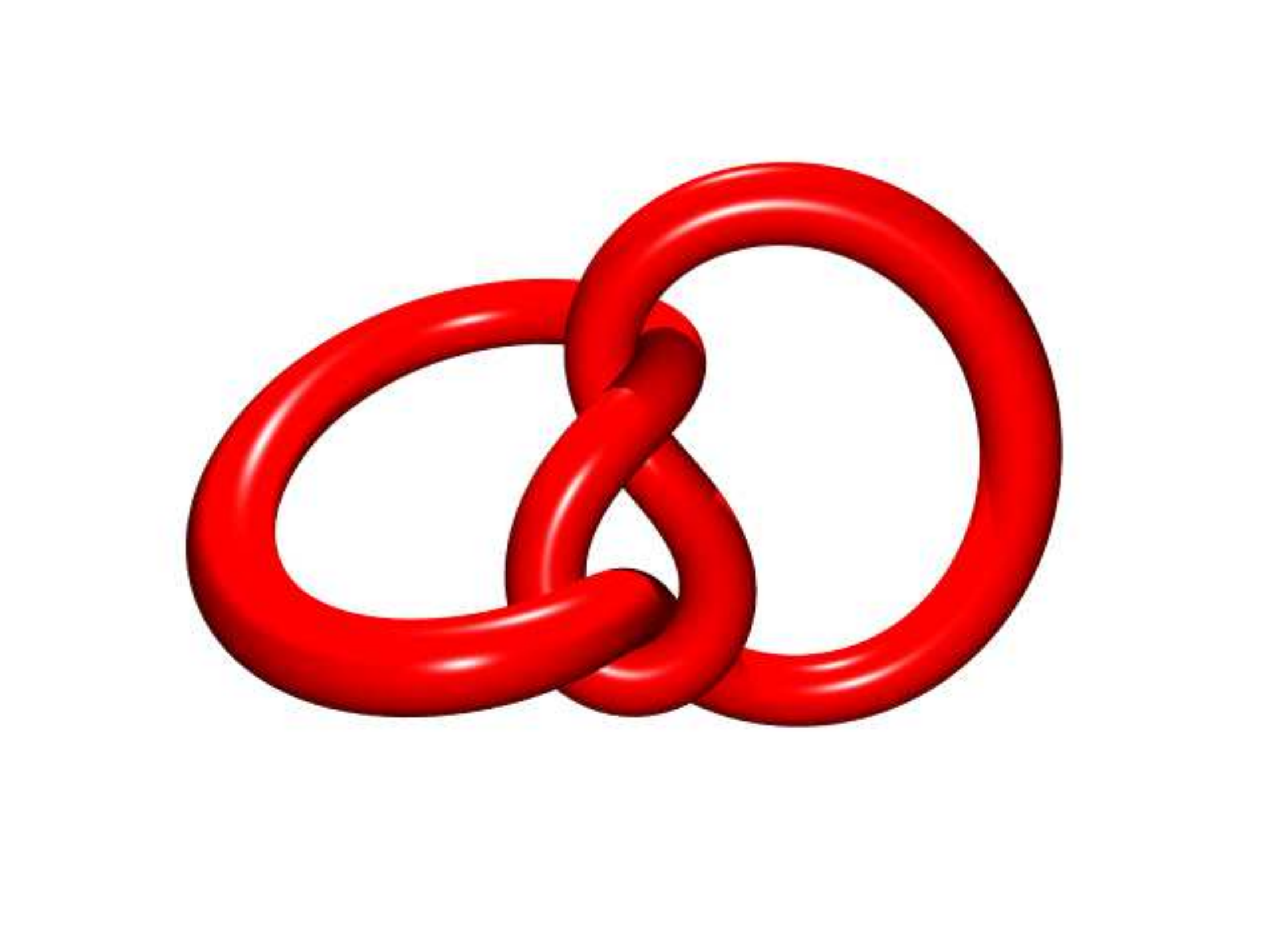} & \includegraphics[width=0.4\textwidth,keepaspectratio]{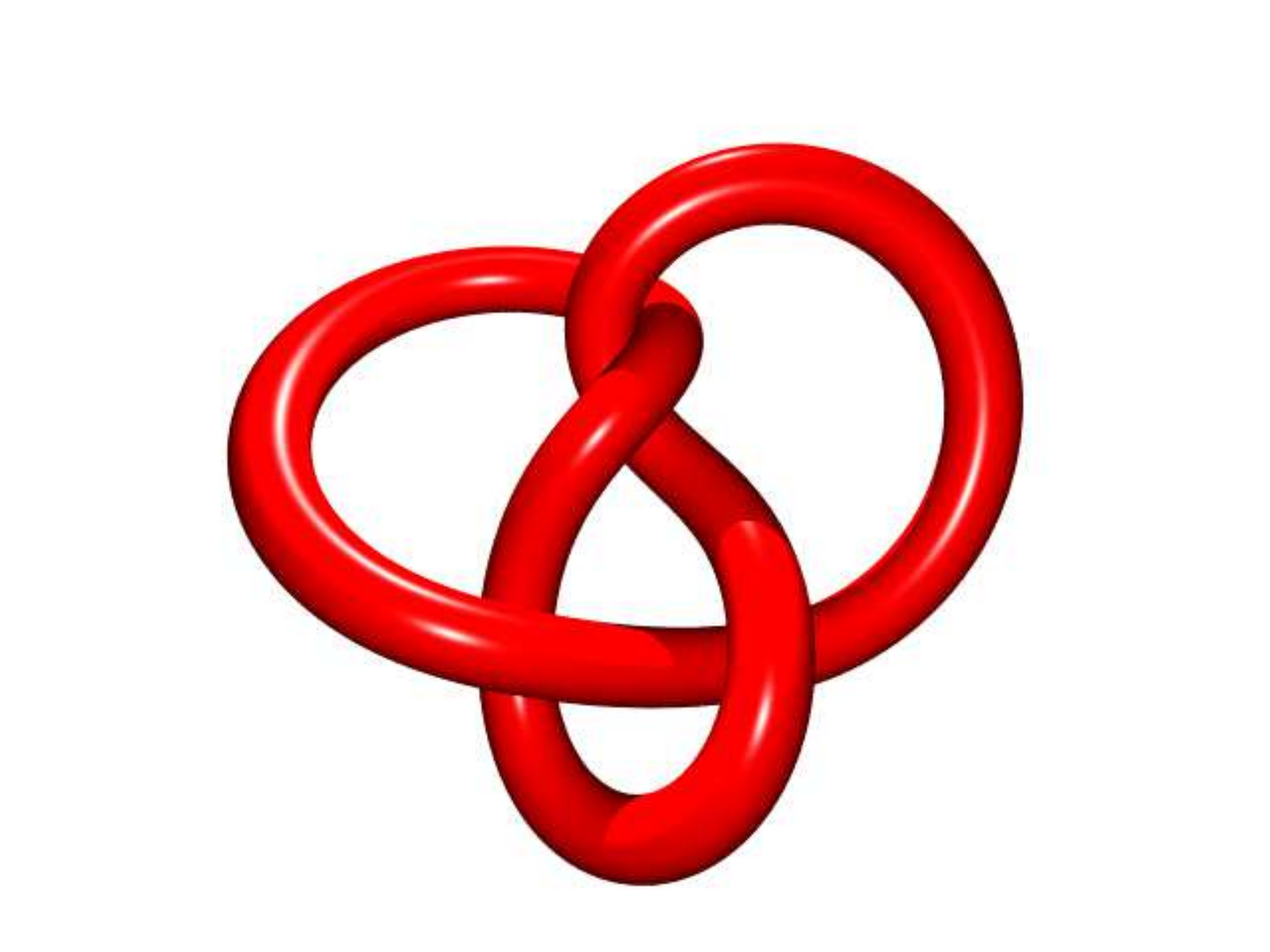}
\end{tabular}
\end{scriptsize}
\end{center}
\caption{A figure-eight -- $p=3.5$}
\end{figure}

We continue with the same starting configuration but from another point of view. Now we use the flow with redistribution for $p=50$.
\begin{figure}[H]
\begin{center}
\begin{scriptsize}
\begin{tabular}{ccc}
0/100000 & 600/100000 & 12000/100000 \\
$\Le(\gamma)\approx 31.68443$ & $\Le(\gamma)\approx 24.68854$ & $\Le(\gamma)\approx 26.32372$ \\
$\E_p(\gamma)\approx 130.00443$ & $\E_p(\gamma)\approx 51.91226$ & $\E_p(\gamma)\approx 48.33161$ \\
$\tau=0.0$ & $\tau=0.33164$ & $\tau=8.93904$ \\
\includegraphics[width=0.33\textwidth,keepaspectratio]{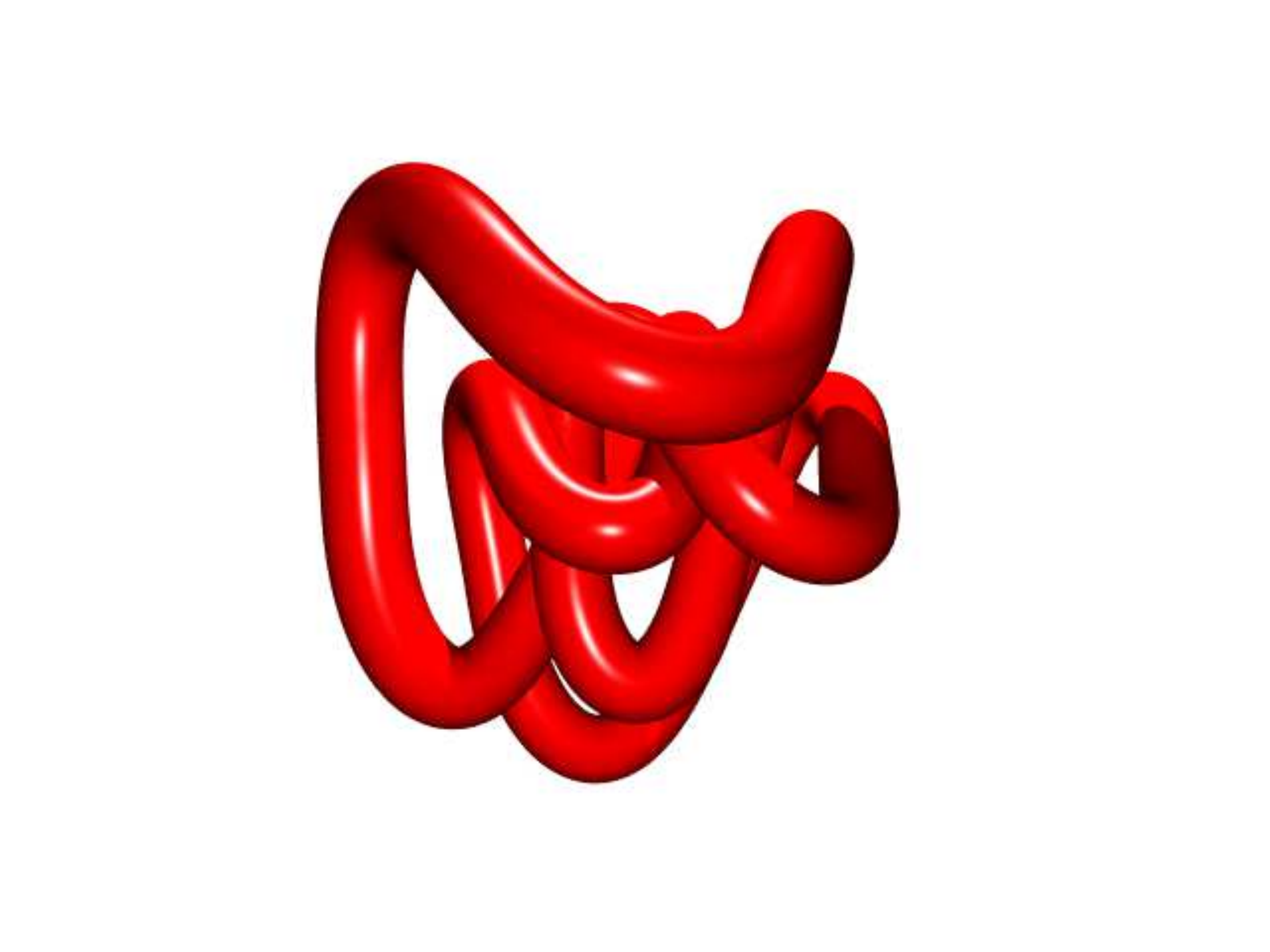} & \includegraphics[width=0.33\textwidth,keepaspectratio]{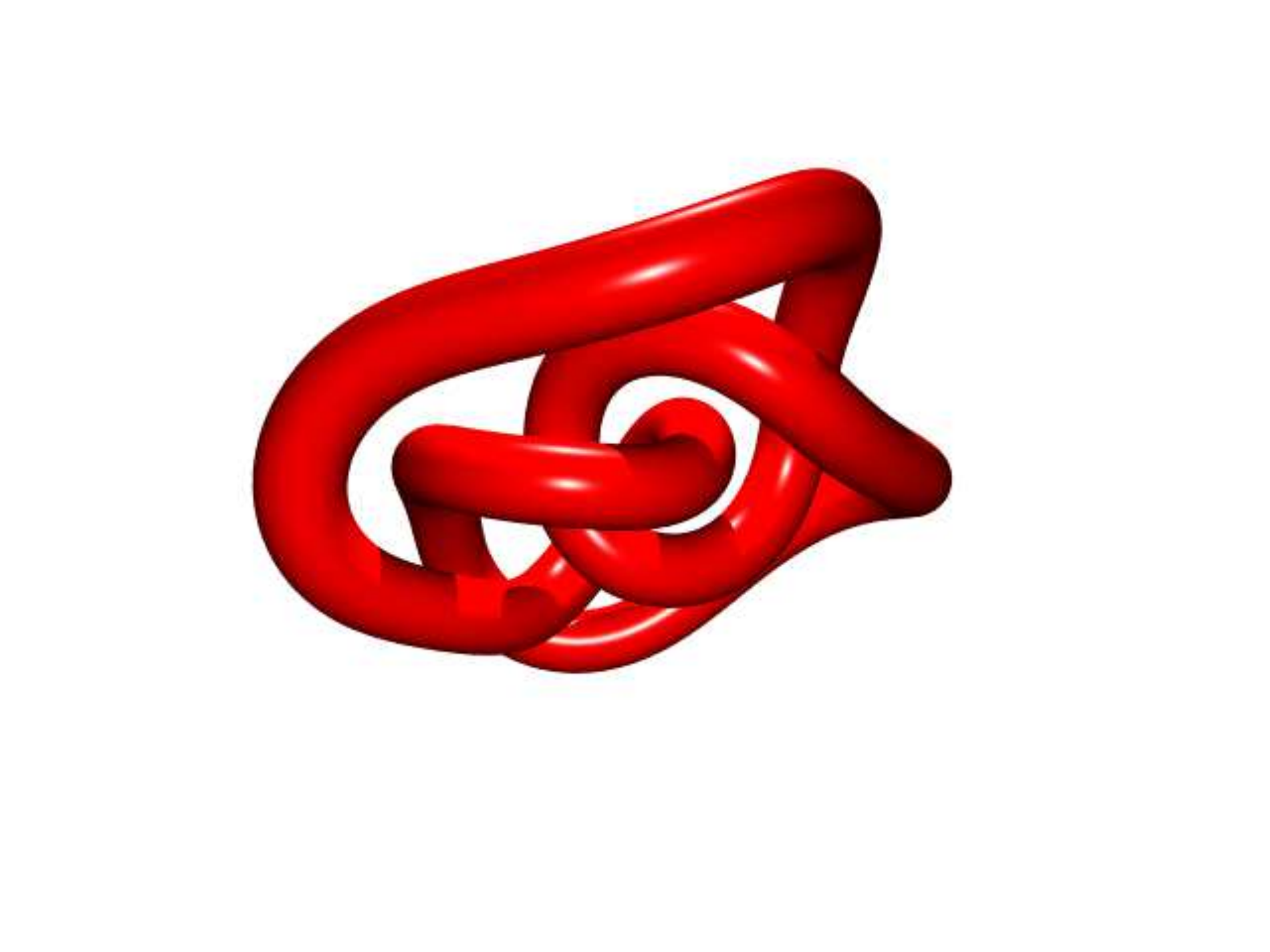} & \includegraphics[width=0.33\textwidth,keepaspectratio]{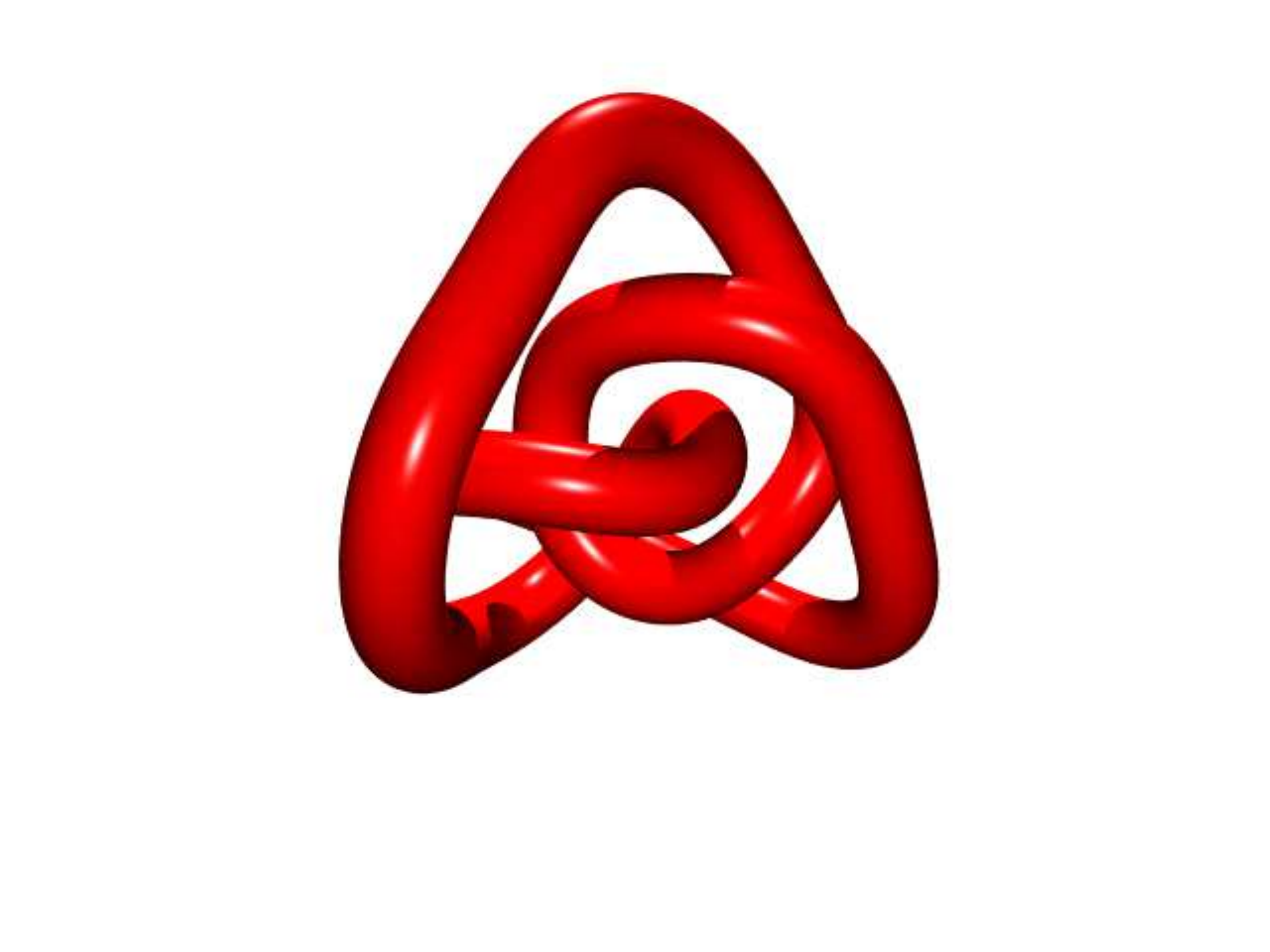} \\
27000/100000 & 29000/100000 & 30000/100000 \\
$\Le(\gamma)\approx 27.65309$ & $\Le(\gamma)\approx 27.07511$ & $\Le(\gamma)\approx 26.43185$ \\
$\E_p(\gamma)\approx 45.54672$ & $\E_p(\gamma)\approx 44.12169$ & $\E_p(\gamma)\approx 41.66007$ \\
$\tau=23.22984$ & $\tau=25.38418$ & $\tau=26.48759$ \\
\includegraphics[width=0.33\textwidth,keepaspectratio]{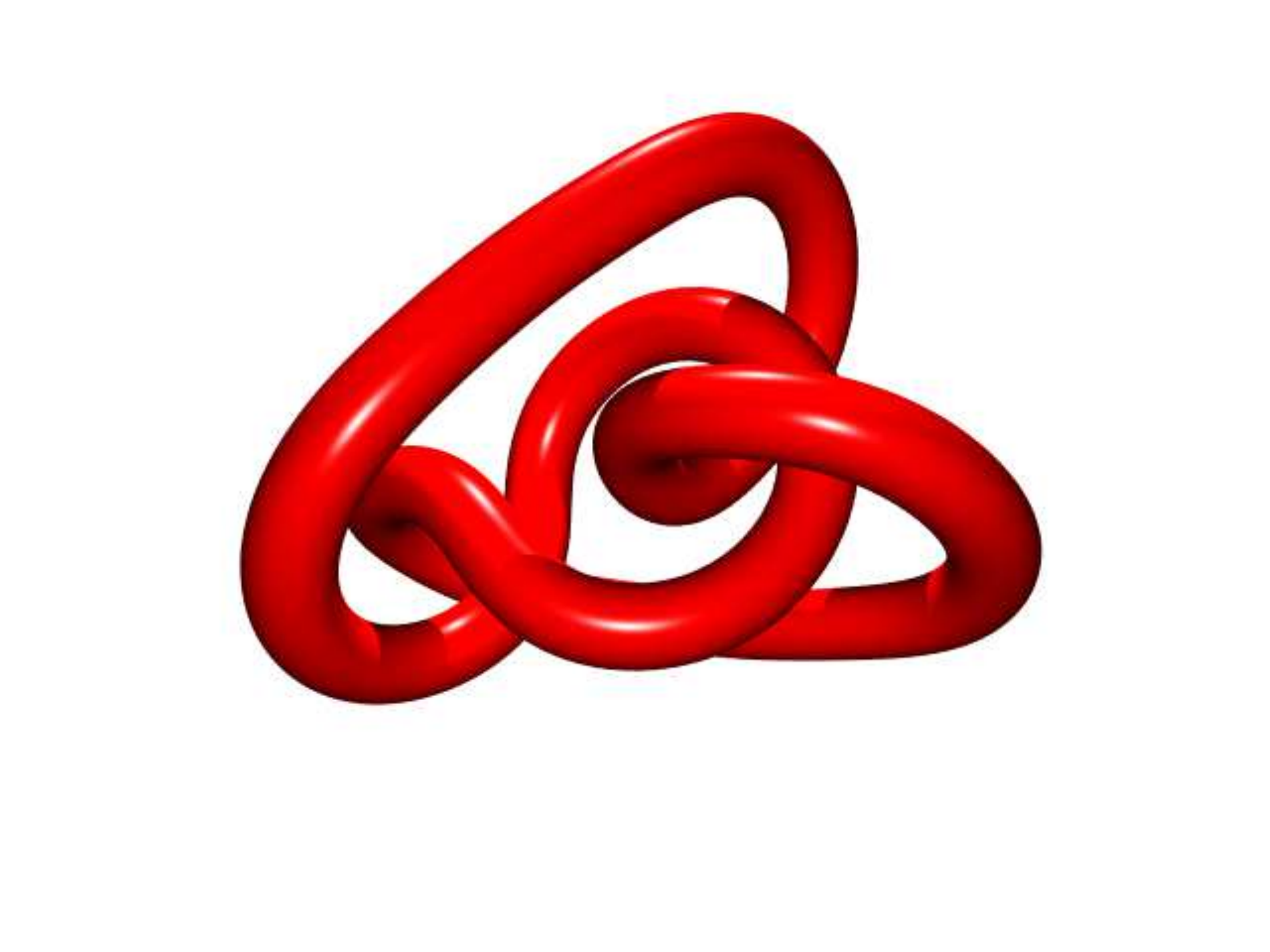} & \includegraphics[width=0.33\textwidth,keepaspectratio]{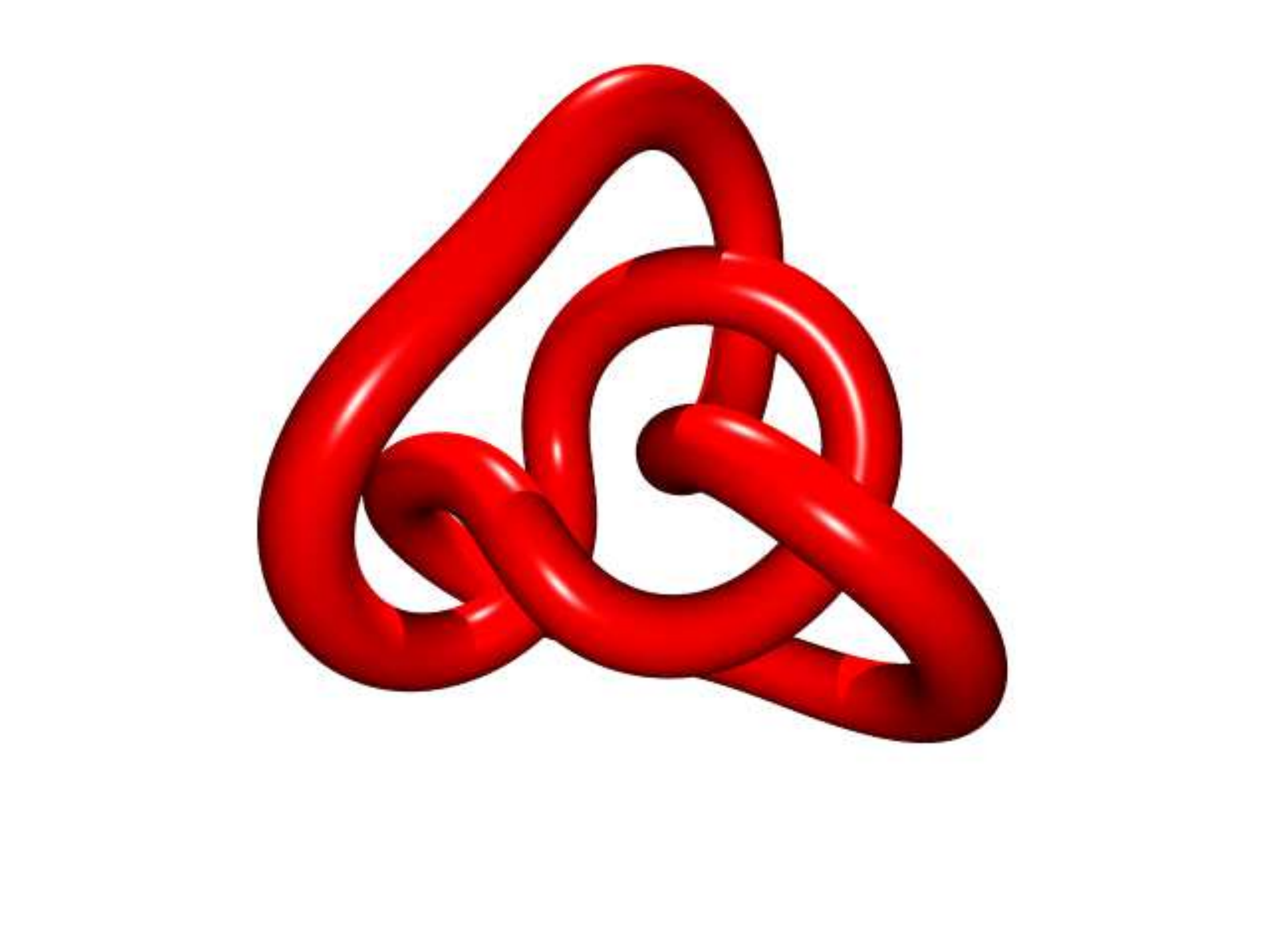} & \includegraphics[width=0.33\textwidth,keepaspectratio]{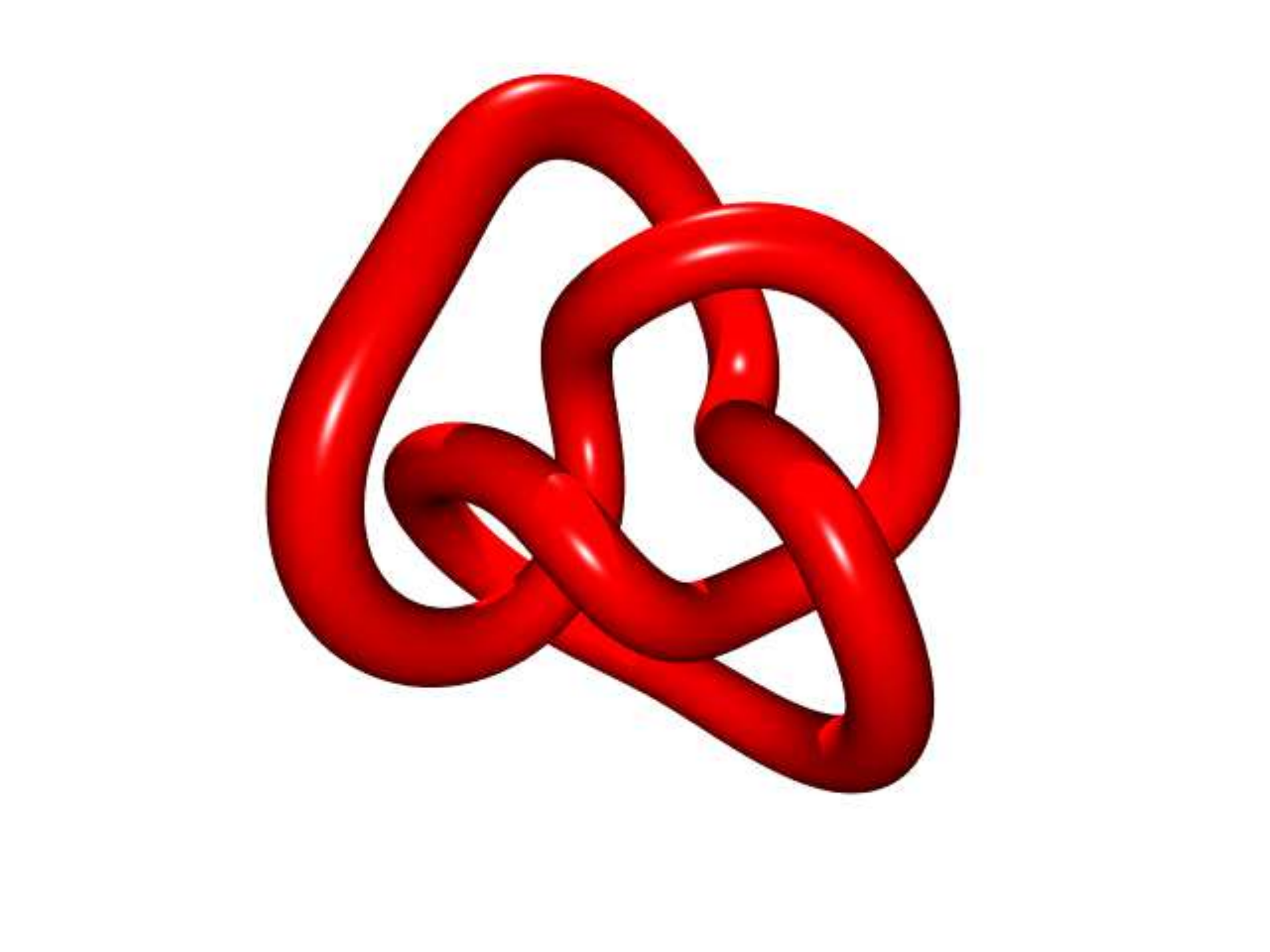} \\
40000/100000 & 53000/100000 & 100000/100000 \\
$\Le(\gamma)\approx 26.12792$ & $\Le(\gamma)\approx 27.4093$ & $\Le(\gamma)\approx 29.29693$ \\
$\E_p(\gamma)\approx 39.96641$ & $\E_p(\gamma)\approx 39.24609$ & $\E_p(\gamma)\approx 37.49673$ \\
$\tau=38.31868$ & $\tau=55.16871$ & $\tau=136.85656$ \\
\includegraphics[width=0.33\textwidth,keepaspectratio]{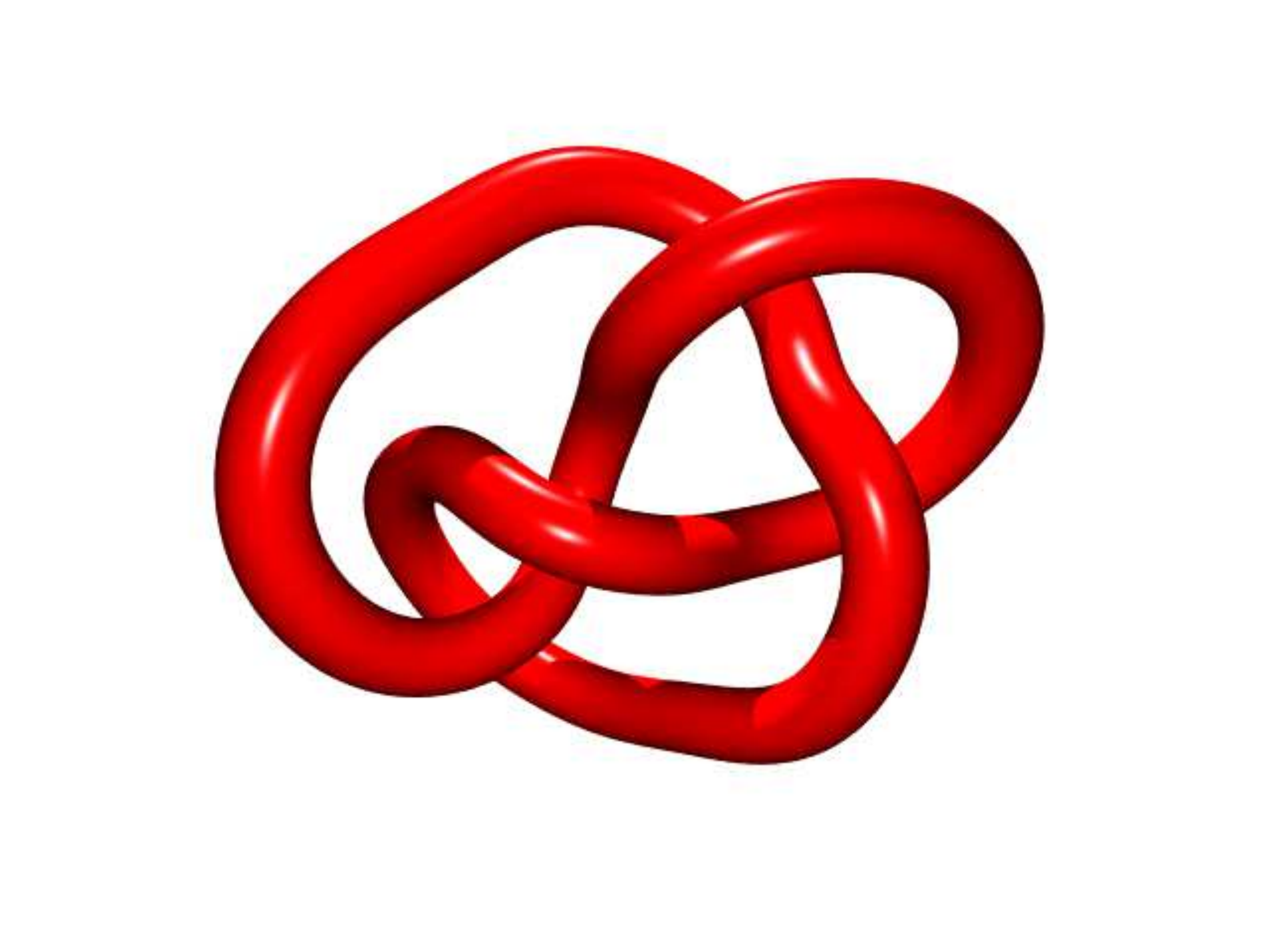} & \includegraphics[width=0.33\textwidth,keepaspectratio]{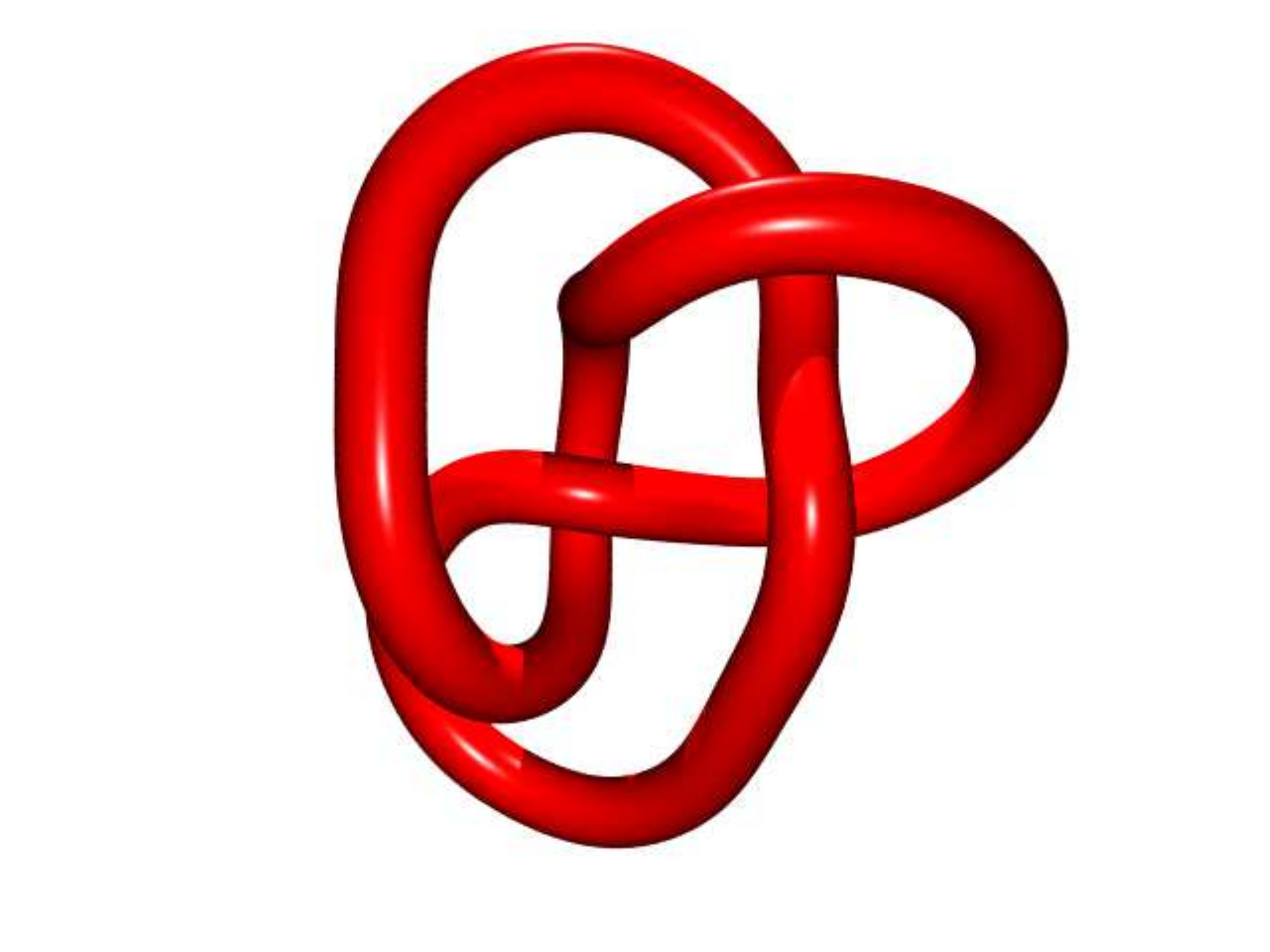} & \includegraphics[width=0.33\textwidth,keepaspectratio]{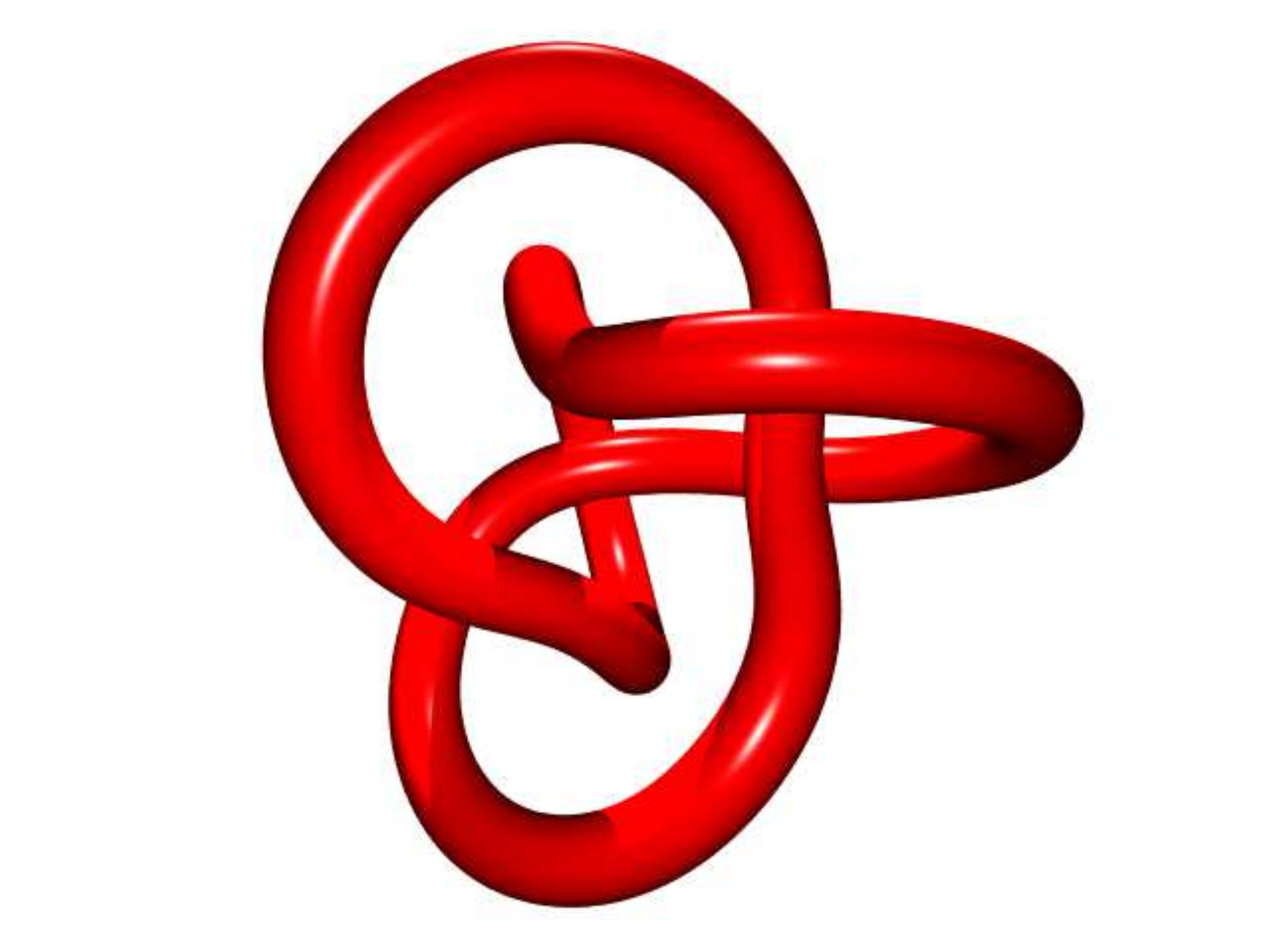}
\end{tabular}
\end{scriptsize}
\end{center}
\caption{A figure-eight -- $p=50.0$}
\end{figure}

Here the energy values stabilises in the end. Although these two representative of this knot class differ strongly at first sight, the shapes of the last configurations of the flows with respect to them are the same for $p=3.5$ and respectively for $p=50$.

\section{$5_1$ knots} \label{knot51}
This time we start with an example of the flow without redistribution for $p=3$.
\begin{figure}[H]
\begin{center}
\begin{scriptsize}
\begin{tabular}{ccc}
0/400000 & 1000/400000 & 100000/400000 \\
$\Le(\gamma)\approx 89.20002$ & $\Le(\gamma)\approx 86.18496$ & $\Le(\gamma)\approx 85.85831$ \\
$\E_p(\gamma)\approx 19.67839$ & $\E_p(\gamma)\approx 19.36133$ & $\E_p(\gamma)\approx 19.33206$ \\
$\tau=0.0$ & $\tau=97.56618$ & $\tau=9573.01393$ \\
\includegraphics[width=0.33\textwidth,keepaspectratio]{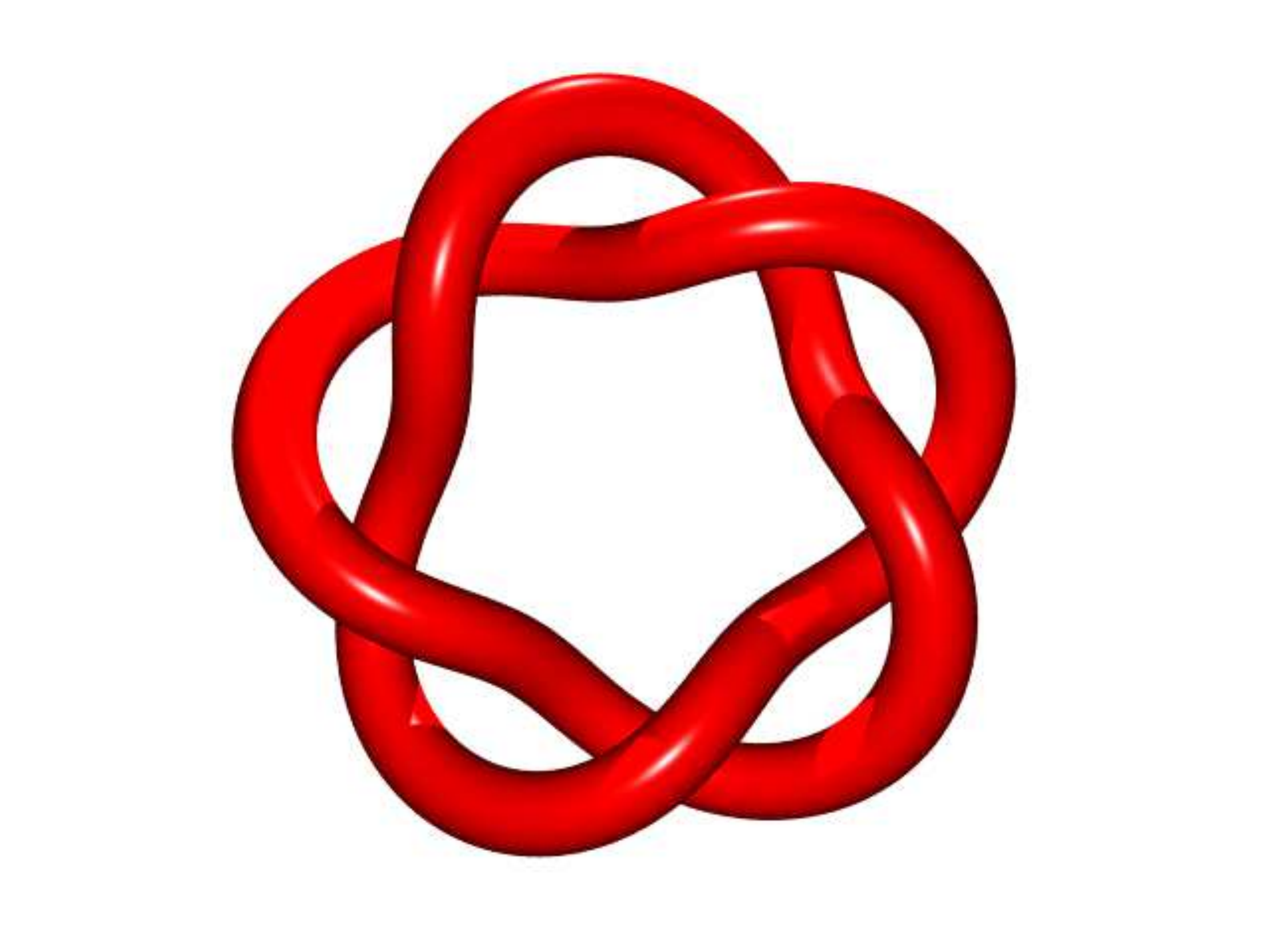} & \includegraphics[width=0.33\textwidth,keepaspectratio]{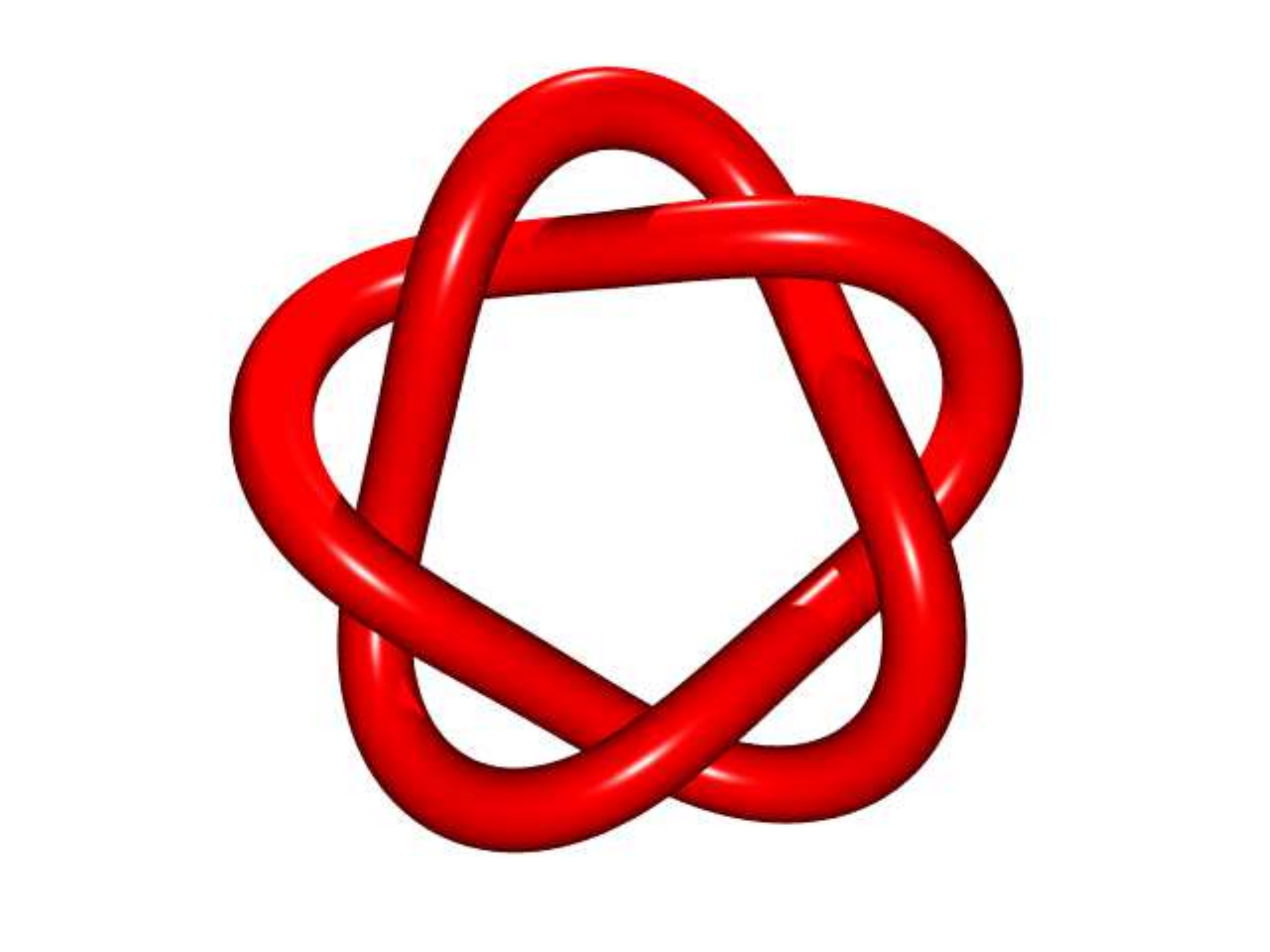} & \includegraphics[width=0.33\textwidth,keepaspectratio]{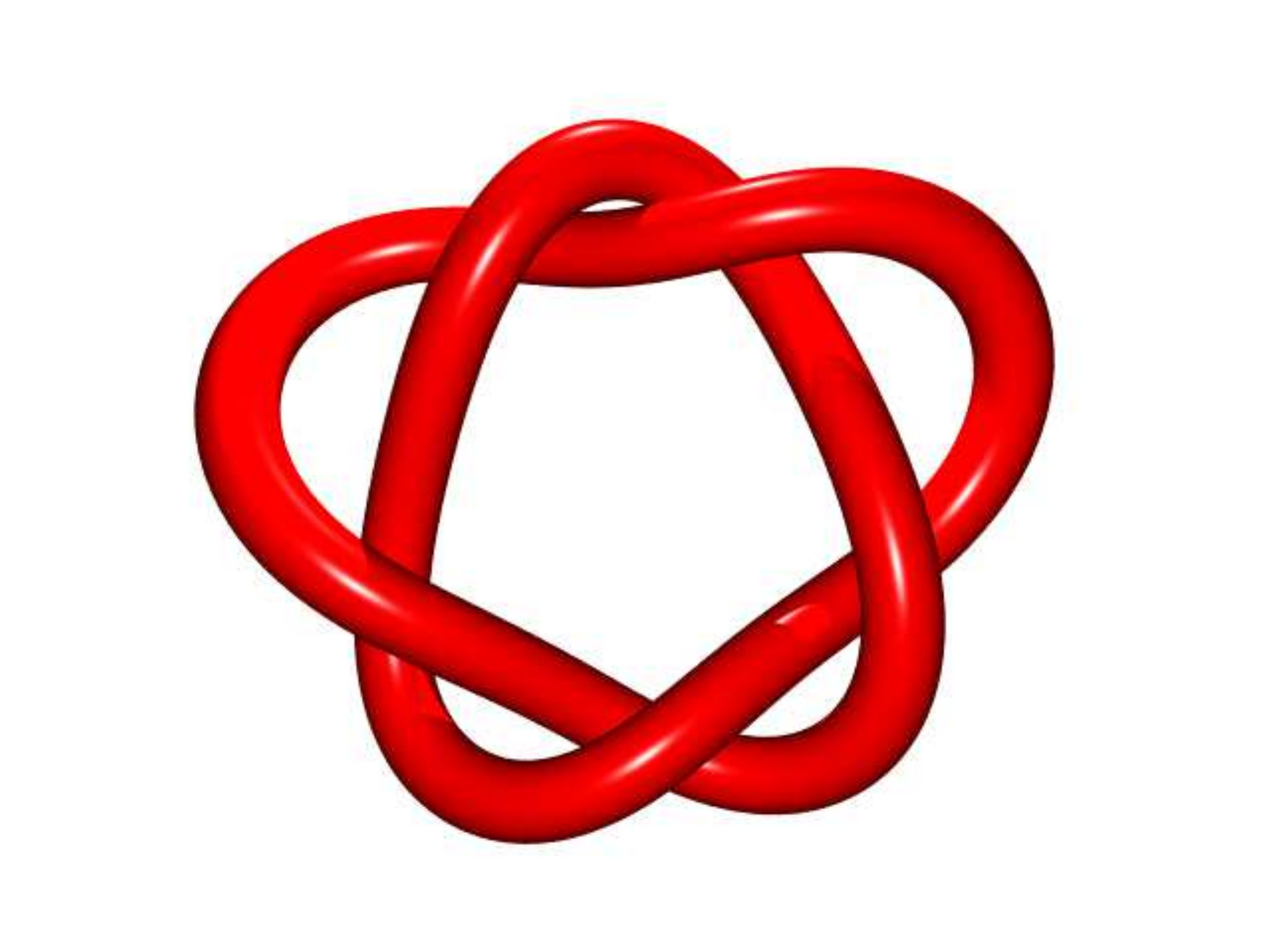} \\
107000/400000 & 125000/400000 & 400000/400000 \\
$\Le(\gamma)\approx 84.20402$ & $\Le(\gamma)\approx 79.685$ & $\Le(\gamma)\approx 75.92362$ \\
$\E_p(\gamma)\approx 19.21474$ & $\E_p(\gamma)\approx 18.98026$ & $\E_p(\gamma)\approx 18.8269$ \\
$\tau=10072.72315$ & $\tau=10477.91399$ & $\tau=10799.57702$ \\
\includegraphics[width=0.33\textwidth,keepaspectratio]{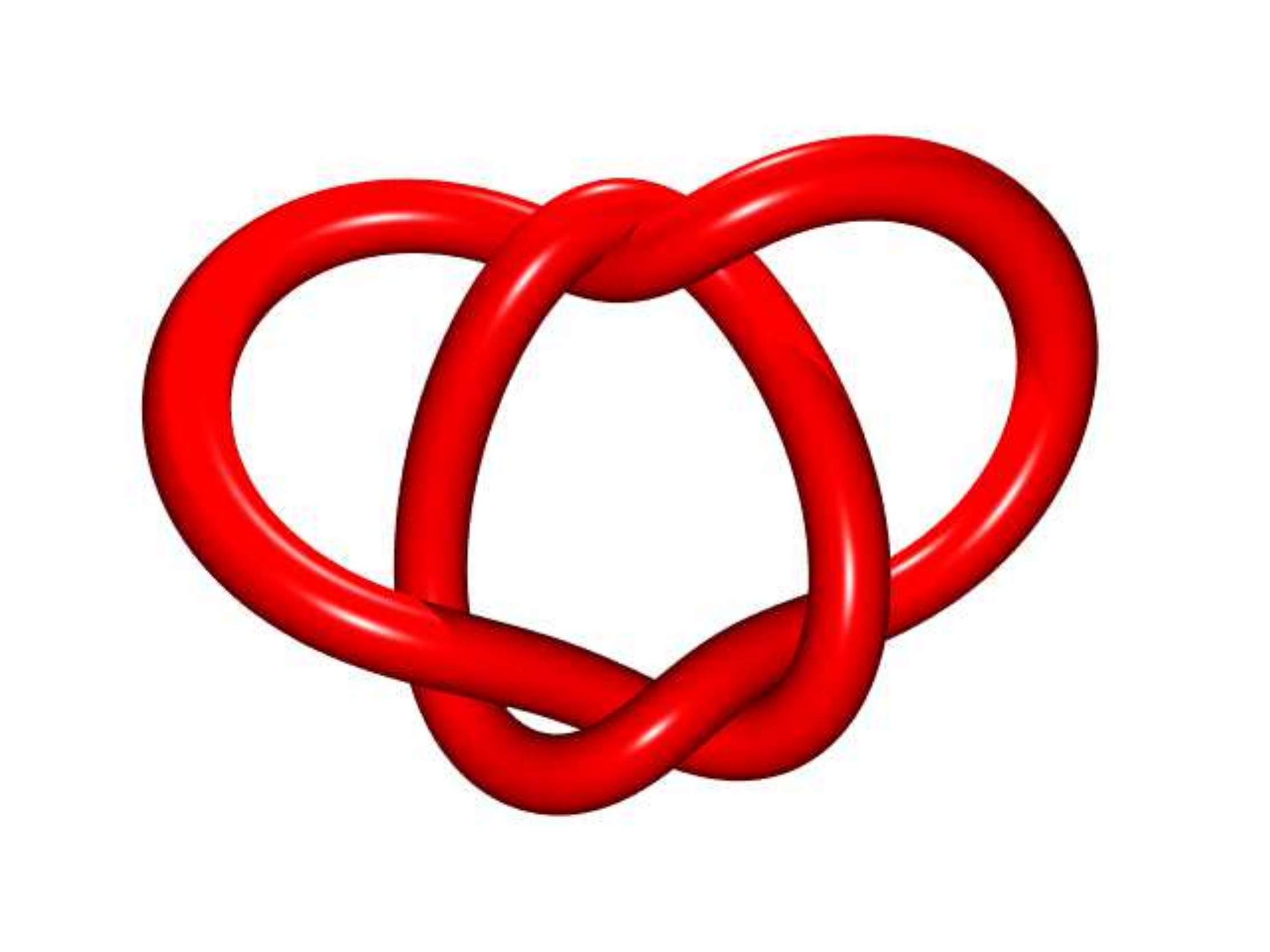} & \includegraphics[width=0.33\textwidth,keepaspectratio]{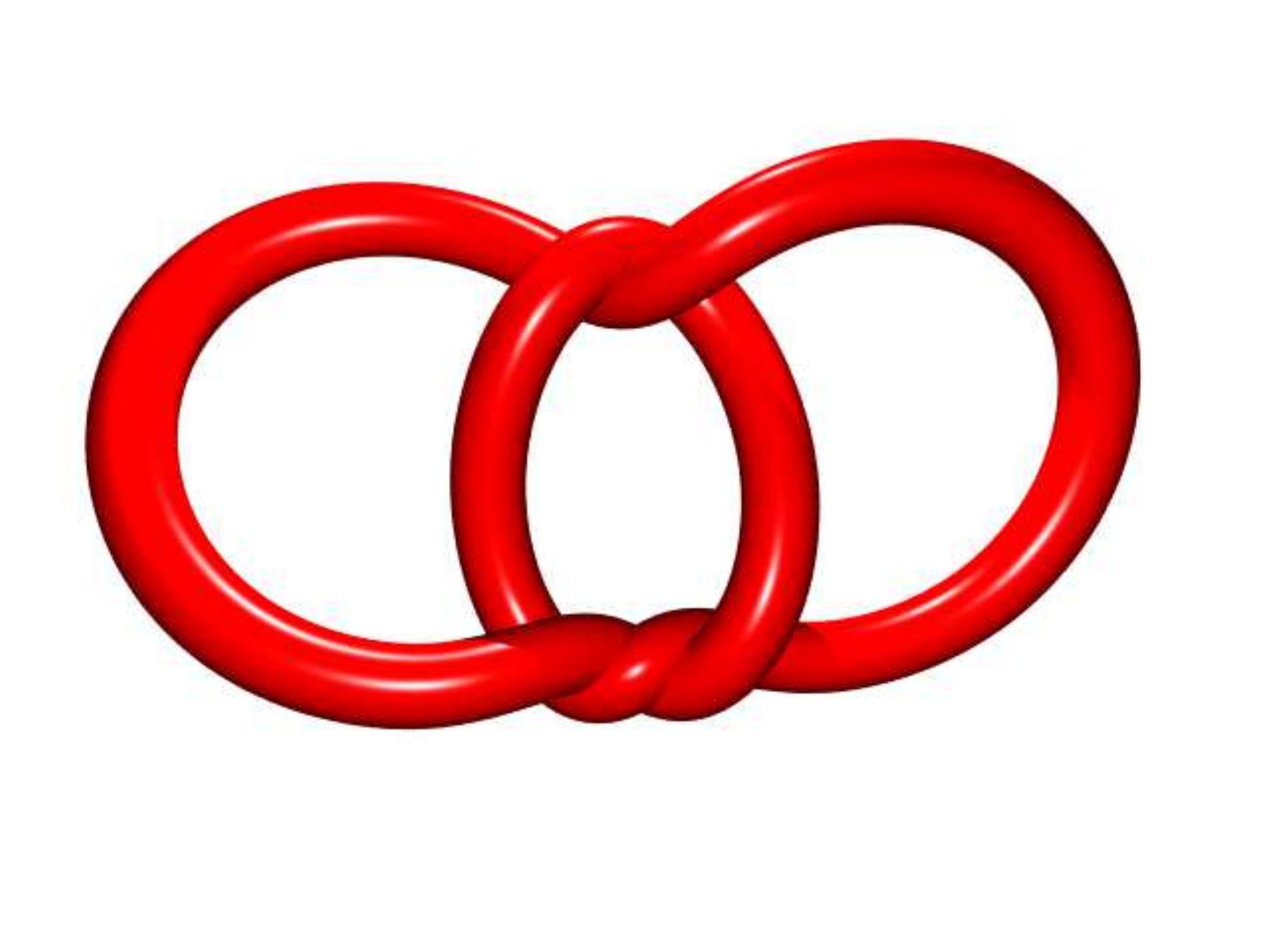} & \includegraphics[width=0.33\textwidth,keepaspectratio]{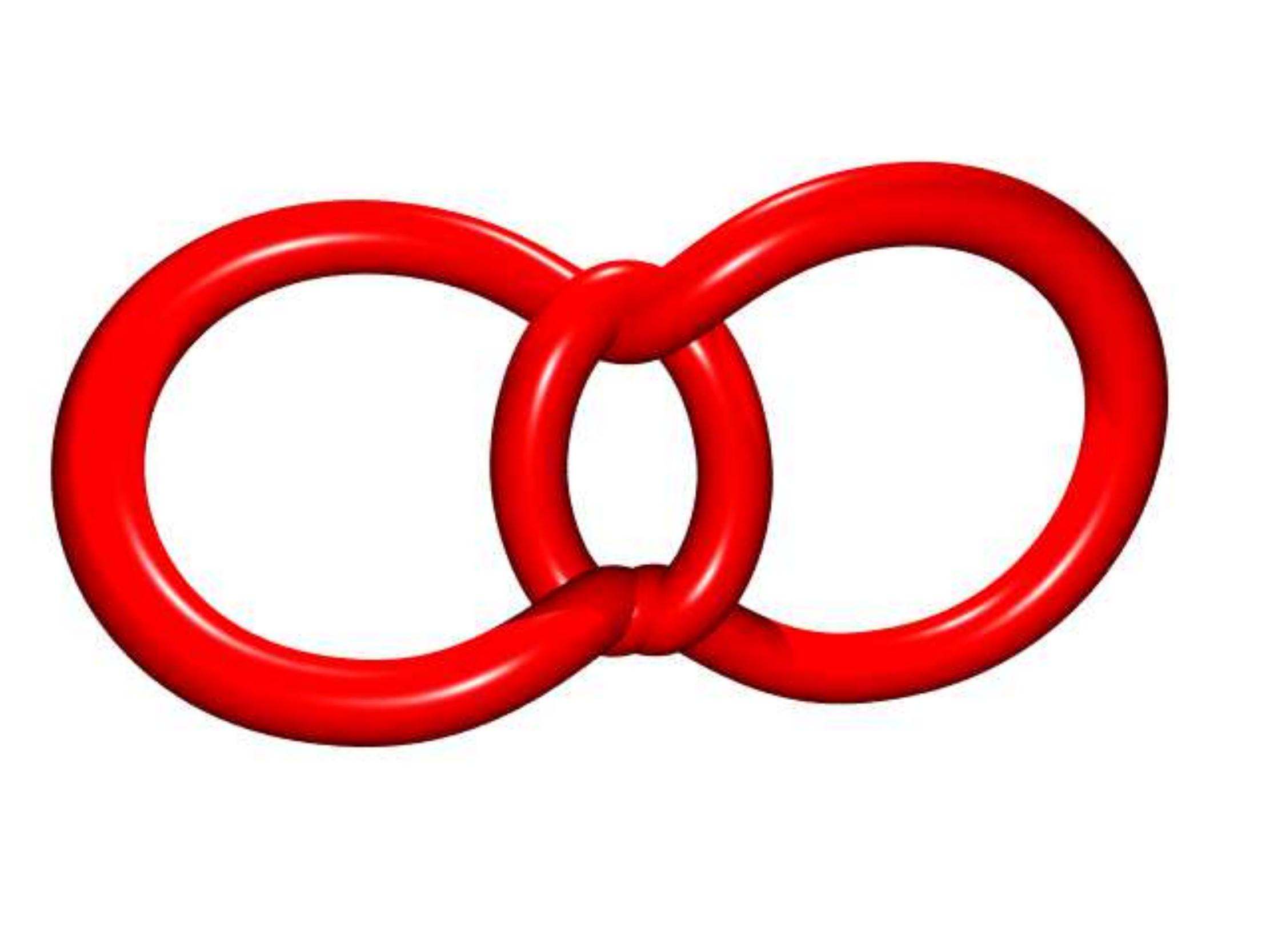}
\end{tabular}
\end{scriptsize}
\end{center}
\caption{A $5_1$ knot -- $p=3.0$ without redistribution ($\tau_{\text{max}}=0.1$, $\eps=0.05$)}
\end{figure}

\newpage
Now we consider the same configuration for the flow with redistribution.
\begin{figure}[H]
\begin{center}
\begin{scriptsize}
\begin{tabular}{ccc}
0/400000 & 1000/400000 & 400000/400000 \\
$\Le(\gamma)\approx 89.19963$ & $\Le(\gamma)\approx 86.1965$ & $\Le(\gamma)\approx 115.71134$ \\
$\E_p(\gamma)\approx 19.67818$ & $\E_p(\gamma)\approx 19.36134$ & $\E_p(\gamma)\approx 18.99994$ \\
$\tau=0.0$ & $\tau=87.12779$ & $\tau=17072.00582$ \\
\includegraphics[width=0.33\textwidth,keepaspectratio]{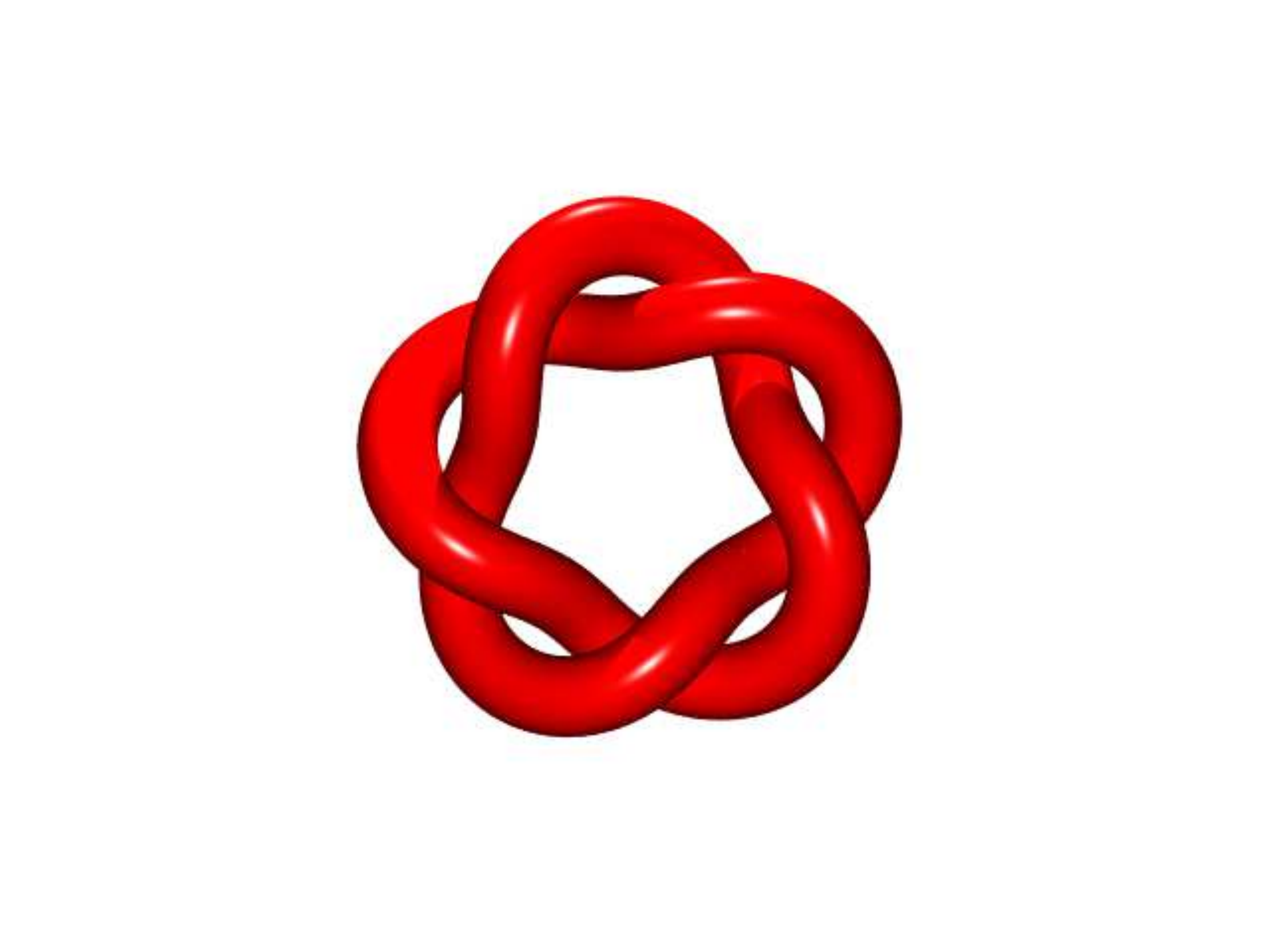} & \includegraphics[width=0.33\textwidth,keepaspectratio]{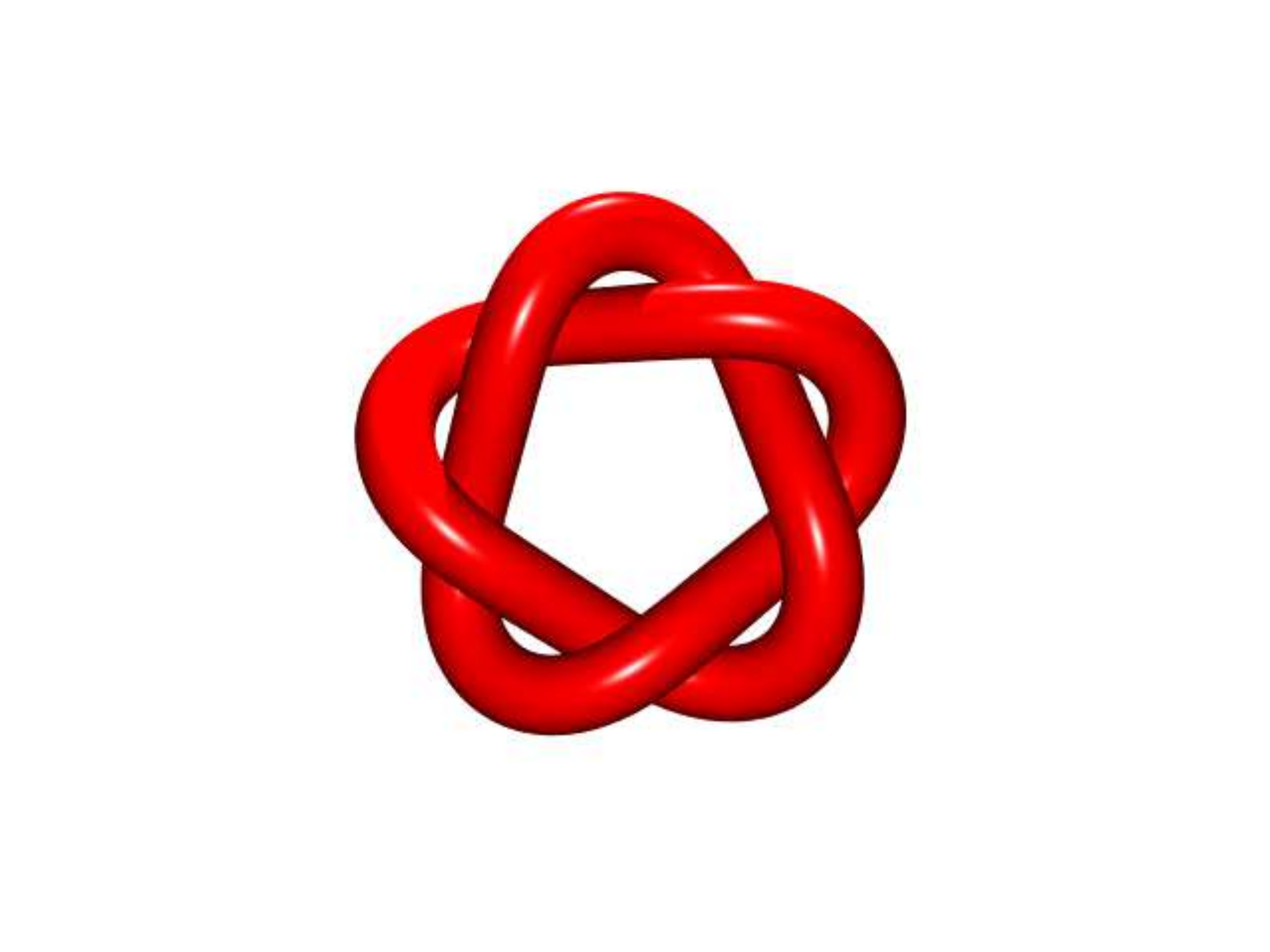} & \includegraphics[width=0.33\textwidth,keepaspectratio]{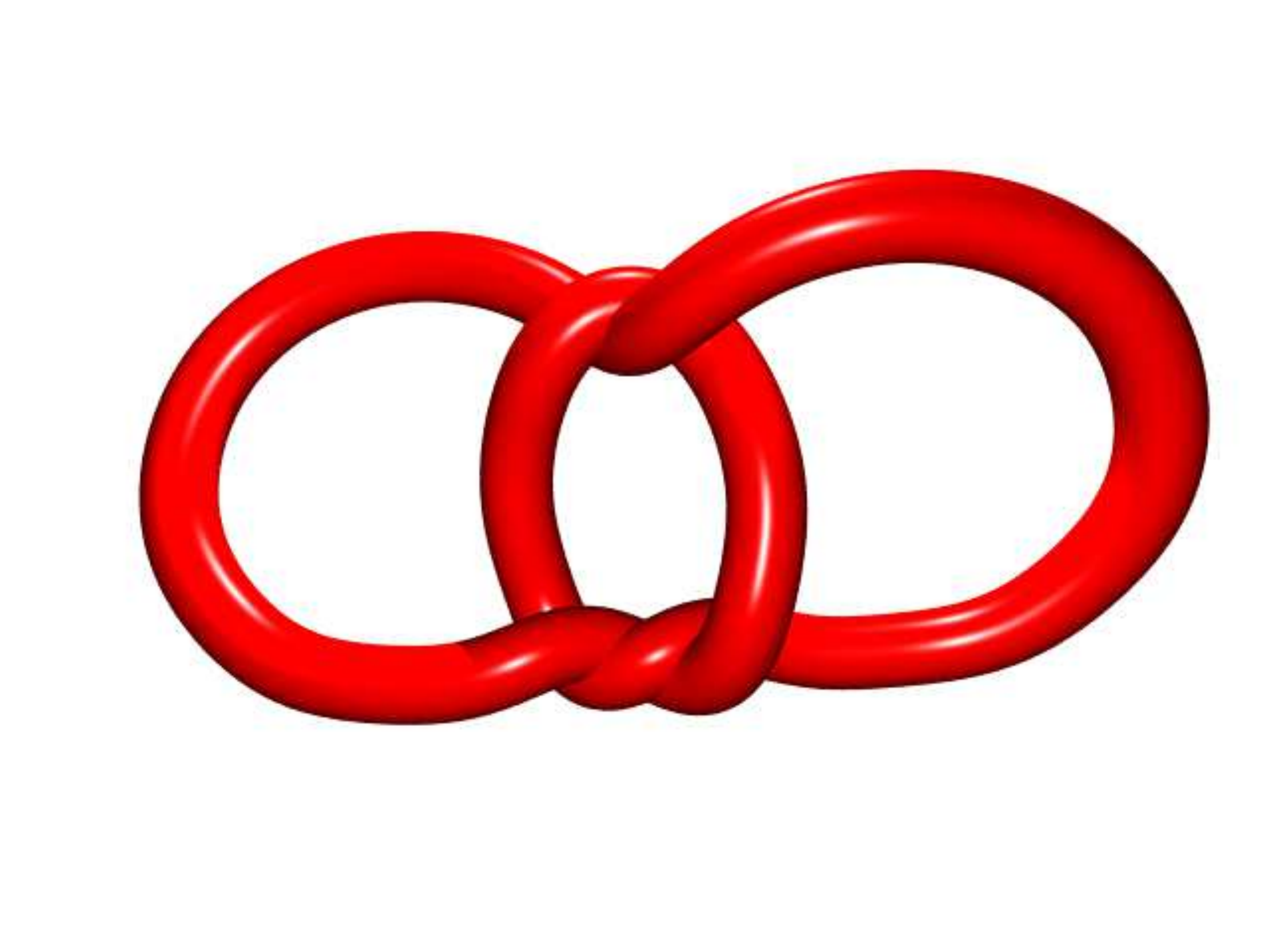}
\end{tabular}
\end{scriptsize}
\end{center}
\caption{A $5_1$ knot -- $p=3.0$ ($\tau_{\text{max}}=0.1$, $\eps=0.05$)}
\end{figure}

We see that after $400.000$ steps in this case for $p=3$ the knot stays in its knot class. Moreover, we obtain that the configuration with redistribution is larger but not that tight.

Again the flow with redistribution for $p=3.5$ stays in the symmetric configuration and the energy value stabilises in the end.
\begin{figure}[H]
\begin{center}
\begin{scriptsize}
\begin{tabular}{ccc}
0/380000 & 1000/380000 & 300000/380000 \\
$\Le(\gamma)\approx 89.19963$ & $\Le(\gamma)\approx 88.89754$ & $\Le(\gamma)\approx 89.05899$ \\
$\E_p(\gamma)\approx 20.93175$ & $\E_p(\gamma)\approx 20.73642$ & $\E_p(\gamma)\approx 20.73638$ \\
$\tau=0.0$ & $\tau=89.40621$ & $\tau=27096.27375$ \\
\includegraphics[width=0.33\textwidth,keepaspectratio]{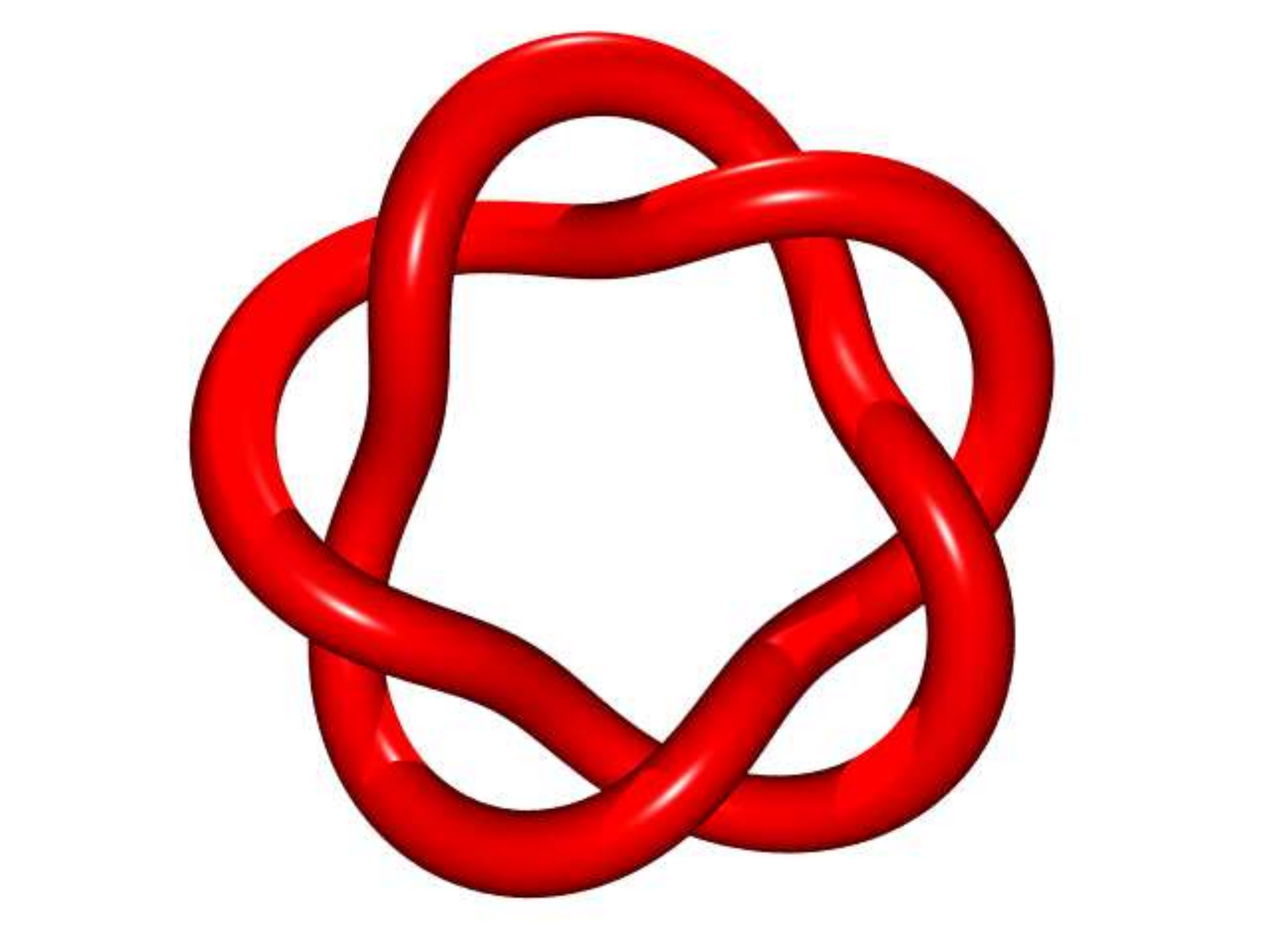} & \includegraphics[width=0.33\textwidth,keepaspectratio]{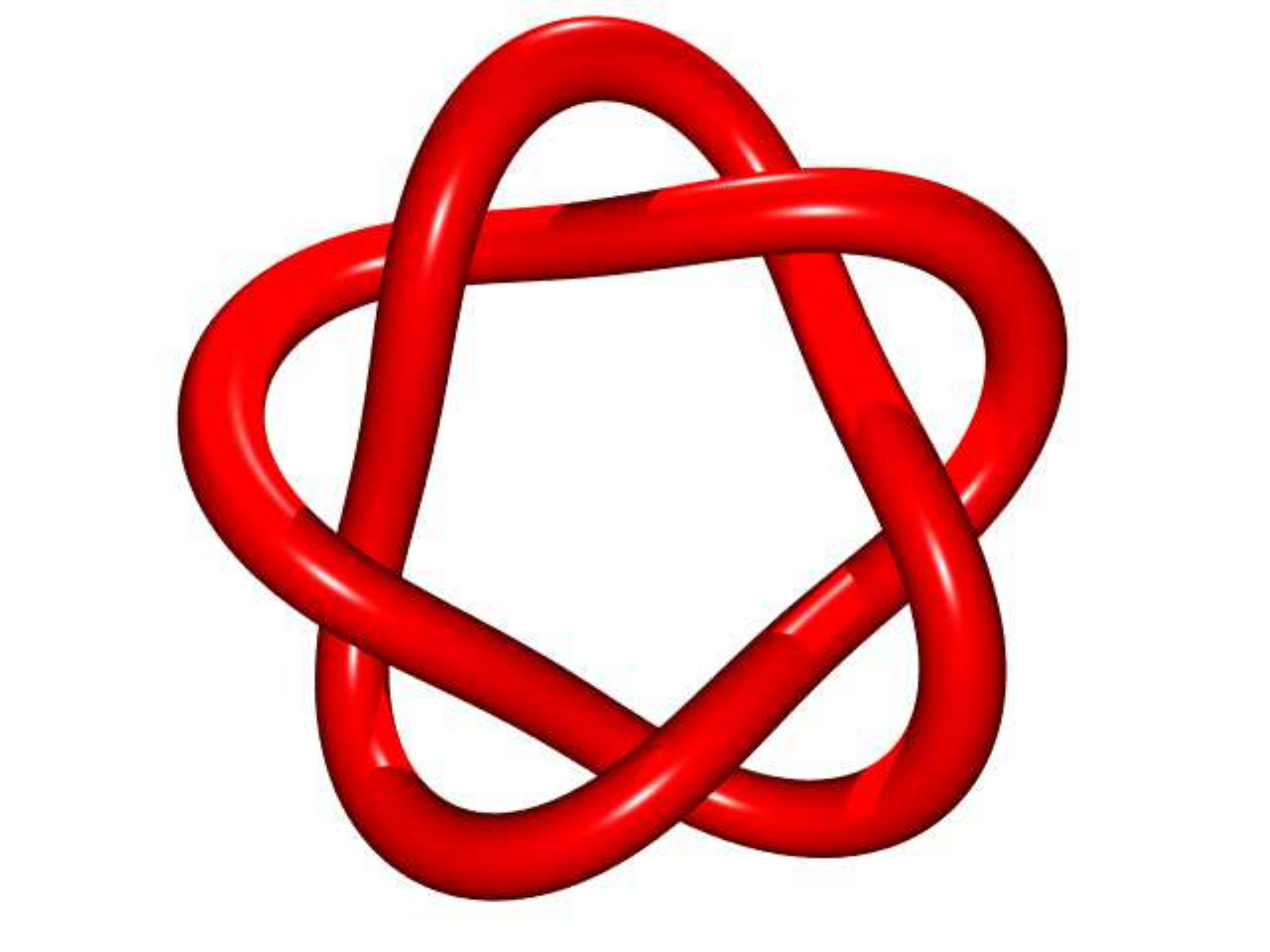} & \includegraphics[width=0.33\textwidth,keepaspectratio]{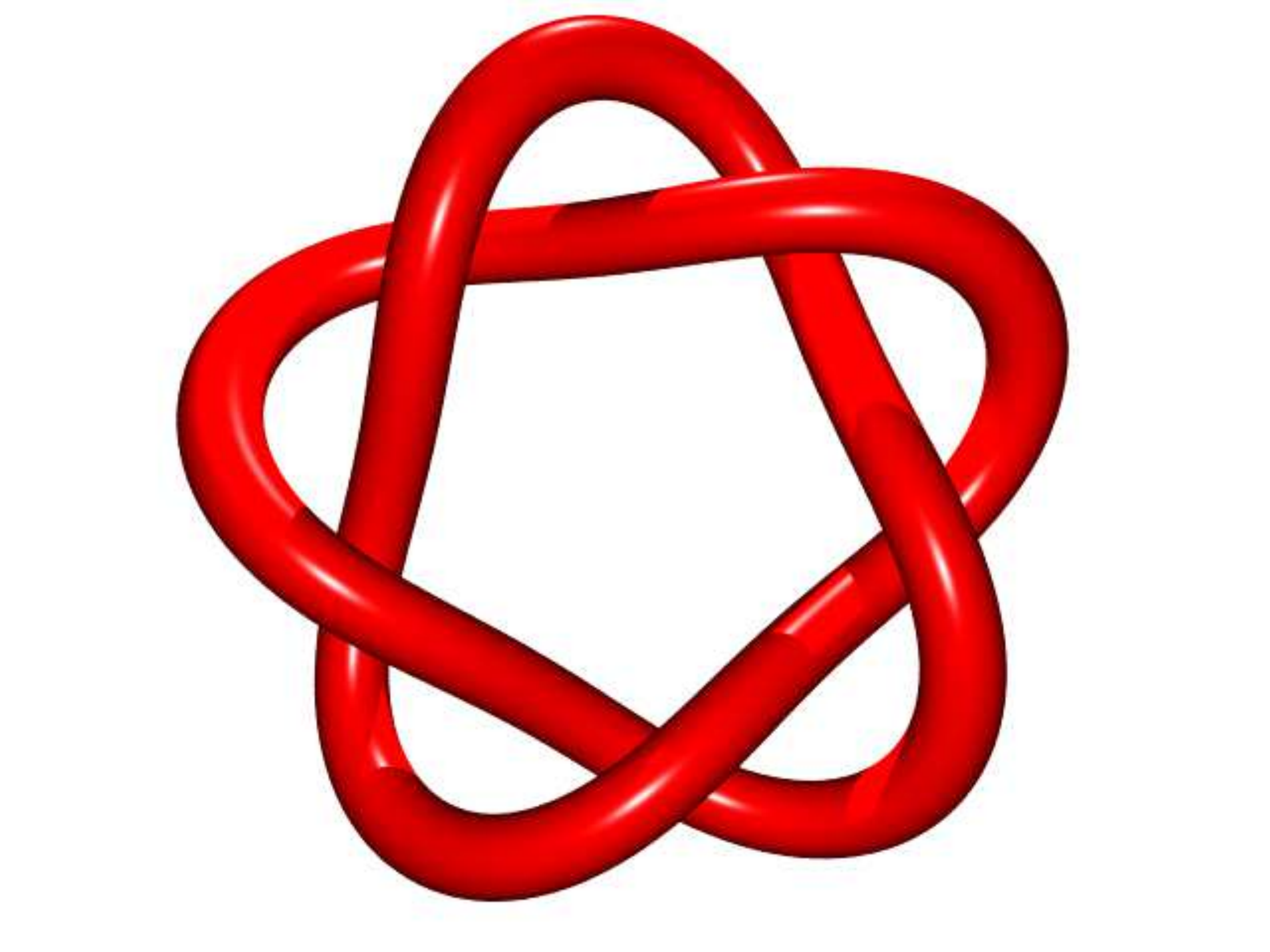}
\end{tabular}
\end{scriptsize}
\end{center}
\caption{A $5_1$ knot -- $p=3.5$}
\end{figure}

Even after $1.000.000$ steps without redistribution the shape is still symmetric and the energy is still $20.73638$.

Also for the case $p=50$ we use the flow without redistribution. However, observe that here the same time step size has been chosen in each step. This is due to the fact that this knot is quit large and therefore, a maximal time step size of $\tau_{\max}=0.01$ is below the computed value of the adaptive time step size algorithm.
\begin{figure}[H]
\begin{center}
\begin{scriptsize}
\begin{tabular}{ccc}
0/300000 & 1000/300000 & 170000/300000 \\
$\Le(\gamma)\approx 89.20002$ & $\Le(\gamma)\approx 98.85687$ & $\Le(\gamma)\approx 98.95859$ \\
$\E_p(\gamma)\approx 51.7359$ & $\E_p(\gamma)\approx 43.05068$ & $\E_p(\gamma)\approx 43.04953$ \\
$\tau=0.0$ & $\tau=10.0$ & $\tau=1700.0$ \\
\includegraphics[width=0.33\textwidth,keepaspectratio]{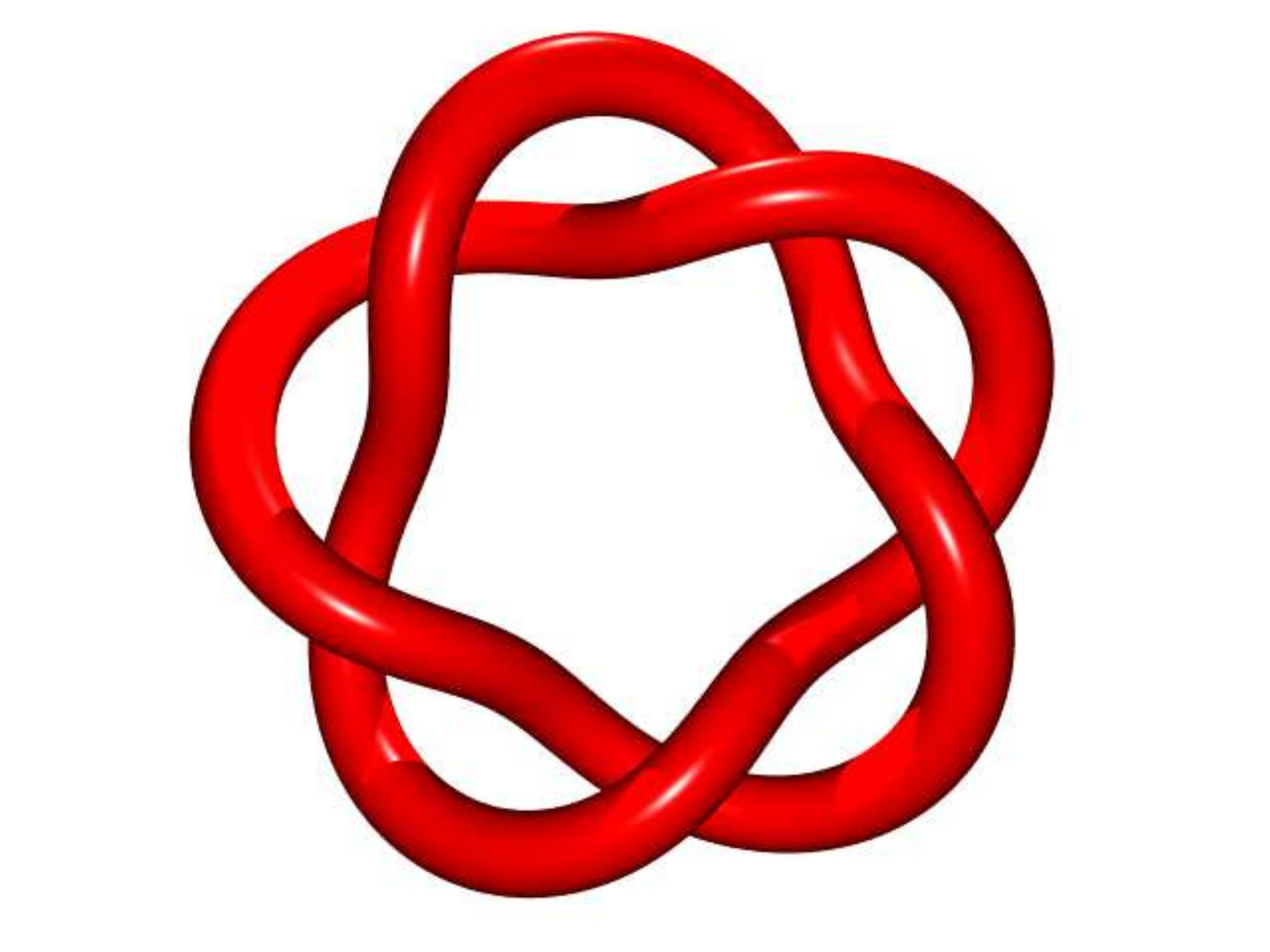} & \includegraphics[width=0.33\textwidth,keepaspectratio]{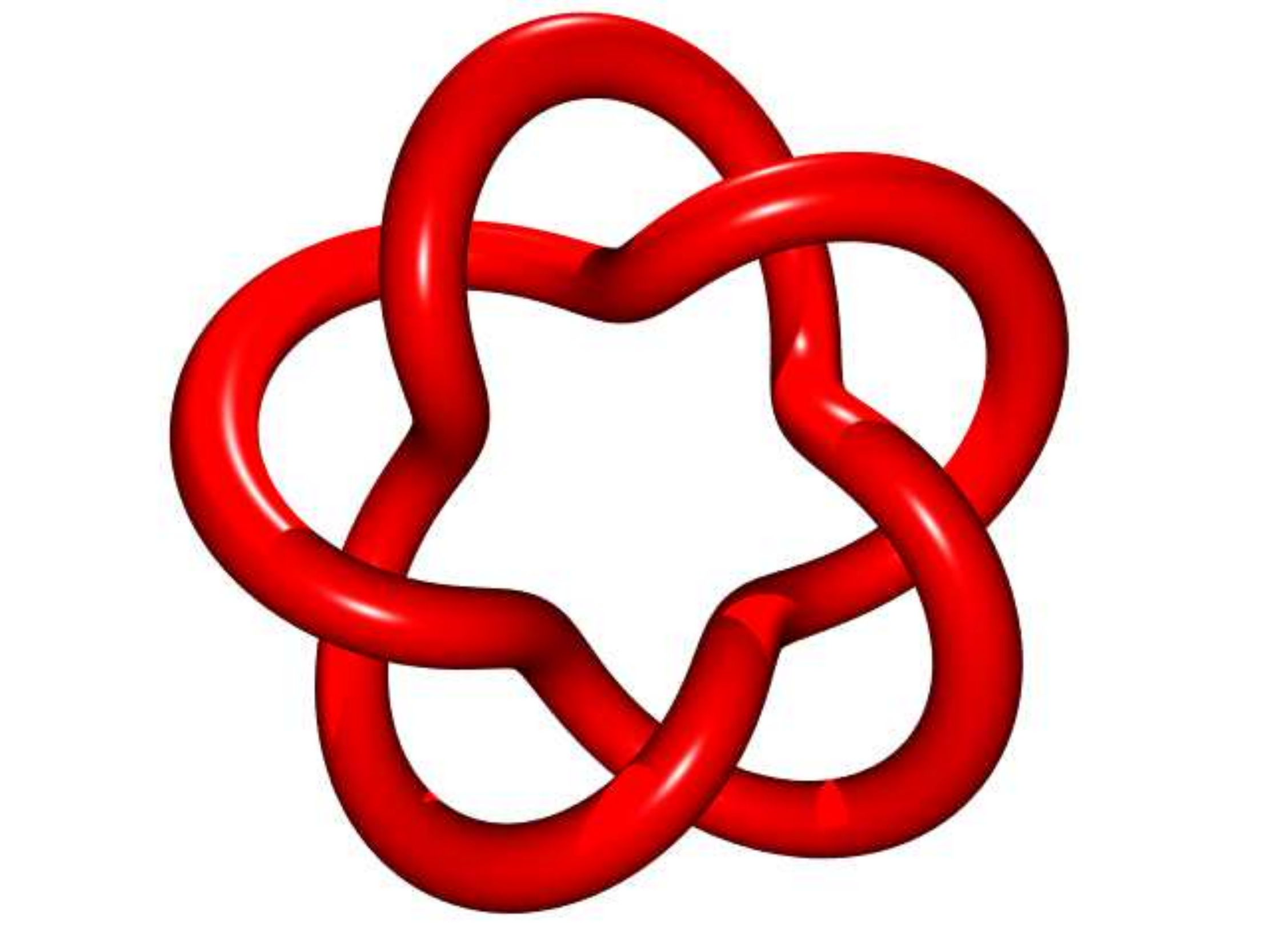} & \includegraphics[width=0.33\textwidth,keepaspectratio]{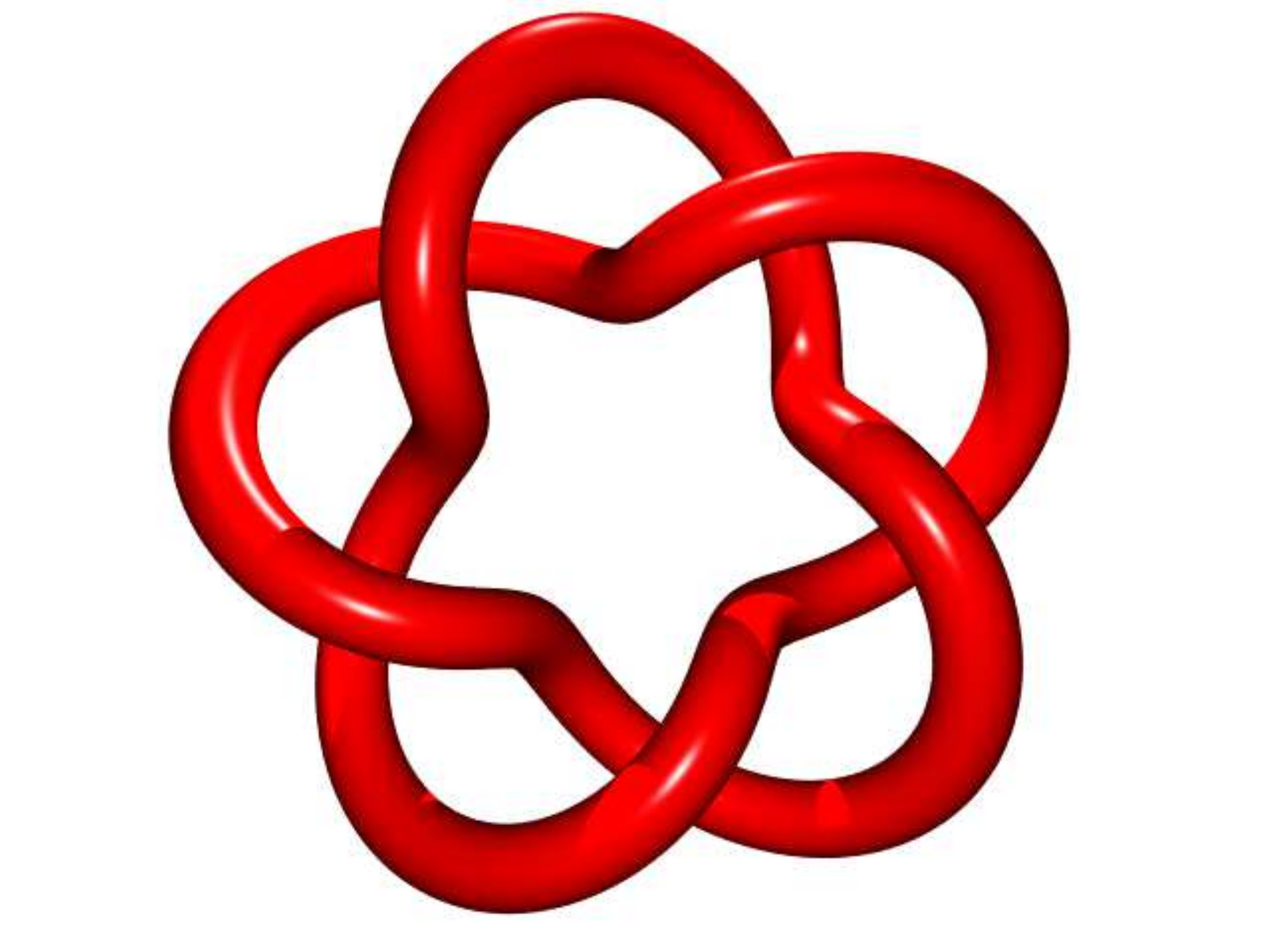} \\
208000/300000 & 220000/300000 & 300000/300000 \\
$\Le(\gamma)\approx 99.87995$ & $\Le(\gamma)\approx 101.13572$ & $\Le(\gamma)\approx 101.46585$ \\
$\E_p(\gamma)\approx 42.72715$ & $\E_p(\gamma)\approx 42.19442$ & $\E_p(\gamma)\approx 42.1562$ \\
$\tau=2080.0$ & $\tau=2200.0$ & $\tau=3000.0$ \\
\includegraphics[width=0.33\textwidth,keepaspectratio]{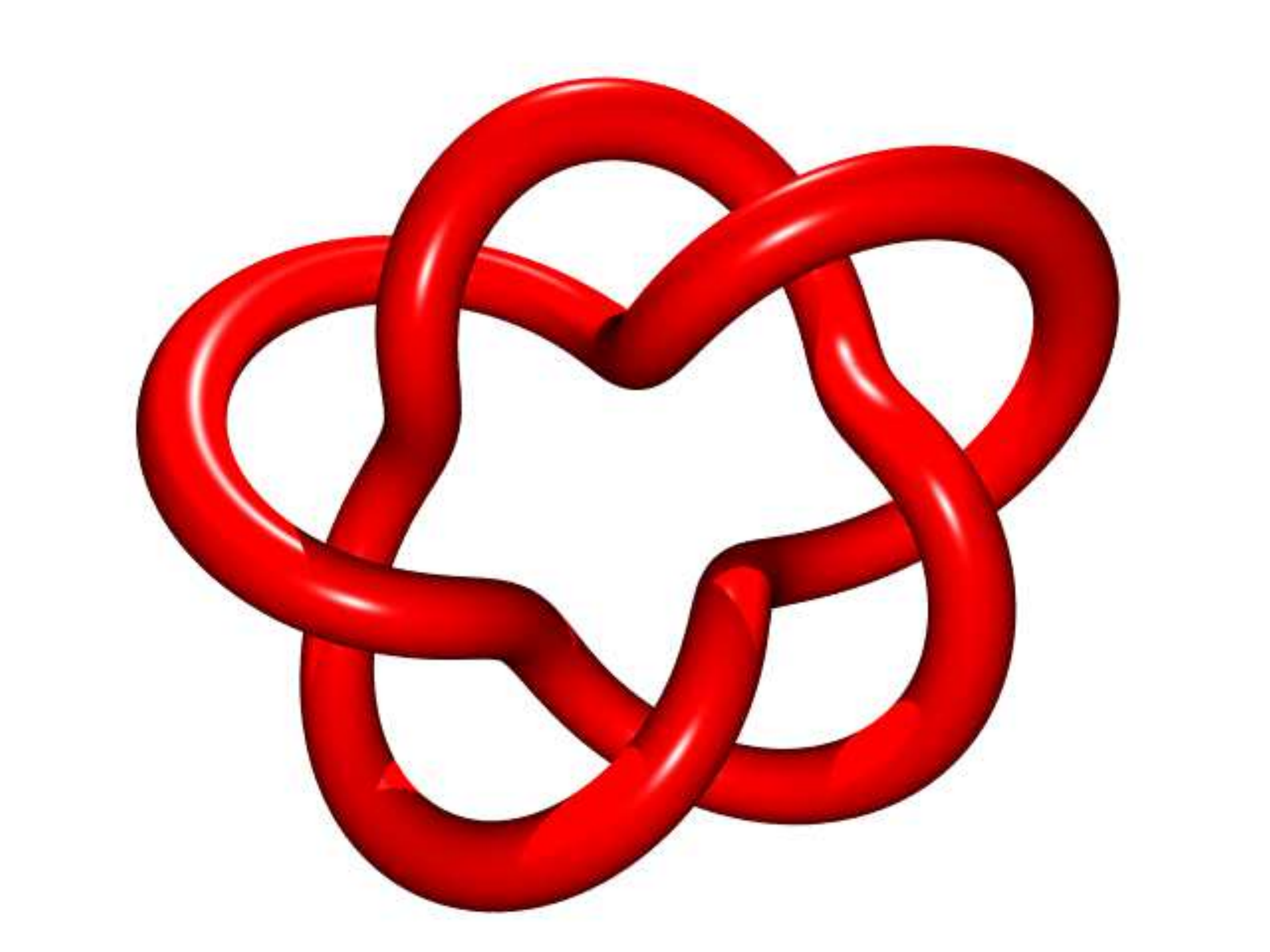} & \includegraphics[width=0.33\textwidth,keepaspectratio]{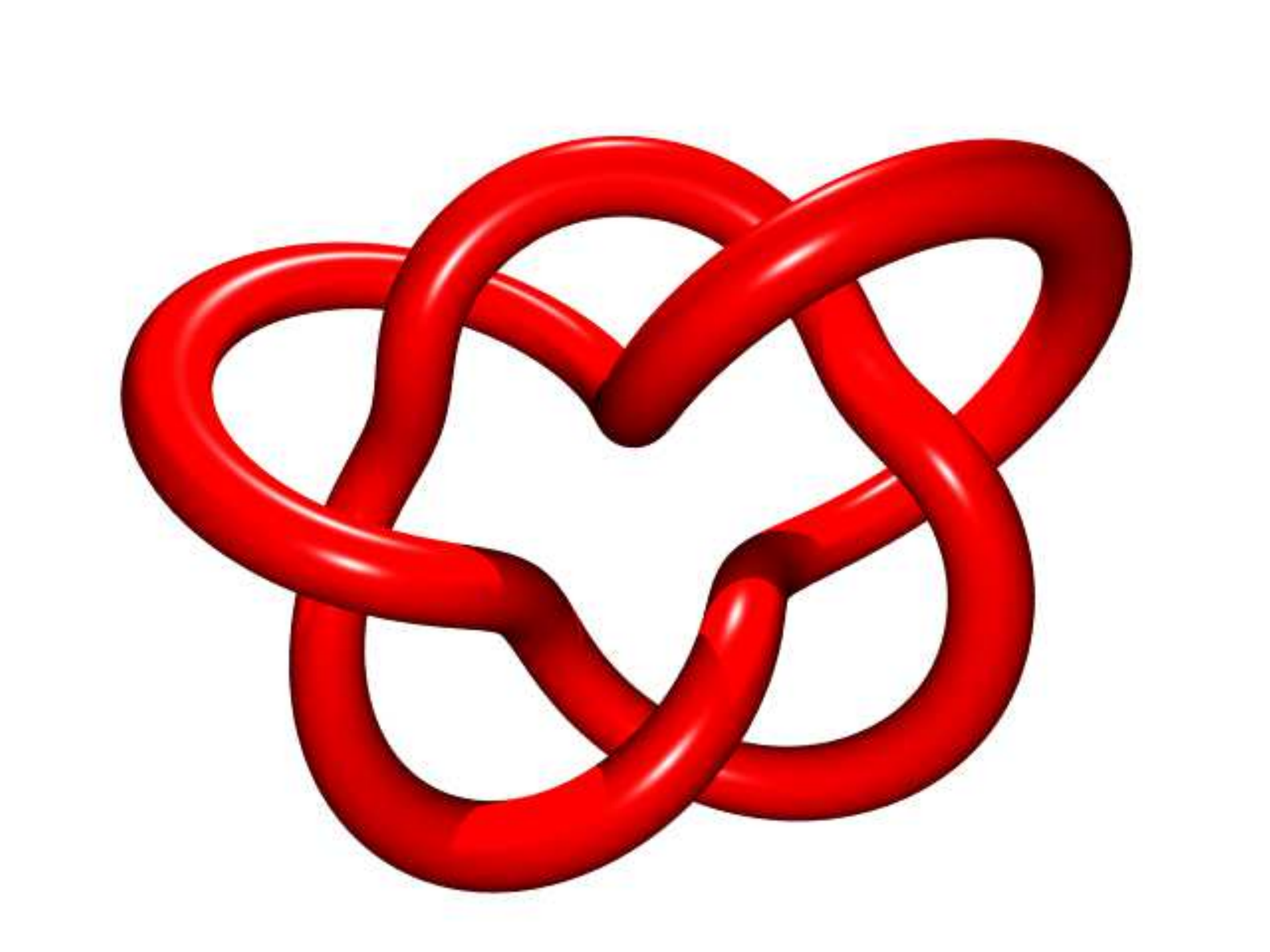} & \includegraphics[width=0.33\textwidth,keepaspectratio]{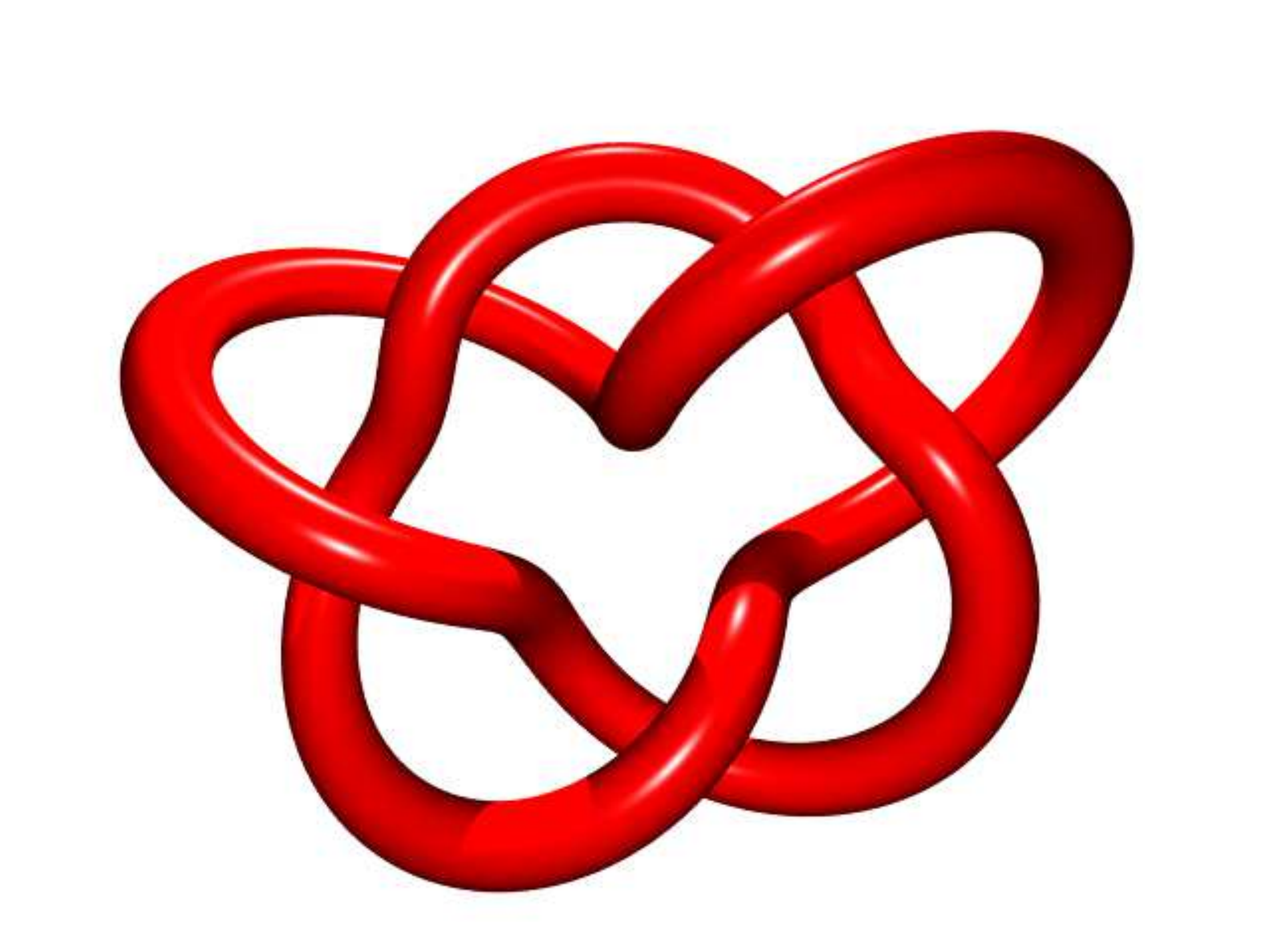}
\end{tabular}
\end{scriptsize}
\end{center}
\caption{A $5_1$ knot -- $p=50.0$ without redistribution}
\end{figure}

Observe that at first the knot stays in the symmetric configuration. It seems to be the case that the continuous flow preserves symmetries and that due to rounding-errors somewhen the symmetry is broken and the knot flows into a ``better'' configuration. With redistribution this takes longer and therefore, and only after about $1.000.000$ steps the symmetric configuration is abandoned. If we consider a rescaled version of this knot the flow with redistribution has reached the ``better'' configuration within $50.000$ steps.

\newpage
\section{$5_2$ knots} \label{knot52}
Again for $p=3$ the knot class is abandoned. However, for $p=3.5$ this is not the case, as we will see in the following pictures
\begin{figure}[H]
\begin{center}
\begin{scriptsize}
\begin{tabular}{ccc}
0/100000 & 500/100000 & 100000/100000 \\
$\Le(\gamma)\approx 24.52692$ & $\Le(\gamma)\approx 22.76258$ & $\Le(\gamma)\approx 22.58074$ \\
$\E_p(\gamma)\approx 22.20712$ & $\E_p(\gamma)\approx 21.39625$ & $\E_p(\gamma)\approx 21.34356$ \\
$\tau=0.0$ & $\tau=0.65197$ & $\tau=94.05908$ \\
\includegraphics[width=0.33\textwidth,keepaspectratio]{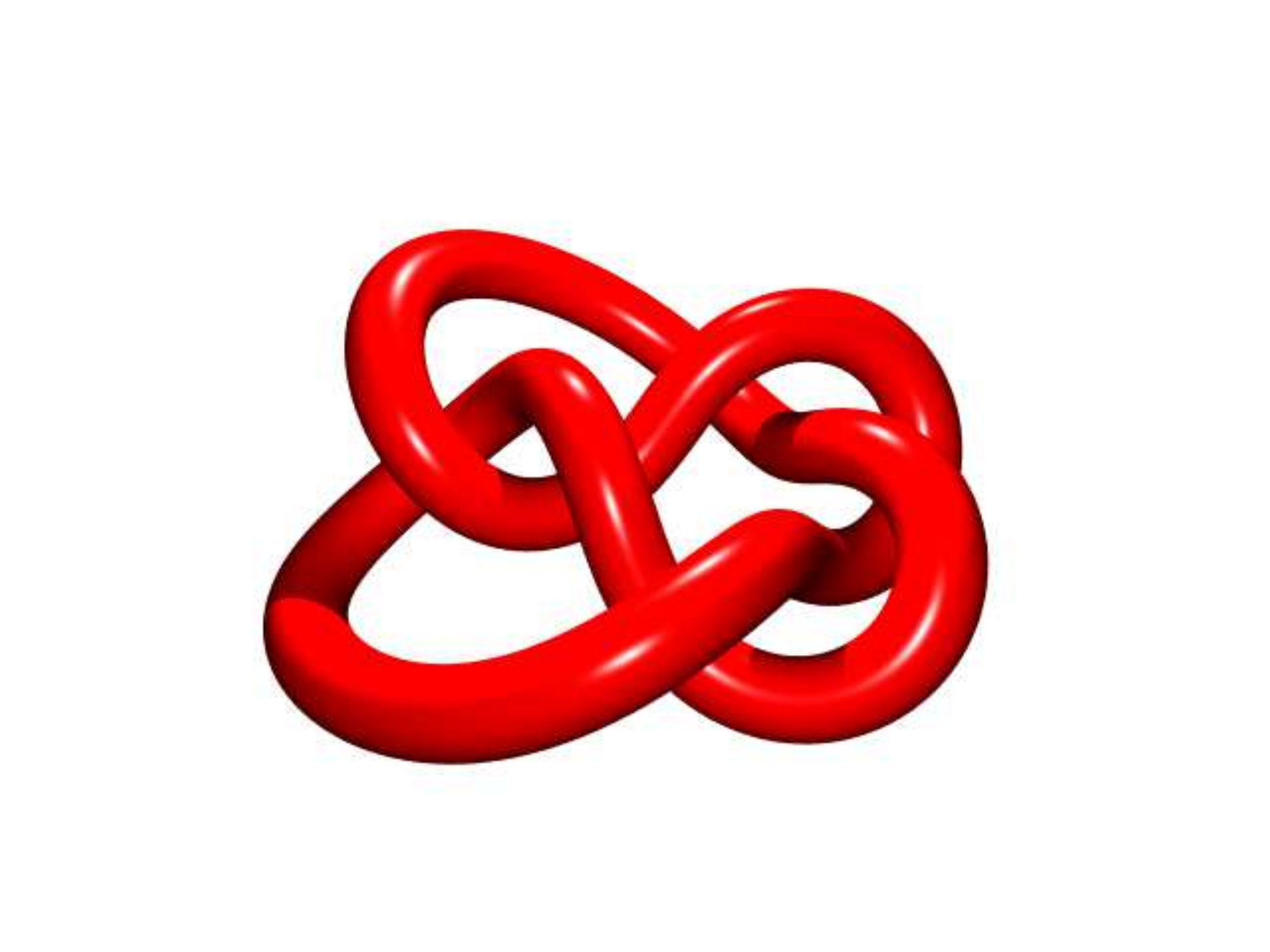} & \includegraphics[width=0.33\textwidth,keepaspectratio]{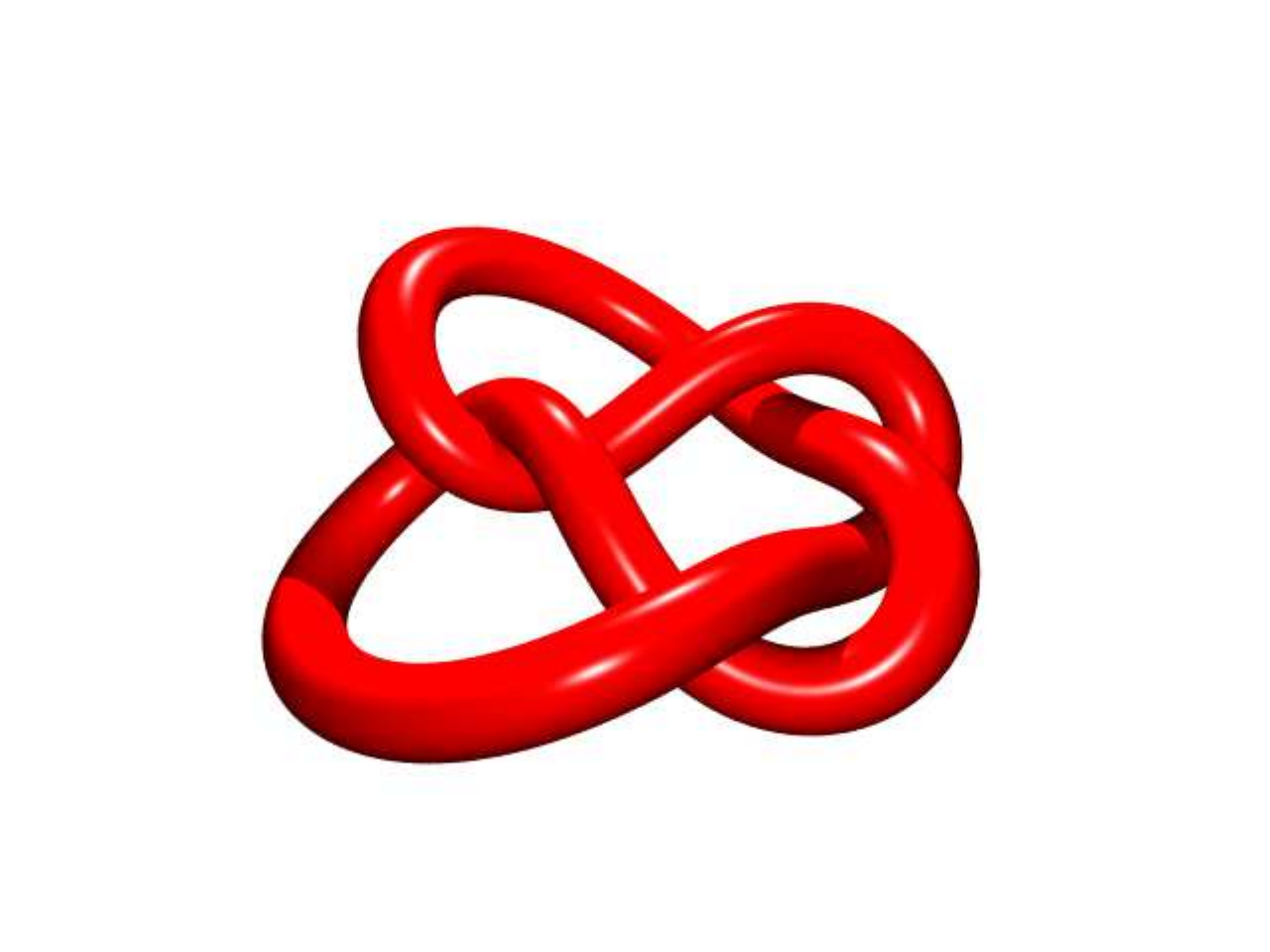} & \includegraphics[width=0.33\textwidth,keepaspectratio]{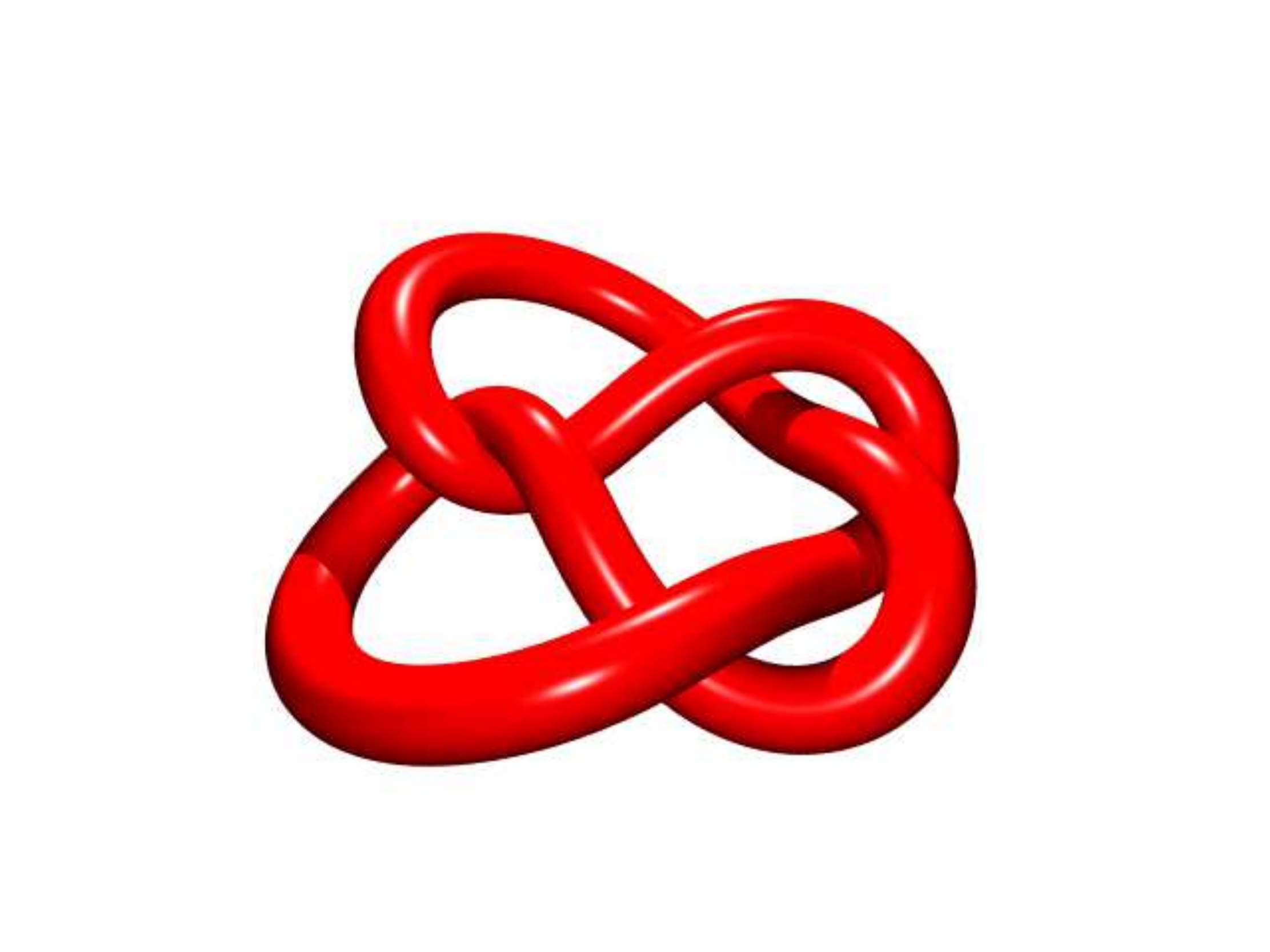}
\end{tabular}
\end{scriptsize}
\end{center}
\caption{A $5_2$ knot -- $p=3.5$}
\end{figure}

Observe that this energy value stay constant up to $300.000$ steps. After $500.000$ steps without redistribution the value $21.3434$ is reached.

\newpage
For $p=50$ the knot is transformed into a completely different configuration
\begin{figure}[H]
\begin{center}
\begin{scriptsize}
\begin{tabular}{cc}
0/50000 & 2000/50000 \\
$\Le(\gamma)\approx 24.52692$ & $\Le(\gamma)\approx 26.77052$ \\
$\E_p(\gamma)\approx 53.24807$ & $\E_p(\gamma)\approx 46.34182$ \\
$\tau=0.0$ & $\tau=1.99485$ \\
\includegraphics[width=0.4\textwidth,keepaspectratio]{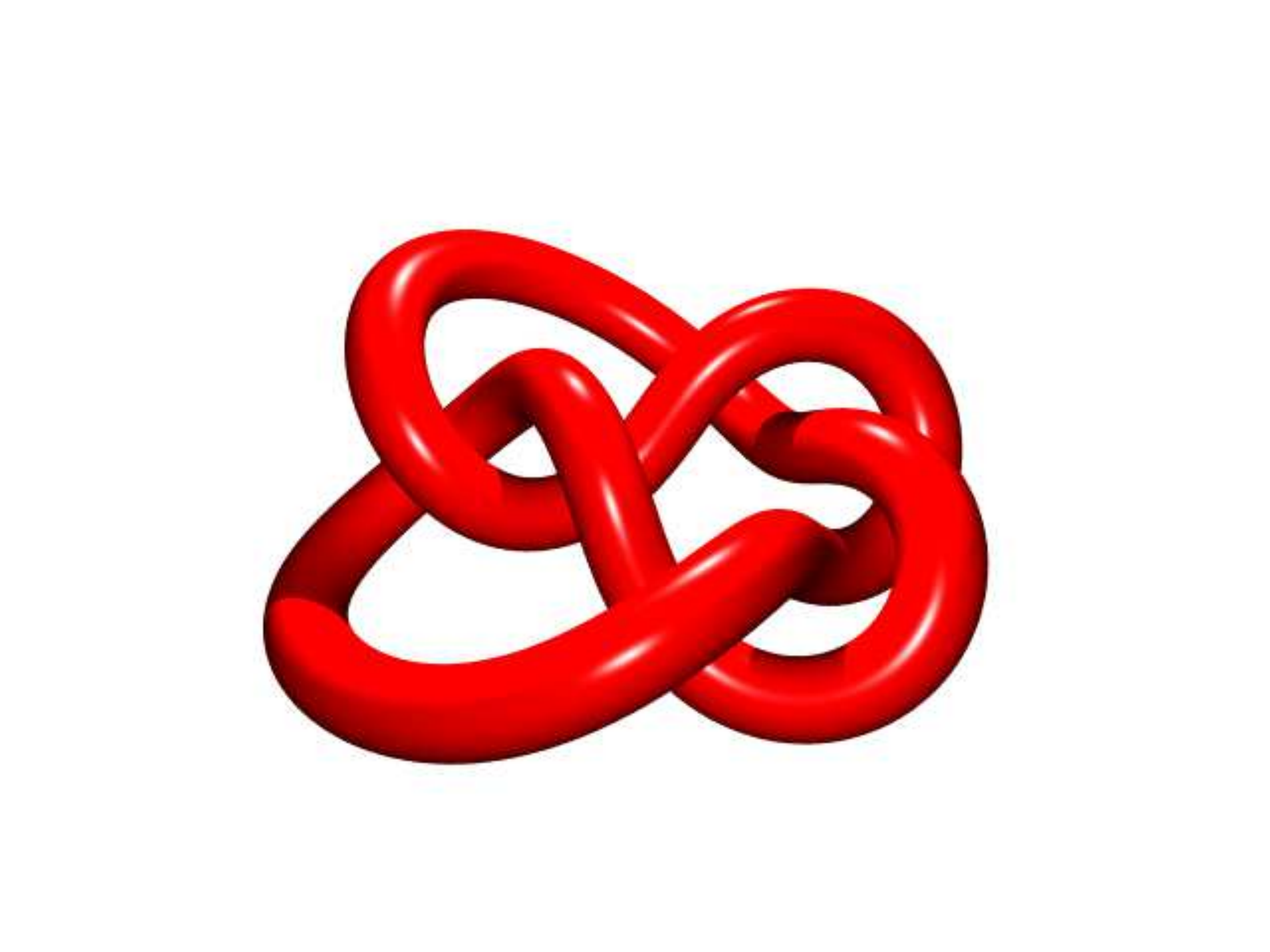} & \includegraphics[width=0.4\textwidth,keepaspectratio]{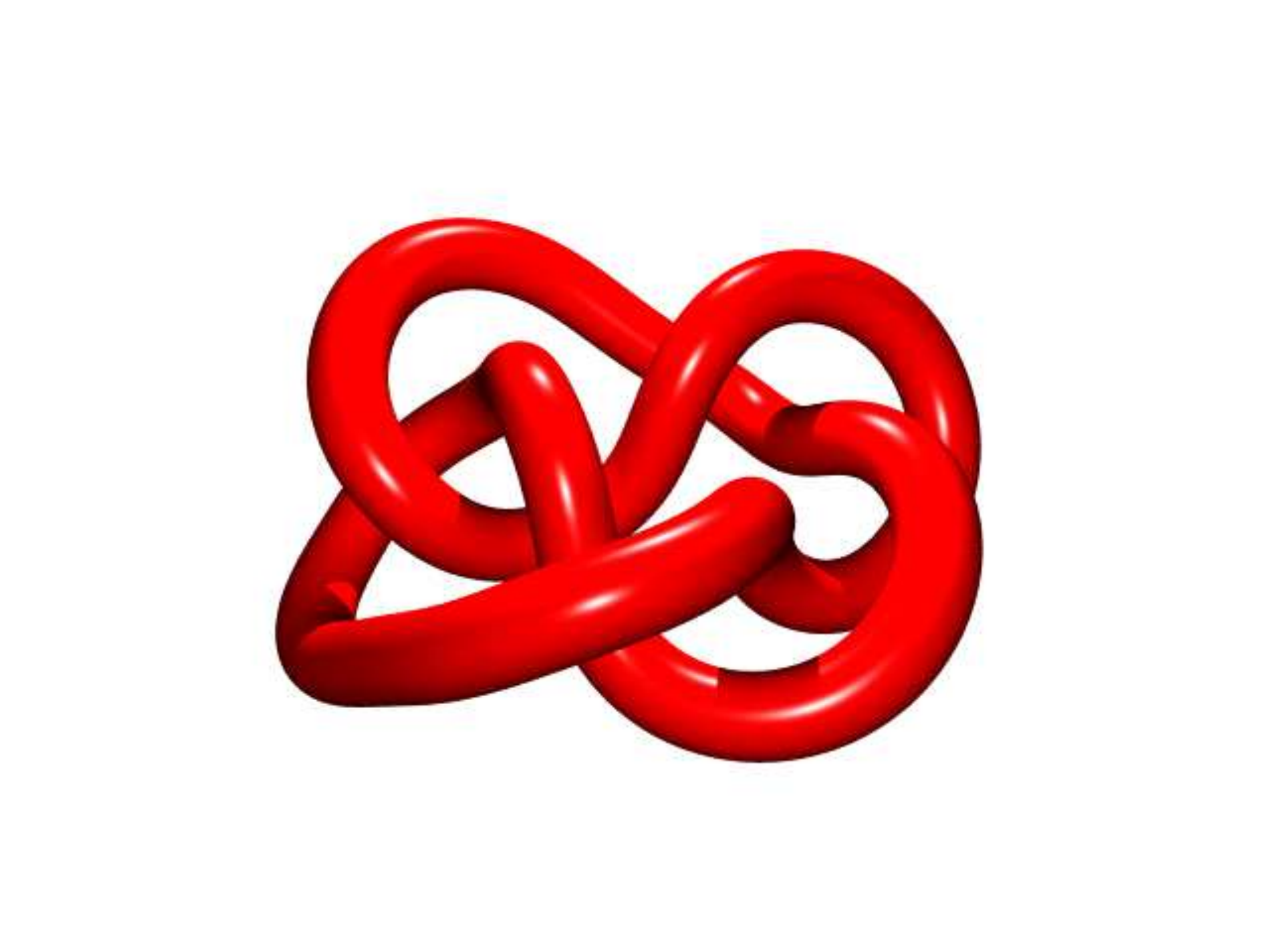} \\
6500/50000 & 10000/50000 \\
$\Le(\gamma)\approx 28.17135$ & $\Le(\gamma)\approx 29.53172$ \\
$\E_p(\gamma)\approx 45.71313$ & $\E_p(\gamma)\approx 44.03581$ \\
$\tau=7.21586$ & $\tau=11.98769$ \\
\includegraphics[width=0.4\textwidth,keepaspectratio]{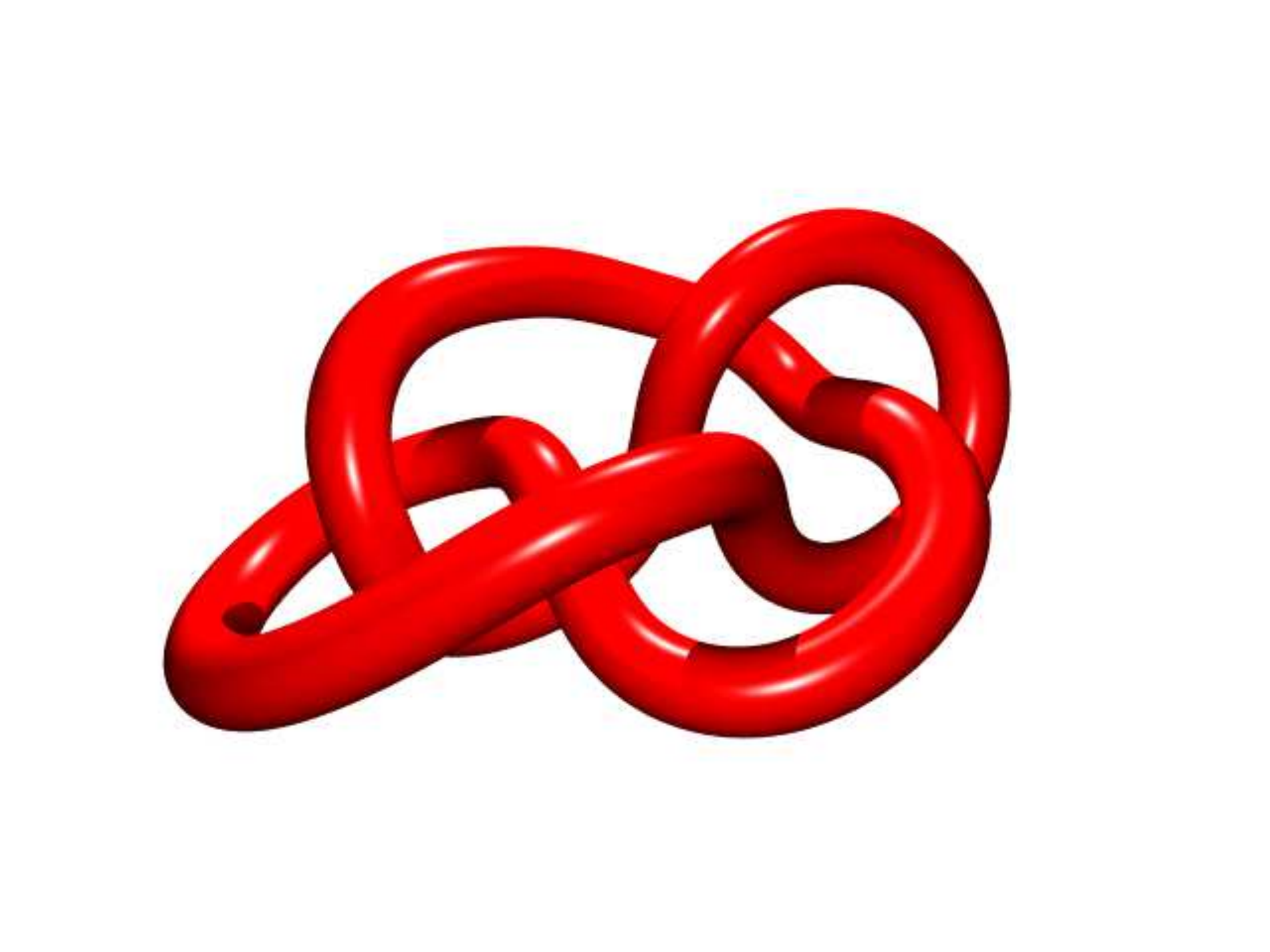} & \includegraphics[width=0.4\textwidth,keepaspectratio]{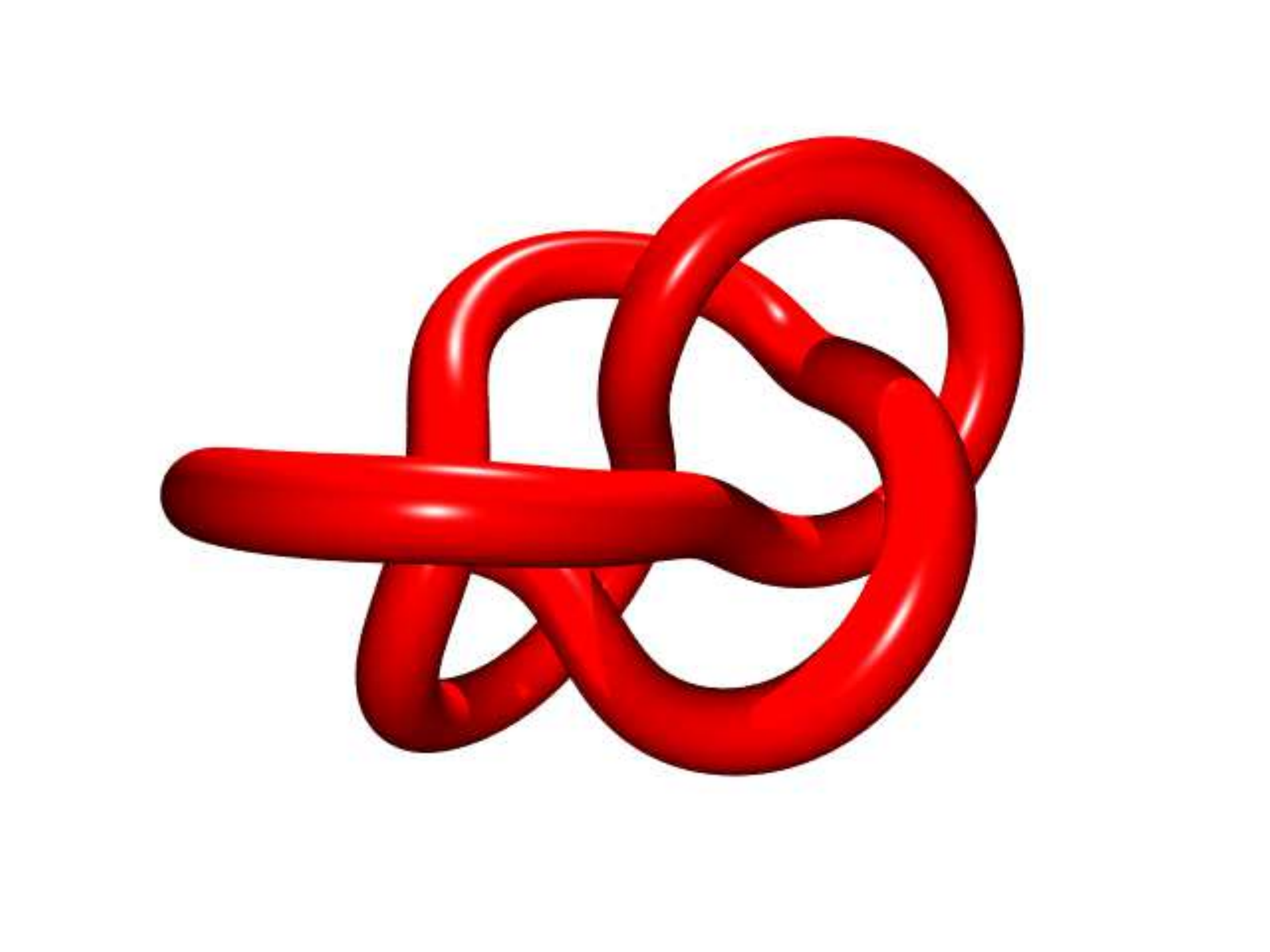} \\
5000/50000 & 50000/50000 \\
$\Le(\gamma)\approx 30.73195$ & $\Le(\gamma)\approx 40.49328$ \\
$\E_p(\gamma)\approx 43.87692$ & $\E_p(\gamma)\approx 43.87673$ \\
$\tau=20.29433$ & $\tau=116.55234$ \\
\includegraphics[width=0.4\textwidth,keepaspectratio]{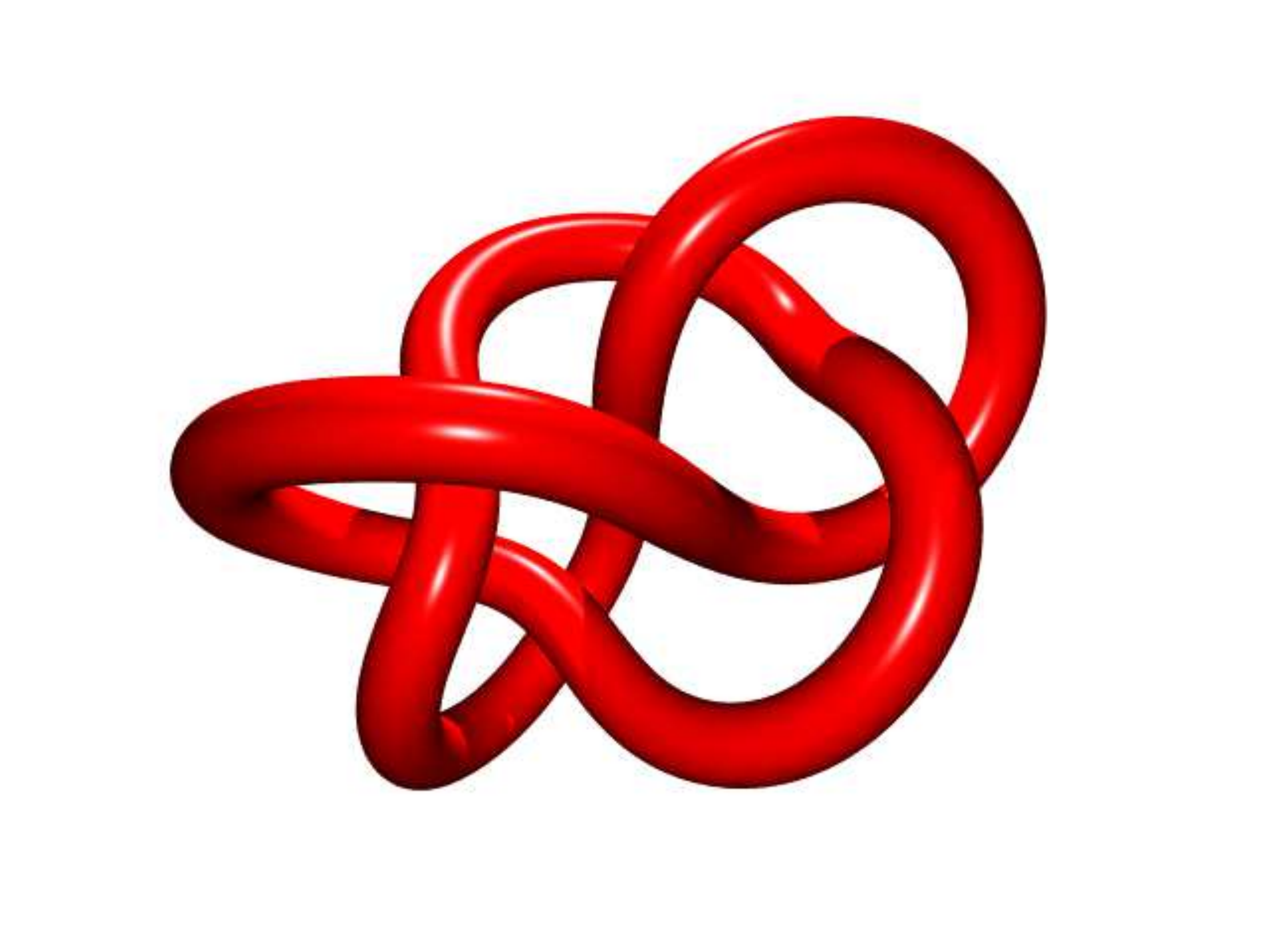} & \includegraphics[width=0.4\textwidth,keepaspectratio]{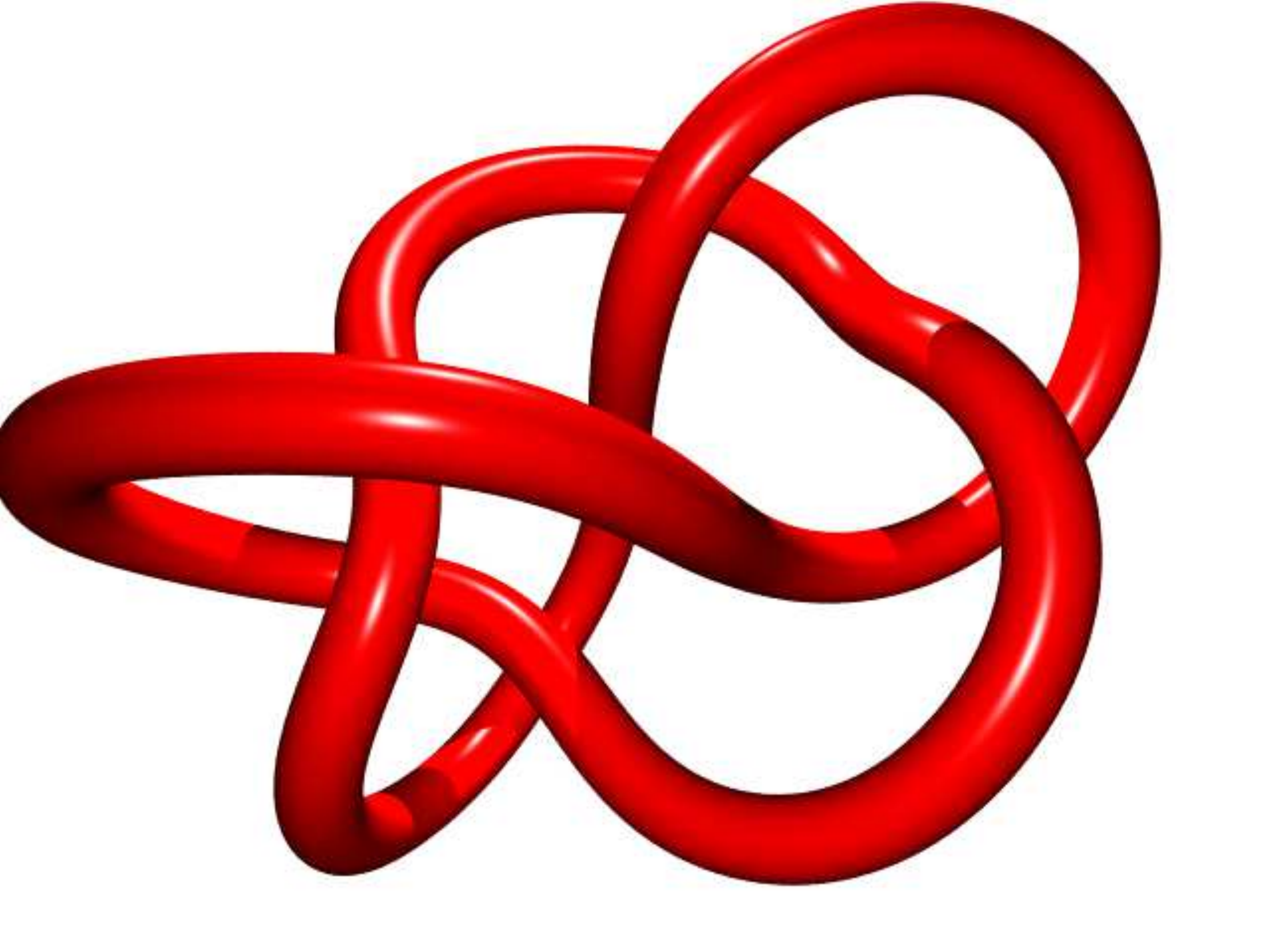}
\end{tabular}
\end{scriptsize}
\end{center}
\caption{A $5_2$ knot -- $p=50.0$}
\end{figure}

Observe, that also in this case the knot starts to grow between the last two pictures and later on the energy value even starts to increase again. After running this configuration without redistribution for $500.000$ steps we reach a value of $43.81185$.

\section{Conclusion}
The flow with the redistribution algorithm seems to be very efficient in order to come close to a stationary point of the considered energy. Nevertheless, in most cases such a stationary point seems to have \name{Fourier} coefficients that correspond to non-equidistantly distributed points on the curve. Moreover, in general the flow with redistribution seems to converge faster. If we apply the flow without redistribution to a final configuration of the flow with redistribution, the speed of convergence of the $\E_p$ energy values is very slow. However, it seems that in this case the differences in shape between the final and the initial configuration are beyond recognition.
Observe, that it is difficult to decide whether a configuration is finial or not, since the energy values can possibly stay nearly constant for a certain number of steps in middle of the flow.

For large $p$ this energy seems to effectively avoid self-penetration during the corresponding gradient flow leading to configurations similar to the ideal knots.

\cleartooddpage[\thispagestyle{empty}]



\addcontentsline{toc}{chapter}{\numberline{}\bibname}
\bibliography{bibliography}


\end{document}